\documentclass[11pt]{amsbook}

\usepackage{amsmath}
\usepackage{amssymb}
\usepackage{baththesis}
\usepackage{amsthm}
\usepackage{graphicx}
\usepackage{float}
\usepackage{fancyhdr}

\numberwithin{section}{chapter} \numberwithin{equation}{chapter}

\newtheorem{thm}{Theorem}[chapter]
\newtheorem{rem}[thm]{Remark}
\newtheorem{defn}[thm]{Definition}
\newtheorem{Lemma}[thm]{Lemma}
\newtheorem{prop}[thm]{Proposition}
\newtheorem{corol}[thm]{Corollary}

\newcommand{\R}{\mathbb{R}}
\newcommand{\C}{\mathbb{C}}

\renewcommand{\P}{\mathbb{P}}

\begin{document}

\frontmatter

\thispagestyle{empty}\vspace*{2pc}

\begin{center}
\Large{\textit{\'{A}urea Casinhas Quintino}}
\end{center}
$\newline$ $\newline$ $\newline$ $\newline$ $\newline$
\begin{center}
\textit{\Huge{Constrained Willmore Surfaces}}
\end{center}
\begin{center}
\textit{\Large{Symmetries of a M\"{o}bius Invariant Integrable
System}}
\end{center}
\begin{center}
$\newline$$\newline\newline\newline\newline
\newline\newline\newline\newline\newline\newline\newline\newline\newline$
$\newline\newline$ \Large{Based on the author's PhD Thesis}
\end{center}

\clearpage

\thispagestyle{empty} \vspace*{8pc}
\begin{center}
$\,\,\,\,\,\,\,\,\,\,\,\,\,\,\,\,
\,\,\,\,\,\,\,\,\,\,\,\,\,\,\,\,\,\,\,\,\,\,\,\,\,\,\,\,\,\,\,\,\,
\,\,\,\,\,\,\,\,\,\,\,\,\,\,\,\,\,\,\,\,\,\,\,\,\,\,\,\,\,
\,\,\,\,\,\,\,\,\,\,\,\,\,\,\,\,\,\,\,\,\,\,\,\,\,\,\,\,\,\textit{To
my parents}$
\end{center}
\cleardoublepage

\thispagestyle{empty} \vspace*{23pc}

$\textit{\textbf{Acknowledgements.}}\,\,$ \textit{My deepest
gratitude goes to Fran, Prof. Francis Burstall, my PhD supervisor,
for his constant support and attention, for sharing with me some of
his outstanding knowledge of mathematics and promising ideas, and
for doing so with the enthusiasm, the articulation and the warmth
that are characteristic of him. Fran is a great mathematician and an
exceptional human being and it was a privilege and a pleasure to
have worked with him for all these years.}

\textit{Special thanks are due to Prof. John C. Wood and Dr. Udo
Hertrich-Jeromin, my PhD examiners, for their thorough reading of my
 thesis and constructive comments to it and for the friendliness.}

\textit{A special mention to all those whose path has crossed mine
at some point of this long journey, who have greatly enriched this
experience to me. And to the squirrels, ducks, spiders, the
occasional mouse, et al., in Hyde Park, London, for the touch of
innocence to my days.}

\textit{A very special mention to my parents, for always being
there, together and well, and to Tiago, my baby nephew, who,
blissfully unaware of it, has changed the lives of so many people
for so much better.}

\textit{My PhD studies were financially supported by
Funda\c{c}\~{a}o para a Ci\^{e}ncia e a Tecnologia, Portugal, with
the scholarship with reference SFRH/BD/6356/2001.}

\cleardoublepage

\thispagestyle{empty} \vspace*{15pc}

$\textit{\textbf{Preface.}}\,\,$ This work is based on the PhD
Thesis with the same title submitted to the University of Bath,
Department of Mathematical Sciences, on September 15, 2008 and
defended viva voce on the (rather snowy) $2^{\mathrm{nd}}$ of
February 2009. The contents of the Thesis available from the
University of Bath has been preserved, with the exception of
Sections \ref{CWspdeform} and \ref{BTCMC}, where the original study
of constrained Willmore spectral deformation and B\"{a}cklund
transformation of CMC surfaces has been extended to a general
multiplier. This version has benefited from the rephrasing or
reformulation of parts of some sections.

$\newline$

$\,\,\,\,\,\,\,\,\,\,\,\,\,\,\,\,\,\,\,\,\,\,\,\,\,\,\,\,\,\,\,\,\,\,\,\,\,\,\,\,\,\,\,\,\,\,\,\,$
Lisbon, on the $29^{\mathrm{th}}$ December 2009

\cleardoublepage

\thispagestyle{empty} \vspace*{16pc}

$\textbf{Abstract.}\,\,$ This work is dedicated to the study of the
M\"{o}bius invariant class of constrained Willmore surfaces and its
symmetries. We define a spectral deformation by the action of a loop
of flat metric connections; \textit{B\"{a}cklund transformations},
by applying a dressing action; and, in $4$-space, \textit{Darboux
transformations}, based on the solution of a Riccati equation. We
establish a permutability between spectral deformation and
B\"{a}cklund transformation and prove that non-trivial Darboux
transformation of constrained Willmore surfaces in $4$-space can be
obtained as a particular case of B\"{a}cklund transformation. All
these transformations corresponding to the zero multiplier preserve
the class of Willmore surfaces. We verify that, for special choices
of parameters, both spectral deformation and B\"{a}cklund
transformation preserve the class of constrained Willmore surfaces
admitting a \textit{conserved quantity}, and, in particular, the
class of CMC surfaces in $3$-dimensional space-form.

\tableofcontents

\chapter*{Introduction}

\markboth{\tiny{A. C. QUINTINO}}{\tiny{CONSTRAINED WILLMORE
SURFACES}}

Among the classes of Riemannian submanifolds, there is that of
Willmore surfaces, named after T. Willmore \cite{willmore2} (1965),
although the topic was mentioned by W. Blaschke \cite{blaschke}
(1929) and by G. Thomsen \cite{thomsen} (1923). Early in the
nineteenth century, S. Germain \cite{germain1}, \cite{germain2}
studied elastic surfaces. On her pioneering analysis, she claimed
that the elastic force of a thin plate is proportional to its mean
curvature. Since then, the mean curvature remains a key concept in
theory of elasticity. In modern literature on the elasticity of
membranes (see, for example, \cite{landau+lifschitz} and
\cite{lipowsky}), a weighted sum of the total mean curvature, the
total squared mean curvature and the total Gaussian curvature is
considered the elastic energy of a membrane. By neglecting the total
mean curvature (by physical considerations) and having in
consideration that the total Gaussian curvature of compact
orientable Riemannian surfaces without boundary is a topological
invariant, T. Willmore defined the Willmore energy of a compact
oriented Riemannian surface, without boundary, isometrically
immersed in $\R^{3}$, to be
$$\mathcal{W}=\int H^{2}dA.$$ The Willmore functional ``extends"
to isometric immersions of compact oriented Riemannian surfaces in
Riemannian manifolds by means of half of the total squared norm of
the trace-free part of the second fundamental form, which, in fact,
amongst surfaces in $\R^{3}$, differs from $\mathcal{W}$ by the
total Gaussian curvature, but still shares then the critical points
with $\mathcal{W}$. Willmore surfaces are the extremals of the
Willmore functional - just like harmonic maps are the extremals of
the energy functional.

Conformal invariance motivates us to move from Riemannian to
M\"{o}bius geometry. Our study is a study of geometrical aspects
that are invariant under M\"{o}bius transformations, with the
exception of the study of constant mean curvature surfaces, in
Sections \ref{sec:casecodim1} and \ref{sec:CMC}. This work restricts
to the study of surfaces conformally immersed in $n$-dimensional
space-forms with $n\geq3$. It starts with a M\"{o}bius description
of space-forms in the projectivized light-cone, following \cite{IS}.
Such description is based on the model of the conformal $n$-sphere
on the projective space $\mathbb{P}(\mathcal{L})$ of the light-cone
$\mathcal{L}$ of $\R^{n+1,1}$,
$$S^{n}\cong\mathbb{P}(\mathcal{L}),$$
due to Darboux \cite{darbouxsphere}, which, in particular, yields a
conformal description of Euclidean $n$-spaces and hyperbolic
$n$-spaces as submanifolds of $\mathbb{P}(\mathcal{L})$. We approach
then a surface conformally immersed in $n$-space as a null line
subbundle $\Lambda$ of $\underline{\R}^{n+1,1}=M\times\R^{n+1,1}$
defining an immersion
$$\Lambda:M\rightarrow\mathbb{P}(\mathcal{L})$$
of an oriented surface $M$, which we provide with the conformal
structure induced by $\Lambda$, into the projectivized light-cone.
The realization of all space-forms as submanifolds of the
projectivized light-cone arises from the realization, cf. \cite{IS},
of all $n$-dimensional space-forms as connected components of conic
sections
$$S_{v_{\infty}}:=\{v\in\mathcal{L}:(v,v_{\infty})=-1\}$$
of the light-cone, with $v_{\infty}\in\mathbb{R}^{n+1,1}$ non-zero.
$S_{v_{\infty}}$ inherits a positive definite metric of constant
sectional curvature $-(v_{\infty},v_{\infty})$ from $\R^{n+1,1}$ and
is either a copy of a sphere, a copy of Euclidean space or two
copies of hyperbolic space, according to the sign of
$(v_{\infty},v_{\infty})$. For each $v_{\infty}$, the canonical
projection $\pi:\mathcal{L}\rightarrow\mathbb{P}(\mathcal{L})$
defines a diffeomorphism
\begin{equation}\label{eq:Svinfdiffeos}
\pi\vert_{S_{v_{\infty}}}:S_{v_{\infty}}\rightarrow \mathbb{P}
(\mathcal {L})\backslash\mathbb{P} (\mathcal {L}\cap \langle
v_{\infty}\rangle ^{\perp}).
\end{equation}
We provide $\mathbb{P}(\mathcal{L})$ with the conformal structure of
the metric induced by $\pi\vert_{S_{v_{\infty}}}$, fixing
$v_{\infty}$ time-like, independently of the choice of $v_{\infty}$,
and, in this way, identify $\mathbb{P}(\mathcal{L})$ with the
conformal $n$-sphere and make each diffeomorphism
\eqref{eq:Svinfdiffeos} - for a general $v_{\infty}$, not
necessarily time-like - into a conformal diffeomorphism.

In this work, we restrict to surfaces $\Lambda$ in $S^{n}$ which are
not contained in any subsphere of $S^{n}$. Such a surface $\Lambda$
defines a surface in any given space-form, by means of a lift, whose
study is M\"{o}bius equivalent to the study of $\Lambda$ and which
will often be considered. Namely, given $v_{\infty}\in\R^{n+1,1}$
non-zero, we have, locally, $(\sigma,v_{\infty})\neq 0$, and
$\Lambda$ defines then a local immersion
$$\sigma_{\infty}:=(\pi\vert_{{S_{v_{\infty}}}})^{-1}\circ
\Lambda=\frac{-1}{(\sigma,v_{\infty})}\,\sigma:M\rightarrow
S_{v_{\infty}},$$ of $M$ into the space-form $S_{v_{\infty}}$.

Having presented our setup, in Chapter \ref{chaptercsc}, we
introduce, following \cite{SD}, the central sphere congruence, a
fundamental construction of M\"{o}bius invariant surface geometry
which will be basic to our study of surfaces. The concept has its
origins in the nineteenth century with the introduction of the mean
curvature sphere of a surface at a point, by S. Germain
\cite{germain3}. By the turn of the century, the family of the mean
curvature spheres of a surface was known as the central sphere
congruence, cf. W. Blaschke \cite{blaschke}. Nowadays, after R.
Bryant's paper \cite{bryant}, it goes as well by the name conformal
Gauss map. The central sphere congruence of a surface in $n$-space,
$$S:M\rightarrow Gr_{(3,1)}(\R^{n+1,1}),$$  defines a decomposition
\begin{equation}\label{eq:dcurlyDcurlyNintrodversio}
d=\mathcal{D}+\mathcal{N},
\end{equation}
of the trivial flat connection on $\underline{\R}^{n+1,1}$ into the
sum of a connection $\mathcal{D}$, with respect to which $S$ and
$S^{\perp}$ are parallel, and a $1$-form $\mathcal{N}$ with values
in $S\wedge S^{\perp}$. Explicitly, $$\mathcal{D}
:=\nabla^{S}+\nabla^{S^{\perp}},\,\,\,\,\mathcal{N}
:=d-\mathcal{D},$$ for $\nabla^{S}$ and $\nabla^{S^{\perp}}$ the
connections induced by $d$ on $S$ and $S^{\perp}$, respectively.
Under the standard identification
$$S^{*}TGr_{(3,1)}(\R^{n+1,1})\cong\mathrm{Hom}(S,S^{\perp})\cong
S\wedge S^{\perp},$$ of bundles provided with a metric and a
connection, we have
\begin{equation}\label{eq:dS=curlyN}
dS=\mathcal{N},
\end{equation}
which establishes a characterization of the harmonicity of $S$ by
$$d^{\mathcal{D}}*\mathcal{N}=0.$$

Chapter 4 is dedicated to the class of Willmore surfaces in
space-forms and its link to the class of harmonic maps into
Grassmannian manifolds via the central sphere congruence. W.
Blaschke \cite{blaschke} established the M\"{o}bius invariance of
the Willmore energy of a surface in spherical $3$-space. B.-Y. Chen
\cite{chen} generalized it to surfaces in constant curvature
Riemnannian manifolds. We present a manifestly conformally invariant
formulation of the Willmore energy of a surface $\Lambda$ in
$n$-dimensional space-form,
$$\mathcal {W}(\Lambda )=\frac{1}{2}\int _{M}\left(\mathcal {N}\wedge *\mathcal{N}\right),$$
following the definition of energy of the mean curvature sphere
congruence of a surface in spherical $4$-space, presented in
\cite{quaternionsbook}. The class of Willmore surfaces in $n$-space
is then established as invariant under the group of M\"{o}bius
transformations of $S^{n}$. As immediately established by
\eqref{eq:dS=curlyN}, and already known to Blaschke \cite{blaschke}
for the particular case of spherical $3$-space, the Willmore energy
of a surface in a space-form coincides with the energy of its
central sphere congruence. Furthermore, a result by Blaschke
\cite{blaschke} (for $n=3$) and N. Ejiri \cite{ejiri} (for general
$n$) characterizes Willmore surfaces in spherical $n$-space by the
harmonicity of the central sphere congruence. Via this
characterization, the class of Willmore surfaces in space-forms is
associated to a class of harmonic maps into Grassmannians. This will
enable us to apply to this class of surfaces the well-developed
integrable systems theory of harmonic maps into Grassmannian
manifolds, with a spectral deformation and B\"{a}cklund
transformations, cf. \cite{uhlenbeck} and \cite{uhlenbeck 89}.

In many occasions throughout this work, we use an interpretation of
loop group theory by F. Burstall and D. Calderbank
\cite{burstall+calderbank} and produce transformations of surfaces
by the action of loops of flat metric connections. Specifically, by
replacing the trivial flat connection by another flat metric
connection $\tilde{d}$ on $\underline{\R}^{n+1,1}$, we transform (in
certain cases) a surface $\Lambda\subset \underline{\R}^{n+1,1}$
into a $\tilde{d}$-\textit{surface} $\tilde{\Lambda}$, or,
equivalently, into another surface $\tilde{\phi}\Lambda$, defined,
up to a M\"{o}bius transformation, for
$$\tilde{\phi}:(\underline{\R}^{n+1,1},\tilde{d})\rightarrow(\underline{\R}^{n+1,1},d)$$
an isomorphism of bundles provided with a metric and a connection.
Many will be the examples in this work of such transformations
preserving the geometrical aspects of a class, i.e., establishing
symmetries of integrable systems. Symmetries of integrable systems
will arise from other constructions, as well. Throughout this work,
by transformation/deformation of a class of surfaces shall be
understood a symmetry of the system, i.e., a
transformation/deformation of the surfaces in the class into new
ones (possibly isomorphic) still in the class.

Chapter $3$ is introductory of the idea of a surface under change of
flat metric connection. A first example, due to F. Burstall and D.
Calderbank \cite{burstall+calderbank}, of a symmetry of an
integrable system arising from the action of a loop of flat metric
connections establishes the class of Willmore surfaces in
space-forms as an integrable system with a spectral deformation, a
fact that was already known to F. Burstall et al. \cite{SD}.
According to K. Uhlenbeck \cite{uhlenbeck 89}, the harmonicity of
$S$ is characterized by the flatness of the real metric connection
$d^{\lambda}:=\mathcal{D} +\lambda ^{-1}\mathcal{N}^{1,0}+\lambda
\mathcal{N}^{0,1}$ on $(\underline{\R}^{n+1,1})^{\C}$, for each
$\lambda \in S^{1}$. The action of this loop of curvature-free
connections defines a $S^{1}$-deformation of $S$ into harmonic maps,
which, as we verify, is the family of central sphere congruences
corresponding to the $S^{1}$-deformation of $\Lambda$ defined by the
action of the loop. The characterization of Willmore surfaces in
space-forms in terms of the harmonicity of the central sphere
congruence gives rise, in this way, to a spectral deformation of
Willmore surfaces. This deformation coincides, up to
reparametrization, with the one presented in \cite{SD}.

B\"{a}cklund transformations of Willmore surfaces will arise from a
more complex construction, following the work of C.-L. Terng and K.
Uhlenbeck \cite{uhlenbeck}.

In Chapter \ref{sec:CWseq}, we introduce constrained Willmore
surfaces, the generalization of Willmore surfaces that arises when
we consider extremals of the Willmore functional with respect to
\textit{infinitesimally conformal} variations,\footnote{To which
references as \textit{conformal variations} can be found in the
literature.} rather than with respect to all variations. A variation
$(\Lambda_{t})_{t}$ of a surface $\Lambda$ through null line
subbundles of $\underline{\R}^{n+1,1}$ defining immersions of $M$
into $\mathbb{P}(\mathcal{L})$ is said to be infinitesimally
conformal if, fixing $Z\in\Gamma(T^{1,0}M)$ (respectively,
$Z\in\Gamma(T^{0,1}M)$), locally never-zero, and, for each $t$, $g
_{t}$ in the conformal class of metrics induced in $M$ by
$\Lambda_{t}$, we have
$$\frac{d}{dt}_{\vert_{t=0}} g_{t}(Z,Z)=0,$$
Conformal variations, characterized by the $g _{t}$-isotropy of
$T^{1,0}M$ (respectively, $T^{0,1}M$), for all $t$, are, in
particular, infinitesimally conformal variations. Constrained
Willmore surfaces form a M\"{o}bius invariant class of surfaces with
strong links to the theory of integrable systems, as we shall
explore in this work.

F. Burstall et al. \cite{SD} established a manifestly conformally
invariant characterization of constrained Willmore surfaces in
space-forms, which, in particular, extended the concept of
constrained Willmore to surfaces that are not necessarily compact.
Chapter 5 is dedicated to deriving from the variational problem the
reformulation of this characterization, by F. Burstall and D.
Calderbank \cite{burstall+calderbank}, presented below. The argument
consists of a generalization to $n$-space of the argument presented
in \cite{christoph2} for the particular case of $n=3$. Set
$$\Lambda^{1,0}:=\Lambda\oplus
d\sigma(T^{1,0}M)\,\,\,\,\mathrm{and}\,\,\,\,\Lambda^{0,1}:=\Lambda\oplus
d\sigma(T^{0,1}M),$$ independently of $\sigma\in\Gamma(\Lambda)$
never-zero, and then
$$\Lambda^{(1)}:=\Lambda^{1,0}+\Lambda^{0,1}.$$
Cf. \cite{burstall+calderbank}, $\Lambda$ is a constrained Willmore
surface if and only if there exists a real form $q\in\Omega
^{1}(\Lambda\wedge\Lambda ^{(1)})$ with
\begin{equation}\label{eq:dcurlyDq0}
d^{\mathcal{D}}q=0,
\end{equation}
such that
\begin{equation}\label{eq:CWeqggah55dkl??pojh}
d^{\mathcal{D}}*\mathcal{N}=2\,[q\wedge *\mathcal{N}].
\end{equation}
In this case, we may refer to $\Lambda$ as, specifically, a
$q$-constrained Willmore surface and to $q$ as a [Lagrange]
multiplier to $\Lambda$. Willmore surfaces are the $0$-constrained
Willmore surfaces. The zero multiplier is not necessarily the only
multiplier to a constrained Willmore surface with no constraint on
the conformal structure. In fact, in Chapter \ref{isoCW}, we
characterize isothermic constrained Willmore surfaces in space-forms
by the non-uniqueness of multiplier. On the other hand, a classical
result by Thomsen \cite{thomsen} characterizes isothermic Willmore
surfaces in $3$-space as minimal surfaces in some space-form.

A multiplier to a surface $\Lambda$ in the projectivized light-cone
is, in particular, a real form
$q\in\Omega^{1}(\Lambda\wedge\Lambda^{(1)})$. For such a $q$,
equations \eqref{eq:dcurlyDq0} and \eqref{eq:CWeqggah55dkl??pojh},
together, encode the flatness of the metric connection
$$d^{\lambda}_{q}:=\mathcal{D}+\lambda^{-1}\mathcal{N}^{1,0}+\lambda\mathcal{N}^{0,1}+(\lambda^{-2}-1)\,q^{1,0}+(\lambda^{2}-1)\,q^{0,1},$$
on $(\underline{\R}^{n+1,1})^{\C}$, for all
$\lambda\in\C\backslash\{0\}$, or, equivalently, for all $\lambda\in
S^{1}$. Constrained Willmore surfaces in space-forms, admitting $q$
as a multiplier, are characterized by the flatness of the
$S^{1}$-family of metric connections $d^{\lambda}_{q}$ on
$\underline{\R}^{n+1,1}$, in an integrable systems interpretation
due to F. Burstall and D. Calderbank \cite{burstall+calderbank}.
This characterization will enable us, in Chapter
\ref{transformsofCW}, to define a spectral deformation of
constrained Willmore surfaces in space-forms, by the action of the
loop of flat metric connections $d^{\lambda}_{q}$, as well as a
\textit{B\"{a}cklund transformation}, by applying a dressing action.

Our transformations of constrained Willmore surfaces will be based
on the \textit{constrained harmonicity} of the central sphere
congruence. Given $\hat{d}$ a flat metric connection on
$\underline{\C}^{n+2}$ and $V$ a non-degenerate subbundle of
$\underline{\C}^{n+2}$, we generalize naturally the decomposition
\eqref{eq:dcurlyDcurlyNintrodversio} to a decomposition
$$\hat{d}=\hat{\mathcal{D}}_{V}+\hat{\mathcal{N}}_{V}$$
and, given $q\in\Omega^{1}(\wedge^{2}V\oplus\wedge^{2}V^{\perp})$,
define then, for each $\lambda\in\C\backslash\{0\}$, a connection
$$\hat{d}^{\lambda,q}_{V}:=\hat{\mathcal{D}}_{V}+\lambda^{-1}\hat{\mathcal{N}}_{V}^{1,0}+\lambda\hat{\mathcal{N}}_{V}^{0,1}+(\lambda^{-2}-1)q^{1,0}+(\lambda^{2}-1)q^{0,1},$$
on $\underline{\C}^{n+2}$, generalizing
$d^{\lambda}_{q}=d^{\lambda,q}_{S}$. We define the bundle $V$ to be
$(q,\hat{d})$-\textit{constrained harmonic} if
$\hat{d}^{\lambda,q}_{V}$ is flat, for all
$\lambda\in\C\backslash\{0\}$, or, equivalently, for all $\lambda\in
S^{1}$. A simple, yet crucial, observation is that, given
$\tilde{d}$ another flat metric connection on $\underline{\C}^{n+2}$
and $\phi: (\underline{\C}^{n+2}, \tilde{d})\rightarrow
(\underline{\C}^{n+2},\hat{d})$ an isomorphism of bundles provided
with a metric and a connection, $V$ is $(q,\tilde{d})$-constrained
harmonic if and only if $\phi V$ is
$(\mathrm{Ad}_{\phi}q,\hat{d})$-constrained harmonic.  The
constrained harmonicity of a bundle applies to the central sphere
congruence, providing a characterization of constrained Willmore
surfaces in space-forms.

The transformations of a constrained Willmore surface $\Lambda$ in
the projectivized light-cone we present are, in particular, pairs
$((\Lambda^{1,0})^{*},(\Lambda^{0,1})^{*})$ of transformations
$(\Lambda^{1,0})^{*}$ and $(\Lambda^{0,1})^{*}$ of $\Lambda^{1,0}$
and $\Lambda^{0,1}$, respectively. The fact that $\Lambda^{1,0}$ and
$\Lambda^{0,1}$ intersect in a rank $1$ bundle will ensure that
$(\Lambda^{1,0})^{*}$ and $(\Lambda^{0,1})^{*}$ have the same
property. The isotropy of $\Lambda^{1,0}$ and $\Lambda^{0,1}$ will
ensure that of $(\Lambda^{1,0})^{*}$ and $(\Lambda^{0,1})^{*}$ and,
therefore, of their intersection. The reality of the bundle
$\Lambda^{1,0}\cap\Lambda^{0,1}$ and the fact that it defines an
immersion of $M$ into $\mathbb{P}(\mathcal{L})$ are preserved by the
spectral deformation, but it is not clear that the same is
necessarily true for B\"{a}cklund transformations. This motivates us
to define \textit{complexified surface}  and, thereafter,
\textit{complexified constrained Willmore surface}.

The spectral deformation defined by the action of the loop of flat
metric connections $d^{\lambda}_{q}$ coincides, up to
reparametrization, with the one presented in \cite{SD}. More
interestingly, we define a B\"{a}cklund transformation of
constrained Willmore surfaces in space-forms. We use a version of
the dressing action theory of C.-L. Terng and K. Uhlenbeck
\cite{uhlenbeck}. We start by defining a local action of a group of
rational maps on the set of flat metric connections of the type
$\hat{d}^{\lambda,q}_{S}$, with $\hat{d}$ flat metric connection on
$\underline{\C}^{n+2}$ and
$q\in\Omega^{1}(\wedge^{2}S\oplus\wedge^{2}S^{\perp})$. Namely,
given $r=r(\lambda)\in\Gamma(O(\underline{\C}^{n+2}))$ holomorphic
at $\lambda=0$ and $\lambda=\infty$ and twisted in the sense that
$\rho r(\lambda)\rho=r(-\lambda)$, for $\rho$ reflection across $S$,
we define a $1$-form $\hat{q}$ with values in $\wedge^{2}S$ (note
that the fact that $r(\lambda)$ is twisted establishes that both
$r(0)$ and $r(\infty)$ preserve $S$) by
$$\hat{q}^{1,0}:=\mathrm{Ad}_{r(0)}q^{1,0},\,\,\,\,\hat{q}\,^{0,1}:=\mathrm{Ad}_{r(\infty)}q^{0,1},$$
and a new family of metric connections from $d^{\lambda,q}_{S}$ by
$\hat{d}^{\lambda,\hat{q}}_{S}:=r(\lambda)\circ
d^{\lambda,q}_{S}\circ r(\lambda)^{-1}$. Obviously, for each
$\lambda$, the flatness of $\hat{d}^{\lambda,\hat{q}}_{S}$ is
equivalent to that of $d^{\lambda,q}_{S}$. Crucially, if
$\hat{d}^{\lambda,\hat{q}}_{S}$ admits a holomorphic extension to
$\lambda\in\C\backslash\{0\}$ through metric connections on
$\underline{\C}^{n+2}$, then the notation
$\hat{d}^{\lambda,\hat{q}}_{S}$ proves to be not merely formal, for
$\hat{d}:=\hat{d}^{1,\hat{q}}_{S}$. In that case, it follows that,
if $\Lambda$ is $q$-constrained Willmore, then $S$ is
$(\hat{q},\hat{d})$-constrained harmonic and, therefore, in the case
$1\in\mathrm{dom} (r)$, $S^{*}:=r(1)^{-1}S$ is $q^{*}$-constrained
harmonic, for
$$q^{*}:=\mathrm{Ad}_{r(1)^{-1}}\hat{q}.$$
The transformation of $S$ into $S^{*}$, preserving constrained
harmonicity, leads, furthermore, to a transformation of $\Lambda$
into a new constrained Willmore surface, provided
that\begin{equation}\label{eq:detr0infty}
\mathrm{det}\,r(0)\vert_{S}=\mathrm{det}\,r(\infty)\vert_{S}.
\end{equation}
Set
$$(\Lambda^{*})^{1,0}:=r(1)^{-1}r(\infty)\Lambda^{1,0},\,\,\,\,(\Lambda^{*})^{0,1}:=r(1)^{-1}r(0)\Lambda^{0,1}$$
and $$\Lambda^{*}:=(\Lambda^{*})^{1,0}\cap (\Lambda^{*})^{0,1}.$$
Condition \eqref{eq:detr0infty} establishes $\Lambda^{*}$ as a line
bundle (the argument is based on the two families of lines on the
Klein quadric). The isotropy of $\Lambda^{1,0}$ and $\Lambda^{0,1}$
ensures that of $\Lambda^{*}$. It is not clear, though, that
$\Lambda^{*}$ is a real bundle. If $\Lambda^{*}$ is a \textit{real
surface}, one proves that $S^{*}$ is the central sphere congruence
of $\Lambda^{*}$ and that the bundles $(\Lambda^{*})^{1,0}$ and
$(\Lambda^{*})^{0,1}$ defined above are not merely formal. The fact
that $q$ is a multiplier to $\Lambda$ establishes, furthermore,
$q^{1,0}\in\Omega^{1,0}(\wedge^{2}\Lambda^{0,1})$ and, therefore,
$(q^{*})^{1,0}\in\Omega^{1,0}(\wedge^{2}(\Lambda^{*})^{0,1})\subset\Omega^{1,0}(\Lambda^{*}\wedge(\Lambda^{*})^{(1)})$.
We conclude that, if, furthermore, $q^{*}$ is real, then
$\Lambda^{*}$ is a $q^{*}$-constrained Willmore surface.

We then construct rational maps $r(\lambda)$ satisfying the
hypothesis of the dressing action, together with reality preserving
conditions. As the philosophy underlying the work of C.-L. Terng and
K. Uhlenbeck \cite{uhlenbeck} suggests, we consider linear
fractional transformations. We define two different types of such
transformations, \textit{type p} and \textit{type q}, each one of
them satisfying the hypothesis of the dressing action with the
exception of condition \eqref{eq:detr0infty}. Iterating the
procedure, in a $2$-step process composing the two different types
of transformations, will produce a desired $r(\lambda)$. A
\textit{Bianchi permutability} of type $p$ and type $q$
transformations of constrained harmonic bundles is established. For
special choices of parameters, the reality of $\Lambda$ as a bundle
proves to establish that of $\Lambda^{*}$, whilst the reality of $q$
establishes that of $q^{*}$. For such a choice of parameters,
$\Lambda^{*}$ is said to be a \textit{B\"{a}cklund transform} of
$\Lambda$, provided that it immerses. For future reference, it is
useful to know that B\"{a}cklund transformation parameters are pairs
$\alpha,L^{\alpha}$ with, in particular, $\alpha\in\C$ and
$L^{\alpha}\subset\underline{\C}^{n+2}$, and that parameters
$\alpha,L^{\alpha}$ and $-\alpha,\rho L^{\alpha}$ give rise to the
same transform. Note that both B\"{a}cklund transformation and
spectral deformation corresponding to the zero multiplier preserve
the class of Willmore surfaces.

In Chapter \ref{transformsofCW}, we define, more generally, a
spectral deformation and a B\"{a}cklund transformation of
constrained harmonic bundles and of complexified constrained
Willmore surfaces. We complete the chapter by establishing a
permutability between spectral deformation and B\"{a}cklund
transformation.

In Chapter \ref{chaptercq}, we introduce the concept of
\textit{conserved quantity} of a constrained Willmore surface in a
space-form, an idea by F. Burstall and D. Calderbank. A conserved
quantity of $\Lambda$ consists of a Laurent polynomial
$$p(\lambda)=\lambda^{-1}v+v_{0}+\lambda\overline{v}$$ with $v_{0}\in\Gamma(S)$ real, $v\in\Gamma(S^{\perp})$
and $p(1)=v_{0}+v+\overline{v}\in\Gamma(\underline{\R}^{n+1,1})$
non-zero, which is $d^{\lambda,q}_{S}$-parallel, for all
$\lambda\in\C\backslash\{0\}$ and some multiplier $q$ to $\Lambda$.
Constrained Willmore surfaces in space-forms admitting a conserved
quantity form a subclass of constrained Willmore surfaces preserved
by both spectral deformation and B\"{a}cklund transformation, for
special choices of parameters.

The existence of a conserved quantity $p(\lambda)$ of $\Lambda$
establishes, in particular, the constancy of $p(1)$. In the
particular case of $n=3$, we verify that $\Lambda$ has constant mean
curvature in the space-form $S_{p(1)}$, that is, the surface defined
by $\Lambda$ in the space-form $S_{p(1)}$ has constant mean
curvature. In fact, in codimension $1$, the class of constrained
Willmore surfaces in space-forms admitting a conserved quantity
consists of the class of constant mean curvature surfaces in
space-forms. Another example of constrained Willmore surfaces
admitting a conserved quantity is that of surfaces with holomorphic
mean curvature vector in some space-form in codimension $2$. In
codimension $2$, the complexification of $S^{\perp}$ admits a unique
decomposition into the direct sum of two null complex lines, complex
conjugate of each other: given $v\in\Gamma(S^{\perp})$ null,
$S^{\perp}=\langle v\rangle\oplus \langle\overline{v}\rangle$. Such
a $v$ defines an almost-complex structure
$$J_{v}:= I\left\{
\begin{array}{ll} i & \mbox{$\mathrm{on}\,\langle v\rangle$}\\ -i &
\mbox{$\mathrm{on}\,\langle\overline{v}\rangle$}\end{array}\right.$$
on $S^{\perp}$. F. Burstall and D. Calderbank
\cite{burstall+calderbank} proved that a codimension $2$ surface in
a space-form, with holomorphic mean curvature vector with respect to
the complex structure induced by $\nabla^{S^{\perp}}$, is
constrained Willmore. We prove it in our setting, proving,
furthermore, that, in $4$-dimensional space-form, the constrained
Willmore surfaces admitting a conserved quantity
$p(\lambda)=\lambda^{-1}v+v_{0}+\lambda\overline{v}$ with $v$ null
are the surfaces with holomorphic mean curvature vector in the
space-form $S_{p(1)}$, with respect to the complex structure on
$(S^{\perp},J_{v},\nabla^{S^{\perp}})$ determined by
Koszul-Malgrange Theorem.

Chapter 8 is dedicated to relating the class of constrained Willmore
surfaces to the isothermic condition. The class of isothermic
surfaces is a M\"{o}bius invariant class of surfaces. A manifestly
conformally invariant formulation of the isothermic condition, by F.
Burstall and U. Pinkall \cite{FPink}, characterizes isothermic
surfaces by the existence of a non-zero real closed $1$-form $\eta$
with values in a certain subbundle of the skew-symmetric
endomorphisms of $\underline{\R}^{n+1,1}$. We establish the set of
multipliers to an isothermic $q$-constrained Willmore surface
$(\Lambda,\eta)$ as the $1$-dimensional affine space $q+\langle
*\eta\rangle_{\R}$. The constrained Willmore spectral deformation is
known to preserve the isothermic condition, cf. \cite{SD}. We derive
it in our setting. As for B\"{a}cklund transformation of constrained
Willmore surfaces, we believe it does not necessarily preserve the
isothermic condition. This shall be the subject of further work.

Following the work of F. Burstall, D. Calderbank and U. Pinkall
\cite{burstall+calderbank}, \cite{FPink}, we characterize isothermic
surfaces by the flatness of a certain $\R$-family of metric
connections on $\underline{\R}^{n+1,1}$ and define, in terms of this
family of connections, both the isothermic spectral deformation,
discovered in the classical setting by Calapso and, independently,
by Bianchi; and the isothermic Darboux transformation.

We dedicate a section to the special class of constant mean
curvature (CMC) surfaces in $3$-dimensional space-forms. CMC
surfaces in $3$-dimensional space-forms are examples of isothermic
constrained Willmore surfaces, as proven by J. Richter
\cite{richter}, with constrained Willmore B\"{a}cklund
transformations; both constrained Willmore and isothermic spectral
deformations; as well as a spectral deformation of their own and, in
the Euclidean case,
 isothermic Darboux transformations and Bianchi-B\"{a}cklund
transformations. The isothermic spectral deformation is known to
preserve the constancy of the mean curvature of a surface in some
space-form, cf. \cite{SD}. Characterized as the class of constrained
Willmore surfaces in $3$-dimensional space-forms admitting a
conserved quantity, the class of CMC surfaces in $3$-space is known
to be preserved by both constrained Willmore spectral deformation
and B\"{a}cklund transformation, for special choices of parameters.
We verify that both the space-form and the mean curvature are
preserved by constrained Willmore B\"{a}cklund transformation and
investigate how these change under constrained Willmore and
isothermic spectral deformation. We present the classical CMC
spectral deformation by means of the action of a loop of flat metric
connections on the class of CMC surfaces in $3$-space (preserving
the space-form and the mean curvature) and observe that the
classical CMC spectral deformation can be obtained as composition of
isothermic and constrained Willmore spectral deformation. These
spectral deformations of CMC surfaces in $3$-space are, in this way,
all closely related and, therefore, closely related to constrained
Willmore B\"{a}cklund transformation. S. Kobayashi and J.-I.
Inoguchi \cite{kobayashi} proved that isothermic Darboux
transformation of CMC surfaces in Euclidean $3$-space is equivalent
to Bianchi-B\"{a}cklund transformation. We believe isothermic
Darboux transformation of a CMC surface in Euclidean $3$-space can
be obtained as a particular case of constrained Willmore
B\"{a}cklund transformation. This shall be the subject of further
work.

We present a $1$-form $\eta_{\infty}$, derived by F. Burstall and D.
Calderbank from a surface $\Lambda$ with constant mean curvature
$H_{\infty}$ in a space-form $S_{v_{\infty}}$, which establishes
$(\Lambda,\eta_{\infty})$ as an isothermic surface and for which
scaling by $H_{\infty}$ provides a multiplier,
$$q_{\infty}:=H_{\infty}\eta_{\infty},$$to $\Lambda$; and,
for each $t\in\R$ and
$$q_{\infty}^{t}:=q_{\infty}+t*\eta_{\infty},$$ we establish a
$q_{\infty}^{t}$-conserved quantity to $\Lambda$.

Lastly, we dedicate Chapter \ref{surfacesinS4} to the special case
of surfaces in $4$-space. Our approach is quaternionic, based on the
model of the conformal $4$-sphere on the quaternionic projective
space, and follows the work of F. Burstall et al.
\cite{quaternionsbook}. We consider the natural identification of
$\mathbb{H}$ with $\R^{4}$ and then the natural identification of
$\mathbb{H}^{2}$ with $\langle 1,i\rangle^{4}=\C^{4}$. We provide
$\wedge^{2}\C^{4}$ with the real structure $\wedge^{2}j$ and define
a metric on $\wedge^{2}\C^{4}$ by $(v_{1}\wedge v_{2},v_{3}\wedge
v_{4}):=-\mathrm{det}(v_{1}, v_{2},v_{3},v_{4})$, for $v_{1},
v_{2},v_{3}, v_{4}\in\C^{4}$, with $\mathrm{det}(v_{1}, v_{2},v_{3},
v_{4})$ denoting the determinant of the matrix whose columns are the
components of $v_{1}, v_{2},v_{3}$ and $v_{4}$, respectively, on the
canonical basis of $\C^{4}$. This metric induces a metric with
signature $(5,1)$ on the space of real vectors of
$\wedge^{2}\C^{4}$, $$\mathrm{Fix}(\wedge^{2}j)=\R^{5,1}.$$ Via the
Pl\"{u}cker embedding, we identify a $j$-stable $2$-plane $L$ in
$\C^{4}$ with the real null line $\wedge^{2}L$ in
$(\mathrm{Fix}(\wedge^{2}j))^{\C}$, presenting, in this way, the
quaternionic projective space $\mathbb{H}P^{1}$ as a model for the
conformal $4$-sphere,
$$\mathbb{H}P^{1}\cong S^{4}.$$Surfaces in $S^{4}$ are described
in this model as the immersed bundles
$$L\cong\wedge^{2}L:M\rightarrow S^{4}$$ of $j$-stable $2$-planes in
$\C^{4}$. As we are in codimension $2$, the complexification of
$S^{\perp}$ admits a unique decomposition
$S^{\perp}=S^{\perp}_{+}\oplus S^{\perp}_{-}$ into the direct sum of
two null complex lines, complex conjugate of each other. Via the
Pl\"{u}cker embedding, we identify $S_{+}^{\perp}$ with some bundle
$S_{+}$ of $2$-planes in $\C^{4}$, and write then
$$S=S_{+}\wedge jS_{+}.$$ We define a $j$-commuting complex
structure on $\underline{\C}^{4}$, which we still denote by $S$, by
the condition of admitting $S_{+}$ as the eigenspace associated to
the eigenvalue $i$ (and, therefore, $jS_{+}$ as the eigenspace
associated to $-i$), together with a certain condition on the sign.

Following the characterization of the harmonicity of $S$ presented
in \cite{quaternionsbook}, we establish, more generally, a
characterization of constrained Willmore surfaces in $S^{4}$ in
terms of the closeness of a certain form, as follows. Under the
standard identification $sl(\C^{4})\cong o(\wedge^{2}\C^{4})$,
$1$-forms with values in $\Lambda\wedge\Lambda^{(1)}$ correspond to
$S$-commuting $1$-forms with values in
$\mathrm{End}_{j}(\underline{\C}^{4}/L,L)$. For such a form $q$,
condition $d^{\mathcal{D}}q=0$ establishes $Sq=*q$. A surface $L$ in
$S^{4}$ is a $q$-constrained Willmore surface, for some $1$-form $q$
with values in $\mathrm{End}_{j}(\underline{\C}^{4}/L,L)$ such that
$Sq=*\,q=qS$, if and only if $$d*(Q+q)=0,$$ for the Hopf field
$Q\in\Omega^{1}(\mathrm{End}_{j}(\underline{\C}^{4}))$. The
closeness of the $1$-form $*(Q+q)$ ensures the existence of
$G\in\Gamma(\mathrm{End}_{j}(\underline{\C}^{4}))$ with
$dG=2*(Q+q)$, as well as the integrability of the Riccati equation
$$dT=\rho T(dG)T-dF+4\rho qT,$$ for each $\rho\in\R\backslash\{0\}$,
fixing such a $G$ and setting $F:=G-S$. For a local solution
$T\in\Gamma(Gl_{j}(\underline{\C}^{4}))$ of the $\rho$-Ricatti
equation, we define the \textit{Darboux transform} of $L$ of
parameters $\rho,T$ by setting $$\hat{L}:=T^{-1}L,$$ extending, in
this way, the Darboux transformation of Willmore surfaces in $S^{4}$
presented in \cite{quaternionsbook} to a transformation of
constrained Willmore surfaces in $4$-space.

We apply, yet again, the dressing action presented in Chapter
\ref{transformsofCW} to define another transformation of constrained
Willmore surfaces in $4$-space, the \textit{untwisted B\"{a}cklund
transformation}, referring then to the original one as the
\textit{twisted B\"{a}cklund transformation}. We verify that, when
both are defined, twisted and untwisted B\"{a}cklund transformations
coincide. We establish a correspondence between Darboux
transformation parameters $\rho,T$ with $\rho > 1$ and pairs
$\alpha,L^{\alpha};\,-\alpha,\rho_{V} L^{\alpha}$ of untwisted
B\"{a}cklund transformation parameters with $\alpha^{2}$ real, and
show that the corresponding transformations coincide. Darboux
transformation of constrained Willmore surfaces with respect to
parameters $\rho,T$ with $\rho \leq 1$ is trivial. Non-trivial
Darboux transformation of constrained Willmore surfaces in $4$-space
is, in this way, established as a particular case of constrained
Wilmore B\"{a}cklund transformation.

The main new or, at least, unpublished\footnote{With the trivial
exception of the PhD Thesis on which this work is based.} (to my
knowledge) notions and results presented in this work can be listed
as follows:

\begin{itemize}
\item Section \ref{sec:cchb} (we define \textit{constrained}
\textit{harmonicity} of a bundle, which will apply to the central
sphere congruence to provide a characterization of constrained
Willmore surfaces in space-forms);

\item Sections \ref{complexifiedsurfacessect} and \ref{CCWSgd}
(we define \textit{complexified surface}, in generalization of
surface in a space-form, defining, thereafter, \textit{complexified
constrained Willmore surface});

\item Section \ref{subsec:cwillmfam} (we define a
$\C\backslash\{0\}$-spectral deformation of complexified constrained
Willmore surfaces by the action of a family of flat metric
connections; for unit parameter, this deformation preserves reality
conditions and, when restricted to \textit{real} surfaces, it
coincides, up to reparametrization, with the one presented in
\cite{SD});

\item Sections \ref{sec:dress} and \ref{sec:tranfscentr} (we
define a \textit{B\"{a}cklund transformation} of complexified
constrained Willmore surfaces, by applying a dressing action; for
special choices of parameters, this transformation preserves reality
conditions);

\item Section \ref{BTvsSp}  (we establish a permutability between
B\"{a}cklund transformation and spectral deformation of complexified
constrained Willmore surfaces);

\item Proposition \ref{quaddiffscoincide} (we prove that the
quadratic differential is preserved under the corresponding
B\"{a}cklund transformation);

\item Section \ref{cqofCWs} (we introduce the concept of
\textit{conserved quantity} of a constrained Willmore surface in a
space-form, an idea by F. Burstall and D. Calderbank);

\item Theorems \ref{specCWwcq} and \ref{BTsofCMCs} (we prove that
the class of constrained Willmore surfaces in $3$-dimensional
space-forms admitting a conserved quantity is preserved by both
spectral deformation and B\"{a}cklund transformation, for special
choices of parameters);

\item Section \ref{HMCsurfs} (we prove that, in codimension $2$,
surfaces with holomorphic mean curvature vector in some space-form
(already known to be constrained Willmore, cf.
\cite{burstall+calderbank}) are examples of constrained Willmore
surfaces admitting a conserved quantity);

\item Proposition \ref{CMChasCQ} and Theorem \ref{CMCsareCWswithCQ} (we establish the class of
CMC surfaces in $3$-dimensional space-forms as the class of
constrained Willmore surfaces in $3$-space admitting a conserved
quantity);

\item Sections \ref{CWspdeform} and \ref{BTCMC} (we verify that both the space-form and the mean
curvature are preserved by constrained Willmore B\"{a}cklund
transformation of CMC surfaces in $3$-dimensional space-forms and
investigate how these change under constrained Willmore spectral
deformation; we verify that the classical CMC spectral deformation
of a CMC surface in $3$-space can be obtained as composition of
isothermic and constrained Willmore spectral deformation);

\item Section \ref{nonisoCW} (we characterize isothermic
constrained Willmore surfaces in space-forms by the non-uniqueness
of multiplier and establish the set of multipliers to an isothermic
$q$-constrained Willmore surface $(\Lambda,\eta)$ as the affine
space $q+\langle
*\eta\rangle_{\R}$);

\item Theorems \ref{CWinS4vkutywsfgdredhgvcertyq2932h} and \ref{CWinS4charact} (we provide
two (equivalent) characterizations of constrained Willmore surfaces
in $4$-space in the quaternionic setting);

\item Section \ref{DTsCWS4} (we extend the Darboux transformation
of Willmore surfaces in $S^{4}$ presented in \cite{quaternionsbook}
to a transformation of constrained Willmore surfaces in $4$-space);

\item Section \ref{BTvsDT} (we prove that Darboux transformation
of parameters $\rho,T$ with $\rho > 1$ is equivalent to B\"{a}cklund
transformation of parameters $\alpha,L^{\alpha}$ with $\alpha^{2}$
real; Darboux transformation of parameters $\rho,T$ with $\rho \leq
1$ is trivial).
\end{itemize}

Our theory is local and, throughout the text, with no need for
further reference, restriction to a suitable non-empty open set
shall be underlying.

\textit{Foundations of Differential Geometry}, Vol.s 1,2, by S.
Kobayashi and K. Nomizu \cite{kobayashi2}, \textit{Riemannian
Geometry}, by T. Willmore \cite{willmore}, and \textit{Selected
topics in Harmonic maps}, by J. Eells and L. Lemaire
\cite{eells+lemaire}, are good references to the basic background.

\mainmatter

\chapter{A bundle approach to conformal surfaces in space-forms}

\markboth{\tiny{A. C. QUINTINO}}{\tiny{CONSTRAINED WILLMORE
SURFACES}}

\thispagestyle{plain}

\pagenumbering{arabic}

Our study is a study of geometrical aspects that are invariant under
M\"{o}bius transformations.\footnote{With the exception of the study
of constant mean curvature surfaces, in Sections
\ref{sec:casecodim1} and \ref{sec:CMC} below.} We present a
M\"{o}bius description of space-forms in the projectivized
light-cone, following \cite{IS}. Such description is based on the
model of the conformal $n$-sphere on the projective space
$\mathbb{P}(\mathcal{L})$ of the light-cone $\mathcal{L}$ of
$\R^{n+1,1}$, due to Darboux \cite{darbouxsphere}, which, in
particular, yields a conformal description of Euclidean $n$-spaces
and hyperbolic $n$-spaces as submanifolds of
$\mathbb{P}(\mathcal{L})$. With this, we approach a surface
conformally immersed in $n$-space as a null line subbundle $\Lambda$
of $\underline{\R}^{n+1,1}=M\times\R^{n+1,1}$ defining an immersion
$\Lambda:M\rightarrow\mathbb{P}(\mathcal{L})$ of an oriented surface
$M$, which we provide with the conformal structure induced by
$\Lambda$, into the projectivized light-cone. The realization of all
space-forms as submanifolds of the projectivized light-cone
considered in this text arises from the realization, cf. \cite{IS},
of all $n$-dimensional space-forms as conic sections
$S_{v_{\infty}}:=\{v\in\mathcal{L}:(v,v_{\infty})=-1\}$ of the
light-cone, with $v_{\infty}\in\mathbb{R}^{n+1,1}$ non-zero.
$S_{v_{\infty}}$ inherits a positive definite metric of constant
sectional curvature $-(v_{\infty},v_{\infty})$ from $\R^{n+1,1}$ and
is either a copy of a sphere, a copy of Euclidean space or two
copies of hyperbolic space, according to the sign of
$(v_{\infty},v_{\infty})$. For each $v_{\infty}$, the canonical
projection $\pi:\mathcal{L}\rightarrow\mathbb{P}(\mathcal{L})$
defines a diffeomorphism
$\pi\vert_{S_{v_{\infty}}}:S_{v_{\infty}}\rightarrow \mathbb{P}
(\mathcal {L})\backslash\mathbb{P} (\mathcal {L}\cap \langle
v_{\infty}\rangle ^{\perp})$. We provide $\mathbb{P}(\mathcal{L})$
with the conformal structure of the (positive definite) metric
induced by $\pi\vert_{S_{v_{\infty}}}$, fixing $v_{\infty}$
time-like, independently of the choice of $v_{\infty}$, and, in this
way, identify $\mathbb{P}(\mathcal{L})$ with the conformal
$n$-sphere and make each diffeomorphism $\pi\vert_{S_{v_{\infty}}}$
- for a general $v_{\infty}$, not necessarily time-like - into a
conformal diffeomorphism.\newline

By \textit{metric} on a $\mathbb{K}$-vector bundle $P$, with
$\mathbb{K}\in\{\R,\C\}$,  we mean a non-degenerate section of
$S^{2}(P,\mathbb{K})$.  With no need for further reference, given a
vector bundle $P$, provided with a metric, the metric considered on
the complexification $P^{\C}$ of $P$ shall be the complex bilinear
extension of the metric on $P$. In particular,
$$\overline{(\alpha_{1}+i\beta_{1},\alpha_{2}+i\beta_{2})}=(\,\overline{\alpha_{1}+i\beta_{1}},\overline{\alpha_{2}+i\beta_{2}}\,),$$
for $\alpha_{i},\beta_{i}\in\Gamma(P)$, for $i=1,2$. More generally,
throughout this text, when appropriated and unless indication to the
contrary is provided, we will move from real tensors to complex
tensors by complex multilinear extension, preserving, therefore,
reality conditions and, in general, notation. With no need for
further reference, given $V$ and $W$ vector bundles provided with a
metric, we shall consider $\mathrm{Hom}(V,W)$ provided with the
canonical metric induced by $V$ and $W$, the metric given by
$$(\alpha,\beta):=\mathrm{tr}\,(\beta^{t}\alpha),$$ for $\alpha,
\beta\in\Gamma(\mathrm{Hom}(V,W))$, with $\beta^{t}$ denoting the
transpose of $\beta$.

As usual, by metric on a manifold $M$ we mean a metric on the
tangent bundle $TM$,  as a  $\R$-bundle. Two positive definite
metrics $g$ and $g'$ on a manifold $M$ are said to be
\textit{conformally equivalent} if $g'=e^{u}g$, for some $u\in
C^{\infty}(M;\R)$. A class of conformally equivalent Riemannian
metrics on $M$ is said to be a \textit{conformal structure} on $M$.
When provided with a conformal structure, $M$ is said to be a
\textit{conformal manifold}. A mapping $\phi$ of $M$ into a manifold
$N$ provided with a metric $(,)_{N}$ defines a section $g_{\phi}$ of
$S^{2}(TM,\R)$, given by
$$g_{\phi}(X,Y):=(d_{X}\phi , d_{Y}\phi )_{N},$$ for $X,Y\in\Gamma
(TM)$, which we refer to as the \textit{metric induced in $M$ by
$\phi$}. If the metric $(,)_{N}$ is positive definite and $\phi$ is
an immersion, then the metric $g_{\phi}$ is positive definite. In
the case $M$ is provided with a Riemannian metric $g$,
$\phi:(M,g)\rightarrow (N,(,)_{N})$ is said to be \textit{conformal}
if $g_{\phi}$ is a positive definite metric conformally equivalent
to $g$.

Our study is a study of geometrical aspects that are invariant under
conformal diffeomorphisms or \textit{M\"{o}bius transformations}. It
is, in particular, a study of angle-preserving transformations,
dedicated to submanifolds of an ambient space equipped with a
conformal class of metrics but not carrying a distinguished metric.
Our point of view is that of \textit{M\"{o}bius geometry}, where
there is an angle measurement, but, in contrast to Euclidean
geometry, there is no measurement of distances. This lack of length
measurement has interesting consequences. For example, from the
point of view of M\"{o}bius geometry, $S^{n}\backslash\{x\}$ and
$\R^{n}$ are no longer distinguished, as the stereographic
projection of pole $x\in S^{1}$, of $S^{n}\backslash\{x\}$ onto
$\R^{n}$, is a conformal diffeomorphism. We refer to invariance
under M\"{o}bius transformations as \textit{M\"{o}bius invariance}
or \textit{conformal invariance}.

For $m\in\mathbb{N}$, let $\R^{m,1}$ be the $(m+1)$-dimensional
Minkowski space with signature $(m,1)$, i.e., a real
$(m+1)$-dimensional vector space equipped with a metric $(\,,\,)$
with respect to which there exists an orthogonal basis
$e_{1},...,e_{m+1}$ with
$$(e_{i},e_{i})
=\left\{
\begin{array}{ll} 1 & \mbox{$i<m+1$}\\ -1 &
\mbox{$i=m+1$}\end{array}\right..$$
 As usual, we refer to a vector $v\in\R^{m,1}$
such that $(v,v)=0$ (respectively, $(v,v)<0$, $(v,v)>0$) as a
light-like (respectively, time-like, space-like) vector. Let
$\mathcal{L}$ be the light-cone in $\R^{m,1}$,
$$\mathcal{L}=\{v\in\R^{m,1}\backslash\{0\}:(v,v)=0\},$$an
$m$-dimensional submanifold of $\R^{m,1}$, and, in the case $m>1$,
let $\mathcal{L}^{+}$ and $\mathcal{L}^{-}$ be the connected
components of $\mathcal{L}$.\footnote{Non-collinear elements
$v_{0},v_{\infty}\in\mathcal{L}$ are in the same component if and
only if $(v_{0},v_{\infty})<0$.} Let $\mathbb{P}(\mathcal{L})$ be
the projectivized light-cone,
$$\mathbb{P}(\mathcal{L})=\{\langle v\rangle :v\in\mathcal{L}\},$$
an $(m-1)$-dimensional submanifold of $\mathbb{P}(\R^{m,1})$.

Recall the sectional curvature $K(\wp)$ of a $2$-plane $\wp$ in the
tangent space $T_{x}M$ to a Riemannian manifold $M$ at a point $x\in
M$,
$$K(\wp)=-(R(X_{1},X_{2})X_{1},X_{2}),$$for $R$
the curvature tensor of $M$ (provided with the Levi-Civita
connection), independently of the choice of an orthonormal basis
$X_{1},X_{2}$ of $T_{x}M$. It is opportune to establish some (usual)
notation: given $X,Y,Z,W$ in $\Gamma(TM)$,
$$R(X,Y,Z,W):=-(R(X,Y)Z,W).$$ If $K(\wp)$ is constant for all planes
$\wp$ in $T_{x}M$ and for all points $x\in M$, then $M$ is said to
be a space of constant curvature. If $K(\wp)$ is constant for all
planes $\wp$ in $T_{x}M$ but possibly depends on $x\in M$, we say
that $M$ has constant sectional curvature $K=K(x)$, with $x\in M$.
That is always the case when $M$ is $2$-dimensional, in which case
$K$ is famously said to be the Gaussian curvature of the surface
$M$.

By \textit{space-form} we mean a connected and simply connected
complete Riemannian manifold of constant sectional curvature. For
simplicity, we may  use $n$-\textit{space} to refer to
$n$-dimensional space-form. Two space-forms with the same curvature
are isometric to each other. Fix $n\in\mathbb{N}$. Our model of flat
$n$-space is the Euclidean space $\R^{n}$. Given $r\in\R^{+}$, our
model of $n$-space of curvature $1/r^{2}$ is the $n$-sphere
$$S^{n}(r):=\{x\in \mathbb{R}^{n+1}:(x,x)=r^{2}\},$$whereas that of
$n$-space of curvature $-1/r^{2}$ is the hyperbolic $n$-space
consisting of either of the two connected components of
$$H^{n}(r):=\{x\in \mathbb{R}^{n,1}:(x,x)=-r^{2}\}.$$

\section{Space-forms in the
conformal projectivized light-cone}\label{subsec:hyper}

\markboth{\tiny{A. C. QUINTINO}}{\tiny{CONSTRAINED WILLMORE
SURFACES}}

Following \cite{IS}, we present a M\"{o}bius description of spheres,
Euclidean spaces and hyperbolic spaces in the projectivized
light-cone. We start by realizing all $n$-dimensional space-forms as
connected components of conic sections of $\mathcal{L}\subset
\R^{n+1,1}$.\newline

Let $\mathcal{L}$ be the light-cone in $\R^{n+1,1}$. Fix
$v_{\infty}\in\mathbb{R}^{n+1,1}$ non-zero. Set
$$S_{v_{\infty}}:=\{v\in\mathcal{L}:(v,v_{\infty})=-1\},$$the conic
section of the light-cone given by the intersection of $\mathcal{L}$
with the hyperplane $\{v\in\R^{n+1,1}:(v,v_{\infty})=-1\}$ of
$\R^{n+1,1}$, which is an $n$-dimensional submanifold of
$\R^{n+1,1}$. Given $v\in\mathcal{L}$,
$$T_{v}\mathcal{L}=\langle v\rangle^{\perp}$$ and,
for $v\in S_{v_{\infty}}$,
$$T_{v}S_{v_{\infty}}=\langle v,v_{\infty}\rangle^{\perp}.$$
The fact that $(v,v_{\infty})\neq 0$ establishes the non-degeneracy
of the subspace $\langle
 v,v_{\infty}\rangle$ of $\R^{n+1,1}$, or, equivalently, a decomposition
\begin{equation}\label{eq:decompwithTvSvinfty}
\R^{n+1,1}=\langle v,v_{\infty}\rangle\oplus T_{v}S_{v_{\infty}}.
\end{equation}
The nullity of $v$ establishes then $\langle
 v,v_{\infty}\rangle$ as a $2$-dimensional space with a metric with signature $(1,1)$, establishing,
therefore, $T_{v}S_{v_{\infty}}$ as isometric to $\R^{n}$.
$S_{v_{\infty}}$ inherits a positive definite metric from
$\R^{n+1,1}$.

\begin{prop}
$S_{v_{\infty}}$ is either a copy of an $n$-sphere, a copy of
Euclidean $n$-space or two copies of hyperbolic $n$-space, according
to the sign of $(v_{\infty},v_{\infty})$.
\end{prop}
In the proof of the proposition, we will show, in particular, that
$S_{v_{\infty}}$ has constant sectional curvature
$-(v_{\infty},v_{\infty})$, and that, if $v_{\infty}$ is space-like,
then $S_{v_{\infty}}\cap \mathcal{L}^{+}$ is a copy of hyperbolic
$n$-space, as well as $S_{v_{\infty}}\cap \mathcal{L}^{-}$.
\begin{proof}
If $v_{\infty}$ is light-like, then, choosing $v_{0}\in
S_{v_{\infty}}$, the map $v\mapsto v-v_{0}+(v,v_{0})v_{\infty}$
defines an isometry
$$S_{v_{\infty}}\rightarrow \langle
v_{0},v_{\infty}\rangle^{\perp}\cong\R^{n}$$whose inverse is given
by $x\rightarrow x+v_{0}+\frac {1}{2}(x,x)v_{\infty}$.

Now contemplate the case $v_{\infty}$ is not light-like. In that
case, $\langle v_{\infty}\rangle$ is non-degenerate, so we have a
decomposition $\mathbb{R}^{n+1,1}=\langle
v_{\infty}\rangle\oplus\langle v_{\infty}\rangle ^{\perp}$. Set
$r':=(v_{\infty},v_{\infty})^{-1}$. For $v\in S_{v_{\infty}}$, write
$v=-r'v_{\infty}+v^{\perp}$ with $v^{\perp}\perp v_{\infty}$ and
note that $0=(v,v)=r'+(v^{\perp},v^{\perp})$. In the particular case
$v_{\infty}$ is time-like,  the projection $v\mapsto v^{\perp}$
defines a diffeomorphism
$$S_{v_{\infty}}\rightarrow S^{n}(r)\subset \langle v_{\infty}\rangle ^{\perp}\cong\R^{n+1}$$
onto the sphere of radius $r:=\sqrt{-r'}$ in $\langle
v_{\infty}\rangle ^{\perp}\cong\R^{n+1},$ which we easily check to
be an isometry. In the case $v_{\infty}$ is space-like, the
projection $v\mapsto v^{\perp}$  defines an isometry
$$S_{v_{\infty}}\rightarrow H^{n}(r)\subset\langle
v_{\infty}\rangle ^{\perp}\cong\R^{n,1},$$for $r:=\sqrt{r'}$.
\end{proof}

Now contemplate the canonical projection
$\pi:\mathcal{L}\rightarrow\mathbb{P}(\mathcal{L})$, which we may,
alternatively, denote by $\pi_{\mathcal{L}}$. Note that, given
$v\in\mathcal{L}$,
\begin{equation}\label{eq:kerdpi}
\mathrm{Ker}\,d\pi_{v}=\langle v \rangle,
\end{equation}
so we have an isomorphism $d\pi_{v}:\langle v\rangle
^{\perp}/\langle v\rangle\cong d\pi_{v}(T_{v}\mathcal{L})$, given by
$u+\langle v\rangle\mapsto d\pi_{v}(u),$ for $u\in\langle v\rangle
^{\perp}$. As $d\pi_{v}(T_{v}\mathcal{L})$ is an $n$-dimensional
subspace of $T_{\langle v\rangle}\mathbb{P}(\mathcal{L})$, we
conclude that $T_{\langle
v\rangle}\mathbb{P}(\mathcal{L})=d\pi_{v}(T_{v}\mathcal{L})$,
\begin{equation}\label{eq:Tangenttoprojectlightcne}
d\pi_{v}:\langle v\rangle ^{\perp}/\langle v\rangle\cong T_{\langle
v\rangle}\mathbb{P}(\mathcal{L})
\end{equation}
and, therefore, that $\pi$ is a submersion.

The map $\pi$ defines a diffeomorphism
\begin{equation}\label{eq:pirestSvinf}
\pi\vert_{S_{v_{\infty}}}:S_{v_{\infty}}\rightarrow \mathbb{P}
(\mathcal {L})\backslash\mathbb{P} (\mathcal {L}\cap \langle
v_{\infty}\rangle ^{\perp})
\end{equation}
whose inverse is given by $$\langle v\rangle \mapsto
S_{v_{\infty}}\cap \langle v\rangle = \frac{-1}{(v,v_{\infty})}\, v
.$$
\begin{Lemma}\label{mathcalLcapvinfperp}
If $v_{\infty}$ is time-like, then $\mathcal{L}\cap \langle
v_{\infty}\rangle^{\perp}=\emptyset$. If $v_{\infty}$ is light-like,
then $\mathcal{L}\cap \langle v_{\infty}\rangle^{\perp}=\langle
v_{\infty}\rangle.$
\end{Lemma}

\begin{proof}
If $v_{\infty}$ is time-like, then $\langle
v_{\infty}\rangle^{\perp}$ is space-like and, therefore, contains no
light-like vectors.

On the other hand, recalling that there are no null subspaces of
$\R^{n+1,1}$ with dimension higher than $1$, we conclude that, if
$v_{\infty}$ is light-like, then $(v,v_{\infty})\neq 0$, for all $v$
in $\mathcal{L}\backslash\langle v_{\infty}\rangle$, completing the
proof.
\end{proof}

Hence:
\begin{prop}\label{Svinfprojectivized}
Once restricted to $S_{v_{\infty}}$, the canonical projection $\pi
:\mathcal {L}\rightarrow \mathbb{P} (\mathcal {L})$ defines a
diffeomorphism onto $\mathbb{P} (\mathcal {L})$, $\mathbb{P}
(\mathcal {L})\backslash\{\langle v_{\infty}\rangle\}$ or
$\mathbb{P} (\mathcal {L})\backslash\mathbb{P} (\mathcal {L}\cap
\langle v_{\infty}\rangle ^{\perp})$, according as $v_{\infty}$ is
time-like, light-like or space-like, respectively.
\end{prop}

This will lead us to a conformal description of spheres, Euclidean
spaces and hyperbolic spaces in the projectivized light-cone. Our
next step is to provide the projectivized light-cone with a
conformal structure.

We provide $\mathbb{P}(\mathcal{L})$ with the conformal structure
$\mathcal{C}_{\mathbb{P}(\mathcal{L})}$ defined by the metric
$g_{\sigma}$, fixing $\sigma$ a never-zero section of the
tautological bundle
$\mathbb{P}(\mathcal{L})_{\mathbb{P}(\mathcal{L})}=(l)_{l\in\mathbb{P}(\mathcal{L})}$.
This is well-defined and, furthermore, makes each diffeomorphism
\eqref{eq:pirestSvinf} into a conformal diffeomorphism. Indeed,
given $\sigma,\sigma':\mathbb{P}(\mathcal{L})\rightarrow \R^{n+1,1}$
never-zero sections of
$\mathbb{P}(\mathcal{L})_{\mathbb{P}(\mathcal{L})}$, we have
$\sigma'=f\sigma$, for some never-zero $f\in
C^{\infty}(\mathbb{P}(\mathcal{L})_{\mathbb{P}(\mathcal{L})},\R)$,
and then
$$g_{\sigma'}(X,Y)=(fd_{X}\sigma+(d_{X}f)\sigma,fd_{Y}\sigma+(d_{Y}f)\sigma)=f^{2}g_{\sigma}(X,Y),$$
as $(\sigma,\sigma)$ vanishes and, therefore, so does
$(\sigma,d\sigma)$. On the other hand, for non-zero
$v_{\infty}\in\R^{n+1,1}$, the diffeomorphism
$(\pi\vert_{S_{v_{\infty}}})^{-1}:\mathbb{P} (\mathcal
{L})\backslash\mathbb{P} (\mathcal {L}\cap \langle v_{\infty}\rangle
^{\perp})\rightarrow S_{v_{\infty}}$ is a never-zero section of
$\mathbb{P}(\mathcal{L})_{\mathbb{P}(\mathcal{L})}$ inducing in
$\mathbb{P}(\mathcal{L})$ a positive definite metric from the one in
$S_{v_{\infty}}$.

We refer to the unit sphere $S^{n}$ in $\mathbb{R}^{n+1}$, when
provided by the conformal class of the round metric, as the
conformal n-sphere. By providing $\mathbb{P}(\mathcal{L})$ with the
conformal structure $\mathcal{C}_{\mathbb{P}(\mathcal{L})}$, we make
the map $\pi\vert_{S_{v_{\infty}}}:S^{n}\cong
S_{v_{\infty}}\rightarrow \mathbb{P}(\mathcal{L})$, defined by
$$x\mapsto\langle v_{\infty}+x\rangle,$$ for $v_{\infty}$ unit
time-like, into a conformal diffeomorphism, identifying the
conformal $n$-sphere with the conformal projectivized light-cone,
$$S^{n}\cong\mathbb{P}(\mathcal{L}),$$in a model due to Darboux
\cite{darbouxsphere}. In this model of the conformal $n$-sphere, the
$k$-spheres $S^{k}$ in $S^{n}$ (intersections of $S^{n}$ with
$(k+1)$-dimensional affine spaces), for $0\leq k\leq n$, are
described as the submanifolds $\mathbb{P} (\mathcal {L}\cap V)$ with
$V$ non-degenerate $(k+2)$-dimensional subspace of
$\mathbb{R}^{n+1,1}$ with signature $(k+1,1)$ (see, for example,
\cite{susana}, Appendix A); hyperbolic $n$-spaces are described as
the submanifolds consisting of either of the two connected
components of
$$S^{n}\backslash S^{n-1}\cong\mathbb{P} (\mathcal
{L})\backslash\mathbb{P} (\mathcal {L}\cap \langle v_{\infty}\rangle
^{\perp})$$ with $v_{\infty}$ space-like; and Euclidean $n$-spaces
are described as those consisting of
$$S^{n}\backslash\{x\}\cong\mathbb{P} (\mathcal
{L})\backslash\{\langle v_{\infty}\rangle\}$$ with $x\in S^{n}$ and
$v_{\infty}$ light-like.

\section{Conformal surfaces in space-forms}\label{sec:imsurf}

\markboth{\tiny{A. C. QUINTINO}}{\tiny{CONSTRAINED WILLMORE
SURFACES}}

Throughout this text, let  $n\geq 3$, $\mathcal{L}$ be the
light-cone of $\R^{n+1,1}$, $M$ be an oriented surface and
$\underline{\mathbb{R}}^{n+1,1}$ denote the trivial bundle $M\times
\mathbb{R}^{n+1,1}$.

\subsection{Oriented conformal surfaces:
generalities}\label{Riemannsurfacegeneralities}

Suppose $M$ is provided with a conformal structure $\mathcal{C}$.

\begin{rem}\label{L-C}
The Levi-Civita connection is not a conformal invariant. In fact,
(see, for example, \cite{willmore}, \S$3.12$), given
$g,\,g':=e^{2u}g\in\mathcal{C}$, for some $u\in C^{\infty}(M,\R)$,
the Levi-Civita connections $\nabla$ and $\nabla '$, on $(M,g)$ and
$(M,g')$, respectively, are related by
\begin{equation}\label{eq:relLCs}
\nabla'_{X}Y=\nabla_{X}Y+(Xu)Y+(Yu)X-g(X,Y)(du)^{*},
\end{equation}
for all $X,Y\in\Gamma(TM)$, with $(du)^{*}$ denoting the
contravariant form of du with respect to $g$,
$(du)^{*}\in\Gamma(TM)$ defined by $g(X,(du)^{*})=d_{X}u$, for all
$X\in\Gamma(TM)$. The conformal variance of the Levi-Civita
connection will, however, as we shall see, constitute, no limitation
to the development of our theory of conformal submanifold geometry.
\end{rem}

As an oriented conformal surface, $M$ is canonically provided with a
complex structure and the corresponding $(1,0)$- and
$(0,1)$-decomposition. Unless indicated otherwise, these shall be
the underlying structures on $(M,\mathcal{C})$. In this section we
recall these basic concepts and a few basic facts on Riemann
surfaces. \newline

\textbf{The canonical complex structure.} The fact that $M$ is an
oriented Riemannian surface enables us to refer to $90^{\circ}$
rotations in the tangent spaces to $M$, a notion that is obviously
invariant under conformal changes of the metric. We define then what
we will refer to as the $\textit{canonical}$ $\textit{complex}$
$\textit{structure}$ on $(M,\mathcal{C})$, $J\in\Gamma
(\mathrm{End}(TM))$, by $90^{\circ}$ rotation in the positive
direction. This terminology is not casual: $J$ is an almost-complex
structure, compatible with the conformal structure and, as it is
well-known in the ambit of Riemannian geometry, parallel with
respect to the connection induced in $\mathrm{End}(TM)$ by the
Levi-Civita connection of $(M,g)$, for each $g\in\mathcal{C}$;
making $(M,J)$ into a complex manifold, according to
Newlander-Nirenberg Theorem.\newline

\textbf{The (1,0)- and (0,1)-decomposition.} The almost-complex
structure $J$ gives rise to a decomposition
\begin{equation}\label{eq:TMCdecomp}
TM^{\C} =T^{1,0}M\oplus T^{0,1}M,
\end{equation}
for $T^{1,0}M$ and $T^{0,1}M$ the eigenspaces associated to $i$ and
$-i$, respectively, of the complex linear extension of $J$
 to $TM^{\C}$, $$J=I\left\{
\begin{array}{ll} i & \mbox{$\mathrm{on}\,T^{1,0}M$}\\ -i &
\mbox{$\mathrm{on}\,T^{0,1}M$}\end{array}\right.;$$ that is,
$$T^{1,0}M=\langle X-iJX\rangle,\,\,\,\,T^{0,1}M=\langle
X+iJX\rangle,$$fixing $X\in \Gamma(TM)$ locally never-zero.  We
refer to a section of $T^{1,0}M$ as a \textit{$(1,0)$-vector field}
and to a section of $T^{0,1}M$ as a \textit{$(0,1)$-vector field}.
The bundles $T^{1,0}M$ and $T^{0,1}M$ are maximal (with respect to
the inclusion relation) isotropic subbundles of $TM^{\C}$. It is
immediate and useful to observe that, in view of the parallelness of
$J$,
\begin{equation}\label{eq:TijMpreservedbyconnection}
\nabla\,\Gamma(T^{1,0}M)\subset\Omega^{1}(T^{1,0}M)\,\,\,,\,\,\,\nabla\,\Gamma(T^{0,1}M)\subset\Omega^{1}(T^{0,1}M),
\end{equation}
for $\nabla$ the Levi-Civita connection on $(M,g)$, fixing
$g\in\mathcal{C}$.

\textit{N.B.:} The decomposition \eqref{eq:TMCdecomp} is a
particular case of a canonical decomposition, which is worth
recalling: given $V$ a non-degenerate complex $2$-plane, $V$ admits
a unique decomposition $V=V_{+}\oplus V_{-}$ into the direct sum of
two null complex lines. Namely, fixing an orthonormal basis
$v_{1},v_{2}$ of $V$, the set of lines $\{V_{+}, V_{-}\}$ coincides
with the set of lines $\{\langle v_{1}+iv_{2}\rangle,\langle
v_{1}-iv_{2}\rangle\}$, $$V=\langle v_{1}+iv_{2}\rangle\oplus\langle
v_{1}-iv_{2}\rangle.$$ The non-degeneracy of $V$ establishes
$V_{+}\cap (V_{-})^{\perp}=\{0\}$. In the particular case $V$ is
real, we can choose the vectors $v_{1}$ and $v_{2}$ to be real, in
which case the spaces $V_{+}$ and $V_{-}$ are complex conjugate of
each other.\newline

\textbf{Conformality.} Given $X\in\Gamma(TM)$ never-zero, $(X, JX)$
constitutes a direct orthogonal frame of $TM$ with $(X,X)=(JX,JX)$,
for all $(\,,\,)\in\mathcal{C}$. Hence a positive definite metric
$(\,,\,)$ on $M$ is in the conformal class $\mathcal{C}$ if and only
if, fixing $X\in\Gamma( TM)$ locally never-zero,
$$(X,X)=(JX,JX),\,\,\,\,(X,JX)=0;$$
or, equivalently, $$(Z,Z)=0$$ (respectively, $(\bar{Z},\bar{Z})=0$),
fixing $Z\in\Gamma(T^{1,0}M)$
(respectively,$\bar{Z}\in\Gamma(T^{0,1}M)$), locally never-zero. The
conformality of a metric on $M$ is equivalent to the isotropy of
$T^{1,0}M$ or, equivalently, of $T^{0,1}M$, with respect to that
metric. In particular, given an immersion $\phi$ of $M$ into a
Riemannian manifold and $i\neq j\in\{0,1\}$, the conformality of
$\phi$ [with respect to any metric in $\mathcal{C}$],
$g_{\phi}\in\mathcal{C}$, can be characterized by
$$(d^{i,j}\phi,d^{i,j}\phi)=0,$$ fixing $i\neq j\in \{0,1\}$, where
we use $(d^{i,j}\phi,d^{i,j}\phi)$ to denote
$g_{\phi}\vert_{T^{i,j}M\times T^{i,j}M}$. In the paragraph below, a
characterization of the conformality of $\phi$ in the light of a
holomorphic chart of $M$ is presented.
\newline

\textbf{Holomorphicity.} Let $z=x+iy:M\rightarrow \C$ be a chart of
$M$. Throughout this text, let $g_{z}$ denote the (positive
definite) metric induced in $M$ by $z:M\rightarrow
(\C,\mathrm{Re}(\,,\,)_{\C})$,
$$g_{z}=dx^{2}+dy^{2}.$$
We set, as usual,
$$\delta _{x}:=dz^{-1}(1),\,\,\,\,\,\delta _{y}:=dz^{-1}(i),$$ defining an orthonormal frame $\delta _{x},\delta _{y}$ of $TM$
with respect to $g_{z}$ ($(dx,dy)$ is the frame of $(TM)^{*}$ dual
to $(\delta_{x},\delta_{y})$). The coordinates $x,y$ are said to be
\textit{conformal coordinates} if the metric $g_{z}$ is in the
conformal class $\mathcal{C}$, or, equivalently,
$$J\delta_{x}=\pm\delta_{y},$$ the chart
$z:(M,J)\rightarrow\C$ is either holomorphic or anti-holomorphic,
respectively, depending on the orientation on $M$.

Suppose $z$ is a holomorphic chart of $(M,J)$. In that case,
$J\delta_{x}=\delta_{y}$, and, therefore,
$$\delta _{z}:=\frac{1}{2}\,(\delta _{x}-i\delta _{y}),\,\,\,\,\,\delta _{\bar {z}}:=\frac{1}{2}\,(\delta _{x}+i\delta _{y})$$defines
a never-zero $(1,0)$-vector field and, respectively, a never-zero
$(0,1)$-vector field. Hence
$$T^{1,0}M=\langle \delta _{z}\rangle,\,\,\,\,\,T^{0,1}M=\langle
\delta _{\bar {z}}\rangle,$$ and, therefore,
$$TM^{\C}=\langle\delta _{z},\delta _{\bar {z}}\rangle .$$
It is, perhaps, worth remarking that, as $(1,0)$- and $(0,1)$-vector
fields, respectively, $\delta_{z}$ and $\delta_{\bar{z}}$ verify
\begin{equation}\label{eq:Jdeltazbarz}
J\delta_{z}=i\delta_{z},\,\,\,\,\,\,J\delta_{\bar{z}}=-i\delta_{\bar{z}}.
\end{equation}
The fact that $\delta _{z},\delta _{\bar {z}}$ constitutes a frame
of $TM^{\C}$ with
\begin{equation}\label{eq:LiBrdelta0}
[\delta _{z},\delta _{\bar {z}}]=0
\end{equation}
will be useful on many occasions.

Given a mapping $\eta$ of $M$, we shall, alternatively, write $\eta
_{x}$, $\eta _{y}$, $\eta _{z}$ and $\eta _{\bar {z}}$ for,
respectively, $d\eta (\delta _{x})$, $d\eta (\delta _{y})$, $d\eta
(\delta _{z})$ and $d\eta (\delta _{\bar{z}})$. In view of the
holomorphicity of $z$, the conformality of a mapping $\eta$ of $M$
into a Riemannian manifold can be characterized by
$$(\eta_{z},\eta_{z})=0,$$ or, equivalently,
$$(\eta_{\bar{z}},\eta_{\bar{z}})=0;$$ whereas the holomorphicity of
a map $\eta\in C^{\infty}(M,\C)$ can be characterized by
$$\eta_{\bar{z}}=0.$$

We complete this section with a basic remark on change of
holomorphic coordinates. Let $\omega$ be another holomorphic chart
of $(M,J)$. The fact that $(dz,\,d\bar{z})$ is the frame of
$((TM)^{\C})^{*}$ dual to $(\delta_{z},\delta_{\bar{z}})$ makes
clear that
$$d\omega =\omega _{z}\,dz$$
and, in particular, that the metric induced in $M$ by $\omega$
relates to the one induced by $z$ by
\begin{equation}\label{eq:gomegagz}
g_{\omega}=\mid\omega_{z}\mid^{2}g_{z}.
\end{equation}
$\newline$

\textbf{The Hodge $*$-operator.} Let $P$ be a vector bundle over
$M$, provided with a metric, and $p\in\{0,1,2\}$.  Given
$\mu\in\Omega^{p}(P)$ and $\eta \in\Omega ^{2-p}(P)$, we define a
$2$-form $(\mu \wedge\eta)\in \Omega ^{2}(\underline {\mathbb
{\C}})$ by, in the case $p=1$,
$$ (\mu\wedge \eta)_{(X,Y)}:=(\mu _{X},\eta _{Y})-(\mu
_{Y},\eta _{X})$$and, in the case $p=2$, $$(\mu\wedge \eta)
_{(X,Y)}:=(\mu _{(X,Y)},\eta),$$ for $X,Y\in\Gamma(TM)$; and, in the
case $p=0$, $(\mu\wedge \eta):=(\eta\wedge\mu)$. Note that, in the
case $p=1$, $(\mu\wedge\eta)=-(\eta\wedge\mu)$. Fix a metric
$g\in\mathcal{C}$ and provide $L(TM,P)$ with the metric canonically
induced by $(TM,g)$ and $P$. Given $k\in\{0,1,2\}$, define a metric
on $A^{k}(TM,P)$ by setting
$$(\xi_{1}\wedge...\wedge\xi_{k},\gamma_{1}\wedge...\wedge\gamma_{k}):=\mathrm{det}((\xi_{i},\gamma_{j})),$$
for $\xi_{i},\gamma_{i}\in\Gamma(L(TM,P))$.  Recall the Hodge
$*$-operator on $p$-forms over $M$, when provided with the metric
$g$, transforming a form $\mu\in\Omega^{p}(P)$ into the form
$*\mu\in\Omega^{2-p}(P)$ defined by the relation
$$(\mu\wedge\eta)=(*\mu,\eta)\,dA,$$for all
$\eta\in\Omega^{2-p}(P)$, with $dA$ denoting the area element of
$(M,g)$. Recall that
$$**\mu=(-1)^{p(2-p)+s}\mu,$$
for $s$ the number of negative eigenvalues of the metric tensor of
$P$. It is useful to observe that, in the particular case $P$ is the
trivial bundle
$\mathrm{End}(\underline{\R}^{n+1,1}):=M\times\mathrm{End}(\R^{n+1,1})$,
$s=2(n+1)$.
\begin{Lemma}\label{hodgeconfinv}
The Hodge $*$-operator on $1$-forms over a surface is invariant
under conformal changes of the metric on the surface.
\end{Lemma}
Before proceeding to the proof of the lemma, it is opportune and
useful to observe how the volume element changes under conformal
changes of the metric.
\begin{Lemma}\label{volconfchanges}
Let $N$ be an $n$-dimensional oriented manifold and $g$ and
$g':=e^{u}g$, for some $u\in C^{\infty}(N,\R)$, be conformally
equivalent positive definite metrics in $N$. Then:
$$\mathrm{d}\mathrm{vol}_{(N,g')}=(e^{u})^{\frac{n}{2}}\,\mathrm{d}\mathrm{vol}_{(N,g)},$$
\end{Lemma}

\begin{proof}
Let $(X_{i})_{i}$ be a direct orthonormal frame of $(TN,g)$ and
$(\xi_{i})_{i}$ be the frame of $(TN)^{*}$ dual to $(X_{i})_{i}$.
Then
$$\mathrm{d}\mathrm{vol}_{(N,g)}=
\xi_{1}\wedge.....\wedge\xi_{n}.$$ On the other hand, clearly,
$(\frac{1}{\sqrt{e^{u}}}\,X_{i})_{i}$ is a direct orthonormal frame
of $(TN,g')$ with dual frame $(\sqrt{e^{u}}\,\xi_{i})_{i}$, and,
therefore,
$\mathrm{d}\mathrm{vol}_{(N,g')}=(\sqrt{e^{u}})^{n}\,\xi_{1}\wedge.....\wedge\xi_{n}$.
\end{proof}
Now we proceed to the proof of Lemma \ref{hodgeconfinv}.
\begin{proof}
Let $g_{1}$ and $g_{2}:=e^{u}g_{1}$, for some $u\in
C^{\infty}(M,\R)$, be conformally equivalent positive definite
metrics in $M$. For $i=1,2$, let $*_{i}$ and $dA_{i}$ denote,
respectively, the Hodge $*$-operator and the area element of
$(M,g_{i})$, and $(\,,)_{i}$ denote the Hilbert-Schmidt metric in
$L((TM,g_{i}),P)$.
 Fix $\mu,\eta\in\Omega^{1}(P)$. The proof will consist of showing that
\begin{equation}\label{eq:hodge21}
(\mu\wedge\eta)=(*_{2}\,\mu,\eta)_{1}\,dA_{1}.
\end{equation}

By definition of $*_{2}$, and according to Lemma
\ref{volconfchanges},
$(\mu\wedge\eta)=(*_{2}\,\mu,\eta)_{2}\,e^{u}dA_{1}$. On the other
hand, given an orthonormal frame $X_{1},X_{2}$ of $(TM,g_{1})$,
$\frac{1}{\sqrt{e^{u}}}\,X_{1},\frac{1}{\sqrt{e^{u}}}\,X_{2}$ is an
orthonormal frame of $(TM,g_{2})$ and, therefore,
$(*_{2}\,\mu,\eta)_{2}=\frac{1}{e^{u}}\,(*_{2}\,\mu,\eta)_{1}$,
completing the proof.
\end{proof}

The Hodge $*$-operator on $1$-forms over $(M,\mathcal{C})$ is
closely related to the canonical complex structure in
$(M,\mathcal{C})$ in a well-known result in the ambit of Riemannian
geometry: $*$ acts on forms $\omega\in\Omega ^{1}(P)$ by
\begin{equation}\label{eq:*vsJ}
*\omega =-\omega J.
\end{equation}

\textbf{(1,0)-,(0,1)- and (1,1)-forms.} The decomposition
\eqref{eq:TMCdecomp} provides the (standard) decomposition of a
$1$-form into its $(1,0)$ and $(0,1)$ parts, as well as that of a
$2$-form into its $(2,0)$, $(1,1)$ and $(0,2)$ parts. Recall that
there are no non-zero $(2,0)$- or $(0,2)$-forms on a surface.

Let $P$ be a vector bundle over $M$, provided with a metric, and
$\xi$ be a $1$-form with values in $P$. Since $\xi^{1,0}$ and
$\xi^{0,1}$ are the complexifications of, respectively, a complex
linear and a complex anti-linear sections of $\mathrm{Hom}(TM,P)$,
we have, by equation \eqref{eq:*vsJ},
$$*\,\xi ^{1,0}=-i\,\xi ^{1,0},\,\,\,\,\,*\,\xi ^{0,1}=i\,\xi ^{0,1},$$ so that $*\,\xi =-i(\xi ^{1,0}-\xi
^{0,1})$ and, therefore,
$$\xi ^{1,0}=\frac{1}{2}\,(\xi+i*\,\xi),\,\,\,\,\,\xi
^{0,1}=\frac{1}{2}\,(\xi-i*\,\xi),$$having in consideration that
$\overline{*\,\xi}=*\,\overline{\xi}$. Note that $\overline{\xi
^{1,0}}=\overline{\xi}\, ^{0,1}$. In particular, $\xi$ is real
($\overline{\xi}=\xi$) if and only if
$\xi^{0,1}=\overline{\xi^{1,0}}$.

\subsection{Conformal immersions of surfaces in the projectivized
light-cone}\label{subsec:confimm}

Having modeled the conformal $n$-sphere on the projectivized
light-cone of $\R^{n+1,1}$, and, in this way, all $n$-dimensional
space-forms on submanifolds of $\mathbb{P}(\mathcal{L})$, we
approach a surface conformally immersed in a space-form as a null
line bundle $\Lambda$ defining an immersion of an oriented surface,
which we provide with the conformal structure induced by $\Lambda$,
into the projectivized light-cone. In this work, we restrict to
surfaces in $S^{n}$ which are not contained in any subsphere of
$S^{n}$. Such a surface defines a surface in any given space-form,
by means of a lift, whose study is M\"{o}bius equivalent to the
study of the surface and which will be often considered. Namely,
given $v_{\infty}\in\R^{n+1,1}$ non-zero, we have, locally,
$(\sigma,v_{\infty})\neq 0$, and $\Lambda$ defines then a local
immersion $\sigma_{\infty}:=(\pi\vert_{S_{v_{\infty}}})^{-1}\circ
\Lambda=\frac{-1}{(\sigma,v_{\infty})}\,\sigma:M\rightarrow
S_{v_{\infty}}$, of $M$ into the space-form
$S_{v_{\infty}}$.\newline

Let $\Lambda$ be a null line subbundle of
$\underline{\mathbb{R}}^{n+1,1}$. In particular, $\Lambda$ defines a
smooth map $\Lambda:M\rightarrow \mathbb{P} (\mathcal {L})$, by
assigning to each $p$ in $M$ the corresponding fibre, $\Lambda_{p}$.
If $\Lambda$ is an immersion, then $\Lambda $ induces naturally,
from the conformal structure on $\mathbb{P}(\mathcal{L})$, a
conformal structure in $M$, which we denote by
$\mathcal{C}_{\Lambda}$, making
$$\Lambda:(M,\mathcal{C}_{\Lambda})\rightarrow\mathbb{P} (\mathcal
{L})$$ into a conformal immersion of $M$ into the projectivized
light-cone.

\begin{prop}
If $\Lambda :M\rightarrow \mathbb{P} (\mathcal {L})$ is an
immersion, then, given a never-zero section $\sigma$ of $\Lambda$,
the metric $g_{\sigma}$ induced in $M$ by $\sigma:M\rightarrow
\mathbb{R}^{n+1,1}$ is in the conformal class $\mathcal
{C}_{\Lambda}$ of metrics,
$$g_{\sigma}\in\mathcal {C}_{\Lambda}.$$
\end{prop}

\begin{proof}
For $t_{0}$ unit time-like vector, $\langle t_{0}\rangle ^{\perp}$
is an Euclidean space and, therefore, $(\sigma,t_{0})$ is
never-zero. Furthermore, $(\pi_{\mathcal{L}}\vert
_{S_{t_{0}}})^{-1}\Lambda =-(\sigma,t_{0})^{-1}\,\sigma,$ and,
therefore,
$$d((\pi_{\mathcal{L}}\vert _{S_{t_{0}}})^{-1}\,\Lambda)=\frac{(d\sigma,t_{0})}{(\sigma
,t_{0})^{2}}\,\sigma -\frac{1}{(\sigma ,t_{0})}\,d\sigma.$$ As
$(\sigma,\sigma)=0=(\sigma,d\sigma)$, we conclude that, for
$X,Y\in\Gamma (TM)$,
$$( d_{X}((\pi_{\mathcal{L}}\vert _{S_{t_{0}}})^{-1}\,\Lambda) ,
d_{Y}((\pi_{\mathcal{L}}\vert _{S_{t_{0}}})^{-1}\,\Lambda))
=\frac{1}{(\sigma ,t_{0})^{2}}\,(d_{X}\sigma ,d_{Y}\sigma)$$or,
equivalently,
$$g_{\sigma}(X,Y)=(\sigma ,t_{0})^{2}( d_{X}\Lambda, d_{Y}\Lambda )
_{t_{0}}=(\sigma ,t_{0})^{2}g_{\Lambda}^{t_{0}}( X,Y),$$ for
$(\,,\,) _{t_{0}}$ the (positive definite) metric induced in
$\mathbb{P} (\mathcal {L})$ by $(\pi_{\mathcal{L}}\vert
_{S_{t_{0}}})^{-1}$ and $g_{\Lambda}^{t_{0}}$ the (positive
definite) metric induced in $M$ by $\Lambda$ from $(\,,\,)
_{t_{0}}$. We conclude that $g_{\sigma}$ is a positive definite
metric conformally equivalent to $g_{\Lambda}^{t_{0}}\in\mathcal
{C}_{\Lambda}$.
\end{proof}

Given $\sigma$ a never-zero section of $\Lambda$,
$\Lambda=\pi_{\mathcal{L}}\,\sigma$ and, therefore, $d\Lambda
=d(\pi_{\mathcal{L}}) _{\sigma}\, d\sigma $. Hence, if $\Lambda$ is
an immersion, then so is $\sigma$, and, conversely, if $\sigma$ is
an immersion, then, in view of \eqref{eq:kerdpi}, $\Lambda$ is an
immersion if and only if $\Lambda$ is in direct sum with
$d\sigma(TM)$. Set
$$\Lambda ^{(1)}:= \langle\sigma ,d\sigma (e_{1}),d\sigma (e_{2})\rangle =\Lambda +d\sigma(TM) ,$$
independently of the choices of a never-zero section $\sigma$ of
$\Lambda$ and of a frame $e_{1},e_{2}$ of $TM$. This is, indeed,
well-defined: given a never-zero $\lambda\in\Gamma
(\underline{\R})$,
$$d_{X}(\lambda\sigma )=\lambda d_{X}\sigma +(d_{X}\lambda )\sigma
=\lambda d_{X}\sigma \,\, \mathrm{mod}\langle\sigma\rangle.$$

\begin{prop}
$\Lambda :M\rightarrow \mathbb{P} (\mathcal {L})$ is an immersion if
and only if $\mathrm{rank}\,\Lambda ^{(1)}=3$.
\end{prop}
\begin{proof}
Fix $\sigma$ a never-zero section of $\Lambda$. If
$\mathrm{rank}\,\Lambda ^{(1)}=3$, then $\mathrm{rank}\,d\sigma
(TM)=2$, or, equivalently, $\sigma $ is an immersion, and $d\sigma
(TM)$ is in direct sum with $\Lambda$. Hence $\Lambda$ is an
immersion.

Conversely, if $\Lambda$ is an immersion, then so is $\sigma$, in
which case $d\sigma(TM)$ is a bundle of Euclidean $2$-planes and,
therefore,  $\mathrm{rank}\,\Lambda ^{(1)}=3$.
\end{proof}

In the case $\Lambda :M\rightarrow \mathbb{P} (\mathcal {L})$ is an
immersion, $\Lambda (M)$ consists of a surface [conformally
immersed] in $\mathbb{P}(\mathcal{L})$,\footnote{If $\Lambda
:M\rightarrow \mathbb{P} (\mathcal {L})$ is an immersion, then
$\Lambda$ defines a surface in spherical $n$-space. Since $n\geq 3$,
we can choose $v_{\infty}\in\mathcal{L}$ such that $\Lambda_{p}\neq
\langle v_{\infty}\rangle $, $$\Lambda _{p}\in
\mathbb{P}(\mathcal{L})\backslash\{\langle v_{\infty}\rangle\}\cong
S^{n}\backslash\{x_{0}\},$$for all $p\in M$, showing that $\Lambda$
defines, furthermore, a surface in Euclidean $n$-space. In the case
$M$ is compact, $\Lambda (M)\subset S^{n}$ is a compact
$2$-dimensional manifold and, therefore, $S^{n}\backslash \Lambda
(M)$ is a non-empty open subset. This shows that $\Lambda (M)$
avoids some hypersphere in $S^{n}$, or, equivalently, that we can
choose a space-like vector $v_{\infty}$ for which
$\Lambda_{p}\nsubseteq \langle v_{\infty}\rangle ^{\perp},\forall
p\in M$. For such a $v_{\infty}$, $\Lambda$ defines a surface in
$\mathbb{P} (\mathcal {L})\backslash\mathbb{P} (\mathcal {L}\cap
\langle v_{\infty}\rangle ^{\perp})$ and, locally (on a connected
component of $M$), in hyperbolic $n$-space.} which we refer to,
alternatively, as the surface $\Lambda$. We restrict ourselves to
surfaces $\Lambda$ in $S^{n}$ for which
\begin{equation}\label{eq:notinsmallerspheres}
\Lambda(M)\nsubseteq S^{k},
\end{equation}
for all $k<n$, and, in particular, to surfaces which do not lie in
any $2$-sphere. This ensures, in particular, that, for a general
non-zero $v_{\infty}\in\R^{n+1,1}$, given $\sigma\in\Gamma(\Lambda)$
never-zero, we have, locally, $(\sigma,v_{\infty})\neq 0$. In fact,
if $(\sigma_{p},v_{\infty})=0$ for all $p\in M$, then
$\Lambda(M)\subset\mathbb{P}(\mathcal{L}\cap \langle
v_{\infty}\rangle^{\perp}),$ which, in the case $v_{\infty}$ is
time-light or light-like, is impossible, according to Lemma
\ref{mathcalLcapvinfperp}, and, in the case $v_{\infty}$ is
space-like, contradicts \eqref{eq:notinsmallerspheres} (in the case
$v_{\infty}$ is space-like, $\mathbb{P}(\mathcal{L}\cap \langle
v_{\infty}\rangle^{\perp})$ is an hypersphere in $S^{n}$).

A surface $\Lambda$ defines a local immersion of $M$ into
$\mathbb{P}(\mathcal{L})\backslash\mathbb{P}(\mathcal{L}\cap\langle
v_{\infty}\rangle ^{\perp})$ and then, via the diffeomorphism
$\pi_{\mathcal{L}}\vert_{S_{v_{\infty}}}:S_{v_{\infty}}\rightarrow
\mathbb{P}(\mathcal{L})\backslash\mathbb{P}(\mathcal{L}\cap\langle
v_{\infty}\rangle^{\perp})$, a local immersion
$$\sigma_{\infty}:=(\pi_{\mathcal{L}}\vert_{S_{v_{\infty}}})^{-1}\circ \Lambda=\frac{-1}{(\sigma,v_{\infty})}\,\sigma:M\rightarrow S_{v_{\infty}},$$of $M$ into the space-form $S_{v_{\infty}}$, which we refer to as
\textit{ the surface defined by $\Lambda$ in $S_{v_{\infty}}$}. The
surfaces $\sigma_{\infty}$, with $v_{\infty}\in\R^{n+1,1}$ non-zero,
defined by $\Lambda$, form a family of M\"{o}bius equivalent
surfaces, whose study is M\"{o}bius equivalent to the study of the
surface $\Lambda$.

\chapter{The central sphere congruence}\label{chaptercsc}

\markboth{\tiny{A. C. QUINTINO}}{\tiny{CONSTRAINED WILLMORE
SURFACES}}

Following \cite{SD}, we introduce the central sphere congruence, a
fundamental construction of M\"{o}bius invariant surface geometry,
which will be basic to our study of surfaces. The concept has its
origins in the nineteenth century with the introduction of the mean
curvature sphere of a surface at a point, by S. Germain
\cite{germain3}. The terminology reflects the central role played by
the mean curvature of a surface. By the turn of the century, the
family of the mean curvature spheres of a surface was known as the
central sphere congruence, cf. W. Blaschke \cite{blaschke}.
Nowadays, after R. Bryant's paper \cite{bryant}, it goes as well by
the name conformal Gauss map.
\newline

We start by recalling some basic concepts in Riemannian geometry.
Given a Riemannian manifold  $\bar{M}$ and an immersion
$f:M\to\bar{M}$, the pull-back bundle $f^*T\bar{M}$ splits into the
direct sum $T_{f}\oplus N_{f}$, where $T_{f}$ and $N_{f}$ denote,
respectively, the tangent bundle, $df(TM)$, and the normal bundle,
$(df(TM))^\perp$, to $f$. Let $\pi _{T_{f}}$ and $\pi _{N_{f}}$
denote the orthogonal projections of $f^*T\bar{M}$ onto $T_{f}$ and
$N_{f}$, respectively. The normal bundle is provided with the
connection
$$\nabla ^{N_{f}}=\pi _{N_{f}}\circ\nabla ^{f^*T\bar{M}}\vert
_{\Gamma (N_{f})},$$ where $\nabla ^{f^*T\bar{M}}$ denotes the
pull-back connection by $f$ of the Levi-Civita connection on
$T\bar{M}$. Recall the second fundamental form of $f$,
$\Pi\in\Gamma(L^2(TM,N_{f}))$, defined by
$$\Pi (X,Y)=\pi _{N_{f}}(\nabla ^{f^*T\bar{M}}_{X}d_{Y}f),$$ and the
mean curvature vector of $f$ (or of $f(M)\subset\bar{M}$, the
surface $M$ immersed in $\bar{M}$ by $f$),
$$\mathcal{H} =\frac{1}{2}\,\mathrm{tr}_{g_{_{f}}}\Pi\in\Gamma (N_{f}),$$
with $\mathrm{tr}_{g_{_{f}}}$ indicating trace computed with respect
to $g_{f}$. Recall that
$$\nabla^{f^{*}T\bar{M}}_{X}d_{Y}f-\nabla^{f^{*}T\bar{M}}_{Y}d_{X}f=d_{[X,Y]}f,$$
for all $X,Y\in\Gamma(TM)$, establishing, in particular, the
symmetry of $\Pi$. Recall the fundamental equation in Riemannian
Geometry:
\begin{equation}\label{eq:isometryvsconnectionsandsff}
\pi_{T_{f}}(\nabla ^{f^*T\bar{M}}_{X}df(Y))=df(\nabla^{TM}_{X}Y),
\end{equation}
for $\nabla^{TM}$ the Levi-Civita connection of $M$ when provided
with the metric induced by $f$ from the one on $\bar{M}$; for all
$X,Y\in\Gamma(TM)$. Fix a unit $\xi\in\Gamma (N_{f})$ and recall the
shape operator $A^{\xi}\in\Gamma (\mathrm {End}\,(TM,T_{f}))$, of
$f$ with respect to $\xi$, given by
$$A^{\xi}(X)=-\pi _{T_{f}}(\nabla ^{f^*T\bar{M}}_{X}\xi),$$ and
the mean curvature of $f$ (or of $f(M)$), with respect to $\xi$,
$$H^{\xi}=\frac{1}{2}\,\mathrm{tr}A^{\xi}\in\Gamma (\underline {\mathbb{R}}).$$
For an isometric immersion, the second fundamental form and the
shape operator with respect to $\xi$ are related by
\begin{equation}\label{eq:secffvsshapeop}
(\Pi(X,Y),\xi) =(A^{\xi}(X),Y).
\end{equation}
In particular, for an isometric immersion,
$$H^{\xi }=(\mathcal{H},\xi).$$
Equation \eqref{eq:secffvsshapeop} establishes, on the other hand,
the symmetry - and consequent diagonalizability - of the shape
operator of an isometric immersion. The shape operators $A^{\xi}$
and $A^{-\xi}$ are symmetrical and, therefore, share eigenspaces and
have symmetrical eigenvalues. Recall that, in the case $f$ is an
isometric immersion of $M$ into $\bar{M}=\R^{3}$, the eigenvalues
$k_{1}$ and $k_{2}$ of $A^{\xi}$ are called the principal curvatures
of $f$, defined up to sign; that a point $p$ in $M$ at which the
principal curvatures are the same is said to be umbilic; and that,
away from umbilic points, the directions defined by the two lines
consisting of the common eigenspaces of $A^{\xi}$ and $A^{-\xi}$ are
called the principal directions of $f$, whilst for umbilic points
all directions are said to be principal. In particular, the mean
curvature of $f:M\rightarrow \R^{3}$ with respect to either of the
two unit normal vector fields to $f$ is given, up to sign, by
\begin{equation}\label{eq:mcineucld3}
H=\frac{k_{1}+k_{2,}}{2},
\end{equation}
the arithmetic mean of the principal curvatures. It is opportune to
recall that, in the case $f$ defines an isometric immersion of $M$
in Euclidean $3$-space, the Gaussian curvature of the surface
$f(M)$, or, equally (cf. \textit{theorema egregium} of Gauss), that
of $M$, can be obtained as
\begin{equation}\label{GaussiancurvinEucli3}
K=\mathrm{det}A^{\xi}=k_{1}k_{2},
\end{equation}
fixing $\xi$ a unit normal vector field to $f$. For that, and for
further reference, recall the Gauss equation,
$$R(X,Y,Z,W)-\bar{R}(df(X),df(Y),df(Z),df(W))$$ $$=(\Pi(Y,W),\Pi
(X,Z))-(\Pi (X,W),\Pi (Y,Z)),$$ for $X,Y,Z,W\in\Gamma(TM)$; relating
the curvature tensors $R$ and $\bar{R}$ of $(M,g_{f})$ and
$\bar{M}$, respectively. It establishes, in particular, that, if
$\bar{M}$ has constant sectional curvature $\bar{K}=\bar{K}(x)$, for
$x\in\bar{M}$, then the Gaussian curvature $K$ of $M$ relates to
$\bar{K}$ by
\begin{equation}\label{eq:KvsbarK}
K-\bar{K}=(\Pi(X_{1},X_{1}),\Pi (X_{2},X_{2}))-(\Pi
(X_{1},X_{2}),\Pi (X_{1},X_{2})),
\end{equation}
fixing an orthonormal frame $X_{1},X_{2}$ of $(TM,g_{f})$. Now
consider the particular case $\bar{M}=\R^{3}$, fix
$\xi\in\Gamma(N_{f})$ unit and consider $X_{1},X_{2}$ to be a frame
along principal directions of $f$,  say $A^{\xi}X_{i}=k_{i}X_{i}$,
for $i=1,2$, (whose existence is established by the symmetry of
$A^{\xi}$). Then, according to \eqref{eq:secffvsshapeop},
$\Pi(X_{i},X_{i})=k_{i}\xi$, for $i=1,2$, and $\Pi(X_{1},X_{2})=0$
and the conclusion follows then from \eqref{eq:KvsbarK}.

\begin{rem}\label{mcvnotconfinv}
Suppose $g$ and $g'$ are conformally equivalent positive definite
metrics on $\bar{M}$, say $g'=e^{2u}g$, for some $u\in
C^{\infty}(\bar{M},\R)$. Following \eqref{eq:relLCs}, we get that
the connections $\nabla$ and $\nabla'$, induced in the pull-back
bundle $f^{*}T\bar{M}$ by the Levi-Civita connections on
$(\bar{M},g)$ and $(\bar{M},g')$, respectively, are related by
$$\nabla'_{X}Y=\nabla_{X}Y+((f^{*}du)Y)df(X)+((f^{*}du)df(X))Y-g(df(X),Y)f^{*}(du)^{*},$$
for $X\in\Gamma(TM)$, $Y\in\Gamma(f^{*}T\bar{M})$, $f^{*}du$ the
pull-back by $f$ of $du$,
$$(f^{*}du)Z:=(x\mapsto du_{f(x)}(Z_{x}))\in C^{\infty}(M,\R),$$
given $Z\in\Gamma(f^{*}T\bar{M})$; and, analogously, $f^{*}(du)^{*}$
the pull-back by $f$ of $(du)^{*}$. It follows that the second
fundamental forms $\Pi$ and $\Pi '$ of the immersions by $f$ of $M$
into $(\bar{M},g)$ and $(\bar{M},g')$, respectively, are related by
\begin{equation}\label{eq:howPichanges}
\Pi '(X,Y)=\Pi (X,Y)-g_{f}(X,Y)\,\pi_{N_{f}}(f^{*}(du)^{*}),
\end{equation}
for all $X,Y\in\Gamma(TM)$, with $g_{f}$ denoting the metric induced
in $M$ by $f$ from the metric $g$. Consequently,
\begin{equation}\label{eq:conformalvarianceoftheMCV}
\mathcal{H}'=e^{-2u\circ f}\mathcal{H}-e^{-2u\circ
f}\pi_{N_{f}}(f^{*}(du)^{*}),
\end{equation}
relating the mean curvature vectors $\mathcal{H}'$ and $\mathcal{H}$
of $f:M\rightarrow (\bar{M},g)$ and $f:M\rightarrow (\bar{M},g')$,
respectively. The conformal variance of the mean curvature vector
will, however, constitute no limitation to the development of our
conformal submanifold theory, as we shall see.
\end{rem}

Let $\Lambda\subset\underline{\R}^{n+1,1}$ be a surface in the
projectivized light-cone.

\section{Central sphere congruence and mean curvature}\label{sec:csc}

\markboth{\tiny{A. C. QUINTINO}}{\tiny{CONSTRAINED WILLMORE
SURFACES}}

Bryant \cite{bryant} established the existence of a congruence of
$2$-spheres, named the conformal Gauss map, tangent to a Riemann
surface isometrically immersed in $S^{3}$, sharing the mean
curvature at each point with the surface. This congruence of spheres
can be generalized to surfaces conformally immersed in $S^{n}$ by
means of the central sphere congruence, which we present in this
section, following \cite{SD}.\newline

A bundle $V\subset\underline{\mathbb{R}}^{n+1,1}$ of $(3,1)$-planes
is said to be an \textit{enveloping} $2$-\textit{sphere}
\textit{congruence} of the surface $\Lambda$ if the $2$-spheres
$\mathbb{P}(\mathcal{L}\cap V_{p})\subset
\mathbb{P}(\mathcal{L})\cong S^{n}$, for $p\in M$, have first order
contact with $\Lambda$, i.e., $\Lambda ^{(1)}\subset V$.

\begin {defn}We define an enveloping $2$-sphere congruence to $\Lambda$,
said to be the $\emph{central}$ $\emph{sphere}$ $\emph{congruence}$,
by
$$S:=\langle \sigma ,d_{e_{1}}\sigma ,d_{e_{2}}\sigma ,\sum _{i}
d_{e_{i}}d_{e_{i}}\sigma\rangle \subset
\underline{\mathbb{R}}^{n+1,1},$$ independently of the choices of a
never-zero $\sigma\in\Gamma (\Lambda)$ and of a local orthonormal
frame $(e_{i})_{i}$ of $TM$ with respect to the metric $g_{\sigma}$.
We may, alternatively, use the notation $S_{\Lambda}$ for $S$.
\end {defn}

We shall now recognize that this is, in fact, well-defined. Fix
$\sigma$ a never-zero section of $\Lambda$ and $(e_{i})_{i}$ an
orthonormal frame of $(TM,g_{\sigma})$. First of all, observe that,
by differentiating $(\sigma , d_{e_{i}}\sigma )=0$ we get $(\sigma
,d_{e_{i}}d_{e_{i}}\sigma )=-( d_{e_{i}}\sigma ,d_{e_{i}}\sigma
)=-1$, for $i=1,2$, and, consequently,
\begin{equation}\label{eq:8765redcvbnml987654redfghjkihgfdew3456789o3456knjhghg}
(\sigma ,\sum _{i}d_{e_{i}}d_{e_{i}}\sigma)=-2,
\end{equation}
which, together with $(\sigma ,\sigma )=0=(\sigma, d\sigma)$, shows
that $\sum _{i} d_{e_{i}}d_{e_{i}}\sigma$ is not a section of
$\Lambda ^{(1)}$. Thus $S =\Lambda ^{(1)}\oplus \langle\sum _{i}
d_{e_{i}}d_{e_{i}}\sigma\rangle.$ Let us now show that $\Lambda
^{(1)}\oplus \langle\sum _{i} d_{e_{i}}d_{e_{i}}\sigma\rangle$ does
not depend on the choices of $\sigma$ and $(e_{i})_{i}$. Consider
the Hessian of $\sigma :(M,g_{\sigma})\rightarrow\R^{n+1,1}$, the
section $\nabla d\sigma$ of $S^2(TM,\underline{\R}^{n+1,1})$, given
by $\nabla d\sigma(X,Y)=d_{X}d_{Y}\sigma-d\sigma(\nabla
^{\sigma}_{X}Y)$, for $\nabla ^{\sigma}$ the Levi-Civita connection
on $(M,g_{\sigma})$. Writing $\mathrm{tr}_{g_{\sigma}}$ for the
trace with respect to the metric $g_{\sigma}$, we have
$$\sum _{i}d_{e_{i}}d_{e_{i}}\sigma=\mathrm{tr}_{g_{\sigma}}\nabla
d\sigma +d\sigma(\sum_{i}\nabla ^{\sigma}
_{e_{i}}e_{i})=(\mathrm{tr}_{g_{\sigma}}\nabla d\sigma)\,\,
\mathrm{mod}\,\Lambda ^{(1)},$$ showing that $\Lambda ^{(1)}\oplus
\langle\sum _{i} d_{e_{i}}d_{e_{i}}\sigma\rangle$ does not depend on
the choice of $(e_{i})_{i}$. On the other hand, given
$\lambda\in\Gamma (\underline {\R})$ never-zero, the metrics induced
by $\sigma$ and $\sigma ':=\lambda\sigma $ are related by $g_{\sigma
'}=\lambda ^{2}g_{\sigma}$, so that $(\lambda ^{-1}e_{i})_{i}$
constitutes an orthonormal frame of $(TM,g_{\sigma '})$. Now
\begin{eqnarray*}
\sum _{i}d_{\lambda ^{-1}e_{i}}d_{\lambda ^{-1}e_{i}}\sigma
'&=&\lambda ^{-1}\sum _{i}d_{e_{i}}d_{e_{i}}\sigma +\lambda
^{-2}\sum _{i}((d_{e_{i}}\lambda )(d_{e_{i}}\sigma
)+(d_{e_{i}}d_{e_{i}}\lambda )\sigma)\\&=& (\lambda ^{-1}\sum
_{i}d_{e_{i}}d_{e_{i}}\sigma )\,\mathrm{mod}\,\Lambda
^{(1)}\end{eqnarray*} shows that $$\Lambda ^{(1)}\oplus\langle\sum
_{i}d_{e_{i}}d_{e_{i}}\sigma \rangle =\Lambda
^{(1)}\oplus\langle\sum _{i}d_{\lambda ^{-1}e_{i}}d_{\lambda
^{-1}e_{i}}\sigma ' \rangle ,$$ showing the independence of $\Lambda
^{(1)}\oplus\langle\sum _{i}d_{e_{i}}d_{e_{i}}\sigma \rangle$ with
respect to the choice of $\sigma$.

To recognize that $S$ consists of a bundle of $(3,1)$-spaces in
$\mathbb{R}^{n+1,1}$, we just need to verify that it is a
non-degenerate $\mathrm {rank}\,\,4$ bundle, for $\sigma\in\Gamma
(S)$ is light-like. The fact that $\Lambda$ is an immersion gives
$\mathrm{rank}\,\Lambda ^{(1)}=3$, whereas, by equation
\eqref{eq:8765redcvbnml987654redfghjkihgfdew3456789o3456knjhghg},
$\mathrm{rank}\,\langle\sum _{i} d_{e_{i}}d_{e_{i}}\sigma\rangle=1$.
Thus $\mathrm {rank}\,S=4$. On the other hand, given
$i,j\in\{1,2\}$,
\begin{eqnarray*}
(d_{e_{j}}\sigma,d_{e_{i}}d_{e_{i}}\sigma)&=&d_{e_{i}}(d_{e_{j}}\sigma,d_{e_{i}}\sigma)-(d_{e_{i}}d_{e_{j}}\sigma,d_{e_{i}}\sigma)\\
&=&d_{e_{i}}\delta_{i,j}-(d_{e_{j}}d_{e_{i}}\sigma,d_{e_{i}}\sigma)\\&=&-\frac{1}{2}\,\,d_{e_{j}}(d_{e_{i}}\sigma,d_{e_{i}}\sigma)\\&=&0,\end{eqnarray*}
and, therefore,
$(d_{e_{j}}\sigma,\sum_{k}d_{e_{k}}d_{e_{k}}\sigma)=0$. It follows
that the matrix of the metric on $S$ in the frame $\sigma
,d_{e_{1}}\sigma ,d_{e_{2}}\sigma ,\sum _{i}
d_{e_{i}}d_{e_{i}}\sigma$ is
$$\begin{pmatrix}
0&0&0&-2\\0&1&0&0\\0&0&1&0\\-2&0&0&a\end{pmatrix}$$ for some
$a\in\R$, whose determinant is $-4$, which shows that $S$ is
non-degenerate.

Lastly, and explicitly, we have $\Lambda ^{(1)}\subset S$, which
completes this verification.

\begin{rem}
Observe from the previous verification that the definition of $S$
is, furthermore, independent of conformal changes of the metric
$g_{\sigma}$. In particular, given a holomorphic chart $z=x+iy$ of
$(M,\mathcal{C}_{\Lambda})$, we have
$$S=\langle\sigma,\sigma_{x},\sigma_{y},\sigma
_{xx}+\sigma _{yy}\rangle.$$ Note that $\sigma _{xx}+\sigma _{yy}
=4\sigma _{z\bar{z}}.$ It follows that the complexification of the
central sphere congruence of $\Lambda$ is given by $$S=\langle
\sigma ,\sigma _{z},\sigma _{\bar{z}},\sigma _{z\bar{z}}\rangle
\subset(\underline{\R}^{n+1,1})^{\C}.$$
\end{rem}

Although the mean curvature vector is not a conformal invariant (cf.
Remark \ref{mcvnotconfinv}), under a conformal change of the metric
on $S^{n}$, it changes in the same way for the surface $M$ immersed
in $S^{n}$ by $\Lambda$ and each $2$-sphere
$\mathbb{P}(\mathcal{L}\cap S_{p})$, with $p\in M$, canonically
immersed in $S^{n}$. The condition $\sigma_{z\bar{z}}\in\Gamma(S)$
establishes that, at each point, the surface $\Lambda$ and the
sphere $\mathbb{P}(\mathcal{L}\cap S_{p})$ share the same mean
curvature vector. $S$ is, in this way, a congruence of
\textit{osculating} spheres to the surface $\Lambda$: for each $p\in
M$, $\mathbb{P}(\mathcal{L}\cap S_{p})$ is the unique $2$-sphere in
$S^{n}$ tangent to $\Lambda(M)$ at $\Lambda(p)$ whose mean curvature
vector at $\Lambda(p)$ is the mean curvature vector of $M$ at $p$.

The central sphere congruence of $\Lambda$ defines naturally a map
$$S:M\rightarrow \mathcal {G}:=Gr_{(3,1)}(\R^{n+1,1})$$ into the connected
component of the Grassmannian of order $4$ of $\R^{n+1,1}$
constituted by the subspaces with signature $(3,1)$. Throughout this
text, let $\pi _{S}$ and $\pi _{S^{\perp}}$ denote the orthogonal
projections of $\underline{\mathbb{R}}^{n+1,1}$ onto $S$ and
$S^{\perp}$, respectively. Let $T$ and $T^{\perp}$ be the bundles
over $\mathcal{G}$ whose fibres at each $V\in\mathcal{G}$ are,
respectively, $V$ and $V^{\perp}$. The tangent bundle to $\mathcal
{G}$ at $T$ is identified with the bundle
$\mathrm{Hom}(T,T^{\perp})$ via the bundle isomorphism given by
\begin{equation}\label{eq:identificTGwithHom}
X\mapsto (\rho\mapsto \pi_{T^{\perp}}(d_{X}\rho)),
\end{equation}
for $\pi_{T^{\perp}}$ the orthogonal projection of $\R^{n+1,1}$ onto
$V^{\perp}$. This is indeed well-defined, as, given $\rho \in\Gamma
(T)$, $(\rho\mapsto \pi _{T^{\perp}}(d_{X}\rho ))$ is tensorial.
This identification makes $\mathcal {G}$ into a pseudo-Riemannian
manifold.\footnote{This identification is the particular case
$\mathcal{G}=Gr_{(3,1)}(\R^{n+1,1})$ of a standard procedure to
induce a pseudo-Riemannian structure in a general Grassmannian
$\mathcal{G}=Gr_{(r,s)}(\R^{p,q})$.} Under this identification, the
pull-back bundle $S^{*}T\mathcal {G}$ is identified with the
pull-back by $S$ of $\mathrm {Hom}(T,T^{\perp})$,$$S^{*}T\mathcal
{G}\cong\mathrm{Hom}(S,S^{\perp}).$$

\begin{prop}\label{conformalcsc}
The central sphere congruence of $\Lambda$ is conformal with respect
to the conformal structure $\mathcal{C}_{\Lambda}$ induced in $M$ by
$\Lambda$.
\end{prop}

Before we proceed to the proof of the proposition, fix $\sigma$ a
never-zero section of $\Lambda$ and $z$ a holomorphic chart of
$(M,\mathcal{C}_{\Lambda})$ and let us spend some moments
contemplating the orthogonality relations of the frame
$\{\sigma,\sigma _{z},\sigma _{\bar{z}},\sigma _{z\bar{z}}\}$ of
$S$, beyond the very well-known $(\sigma,\sigma)=0$ and subsequent
$$(\sigma , \sigma _{z})=0=(\sigma , \sigma _{\bar{z}}).$$
The conformality of $\sigma :(M,\mathcal{C}_{\Lambda})\rightarrow
\R^{n+1,1}$ gives $$(\sigma _{z},\sigma _{z})=0=(\sigma
_{\bar{z}},\sigma _{\bar{z}}),$$and differentiation shows then that
$$(\sigma _{z},\sigma _{z\bar{z}})=0=(\sigma _{\bar{z}},\sigma
_{z\bar{z}}).$$The fact that $g_{\sigma}$ is conformally equivalent
to $g_{z}$ ensures that
\begin{equation}\label{eq:sigmazsigmazbar}
(\sigma _{z},\sigma
_{\bar{z}})=\frac{1}{4}\,(g_{\sigma}(\delta_{x},\delta_{x})+g_{\sigma}(\delta_{y},\delta_{y}))
\end{equation}
is never-zero, whilst differentiation of $(\sigma , \sigma _{z})=0$
shows that
$$(\sigma ,\sigma _{z\bar{z}})=-(\sigma _{z},\sigma
_{\bar{z}}).$$ As for $(\sigma _{z\bar{z}},\sigma _{z\bar{z}}),$
nothing can be ensured in a general situation. As a final remark, we
introduce a specific choice of a never-zero section of $\Lambda$, in
relation to the chart $z$, that provides an extra specificity on the
orthogonality relations presented above and that will, for that
reason, be useful in some moments in the future. As we know, each
never-zero section of $\Lambda$ induces in $M$ a metric conformally
equivalent to the metric $g_{z}$ induced by $z$. If we make a choice
of one of the two components of the light-cone, say
$\mathcal{L}^{+}$, then there is a unique section $\sigma
^{z}:M\rightarrow \mathcal{L}^{+}$ of $\Lambda$ whose induced metric
coincides with the metric induced by $z$,$$g_{\sigma^{z}}=g_{z}.$$
We refer to $\sigma ^{z}$ as the $\textit{normalized}$
\textit{section} of $\Lambda$ with respect to $z$. According to
\eqref{eq:sigmazsigmazbar}, $(\sigma^{z}_{z},\sigma^{z}_{\bar{z}})$
is constantly equal to $\frac{1}{2},$
$$(\sigma^{z}_{z},\sigma^{z}_{\bar{z}})=\frac{1}{2}.$$
In particular,
$(\sigma^{z}_{z},\sigma^{z}_{\bar{z}})_{z}=0=(\sigma^{z}_{z},\sigma^{z}_{\bar{z}})_{\bar{z}}$
or, equivalently,
$$(\sigma_{z}^{z},\sigma_{\bar{z}\bar{z}}^{z})=0.$$

Now we proceed to the proof of Proposition \ref{conformalcsc}.
\begin{proof}Fix a holomorphic chart $z=x+iy$ of
$(M,\mathcal{C}_{\Lambda})$. The proof will consist of showing that
$(S_{z},S_{z})=0$, or, equivalently, that $\mathrm{tr}(S_{z}^{t}\,
S_{z})=0$.

Fix a never-zero section $\sigma$ of $\Lambda$. According to the
orthogonality relations of the frame $\sigma,\sigma _{z},\sigma
_{\bar{z}},\sigma _{z\bar{z}}$ of $S$, we have $\langle\sigma,\sigma
_{\bar{z}}\rangle ^{\perp}\cap S=\langle\sigma,\sigma
_{\bar{z}}\rangle$. On the other hand,
\begin{equation}\label{eq:34567gfs4335678wsqw4d4330912ertyuihg5tgt654321111211112?}
S_{z}(\sigma)=\pi_{S^{\perp}}(\sigma
_{z})=0=\pi_{S^{\perp}}(\sigma_{\bar{z}z})=S_{z}(\sigma_{\bar{z}}),
\end{equation}
and, therefore, $\langle\sigma,\sigma _{\bar{z}}\rangle\subset
\mathrm{ker}\, S_{z}$. It follows that
$\mathrm{Im}\,S_{z}^{t}\subset(\ker S_{z})^{\perp}\cap
S\subset\langle\sigma,\sigma _{\bar{z}}\rangle$ and, consequently,
by
\eqref{eq:34567gfs4335678wsqw4d4330912ertyuihg5tgt654321111211112?},
that $$S_{z}\,S_{z}^{t}=0,$$which completes the proof.
\end{proof}

It was Bryant \cite{bryant} who established the existence of a
congruence of $2$-spheres tangent to a Riemann surface isometrically
immersed in $S^{3}$ and with the same mean curvature as the surface
at each point. Bryant named it the conformal Gauss map. This
justifies the alternative terminology, after Bryant's paper
\cite{bryant}, of conformal Gauss map for the central sphere
congruence, although the central sphere congruence carries, not only
first order contact information, but second order as well.

\section{The normal bundle to the central sphere
congruence}\label{normalbundle}

\markboth{\tiny{A. C. QUINTINO}}{\tiny{CONSTRAINED WILLMORE
SURFACES}}

The bundle $S^{\perp}$, normal to the central sphere congruence of
$\Lambda$, can be identified with the normal bundle to $\Lambda$,
when regarded as a surface in a space-form, via an isometric
isomorphism of bundles with connections, as we present next,
following \cite{SD}.\newline

We provide $S$ and $S^{\perp}$ with the connections $\nabla ^{S}$
and $\nabla^{S^{\perp}}$, respectively, defined by orthogonal
projection of the trivial flat connection $d$ on
$\underline{\R}^{n+1,1}$,
$$\nabla
^{S}:=\pi _{S}\circ d\vert _{\Gamma (S)},\,\,\,\,\,\,\nabla
^{S^{\perp}}:=\pi _{S^{\perp}}\circ d\vert _{\Gamma
(S^{\perp})},$$which we immediately verify to be metric connections.

Fix a non-zero $v_{\infty}$ in $\mathbb{R}^{n+1,1}$ and consider the
surface $\sigma _{\infty}:M\rightarrow S_{v_{\infty}}$, in the
space-form
 $S_{v_{\infty}}$, defined by $\Lambda$. For simplicity, we write
$g_{\infty}$, rather than $g_{\sigma _{\infty}}$, for the metric
induced in $M$ by $\sigma _{\infty}$, as well as $N_{\infty}$ for
the normal bundle to $\sigma _{\infty}$. Let $\Pi_{\infty}$  and
$\mathcal{H}_{\infty}$ denote, respectively, the second fundamental
form and the mean curvature vector of $\sigma _{\infty}$. We shall
keep this notation throughout this text.

\begin{prop}\label{identofnormals}
The normal bundle $N_{\infty}$ to $\sigma _{\infty}$ is identified
with the bundle $S^{\perp}$ normal to the central sphere congruence
of $\Lambda$,
$$N_{\infty}\cong S^{\perp},$$ by the isomorphism $$\mathcal {Q}:N_{\infty}\rightarrow
S^{\perp}$$ of bundles provided with a metric and a connection
defined by
$$\xi\mapsto \xi+(\xi,\mathcal {H_{\infty}})\sigma _{\infty}.$$
\end{prop}
\begin{proof}
The pull-back bundle by $\sigma_{\infty}$ of the tangent bundle
$TS_{v_{\infty}}$ consists of the orthogonal complement in
$\underline{\R}^{n+1,1}$ of the non-degenerate bundle
$\langle\sigma_{\infty},v_{\infty}\rangle$ (cf. section
\ref{subsec:hyper}). Let $\pi _{N_{\infty}}$ denote the orthogonal
projection of
$$\underline {\mathbb{R}}^{n+1,1}=d\sigma _{\infty}(TM)\oplus N
_{\infty}\oplus \langle v_{\infty},\sigma _{\infty}\rangle $$ onto
$N_{\infty}$. Since the metric in $S_{v_{\infty}}$ is the one
inherited from $\mathbb{R}^{n+1,1}$, $\Pi _{\infty}$ is simply given
by
$$\Pi _{\infty}(X,Y)=\pi _{N_{\infty}}(d_{X}d_{Y}\sigma
_{\infty}),$$ for $X,Y\in \Gamma (TM)$, so that, for $\xi\in\Gamma
(N_{\infty})$ and $(e_{i})_{i}$ an orthonormal frame of
$(TM,g_{\infty})$, $(\xi,\sum _{i} d_{e_{i}}d_{e_{i}}\sigma
_{\infty})=2(\xi,\mathcal {H}_{\infty})$ and, therefore, by
\eqref{eq:8765redcvbnml987654redfghjkihgfdew3456789o3456knjhghg},
$(\xi+(\xi,\mathcal {H_{\infty}})\sigma _{\infty},\sum _{i}
d_{e_{i}}d_{e_{i}}\sigma _{\infty})=0$. Together with the fact that
$$N_{\infty}\subset \sigma_{\infty}^{*}TS_{v_{\infty}}=\langle\sigma _{\infty},v_{\infty}\rangle
^{\perp},$$  this shows that $\xi+(\xi,\mathcal {H_{\infty}})\sigma
_{\infty}$ is, in fact, a section of $S^{\perp}$.

Clearly, $\mathcal {Q}$ is isometric, and, therefore, injective, as
$N_{\infty}$ is non-degenerate. Now
$$\mathrm {rank}\, N_{\infty}=\mathrm {rank}\,\sigma_{\infty}^{*}TS_{v_{\infty}}-\mathrm {rank}\,d\sigma
_{\infty}(TM)=n-2=\mathrm {rank}\,S^{\perp}$$ shows that $\mathcal
{Q}$ is an isometric isomorphism. Furthermore, for $\xi\in\Gamma
(N_{\infty})$,
$$ \nabla
^{S^{\perp}}(\mathcal {Q}(\xi))=\pi
_{S^{\perp}}(d\xi)+d(\xi,\mathcal {H_{\infty}})\pi
_{S^{\perp}}\sigma_{\infty}+(\xi,\mathcal {H_{\infty}})\pi
_{S^{\perp}}(d\sigma _{\infty})=\pi _{S^{\perp}}(d\xi),$$ whilst
\begin{eqnarray*}
\mathcal {Q}(\nabla ^{N_{\infty}}\xi)=\pi _{N_{\infty}}(d\xi)+(\pi
_{N_{\infty}}(d\xi),\mathcal {H_{\infty}})\sigma _{\infty}\in\Gamma
(S^{\perp}).
\end{eqnarray*}
To show that $\mathcal {Q}$ preserves connections, we just need to
verify that $d\xi -\pi _{N_{\infty}}(d\xi)\in\Gamma (S)$. That is
immediate: for $\xi \in \Gamma (N_{\infty})\subset \Gamma
(\langle\sigma _{\infty},v_{\infty}\rangle ^{\perp})$, $d\xi$ is
still a section of $\langle\sigma _{\infty},v_{\infty}\rangle
^{\perp}$,
$$(d\xi,\sigma_{\infty})=(d\xi,\sigma_{\infty})+(\xi,d\sigma_{\infty})=0=(d\xi,v_{\infty})+(\xi,dv_{\infty})=(d\xi,v_{\infty});$$
and, therefore,
$$d\xi -\pi _{N_{\infty}}(d\xi)=\pi _{d\sigma
_{\infty}(TM)}(d\xi).$$
\end{proof}

\section{The Gauss-Ricci and Codazzi equations}\label{subsec:grc}
\subsection{The exterior power $\wedge^{2}\R^{n+1,1}$ et al.: a few
utilities}\label{extalgebra}

\markboth{\tiny{A. C. QUINTINO}}{\tiny{CONSTRAINED WILLMORE
SURFACES}}

This section consists of a collection of useful, well-known facts
involving exterior products. \newline

The space $\wedge ^{2}\mathbb{R}^{n+1,1}$ can be identified with the
orthogonal algebra $o(\mathbb{R}^{n+1,1})$,
\begin{equation}\label{eq:wedgeRn+11andorthogonalalgebar}
\wedge^{2}\R^{n+1,1}\cong o(\mathbb{R}^{n+1,1}),
\end{equation}
via
\begin{equation}\label{eq:isomwedge}
w\mapsto (v_{1}\wedge v_{2})(w):=(w, v_{1})v_{2}-(w,v_{2})v_{1},
\end{equation}
which assigns to $v_{1}\wedge v_{2}$ a skew-symmetric
transformation. We shall consider this identification throughout
this work.

Under the identification \eqref{eq:wedgeRn+11andorthogonalalgebar}
defined by \eqref{eq:isomwedge}, given a non-degenerate subbundle
$V$ of $\underline{\mathbb{R}}^{n+1,1}$, we have
$$\Gamma (V\wedge V^{\perp})=\{\xi \in \Gamma
(o(\underline{\mathbb{R}}^{n+1,1})): \xi (V)\subset V^{\perp}, \xi
(V^{\perp})\subset V\},$$as well as $$\Gamma (\wedge ^{2}V\oplus
\wedge^{2}V^{\perp})=\{\xi \in \Gamma
(o(\underline{\mathbb{R}}^{n+1,1})): \xi (V)\subset V, \xi
(V^{\perp})\subset V^{\perp}\},$$and
\begin{equation}\label{eq:PoplusM}
o(\underline{\mathbb{R}}^{n+1,1})=\wedge^{2}V\oplus
\wedge^{2}V^{\perp}\oplus V\wedge V^{\perp},
\end{equation}
for the trivial bundle $o(\underline{\mathbb{R}}^{n+1,1}):=M\times
o(\mathbb{R}^{n+1,1})$. It is immediate and very useful to observe
that, given $\xi\in\Gamma(V\wedge V^{\perp})$,
\begin{equation}\label{eq:transposeVwedgeVperp}
\xi_{\vert_{V^{\perp}}}=-(\xi_{\vert_{V}})^{t},
\end{equation}
for $(\xi_{\vert_{V}})^{t}$ the transpose of
$\xi_{\vert_{V}}\in\Gamma(\mathrm{Hom}(V,V^{\perp}))$ with respect
to the metric on $\underline{\R}^{n+1,1}$.

Given $V$ and $W$ subbundles of $\underline{\R}^{n+1,1}$, provided
with connections $\nabla^{V}$ and $\nabla^{W}$, respectively, the
bundle $V\wedge W$ is provided with the metric induced from the one
on $\mathrm{End}(\underline{\R}^{n+1,1})$ and, canonically, with the
connection given by
$$\nabla(v\wedge w):=(\nabla^{V} v)\wedge w+v\wedge(\nabla^{W}
w),$$for $v\in\Gamma(V),w\in\Gamma(W)$.  In the particular case
$V=\underline{\R}^{n+1,1}=W$ and $\nabla^{V}=\nabla^{W}$ is a metric
connection, $\nabla$ coincides with the connection canonically
induced by $\nabla^{V}$ in $\mathrm{End}(\underline{\R}^{n+1,1})$,
over sections of the trivial bundle
$M\times\wedge^{2}\R^{n+1,1}=:\wedge^{2}\underline{\R}^{n+1,1}\cong
o(\underline{\mathbb{R}}^{n+1,1})$. If $\nabla^{V}$ is metric and
$V^{\perp}$ is provided with a connection, the correspondence
\begin{equation}\label{eq:VwedgeVperpHomVVperp}
\eta\mapsto\eta\vert_{V}
\end{equation}
defines a bundle isomorphism $V\wedge
V^{\perp}\rightarrow\mathrm{Hom}(V,V^{\perp})$ preserving metrics
and connections, providing an identification
$$V\wedge V^{\perp}\cong\mathrm{Hom}(V,V^{\perp})$$ of bundles
provided with a metric and a connection. In fact, given
$v,v^{*}\in\Gamma(V)$,
$v^{\perp},v^{\perp}_{*}\in\Gamma(V^{\perp})$, $\eta=v\wedge
v^{\perp}$ and $\eta^{*}=v^{*}\wedge v^{\perp}_{*}$,
\begin{eqnarray*}
(\nabla \eta\vert_{V})v^{*}&=&\nabla^{V^{\perp}}(\eta
v^{*})-\eta(\nabla^{V}v^{*})\\&=&\nabla^{V^{\perp}}(v,v^{*})v^{\perp}-(v,\nabla^{V}v^{*})v^{\perp}\\&=&
(v,v^{*})\nabla^{V^{\perp}}v^{\perp}+d(v,v^{*})v^{\perp}-(v,\nabla^{V}v^{*})v^{\perp}
\end{eqnarray*}
and, therefore, as $\nabla^{V}$ is metric,
$$
(\nabla
\eta\vert_{V})v^{*}=(v,v^{*})\nabla^{V^{\perp}}v^{\perp}+(\nabla^{V}v,v^{*})v^{\perp}=
(\nabla \eta)v^{*};$$whilst, on the other hand,
$(\eta^{*}\vert_{V})^{t}\circ
\eta\vert_{V^{\perp}}=(\eta{*})^{t}\vert_{V^{\perp}}\circ
\eta\vert_{V^{\perp}}=0$ and, therefore,
$$(\eta\vert_{V},\eta^{*}\vert_{V})=\mathrm{tr}((\eta^{*}\vert_{V})^{t}\circ\eta\vert_{V})=\mathrm{tr}((
\eta^{*})^{t}\circ\eta)=(\eta,\eta^{*}).$$

It will be useful to note that, as a straightforward computation
shows, given $a,b\in\R^{n+1,1}$ and $T\in o(\mathbb{R}^{n+1,1})$,
\begin{equation}\label{eq:wegdebrack}
[T,a\wedge b]=(Ta)\wedge b+a\wedge (Tb),
\end{equation}
for the Lie bracket in $o(\mathbb{R}^{n+1,1})$; as well as, for
$T\in O(\mathbb{R}^{n+1,1})$,
\begin{equation}\label{eq:adjwedge}
\mathrm{Ad}_{T}(a\wedge b)=Ta\wedge Tb.
\end{equation}

Given a vector bundle $P$ over $M$ whose fibres are Lie algebras and
$\mu ,\eta \in\Omega ^{1}(P)$, we define a $2$-form $[\mu
\wedge\eta]\in \Omega ^{2}(P)$ by
$$ [\mu\wedge \eta] _{(X,Y)}:=[\mu _{X},\eta _{Y}]-[\mu
_{Y},\eta _{X}]$$ for $X,Y\in\Gamma (TM)$. Note that $$[\eta\wedge
\mu]=[\mu\wedge \eta]$$ and that $$[\mu\wedge \mu]_{(X,Y)}=2[\mu
_{X},\mu_{Y}].$$ In the case $P$ is a bundle of endomorphisms, we
define another $2$-form with values in $P$,
$\mu\wedge\eta\in\Omega^{2}(P)$, using composition of endomorphisms
to multiply coefficients in the exterior product:
$$\mu\wedge\eta _{(X,Y)}:=\mu_{X}\eta_{Y}-\mu_{Y}\eta_{X},$$for all
$X,Y$. Note that, in that case,
\begin{equation}\label{eq:[wedge]vswedge}
[\mu\wedge \eta]=\mu\wedge\eta+\eta\wedge\mu.
\end{equation}
Suppose $M$ is provided with a conformal structure $\mathcal{C}$.
Observe that
\begin{equation}\label{eq:starseliebracks}
[\mu\wedge *\,\eta]=-[*\,\mu\wedge\eta ].
\end{equation}
For this, fix a (locally) never-zero $Z\in\Gamma(T^{1,0}M)$ and
verify that $[\mu\wedge *\,\eta](Z,\overline{Z})=-[*\,\mu\wedge\eta
](Z,\overline {Z})$, or, equivalently, that $-[\mu\wedge (\eta
J)](Z,\overline{Z})=[(\mu J)\wedge\eta ](Z,\overline {Z})$, which
is, in fact, an immediate consequence of the fact that $Z$ and
$\overline{Z}$ are eigenvectors of $J$ associated to the eigenvalues
$i$ and $-i$, respectively. In particular,
\begin{equation}\label{eq:bracstarstarbrac}
[*\mu\wedge *\,\eta]=[\,\mu\wedge\eta ].
\end{equation}
Note that, if $\mu$ and $\eta$ are both either $(1,0)$-forms or
$(0,1)$-forms, then $[\mu\wedge \eta]$ vanishes:
\begin{equation}\label{eq:ijwedgeij0}
[\mu ^{1,0}\wedge \eta ^{1,0}]=0=[\mu ^{0,1}\wedge \eta ^{0,1}],
\end{equation}
for all $\mu$ and $\eta$. Note also that, given
$\xi_{1},\xi_{2}\in\Omega^{1}(o(\underline{\mathbb{R}}^{n+1,1}))$
and $T\in\mathrm{End}(\R^{n+1,1})$,
\begin{equation}\label{eq:Adwedge}
[\,\mathrm{Ad}_{T}\,\xi_{1}\,\wedge\,\mathrm{Ad}_{T}\,\xi_{2}\,]=\mathrm{Ad}_{T}\,[\,\xi_{1}\wedge\xi_{2}\,].
\end{equation}
Lastly, suppose that $\nabla ^{1}$ and $\nabla ^{2}$ are connections
on $\underline{\R}^{n+1,1}$ related by
$$\nabla ^{1}=\nabla ^{2}+A,$$ for some $A\in\Omega
^{1}(\mathrm{End}(\underline{\R}^{n+1,1}))$. The respective
curvature tensors, $R^{\nabla ^{1}}$ and $R^{\nabla ^{2}}$, are
related by
\begin{equation}\label{eq:curvtens}
R^{\nabla ^{1}}=R^{\nabla ^{2}}+d^{\nabla
^{2}}A+\frac{1}{2}\,[A\wedge A],
\end{equation}
whilst the corresponding exterior derivatives, $d^{\nabla ^{1}}$ and
$d^{\nabla ^{2}}$, relate by
\begin{equation}\label{eq:ext}
d^{\nabla ^{1}}\xi=d^{\nabla ^{2}}\xi+[A\wedge \xi],
\end{equation}
for $\xi\in\Omega ^{1}(\mathrm{End}(\underline{\R}^{n+1,1}))$. As a
final remark, note that, in the case $\nabla^{2}$ is a metric
connection, the connection $\nabla^{1}$ is metric if and only if $A$
is skew-symmetric.

\subsection{The Gauss-Ricci and Codazzi equations}
The central sphere congruence $S:M\rightarrow
Gr_{(3,1)}(\R^{n+1,1})$ of a surface in $n$-space defines a
decomposition $d=\mathcal{D}+\mathcal{N}$ of the trivial flat
connection on $\underline{\R}^{n+1,1}$ into the sum of a connection
$\mathcal{D}$, with respect to which $S$ and $S^{\perp}$ are
parallel, and a $1$-form $\mathcal{N}$ with values in $S\wedge
S^{\perp}$. Explicitly, $\mathcal{D}
:=\nabla^{S}+\nabla^{S^{\perp}}$, for $\nabla^{S}$ and
$\nabla^{S^{\perp}}$ the connections on $S$ and $S^{\perp}$,
respectively, defined by orthogonal projection of $d$, and
$\mathcal{N} :=d-\mathcal{D}$. The flatness of $d$ encodes two
structure equations on $\mathcal{D}$ and $\mathcal{N}$.
\newline

Define a connection $\mathcal {D}$ on
$\underline{\mathbb{R}}^{n+1,1}$ by
$$\mathcal {D}:=\nabla ^{S}\circ \pi _{S}+
\nabla ^{S^{\perp}}\circ \pi _{S^{\perp}}.$$ For simplicity, and
only temporarily, denote $\pi _{S}$ and $\pi _{S^{\perp}}$ by
$(\,)^{T}$ and $(\,)^{\perp}$, respectively. Given
$\mu\in\Gamma(\underline{\mathbb{R}}^{n+1,1})$,
$$d\mu=(d\mu ^{T})^{T}+(d\mu ^{T})^{\perp}+(d\mu ^{\perp})^{T}+(d\mu
^{\perp})^{\perp}=\mathcal{D}\mu+(d\mu ^{T})^{\perp}+(d\mu
^{\perp})^{T}$$ and, given
$\eta\in\Gamma(\underline{\mathbb{R}}^{n+1,1})$, $$((d\mu
^{T})^{\perp},\eta)=(d\mu
^{T},\eta^{\perp})=-(\mu^{T},d\eta^{\perp})=-(\mu,(d\eta^{\perp})^{T}).$$
It is then clear that, for
$\mu,\eta\in\Gamma(\underline{\mathbb{R}}^{n+1,1})$,
$$d(\mu ,\eta )=(d\mu ,\eta )+(\mu ,d\eta )=(\mathcal {D}\mu
,\eta)+(\mu ,\mathcal {D}\eta);$$ $\mathcal {D}$ is a metric
connection. Thus
\begin{equation}\label{eq:dcurlyDcurlyN}
d=\mathcal {D}+\mathcal {N}
\end{equation}
defines a $1$-form $\mathcal
{N}\in\Omega^{1}(o(\underline{\mathbb{R}}^{n+1,1}))$. In fact,
$$\mathcal {N}=\pi _{S^{\perp}}\circ d \circ \pi _{S}+\pi _{S}\circ
d\circ \pi _{S^{\perp}}\in\Omega ^{1}(S\wedge S^{\perp}).$$ We may,
alternatively, use, specifically, the notations $\mathcal{D}_{S}$
and $\mathcal{N}_{S}$ for, respectively, $\mathcal{D}$ and
$\mathcal{N}$. It is very simple but very useful to note that
\begin{equation}\label{eq:calNLambda0}
\mathcal {N}\Lambda=0.
\end{equation}
Indeed, given $\sigma\in\Gamma(\Lambda)$ never-zero,
$\mathcal{N}\sigma=\pi _{S^{\perp}}\circ d\sigma$ and
$d\sigma\in\Omega^{1}(S)$. By the skew-symmetry of $\mathcal{N}$, it
follows, in particular, that $\mathrm{Im}\,\mathcal{N}\subset
\Lambda^{\perp}$ and, consequently, that
$\mathcal{N}S^{\perp}\subset S\cap\Lambda^{\perp}=\Lambda^{(1)}$.
Hence
\begin{equation}\label{eq:caklNLambdaperpSperp}
\mathcal{N}\in\Omega^{1}(\Lambda^{(1)}\wedge S^{\perp}).
\end{equation}

The flatness of $d$, characterized by
$$0=R^{\mathcal{D}}+d^{\mathcal{D}}\mathcal{N}+\frac{1}{2}\,[\mathcal{N}\wedge \mathcal{N}],$$
encodes two structure equations, as follows. The
$\mathcal{D}$-parallelness of $S$ and $S^{\perp}$ establishes
$\mathcal{D}(\Gamma (S\wedge S^{\perp}))\subset \Omega ^{1} (S\wedge
S^{\perp})$ and, therefore,
$d^{\mathcal{D}}\mathcal{N}\in\Omega^{2}(S\wedge S^{\perp})$.
Together with the fact that $\mathcal{D}$ is metric, it establishes,
on the other hand, $R^{\mathcal{D}}\in\Omega^{2}(\wedge^{2}S\oplus
\wedge^{2}S^{\perp})$. By equation \eqref{eq:wegdebrack}, we verify
that $[\mathcal{N}\wedge \mathcal{N}]\in\Omega
^{2}(\wedge^{2}S\oplus \wedge^{2}S^{\perp})$. According to the
decomposition \eqref{eq:PoplusM}, it follows then that:

\begin{prop}(Gauss-Ricci equation)
$$R^{\mathcal{D}}+\frac{1}{2}\,[\mathcal{N}\wedge \mathcal{N}]=0;$$
\end{prop}
and
\begin{prop}(Codazzi equation)
$$d^{\mathcal{D}}\mathcal{N}=0.$$
\end{prop}

\chapter{Surfaces under change of flat metric
connection}\label{subsec:tranfsfmc}

\markboth{\tiny{A. C. QUINTINO}}{\tiny{CONSTRAINED WILLMORE
SURFACES}}

In many occasions throughout this work, we use an interpretation of
loop group theory by F. Burstall and D. Calderbank
\cite{burstall+calderbank} and produce transformations of surfaces
by the action of loops of flat metric connections. Specifically, by
replacing the trivial flat connection by another flat metric
connection $\tilde{d}$ on $\underline{\R}^{n+1,1}$, we transform (in
certain cases) a surface $\Lambda\subset \underline{\R}^{n+1,1}$
into a $\tilde{d}$-\textit{surface} $\tilde{\Lambda}$, or,
equivalently, into another surface $\tilde{\phi}\Lambda$, defined,
up to a M\"{o}bius transformation,\footnote{At some point, we shall
omit the indication "up to a M\"{o}bius transformation" and assume a
M\"{o}bius geometry point of view.} for
$\tilde{\phi}:(\underline{\R}^{n+1,1},\tilde{d})\rightarrow(\underline{\R}^{n+1,1},d)$
an isomorphism of bundles provided with a metric and a connection.
Many will be the examples in this work of such transformations
preserving the geometrical aspects of a class, i.e., establishing
symmetries of integrable systems. This tiny chapter is merely
introductory of the concept of $\tilde{d}$-surface.
\newline

Recall that the flatness of a bundle (over $M$) ensures the
existence of a local frame (defined on a simply connected component
of $M$) made up of parallel sections. Recall as well that, if the
bundle is also provided with a metric with respect to which the
connection is a metric connection, then there exists an orthonormal
local frame constituted by parallel sections.

Let $\tilde{d}$ be a flat metric connection on
$\underline{\R}^{n+1,1}$. Let $\mathcal{B}=(e_{i})_{i}$ and
$\tilde{\mathcal{B}}=(\tilde{e}_{i})_{i}$ be orthonormal local
frames of $\underline{\R}^{n+1,1}$, parallel with respect to $d$ and
$\tilde{d}$, respectively, $\tilde{d}\,\tilde{e}_{i}=0=d\,e_{i},$
for $i=1,...,n+2$; with $e_{n+2}$ and $\tilde{e}_{n+2}$ time-like.
Define an isometry
$\phi_{\tilde{\mathcal{B}}\,\mathcal{B}}:(\underline{\R}^{n+1,1},\tilde{d})\rightarrow
(\underline{\R}^{n+1,1},d)$ of bundles, preserving connections,
$$\phi_{\tilde{\mathcal{B}}\,\mathcal{B}}
\circ\tilde{d}=d\circ\phi_{\tilde{\mathcal{B}}\,\mathcal{B}},$$by
setting
$\phi_{\tilde{\mathcal{B}}\,\mathcal{B}}(\tilde{e}_{i})=e_{i}$, for
$i=1,...,n+2$. Observe that, although
$\phi_{\tilde{\mathcal{B}}\,\mathcal{B}}$ depends on the choice of
the frames $\mathcal{B}$ and $\tilde{\mathcal{B}}$, it is uniquely
determined up to a M\"{o}bius transformation. For that, first
suppose $\mathcal{B}'=(e'_{i})_{i}$ is another $d$-parallel
orthonormal local frame of $\underline{\R}^{n+1,1}$, with $e'_{n+2}$
time-like, define $T\in\Gamma (O(\underline {\R}^{n+1,1}))$ by
$T(e_{i})=e'_{i}$, for all $i$, and note that
$\phi_{\tilde{\mathcal{B}}\,\mathcal{B}'}=T\phi_{\tilde{\mathcal{B}}\,\mathcal{B}}$
and that $T$ is constant:
$$T\circ d=\phi_{\tilde{\mathcal{B}}\,\mathcal{B}'}
\phi_{\tilde{\mathcal{B}}\,\mathcal{B}}^{-1}\circ
d=\phi_{\tilde{\mathcal{B}}\,\mathcal{B}'}\circ
\tilde{d}\circ\phi_{\tilde{\mathcal{B}}\,\mathcal{B}}^{-1}=d\circ\phi_{\tilde{\mathcal{B}}\,\mathcal{B}'}\,
\phi_{\tilde{\mathcal{B}}\,\mathcal{B}}^{-1}=d\circ T.$$A similar
argument shows the independence, up to a M\"{o}bius transformation,
of $\phi_{\tilde{\mathcal{B}}\,\mathcal{B}}$ with respect to the
choice of the frame $\tilde{\mathcal{B}}$. Observe, on the other
hand, that any isomorphism
$\phi:(\underline{\R}^{n+1},\tilde{d})\rightarrow(\underline{\R}^{n+1},d)$
of bundles provided with a metric and a connection is of the form
$\phi_{\tilde{\mathcal{B}}\mathcal{B}}$ for some $\mathcal{B}$ and
$\tilde{\mathcal{B}}$: fixing an orthonormal $\tilde{d}$-parallel
local frame $\tilde{\mathcal{B}}=(\tilde{e}_{i})_{i}$ of
$\underline{\R}^{n+1,1}$ and setting $\mathcal{B}:=(\phi
(\tilde{e}_{i}))_{i}$, we define an orthonormal $d$-parallel local
frame of $\underline{\R}^{n+1,1}$ such that
$\phi=\phi_{\tilde{\mathcal{B}}\mathcal{B}}$.

Throughout this text, given a vector bundle $P$, provided with a
metric, and connections $\nabla$ and $\nabla'$ on $P$, by
isomorphism $(P,\nabla)\rightarrow (P,\nabla')$ shall be understood
isomorphism of bundles provided with a metric and a connection.

\begin{defn}
Given $V\subset\underline{\R}^{n+1,1}$ and an isomorphism
$$\phi_{\tilde{d}}:(\underline{\R}^{n+1,1},\tilde{d})\rightarrow
(\underline{\R}^{n+1,1},d),$$ the bundle
$\tilde{V}:=\phi_{\tilde{d}}\,V$ is said to be \emph{the
transformation of $V$ defined, up to a M\"{o}bius transformation, by
the flat metric connection $\tilde{d}$}.
\end{defn}

Henceforth, we shall omit the indication "up to a M\"{o}bius
transformation" and assume a M\"{o}bius geometry point of view.

\begin{rem}
Obviously, given $\nabla$ and $\nabla'$ connections on
$\underline{\R}^{n+1,1}$ and $\phi$ an endomorphism of
$\underline{\R}^{n+1,1}$, the hypothesis $\phi\circ\nabla\circ
\phi^{-1}=d=\phi\circ\nabla'\circ \phi^{-1}$ forces
$\nabla=\nabla'$. In particular, given a flat metric connection
$\hat{d}\neq\tilde{d}$ on $\underline{\R}^{n+1,1}$ and isomorphisms
$\phi_{\hat{d}}:(\underline{\R}^{n+1,1},\hat{d})\rightarrow
(\underline{\R}^{n+1,1},d)$ and
$\phi_{\tilde{d}}:(\underline{\R}^{n+1,1},\tilde{d})\rightarrow
(\underline{\R}^{n+1,1},d)$,
$$\phi_{\hat{d}}\neq\phi_{\tilde{d}}.$$
Of course, this does not exclude the possibility of, given a
subbundle $V$ of $\underline{\R}^{n+1,1}$, the transformations of
$V$ defined by $\phi_{\tilde{d}}$ and $\phi_{\hat{d}}$ being the
same.
\end{rem}

Let us concentrate on the particular case $V=\Lambda,$ a null line
bundle, not necessarily defining an immersion into the projectivized
light-cone; and on its transformation
$$\tilde{\Lambda}:=\phi_{\tilde{d}}\,\Lambda,$$into another null
line subbundle of $\underline{\R}^{n+1,1}$.

\begin{defn}
We say that $\Lambda$ is a $\tilde{d}$-surface if
$\mathrm{rank}\,\Lambda^{(1)}_{\tilde{d}}=3$, for
$$\Lambda^{(1)}_{\tilde{d}}:=\langle\sigma,\tilde{d}_{e_{1}}\sigma,\tilde{d}_{e_{2}}\sigma\rangle,$$
defined independently of the choices of a never-zero
$\sigma\in\Gamma (\Lambda)$ and of a local frame $(e_{i})_{i}$ of
$TM$. In the particular case $\tilde{d}=d$, we shall, alternatively,
omit the reference to $\tilde{d}$.
\end{defn}

It is, perhaps, worth remarking that a $\tilde{d}$-surface is not
necessarily a surface.

The fact that $\phi_{\tilde{d}}$ preserves the connections
$\tilde{d}$ and $d$ establishes
\begin{equation}\label{eq:phi1}
(\phi_{\tilde{d}}\,\Lambda)^{(1)}=\phi_{\tilde{d}}\,(\Lambda^{(1)}_{\tilde{d}}),
\end{equation}
and, therefore,
$\mathrm{rank}\,\tilde{\Lambda}^{(1)}=\mathrm{rank}\,\Lambda^{(1)}_{\tilde{d}}$,
and, ultimately:
\begin{prop}
$\tilde{\Lambda}$ is a surface if and only if $\Lambda$ is a
$\tilde{d}$-surface.
\end{prop}

An alternative perspective on the transformation of $\Lambda$ into
$\phi_{\tilde{d}}\,\Lambda$, is, in this way, that of the
transformation
$$\Lambda\subset(\underline{\R}^{n+1,1},d)\mapsto\Lambda\subset(\underline{\R}^{n+1,1}, \tilde{d}),$$
consisting of the change of the trivial flat connection on
$\underline{\R}^{n+1,1}$ into the flat metric connection
$\tilde{d}$.

In what little is left in this section, we introduce a few concepts
on $\tilde{d}$-surfaces. Suppose $\Lambda$ is a $\tilde{d}$-surface.
In that case, given a never-zero section $\sigma$ of $\Lambda$, we
define a positive definite metric $g_{\sigma}^{\tilde{d}}$ by
$$g_{\sigma}^{\tilde{d}}(X,Y):=(\tilde{d}_{X}\sigma
,\tilde{d}_{Y}\sigma ),$$ for $X,Y\in\Gamma (TM)$. Indeed,
\begin{equation}\label{eq:metricaeconformalclasstilde}
g_{\sigma}^{\tilde{d}}=g_{\phi_{\tilde{d}}\sigma}\in
\mathcal{C}_{\tilde{\Lambda}}=:\mathcal{C}_{\Lambda}^{\tilde{d}}.
\end{equation}

\begin{defn}
We define the $\tilde{d}$-$\emph{central}$ $\emph{sphere}$
$\emph{congruence}$ of $\Lambda$ by
\begin{equation}\label{eq:csc}
S^{\tilde{d}}:=\langle \sigma ,\tilde{d}_{e_{1}}\sigma
,\tilde{d}_{e_{2}}\sigma ,\sum _{i}
\tilde{d}_{e_{i}}\tilde{d}_{e_{i}}\sigma\rangle=\phi_{\tilde{d}}^{-1}\,S_{\tilde{\Lambda}},
\end{equation}
independently of the choices of a never-zero $\sigma\in\Gamma
(\Lambda)$ and of a local orthonormal frame $(e_{i})_{i}$ of $TM$
with respect to $g_{\sigma} ^{\tilde{d}}$.
\end {defn}

The non-degeneracy of $S_{\tilde{\Lambda}}$ ensures that of
$S^{\tilde{d}}$. Let $\pi _{S^{\tilde{d}}}$ and $\pi
_{(S^{\tilde{d}})^{\perp}}$ be the orthogonal projections of
$$\underline{\R}^{n+1,1}=S^{\tilde{d}}\oplus\,
(S^{\tilde{d}})^{\perp}$$ onto $S^{\tilde{d}}$ and
$(S^{\tilde{d}})^{\perp}$, respectively. We define a connection
$\mathcal{D}^{\tilde{d}}$ on $\underline{\R}^{n+1,1}$ by
$$\mathcal{D}^{\tilde{d}}:=\pi _{S^{\tilde{d}}}\circ
\tilde{d}\circ\pi_{S^{\tilde{d}}}+\pi
_{(S^{\tilde{d}})^{\perp}}\circ
\tilde{d}\circ\pi_{(S^{\tilde{d}})^{\perp}}$$ and a $1$-form
$\mathcal{N}^{\tilde{d}}\in\Omega^{1}(\mathrm{End}(\underline{\R}^{n+1,1}))$
by $$\mathcal{N}^{\tilde{d}}:=\tilde{d}-\mathcal{D}^{\tilde{d}}.$$
Note that
\begin{equation}\label{eq:DeNtildevsDLambdatilde}
\mathcal{D}^{\tilde{d}}=\phi_{\tilde{d}}^{-1}\circ
\mathcal{D}_{\tilde{\Lambda}}\circ\phi_{\tilde{d}},\,\,\,\,\mathcal{N}^{\tilde{d}}=\phi_{\tilde{d}}^{-1}
\mathcal{N}_{\tilde{\Lambda}}\phi_{\tilde{d}}.
\end{equation}

\chapter{Willmore surfaces}\label{willmsurf}

\markboth{\tiny{A. C. QUINTINO}}{\tiny{CONSTRAINED WILLMORE
SURFACES}}

Among the classes of Riemannian submanifolds, there is that of
Willmore surfaces, named after T. Willmore \cite{willmore2} (1965),
although the topic was mentioned by W. Blaschke \cite{blaschke}
(1929) and by G. Thomsen \cite{thomsen} (1923). Early in the
nineteenth century, S. Germain \cite{germain1}, \cite{germain2}
studied elastic surfaces. On her pioneering analysis, she claimed
that the elastic force of a thin plate is proportional to its mean
curvature. Since then, the mean curvature remains a key concept in
theory of elasticity. In modern literature on the elasticity of
membranes (see, for example, \cite{landau+lifschitz} and
\cite{lipowsky}), a weighted sum of the total mean curvature, the
total squared mean curvature and the total Gaussian curvature is
considered the elastic energy of a membrane. By neglecting the total
mean curvature (by physical considerations) and having in
consideration that the total Gaussian curvature of compact
orientable Riemannian surfaces without boundary is a topological
invariant, T. Willmore defined the Willmore energy of a compact
oriented Riemannian surface, without boundary, isometrically
immersed in $\R^{3}$, to be $\mathcal{W}=\int H^{2}dA$. The Willmore
functional ``extends" to isometric immersions of compact oriented
Riemannian surfaces in Riemannian manifolds by means of half of the
total squared norm of the trace-free part of the second fundamental
form, which, in fact, amongst surfaces in $\R^{3}$, differs from
$\mathcal{W}$ by the total Gaussian curvature, but still shares then
the critical points with $\mathcal{W}$. Willmore surfaces are the
extremals of the Willmore functional. W. Blaschke \cite{blaschke}
established the M\"{o}bius invariance of the Willmore energy of a
surface in spherical $3$-space. B.-Y. Chen \cite{chen} generalized
it to surfaces in constant curvature Riemnannian manifolds. We
present a manifestly conformally invariant formulation of the
Willmore energy of a surface in $n$-dimensional space-form,
$\mathcal {W}(\Lambda )=\frac{1}{2}\int _{M}(\mathcal {N}\wedge
*\mathcal{N})$. The class of Willmore surfaces in $n$-space is then
established as invariant under the group of M\"{o}bius
transformations of $S^{n}$.\footnote{In fact, we verify the
M\"{o}bius invariance of the Willmore energy of a general surface
and establish then the M\"{o}bius invariance of the class of
Willmore surfaces.}  As already known to Blaschke \cite{blaschke}
for the particular case of spherical $3$-space, the Willmore energy
of a surface in a space-form coincides with the energy of its
central sphere congruence. Furthermore, a result by Blaschke
\cite{blaschke} (for $n=3$) and N. Ejiri \cite{ejiri} (for general
$n$) characterizes Willmore surfaces in spherical $n$-space by the
harmonicity of the central sphere congruence. Via this
characterization, the class of Willmore surfaces in space-forms is
then associated to a class of harmonic maps into Grassmannians. This
enables us to apply to this class of surfaces the well-developed
integrable systems theory of harmonic maps into Grassmannian
manifolds, with a spectral deformation and B\"{a}cklund
transformations, cf. \cite{uhlenbeck} and \cite{uhlenbeck 89}. We
define in this way a spectral deformation of Willmore surfaces,
which we verify to coincide, up to reparametrization, with the one
presented in \cite{SD}, as well as \textit{B\"{a}cklund
transformations}, the latter arising from a more complex
construction, presented in a chapter below.\newline

\section{The Willmore functional}\label{willmfunct}

\markboth{\tiny{A. C. QUINTINO}}{\tiny{CONSTRAINED WILLMORE
SURFACES}}

In this section, we present a manifestly conformally invariant
formulation of the Willmore energy of a surface in a
space-form.\newline

We start by recalling the classical concept of Willmore energy of an
isometric immersion of a Riemannian surface into a Riemannian
manifold.

Consider a Riemannian manifold $(\bar{M},g)$ and an immersion
$f:M\to\bar{M}$. Provide $M$ with the metric $g_{_{f}}$ induced by
$f$ from $g$, making $f$ into an isometric immersion. Recall the
trace-free part of the second fundamental form of $f$,
$$\Pi ^{0}=\Pi
-g_{_{f}}\otimes\mathcal{H}\in\Gamma(L^2(TM,N_{f})).$$Suppose $M$ is
compact. The Willmore energy of $f$ is defined to be\footnote{In
fact, the Willmore energy is classically defined as \textit{half} of
the total squared norm of the trace-free part of the second
fundamental form. In either case, it generalizes the Willmore energy
of a surface in $\R^{3}$, as defined by T. Willmore, only up to some
constant and, in this case, some scaling (as we shall verify later
on). Although not sharing extremes, all these different energies
share extremals. The reason for this scaling of the Willmore energy
by $2$ is avoiding some scaling when comparing the Willmore energy
of a surface to the energy of its central sphere congruence, to take
place in section \ref{subsec:willmenergy} below.}
$$\mathcal {W}(f):=\int_{M}|\Pi^{0}|^2\mathrm{d}A,$$ for $$|\Pi
^{0}|^2:=\sum _{i,j}( \Pi ^{0}(X_{i},X_{j}),\Pi
^{0}(X_{i},X_{j})),$$ defined independently of the choice of a local
orthonormal frame $(X_{i})_{i=1,2}$ of $TM$, and $\mathrm{d}A$ the
area element of $M$. Let $g'$ be a metric on $\bar{M}$ conformally
equivalent to $g$, $g'=e^{2u}g$, for some $u\in
C^{\infty}(\bar{M},\R)$. Let $g'_{f}$ denote the metric induced in
$M$ by $f$ from $g'$ and $\Pi'$ denote the second fundamental form
of $f:M\rightarrow (\bar{M},g')$. Following
\eqref{eq:conformalvarianceoftheMCV}, we conclude that the
trace-free part of the second fundamental form is invariant under
conformal changes of the metric,
$$\Pi^{0}=(\Pi^{0})',$$ for $\Pi^{0}$ and $(\Pi^{0})'$ the
trace-free parts of $\Pi$ and $\Pi'$, respectively. Ultimately, we
conclude that
$$(|(\Pi^{0})'|')^2=e^{-2u\circ f}\,|\Pi^{0}|^{2},$$with $'$ indicating, yet
again, ``with respect to $g'$ ". On the other hand, according to
Lemma \ref{volconfchanges},
$$dA'=e^{2u\circ f}dA,$$relating the area element $dA'$ of $(M,g'_{f})$ to the area element
of $(M,g_{f})$. Under a conformal change of the metric, the square
of the length of $\Pi^{0}$ and the area element of $M$ change in an
inverse way, leaving the Willmore energy unchanged. It follows that:
\begin{thm}\label{Wenpres}
The Willmore energy is a M\"{o}bius invariant.
\end{thm}

The M\"{o}bius invariance of the Willmore energy of a surface in
spherical $3$-space was first established by W. Blaschke
\cite{blaschke}. B.-Y. Chen \cite{chen} generalized it to surfaces
in constant curvature Riemannian manifolds\footnote{We shall compute
the Willmore energy in this special case later on in this section.}.

Next we present a manifestly conformally invariant formulation of
the Willmore energy of a surface in a space-form. Let
$\Lambda\subset\underline{\R}^{n+1,1}$ be a surface in the
projectivized light-cone. We consider $o(\underline
{\mathbb{R}}^{n+1,1})$ provided with the metric induced by
$\mathrm{End}(\underline{\R}^{n+1,1})$: given $\alpha,\beta\in o
(\mathbb{R}^{n+1,1})$, $(\alpha,\beta)=-\mathrm{tr}\,\alpha\beta$.
Fixing a conformal structure in $M$, and having in consideration the
invariance of the Hodge $*$-operator on $1$-forms over $M$ under
conformal changes of the metric on $M$, we have well-defined a
$2$-form $(\mathcal {N}\wedge *\mathcal{N})$ over $M$ with values in
$\R$.

\begin{defn}\label{defWenofLambda} We define
the \emph{Willmore energy} of $\Lambda$ to be
$$\mathcal {W}(\Lambda ):=\frac{1}{2}\int _{M}(\mathcal {N}\wedge *\mathcal{N}),$$
with respect to the conformal structure induced in $M$ by $\Lambda$.
\end {defn}

The previous definition follows the definition of energy of the mean
curvature sphere congruence\footnote{For the relationship between
the mean curvature sphere congruence and the central sphere
congruence, see Section \ref{mcscVScsc}.} of a surface in spherical
$4$-space, presented in \cite{quaternionsbook}. Now fix a non-zero
$v_{\infty}$ in $\mathbb{R}^{n+1,1}$ and consider the surface
$\sigma _{\infty}:M\rightarrow S_{v_{\infty}}$, in the space-form
$S_{v_{\infty}}$, defined by $\Lambda$. The immersions $\Lambda$ and
$\sigma_{\infty}$ are related by
$$\Lambda=\pi_{\mathcal{L}}\vert_{S_{v_{\infty}}}\circ \sigma_{\infty}$$ via the
conformal diffeomorphism
$\pi_{\mathcal{L}}\vert_{S_{v_{\infty}}}:S_{v_{\infty}}\rightarrow
\mathbb{P}(\mathcal{L})\backslash\mathbb{P}(\mathcal{L}\cap\langle
v_{\infty}\rangle^{\perp})$. Hence the Willmore energy of
$\Lambda:M\rightarrow(\mathbb{P}(\mathcal{L}),h)$, fixing
$h\in\mathcal{C}_{\mathbb{P}(\mathcal{L})}$, coincides with the
Willmore energy of $\sigma_{\infty}$. The Willmore energy of the
conformal immersion $\Lambda:(M,\mathcal{C}_{\Lambda})\rightarrow
(\mathbb{P}(\mathcal{L}),\mathcal{C}_{\mathbb{P}(\mathcal{L})})$
consists of the Willmore energy of $\Lambda$ as an immersion of $M$
into the projectivized light-cone provided with a metric in
$\mathcal{C}_{\mathbb{P}(\mathcal{L})}$ (chosen arbitrarily), as
established next:

\begin{thm}\label{WsigmaeWLambda}
The Willmore energy of the conformal immersion $\Lambda$ coincides
with the Willmore energy of $\sigma_{\infty}$,
\begin{equation}\label{eq:relationbetweenWillmoreenergies}
\mathcal {W}(\Lambda )=\mathcal {W}(\sigma _{\infty}).
\end{equation}
\end{thm}
\begin{proof}
The Willmore energy of $\sigma_{\infty}$ is given by
$\int_{M}|\Pi^{0}_{\infty}|^2\mathrm{d}A_{\infty}$, for $\Pi
^{0}_{\infty}$ the trace-free part of $\Pi _{\infty}$ and
$\mathrm{d}A_{\infty}$ the area element of $M$ when provided with
the metric $g_{\infty}$. On the other hand, $(\mathcal {N}\wedge
*\mathcal{N})=-(*\mathcal{N}\wedge\mathcal{N})$ is a conformally
invariant way of writing $(\mathcal {N},\mathcal {N})_{g}dA_{g}$,
$$(\mathcal {N}\wedge *\mathcal{N})=(\mathcal {N},\mathcal
{N})_{g}dA_{g},$$for $g$ in $\mathcal {C}_{\Lambda}$, with $dA_{g}$
denoting the area element of $(M,g)$ and $(\,,\,)_{g}$ denoting the
Hilbert-Schmidt metric on $L((TM,g), o(\underline
{\mathbb{R}}^{n+1,1}))$. In particular, $(\mathcal {N}\wedge
*\mathcal{N})=(\mathcal {N},\mathcal
{N})_{g_{_{\infty}}}dA_{\infty}$. The proof of the theorem will
consist of showing that $(\mathcal {N},\mathcal
{N})_{g_{\infty}}=2|\Pi ^{0}_{\infty}|^2.$

For simplicity, set $\alpha :=\pi _{S^{\perp}}\circ d \circ \pi
_{S}\in\Omega ^{1}(\mathrm {End}(\underline{\mathbb{R}}^{n+1,1}))$.
According to equation \eqref{eq:transposeVwedgeVperp},
$$\mathcal{N}=\mathcal{N}\vert_{S}+\mathcal{N}\vert_{S^{\perp}}=\alpha -\alpha ^{t},$$ where $t$ indicates the transpose with
respect to the metric on $\R ^{n+1,1}$. Fix a local orthonormal
frame $\{X_{i}\}_{i}$ of $TM$ with respect to $g_{\infty}$.
Obviously, $\alpha\alpha=0$, so
$$(\mathcal {N},\mathcal {N})_{g_{\infty}}=\sum _{i}(\mathcal
{N}_{X_{i}},\mathcal {N}_{X_{i}})=-\sum _{i}\mathrm {tr}(\mathcal
{N}_{X_{i}}\mathcal {N}_{X_{i}}) =\sum _{i}(\mathrm {tr}(\alpha
_{X_{i}}\alpha ^{t}_{X_{i}}) +\mathrm {tr}(\alpha ^{t}_{X_{i}}
\alpha _{X_{i}}))$$ and, therefore,
$$(\mathcal {N},\mathcal {N})_{g_{\infty}}=2\sum _{i}\mathrm {tr}(\alpha
^{t}_{X_{i}} \alpha _{X_{i}}\vert _{S}),$$ having in consideration
that $\alpha ^{t}_{X_{i}} \alpha _{X_{i}}$ vanishes on $S^{\perp}$.
Recall that, if $(e_{i})_{i}$ and $(\hat e_{i})_{i}$ are dual basis
of a vector space $E$ provided with a metric $(\,,\,)$ (possibly
with signature), i.e., bases related by $(e_{i},\hat e_{j})=\delta
_{ij},\forall i,j$, then, given $\mu\in \mathrm {End}(E)$, $\mathrm
{tr}(\mu)=\sum _{i}(\mu (e_{i}),\hat e_{i})$. We shall now get two
dual frames of $S$, in order to compute $\mathrm {tr}(\alpha
^{t}_{X_{i}}\alpha _{X_{i}}\vert _{S}).$ Observe that, together, the
conditions $(\hat {\sigma}_{\infty},\hat {\sigma}_{\infty})=0$,
$(\sigma_{\infty},\hat {\sigma}_{\infty})=-1$ and $(\hat
{\sigma}_{\infty},d\sigma _{\infty})=0$ determine uniquely a section
$\hat {\sigma}_{\infty}$ of $S$. In  fact, as $d\sigma
_{\infty}(TM)$ is a bundle of $(2,0)$-planes, its orthogonal
complement in $S$ is a bundle of $(1,1)$-planes,
$$(d\sigma _{\infty}(TM))^{\perp}\cap S=
\underline{\R}^{1,1},$$ which, by the nullity of $\hat
{\sigma}_{\infty}$, restricts $\hat {\sigma}_{\infty}$ to two light
lines, one of which is $\langle\sigma_{\infty}\rangle$. The
condition $(\sigma_{\infty},\hat {\sigma}_{\infty})=-1$ shows that
$\hat {\sigma}_{\infty}$ is not in $\langle\sigma_{\infty}\rangle$
and, ultimately, determines $\hat {\sigma}_{\infty}$. The metric
relation between $\sigma _{\infty}$ and $\hat \sigma _{\infty}$
shows that $\hat \sigma _{\infty}\notin\Lambda ^{(1)}$, telling us
that $(\sigma _{\infty},d_{X_{1}}\sigma _{\infty}, d_{X_{2}}\sigma
_{\infty}, \hat\sigma _{\infty})$ forms a frame of $S$. The dual
frame of $S$ is the frame $(-\hat \sigma _{\infty},d_{X_{1}}\sigma
_{\infty}, d_{X_{2}}\sigma _{\infty}, -\sigma _{\infty})$. Thus
\begin{eqnarray*}
\mathrm {tr}(\alpha ^{t}_{X_{i}}\alpha _{X_{i}}\vert _{S})&=&
(\alpha ^{t}_{X_{i}}\alpha _{X_{i}}(\sigma _{\infty}),-\hat \sigma
_{\infty})+(\alpha _{X_{i}}(d_{X_{1}}\sigma _{\infty}),\alpha
_{X_{i}}(d_{X_{1}}\sigma _{\infty}))\\ & & \mbox{}+ (\alpha
_{X_{i}}(d_{X_{2}}\sigma _{\infty}),\alpha _{X_{i}}(d_{X_{2}}\sigma
_{\infty}))+(\alpha _{X_{i}}(\hat \sigma _{\infty}),-\alpha
_{X_{i}}(\sigma _{\infty}))
\end{eqnarray*}
and, consequently,
\begin{eqnarray*}
\mathrm {tr}(\alpha ^{t}_{X_{i}}\alpha _{X_{i}}\vert _{S})
 &=&\sum
_{j}(\alpha _{X_{i}}(d_{X_{j}}\sigma _{\infty}),\alpha
_{X_{i}}(d_{X_{j}}\sigma _{\infty}))\\&=&\sum _{j}(\mathcal
{N}_{X_{i}}(d_{X_{j}}\sigma _{\infty}),\mathcal
{N}_{X_{i}}(d_{X_{j}}\sigma _{\infty})).
\end{eqnarray*}

Now recall the identification between $N_{\infty}$ and $S^{\perp}$
via the isometric isomorphism $\mathcal{Q}$, presented in
Proposition \ref{identofnormals}. For $\xi\in
N_{\infty}\subset\langle\sigma_{\infty},v_{\infty}\rangle ^{\perp}$,
$$(\mathcal {Q}(\xi ),\pi _{S^{\perp}}(v_{\infty}))= (\mathcal
{Q}(\xi ),v_{\infty})=(\xi ,v_{\infty})+(\xi ,\mathcal
{H}_{\infty})(\sigma _{\infty},v_{\infty})= -(\mathcal
{Q}(\xi),\mathcal {Q}(\mathcal {H}_{\infty})),$$ establishing
$(\mu,\pi _{S^{\perp}}(v_{\infty})+\mathcal {Q}(\mathcal
{H}_{\infty}))=0$, for all $\mu\in S^{\perp},$ and, therefore,
\begin{equation}\label{eq:piperpvinfisQHinf}
\pi _{S^{\perp}}(v_{\infty})=-\mathcal {Q}(\mathcal {H}_{\infty}).
\end{equation}
Let $\pi _{N_{\infty}}$ denote the orthogonal projection of
$\underline {\mathbb{R}}^{n+1,1}=d\sigma _{\infty}(TM)\oplus N
_{\infty}\oplus \langle v_{\infty},\sigma _{\infty}\rangle$ onto
$N_{\infty}$. For arbitrary $X,Y\in \Gamma (TM)$, write
$d_{X}d_{Y}\sigma _{\infty}=\gamma +\eta +\beta v_{\infty}+\lambda
\sigma _{\infty}$, with $\gamma\in\Gamma (d\sigma _{\infty}(TM)),
\eta \in\Gamma (N _{\infty})$ and $\beta ,\lambda\in \Gamma
(\underline {\mathbb {R}})$. In fact, we can be more precise:
$$\beta =-(d_{X}d_{Y}\sigma _{\infty},\sigma
_{\infty})=(d_{Y}\sigma _{\infty},d_{X}\sigma
_{\infty})=g_{\infty}(X,Y).$$By equation
\eqref{eq:piperpvinfisQHinf}, $\pi _{S^{\perp}}(d_{X}d_{Y}\sigma
_{\infty})=\pi _{S^{\perp}}(\eta)-g_{\infty}(X,Y)\mathcal
{Q}(\mathcal {H}_{\infty})$. On the other hand, $\eta-\mathcal
{Q}(\eta)=(\eta ,\mathcal {H}_{\infty})\sigma _{\infty}\in \Gamma
(S)$ and, therefore, $\pi _{S^{\perp}}(\eta)=\mathcal {Q}\,(\pi
_{N_{\infty}}(d_{X}d_{Y}\sigma _{\infty}))$. Thus $\pi
_{S^{\perp}}(d_{X}d_{Y}\sigma _{\infty})=\mathcal {Q}( \pi _{N
_{\infty}}(d_{X}d_{Y}\sigma _{\infty})-g_{\infty}(X,Y)\mathcal
{H}_{\infty})$, or, equivalently, $$\mathcal {N}_{X}(d_{Y}\sigma
_{\infty})=\mathcal {Q}(\Pi ^{0}_{\infty}(X,Y)).$$  It follows that
\begin{equation}\label{eq:formtracefreeofsffofsigmainf}
|\Pi ^{0}_{\infty}|^2=\sum _{i,j}(\mathcal
{N}_{X_{i}}(d_{X_{j}}\sigma _{\infty}),\mathcal
{N}_{X_{i}}(d_{X_{j}}\sigma _{\infty})),
\end{equation}
completing the proof.
\end{proof}

We complete this section by computing the Willmore energy of a
compact surface immersed in $\mathbb{R} ^{3}$, specially popular in
the literature. Consider an immersion $f$ of $M$ into $\R^{3}$ and
provide $M$ with the metric $g_{f}$ induced by $f$. Fixing an
orthonormal frame $X_{1},X_{2}$ of $TM$, we have
\begin{eqnarray*}
|\Pi ^{0}|^2&=&\sum _{i,j}(\Pi (X_{i},X_{j})-\delta
_{ij}\mathcal {H},\Pi (X_{i},X_{j})-\delta _{ij}\mathcal {H}) \\
&=&\sum _{i,j}(\Pi (X_{i},X_{j}),\Pi (X_{i},X_{j})) -2(\sum
_{i} \Pi (X_{i},X_{i}),\mathcal {H}) +2|\mathcal{H}|^2 \\
&=&\sum _{i,j}(\Pi (X_{i},X_{j}),\Pi (X_{i},X_{j}))
-2|\mathcal{H}|^2,
\end{eqnarray*}
and, consequently,
\begin{eqnarray*}
|\Pi ^{0}|^2-2|\mathcal{H}|^2&=&\sum _{i,j}( \Pi (X_{i},X_{j}),\Pi
(X_{i},X_{j})) -\sum _{i, j}( \Pi (X_{i},X_{i}),\Pi
(X_{j},X_{j}))\\&=&\sum _{i\neq j}( \Pi (X_{i},X_{j}),\Pi
(X_{i},X_{j})) -\sum _{i\neq j}( \Pi (X_{i},X_{i}),\Pi
(X_{j},X_{j}))
\\&=&2( \Pi (X_{1},X_{2}),\Pi (X_{1},X_{2})) -2( \Pi
(X_{1},X_{1}),\Pi (X_{2},X_{2})).
\end{eqnarray*}
Hence, by equation \eqref{eq:KvsbarK},
$$|\Pi ^{0}|^2=2(|\mathcal{H}|^2-K+\bar{K})$$and,
therefore,
$$\mathcal {W}(f)=2\int_{M}(|\mathcal{H}|^2-K+\bar{K})\,\mathrm{d}A.$$
In the particular case $\bar {M}=\R ^{3}$, we get twice as much the
famous Willmore energy of a compact surface immersed in $\mathbb{R}
^{3}$:
$$\mathcal {W}(f)=2\int_{M}(H^2-K)\,\mathrm{d}A,$$ where $H^{2}$ denotes
the square of the mean curvature of $f$ (with respect to either of
the two unit normal vector fields to $f$). Amongst compact surfaces
without boundary in $\R^{3}$, and since for these the total Gaussian
curvature is a topological invariant (cf. Gauss-Bonnet theorem), the
Willmore functional shares critical points with the functional
$\tilde{\mathcal{W}}$ given by
$$\tilde{\mathcal{W}}(f):=\int_{M}H^2\,\mathrm{d}A,$$ which is what T.
Willmore \cite{willmore2} defined as the Willmore energy of a
compact surface, without boundary, immersed in $\R^{3}$.

\section{Willmore surfaces: definition and examples}

\markboth{\tiny{A. C. QUINTINO}}{\tiny{CONSTRAINED WILLMORE
SURFACES}}

Willmore surfaces are the critical points of the Willmore
functional. In view of the M\"{o}bius invariance of the Willmore
energy, Willmore surfaces form a M\"{o}bius invariant class of
surfaces. Minimal surfaces in $3$-dimensional space-forms, and so
their M\"{o}bius transforms, are examples of Willmore
surfaces.\newline

Let $\Lambda\subset\underline{\R}^{n+1,1}$ be a surface in the
projectivized light-cone. Suppose $M$ is compact.
\begin {defn}
$\Lambda$ is said to be a \emph{Willmore surface} if
$$\frac{d}{dt}_{\mid_{t=0}}\mathcal{W}(\Lambda _{t})=0$$ for every variation
$(\Lambda _{t})_{t}$ of $\Lambda$  through immersions of $M$ in
$\mathbb{P}(\mathcal{L})$.
\end {defn}

Classically, a Willmore surface is defined to be an immersion
$f:M\rightarrow \bar{M}$ of $M$ into a Riemannian manifold
$\bar{M}$, for which $\frac{d}{dt}_{\mid_{t=0}}\mathcal{W}(f
_{t})=0$, for every variation $(f _{t})_{t}$ of $f$ through
immersions $f_{t}:M\rightarrow \bar{M}$. It is immediate from
Theorem \ref{Wenpres} that conformal diffeomorphisms transform
Willmore surfaces into Willmore surfaces.
\begin{thm}\label{Willmandconfdiffeom}
The class of Willmore surfaces is M\"{o}bius invariant.
\end{thm}
In particular:
\begin{prop}\label{Willmtheo}
$\Lambda$ is a Willmore surface if and only if, fixing
$v_{\infty}\in\mathbb{R}^{n+1,1}$ non-zero, so is the surface in
$S_{v_{\infty}}$ defined by $\Lambda$.
\end{prop}

In Section \ref{Wilmmunderchange}, we extend the concept of Willmore
surface to surfaces that are, in particular, not necessarily
compact. We verify that minimal surfaces in space-forms, and so
their M\"{o}bius transforms, are examples of Willmore surfaces (see
Section \ref{sec:CMC}). In particular, the stereographic projection
of a minimal surface in $S^{n}$ is a Willmore surface in $\R^{n}$.
The Clifford torus  is embedded in $S^{3}$ as a minimal surface. It
projects stereographically onto the $\sqrt{2}$ anchor ring, which is
then a Willmore surface in $\R^{3}$ (as well as its M\"{o}bius
transforms).

B. Lawson \cite{lawson} proved that there are minimal embeddings
into $S^{3}$ of surfaces of arbitrary genus. Thus there exist in
$\R^{3}$ Willmore surfaces of arbitrary genus. But are all Willmore
surfaces in $\R^{3}$ obtainable as stereographics projections of
minimal surfaces in $S^{3}$? The answer is ``no". In fact, J. Langer
and D. Singer \cite{langer+singer} showed that there are infinitely
many closed curves on $S^{2}$ whose corresponding \textit{Hopf
torus} is a Willmore surface in $S^{3}$. And U. Pinkall
\cite{pinkall} showed then that, with one exception, the Willmore
surfaces obtained by stereographic projection of these Willmore tori
in $S^{3}$ cannot possibly be obtained by stereographic projection
of minimal surfaces in $S^{3}$. The exception is the $\sqrt{2}$
anchor ring.

\section{Willmore energy vs. energy of the central sphere congruence}\label{subsec:willmenergy}

\markboth{\tiny{A. C. QUINTINO}}{\tiny{CONSTRAINED WILLMORE
SURFACES}}

Under the standard identification
$S^{*}T\mathcal{G}\cong\mathrm{Hom}(S,S^{\perp})\cong S\wedge
S^{\perp}$, of bundles provided with a metric, $dS= \mathcal{N}$,
which establishes the Willmore energy of a surface conformally
immersed in a space-form as the energy of its central sphere
congruence.
\newline

Recall that, given $P$ a compact oriented Riemannian manifold and
$Q$ a pseudo-Riemannian manifold, the energy of a smooth map $\phi
:P\to Q$ is defined as
$$E(\phi)=\frac{1}{2}\int_{P}|d\phi|^2\mathrm{d}\mathrm{vol}_{P},$$
for the Hilbert-Schmidt norm $|d\phi|$ of $d\phi\in \Gamma
(\mathrm{Hom}(TP,\phi^*TQ))$ induced by the pseudo-Riemannian
structures of $P$ and $Q$, and $\mathrm{d}\mathrm{vol}_{P}$ the
volume element of $P$. If we think of the map $\phi$ as a way to
confine and stretch an elastic $P$ inside $Q$, then $E(\phi)$
represents an elastic deformation energy. As observed, in
particular, by J. Eells and J. Sampson \cite{eells+sampson}:
\begin{Lemma}
In the case $P$ is 2-dimensional, the energy of $\phi :P\to Q$ can
be invariantly defined with respect to a conformal class of metrics
in $P$.
\end{Lemma}
\begin{proof}
If $P$ is $2$-dimensional, then under conformal changes of the
metric on $P$, the volume element and the square norm of $d\phi$
vary in an inverse way. In fact, given $g$ and $g':=e^{u}g$, for
some $u\in\Gamma(\underline{\R})$, conformally equivalent metrics on
$P$, we have, according to Lemma \ref{volconfchanges},
$$\mathrm{d}\mathrm{vol}_{(P,g')}=e^{u}\mathrm{d}\mathrm{vol}_{(P,g)},$$
whereas, clearly, $$|d\phi|_{g'}^2=\frac{1}{e^{u}}\,|d\phi|_{g}^2.$$
\end{proof}
We are then, in the case $M$ is compact, in a position to discuss
the energy $E(S,\mathcal{C}_{\Lambda})$ of the central sphere
congruence $S:M\rightarrow \mathcal {G}$ of a surface
$\Lambda:M\rightarrow\mathbb{P}(\mathcal{L})$, when providing $M$
with the conformal structure $\mathcal{C}_{\Lambda}$ induced in $M$
by $\Lambda$. As already known to Blaschke \cite{blaschke} in the
particular case of spherical $3$-space:

\begin{thm}\label{willmmigualenergy}
The Willmore energy of a surface $\Lambda:M\rightarrow
\mathbb{P}(\mathcal{L})$ conformally immersed in the projectivized
light-cone coincides with the energy of its central sphere
congruence $S:M\rightarrow \mathcal{G}$,
$$\mathcal {W}(\Lambda)=E(S,\mathcal{C}_{\Lambda}).$$
\end{thm}

The proof of the theorem will be immediate after a few
considerations, as follows.

Under the identification $S^{*}T\mathcal{G}\cong
\mathrm{Hom}(S,S^{\perp})$, of bundles provided with a metric,
defined in Section \ref{sec:csc}, we have
$(d_{X}S)\xi=\pi_{S^{\perp}}(d_{X}\xi)=\mathcal{N}_{X}\xi$, for
$\xi\in\Gamma (S)$, and, therefore,
$d_{X}S=\mathcal{N}_{X}\vert_{S}$, given $X\in\Gamma(TM)$. Under,
furthermore, the identification $\mathrm{Hom}(S,S^{\perp})\cong
S\wedge S^{\perp}$, of bundles provided with a metric, defined in
Section \ref{extalgebra}, we have $d_{X}S=\mathcal{N}_{X}$, for all
$X\in\Gamma(TM)$.  Hence, under the identification
$$S^{*}T\mathcal{G}\cong\mathrm{Hom}(S,S^{\perp})\cong
S\wedge S^{\perp},$$ of bundles provided with a metric, we have
\begin{equation}\label{eq:dS=N}
dS=\mathcal{N}
\end{equation}
and, therefore, fixing a metric on $M$,
$$|dS|^{2}=|\mathcal{N}|^{2}.$$

The proof of Theorem \ref{willmmigualenergy} is now immediate:
\begin{proof}
Fixing a metric in $\mathcal{C}_{\Lambda}$,
$$(\mathcal{N}\wedge
*\mathcal{N})=-(*\mathcal{N}\wedge\mathcal{N})=(\mathcal{N},\mathcal{N})\,dA=|dS|^{2}dA.$$
\end{proof}

\section{Willmore surfaces and harmonicity}\label{subsec:willmharm}

\markboth{\tiny{A. C. QUINTINO}}{\tiny{CONSTRAINED WILLMORE
SURFACES}}

Willmore surfaces are the extremals of the Willmore functional, just
like harmonic maps are the extremals of the energy functional. The
Willmore energy of a surface conformally immersed in a space-form
coincides with the energy of its central sphere congruence.
Furthermore, a result by Blaschke \cite{blaschke} (for $n=3$) and N.
Ejiri \cite{ejiri} (for general $n$) characterizes Willmore surfaces
isometrically immersed in spherical $n$-space by the harmonicity of
the central sphere congruence. This characterization will enable us,
in the sections below, to apply to the class of Willmore surfaces
conformally immersed in space-forms the well-developed integrable
systems theory of harmonic maps into Grassmanian manifolds and to
prove that Willmore surfaces constitute an integrable
system.\newline

Recall that, given $P$ a compact oriented Riemannian manifold and
$Q$ a pseudo-Riemannian manifold, a smooth map $\phi :P\to Q$ is
said to be harmonic if it extremises the energy functional,
$$\frac{d}{dt}_{\mid_{t=0}}E(\phi _{t})=0,$$ for every variation
$(\phi_{t})_{t}$ of $\phi$ through smooth maps from $P$ to $Q$. The
associated Euler-Lagrange equation is
\begin{equation}\label{eq:ELforharm}
\mathrm{tr}\,\nabla d\phi =0,
\end{equation}
where $\nabla d\phi$ denotes the Hessian of $\phi$, the section of
$S^2(TP,\phi^*TQ)$ defined by
$$\nabla d\phi(X,Y):=\nabla ^{\phi ^*TQ}_{X}d\phi(Y)-d\phi(\nabla
^{TP}_{X}Y),$$ for the Levi-Civita connection $\nabla ^{TP}$ on $P$
and the connection $\nabla ^{\phi ^*TQ}$ induced in the pull-back
bundle $\phi ^*TQ$ by the Levi-Civita connection on $Q$. The section
$\tau_{\phi}:=\mathrm{tr}\,\nabla d\phi$ of $\phi^{-1}TQ$ is called
the tension field of $\phi$. Behind the characterization of the
harmonicity of $\phi$ provided by equation \eqref{eq:ELforharm} is
the classical formula
\begin{equation}\label{eq:enandvariational}
\frac{d}{dt}_{\mid_{t=0}}E(\phi _{t})=-\int _{P}(\dot
{\phi},\mathrm{tr}\nabla d\phi )\mathrm{dvol}_{P},
\end{equation}
relating the variation of energy through a variation of $\phi$ to
the respective variational vector field, $$\dot
{\phi}:=\frac{d}{dt}_{\mid_{t=0}}\phi _{t}\in\Gamma (\phi ^{*}TQ).$$
Equation \eqref{eq:enandvariational} will be useful in the future.
It will also be useful to recall that, as we travel along all the
variations of $\phi$, the variational vector field $\dot {\phi}$
travels along all the sections of $\phi ^*TQ$: given $\eta\in\Gamma
(\phi ^*TQ)$, by setting
$$\phi _{t}(p):=\mathrm{exp}^{Q}_{\phi (p)}(t\eta _{p})$$ for $t\in\R, p\in P$,
we define a variation $(\phi _{t})_{t}$ of $\phi$ through smooth
maps from $P$ to $Q$ for which $\dot {\phi}=\eta$.

Although the Levi-Civita connection is not a conformal invariant,
the harmonicity of a map defined on a $2$-dimensional manifold is
preserved by conformal changes of the metric on that manifold. In
fact, as observed by J. Eells and J. Sampson \cite{eells+sampson},
energy and, therefore, harmonicity of a map of a surface are
preserved by conformal diffeomorphisms. It is well known that
\begin{equation}\label{eq:tracoehessiana}
d^{\nabla}*d\phi =-*(\mathrm{tr}\nabla d\phi ),
\end{equation}
denoting $\nabla ^{\phi ^*TQ}$ by $\nabla$, which leads us to the
following:

\begin{Lemma} In the case $P$ is $2$-dimensional, the equation
\begin{equation}\label{eq:harmconfequiveq}
d^{\nabla}*d\phi=0
\end{equation}
constitutes a characterization of the harmonicity of $\phi$,
manifestly invariant under a conformal change of the metric on $P$.
\end{Lemma}
\begin{proof}
Equation \eqref{eq:harmconfequiveq} provides a characterization of
the harmonicity of $\phi$, according to equation
\eqref{eq:tracoehessiana}. On the other hand, in the case $P$ is
$2$-dimensional, the Hodge $*$-operator on $1$-forms over $P$ is a
conformal invariant. Since $d^{\nabla}$ depends only on the
pseudo-Riemannian structure on $Q$, we conclude that, under a
conformal change of the metric on $P$, $d^{\nabla}*d\phi$ remains
invariant (and so does then the harmonicity of $\phi$).
\end{proof}

There is then no ambiguity in the following statement. Let
$\Lambda\subset\underline{\R}^{n+1,1}$ be a surface in the
projectivized light-cone. Suppose $M$ is compact.

\begin{thm}\label{willmorevsharmonicitytheorem}
$\Lambda$ is a Willmore surface if and only if its central sphere
congruence $S:M\rightarrow \mathcal {G}$ is harmonic  with respect
to the conformal structure induced in $M$ by $\Lambda$.
\end{thm}

The characterization of Willmore surfaces isometrically immersed in
spherical $n$-space by the harmonicity of the central sphere
congruence is due to  W. Blaschke \cite{blaschke} (for $n=3$) and N.
Ejiri \cite{ejiri} (for general $n$). The proof of Theorem
\ref{willmorevsharmonicitytheorem} we present next is a
generalization in the light-cone picture of the proof presented in
\cite{quaternionsbook}, in the quaternionic setting, for the
particular case of surfaces in $S^{4}$. The conclusion will follow
easily from three useful lemmas we present next.
\begin{Lemma}\label{dotWanddotE}
Let $(\Lambda_{t})_{t}$ be a variation of $\Lambda$ through
immersions of $M$ in $\mathbb{P}(\mathcal{L})$ and $(S_{t})_{t}$ be
the corresponding variation of $S$ through central sphere
congruences. Then
\begin{equation}\label{eq:derivwillmvsderivenergy}
\frac{d}{dt}_{\vert_{t=0}}\mathcal
{W}(\Lambda_{t})=\frac{d}{dt}_{\mid_{t=0}}E(S_{t},\mathcal
{C}_{\Lambda}),
\end{equation}
for $E(S_{t},\mathcal {C}_{\Lambda})$ the energy of
$S_{t}:M\rightarrow \mathcal{G}$ when providing $M$ with the
conformal structure $\mathcal {C}_{\Lambda}$.
\end{Lemma}

\begin{proof}
Write $\mathcal {C}_{t}$ for the conformal structure induced in $M$
by $\Lambda_{t}$, writing also $\mathcal {C}_{\Lambda}$ for
$\mathcal {C}_{0}$. According to Theorem \ref{willmmigualenergy},
for each $t$, $\mathcal {W}(\Lambda_{t})=E(S_{t},\mathcal {C}_{t})$.
The proof will consist of showing that
$$\frac{d}{dt}_{\mid_{t=0}}E(S_{t},\mathcal
{C}_{t})=\frac{d}{dt}_{\mid_{t=0}}E(S_{t},\mathcal {C}_{\Lambda}).$$

For each $t$,
$$E(S_{t},\mathcal
{C}_{t})=\int_{M}\frac{1}{2}\,(dS_{t},dS_{t})_{t}\mathrm{dA}_{t}=\int_{M}\frac{1}{2}\,(dS_{t}\wedge
*_{t}dS_{t}),$$ for $(,)_{t}$ the Hilbert-Schmidt metric on
$\mathrm{Hom}((TM,g_{t}),\mathrm{Hom}(S_{t},S_{t}^{\perp}))$,
$\mathrm{dA}_{t}$ and $*_{t}$ the area element and the Hodge
$*$-operator of $(M,g_{t})$, respectively, fixing
$g_{t}\in\mathcal{C}_{t}$. Thus
$$\frac{d}{dt}_{\mid_{t=0}}E(S_{t},\mathcal
{C}_{t})=\frac{1}{2}\int _{M}((d\dot{S}\wedge
*dS)+(dS\wedge\dot{*}dS)+(dS\wedge*d\dot{S})),$$ abbreviating
$\frac{d}{dt}_{\mid_{t=0}}$ by a dot and writing $*$ for $*_{0}$.
Now we verify that
\begin{equation}\label{eq:dSwedgdedS=0}
(dS\wedge\dot{*}dS)=0.
\end{equation}
Fix $X\in\Gamma(TM)$ locally never-zero, so that $X,JX$ provides a
local frame of $TM$, for $J$ the canonical complex structure in
$(M,\mathcal {C}_{\Lambda})$. The $2$-form $(dS\wedge\dot{*}dS)$
vanishes if and only if $(dS\wedge\dot{*}dS)(X,JX)=0$, or,
equivalently, $(dS\wedge\dot{*}dS)(X+iJX,X-iJX)=0$. For each $t$,
let $J_{t}\in\Gamma (\mathrm{End}(TM))$ be the canonical complex
structure in $(M,\mathcal {C}_{t})$, writing also $J$ for $J_{0}$.
Differentiation at $t=0$ of $*_{t}dS_{t}=-(dS_{t}) J_{t}$ gives
$\dot{*}dS+*d\dot{S}=-(d\dot{S})
J-(dS)\dot{J}=*d\dot{S}-(dS)\dot{J}$ and, therefore,
$$\dot{*}dS=-(dS)\dot{J}.$$ Hence equation \eqref{eq:dSwedgdedS=0} holds if
and only if
$-(d_{X^{0,1}}S,d_{\dot{J}X^{1,0}}S)+(d_{X^{1,0}}S,d_{\dot{J}X^{0,1}}S)=0$,
for $X^{0,1}:=X+iJX$ and $X^{1,0}:=X-iJX$. Now differentiation at
$t=0$ of $J_{t}^2=-I$ gives
$$\dot{J} J=-J\dot{J}$$and, consequently, that $\dot{J}$ intertwines the
eigenspaces of $J$:  given $X\in\Gamma (TM)$,
$$J(\dot{J}(X\pm iJX))=-\dot{J}(J(X\pm iJX))=\pm i\dot{J}(X\pm
iJX),$$respectively, showing that $\dot{J}(T^{1,0}M)\subset
T^{0,1}M,\,\,\,\dot{J}(T^{0,1}M)\subset T^{1,0}M$. By the
conformality of $S:(M,\mathcal {C}_{\Lambda})
\rightarrow\mathcal{G}$, it follows that
$(d_{X^{0,1}}S,d_{\dot{J}X^{1,0}}S)=0=(d_{X^{1,0}}S,d_{\dot{J}X^{0,1}}S)$.
We establish  \eqref{eq:dSwedgdedS=0} and, consequently, that
\begin{eqnarray*}
\frac{d}{dt}_{\mid_{t=0}}E(S_{t},\mathcal {C}_{t})&=&\frac{1}{2}\int
_{M}((d\dot{S}\wedge
*dS)+(dS\wedge*d\dot{S}))\\&=&\frac{d}{dt}_{\mid _{t=0}}\int _{M}
\frac{1}{2}\,(dS_{t}\wedge
*dS_{t})\\&=&\frac{d}{dt}_{\mid_{t=0}}E(S_{t},\mathcal
{C}_{\Lambda}),
\end{eqnarray*}
completing the proof.
\end{proof}

Lemma \ref{dotWanddotE} establishes, in particular, that $\Lambda$
is a Willmore surface as soon as
$S:(M,\mathcal{C}_{\Lambda})\rightarrow \mathcal{G}$ is harmonic.
However, it does not prove Theorem
\ref{willmorevsharmonicitytheorem}, as a variation of the central
sphere congruence is not necessarily a variation through central
sphere congruences.

\begin{Lemma}\label{variationalsrelated446190}
Let $(\Lambda_{t})_{t}$ be a variation of $\Lambda$ through
immersions of $M$ in $\mathbb{P}(\mathcal{L})$ and $(S_{t})_{t}$ be
the corresponding variation of $S$ through central sphere
congruences. For each $t$, let $\sigma_{t}$ be a never-zero section
of $\Lambda_{t}$, writing also $\sigma$ for $\sigma_{0}$. The
variational vector fields of $(\sigma_{t})_{t}$ and $(S_{t})_{t}$
are related by
$$\dot{S}(\sigma)=\pi_{S^{\perp}}(\dot{\sigma}).$$
\end{Lemma}

\begin{proof}
Let $\mathcal{S}:M\times ]-\varepsilon,\varepsilon[\rightarrow
\mathcal{G}$ be defined by $\mathcal{S}(p,t):=S_{t}(p)$. Under the
identification of $S^{*}T\mathcal{G}$ with
$\mathrm{Hom}(S,S^{\perp})$ defined by \eqref{eq:identificTGwithHom}
for the case $T=S$, given $\xi:M\times
]-\varepsilon,\varepsilon[\rightarrow \R^{n+1,1}$ such that
$\xi_{0}:=(x\mapsto \xi(x,0))$ is a section of $S$, we have
$d\mathcal{S}_{(p,0)}(u,k)(\xi_{0}(p))=\pi_{S^{\perp}}(d\xi_{(p,0)}(u,k))$,
for all $p\in M$, $u\in T_{p}M$ and $k\in \R$. In particular, for
$\xi$ defined by $\xi(p,t):=\sigma_{t}(p)$, for $u=0$ and
$k=(d/dt)_{t=0}$, we get
$$d(\mathcal{S}^{p})_{0}((d/dt)_{t=0})\sigma
(p)=\pi_{S^{\perp}}(d(\xi^{p})_{0}((d/dt)_{t=0}),$$ for
$\mathcal{S}^{p}:=(t\mapsto \mathcal{S}(p,t))$ and
$\xi^{p}:=(t\mapsto \xi(p,t))$. Equivalently,
$$(d/dt)_{t=0}\,S_{t}(p)\sigma(p)=\pi_{S^{\perp}}((d/dt)_{t=0}\,\sigma_{t}(p)).$$
\end{proof}

\begin{rem}\label{dotsigmaindepi?ofsigmatuptoLambda}
Lemma \ref{variationalsrelated446190} establishes, in particular,
that $\pi_{S^{\perp}}\dot{\sigma}$ does not depend on the variation
$(\sigma_{t})_{t}$ of $\sigma$, only on the variation
$(\langle\sigma_{t}\rangle)_{t}=(\Lambda_{t})_{t}$ of
$\langle\sigma\rangle=\Lambda$. Furthermore, given a variation
$(\sigma'_{t})_{t}=(\lambda_{t}\sigma_{t})_{t}$ of $\sigma$, with
$\lambda_{t}\in\Gamma(\underline{\R})$ never-zero for all $t$, the
respective variational vector field $\dot{\sigma}'$ relates to
$\dot{\sigma}$ by
$\dot{\sigma}'=(\frac{d}{dt}_{\vert_{t=0}}\lambda_{t})\sigma'_{0}+\lambda_{0}\dot{\sigma}$
and, therefore,
$$\dot{\sigma}'=\dot{\sigma}\,\mathrm{mod}\Lambda.$$
\end{rem}

Given $z$ a holomorphic chart of $(M,\mathcal{C}_{\Lambda})$, we use
$\tau_{z}$ to denote the tension field of $S:(M,g_{z})\rightarrow
\mathcal{G}$,
$$\tau _{z}=\mathrm{tr}_{z}\nabla ^{z}dS\in\Gamma(S^{*}T\mathcal{G}),$$for $\nabla ^{z}dS$ the Hessian of
$S:(M,g_{z})\rightarrow \mathcal{G}$ and $\mathrm{tr}_{z}$
indicating trace computed with respect to
$g_{z}\in\mathcal{C}_{\Lambda}$. For simplicity, let  $\nabla$
denote the pull-back connection on $S^*T\mathcal {G}$ induced by the
Levi-Civita connection on $\mathcal {G}$ (when provided with the
pseudo-Riemannian structure defined in Section \ref{sec:csc}).

\begin{Lemma}\label{tauzallaboutit}
Given $z$ a holomorphic chart of $(M,\mathcal{C}_{\Lambda})$,
\begin{equation}\label{eq:tauinnablaS}
4\,\nabla _{\delta _{z}}S_{\bar{z}}=\tau_{z}=4\,\nabla _{\delta
_{\bar{z}}}S_{z}.
\end{equation}
It follows that, under the usual identification
$S^{*}T\mathcal{G}\cong\mathrm {Hom}(S,S^{\perp})$,
\begin{equation}\label{eq:kertauz}
\Lambda^{(1)}\subset \ker\tau_{z}
\end{equation}
and, consequently,
\begin{equation}\label{eq:ImTsubsetperpKer}
\mathrm{Im}\,\tau_{z}^{t}\subset \Lambda.
\end{equation}
\end{Lemma}
\begin{proof}
Fix a holomorphic chart $z=x+iy$ of $(M,\mathcal{C}_{\Lambda})$.
First of all, note that, as $\delta_{x},\delta_{y}$ is an
orthonormal frame of $(TM,g_{z})$, we have
$$\tau _{z}=\nabla _{\delta
_{x}}S_{x}-dS(\nabla^{g_{z}}_{\delta _{x}}\delta _{x})+\nabla
_{\delta _{y}}S_{y}-dS(\nabla^{g_{z}}_{\delta _{y}}\delta _{y}),$$
for $\nabla^{g_{z}}$ the Levi-Civita connection on $(M,g_{z})$. On
the other hand, for $J$ the canonical complex-structure in
$(M,\mathcal{C}_{\Lambda})$, $(M,g_{z},J)$ is a K\"{a}hler manifold
and, therefore,
$$J(\nabla^{g_{z}}_{\delta _{x}}\delta _{x}+\nabla^{g_{z}}_{\delta
_{y}}\delta _{y})=\nabla^{g_{z}}_{\delta _{x}}J\delta
_{x}+\nabla^{g_{z}}_{\delta _{y}}J\delta _{y}=\nabla^{g_{z}}_{\delta
_{x}}\delta _{y}-\nabla^{g_{z}}_{\delta _{y}}\delta _{x}=[\delta
_{x},\delta _{y}]=0.$$ Thus
$$\tau_{z}=\nabla
_{\delta _{x}}S_{x}+\nabla _{\delta _{y}}S_{y}.$$By the symmetry of
the Hessian,
$$\nabla_{\delta_{x}}S_{y}-dS(\nabla^{g_{z}}_{\delta_{x}}\delta_{y})=\nabla_{\delta_{y}}S_{x}-dS(\nabla^{g_{z}}_{\delta_{y}}\delta_{x}),$$
or, equivalently, by the torsion-free property of the Levi-Civita
connection,
$$\nabla_{\delta_{x}}S_{y}-\nabla_{\delta_{y}}S_{x}=dS([\delta
_{x},\delta _{y}])=0.$$ This establishes \eqref{eq:tauinnablaS}.
Next observe that, whilst, for a never-zero section $\sigma$ of
$\Lambda$,
$(\sigma_{zz},\sigma)=(\sigma,\sigma_{z})_{z}-(\sigma_{z},\sigma_{z})=0$,
as well as $ (\sigma
_{zz},\sigma_{z})=\frac{1}{2}\,(\sigma_{z},\sigma_{z})_{z}=0$; for
the particular case of the normalized section $\sigma ^{z}$ of
$\Lambda$ with respect to $z$, we have, furthermore $ (\sigma^{z}
_{zz},\sigma^{z}_{\bar{z}})=(\sigma ^{z}_{z},\sigma
^{z}_{\bar{z}})_{z}-(\sigma^{z} _{z},\sigma^{z} _{z\bar{z}})=0$, as
$(\sigma^{z}_{z},\sigma^{z}_{\bar{z}})$ is constant. Hence $\pi
_{S}(\sigma ^{z}_{zz})$ is orthogonal to $\sigma^{z}$,
$\sigma^{z}_{z}$ and $\sigma^{z} _{\bar{z}}$, so that $ \pi
_{S}(\sigma ^{z}_{zz})\in\Gamma (\Lambda)$ and, therefore, $(\pi
_{S}(\sigma ^{z}_{zz}))_{\bar{z}}\in\Gamma (S)$. It follows that
\begin{eqnarray*}
\tau _{z}(\sigma ^{z}_{z})&=&4\,(\nabla _{\delta
_{z}}S_{\bar{z}})\sigma^{z}_{z}\\&=&4\,(\nabla
^{S^{\perp}}_{\delta_{z}}(S_{\bar{z}}\sigma
^{z}_{z})-S_{\bar{z}}(\nabla^{S}_{\delta_{z}}\sigma^{z}_{z}))\\&=&4\,(\nabla
^{S^{\perp}}_{\delta_{z}}(\pi _{S^{\perp}}(\sigma^{z}
_{z\bar{z}}))-S_{\bar{z}}(\pi _{S}(\sigma^{z}_{zz})))\\&=&-4\,\pi
_{S^{\perp}}((\pi
_{S}(\sigma^{z}_{zz}))_{\bar{z}})\\&=&0.\end{eqnarray*} The
orthogonality of  $\pi _{S}(\sigma ^{z}_{zz})$ to $\sigma^{z}$,
$\sigma^{z}_{z}$ and $\sigma^{z} _{\bar{z}}$ establishes that of
$\pi _{S}(\sigma^{z}_{\bar{z}\bar{z}})$, establishing, ultimately,
$$\tau _{z}(\sigma ^{z}_{\bar{z}})=4\,(\nabla _{\delta
_{\bar{z}}}S_{z})\sigma^{z}_{\bar{z}}=-4\,\pi _{S^{\perp}}((\pi
_{S}(\sigma^{z}_{\bar{z}\bar{z}}))_{z})=0.$$On the other hand,
$$\tau _{z}(\sigma ^{z})=4\,(\nabla _{\delta
_{z}}^{S^{\perp}}(\pi_{S^{\perp}}(\sigma^{z}_{\bar{z}}))-\pi_{S^{\perp}}(\sigma
^{z}_{z\bar{z}}))=0.$$We conclude that $\Lambda^{(1)}\subset
\ker\tau_{z}$ and, consequently, that $$\mathrm{Im}\,\tau
_{z}^{t}\subset(\ker\tau_{z})^{\perp}\cap S\subset(\Lambda^{(1)})
^{\perp}\cap S=\Lambda.$$
\end{proof}

Now we proceed to the proof of Theorem
\ref{willmorevsharmonicitytheorem}.

\begin{proof}
Suppose $S:(M,\mathcal{C}_{\Lambda})\rightarrow\mathcal{G}$ is
harmonic. Then, in particular, given an arbitrary variation
$(\Lambda_{t})_{t}$ of $\Lambda$ through surfaces in the
projectivized light-cone, we have
\begin{equation}\label{eq:sfdshp'029uyhn xcqt}
\frac{d}{dt}_{\mid_{t=0}}E(S_{t},\mathcal {C}_{\Lambda})=0,
\end{equation}
for the corresponding variation $(S_{t})_{t}$ of $S$ through central
sphere congruences. By equation \eqref{eq:derivwillmvsderivenergy},
we conclude that $\Lambda$ is a Willmore surface.

Conversely, suppose that $\Lambda$ is a Willmore surface. Fix a
holomorphic chart $z$ of $(M,\mathcal{C}_{\Lambda})$. To prove that
$S:(M,\mathcal{C}_{\Lambda})\rightarrow\mathcal{G}$ is harmonic, we
consider the usual identification $S^{*}T\mathcal{G}\cong\mathrm
{Hom}(S,S^{\perp})$ and show that $\tau _{z}\in\Gamma(\mathrm
{Hom}(S,S^{\perp}))$ vanishes. For that, and aiming for a
contradiction, suppose that $\tau _{z}$ is non-zero. Then so is
$\tau _{z}^{t}\in\Gamma (\mathrm {Hom}(S^{\perp},\Lambda))$. Fix a
never-zero section $\sigma$ of $\Lambda$ and a variation $(\sigma
_{t})_{t}$ of $\sigma$ through smooth maps $\sigma _{t}:M\rightarrow
\mathcal{L}$ with
$\dot{\sigma}\in\Gamma(\langle\sigma\rangle^{\perp})=\Gamma(\Lambda^{(1)}\oplus
S^{\perp})$ a section of $S^{\perp}$ such that $\tau
_{z}^{t}(\pi_{S^{\perp}}\dot {\sigma})=\lambda \sigma$ for some
positive $\lambda\in C ^{\infty}(M,\R)$. Define a variation of
$\Lambda$ through surfaces in the projectivized light-cone by
setting $\Lambda _{t}:=\langle \sigma _{t}\rangle$, for each $t$.
Let $(S_{t})_{t}$ be the corresponding variation of $S$ through
central sphere congruences and $\dot {S}$ be the corresponding
variational vector field. According to Lemma
\ref{variationalsrelated446190}, $\tau _{z}^{t}\, \dot
{S}(\sigma)=\lambda \sigma$. On the other hand,  yet again according
to \eqref{eq:ImTsubsetperpKer}, $\mathrm{tr}(\tau _{z}^{t}\, \dot
{S})$ is simply the component of $\tau _{z}^{t}\,\dot {S}(\sigma)$
with respect to $\sigma$. Hence $\mathrm{tr}(\tau _{z}^{t}\,\dot
{S})=\lambda$ is positive. Lastly, the fact that $\Lambda$ is a
Willmore surface intervenes to establish  \eqref{eq:sfdshp'029uyhn
xcqt}. On the other hand, according to equation
\eqref{eq:enandvariational},
\begin{equation}\label{eq:gafrr41bb?pagcl?sbtce9}
\frac{d}{dt}_{\mid_{t=0}}E(S_{t},\mathcal {C}_{\Lambda})=- \int
_{M}(\dot {S},\tau _{z} )\,\mathrm {dA}_{z}=-\int
_{M}\mathrm{tr}(\tau _{z}^{t}\, \dot {S})\,\mathrm {dA}_{z},
\end{equation}
for $dA_{z}$ the area element of $(M,g_{z})$. It follows that
\begin{equation}\label{eq:ca329nv62392mv96321fn p098543sghj}
\int _{M}\mathrm{tr}(\tau _{z}^{t}\, \dot {S})\,\mathrm {dA}_{z}=0,
\end{equation}
which contradicts the conclusion of the positiveness of
$\mathrm{tr}(\tau _{z}^{t}\,\dot {S})$, completing the proof.
\end{proof}

\section{The Willmore surface equation}\label{subsec:willmeq}

\markboth{\tiny{A. C. QUINTINO}}{\tiny{CONSTRAINED WILLMORE
SURFACES}}

Having characterized conformal Willmore surfaces in the
projectivized light-cone by the harmonicity of the central sphere
congruence, we have, in particular, deduced the Willmore surface
equation for a conformal immersion
$\Lambda:M\rightarrow\mathbb{P}(\mathcal{L})$:
$$d^{\nabla ^{S^*T\mathcal {G}}}*dS=0,$$
for $\nabla ^{S^*T\mathcal {G}}$ the pull-back connection on
$S^*T\mathcal {G}$ induced by the Levi-Civita connection on
$\mathcal {G}$ (when provided with the pseudo-Riemannian structure
defined in Section \ref{sec:csc}). It is well-known (see, for
example, \cite{burstall+rawnsley}) that the usual identification
$S^{*}T\mathcal{G}\cong\mathrm{Hom}(S,S^{\perp})$, of bundles
provided with a metric,  respects connections,\footnote{This is the
particular case $\phi=S$ and
$\mathcal{G}=\mathrm{Gr}_{(3,1)}(\R^{n+1,1})$ of a fact regarding  a
general map $\phi:M\rightarrow\mathcal{G}$ into a general
Grassmannian $\mathcal{G}=Gr_{(r,s)}(\R^{p,q})$.} i.e., $\nabla
^{S^*T\mathcal {G}}$ consists of the connection induced canonically
in $\mathrm{Hom}(S,S^{\perp})$ by $\nabla^{S}$ and
$\nabla^{S^{\perp}}$,
$$\nabla ^{S^*T\mathcal {G}}\xi=\nabla^{S^{\perp}}\circ\xi-\xi\circ\nabla^{S}=
\mathcal{D}\circ\xi-\xi\circ\mathcal {D},$$ for all
$\xi\in\Gamma(\mathrm{Hom}(S,S^{\perp}))$. That is,
\begin{equation}\label{eq:nablaS+iscurlyD}
\nabla ^{S^*T\mathcal {G}}=\mathcal{D},
\end{equation}
for the connection induced naturally in $\mathrm{Hom}(S,S^{\perp})$
by the connection $\mathcal{D}$ on $\underline{\R}^{n+1,1}$. Note
that the connection induced naturally in $S\wedge S^{\perp}$ by
$\mathcal{D}$ coincides with the one induced naturally by
$\nabla^{S}$ and $\nabla^{S^{\perp}}$. By \eqref{eq:dS=N}, we
conclude that, under the usual identification
\begin{equation}\label{eq:bbusualidentificbbvvvvv}
S^{*}T\mathcal{G}\cong\mathrm{Hom}(S,S^{\perp})\cong S\wedge
S^{\perp},
\end{equation}
of bundles provided with a metric and a connection, $d^{\mathcal
{D}}*\mathcal{N} =0$, or, equivalently (cf. \eqref{eq:ext}),
$d*\mathcal{N} =0$, provides a characterization of Willmore surfaces
in the projectivized light-cone. In view of
$\mathcal{N}^{1,0}=\frac{1}{2}\,(\mathcal{N}+i*\mathcal{N})$ (and,
therefore,
$\mathcal{N}^{0,1}=\frac{1}{2}\,(\mathcal{N}-i*\mathcal{N})$),
Codazzi equation establishes
\begin{equation}\label{eq:codcurlyN}
d^{\mathcal{D}}\mathcal{N}^{1,0}=\frac{i}{2}\,d^{\mathcal{D}}*\mathcal{N}=-d^{\mathcal{D}}\mathcal{N}^{0,1}.
\end{equation}
It follows that:
\begin{thm}\label{Willmoreeq}
Willmore surfaces in the projectivized light-cone are characterized,
equivalently, by any of the following equations:

$i)$\,\,$d*\mathcal{N} =0;$

$ii)$\,\,$d^{\mathcal {D}}*\mathcal{N}=0;$

$iii)$\,\,$d^{\mathcal {D}}\mathcal{N}^{1,0}=0;$

$iv)$\,\,$d^{\mathcal {D}}\mathcal{N}^{0,1}=0.$
\end{thm}

\begin{rem}\label{taucharactofWillm}
According to Lemma \ref{tauzallaboutit}, together with
\eqref{eq:nablaS+iscurlyD}, the harmonicity of the central sphere
congruence $S:(M,\mathcal{C}_{\Lambda})\rightarrow \mathcal{G}$ of
$\Lambda$ can be characterized by $\mathcal{D} _{\delta
_{z}}S_{\bar{z}}=0$, or, equivalently, $(\mathcal{D}_{\delta
_{z}}S_{\bar{z}})\,\sigma_{z\bar{z}}=0$, fixing a never-zero section
$\sigma$ of $\Lambda$ and a holomorphic chart $z$ of
$(M,\mathcal{C}_{\Lambda})$.
\end{rem}

\begin{rem}
Let $v_{\infty}\in\R^{n+1,1}$ be non-zero and
$\sigma_{\infty}:M\rightarrow S_{v_{\infty}}$ be the surface defined
by $\Lambda$ in the space-form $S_{v_{\infty}}$. Let
$\Delta_{\infty}$ be the Laplacian in $N_{\infty}$ and
$\tilde{A}_{\infty}:=A_{\infty}^{*}\circ A_{\infty}$, for
$A_{\infty}$ mapping a unit $\xi\in\Gamma(N_{\infty})$ to
$A_{\infty}^{\xi}$, the shape operator of $\sigma_{\infty}$ with
respect to $\xi$.  Cf. \cite{weiner},
$$\Delta_{\infty}\mathcal{H}_{\infty}-2\vert\mathcal{H}_{\infty}\vert^{2}\mathcal{H}_{\infty}+\tilde{A}_{\infty}(\mathcal{H}_{\infty})=0$$
is a Willmore surface equation for $\sigma_{\infty}$, providing,
therefore, yet another Willmore surface equation for $\Lambda$. One
which takes us out of the path of this text, though.
\end{rem}

We dedicate what is left in this section to contemplating the
variational Willmore energy, supposing $M$ is compact. Let
$(\Lambda_{t})_{t}$ be a variation of $\Lambda$ through immersions
of $M$ in $\mathbb{P}(\mathcal{L})$ and $\dot{\mathcal{W}}$ be the
corresponding variational Willmore energy,
$$\dot{\mathcal{W}}=\frac{d}{dt}_{\vert_{t=0}}\mathcal {W}(\Lambda_{t}).$$
For each $t$, let $\sigma_{t}$ be a never-zero section of
$\Lambda_{t}$, writing also $\sigma$ for $\sigma_{0}$. Let
$\dot{\sigma}$ be the variational vector field of the variation
$(\sigma_{t})_{t}$. Differentiation of $(\sigma_{t},\sigma_{t})=0$
establishes $(\sigma,\dot{\sigma})=0$. In view of Remark
\ref{dotsigmaindepi?ofsigmatuptoLambda}, define
\begin{equation}\label{eq:dotLambda}
\dot{\Lambda}\in\Gamma(\mathrm{Hom}(\Lambda,\Lambda^{\perp}/\Lambda))
\end{equation}
by $\dot{\Lambda}\,\sigma:=\dot{\sigma}\,\mathrm{mod}\Lambda$. The
notation is not casual. In fact, under the isomorphism
$$d\pi_{\sigma}:\Lambda^{\perp}/\Lambda\cong T_{\Lambda}\mathbb{P}(\mathcal{L})$$
(cf. \eqref{eq:Tangenttoprojectlightcne}), the variational vector
field of $(\Lambda_{t})_{t}$ is
$\dot{\sigma}\,\mathrm{mod}\Lambda\in\Gamma(\mathrm{Hom}(\Lambda^{\perp}/\Lambda))$.
Set
$$\nu:=\pi\dot{\Lambda}\in \Gamma(\mathrm{Hom}(\Lambda,S^{\perp})),$$
for the canonical projection
$$\pi:\Gamma(\mathrm{Hom}(\Lambda,\Lambda^{\perp}/\Lambda)=\mathrm{Hom}(\Lambda,\Lambda^{(1)}/\Lambda))\oplus
\Gamma(\mathrm{Hom}(\Lambda,S^{\perp}))\rightarrow
\Gamma(\mathrm{Hom}(\Lambda,S^{\perp})).$$ Observe that
$$\nu\sigma=\pi_{S^{\perp}}(\dot{\Lambda}\sigma)=\pi_{S^{\perp}}\dot{\sigma}.$$

Having said so, let $(S_{t})_{t}$ be the variation of $S$ through
central sphere congruences corresponding to the variation
$(\Lambda_{t})_{t}$ of $\Lambda$ and $\dot{S}$ be the corresponding
variational central sphere congruence. Fix a holomorphic chart $z$
of $(M,\mathcal{C}_{\Lambda})$. According to  Lemma
\ref{dotWanddotE}, together with \eqref{eq:gafrr41bb?pagcl?sbtce9},
$$\dot{\mathcal{W}}=-\int _{M}\mathrm{tr}(\tau _{z}^{t}\, \dot
{S})\,\mathrm {dA}_{z}$$ and, therefore, by Lemma
\ref{tauzallaboutit}, followed by Lemma
\ref{variationalsrelated446190},
$$\dot{\mathcal{W}}=-\int _{M}(\tau _{z}^{t}\, \dot
{S}\,\sigma,\sigma_{z\bar{z}})(\sigma,\sigma_{z\bar{z}})^{-1}\,\mathrm
{dA}_{z}=-\int_{M}(\dot{\sigma},\tau_{z}\sigma_{z\bar{z}})(\sigma,\sigma_{z\bar{z}})^{-1}dA_{z}.$$
The skew-symmetry of
$\tau_{z}\in\Gamma(\mathrm{Hom}(S,S^{\perp})\cong S\wedge
S^{\perp})$ establishes then
$$\dot{\mathcal{W}}=\int_{M}(\tau_{z}\pi_{S^{\perp}}\dot{\sigma},\sigma_{z\bar{z}})(\sigma,\sigma_{z\bar{z}})^{-1}dA_{z}.$$
Let $*_{z}$ be the Hodge $*$-operator on forms over $(M,g_{z})$.
According to equation \eqref{eq:tracoehessiana}, $d^{\nabla
^{S^*T\mathcal {G}}}*dS=-*_{z}\tau_{z}$ and, therefore, under the
identification \eqref{eq:bbusualidentificbbvvvvv},
\begin{equation}\label{eq:tensionfieldvsdcurlyD*curlyN}
\tau_{z}=*_{z}\,d^{\mathcal{D}}*\mathcal{N}\in\Gamma(S\wedge
S^{\perp}).
\end{equation}
It follows that
$$\dot{\mathcal{W}}=\int_{M}((*_{z}\,d^{\mathcal{D}}*\mathcal{N})\nu\sigma,\sigma_{z\bar{z}})(\sigma,\sigma_{z\bar{z}})^{-1}dA_{z}.$$
As  we know, since $\mathcal{N}$ takes values in $S\wedge S^{\perp}$
and $S$ and $S^{\perp}$ are $\mathcal{D}$-parallel, the $2$-form
$d^{\mathcal{D}}*\mathcal{N}$ takes values in $S\wedge S^{\perp}$,
and so does, therefore, $*_{z}\,d^{\mathcal{D}}*\mathcal{N}$. In
view of equations \eqref{eq:kertauz} and
\eqref{eq:tensionfieldvsdcurlyD*curlyN}, we conclude, furthermore,
that
\begin{equation}\label{eq:starblablastar}
*_{z}\,d^{\mathcal{D}}*\mathcal{N}\in\Omega^{0}(\Lambda\wedge S^{\perp}).
\end{equation}
Hence $(*_{z}\,d^{\mathcal{D}}*\mathcal{N})\circ
\nu\in\Gamma(\mathrm{End}(\Lambda))$ and
\begin{eqnarray*}
\dot{\mathcal{W}}&=&\int_{M}\mathrm{tr}\,((*_{z}\,d^{\mathcal{D}}*\mathcal{N})\circ\nu)\,dA_{z}\\&=&
\int_{M}(\nu,(*_{z}\,d^{\mathcal{D}}*\mathcal{N})^{t})\,dA_{z}\\&=&-\int_{M}(*_{z}\,d^{\mathcal{D}}*\mathcal{N},\nu)\,dA_{z}
\end{eqnarray*}
and, ultimately,
$$\dot{\mathcal{W}}=-\int_{M}((d^{\mathcal{D}}*\mathcal{N})\wedge\nu).$$
As $*_{z}\,d^{\mathcal{D}}*\mathcal{N}$ vanishes on $\Lambda^{(1)}$,
we have
$$(*_{z}\,d^{\mathcal{D}}*\mathcal{N},\dot{\Lambda}-\nu)=\mathrm{tr}\,((*_{z}\,d^{\mathcal{D}}*\mathcal{N})^{T}(\dot{\Lambda}-\nu))=-\mathrm{tr}\,(*_{z}\,d^{\mathcal{D}}*\mathcal{N}\circ (\dot{\Lambda}-\nu))=0$$
and we conclude that the variational Willmore energy relates to the
variational surface by
\begin{equation}\label{eq:dotWvsdotLambda}
\dot{\mathcal{W}}=-\int_{M}((d^{\mathcal{D}}*\mathcal{N})\wedge\dot{\Lambda}).
\end{equation}
As a final remark, note that, according to Lemma \ref{dotWanddotE}
and \eqref{eq:gafrr41bb?pagcl?sbtce9}, on the other hand,
$\dot{\mathcal{W}}=- \int _{M}(\dot {S},\tau _{z} )\,\mathrm
{dA}_{z}$ and, therefore, by
\eqref{eq:tensionfieldvsdcurlyD*curlyN},
$$\dot{\mathcal{W}}=-\int_{M}((d^{\mathcal{D}}*\mathcal{N})\wedge \dot{S}).$$

\section{Willmore surfaces under change of flat metric
connection}\label{Wilmmunderchange}

\markboth{\tiny{A. C. QUINTINO}}{\tiny{CONSTRAINED WILLMORE
SURFACES}}

Let $\Lambda$ be a null line subbundle of the trivial bundle
$M\times\R^{n+1,1}$, not necessarily defining an immersion into
$\mathbb{P}(\mathcal{L})$. Let $\tilde{d}$ be a flat metric
connection on $\underline{\R}^{n+1,1}$.
\begin{defn}\label{eq:tildewilmm}
Suppose $\Lambda$ is a $\tilde{d}$-surface. $\Lambda$ is said to be
a \emph{Willmore $\tilde{d}$-surface} if
$$d^{\mathcal{D}^{\tilde{d}}}*_{\tilde{d}}\mathcal{N}^{\tilde{d}}=0,$$
where $*_{\tilde{d}}$ denotes the Hodge $*$-operator on $1$-forms
over $(M,\mathcal{C}_{\Lambda}^{\tilde{d}})$.
\end {defn}

This definition is a generalization of the characterization of a
Willmore surface in the projectivized light-cone provided by
$d^{\mathcal{D}}*\mathcal{N}=0$, corresponding to the particular
case $\tilde{d}=d$ and $M$ is compact.

Let $\tilde{\phi}:(\underline{\R}^{n+1,1},\tilde{d})\rightarrow
(\underline{\R}^{n+1,1},d)$ be an isomorphism. As observed in
Section \ref{subsec:tranfsfmc}, $\Lambda$ is a $\tilde{d}$-surface
if and only if $\tilde{\phi}\Lambda$ is a surface. Furthermore:

\begin{prop}\label{prop4.6.1} Suppose $\Lambda$ is a
$\tilde{d}$-surface (or, equivalently, $\tilde{\phi}\Lambda$ is a
surface). In that case, $\Lambda$ is a Willmore $\tilde{d}$-surface
if and only if $\tilde{\phi}\Lambda$ is a Willmore surface.
\end{prop}

\begin{proof}
Set $\tilde{\Lambda}=\tilde{\phi}\Lambda$. By
\eqref{eq:DeNtildevsDLambdatilde}, relating
$\mathcal{D}_{\tilde{\Lambda}}$ to $\mathcal{D}^{\tilde{d}}$ and
$\mathcal{N}_{\tilde{\Lambda}}$ to $\mathcal{N}^{\tilde{d}}$, we
have, given $X,Y\in\Gamma(TM)$,
$$d^{\mathcal{D}_{\tilde{\Lambda}}}*_{\tilde{d}}\mathcal{N}_{\tilde{\Lambda}}(X,Y)=
d^{\mathcal{D}_{\tilde{\Lambda}}}\,\tilde{\phi}\,(*_{\tilde{d}}\,\mathcal{N}^{\tilde{d}})\,\tilde{\phi}^{-1}(X,Y)=
\tilde{\phi}\,\,(d^{\mathcal{D}^{\tilde{d}}}*_{\tilde{d}}\mathcal{N}^{\tilde{d}}(X,Y))\,\tilde{\phi}^{-1}.$$
The fact that
$\mathcal{C}_{\Lambda}^{\tilde{d}}=\mathcal{C}_{\tilde{\Lambda}}$
(cf. \eqref{eq:metricaeconformalclasstilde}) completes the proof.
\end{proof}

\section{Spectral deformation of Willmore
surfaces}\label{subsec:willmfam}

\markboth{\tiny{A. C. QUINTINO}}{\tiny{CONSTRAINED WILLMORE
SURFACES}}

Let $\phi :M\rightarrow Gr_{(r,s)}(\R^{p,q})$ be a map into the
Grassmannian $Gr_{(r,s)}(\R^{p,q})$. Let $\pi_{\phi}$ and
$\pi_{\phi^{\perp}}$ be the orthogonal projections of
$\underline{\R} ^{p,q}$ onto $\phi$ and $\phi^{\perp}$,
respectively. Provide $\phi$ and $\phi^{\perp}$ with the connections
$\nabla^{\phi}:=\pi _{\phi}\circ d\circ\pi_{\phi}$ and
$\nabla^{\phi^{\perp}}:=\pi _{\phi ^{\perp}}\circ d\circ\pi_{\phi
^{\perp}}$, respectively. Set $\mathcal{D}_{\phi}
:=\nabla^{\phi}+\nabla^{\phi^{\perp}}$ and $\mathcal{N}_{\phi}
:=d-\mathcal{D}_{\phi}$. Under the standard identification
$\phi^{*}TGr_{(r,s)}(\R^{p,q})\cong\mathrm{Hom}(\phi,\phi^{\perp})\cong
\phi\wedge\phi^{\perp}$, of bundles provided with a metric and a
connection, $d\phi=\mathcal{N}_{\phi}$, so that the harmonicity of
$\phi$ with respect to a given conformal structure in $M$, in the
case $M$ is compact, is characterized by
$d^{\mathcal{D}_{\phi}}*\mathcal{N}_{\phi} =0$ (noting that the
connection induced naturally in $\phi\wedge \phi^{\perp}$ by
$\nabla^{\phi}$ and $\nabla^{\phi^{\perp}}$ coincides with the one
induced naturally by $\mathcal{D}_{\phi}$). K. Uhlenbeck
\cite{uhlenbeck 89} proved that the harmonicity of $\phi$ is
characterized, equivalently, by the flatness of the metric
connection $d^{\lambda}_{\phi}:=\mathcal{D}_{\phi} +\lambda
^{-1}\mathcal{N}_{\phi} ^{1,0}+\lambda \mathcal{N}_{\phi} ^{0,1}$,
for each $\lambda \in S^{1}$. Furthermore,  such loop of flat metric
connections gives rise to a $S^{1}$-family of harmonic maps into
$Gr_{(r,s)}(\R^{p,q})$, cf. \cite{uhlenbeck 89}. Harmonic maps into
Grassmannian manifolds come in $S^{1}$-families. In this section, we
show that, if $S:M\rightarrow \mathcal{G}$ is harmonic, then the
$S^{1}$-deformation of $S$ defined by the loop of flat metric
connections $d^{\lambda}_{S}$, with $\lambda\in S^{1}$, is the
family of (harmonic) central sphere congruences corresponding to the
$S^{1}$-deformation of $\Lambda$ defined by the loop
$d^{\lambda}_{S}$. The characterization of Willmore surfaces in
space-forms in terms of the harmonicity of the central sphere
congruence gives rise, in this way, to a spectral deformation of
Willmore surfaces in space-forms. As we shall verify in section
\ref{realspecofCW} below, this deformation coincides, up to
reparametrization, with the one presented in \cite{SD}.
\newline

Let $\Lambda\subset\underline{\R}^{n+1,1}$ be a surface in the
projectivized light-cone. Consider $M$ provided with the conformal
structure $\mathcal{C}_{\Lambda}$. For each $\lambda\in S^{1}$,
define a connection on $\underline{\R}^{n+1,1}$ by
$$d^{\lambda}:=\mathcal{D}+\lambda ^{-1}\mathcal{N}^{1,0}+\lambda
\mathcal{N}^{0,1},$$ noting that, $d^{\lambda}$ is, indeed, real:
$\mathcal{D}$ and $\mathcal{N}$ are real and, as $\lambda$ is unit,
$\overline{\lambda}=\lambda ^{-1}$, so that, given
$\mu\in\Gamma(\underline{\R}^{n+1,1})$,
$\overline{d^{\lambda}\mu}=d^{\lambda}\mu$.

The next result is, in view of Theorem
\ref{willmorevsharmonicitytheorem}, the particular case $\phi=S$ of
the characterization of the harmonicity of a map
$\phi:M\rightarrow\mathcal{G}$ in terms of the flatness of the
$S^{1}$-family of metric connections $d^{\lambda}_{\phi}$, by K.
Uhlenbeck \cite{uhlenbeck 89}.

\begin{thm}\label{willmoreiffdlambdaflat}
$\Lambda$ is a Willmore surface if and only if $d^{\lambda}$ is a
flat connection, for each $\lambda\in S^{1}$.
\end{thm}
\begin{proof}
According to \eqref{eq:curvtens}, and having in consideration that
there are no non-zero $(2,0)$- or $(0,2)$-forms over a surface, the
curvature tensor of $d^{\lambda}$ is given by
$$R^{d^{\lambda}}=R^{\mathcal{D}}+\lambda
^{-1}d^{\mathcal{D}}\mathcal{N}^{1,0}+\lambda\,
d^{\mathcal{D}}\mathcal{N}^{0,1}+\frac{1}{2}\,[\mathcal{N}\wedge\mathcal{N}].$$
By Gauss-Ricci and Codazzi equations, it follows that
$R^{d^{\lambda}}=(\lambda
^{-1}-\lambda)\,d^{\mathcal{D}}\mathcal{N}^{1,0}$.  We conclude that
$d^{\lambda}$ is flat for all $\lambda$ in $S^{1}$ if and only if
$d^{\mathcal{D}}\mathcal{N}^{1,0}=0$, which, according to Theorem
\ref{Willmoreeq}, completes the proof.
\end{proof}

Since $\mathcal{N}$ is skew-symmetric, the fact that $\mathcal{D}$
is a metric connection ensures that so is $d^{\lambda}$. Therefore,
if $\Lambda$ is a Willmore surface, the family of connections
$d^{\lambda}$, with $\lambda\in S^{1}$, consists of a $S^{1}$-family
of flat metric connections, defining then a $S^{1}$-family of
transformations of $\Lambda$, by setting, for each $\lambda\in
S^{1}$,
$$\Lambda _{\lambda}:=\phi _{d^{\lambda}}\Lambda ,$$ for
some isomorphism $\phi
_{d^{\lambda}}:(\underline{\R}^{n+1,1},d^{\lambda})\rightarrow
(\underline{\R}^{n+1,1},d)$. Observe that, for each $\lambda\in
S^{1}$, $\Lambda _{\lambda}$ consists of a transformation of
$\Lambda$ into another surface. In fact, given a never-zero section
$\sigma$ of $\Lambda$, we have
\begin{equation}\label{eq:dd}
d^{\lambda}\sigma=\mathcal{D}\sigma =d\sigma,
\end{equation}
and, therefore, $\Lambda^{(1)}_{d^{\lambda}}=\Lambda^{(1)}$, showing
that, $\Lambda$ is a $d^{\lambda}$-surface for all $\lambda\in
S^{1}$. Furthermore:
\begin{thm}
If $\Lambda$ is a Willmore surface, then so is the transformation
$\Lambda _{\lambda}$ of $\Lambda$ defined by the flat metric
connection $d^{\lambda}$, for each $\lambda\in S^{1}$.
\end{thm}

\begin{proof}Suppose $\Lambda$ is a Willmore surface, in which case such a
transformation $\Lambda _{\lambda}$ of $\Lambda$ is defined. Fix
$\lambda\in S^{1}$ and $\sigma$ a never-zero section of $\Lambda$.
First of all, note that, according to \eqref{eq:dd},
$g_{\sigma}^{d^{\lambda}}=g_{\sigma}$ and, therefore,
$\mathcal{C}_{\Lambda}^{d^{\lambda}}=\mathcal{C}_{\Lambda}$. For
simplicity, write $S^{\lambda}$, $\mathcal{D}^{\lambda}$ and
$\mathcal{N}^{\lambda}$ for, respectively, $S^{d^{\lambda}}$,
$\mathcal{D}^{d^{\lambda}}$ and $\mathcal{N}^{d^{\lambda}}$. The
proof will consist of showing that
$d^{\mathcal{D}^{\lambda}}*\mathcal{N}^{\lambda}=0$, for $*$ the
Hodge $*$-operator on $1$-forms over $(M,\mathcal{C}_{\Lambda})$.
The result will then follow from Proposition \ref{prop4.6.1}.

The crucial observation is that the $d^{\lambda}$-central sphere
congruence of $\Lambda $ coincides with the central sphere
congruence of $\Lambda$,
\begin{equation}\label{eq:Slambda=S}
S^{\lambda}=S,
\end{equation}
as we shall see next. For that, it is enough to fix an orthonormal
frame $(e_{i})_{i}$ of $TM$ with respect to $g_{\sigma}$ and to show
that $\sum _{i}d^{\lambda}_{e_{i}}d_{e_{i}}\sigma=\sum
_{i}d_{e_{i}}d_{e_{i}}\sigma $, or, equivalently, that $\sum
_{i}(\lambda ^{-1}\mathcal{N}^{1,0}_{e_{i}}d_{e_{i}}\sigma
+\lambda\mathcal{N}^{0,1}_{e_{i}}d_{e_{i}}\sigma)=0$, having in
consideration that $\sum _{i}\mathcal{N}_{e_{i}}d_{e_{i}}\sigma=0$.
Set $Z:=e_{1}-i\,e_{2}$. Note that, given a $2$-tensor $T$ on $M$,
$\sum
_{i}T(e_{i},e_{i})=\frac{1}{2}\,(T(Z,\overline{Z})+T(\overline{Z},Z))$,
so that, in particular, if $T$ is symmetric, $T(Z,\overline{Z})=\sum
_{i}T(e_{i},e_{i})=T(\overline{Z},Z)$. Note that the $2$-tensor
$((X,Y)\mapsto \mathcal{N}_{X}(d_{Y}\sigma))$ is symmetric: given
$X,Y\in\Gamma (TM)$,
$$\mathcal{N}_{Y}(d_{X}\sigma)-\mathcal{N}_{X}(d_{Y}\sigma)=-\pi_{S^{\perp}}(d_{[X,Y]}\sigma)=0.$$
Choosing the frame $(e_{1},e_{2})$ to be direct, we have
$Je_{1}=e_{2}$ and $Je_{2}=-e_{1}$, for $J$ the canonical complex
structure in $(M,\mathcal{C}_{\Lambda})$, and, therefore,
$Z\in\Gamma (T^{1,0})$. It follows, on the one hand, that
$$\sum
_{i}\mathcal{N}^{1,0}_{e_{i}}d_{e_{i}}\sigma=\frac{1}{2}\,(\mathcal{N}^{1,0}_{Z}d_{\overline{Z}}\sigma+\mathcal{N}^{1,0}_{\overline{Z}}d_{Z}\,\sigma)=\frac{1}{2}\,\mathcal{N}^{1,0}_{Z}d_{\overline{Z}}\sigma,$$
and, on the other hand,
$$\mathcal{N}^{1,0}_{Z}d_{\overline{Z}}\sigma=\mathcal{N}_{e_{1}}d_{e_{1}}\sigma+i\mathcal{N}_{e_{1}}d_{e_{2}}\sigma-i\mathcal{N}_{e_{2}}d_{e_{1}}\sigma+\mathcal{N}_{e_{2}}d_{e_{2}}\sigma=\sum
_{i}\mathcal{N}^{1,0}_{e_{i}}d_{e_{i}}\sigma.$$ Analogously,
$$\frac{1}{2}\,\mathcal{N}^{0,1}_{\overline{Z}}d_{Z}\,\sigma=\sum
_{i}\mathcal{N}^{0,1}_{e_{i}}d_{e_{i}}\sigma=\mathcal{N}^{0,1}_{\overline{Z}}d_{Z}\,\sigma.$$
Thus $$\sum _{i}\mathcal{N}^{1,0}_{e_{i}}d_{e_{i}}\sigma=0=\sum
_{i}\mathcal{N}^{0,1}_{e_{i}}d_{e_{i}}\sigma,$$completing the
verification of \eqref{eq:Slambda=S}. Now note that, as
$\mathcal{N}$ intertwines $S$ and $S^{\perp}$, we have
$$\pi_{S}\circ d^{\lambda}\circ\pi_{S}=\pi_{S}\circ
d\circ\pi_{S},\,\,\,\,\,\pi_{S^{\perp}}\circ
d^{\lambda}\circ\pi_{S^{\perp}}=\pi_{S^{\perp}}\circ
d\circ\pi_{S^{\perp}}.$$ We conclude that
$$\mathcal{D}^{\lambda}=\mathcal{D}$$ and, consequently,
$$\mathcal{N}^{\lambda}=\lambda^{-1}\mathcal{N}^{1,0}+\lambda
\mathcal{N}^{0,1}.$$ Hence
\begin{eqnarray*}
d^{\mathcal{D}^{\lambda}}*\mathcal{N}^{\lambda}&=&\lambda^{-1}d^{\mathcal{D}}*\mathcal{N}^{1,0}+\lambda\,
d^{\mathcal{D}}*\mathcal{N}^{0,1}=\\&=&-i\lambda^{-1}d^{\mathcal{D}}\mathcal{N}^{1,0}+i\lambda\,
d^{\mathcal{D}}\mathcal{N}^{0,1}=\\&=&-i(\lambda^{-1}+\lambda
)d^{\mathcal{D}}\mathcal{N}^{1,0},\end{eqnarray*}which, according to
Theorem \ref{Willmoreeq}, completes the proof.
\end{proof}

The harmonicity of $S:(M,\mathcal{C}_{\Lambda})\rightarrow
\mathcal{G}$, characterized, cf. K. Uhlenbeck \cite{uhlenbeck 89},
by the flatness of $d^{\lambda}_{S}$ (spelt out with respect to
$\mathcal{C}_{\Lambda}$), for all $\lambda\in S^{1}$, establishes,
equivalently, the flatness of $$d^{\mu}_{\phi _{d^{\lambda}}S}=\phi
_{d^{\lambda}}\circ d^{\lambda\mu}_{S} \circ \phi
_{d^{\lambda}}^{-1},$$ for all $\lambda,\mu\in S^{1}$, or,
equivalently, the harmonicity of $\phi
_{d^{\lambda}}S:(M,\mathcal{C}_{\Lambda})\rightarrow\mathcal{G}$,
for all $\lambda$. According to \eqref{eq:Slambda=S}, the
$S^{1}$-deformation $\phi _{d^{\lambda}}S$ of $S$ is the family of
central sphere congruences corresponding to the $S^{1}$-deformation
$\phi _{d^{\lambda}}\Lambda$ of $\Lambda$. On the other hand, in
view of \eqref{eq:dd}, $g_{\phi _{d^{\lambda}}\Lambda}=g_{\sigma}$
and, therefore,
$$\mathcal{C}_{\phi _{d^{\lambda}}\Lambda}=\mathcal{C}_{\Lambda}.$$Hence the
harmonicity of $S$ with respect to $\mathcal{C}_{\Lambda}$
establishes the harmonicity of $\phi _{d^{\lambda}}S$ with respect
to $\mathcal{C}_{\phi _{d^{\lambda}}\Lambda}$, for all $\lambda$.
The loop of flat metric connections $d^{\lambda}$ defines, in this
way, a conformal $S^{1}$-deformation of a Willmore surface into a
family of Willmore surfaces. As we shall verify in Section
\ref{realspecofCW} below, this deformation coincides, up to
reparametrization, with the one presented in \cite{SD}.

\chapter{The constrained Willmore surface equation}\label{sec:CWseq}

\markboth{\tiny{A. C. QUINTINO}}{\tiny{CONSTRAINED WILLMORE
SURFACES}}

In this chapter, we introduce constrained Willmore surfaces, the
generalization of Willmore surfaces that arises when we consider
extremals of the Willmore functional with respect to infinitesimally
conformal variations,\footnote{To which references as
\textit{conformal variations} can be found in the literature.}
rather than with respect to all variations.  The topic is mentioned
very briefly in \cite{willmore}. Some results on constrained
Willmore surfaces are contained in \cite{richter}, \cite{SD},
\cite{christoph2} and \cite{burstall+calderbank}. Constrained
Willmore surfaces in space forms form a M\"{o}bius invariant class
of surfaces, with strong links to the theory of integrable systems,
which we shall explore throughout this work. F. Burstall et al.
\cite{SD} established a manifestly conformally invariant
characterization of constrained Willmore surfaces in space-forms,
which, in particular, extended the concept of constrained Willmore
to surfaces that are not necessarily compact. This chapter is
dedicated to deriving the characterization mentioned above, or
rather a reformulation of it by F. Burstall and D. Calderbank
\cite{burstall+calderbank}, from the variational problem.\newline

Let $\Lambda\subset\underline{\R}^{n+1,1}$ be a surface in the
projectivized light-cone.  Provide $M$ with the conformal structure
$\mathcal{C}_{\Lambda}$, induced by $\Lambda$. Suppose $M$ is
compact.

\section{Constrained Willmore surfaces: definition and examples}

\markboth{\tiny{A. C. QUINTINO}}{\tiny{CONSTRAINED WILLMORE
SURFACES}}

Constrained Willmore surfaces are the critical points of the
Willmore functional with respect to infinitesimally conformal
variations. They form a M\"{o}bius invariant class of surfaces.
Willmore surfaces and constant mean curvature surfaces in
$3$-dimensional space-forms, and so their M\"{o}bius transforms, are
examples of constrained Willmore surfaces. \newline

A variation $(\Lambda_{t})_{t}$ of $\Lambda$ through immersions of
$M$ in $\mathbb{P}(\mathcal{L})$ is said to be a \emph{conformal
variation} if it preserves the conformal structure induced in $M$,
or, equivalently, it preserves the isotropy of $T^{1,0}M$
(respectively, $T^{0,1}M$), i.e., fixing $Z\in\Gamma(T^{1,0}M)$
(respectively, $Z\in\Gamma(T^{0,1}M)$), locally never-zero, and, for
each $t$, $g _{t}$ in the conformal class of metrics induced in $M$
by $\Lambda_{t}$,
$$g_{t}(Z,Z)=0.$$The constraint on the conformal structure we are interested
in is weaker than conformality.

\begin{defn}
A variation $(\Lambda_{t})_{t}$ of $\Lambda$ through immersions of
$M$ in $\mathbb{P}(\mathcal{L})$ is said to be an
\emph{infinitesimally conformal variation} if, fixing
$Z\in\Gamma(T^{1,0}M)$ (respectively, $Z\in\Gamma(T^{0,1}M)$),
locally never-zero,  and, for each $t$, $g _{t}$ in the conformal
class of metrics induced in $M$ by $\Lambda_{t}$, we have
$$\frac{d}{dt}_{\mid_{t=0}} g_{t}(Z,Z)=0.$$
\end {defn}

Note that this is, indeed, a good definition, as, given
$\lambda=(t\mapsto\lambda_{t})$, with $\lambda_{t}:M\rightarrow \R$
positive for each $t$,
$$\frac{d}{dt}_{\mid_{t=0}} \lambda_{t}g_{t}(Z,Z)=\lambda'(0)g(Z,Z)+\lambda(0)\frac{d}{dt}_{\mid_{t=0}} g_{t}(Z,Z),$$
with $g\in\mathcal{C}_{\Lambda}$, so that $\frac{d}{dt}_{\mid_{t=0}}
\lambda_{t}g_{t}(Z,Z)=\lambda(0)\frac{d}{dt}_{\mid_{t=0}}
g_{t}(Z,Z)$, which vanishes if and only if
$\frac{d}{dt}_{\mid_{t=0}} g_{t}(Z,Z)$ does.

\begin{defn}
The surface $\Lambda$ is said to be a \emph{constrained Willmore
surface} if
$$\frac{d}{dt}_{\mid_{t=0}}W(\Lambda _{t})=0$$ for every infinitesimally conformal variation
$(\Lambda _{t})_{t}$ of $\Lambda$ through immersions of $M$ in
$\mathbb{P} (\mathcal {L})$.
\end{defn}

It is immediate from Theorem \ref{Wenpres} that conformal
diffeomorphisms transform constrained Willmore surfaces into
constrained Willmore surfaces.

\begin{thm}\label{Willmandconfdiffeom}
The class of constrained Willmore surfaces is M\"{o}bius invariant.
\end{thm}
In particular:
\begin{prop}\label{LambdaconstWillmiffsigmaconstWillm}
$\Lambda$ is a constrained Willmore surface if and only if, fixing a
non-zero $v_{\infty}\in\mathbb{R}^{n+1,1}$, so is the surface in
$S_{v_{\infty}}$ defined by $\Lambda$.
\end{prop}

In Section \ref{CCWS}, we extend the concept of constrained Willmore
surface to surfaces that are, in particular, not necessarily
compact.

Willmore surfaces are, obviously, examples of constrained Willmore
surfaces. But there are more: constant mean curvature (CMC) surfaces
in $3$-dimensional space-forms are constrained Willmore (and also
isothermic). Section \ref{sec:CMC} is dedicated to this special
class of surfaces. We believe one can obtain non-isothermic,
non-Willmore constrained Willmore surfaces as \textit{B\"{a}cklund
transforms} of non-minimal CMC surfaces in space-forms, following
Section \ref{isoCWtransforms} below, but this is not clear, though
(it shall be the subject of further work). Section \ref{HMCsurfs} is
dedicated to another class of constrained Willmore surfaces, that of
codimension $2$ surfaces with holomorphic mean curvature vector in
space-forms.\newline

\section{The Hopf differential and the Schwarzian derivative}

\markboth{\tiny{A. C. QUINTINO}}{\tiny{CONSTRAINED WILLMORE
SURFACES}}

In \cite{SD}, a characterization of constrained Willmore surfaces,
isothermic surfaces and constant mean curvature surfaces in
space-forms in terms of the Hopf differential and the Schwarzian
derivative is established. It is a uniform characterization of these
three classes of surfaces and, for this reason, it will be presented
in this text, in parallel to what is our main approach. In this
section, we introduce the Hopf differential and the Schwarzian
derivative, cf. \cite{SD}.\newline

Fix $z=x+iy$ a holomorphic chart of $M$ and consider $\sigma^{z}$,
the normalized section of $\Lambda$ with respect to $z$. Write
$$\sigma_{zz}^{z}=a\sigma^{z}+b\sigma_{z}^{z}+c\sigma_{\bar{z}}^{z}+d\sigma^{z}_{z\bar{z}}+\pi_{S^{\perp}}\sigma_{zz}^{z},$$
with $a,b,c,d\in\Gamma(\underline{\C})$. By the orthogonality
relations of the frame $\{\sigma,\sigma _{z},\sigma
_{\bar{z}},\sigma _{z\bar{z}}\}$, we have
$$b\frac{1}{2}=(\sigma^{z}_{zz},
\sigma^{z}_{\bar{z}})=(\sigma^{z}_{z},
\sigma^{z}_{\bar{z}})_{z}-(\sigma_{z}^{z},\sigma_{z\bar{z}})=0$$
and, therefore, $b=0$; on the other hand,
$$c\frac{1}{2}=(\sigma_{zz}^{z},
\sigma_{z}^{z})=\frac{1}{2}\,(\sigma_{z}^{z},\sigma_{z}^{z})_{z}=0$$
and, therefore, $c=0$; and
$$-d\frac{1}{2}=(\sigma_{zz}^{z},
\sigma^{z})=(\sigma_{z}^{z},\sigma^{z})_{z}-(\sigma_{z}^{z},
\sigma_{z}^{z})=0,$$showing that $d=0$. Thus $\sigma_{zz}^{z}$
satisfies an equation
\begin{equation}\label{eq:sigmazzck}
\sigma_{zz}^{z}=-\frac{1}{2}\,c^{z}\sigma^{z}+k^{z},
\end{equation}
defining a complex function
$$c^{z}:=4(\sigma_{zz}^{z},\sigma_{z\bar{z}}^{z})\in\Gamma(\underline{\C}),$$
the \textit{Schwarzian derivative} with respect to $z$, and a
section
$$k^{z}:=\pi_{S^{\perp}}(\sigma_{zz}^{z})\in\Gamma(S^{\perp}),$$of the complexification of the normal bundle to the central
sphere congruence of $\Lambda$, called the \textit{Hopf}
\textit{differential} of $\Lambda$ with respect to $z$.\footnote{The
terminology for the latter is motivated by the relation established
in Lemma \ref{hopfdiffervsclassical} below, in Appendix
\ref{appHopf+umbilics}.}

It is useful to understand how the Hopf differential changes under a
change of holomorphic coordinate. Let $\omega$ be another
holomorphic chart of $M$. Following equation \eqref{eq:gomegagz}, we
have $$\sigma^{w}=\mid\omega_{z}\mid\sigma^{z},$$ so that
$\sigma^{\omega}_{\omega}=\omega_{z}^{-1}\mid\omega_{z}\mid_{z}\sigma^{z}+\omega_{z}^{-1}\mid\omega_{z}\mid\sigma^{z}_{z}$
and, consequently,
\begin{eqnarray*}
\sigma^{\omega}_{\omega\omega}&=&
\frac{\mid\omega_{z}\mid}{\omega_{z}^{2}}\left(\left(\frac{\omega_{zz}}{2\omega_{z}}\right)_{z}-\left(\frac{\omega_{zz}}{2\omega_{z}}\right)^{2}-\frac{1}{2}\,c^{z}\right)
\sigma^{z} +\frac{\mid\omega_{z}\mid}{\omega_{z}^{2}}\,k^{z}.
\end{eqnarray*}
Hence
\begin{equation}\label{eq:komega}
k^{\omega}=\frac{\mid\omega_{z}\mid}{\omega_{z}^{2}}\,k^{z}
\end{equation}

We complete this section with a fundamental result of conformal
surface theory. As established in \cite{burstall+calderbank},
\begin{Lemma}\label{thmonHopfeSchwarz}
If $\Lambda_{1},\Lambda_{2}:M\rightarrow
\mathbb{P}(\mathcal{L})\cong S^{n}$ are two conformal immersions
inducing the same Hopf differential, Schwarzian derivative and
normal connection $\nabla^{S^{\perp}}$, then there is a M\"{o}bius
transformation $T$ such that $\Lambda_{2}=T\Lambda_{1}$. (In the
particular case of codimension $1$ (i.e., $n=3$), the condition on
the normal connection is vacuous.)
\end{Lemma}

\section{The Euler-Lagrange constrained Willmore surface
equation}\label{subsec:cwseq}

\markboth{\tiny{A. C. QUINTINO}}{\tiny{CONSTRAINED WILLMORE
SURFACES}}

F. Burstall et al. \cite{SD} established a manifestly conformally
invariant characterization of constrained Willmore surfaces in
space-forms, which, in particular, extended the concept of
constrained Willmore to surfaces that are not necessarily compact.
In this section, we derive, from the variational problem, a
reformulation of this characterization, due to F. Burstall and D.
Calderbank \cite{burstall+calderbank}. The argument consists of a
generalization to $n$-space of the argument presented in
\cite{christoph2} for the particular case of $n=3$.\newline

As established in \cite{burstall+calderbank}:

\begin{thm}\label{CWeq}
The surface $\Lambda$ is a constrained Willmore surface if and only
if there exists a real form $q\in\Omega ^{1}(\Lambda\wedge\Lambda
^{(1)})$ with
\begin{equation}\label{eq:curlyDextderivofq}
d^{\mathcal{D}}q=0
\end{equation}
such that
\begin{equation}\label{eq:mainCWeq}
d^{\mathcal{D}}*\mathcal{N}=2\,[q\wedge *\mathcal{N}].
\end{equation}
In this case, we may refer to $\Lambda$ as, specifically, a
\emph{$q$-constrained Willmore surface} and to $q$ as a
\emph{[Lagrange] multiplier}\footnote{Named after the method of
Lagrange multipliers for finding the critical points of a function
subject to a constraint.} to $\Lambda$.
\end{thm}

Willmore surfaces are the $0$-constrained Willmore surfaces. The
zero multiplier is not necessarily the only multiplier to a
constrained Willmore surface with no constraint on the conformal
structure. The uniqueness of multiplier is discussed in section
\ref{nonisoCW}.

\begin{rem}\label{remmuyutilcurlyDNqCWeq}
Let $q$ be a $1$-form with values in $\Lambda\wedge\Lambda^{(1)}$.
According to \eqref{eq:starblablastar},
\begin{equation}\label{eq:dddfddvlllaalaooaoajjj}
d^{\mathcal{D}}*\mathcal{N}\in\Omega^{2}(\Lambda\wedge S^{\perp}).
\end{equation}
On the other hand, in view of \eqref{eq:wegdebrack}, we conclude
from \eqref{eq:caklNLambdaperpSperp} that
\begin{equation}\label{eq:????????????b}
[q\wedge *\mathcal{N}]\in\Omega^{2}(\Lambda\wedge S^{\perp}).
\end{equation}
In particular, both $d^{\mathcal{D}}*\mathcal{N}$ and $[q\wedge
*\mathcal{N}]$ vanish on $\Lambda^{(1)}$ and are, therefore,
determined by the respective restrictions to $\langle u\rangle\oplus
S^{\perp}$, fixing $u\in S\backslash\Lambda^{\perp}$ (in particular,
for $u=\sigma_{z\bar{z}}$, fixing $\sigma\in\Gamma(\Lambda)$
never-zero and $z$ a holomorphic chart of $M$). Equation \eqref
{eq:transposeVwedgeVperp} establishes, furthermore, that both
$d^{\mathcal{D}}*\mathcal{N}$ and $[q\wedge *\mathcal{N}]$ are
determined by the respective restrictions to $\langle u\rangle$,
fixing $u\in S\backslash\Lambda^{\perp}$. In particular, equation
\eqref{eq:mainCWeq} holds if and only if, fixing such a $u$,
$(d^{\mathcal{D}}*\mathcal{N})u=2\,[q\wedge *\mathcal{N}]u$.
\end{rem}

The proof of Theorem \ref{CWeq} we present follows essentially from
the so called Weyl's lemma (see, for example, \cite{livroweyl}, \S
$9$), which, in particular, establishes the image of the operator
$\bar{\partial}$ as the orthogonal space to the vector space of the
holomorphic quadratic differentials, with respect to some
non-degenerate pairing. We start by establishing a $1-1$
correspondence between quadratic differentials and real $1$-forms
taking values in $\Lambda\wedge\Lambda^{(1)}$ whose $(1,0)$-part
takes values in $\Lambda\wedge\Lambda^{0,1}$ (for $\Lambda^{0,1}$ as
defined next). We verify that condition \eqref{eq:curlyDextderivofq}
on $q$ forces $q^{1,0}$ to take values in
$\Lambda\wedge\Lambda^{0,1}$ and that, under the correspondence
above, condition \eqref{eq:curlyDextderivofq} is equivalent to the
holomorphicity of the quadratic differential $q$. And then we
proceed to the proof of the theorem, showing that $\Lambda$ is a
constrained Willmore surface if and only if there exists a
holomorphic quadratic differential $q$ satisfying equation
\eqref{eq:mainCWeq}, under the correspondence mentioned above.

First of all, fix a never-zero section $\sigma$ of $\Lambda$ and a
holomorphic chart $z$ of $M$. Independently of the choice of such a
$\sigma$ and such a $z$, set
$$\Lambda ^{1,0}:=\langle \sigma ,\sigma _{z}\rangle$$ and
$$\Lambda ^{0,1}:=\langle \sigma ,\sigma _{\bar{z}}\rangle
,$$ defining in this way two isotropic complex rank $2$ subbundles
of $S$, complex conjugate of each other. In view of the
non-degeneracy of $S$, given $i\neq j\in\{0,1\}$,
$\mathrm{rank}\,S=\mathrm{rank}\,\Lambda^{i,j}+\mathrm{rank}\,(\Lambda^{i,j})^{\perp}$,
showing that the isotropy of $\Lambda^{1,0}$ and $\Lambda^{0,1}$
establishes, furthermore, their maximal isotropy in $S$,
$$(\Lambda^{1,0})^{\perp}\cap S=\Lambda^{1,0},\,\,\,(\Lambda^{0,1})^{\perp}\cap S=\Lambda^{0,1}.$$
Note that
$$\Lambda =\Lambda ^{1,0}\cap\Lambda ^{0,1}.$$Clearly, $$\Lambda ^{1,0}\wedge \Lambda ^{1,0}=\langle \sigma \wedge
\sigma _{z}\rangle =\Lambda \wedge \Lambda ^{1,0},$$ as well as
$$\Lambda ^{0,1}\wedge\Lambda ^{0,1}=\langle \sigma \wedge \sigma
_{\bar{z}}\rangle =\Lambda \wedge \Lambda ^{0,1}.$$The orthogonality
relations of the frame $\sigma,\sigma _{z},\sigma _{\bar{z}}$ of
$$\Lambda ^{(1)}=\langle\sigma ,\sigma _{z},\sigma
_{\bar{z}}\rangle$$ show that $\sigma \wedge \sigma _{z}$ and
$\sigma \wedge \sigma _{\bar{z}}$ are linearly independent and,
therefore, that
$$\Lambda \wedge\Lambda ^{(1)}=\Lambda \wedge \Lambda
^{1,0}\oplus\Lambda \wedge \Lambda ^{0,1}.$$

\begin{rem}\label{e10determinedby}
Given $\xi$ a section of $\Lambda\wedge\Lambda^{1,0}$ (respectively,
$\Lambda\wedge\Lambda^{0,1}$),  $\xi$ vanishes everywhere outside
$\langle\sigma_{\bar{z}},\sigma_{z\bar{z}}\rangle$ (respectively,
$\langle\sigma_{z},\sigma_{z\bar{z}}\rangle$) and
$$\xi\sigma_{\bar{z}}=\lambda\sigma,\,\,\,\,\,\xi\sigma_{z\bar{z}}=\lambda\sigma_{z},$$
(respectively,
$\xi\sigma_{z}=\lambda\sigma,\,\,\,\xi\sigma_{z\bar{z}}=\lambda\sigma_{\bar{z}}$),
for some $\lambda\in C^{\infty}(M,\R)$. In particular, $\xi$ is
determined by $\xi\sigma_{\bar{z}}$ (respectively, $\xi
\sigma_{z}$), or, equivalently, by $\xi\sigma_{z\bar{z}}$, or, yet
again, equivalently, by $\xi u$, fixing $u\in S\backslash
\Lambda^{\perp}$.
\end{rem}

Observe that
\begin{equation}\label{eq:mathcalDijpreservesLambdaij}
\mathcal{D}^{1,0}\,\Gamma (\Lambda ^{1,0})\subset \Omega
^{1,0}(\Lambda ^{1,0}),\,\,\,\,\,\mathcal{D}^{0,1}\,\Gamma (\Lambda
^{0,1})\subset \Omega ^{0,1}(\Lambda ^{0,1}).
\end{equation}
Indeed, given $\lambda,\mu\in\Gamma (\underline{\C})$,
$\mathcal{D}_{\delta _{z}}(\lambda\sigma^{z} +\mu\sigma
_{z}^{z})=\lambda _{z}\,\sigma^{z} +(\lambda+\mu _{z})\,\sigma^{z}
_{z}+\mu\,\pi _{S}(\sigma^{z} _{zz})\in\Gamma(\Lambda^{1,0})$ and,
similarly, we verify that $\mathcal{D}_{\delta
_{\bar{z}}}(\lambda\sigma^{z} +\mu\sigma
_{\bar{z}}^{z})\in\Gamma(\Lambda^{0,1})$.
 Following
\eqref{eq:mathcalDijpreservesLambdaij}, we get
\begin{equation}\label{eq:curlyD1oLambdaijblablalbla}
\mathcal{D}^{1,0}\,\Gamma (\Lambda \wedge\Lambda^{1,0})\subset\Omega
^{1,0} (\Lambda \wedge\Lambda
^{1,0}),\,\,\,\,\,\mathcal{D}^{0,1}\,\Gamma (\Lambda \wedge\Lambda ^
{0,1})\subset\Omega ^{0,1} (\Lambda \wedge\Lambda ^{0,1}).
\end{equation}
Observe, on the other hand, that, as $\sigma_{\bar{z}}$ is a section
of $S$,
$\mathcal{N}_{\delta_{z}}\sigma_{\bar{z}}=\pi_{S^{\perp}}(d_{\delta_{z}}\sigma_{\bar{z}})=0$,
and, therefore, in view of \eqref{eq:calNLambda0},
$\mathcal{N}^{1,0}\Lambda^{0,1}=0$, or, equivalently,
$\mathcal{N}^{0,1}\Lambda^{0,1}=0$,
\begin{equation}\label{eq:curlyN10Lambda01bbbbb09rgg9987654}
\mathcal{N}^{1,0}\Lambda^{0,1}=0=\mathcal{N}^{0,1}\Lambda^{0,1}.
\end{equation}
It is opportune to observe that, in view of the skew-symmetry of
$\mathcal{N}$ and of the maximal isotropy of $\Lambda^{1,0}$ and
$\Lambda^{0,1}$ in $S$, it follows, in particular,
\begin{equation}\label{eq:curlyN10SperpinLambda01}
\mathcal{N}^{1,0}S^{\perp}\subset
\Lambda^{0,1},\,\,\,\,\mathcal{N}^{0,1}S^{\perp}\subset
\Lambda^{1,0}.
\end{equation}

Now recall that a quadratic differential is a  $2$-tensor
represented locally, in the domain of $z$, as
$Q:=f^{z}dz^{2}+\overline{f^{z}}d\bar{z}^{2}$, with $f^{z}\in
C^{\infty}(M,\C)$. We may refer to $Q$ as \textit{the quadratic
differential defined locally by $f^{z}dz^{2}$}. The transformation
rule for $f^{z}$ under change of holomorphic coordinates is
\begin{equation}\label{eq:qwvsqz}
f^{\omega}=\omega_{z}^{-2}\,f^{z},
\end{equation}
given another holomorphic chart $\omega$. $Q$ is said to be
holomorphic if $f^{z}$ is holomorphic.\footnote{Noting that, given
another holomorphic chart $\omega$,
$f^{\omega}_{\bar{z}}=\omega_{z}^{-2}f^{z}_{\bar{z}}$, so that
$f^{\omega}_{\bar{z}}$ vanishes if and only if so does
$f^{z}_{\bar{z}}$.} We denote the vector space of holomorphic
quadratic differentials on $M$ by $H^{0}(K)$. It is well-known that
the vector space of holomorphic quadratic differentials on a compact
Riemann surface is finite dimensional (see, for example,
\cite{jost}).

Given a $1$-form $q\in\Omega^{1}(\Lambda\wedge\Lambda^{(1)})$, we
define a quadratic differential by
$$q_{Q}:=q^{z}dz^{2}+\overline{q^{z}}d\bar{z}^{2},$$ for $q^{z}\in C^{\infty}(M,\C)$
defined by\footnote{Scaling $q_{\delta_{z}}\sigma_{z}$ by $-2$ will
avoid some extra scaling in future equations.}
\begin{equation}\label{eq:defdeQforrealcase}
q^{z}\sigma:=-2\,q_{\delta_{z}}\sigma_{z}.
\end{equation}
We shall verify that this is, indeed, well-defined. The independence
of \eqref {eq:defdeQforrealcase} with respect to $\sigma$ is a
consequence of the fact that $q\Lambda=0$. On the other hand, the
fact that $\sigma_{z}$ is a section of $\Lambda^{\perp}$ ensures
that $q_{\delta_{z}}\sigma_{z}\in\Gamma(\langle\sigma\rangle)$,
determining $q^{z}\in\Gamma(\underline{\C})$.

A simple, yet crucial, observation is that, in view of Remark
\ref{e10determinedby}, in the case $q^{1,0}\in\Omega ^{1,0}(\Lambda
\wedge\Lambda ^{0,1})$, $q^{z}$ is, equivalently, defined by
\begin{equation}\label{eq:newdefofQqrealcase}
q^{z}\tau=-2\,q_{\delta_{z}}(\mathcal{D}_{\delta_{z}}\tau),
\end{equation}
for every $\tau\in\Gamma(\Lambda^{0,1})$.

Conversely, a quadratic differential
$Q=f^{z}dz^{2}+\overline{f^{z}}d\bar{z}^{2}$ determines a real
$1$-form $q\in\Omega^{1}(\Lambda\wedge\Lambda^{(1)})$ with
$q^{1,0}\in\Omega^{1,0}(\Lambda\wedge\Lambda^{0,1})$ and $Q=q_{Q}$
by setting $q_{\delta_{z}}\sigma_{z}:=-\frac{1}{2}\,f^{z}\sigma$. We
have established a $1-1$ correspondence
\begin{equation}\label{eq:correspondenec between qzand qsharp}
q\leftrightarrow q_{Q}
\end{equation}
between real forms $q\in\Omega^{1}(\Lambda\wedge\Lambda^{(1)})$ with
$q^{1,0}\in\Omega^{1,0}(\Lambda\wedge\Lambda^{0,1})$ and quadratic
differentials.

\begin{Lemma}\label{holomvsextderivq}
Suppose $q\in\Omega^{1}(\Lambda\wedge\Lambda^{(1)})$ is real and
$q^{1,0}\in\Omega ^{1,0}(\Lambda \wedge\Lambda ^{0,1})$. In that
case, the quadratic differential $q_{Q}$ is holomorphic if and only
if $d^{\mathcal{D}}q=0$.
\end{Lemma}

\begin{proof}
In view of equation \eqref{eq:LiBrdelta0}, the $2$-form
$d^{\mathcal{D}}q$ vanishes if and only if
$$d^{\mathcal{D}}q\,(\delta_{z},\delta_{\bar{z}})=\mathcal{D}_{\delta_{z}}\circ
q_{\delta_{\bar{z}}}-q_{\delta_{\bar{z}}}\circ\mathcal{D}_{\delta_{z}}-\mathcal{D}_{\delta_{\bar{z}}}\circ
q_{\delta_{z}}+q_{\delta_{z}}\circ\mathcal{D}_{\delta_{\bar{z}}}$$
does. As $q$ is real and $q^{1,0}$ takes values in $\Lambda
\wedge\Lambda ^{0,1}$,
$q^{0,1}=\overline{q^{1,0}}\in\Omega^{0,1}(\Lambda\wedge\Lambda^{1,0})$
and, therefore,
$$q^{1,0}\Lambda^{0,1}=0=q^{0,1}\Lambda^{1,0};$$
and, on the other hand, $$qS^{\perp}=0.$$ In particular, as
$S^{\perp}$ is $\mathcal{D}$-parallel,
$$d^{\mathcal{D}}q\,(\delta_{z},\delta_{\bar{z}})\,\sigma^{z}=0=d^{\mathcal{D}}q\,(\delta_{z},\delta_{\bar{z}})\,s^{\perp},$$
for all $s^{\perp}\in\Gamma(S^{\perp})$. On the other hand, having
in consideration that $\pi_{S}(\sigma_{zz}^{z})\in\Gamma(\Lambda)$
and according to equation \eqref{eq:newdefofQqrealcase},
$$d^{\mathcal{D}}q\,(\delta_{z},\delta_{\bar{z}})\,\sigma_{z}^{z}=-\mathcal{D}_{\delta_{\bar{z}}}(
q_{\delta_{z}}\sigma_{z}^{z})+q_{\delta_{z}}\sigma_{z\bar{z}}^{z}=\frac{1}{2}\,\mathcal{D}_{\delta_{\bar{z}}}(q^{z}\sigma^{z})-\frac{1}{2}\,q^{z}\sigma_{\bar{z}}^{z}=
\frac{1}{2}\,q^{z}_{\bar{z}}\sigma^{z}.$$ We conclude that $q_{Q}$
is holomorphic, $q^{z}_{\bar{z}}=0$, if and only if
$d^{\mathcal{D}}q\,(\delta_{z},\delta_{\bar{z}})\,\sigma_{z}^{z}$
vanishes. In its turn, by the reality of $q$ (and that of
$\mathcal{D}$ and of $\sigma^{z}$),
$$d^{\mathcal{D}}q(\delta_{z},\delta_{\bar{z}})\,\sigma_{\bar{z}}^{z}=-\overline{d^{\mathcal{D}}q(\delta_{z},\delta_{\bar{z}})\,\sigma_{z}^{z}}.$$
Lastly, we contemplate
$$d^{\mathcal{D}}q(\delta_{z},\delta_{\bar{z}})\,\sigma_{z\bar{z}}^{z}=
\overline{\mathcal{D}_{\delta_{\bar{z}}}(q_{\delta_{z}}\sigma_{z\bar{z}}^{z})}-q_{\delta_{\bar{z}}}(\pi_{S}(\sigma_{zz\bar{z}}^{z}))-\mathcal{D}_{\delta_{\bar{z}}}(q_{\delta_{z}}\sigma_{z\bar{z}}^{z})+\overline{q_{\delta_{\bar{z}}}(\pi_{S}(\sigma_{zz\bar{z}}^{z}))}.$$
First of all,
$$\pi_{S}(\sigma_{zz\bar{z}}^{z})=\pi_{S}(-\frac{1}{2}\,c^{z}\sigma^{z}+k^{z})_{\bar{z}}=-\frac{1}{2}\,c^{z}\sigma^{z}_{\bar{z}}-\frac{1}{2}\,c^{z}_{\bar{z}}\,\sigma^{z}+\pi_{S}(k^{z}_{\bar{z}}),$$
so that
$$q_{\delta_{\bar{z}}}(\pi_{S}(\sigma_{zz\bar{z}}^{z}))=-\frac{1}{2}\,c^{z}q_{\delta_{\bar{z}}}\sigma^{z}_{\bar{z}}+q_{\delta_{\bar{z}}}k^{z}_{\bar{z}}=\frac{1}{4}\,c^{z}\,\overline{q^{z}}\,\sigma^{z}+q_{\delta_{\bar{z}}}k^{z}_{\bar{z}}.$$
Together with
$(k^{z},\sigma^{z}_{\bar{z}})=0=(k^{z},\sigma^{z}_{z\bar{z}})$,
differentiation of $(k^{z},\sigma^{z})=0=(k^{z},\sigma^{z}_{z})$
shows that
$(k^{z}_{\bar{z}},\sigma^{z})=0=(k^{z}_{\bar{z}},\sigma^{z}_{z})$,
or, equivalently, that
\begin{equation}\label{eq:piSkbarzinLambda10}
\pi_{S}k^{z}_{\bar{z}}\in\Gamma(\Lambda^{1,0}).
\end{equation}
Hence $q_{\delta_{\bar{z}}}k^{z}_{\bar{z}}=0$. On the other hand,
$$\mathcal{D}_{\delta_{\bar{z}}}(q_{\delta_{z}}\sigma_{z\bar{z}}^{z})=-\frac{1}{2}\,q^{z}\pi_{S}(\sigma_{\bar{z}\bar{z}}^{z})-\frac{1}{2}\,q^{z}_{\bar{z}}\,\sigma_{\bar{z}}^{z}=\frac{1}{4}\,\overline{c^{z}}\,q^{z}\sigma^{z}-\frac{1}{2}\,q^{z}_{\bar{z}}\,\sigma_{\bar{z}}^{z}.$$
We conclude that
$$d^{\mathcal{D}}q(\delta_{z},\delta_{\bar{z}})\,\sigma_{z\bar{z}}^{z}=-\frac{1}{2}\,\overline{q^{z}_{\bar{z}}}\,\sigma_{z}^{z}+\frac{1}{2}\,q^{z}_{\bar{z}}\,\sigma_{\bar{z}}^{z}$$
vanishes if and only if $q^{z}_{\bar{z}}$ does, which completes the
proof.
\end{proof}

The next result establishes, in particular, that if $q$ is a
multiplier to $\Lambda$ then $q^{1,0}$ takes values in $\Lambda
\wedge\Lambda ^{0,1}$.
\begin{Lemma}\label{withvswithoutdecomps}
Given $q\in\Omega ^{1}(\Lambda\wedge\Lambda ^{(1)})$ real,

$i)$\,\,if $d^{\mathcal{D}}q=0$ then $q^{1,0}\in\Omega
^{1,0}(\Lambda \wedge\Lambda ^{0,1})$ or, equivalently,
$q^{0,1}\in\Omega ^{0,1}(\Lambda \wedge\Lambda ^{1,0})$;

$ii)$\,\,$d^{\mathcal{D}}q=0$ if and only if
$d^{\mathcal{D}}q^{1,0}=0$, or, equivalently,
$d^{\mathcal{D}}q^{0,1}=0;$

$iii)$\,\,$d^{\mathcal{D}}q=0$ if and only if $d^{\mathcal{D}}*q=0$.
\end{Lemma}

Before we proceed to the proof of the lemma, observe that a section
$\xi$ of $\Lambda \wedge\Lambda ^{(1)}$ is a section of $\Lambda
\wedge \Lambda ^{1,0}$, or, equivalently, $\overline{\xi}$ is a
section of $\Lambda \wedge \Lambda ^{0,1}$, if and only if
$\xi(\sigma _{z})=0$.

Now we proceed to the proof of Lemma \ref{withvswithoutdecomps}.
\begin{proof}
To prove $i)$, we prove, equivalently, that, if
$d^{\mathcal{D}}q=0$, then $q_{\delta_{\bar{z}}}\sigma_{z}^{z}=0$.
First, if $d^{\mathcal{D}}q=0$, then, in particular,
$d^{\mathcal{D}}q\,(\delta_{z},\delta_{\bar{z}})\,\sigma^{z}=0$, or,
equivalently, $$\mathcal{D}_{\delta_{z}}(q_{\delta
_{\bar{z}}}\sigma^{z})-q_{\delta_{\bar{z}}}(\mathcal{D}_{\delta_{z}}\sigma^{z})-\mathcal{D}_{\delta_{\bar{z}}}(q_{\delta_{z}}\sigma^{z})+q_{\delta_{z}}(\mathcal{D}_{\delta_{\bar{z}}}\sigma^{z})=0,$$
establishing
\begin{equation}\label{eq:qsigmazdeltabarzigualqsigmabarzdeltaz}
q_{\delta_{z}}\sigma^{z}_{\bar{z}}=q_{\delta_{\bar{z}}}\sigma_{z}^{z}.
\end{equation}
In its turn,
 $d^{\mathcal{D}}q\,(\delta_{z},\delta_{\bar{z}})\,\sigma_{z}^{z}=0$ implies
$$\mathcal{D}_{\delta_{z}}(q_{\delta
_{\bar{z}}}\sigma_{z}^{z})-q_{\delta_{\bar{z}}}(\pi_{S}(\sigma^{z}
_{zz}))-\mathcal{D}_{\delta_{\bar{z}}}(q_{\delta_{z}}\sigma_{z}^{z})+q_{\delta_{z}}
\sigma_{z\bar{z}}^{z}=0,$$or, equivalently,
$$\mathcal{D}_{\delta_{z}}(q_{\delta
_{\bar{z}}}\sigma_{z}^{z})-\mathcal{D}_{\delta_{\bar{z}}}(q_{\delta_{z}}\sigma_{z}^{z})+q_{\delta_{z}}
\sigma_{z\bar{z}}^{z}=0.$$ On the other hand, the orthogonality
relations between $\sigma^{z}$, $\sigma_{z}^{z}$,
$\sigma_{\bar{z}}^{z}$ and $\sigma_{z\bar{z}}^{z}$ show that
\begin{equation}\label{eq:que}
q\sigma_{z\bar{z}}^{z}=\mu\sigma_{z}+\eta\sigma_{\bar{z}}^{z},
\end{equation}
for some $\mu,\eta\in\Omega^{1}(\underline{\C})$. Hence
$$\mathcal{D}_{\delta_{z}}(q_{\delta
_{\bar{z}}}\sigma_{z}^{z})+\mu_{\delta
_{z}}\sigma_{z}^{z}=\mathcal{D}_{\delta_{\bar{z}}}(q_{\delta_{z}}
\sigma_{z}^{z})-\eta_{\delta_{z}}\sigma_{\bar{z}}^{z}.$$ It is
obvious that a section of $\Lambda\wedge\Lambda ^{(1)}$ transforms
sections of $\Lambda^{(1)}$ into sections of $\Lambda$, so that, in
particular, both $q_{\delta _{\bar{z}}}\sigma_{z}^{z}$ and
$q_{\delta_{z}} \sigma_{z}^{z}$ are sections of $\Lambda$. By
\eqref{eq:mathcalDijpreservesLambdaij}, we conclude that
$\mathcal{D}_{\delta_{z}}(q_{\delta
_{\bar{z}}}\sigma_{z}^{z})+\mu_{\delta _{z}}\sigma_{z}^{z}$ is a
section of $\Lambda^{1,0}\cap\Lambda^{0,1}=\Lambda$. Write
\begin{equation}\label{eq:ri}
q_{\delta _{\bar{z}}}\sigma_{z}^{z}=\lambda\sigma^{z},
\end{equation}
with $\lambda\in\Gamma(\underline{\C})$. Then $\lambda_{z}\sigma^{z}
+(\lambda +\mu_{\delta _{z}})\sigma_{z}^{z}=\gamma\sigma^{z}$, for
some $\gamma\in\Gamma (\underline{\C})$. In particular,
$\lambda=-\mu_{\delta _{z}}$. Equation \eqref{eq:que} shows, on the
other hand, that $q=-2\mu\,
\sigma^{z}\wedge\sigma^{z}_{z}-2\eta\,\sigma^{z}\wedge\sigma^{z}_{\bar{z}}$
and, consequently, equations
\eqref{eq:qsigmazdeltabarzigualqsigmabarzdeltaz} and \eqref{eq:ri}
together give $\lambda=\mu_{\delta_{z}}$, completing the proof of
$i)$.

Next we prove $ii)$. In view of
\eqref{eq:curlyD1oLambdaijblablalbla}, and following $i)$, we have
$$d^{\mathcal{D}}q^{1,0}\in\Omega ^{2}(\Lambda\wedge\Lambda
^{0,1}),\,\,\,\,\,d^{\mathcal{D}}q^{0,1}\in\Omega ^{2}(\Lambda
\wedge\Lambda ^{1,0}).$$ Hence, as $\Lambda \wedge\Lambda ^{1,0}$
and $\Lambda \wedge\Lambda ^{0,1}$ are complementary in
$\Lambda\wedge\Lambda^{(1)}$, $d^{\mathcal{D}}q=0$ forces
$d^{\mathcal{D}}q^{1,0}$ and $d^{\mathcal{D}}q^{0,1}$ to vanish
separately, $d^{\mathcal{D}}q^{1,0}=0=d^{\mathcal{D}}q^{0,1}$. On
the other hand, as $q$ is real,
$$d^{\mathcal{D}}q^{0,1}(\delta_{z},\delta_{\bar{z}})=-\overline{d^{\mathcal{D}}q^{1,0}(\delta_{z},\delta_{\bar{z}})},$$
so that $d^{\mathcal{D}}q^{1,0}$ vanishes if and only if
$d^{\mathcal{D}}q^{0,1}$ does. In particular, if
$d^{\mathcal{D}}q^{1,0}$ vanishes, then, obviously, so does
$d^{\mathcal{D}}q=d^{\mathcal{D}}q^{1,0}+d^{\mathcal{D}}q^{0,1}$.

As for $iii)$, it is immediate from $ii)$, as
$*q=-iq^{1,0}+iq^{0,1}\in\Omega^{1}(\Lambda\wedge\Lambda^{(1)})$ is
real, as well as $q$.
\end{proof}

The $1-1$ correspondence given by \eqref{eq:correspondenec between
qzand qsharp} establishes, in particular, a correspondence between
holomorphic quadratic differentials and real forms
$q\in\Omega^{1}(\Lambda\wedge\Lambda^{(1)})$ with
$d^{\mathcal{D}}q=0$.

Lastly, we proceed to the proof of Theorem \ref{CWeq}.

\begin{proof}
Let $(\Lambda_{t})_{t}$ be a variation of $\Lambda$ through
immersions of $M$ into the projectivized light-cone and
$\dot{\Lambda}\in\Gamma(\mathrm{Hom}(\Lambda,\Lambda^{\perp}/\Lambda))$
be the corresponding variational vector field, cf.
\eqref{eq:dotLambda}. The variation $(\Lambda_{t})_{t}$ is
infinitesimally conformal if and only if
$$\frac{d}{dt}_{\mid_{t=0}} (d(\Lambda_{t})(\delta
_{\bar{z}}),d(\Lambda_{t})(\delta _{\bar{z}}))=0,$$ or,
equivalently,
$$(\dot{\Lambda}_{\bar{z}},\Lambda_{\bar{z}})=0,$$
noting that $\frac{d}{dt}_{\mid_{t=0}}$ and $d_{\delta_{\bar{z}}}$
commute, as $z$ is independent of $t$. Write
$$\dot{\Lambda}=d\Lambda(X)+\nu$$ with $X\in\Gamma(TM)$ and
$\nu\in\Gamma(\mathrm{Hom}(\Lambda,S^{\perp}))$ as defined in
Section \ref{subsec:willmeq}.  Let $\pi_{T_{\Lambda}}$ and
$\pi_{N_{\Lambda}}$ denote the orthogonal projections of
$\Lambda^{*}T\mathbb{P}(\mathcal{L})=T_{\Lambda}\oplus N_{\Lambda}$
onto $T_{\Lambda}$ and $N_{\Lambda}$, respectively. By equation
\eqref{eq:isometryvsconnectionsandsff},
$$\dot{\Lambda}_{\bar{z}}=d\Lambda(\bar{\partial}_{\delta_{\bar{z}}}X)+\pi_{N_{\Lambda}}(d\Lambda(X))_{\bar{z}}+\pi_{T_{\Lambda}}(\nu_{\bar{z}})+\pi_{N_{\Lambda}}(\nu_{\bar{z}}),$$
and, therefore,
$(\dot{\Lambda}_{\bar{z}},\Lambda_{\bar{z}})=(d\Lambda(\bar{\partial}_{\delta_{\bar{z}}}X),
\Lambda_{\bar{z}})-( A^{\nu}(\delta_{\bar{z}}),\Lambda_{\bar{z}}).$
Let $X^{1,0}$ be the projection of $X$ onto $T^{1,0}M$. By the
isotropy of $T^{0,1}M$ and in view of
\eqref{eq:TijMpreservedbyconnection}, it follows that
$(\Lambda_{t})_{t}$
 is infinitesimally conformal if and only if
\begin{equation}\label{eq:nnabg4bn0a3viARE}
(d\Lambda(\bar{\partial}_{\delta_{\bar{z}}}X^{1,0})-A_{0}^{\nu}(\delta_{\bar{z}}),\Lambda_{\bar{z}})=0,
\end{equation}
for $A_{0}^{\nu}:=A^{\nu}-H^{\nu}I$, the trace-free part of the
shape operator $A^{\nu}$. Equation \eqref{eq:secffvsshapeop} shows
that, as the second fundamental form is symmetric, so is the shape
operator,
$$(A^{\nu}X,Y)=(X,A^{\nu}Y),$$for all $X,Y\in\Gamma(TM)$; and,
therefore, so is as well the trace-free part of the shape operator.
It follows that $$A^{\nu}_{0}J=-JA^{\nu}_{0},$$showing that
$A^{\nu}_{0}$ intertwines the eigenspaces of $J$. In view of the
isotropy of $T^{1,0}M$, we conclude that equation
\eqref{eq:nnabg4bn0a3viARE} holds if and only if
$(d\Lambda(\bar{\partial}_{\delta_{\bar{z}}}X^{1,0})-A_{0}^{\nu}(\delta_{\bar{z}}),d\Lambda
(Y))=0$, for all $Y\in\Gamma(TM)$, or, equivalently,
$$d\Lambda(\bar{\partial}X^{1,0})=(A_{0}^{\nu})^{0,1}.$$
According to \eqref{eq:curlyN10SperpinLambda01},
$\mathcal{N}^{0,1}\nu\in\Omega^{0,1}(\mathrm{Hom}(\Lambda,\Lambda^{1,0}))=\Omega^{0,1}(\mathrm{Hom}(\Lambda,d\Lambda(T^{1,0}M)/\Lambda)).$
Thus
$(\mathcal{N}_{\delta_{\bar{z}}}\nu,\Lambda_{\bar{z}})=(-A^{\nu}_{0}(\delta_{\bar{z}}),\Lambda_{\bar{z}}).$
We conclude that infinitesimally conformal variations through
immersions are characterized by the equation
$$d\Lambda(\bar{\partial}X^{1,0})=-\mathcal{N}^{0,1}\nu;$$
or, equivalently,
$$\mathcal{A}\nu\in\Omega^{2}( \mathrm{Im}\,\bar{\partial}),$$
for $\mathcal{A}:=-d\Lambda^{-1}\circ\mathcal{N}^{0,1}$.

Having said so, recall equation \eqref{eq:dotWvsdotLambda},
establishing
$$\dot{W}=\ll d^{\mathcal{D}}*\mathcal{N},\dot{\Lambda}\gg=\ll d^{\mathcal{D}}*\mathcal{N},\nu\gg$$ for $\ll\,,\,\gg$ the non-degenerate pairing between $2$-forms over $M$ and normal variations
defined in \cite{christoph}. We conclude that $\Lambda$ is
constrained Willmore if and only
$d^{\mathcal{D}}*\mathcal{N}\perp_{\ll\,,\,\gg}
\{\nu:\mathcal{A}\nu\in\Omega^{2}(
\mathrm{Im}\,\bar{\partial})\}$.\footnote{The Willmore surface
equation follows immediately as the particular case where no
constraint is considered.} The reference to the pairing shall be
omitted from now on.

From this point, the proof of the theorem consists of a
straightforward generalization to $n$-space of the argument
presented in \cite{christoph2} for the particular case of $n=3$.
Basically, in \cite{christoph2}, it is presented a pairing between
the space of quadratic differentials and the space of
$J$-anti-commuting endomorphisms of $TM$, with respect to which, by
Weyl's lemma,
\begin{equation}\label{eq:IM=perp}
\mathrm{Im}\,\bar{\partial}=(H^{0}K)^{\perp}.
\end{equation}
In view of equation \eqref{eq:IM=perp}, $\Lambda$ is constrained
Willmore if and only if $\nu\in(\mathcal{A}^{*}H^{0}(K))^{\perp}$,
or, equivalently,
$$d^{\mathcal{D}}*\mathcal{N}\in\Omega^{2}((\mathcal{A}^{*}H^{0}(K))^{\perp\perp}).$$
As $M$ is compact,  $H^{0}(K)$ is finite dimensional and, therefore,
$$(\mathcal{A}^{*}H^{0}(K))^{\perp\perp}=\mathcal{A}^{*}H^{0}(K).$$
We conclude that the surface $\Lambda$ is constrained Willmore if
and only if there exists some homomorphic quadratic differential $q$
such that $d^{\mathcal{D}}*\mathcal{N}=\mathcal{A}^{*}q$. The final
conclusion follows then, after some computation involving the
pairings mentioned above.
\end{proof}

\section{A constrained Willmore surface equation on the Hopf
differential and the Schwarzian derivative}\label{cW+kc}

\markboth{\tiny{A. C. QUINTINO}}{\tiny{CONSTRAINED WILLMORE
SURFACES}}

Theorem \ref{CWeq} consists of a reformulation of the
characterization of constrained Willmore surfaces in space-forms
established in \cite{SD}, which we deduce in this section.\newline

Let $q\in\Omega^{1}(\Lambda\wedge\Lambda^{(1)})$ be real with
$q^{1,0}\in\Omega^{1,0}(\Lambda\wedge\Lambda^{0,1})$. According to
Remark \ref{remmuyutilcurlyDNqCWeq}, equation \eqref{eq:mainCWeq}
holds if and only if, fixing a holomorphic chart $z$ of $M$,
$$d^{\mathcal{D}}*\mathcal{N}(\delta_{z},\delta_{\bar{z}})\,\sigma^{z}_{z\bar{z}}=2[q\wedge
*\mathcal{N}](\delta_{z},\delta_{\bar{z}})\,\sigma^{z}_{z\bar{z}},$$
or, equivalently,
$$(\mathcal{D}_{\delta_{z}}\circ\mathcal{N}_{\delta_{\bar{z}}}-\mathcal{N}_{\delta_{\bar{z}}}\circ\mathcal{D}_{\delta_{z}}+\overline{\mathcal{D}_{\delta_{z}}\circ\mathcal{N}_{\delta_{\bar{z}}}-\mathcal{N}_{\delta_{\bar{z}}}\circ\mathcal{D}_{\delta_{z}}})\,\sigma^{z}_{z\bar{z}}=2[q_{\delta_{z}},\mathcal{N}_{\delta_{\bar{z}}}]\sigma^{z}_{z\bar{z}}+2\overline{[q_{\delta_{z}},\mathcal{N}_{\delta_{\bar{z}}}]\sigma^{z}_{z\bar{z}}}.$$
On the one hand,
$$\mathcal{D}_{\delta_{z}}\circ\mathcal{N}_{\delta_{\bar{z}}}\sigma^{z}_{z\bar{z}}=\mathcal{D}_{\delta_{z}}\circ\pi_{S^{\perp}}\overline{(\sigma_{zz}^{z})_{\bar{z}}}=\mathcal{D}_{\delta_{z}}\circ\pi_{S^{\perp}}(\,\overline{k^{z}_{\bar{z}}}\,)=\overline{\nabla^{S^{\perp}}_{\delta_{\bar{z}}}\nabla^{S^{\perp}}_{\delta_{\bar{z}}}k^{z}}.$$
On the other hand,
$$\mathcal{N}_{\delta_{\bar{z}}}\circ\mathcal{D}_{\delta_{z}}\sigma_{z\bar{z}}^{z}=\mathcal{N}_{\delta_{\bar{z}}}\circ\pi_{S}(\sigma_{zz}^{z})_{\bar{z}}=
-\frac{1}{2}\,c^{z}\mathcal{N}_{\delta_{\bar{z}}}\sigma_{\bar{z}}^{z}+\mathcal{N}_{\delta_{\bar{z}}}\circ\pi_{S}k^{z}_{\bar{z}}.$$
Note that, as $\sigma_{z}^{z}$ and $\sigma_{z\bar{z}}^{z}$ are
sections of $S$,
$\mathcal{N}_{\delta_{\bar{z}}}\sigma_{z}^{z}=\pi_{S^{\perp}}(d_{\delta_{\bar{z}}}\sigma_{z}^{z})=0$.
Together with \eqref{eq:calNLambda0}, this establishes
$\mathcal{N}^{0,1}\Lambda^{1,0}=0$. By
\eqref{eq:piSkbarzinLambda10}, it follows that
$\mathcal{N}_{\delta_{\bar{z}}}\circ\mathcal{D}_{\delta_{z}}\sigma_{z\bar{z}}^{z}=-\frac{1}{2}\,c^{z}\overline{k^{z}}$.
Lastly,
$$[q_{\delta_{z}},\mathcal{N}_{\delta_{\bar{z}}}]\sigma^{z}_{z\bar{z}}=-\mathcal{N}_{\delta_{\bar{z}}}q_{\delta_{z}}\sigma^{z}_{z\bar{z}}=\frac{1}{2}\,q^{z}\mathcal{N}_{\delta_{\bar{z}}}\sigma_{\bar{z}}^{z}=\frac{1}{2}\,q^{z}\,\overline{\pi_{S^{\perp}}\sigma_{zz}^{z}}=\frac{1}{2}\,q^{z}\,\overline{k^{z}}.$$
We conclude that equation \eqref{eq:mainCWeq} holds if and only if
$$\mathrm{Re}(\nabla^{S^{\perp}}_{\delta_{\bar{z}}}\nabla^{S^{\perp}}_{\delta_{\bar{z}}}k^{z}+\frac{1}{2}\,\overline{c^{z}}k^{z})=\mathrm{Re}(\overline{q^{z}}\,k^{z}).$$
Equation (33b) in \cite{SD}, establishing the reality of
$\nabla^{S^{\perp}}_{\delta_{\bar{z}}}\nabla^{S^{\perp}}_{\delta_{\bar{z}}}k^{z}+\frac{1}{2}\,\overline{c^{z}}k^{z}$,
leads us to the following constrained Willmore surface equation in
terms of the Hopf differential and the Schwarzian derivative,
presented in \cite{SD}:\footnote{The argument above shows that the
constrained Willmore surface equation on a holomorphic quadratic
differential, presented in \cite{SD} with no reference to a
multiplier, is equivalent to equation \eqref{eq:mainCWeq}, under the
correspondence given by \eqref{eq:correspondenec between qzand
qsharp}.}

\begin{Lemma}\label{ckCWeq} Let
$q$ be a real $1$-form with values in $\Lambda\wedge\Lambda^{(1)}$
with $q^{1,0}\in\Omega^{1,0}(\Lambda\wedge\Lambda^{0,1})$. $\Lambda$
is a $q$-constrained Willmore surface if and only if, fixing a
holomorphic chart $z$ of $M$, $q^{z}$ is holomorphic and
\begin{equation}\label{CWviacharts}
\nabla^{S^{\perp}}_{\delta_{\bar{z}}}\nabla^{S^{\perp}}_{\delta_{\bar{z}}}k^{z}+\frac{\overline{c^{z}}}{2}\,k^{z}=\mathrm{Re}\,(\overline{q^{z}}\,k^{z}).
\end{equation}
\end{Lemma}

The surface $\Lambda$ is constrained Willmore if and only there
exists a holomorphic map $q^{z}\in C^{\infty}(M,\C)$ satisfying
equation \eqref{CWviacharts}, in which case, the quadratic
differential defined locally by $q^{z}dz^{2}$ is, under the
correspondence given by \eqref{eq:correspondenec between qzand
qsharp}, a multiplier to $\Lambda$.

\chapter{Constrained Willmore
surfaces: spectral deformation and B\"{a}cklund
transformation}\label{transformsofCW}

\markboth{\tiny{A. C. QUINTINO}}{\tiny{CONSTRAINED WILLMORE
SURFACES}}

In \cite{burstall+calderbank}, F. Burstall and D. Calderbank present
a characterization of constrained Willmore surfaces in spherical
space in terms of the flatness of a certain loop of metric
connections. In view of this characterization, we characterize
constrained Willmore surfaces in space-forms in terms of the
\textit{constrained} \textit{harmonicity} of the central sphere
congruence, generalizing the characterization of Willmore surfaces
in space-forms in terms of the harmonicity of the central sphere
congruence.  This will enable us to define a spectral deformation of
constrained Willmore surfaces in space-forms, by the action of the
loop of flat metric connections above, as well as a
\textit{B\"{a}cklund transformation}, by applying a dressing action.
Both B\"{a}cklund transformation and spectral deformation
corresponding to the zero multiplier preserve the class of Willmore
surfaces. We establish a permutability between spectral deformation
and B\"{a}cklund transformation of constrained Willmore surfaces in
space-forms.\newline

Let $\tilde{d}$ be a flat metric connection on
$\underline{\C}^{n+2}$. It is obvious, but, nevertheless, opportune
to remark that $\tilde{d}$ defines a connection on
$\underline{\R}^{n+1,1}$, i.e., $\tilde{d}$ preserves the reality of
sections of $\underline{\C}^{n+2}=(\underline{\R}^{n+1,1})^{\C}$, if
and only if $\tilde{d}$ is real,
$$\overline{\tilde{d}}=\tilde{d}.$$Equally basic is to remark that,
if $\hat{d}_{1}$ and $\hat{d}_{2}$ are real connections on
$\underline{\C}^{n+2}$, then an isomorphism
$\phi:(\underline{\C}^{n+2},\hat{d}_{1})\rightarrow(\underline{\C}^{n+2},\hat{d}_{2})$
is real, $$\overline{\phi}=\phi,$$ if and only if it defines an
isomorphism
$\phi:(\underline{\R}^{n+1,},\hat{d}_{1})\rightarrow(\underline{\R}^{n+1,1},\hat{d}_{2})$.
Note also that the real bundles in $\underline{\C}^{n+2}$ - those
preserved by complex conjugation - are the complexifications of
bundles in $\underline{\R}^{n+1,1}$: given
$W\subset\underline{\C}^{n+2}$ real,
$$W=W\cap \underline{\R}^{n+1,1}\oplus W\cap
i\underline{\R}^{n+1,1}=(W\cap \underline{\R}^{n+1,1})^{\C}.$$

Throughout this chapter, let $V$ be a non-degenerate subbundle of
$\underline{\C}^{n+2}$, $$\underline{\C}^{n+2}=V\oplus V^{\perp},$$
and $\pi_{V}$ and $\pi_{V^{\perp}}$ denote the orthogonal
projections of $\underline{\C}^{n+2}$ onto $V$ and $V^{\perp}$,
respectively.

Fix a conformal structure $\mathcal{C}$ in $M$.

\section{Constrained harmonicity of bundles}\label{sec:cchb}

\markboth{\tiny{A. C. QUINTINO}}{\tiny{CONSTRAINED WILLMORE
SURFACES}}

A multiplier to a surface $\Lambda$ in the projectivized light-cone
is, in particular, a real form
$q\in\Omega^{1}(\Lambda\wedge\Lambda^{(1)})$. For such a $q$,
equations \eqref{eq:curlyDextderivofq} and \eqref{eq:mainCWeq},
together, encode the flatness of the connection
$d^{\lambda}_{q}:=\mathcal{D}+\lambda^{-1}\mathcal{N}^{1,0}+\lambda\mathcal{N}^{0,1}+(\lambda^{-2}-1)\,q^{1,0}+(\lambda^{2}-1)\,q^{0,1}$,
on $(\underline{\R}^{n+1,1})^{\C}$, for all
$\lambda\in\C\backslash\{0\}$, or, equivalently, for all $\lambda\in
S^{1}$. Constrained Willmore surfaces in space-forms, admitting $q$
as a multiplier, are characterized by the flatness of the
$S^{1}$-family of metric connections $d^{\lambda}_{q}$ on
$\underline{\R}^{n+1,1}$, in an integrable systems interpretation
due to F. Burstall and D. Calderbank \cite{burstall+calderbank}.
This characterization will enable us to define a spectral
deformation of constrained Willmore surfaces in space-forms, by the
action of the loop of flat metric connections $d^{\lambda}_{q}$, as
well as a \textit{B\"{a}cklund transformation}, by applying a
dressing action. Our transformations of constrained Willmore
surfaces will be based on the \textit{constrained harmonicity} of
the central sphere congruence. Given $\hat{d}$ a flat metric
connection on $\underline{\C}^{n+2}$ and $V$ a non-degenerate
subbundle of $\underline{\C}^{n+2}$, we generalize naturally the
decomposition \eqref{eq:dcurlyDcurlyN} to a decomposition
$\hat{d}=\hat{\mathcal{D}}_{V}+\hat{\mathcal{N}}_{V}$ and, given
$q\in\Omega^{1}(\wedge^{2}V\oplus\wedge^{2}V^{\perp})$, define then,
for each $\lambda\in\C\backslash\{0\}$, a connection
$\hat{d}^{\lambda,q}_{V}:=\hat{\mathcal{D}}_{V}+\lambda^{-1}\hat{\mathcal{N}}_{V}^{1,0}+\lambda\hat{\mathcal{N}}_{V}^{0,1}+(\lambda^{-2}-1)q^{1,0}+(\lambda^{2}-1)q^{0,1}$,
on $\underline{\C}^{n+2}$, generalizing
$d^{\lambda}_{q}=d^{\lambda,q}_{S}$. We define $V$ to be
$(q,\hat{d})$-\textit{constrained harmonic} if
$\hat{d}^{\lambda,q}_{V}$ is flat, for all
$\lambda\in\C\backslash\{0\}$, or, equivalently, for all $\lambda\in
S^{1}$. A simple, yet crucial, observation is that, given
$\tilde{d}$ another flat metric connection on $\underline{\C}^{n+2}$
and $\phi: (\underline{\C}^{n+2}, \tilde{d})\rightarrow
(\underline{\C}^{n+2},\hat{d})$ an isomorphism of bundles provided
with a metric and a connection, $V$ is $(q,\tilde{d})$-constrained
harmonic if and only if $\phi V$ is
$(\mathrm{Ad}_{\phi}q,\hat{d})$-constrained harmonic.  The
constrained harmonicity of a bundle applies to the central sphere
congruence, providing a characterization of constrained Willmore
surfaces in space-forms.\newline

Consider the decomposition
$$\tilde{d}=\mathcal{D}_{V}^{\tilde{d}}+\mathcal{N}_{V}^{\tilde{d}}$$ for $\mathcal{D}_{V}^{\tilde{d}}$ the connection on $\underline{\C}^{n+2}$ defined by
$$\mathcal{D}_{V}^{\tilde{d}}:=\pi
_{V}\circ \tilde{d} \circ\pi _{V}+\pi _{V^{\perp}}\circ \tilde{d}
\circ\pi _{V^{\perp}}.$$ Note that, as $\tilde{d}$ is a metric
connection, so is $\mathcal{D}_{V}^{\tilde{d}}$: given
$\eta,\mu\in\Gamma(\underline{\C}^{n+2})$,
\begin{eqnarray*}
d(\eta,\mu)&=&d(\pi_{V}\eta,\pi_{V}\mu)+d(\pi_{V^{\perp}}\eta,\pi_{V^{\perp}}\mu)\\&=&
(\tilde{d}(\pi_{V}\eta),\pi_{V}\mu)+(\pi_{V}\eta,\tilde{d}(\pi_{V}\mu))+(\tilde{d}(\pi_{V^{\perp}}\eta),\pi_{V^{\perp}}\mu)+(\pi_{V^{\perp}}\eta,\tilde{d}(\pi_{V^{\perp}}\mu))\\&=&
(\mathcal{D}_{V}^{\tilde{d}}\eta,\mu)+(\eta,\mathcal{D}_{V}^{\tilde{d}}\mu).
\end{eqnarray*}
Equivalently,
$$\mathcal{N}_{V}^{\tilde{d}}:=d-\mathcal{D}_{V}^{\tilde{d}}=\pi
_{V^{\perp}}\circ \tilde{d} \circ\pi _{V}+\pi _{V}\circ \tilde{d}
\circ\pi _{V^{\perp}}$$ is skew-symmetric,
$$\mathcal{N}_{V}^{\tilde{d}}\in\Omega^{1}(V\wedge V^{\perp}).$$ In
the particular case $\tilde{d}=d$, the trivial flat connection, we
shall omit the reference to $\tilde{d}$.

Next we present the concept of \textit{constrained}
\textit{harmonicity} of a bundle, which will apply to the central
sphere congruence to provide a characterization of constrained
Willmore surfaces, in view of the characterization of these surfaces
in terms of the flatness of a certain loop of metric connections
presented in \cite{burstall+calderbank}, which we shall address
later on. In fact, the following definition encodes the
characterization of constrained harmonicity for a general bundle in
terms of the flatness of a loop of metric connections generalizing
the one above, as we shall verify later on.

Let $q$ be a $1$-form with values in $\wedge^{2}V\oplus
\wedge^{2}V^{\perp}\subset o(\underline{\C}^{n+2})$.
\begin{defn}\label{cccchadddafhh}
$V$ is said to be \emph{$(q,\tilde{d})$-constrained harmonic}
if\newline

$i)$\,\,$d^{\mathcal{D}_{V}^{\tilde{d}}}q^{1,0}=\frac{1}{2}\,[q\wedge
q]=d^{\mathcal{D}_{V}^{\tilde{d}}}q^{0,1};$\newline

$ii)$\,\,$d^{\mathcal{D}_{V}^{\tilde{d}}}*\mathcal{N}_{V}^{\tilde{d}}=2\,[q\wedge*\mathcal{N}_{V}^{\tilde{d}}].$
\end{defn}

By $\tilde{d}$-\textit{constrained} \textit{harmonicity} of $V$ we
mean the existence of a $\textit{multiplier}$ to $V$ with respect to
$\tilde{d}$, i.e., a $1$-form $q$ with values in $\wedge^{2}V\oplus
\wedge^{2}V^{\perp}$ for which $V$ is $(q,\tilde{d})$-constrained
harmonic. In the particular case of $\tilde{d}=d$, we shall,
alternatively, omit the reference to $\tilde{d}$ and refer simply to
\textit{constrained} \textit{harmonicity} or, when specifying
$q\in\Omega^{1}(\wedge^{2}V\oplus \wedge^{2}V^{\perp})$, to
$q$-\textit{constrained} \textit{harmonicity}. We shall refer to
$(0,\tilde{d})$-constrained harmonic bundles, alternatively, as
$\tilde{d}$-\textit{harmonic} \textit{bundles}.

It is useful to observe that, as $\mathcal{N}^{\tilde{d}}_{V}$ and
$q$ take values in $V\wedge V^{\perp}$ and $\wedge^{2}V\oplus
\wedge^{2}V^{\perp}$, respectively, and $V$ and $V^{\perp}$ are
$\mathcal{D}^{\tilde{d}}_{V}$-parallel,
\begin{equation}\label{eq:cdfghyjuklguyeifrojhgfdxmnbvcxcghyrd5tr3276y348iu549549886}
d^{\mathcal{D}^{\tilde{d}}_{V}}*\mathcal{N}^{\tilde{d}}_{V},\,
[q\wedge *\,\mathcal{N}^{\tilde{d}}_{V}]\in\Omega^{2}(V\wedge
V^{\perp}),
\end{equation}
whereas
\begin{equation}\label{eq:11111111qwertyuiopmjsj55s}
d^{\mathcal{D}^{\tilde{d}}_{V}}q^{1,0},
d^{\mathcal{D}^{\tilde{d}}_{V}}q^{0,1},[q\wedge
q]\in\Omega^{2}(\wedge^{2}V\oplus \wedge^{2}V^{\perp}).
\end{equation}
The same argument establishes both $R^{\mathcal{D}^{\tilde{d}}_{V}}$
and $[\mathcal{N}^{\tilde{d}}_{V}\wedge\mathcal{N}^{\tilde{d}}_{V}]$
as $2$-forms with values in $\Omega^{2}(\wedge^{2}V\oplus
\wedge^{2}V^{\perp})$ and
$d^{\mathcal{D}^{\tilde{d}}_{V}}\mathcal{N}^{\tilde{d}}_{V}$ as a
$2$-form with values in $\Omega^{2}(V\wedge V^{\perp})$. We conclude
that the flatness of $\tilde{d}$ encodes both

\begin{prop}(Gauss-Ricci equation)
$$R^{\mathcal{D}^{\tilde{d}}_{V}}+\frac{1}{2}\,[\mathcal{N}^{\tilde{d}}_{V}\wedge
\mathcal{N}^{\tilde{d}}_{V}]=0;$$
\end{prop}

and

\begin{prop}(Codazzi equation)
$$d^{\mathcal{D}^{\tilde{d}}_{V}}\mathcal{N}^{\tilde{d}}_{V}=0.$$
\end{prop}

According to the Codazzi equation,
$$d^{\mathcal{D}^{\tilde{d}}_{V}}(\mathcal{N}^{\tilde{d}}_{V})^{1,0}=-d^{\mathcal{D}^{\tilde{d}}_{V}}(\mathcal{N}^{\tilde{d}}_{V})^{0,1},$$
making clear that:

\begin{prop}\label{complharm}
The $\tilde{d}$-harmonicity of $V$ is characterized, equivalently,
by any of the following equations:\newline

$i)$\,\,$d^{\mathcal
{D}_{V}^{\tilde{d}}}*\mathcal{N}_{V}^{\tilde{d}}=0;$\newline

$ii)$\,\,$d^{\mathcal
{D}_{V}^{\tilde{d}}}(\mathcal{N}_{V}^{\tilde{d}})^{1,0}=0;$\newline

$iii)$\,\,$d^{\mathcal
{D}_{V}^{\tilde{d}}}(\mathcal{N}_{V}^{\tilde{d}})^{0,1}=0.$

\end{prop}

Define, for each $\lambda\in\C\backslash \{0\}$, a metric connection
on $\underline{\C}^{n+2}$ by
$$\tilde{d}^{\lambda ,q}_{V}:=\mathcal{D}_{V}^{\tilde{d}}+\lambda ^{-1}(\mathcal{N}_{V}^{\tilde{d}})^{1,0}+\lambda(\mathcal{N}_{V}^{\tilde{d}})^{0,1}+(\lambda ^{-2}-1)q^{1,0}+(\lambda ^{2}-1)q^{0,1}.$$
In the particular case of $q=0$, we shall omit the reference to $q$
in $\tilde{d}^{\lambda ,q}_{V}$.

In the particular case $\tilde{d}=d$ and $V=S$, the complexification
of the central sphere congruence of a surface $\Lambda$, we shall,
alternatively, omit the reference to $S$ and refer to
$\tilde{d}^{\lambda ,q}_{V}$ as $d^{\lambda}_{q}$,
$$d^{\lambda}_{q}:=\mathcal{D}+\lambda
^{-1}\mathcal{N}^{1,0}+\lambda\mathcal{N}^{0,1}+(\lambda
^{-2}-1)q^{1,0}+(\lambda ^{2}-1)q^{0,1}.$$ Of course, the set of
connections $d^{\lambda}{q}$, with $\lambda\in S^{1}$, defines a
group, using multiplication of parameters to define the group law.
According to \cite{burstall+calderbank}, if $q$ is a real form
taking values in $\Lambda\wedge\Lambda^{(1)}$, then the flatness of
the loop of metric connections $d^{\lambda}_{q}$ characterizes
$\Lambda$ as a $q$-constrained Willmore surface.

Before proceeding any further, we dedicate a few moments to some
useful general remarks about the family of connections
$\tilde{d}^{\lambda ,q}_{V}$. First note that
$$\tilde{d}^{1,q}_{V}=\tilde{d}.$$
Note, on the other hand, that, in view of the
$\mathcal{D}_{V}^{\tilde{d}}$-parallelness of $V$ and $V^{\perp}$,
together with the intertwining of $V$ and $V^{\perp}$ by
$\mathcal{N}_{V}^{\tilde{d}}$, and the fact that $q$ preserves $V$
and $V^{\perp}$, we have
$$\mathcal{D}^{\tilde{d}^{\lambda
,q}_{V}}_{V}=\mathcal{D}^{\tilde{d}}_{V}+(\lambda
^{-2}-1)q^{1,0}+(\lambda^{2}-1)q^{0,1},\,\,\,\,
\mathcal{N}_{V}^{\tilde{d}^{\lambda
,q}_{V}}=\lambda^{-1}(\mathcal{N}_{V}^{\tilde{d}})^{1,0}+\lambda(\mathcal{N}_{V}^{\tilde{d}})^{0,1}.$$
Set $$q_{\lambda}:=\lambda ^{-2}q^{1,0}+\lambda^{2}q^{0,1}.$$Then
\begin{eqnarray*}
(\tilde{d}^{\lambda
,q}_{V})^{\mu,q_{\lambda}}_{V}&=&\mathcal{D}^{\tilde{d}^{\lambda
,q}_{V}}_{V}+\mu^{-1}(\mathcal{N}^{\tilde{d}^{\lambda
,q}_{V}}_{V})^{1,0}+\mu(\mathcal{N}^{\tilde{d}^{\lambda
,q}_{V}}_{V})^{0,1}+(\mu^{-2}-1)q^{1,0}_{\lambda}+(\mu^{2}-1)q^{0,1}_{\lambda}\\&=&
\mathcal{D}^{\tilde{d}}_{V}+(\lambda\mu)^{-1}(\mathcal{N}^{\tilde{d}}_{V})^{1,0}+
\lambda\mu(\mathcal{N}^{\tilde{d}}_{V})^{0,1}+((\lambda\mu)^{-2}-1)q^{1,0}+((\lambda\mu)^{2}-1)q^{0,1}
\end{eqnarray*}
and, ultimately,
\begin{equation}\label{eq:conndeconn}
(\tilde{d}^{\lambda
,q}_{V})^{\mu,q_{\lambda}}_{V}=\tilde{d}^{\lambda\mu ,q}_{V}.
\end{equation}
It will also be useful to observe that, given another flat metric
connection $\hat{d}$ on $\underline{\C}^{n+2}$ and an isomorphism
$\psi:(\underline{\C}^{n+2}, \hat{d})\rightarrow
(\underline{\C}^{n+2}, \tilde{d})$ of bundles preserving
connections, we have
\begin{equation}\label{eq:isomV}
\mathcal{D}_{\psi
V}^{\tilde{d}}=\psi\circ\mathcal{D}_{V}^{\hat{d}}\circ\psi^{-1},\,\,\,\,\,\,\mathcal{N}_{\psi
V}^{\tilde{d}}=\psi\,\mathcal{N}_{V}^{\hat{d}}\,\psi^{-1}
\end{equation}
and, therefore,
\begin{equation}\label{eq:isomVd}
\tilde{d}_{\psi
V}^{\lambda,q}=\psi\circ\hat{d}\,^{\lambda,\mathrm{Ad}_{\psi^{-1}}q}_{V}\circ\psi^{-1}.
\end{equation}

\begin{rem}\label{s1familyofconnections}

If $V$ is real, then, obviously, so is $V^{\perp}$, so that, in
particular, $\overline{\pi_{V}}=\pi
_{V},\,\,\,\overline{\pi_{V^{\perp}}}=\pi _{V^{\perp}}$ and
therefore, if $\tilde{d}$ is real, then
$\overline{\mathcal{D}_{V}^{\tilde{d}}}=\mathcal{D}_{V}^{\tilde{d}},\,\,\,\,\,\,\overline{\mathcal{N}_{V}^{\tilde{d}}}=\mathcal{N}_{V}^{\tilde{d}}$.
If $V$, $q$ and $\tilde{d}$ are real, then so is $\tilde{d}^{\lambda
,q}_{V}$, for all $\lambda\in S^{1}$.
\end{rem}

The characterization of the harmonicity of $V$, as a map into a
Grassmannian, in terms of the flatness of the $S^{1}$-family of
metric connections $d^{\lambda}_{V}$, cf. K. Uhlenbeck
\cite{uhlenbeck 89}, generalizes to a characterization of
$\tilde{d}$-constrained harmonicity, as follows (see also Remark
\ref{CoS1}):

\begin{thm}\label{thm5.2}
$V$ is $(q,\tilde{d})$-constrained harmonic if and only if
$\tilde{d}^{\lambda ,q}_{V}$ is a flat connection, for each
$\lambda\in\C\backslash \{0\}$.
\end{thm}

The proof of the theorem will consist of showing that the
$(q,\tilde{d})$-constrained harmonicity of $V$ establishes and, in
fact, encodes, in view of the flatness of $\tilde{d}$ (Gauss-Ricci
and Codazzi equations), the flatness of $\tilde{d}^{\lambda,q}_{V}$,
for all $\lambda\in\C\backslash\{0\}$.
\begin{proof}
For simplicity, write $\tilde{\mathcal{D}}_{V}$ and
$\tilde{\mathcal{N}}_{V}$ for $\mathcal{D}_{V}^{\tilde{d}}$ and
$\mathcal{N}_{V}^{\tilde{d}}$. The curvature tensor of
$\tilde{d}^{\lambda ,q}_{V}$ is given by
$$R^{\tilde{d}^{\lambda ,q}_{V}}=R^{\tilde{\mathcal{D}}_{V}}+d^{\tilde{\mathcal{D}}_{V}}\delta
^{\lambda}_{q}+\frac{1}{2}\,[\delta ^{\lambda}_{q}\wedge\delta
^{\lambda}_{q}],$$ for $\delta ^{\lambda}_{q}:=\tilde{d}^{\lambda
,q}_{V}-\tilde{\mathcal{D}}_{V}$. Since there are no non-zero
$(2,0)$- or $(0,2)$-forms over a surface,
$$\frac{1}{2}[\delta ^{\lambda}_{q}\wedge\delta
^{\lambda}_{q}]=[\tilde{\mathcal{N}}_{V}^{1,0}\wedge
\tilde{\mathcal{N}}_{V}^{0,1}]+(\lambda
^{-1}-\lambda)([q^{1,0}\wedge \tilde{\mathcal{N}}_{V}^{0,1} ]-[
q^{0,1}\wedge\tilde{\mathcal{N}}_{V}^{1,0}])+(2-\lambda
^{-2}-\lambda ^{2})[q^{1,0}\wedge q^{0,1}],$$ and Gauss-Ricci
equation establishes then
$$R^{\tilde{d}^{\lambda ,q}_{V}}=d^{\tilde{\mathcal{D}}_{V}}\delta ^{\lambda}_{q}+(\lambda
^{-1}-\lambda)([q^{1,0}\wedge\tilde{\mathcal{N}}_{V}^{0,1}
]-[q^{0,1}\wedge\tilde{\mathcal{N}}_{V}^{1,0}
])+\frac{1}{2}\,(2-\lambda ^{-2}-\lambda ^{2})[q\wedge q].$$In its
turn, Codazzi equation gives
$$d^{\tilde{\mathcal{D}}_{V}}\tilde{\mathcal{N}}_{V}^{1,0}=\frac{i}{2}\,d^{\tilde{\mathcal{D}}_{V}}*\tilde{\mathcal{N}}_{V}=-d^{\tilde{\mathcal{D}}_{V}}\tilde{\mathcal{N}}_{V}^{0,1},$$
We conclude that
\begin{eqnarray*}
R^{\tilde{d}^{\lambda ,q}_{V}}&=&\frac{\lambda
^{-1}-\lambda}{2}\,i\,(d^{\tilde{\mathcal{D}}_{V}}*\tilde{\mathcal{N}}_{V}-2[q\wedge
*\tilde{\mathcal{N}}_{V}])\\
&& \mbox{}+(\lambda
^{-2}-1)\,d^{\tilde{\mathcal{D}}_{V}}q^{1,0}+(\lambda
^{2}-1)\,d^{\tilde{\mathcal{D}}_{V}}q^{0,1}+\frac{1}{2}\,(2-\lambda
^{-2}-\lambda ^{2})\,[q\wedge q].
\end{eqnarray*}
By
\eqref{eq:cdfghyjuklguyeifrojhgfdxmnbvcxcghyrd5tr3276y348iu549549886}
and \eqref{eq:11111111qwertyuiopmjsj55s}, it follows that
$R^{\tilde{d}^{\lambda ,q}_{V}}=0$ if and only if both
\begin{equation}\label{eq:CWeqlambda1}
\frac{\lambda
^{-1}-\lambda}{2}\,i\,(d^{\tilde{\mathcal{D}}_{V}}*\tilde{\mathcal{N}}_{V}-2[q\wedge
*\tilde{\mathcal{N}}_{V}])=0
\end{equation}
and
\begin{equation}\label{eq:CWeqlambda2}
(\lambda ^{-2}-1)\,d^{\tilde{\mathcal{D}}_{V}}q^{1,0}+(\lambda
^{2}-1)\,d^{\tilde{\mathcal{D}}_{V}}q^{0,1}+\frac{1}{2}\,(2-\lambda
^{-2}-\lambda ^{2})\,[q\wedge q]=0
\end{equation}
hold. Organizing equations \eqref{eq:CWeqlambda1} and
\eqref{eq:CWeqlambda2} by powers of $\lambda$ leads us to the final
conclusion of the equivalence between the flatness of
$\tilde{d}^{\lambda ,q}_{V}$ for all $\lambda\in \C\backslash \{0\}$
and the $(q,\tilde{d})$-constrained complex-harmonicity of $V$.
\end{proof}

\begin{rem}\label{CoS1}
From the proof of Theorem \ref{thm5.2}, we readily verify that the
$(q,\tilde{d})$-constrained harmonicity of $V$ is equivalently
characterized by the flatness of $\tilde{d}^{\lambda ,q}_{V}$ for
all $\lambda\in S^{1}$.
\end{rem}

Theorem \ref{thm5.2} will play a crucial role in what follows in the
chapter.

\subsection{Spectral deformation of constrained harmonic
bundles}

Having observed that the $(q,\tilde{d})$-constrained harmonicity of
$V$ ensures the flatness of the metric connection
$\tilde{d}^{\lambda ,q}_{V}$ on $\underline{\C}^{n+2}$, it is
natural to ask about the $\tilde{d}^{\lambda ,q}_{V}$-constrained
harmonicity of $V$.

\begin{thm}\label{th5.0.4}
If  $V$ is $(q,\tilde{d})$-constrained harmonic then $V$ is
$(q_{\lambda},\tilde{d}^{\lambda ,q}_{V})$-constrained harmonic, for
all $\lambda\in\C\backslash\{0\}$.
\end{thm}

\begin{proof}
It is immediate from Theorem \ref{thm5.2} and equation
\eqref{eq:conndeconn}.
\end{proof}

\begin{corol}\label{cor5.0.1}
If $V$ is $\tilde{d}$-harmonic then $V$ is
$\tilde{d}^{\lambda}_{V}$-harmonic for any $\lambda\in\C\backslash
\{0\}$.
\end{corol}
For a general flat metric connection $\hat{d}$ on
$\underline{\C}^{n+2}$, the $\hat{d}$-harmonicity of $V$ follows
from its $\tilde{d}$-harmonicity if $V=\phi\,V$ for some isomorphism
$\phi:(\underline{\C}^{n+2},\tilde{d})\rightarrow
(\underline{\C}^{n+2},\hat{d})$, as established, in particular, in
the following theorem, which constitutes a simple, yet crucial,
result.

\begin{thm}\label{prop5.5}
Let $\hat{d}$ be a flat metric connection on $\underline{\C}^{n+2}$
and $$\phi:(\underline{\C}^{n+2},\tilde{d})\rightarrow
(\underline{\C}^{n+2},\hat{d})$$ be an isomorphism. The bundle $V$
is $(q,\tilde{d})$-constrained harmonic if and only if $\phi\,V$ is
$(\mathrm{Ad}_{\phi}\,q,\hat{d})$-constrained harmonic.
\end{thm}

\begin{proof}
It is immediate from theorem \ref{thm5.2} and equation
\eqref{eq:isomVd}.
\end{proof}

Theorem \ref {prop5.5} combines with Theorem \ref {th5.0.4} to
provide the definition, up to isomorphisms of bundles with a metric
and a connection, of a $\C\backslash\{0\}$-deformation of
$\tilde{d}$-constrained harmonic bundles. In fact, if $V$ is
$\tilde{d}$-constrained harmonic with multiplier $q$, then so is the
transformation $\tilde{\phi}_{\lambda}^{q}\,V$ of $\,V$, for
$\tilde{\phi}_{\lambda}^{q}:(\underline{\C}^{n+2},\tilde{d}^{\lambda,q}_{V})\rightarrow
(\underline{\C}^{n+2},\tilde{d})$ an isomorphism and $\lambda$ in
$\C\backslash \{0\}$. Note that this spectral deformation of
$\tilde{d}$-constrained harmonic bundles provides, in particular, a
$\C\backslash\{0\}$-deformation of $\tilde{d}$-harmonic bundles into
$\tilde{d}$-harmonic bundles.

\section{Complexified surfaces}\label{complexifiedsurfacessect}

\markboth{\tiny{A. C. QUINTINO}}{\tiny{CONSTRAINED WILLMORE
SURFACES}}

The transformations of a constrained Willmore surface $\Lambda$ in
the projectivized light-cone we present in this chapter are, in
particular, pairs $((\Lambda^{1,0})^{*},(\Lambda^{0,1})^{*})$ of
transformations $(\Lambda^{1,0})^{*}$ and $(\Lambda^{0,1})^{*}$ of
$\Lambda^{1,0}$ and $\Lambda^{0,1}$, respectively. The fact that
$\Lambda^{1,0}$ and $\Lambda^{0,1}$ intersect in a rank $1$ bundle
will ensure that $(\Lambda^{1,0})^{*}$ and $(\Lambda^{0,1})^{*}$
have the same property. The isotropy of $\Lambda^{1,0}$ and
$\Lambda^{0,1}$ will ensure that of $(\Lambda^{1,0})^{*}$ and
$(\Lambda^{0,1})^{*}$ and, therefore, of their intersection. The
reality of the bundle $\Lambda^{1,0}\cap\Lambda^{0,1}$ and the fact
that it defines an immersion of $M$ into $\mathbb{P}(\mathcal{L})$
are preserved by the spectral deformation, but it is not clear that
the same is necessarily true for B\"{a}cklund transformation. This
motivates us to define \textit{complexified surface}.
\newline

In the case $\tilde{d}$ is real, and given a $\tilde{d}$-surface
$\Lambda$, set
$$\Lambda ^{1,0}_{\tilde{d}}:=\langle\sigma ,\tilde{d}_{\delta
_{z}}\sigma\rangle,\,\,\,\, \,\,\,\Lambda
^{0,1}_{\tilde{d}}:=\langle\sigma ,\tilde{d}_{\delta
_{\bar{z}}}\sigma\rangle=\overline{\Lambda
^{1,0}_{\tilde{d}}},$$independently of the choices of a never-zero
section $\sigma$ of $\Lambda$ and of a holomorphic chart $z$ of
$(M,\mathcal{C}_{\Lambda}^{\tilde{d}})$, defining in this way two
subbundles of the bundle
$$\Lambda^{(1)}_{\tilde{d}}=\langle\sigma ,\tilde{d}_{\delta
_{z}}\sigma ,\tilde{d}_{\delta _{\bar{z}}}\sigma\rangle$$in
$$S^{\tilde{d}}=\langle\sigma ,\tilde{d}_{\delta _{z}}\sigma
,\tilde{d}_{\delta _{\bar{z}}}\sigma ,\tilde{d}_{\delta
_{\bar{z}}}\tilde{d}_{\delta
_{z}}\sigma\rangle\subset\underline{\C}^{n+2},$$ the
complexification of the $\tilde{d}$-central sphere congruence of
$\Lambda$. The $\tilde{d}$-surface condition on $\Lambda$,
$\mathrm{rank}\,_{\C}\,\Lambda^{(1)}_{\tilde{d}}=3$, shows that
$\Lambda ^{1,0}_{\tilde{d}}$ and $\Lambda ^{0,1}_{\tilde{d}}$ are
complex $\mathrm{rank}$ $2$ bundles. On the other hand, the fact
that $\tilde{d}$ is metric connection gives
$$(\tilde{d}_{\delta_{z}}\sigma,\sigma)=0=(\tilde{d}_{\delta_{\bar{z}}}\sigma,\sigma),$$
whereas the conformality of
$\tilde{\phi}\sigma:(M,\mathcal{C}_{\tilde{\phi}\sigma}=\mathcal{C}_{\Lambda}^{\tilde{d}})\rightarrow\R^{n+1,1}$,
fixing an isomorphism
$\tilde{\phi}:(\R^{n+1,1},\tilde{d})\rightarrow(\R^{n+1,1},d)$,
gives
$$(\tilde{d}_{\delta_{z}}\sigma,\tilde{d}_{\delta_{z}}\sigma)=0=(\tilde{d}_{\delta_{\bar{z}}}\sigma,\tilde{d}_{\delta_{\bar{z}}}\sigma).$$
We conclude that $\Lambda ^{1,0}_{\tilde{d}}$ and $\Lambda
^{0,1}_{\tilde{d}}$ are isotropic. The fact that $S^{\tilde{d}}$ has
complex rank $4$ shows that $\Lambda ^{1,0}_{\tilde{d}}$ and
$\Lambda ^{0,1}_{\tilde{d}}$ intersect as the complexification
$$\Lambda=\Lambda
^{1,0}_{\tilde{d}}\cap\Lambda ^{0,1}_{\tilde{d}},$$ of $\Lambda$.

Notation: given $i\neq j\in\{0,1\}$,
$\tilde{d}^{i,j}:=\tilde{d}\vert_{\Gamma(T^{i,j}M)}.$
\begin{defn}\label{defcomplexsurface}
We define a \emph{complexified} $\tilde{d}$-\emph{surface} to be a
pair $(\Delta^{1,0},\Delta^{0,1})$ of isotropic rank $2$ subbundles
of $\underline{\C}^{n+2}$ intersecting in a rank $1$ bundle
$$\Delta:=\Delta^{1,0}\cap\Delta^{0,1}$$ such that
\begin{equation}\label{eq:defcomplsurf}
\tilde{d}^{1,0}\Gamma(\Delta)\subset\Omega^{1,0}(\Delta^{1,0}),\,\,\,\,\,\,\tilde{d}^{0,1}\Gamma(\Delta)\subset\Omega^{0,1}(\Delta^{0,1}).
\end{equation}
In the particular case of $\tilde{d}=d$, we shall, alternatively,
omit the reference to $\tilde{d}$.
\end{defn}

It is, perhaps, worth remarking that a complexified surface does not
necessarily define an immersion $\Delta$ of $M$ in
$\mathbb{P}(\mathcal{L})$.

Obviously, if $\tilde{d}$ is real, then, given a $\tilde{d}$-surface
$\Lambda$, the pair
$(\Lambda^{1,0}_{\tilde{d}},\Lambda^{0,1}_{\tilde{d}})$ consists of
a complexified $\tilde{d}$-surface with respect to
$\mathcal{C}_{\Lambda}^{\tilde{d}}$. Observe that the isotropy of
$\Delta^{1,0}$ establishes, in particular,
\begin{equation}\label{eq:tildedsigma10}
(\tilde{d}^{1,0}\sigma,\tilde{d}^{1,0}\sigma)=0,
\end{equation}
fixing $\sigma\in\Gamma(\Delta)$ never-zero. In the particular case
$\tilde{d}$ is real and $\Delta$ is a $\tilde{d}$-surface, (we may
refer to $\mathcal{C}_{\Delta}^{\tilde{d}}$ and) equation
\eqref{eq:tildedsigma10} is equivalent to
$$\mathcal{C}=\mathcal{C}_{\Delta}^{\tilde{d}},$$ which, together with
$\Delta^{1,0}\supset\langle\sigma\rangle\,+\,\tilde{d}^{1,0}\sigma(TM)$
and with
$\Delta^{0,1}\supset\langle\sigma\rangle\,+\,\tilde{d}^{0,1}\sigma(TM)$,
shows that
$$\Delta^{1,0}=\Delta^{1,0}_{\tilde{d}},\,\,\,\,\,\,\Delta^{0,1}=\Delta^{0,1}_{\tilde{d}}$$
and, in particular, that  $(\Delta^{1,0},\Delta^{0,1})$ is uniquely
determined by $\Delta^{1,0}\cap\Delta^{0,1}$. We conclude that, for
$\tilde{d}$ real and under the correspondence given by
$$(\Delta^{1,0},\Delta^{0,1})\leftrightarrow
\Delta^{1,0}\cap\Delta^{0,1},$$ $\tilde{d}$-surfaces are the
complexified $\tilde{d}$-surfaces given by a pair intersecting as a
real line subbundle $\Delta$ of $\underline{\C}^{n+2}$,
$$\Delta=\langle\sigma\rangle ^{\C},$$ for some
$\sigma\in\Gamma(\underline{\R}^{n+1,1})$, such that
$$\mathrm{rank}(\langle\sigma\rangle+\tilde{d}^{1,0}\sigma
(TM)+\tilde{d}^{0,1}\sigma(TM))=3.$$In particular, $d$-surfaces are
the complexified $d$-surfaces given by a pair intersecting as a real
line subbundle of $\underline{\C}^{n+2}$ defining an immersion of
$M$ into the projectivized light-cone. Henceforth, we drop the term
``complexified", referring, when necessary, to \textit{real}
$\tilde{d}$-surfaces, in order to make a distinction.

In what follows in this section, let $(\Delta^{1,0},\Delta^{0,1})$
be a $\tilde{d}$-surface.

\begin{defn}
A non-degenerate rank $\,4$ subbundle $V$ of $\underline{\C}^{n+2}$
is said to be an \emph{enveloping}
 \emph{sphere} \emph{congruence} of
$(\Delta^{1,0},\Delta^{0,1})$ if $\Delta^{1,0}+\Delta^{0,1}\subset
V$.
\end{defn}

Note that, if $V$ is an enveloping sphere congruence of
$(\Delta^{1,0},\Delta^{0,1})$, then
\begin{equation}\label{eq:NLambda0}
\mathcal{N}^{\tilde{d}}_{V}\,\Delta=0.
\end{equation}

\begin{prop}\label{curlyDpreservam1001}
If $V$ is an enveloping sphere congruence of
$(\Delta^{1,0},\Delta^{0,1})$, then
$$(\mathcal{D}_{V}^{\tilde{d}})^{1,0}\Gamma(\Delta^{1,0})\subset\Omega^{1,0}(\Delta^{1,0}),\,\,\,\,\,\,(\mathcal{D}_{V}^{\tilde{d}})^{0,1}\,\Gamma(\Delta^{0,1})\subset\Omega^{0,1}(\Delta^{0,1}).$$
\end{prop}
Before proceeding to the proof of the proposition, observe that, if
$V$ is an enveloping sphere congruence of
$(\Delta^{1,0},\Delta^{0,1})$, then
$\mathrm{rank}\,\Delta^{1,0}=\frac{1}{2}\,\mathrm{rank}\,V=\mathrm{rank}\,\Delta^{0,1}$,
which, together with the isotropy of $\Delta^{1,0}$ and
$\Delta^{0,1}$, and having in consideration the non-degeneracy of
$V$, gives
$$(\Delta^{1,0})^{\perp}\cap V=\Delta^{1,0},\,\,\,(\Delta^{0,1})^{\perp}\cap V=\Delta^{0,1},$$
$\Delta^{1,0}$ and $\Delta^{0,1}$ are maximal isotropic in $V$.

Now we proceed to the proof of Proposition
\ref{curlyDpreservam1001}.
\begin{proof}
Condition \eqref {eq:defcomplsurf}, together with
\eqref{eq:NLambda0}, establishes
$(\mathcal{D}_{V}^{\tilde{d}})^{1,0}\,\Gamma(\Delta)\subset\Omega^{1,0}(\Delta^{1,0})$
and
$(\mathcal{D}_{V}^{\tilde{d}})^{0,1}\,\Gamma(\Delta)\subset\Omega^{0,1}(\Delta^{0,1})$.
Write $\Delta^{1,0}=\langle\sigma,\tau\rangle$ with
$\sigma\in\Gamma(\Delta)$ never-zero. To conclude that
$(\mathcal{D}_{V}^{\tilde{d}})^{1,0}\Gamma(\Delta^{1,0})\subset\Omega^{1,0}(\Delta^{1,0})$,
we are left to verify that $(\mathcal{D}_{V}^{\tilde{d}})^{1,0}\tau$
takes values in $\Delta^{1,0}$, or, equivalently, in
$(\Delta^{1,0})^{\perp}$. The fact that
$\mathcal{D}_{V}^{\tilde{d}}$ is metric and $\Delta^{1,0}$ is
isotropic establishes
$$((\mathcal{D}_{V}^{\tilde{d}})^{1,0}\tau,\tau)=\frac{1}{2}\,d^{1,0}(\tau,\tau)=0$$
and
$$0=d^{1,0}(\tau,\sigma)=((\mathcal{D}_{V}^{\tilde{d}})^{1,0}\tau,\sigma)+(\tau,
(\mathcal{D}_{V}^{\tilde{d}})^{1,0}\sigma),$$ and, therefore,
$((\mathcal{D}_{V}^{\tilde{d}})^{1,0}\tau,\sigma)=0$. A similar
argument applies to $\Delta^{0,1}$, completing the proof.
\end{proof}

Consider, for a moment, a real surface $\Lambda$. Let $\sigma$ be a
never-zero section of $\Lambda$ and $z$ be a holomorphic chart of
$(M, \mathcal{C}_{\Lambda})$. The central sphere congruence $S$ of
$\Lambda$ is characterized by having $\mathrm{rank}$ $4$ and
satisfying the conditions $\Lambda^{1,0}+\Lambda^{0,1}\subset S$ and
$\sigma_{z\bar{z}}\in\Gamma(S)$. In view of \eqref{eq:calNLambda0},
the condition $\sigma_{z\bar{z}}\in\Gamma(S)$ can be reformulated as
$\mathcal{N}^{1,0}\Lambda^{0,1}=0$, or, equivalently,
$\mathcal{N}^{0,1}\Lambda^{0,1}=0$. We conclude that the central
sphere congruence of $\Lambda$ is characterized, equivalently, by
enveloping $\Lambda$, together with satisfying
\eqref{eq:curlyN10Lambda01bbbbb09rgg9987654}. This motivates the
next definition.

\begin{defn}\label{defnccsc}
An enveloping sphere congruence $V$ of $(\Delta^{1,0},\Delta^{0,1})$
is said to be a \emph{$\tilde{d}$-central sphere congruence} of
$(\Delta^{1,0},\Delta^{0,1})$ if
$$(\mathcal{N}_{V}^{\tilde{d}})^{1,0}\Delta^{0,1}=0=(\mathcal{N}_{V}^{\tilde{d}})^{0,1}\Delta^{1,0}.$$
\end{defn}

As usual, in the particular case of $\tilde{d}=d$, we shall,
alternatively, refer to $V$ as a central sphere congruence.

Observe that, if, given $\sigma\in\Gamma(\Delta)$,
$(\tilde{d}^{1,0}\sigma,\tilde{d}^{0,1}\sigma)$ is locally
never-zero, or, equivalently,
\begin{equation}\label{eq:condforVdet}
(\tilde{d}^{1,0}\sigma,\tilde{d}^{0,1}\sigma)\neq 0,
\end{equation}
at some point, then $(\Delta^{1,0},\Delta^{0,1})$ admits a unique
 $\tilde{d}$-central sphere
congruence, that is,
$V=\Delta^{1,0}+\Delta^{0,1}+\tilde{d}^{1,0}\tilde{d}^{0,1}\sigma(TM\times
TM)$. In fact, the centrality of $V$, equivalent to
$$\tilde{d}^{1,0}\Gamma(\Delta^{0,1}),\,
\tilde{d}^{0,1}\Gamma(\Delta^{1,0})\subset\Omega^{1}(V),$$ensures,
in particular, that
$\tilde{d}^{1,0}\tilde{d}^{0,1}\sigma\in\Omega^{2}(V)$. In its turn,
equation \eqref{eq:condforVdet} ensures that $\tilde{d}^{1,0}\sigma$
does not take values in $\langle\sigma\rangle$, nor does
$\tilde{d}^{0,1}\sigma$ in
$\langle\sigma,\tilde{d}^{1,0}\sigma\rangle$, and, therefore,
$\mathrm{rank}\,(\Delta^{1,0}+\Delta^{0,1})\geq 3$, leaving us to
verify that $\tilde{d}^{1,0}\tilde{d}^{0,1}\sigma$ does not take
values in $\Delta^{1,0}+\Delta^{0,1}$. Suppose it did. In that case,
$(\tilde{d}^{1,0}\tilde{d}^{0,1}\sigma,\sigma)=0$, which contradicts
$$(\tilde{d}^{1,0}\tilde{d}^{0,1}\sigma,\sigma)=d^{1,0}(\tilde{d}^{0,1}\sigma,\sigma)-(\tilde{d}^{0,1}\sigma,\tilde{d}^{1,0}\sigma)=-(\tilde{d}^{0,1}\sigma,\tilde{d}^{1,0}\sigma)\neq
0.$$

Observe, on the other hand, that condition \eqref{eq:condforVdet} is
equivalent to
$$(\tilde{d}_{\delta_{x}}\sigma,\tilde{d}_{\delta_{x}}\sigma)+(\tilde{d}_{\delta_{y}}\sigma,\tilde{d}_{\delta_{y}}\sigma)\neq 0,$$
fixing $z=x+iy$ a holomorphic chart of $(M,\mathcal{C})$, and is,
therefore, trivially satisfied in the particular case $\tilde{d}$ is
real and $\Delta$ is a real $\tilde{d}$-surface, as, in that case,
$g_{z}\in\mathcal{C}=[g_{\sigma}^{\tilde{d}}]$. It follows, and it
is worth emphasizing, that:
\begin{rem}
In the particular case $\tilde{d}$ is real and
$(\Delta^{1,0},\Delta^{0,1})$ defines a real $\tilde{d}$-surface
$\Delta^{1,0}\cap\Delta^{0,1}=:\Delta$, not only $\mathcal{C}$,
$\Delta^{1,0}$ and $\Delta^{0,1}$ are determined by $\Delta$, as
observed previously, but so is $V$, \emph{the} $\tilde{d}$-central
sphere congruence of $(\Delta^{1,0},\Delta^{0,1})$: it is the
complexification of the $\tilde{d}$-central sphere congruence of
$\Delta$, $ V=S^{\tilde{d}}_{\Delta}.$  In fact, the identities
$$\Delta^{1,0}=\langle\sigma ,\tilde{d}_{\delta
_{z}}\sigma\rangle,\,\,\,\,\,\,\Delta^{0,1}=\langle\sigma
,\tilde{d}_{\delta _{\bar{z}}}\sigma\rangle
$$
and
$$
V=\langle\sigma ,\tilde{d}_{\delta _{z}}\sigma ,\tilde{d}_{\delta
_{\bar{z}}}\sigma ,\tilde{d}_{\delta _{\bar{z}}}\tilde{d}_{\delta
_{z}}\sigma\rangle,
$$
for $\sigma\in\Gamma(\Delta)$ never-zero and $z$ a holomorphic chart
of $(M,\mathcal{C})$, still hold, at least, locally, in the case of
a general $\tilde{d}$-surface $(\Delta^{1,0},\Delta^{0,1})$,
provided that, clearly independently of the choice of $\sigma$, we
have, at some point, $\sigma\wedge\tilde{d}^{1,0}\sigma\neq 0$ and
$(\tilde{d}^{1,0}\sigma,\tilde{d}^{0,1}\sigma)\neq 0$.
\end{rem}

\section{Complexified constrained Willmore surfaces}\label{CCWSgd}

\markboth{\tiny{A. C. QUINTINO}}{\tiny{CONSTRAINED WILLMORE
SURFACES}}

In generalization of the characterization of Willmore surfaces in
space-forms in terms of the harmonicity of the central sphere
congruence, a surface $\Lambda$ in the projectivized light-cone is a
$q$-constrained Willmore surface, for some real form
$q\in\Omega^{1}(\Lambda\wedge\Lambda^{(1)})$, if and only if $S$ is
$q$-constrained harmonic with respect to the conformal structure
$\mathcal{C}_{\Lambda}$. Generalizing the class of constrained
Willmore surfaces in space-forms, we define \textit{complexified
constrained Willmore surface} by the property of admitting a
constrained harmonic central sphere congruence with a multiplier
satisfying certain specificities, as presented in this
section.\newline

\subsection{Complexified constrained Willmore surfaces and constrained
harmonicity}\label{CCWS}

Note that, given a real surface $\Lambda$, and in view of the fact
that $\mathrm{rank}\,\Lambda =1$, we have, according to
\eqref{eq:wegdebrack},
\begin{equation}\label{eq:qq=0}
[\Lambda\wedge\Lambda^{(1)},\Lambda\wedge\Lambda^{(1)}]\subset
\Lambda \wedge\Lambda =\{0\}.
\end{equation}
In particular, if $q$ is a multiplier to $\Lambda$, then $[q\wedge
q]=0$. Furthermore, according to Lemma \ref{withvswithoutdecomps},
$d^{\mathcal{D}}q=0$ if and only if
$d^{\mathcal{D}}q^{1,0}=0=d^{\mathcal{D}}q^{0,1}$, considering
$(1,0)$- and $(0,1)$-decomposition with respect to
$\mathcal{C}_{\Lambda}$. In generalization of Theorem
\ref{willmorevsharmonicitytheorem}, it follows that:
\begin{thm}\label{ffffffffadeaeddaeae}
A real surface $\Lambda$ is $q$-constrained Willmore, for some real
form $q\in\Omega^{1}(\Lambda\wedge\Lambda^{(1)})$, if and only if
its central sphere congruence is $q$-constrained harmonic with
respect to the conformal structure $\mathcal{C}_{\Lambda}$.
\end{thm}

By Theorem \ref{thm5.2}, together with Remark \ref{CoS1}, it
follows, in generalization of Theorem \ref{willmoreiffdlambdaflat},
that:

\begin{thm}\label{originallybyfran+calderbankCWviaflatness}
A real surface $\Lambda$ is $q$-constrained Willmore, for some real
form $q\in\Omega^{1}(\Lambda\wedge\Lambda^{(1)})$, if and only if,
considering $(1,0)$- and $(0,1)$-decomposition with respect to
$\mathcal{C}_{\Lambda}$, $d^{\lambda}_{q}$ is a flat connection, for
each $\lambda\in\C\backslash\{0\}$ or, equivalently, for each
$\lambda\in S^{1}$.
\end{thm}

As referred previously, Theorem
\ref{originallybyfran+calderbankCWviaflatness} was originally
established by F. Burstall and D. Calderbank
\cite{burstall+calderbank}.

\begin{rem}
Theorem \ref{CWeq} and Lemma \ref{withvswithoutdecomps} combine to
establish, in particular, that if $q$ is a multiplier to a real
surface $\Lambda$, then, considering $(1,0)$- and
$(0,1)$-decomposition with respect to $\mathcal{C}_{\Lambda}$,
$q^{1,0}\in\Omega^{1,0}(\Lambda\wedge\Lambda^{0,1})$, or,
equivalently, $q^{0,1}\in\Omega^{0,1}(\Lambda\wedge\Lambda^{1,0})$.
\end{rem}

In view of Theorem \ref{ffffffffadeaeddaeae}, we generalize the
class of constrained Willmore surfaces in space-forms with the
following definition:

\begin{defn}
A  $\tilde{d}$-surface $(\Delta^{1,0},\Delta^{0,1})$ is said to be a
constrained Willmore $\tilde{d}$-surface if there exist
\begin{equation}\label{eq:q10q01forcomplexcW}
q^{1,0}\in\Omega^{1,0}(\wedge^{2}\Delta^{0,1}),\,\,\,\,\,\,q^{0,1}\in\Omega^{0,1}(\wedge^{2}\Delta^{1,0})
\end{equation}
for which, setting $q:=q^{1,0}+q^{0,1}$,
$(\Delta^{1,0},\Delta^{0,1})$ admits a $(q,\tilde{d})$-constrained
harmonic $\tilde{d}$-central sphere congruence.
\end{defn}

In the conditions of the definition above, we may refer to
$(\Delta^{1,0},\Delta^{0,1})$ as a
$(q,\tilde{d})$-\textit{constrained} \textit{Willmore}
\textit{surface}. In the particular case $\tilde{d}$ is real,
$(\Delta^{1,0},\Delta^{0,1})$ is a real $\tilde{d}$-surface and $q$
is real, we may refer to $(\Delta^{1,0},\Delta^{0,1})$ as a real
constrained Willmore $\tilde{d}$-surface or a real
$(q,\tilde{d})$-constrained Willmore surface. For simplicity, we
may, alternatively, refer to a $(q,\tilde{d})$-constrained harmonic
$\tilde{d}$-central sphere congruence as a
$(q,\tilde{d})$-\textit{central sphere congruence}. The form $q$ is
said to be a \textit{multiplier} to the constrained Willmore
$\tilde{d}$-surface $(\Delta^{1,0},\Delta^{0,1})$. In the particular
case $\tilde{d}=d$, we shall omit the reference to $\tilde{d}$ and
refer to $(\Delta^{1,0},\Delta^{0,1})$ as a constrained Willmore
surface or a $q$-constrained Willmore surface, or, in the case
$\tilde{d}$, $(\Delta^{1,0},\Delta^{0,1})$ and $q$ are real, as a
real $q$-constrained Willmore surface. In the light of this
terminology, the surfaces studied in Chapter \ref{sec:CWseq} shall
be renamed \textit{real constrained Willmore surfaces}. In the
particular case $q=0$ we may refer to $(\Delta^{1,0},\Delta^{0,1})$
as a Willmore $\tilde{d}$-surface, in line with Definition
\ref{eq:tildewilmm}; or, alternatively, in the case $\tilde{d}$ is
real and $(\Delta^{1,0},\Delta^{0,1})$ is a real
$\tilde{d}$-surface, as a \textit{real Willmore
$\tilde{d}$-surface}. When referring to a surface in the
projectivized light-cone, or, equivalently, in some space-form, it
shall be understood a real surface, with no need to express it.

\subsection{Complexified constrained Willmore surfaces under change
of flat metric connection}

Let $(\Delta^{1,0},\Delta^{0,1})$ be a pair of isotropic
$\mathrm{rank}\,2$ subbundles of $\underline{\C}^{n+2}$,
intersecting in a rank $1$ bundle, and
$$\tilde{\phi}:(\underline{\C}^{n+2},\tilde{d})\rightarrow
(\underline{\C}^{n+2},d)$$ be an isomorphism. Obviously,
$(\tilde{\phi}\,\Delta^{1,0},\tilde{\phi}\,\Delta^{0,1})$ is another
pair of isotropic $\mathrm{rank}\,2$ subbundles of
$\underline{\C}^{n+2}$. It is clear that
$(\Delta^{1,0},\Delta^{0,1})$ is a $\tilde{d}$-surface if and only
if $(\tilde{\phi}\,\Delta^{1,0},\tilde{\phi}\,\Delta^{0,1})$ is a
$d$-surface, and that, according to \eqref {eq:isomV}, given a
$\tilde{d}$-central sphere congruence $V$ of
$(\Delta^{1,0},\Delta^{0,1})$, $\tilde{\phi}\,V$ is a central sphere
congruence of
$(\tilde{\phi}\,\Delta^{1,0},\tilde{\phi}\,\Delta^{0,1})$.
Furthermore:

\begin{prop}\label{CWunderphi}
Suppose $(\Delta^{1,0},\Delta^{0,1})$ is a $\tilde{d}$-surface. In
that case, $(\Delta^{1,0},\Delta^{0,1})$ is a
$(q,\tilde{d})$-constrained Willmore surface admitting $V$ as a
$(q,\tilde{d})$-central sphere congruence if and only if
$(\tilde{\phi}\,\Delta^{1,0},\tilde{\phi}\,\Delta^{0,1})$ is a
$\mathrm{Ad}_{\tilde{\phi}}\,q$-constrained Willmore surface
admitting $\tilde{\phi}\,V$ as a
$\mathrm{Ad}_{\tilde{\phi}}\,q$-central sphere congruence.
\end{prop}

\begin{proof}
According to equation \eqref {eq:adjwedge},
$\mathrm{Ad}_{\tilde{\phi}}\,q^{i,j}\in\Omega^{i,j}(\wedge^{2}\tilde{\phi}\,\Delta^{j,i})$,
for $i\neq j\in\{0,1\}$. The result comes as an immediate
consequence of Theorem \ref{prop5.5}.
\end{proof}

If $\tilde{d}$ is real, one can take $\tilde{\phi}$ to be real, in
which case $(\Delta^{1,0},\Delta^{0,1})$ is a real
$\tilde{d}$-surface if and only if
$(\tilde{\phi}\,\Delta^{1,0},\tilde{\phi}\,\Delta^{0,1})$ is a real
$d$-surface:
$$\overline{\tilde{\phi}\,\Delta^{1,0}\cap\tilde{\phi}\,\Delta^{0,1}}=
\tilde{\phi}\,\overline{\Delta^{1,0}\cap\Delta^{0,1}}$$ and, given
$\sigma\in\Gamma(\Delta^{1,0}\cap\Delta^{0,1})$ never-zero,
$$\langle\tilde{\phi}\sigma\rangle+d^{1,0}(\tilde{\phi}\sigma)(TM)+d^{0,1}(\tilde{\phi}\sigma)(TM)=
\tilde{\phi}(\langle\sigma\rangle+\tilde{d}^{1,0}\sigma(TM)+\tilde{d}^{0,1}\sigma(TM)).$$
Furthermore, $(\Delta^{1,0},\Delta^{0,1})$ is a real
$(q,\tilde{d})$-constrained Willmore surface if and only if
$(\tilde{\phi}\,\Delta^{1,0},\tilde{\phi}\,\Delta^{0,1})$ is a real
$\mathrm{Ad}_{\tilde{\phi}}\,q$-constrained Willmore surface.

\section{Spectral deformation of complexified constrained Willmore
surfaces}\label{subsec:cwillmfam}

\markboth{\tiny{A. C. QUINTINO}}{\tiny{CONSTRAINED WILLMORE
SURFACES}}

Complexified constrained Willmore surfaces are characterized by the
flatness of the metric connection $d^{\lambda,q}_{V}$, for all
$\lambda\in\C\backslash\{0\}$, for some central sphere congruence
$V$ and some $q\in\Omega^{1}(\wedge^{2}V\oplus\wedge^{2}V^{\perp})$
satisfying certain specificities. Given a complexified constrained
Willmore surface, such family of flat metric connections provides a
spectral deformation of the surface into new complexified
constrained Willmore surfaces, which, in the case of a complexified
Willmore surface and of the zero multiplier, remains within the
class of complexified Willmore surfaces, and, for $\lambda\in
S^{1}$, preserves reality conditions. The deformation defined by the
loop of flat metric connections $d^{\lambda}_{q}$ coincides, up to
reparametrization, with the spectral deformation of a
$q$-constrained Willmore surface in spherical space presented in
\cite{SD}.\newline

Suppose $(\Delta^{1,0},\Delta^{0,1})$ is a constrained Willmore
surface admitting $q$ as a multiplier and $V$ as a $(q,d)$-central
sphere congruence, in which case, according to Theorem \ref{thm5.2},
the $\C\backslash \{0\}$-family
$$d^{\lambda,q}_{V}=\mathcal{D}_{V}+\lambda
^{-1}\mathcal{N}^{1,0}_{V}+\lambda\mathcal{N}^{0,1}_{V}+(\lambda
^{-2}-1)q^{1,0}+(\lambda ^{2}-1)q^{0,1},$$ on
$\lambda\in\C\backslash \{0\}$, consists of a family of flat metric
connections on $\underline{\C}^{n+2}$. Consider an isomorphism
$$\phi^{\lambda}_{q}:(\underline{\C}^{n+2},d^{\lambda,q}_{V})\rightarrow
(\underline{\C}^{n+2},d).$$ Since $q^{i,j}$ takes values in
$\wedge^{2}\Delta^{j,i}$ and $\Delta^{j,i}$ is isotropic, for $i\neq
j\in\{0,1\}$, we have $q^{1,0}\Delta=0=q^{0,1}\Delta$, which,
together with \eqref {eq:NLambda0}, shows that, given
$\sigma\in\Gamma(\Delta)$ never-zero,
$d^{\lambda,q}_{V}\sigma=\mathcal{D}_{V}\sigma$. Equation \eqref
{eq:NLambda0} establishes, on the other hand,
$\mathcal{D}_{V}\sigma=d\sigma$. Hence
$$d^{i,j}(\phi^{\lambda}_{q}\sigma)=\phi^{\lambda}_{q}((d^{\lambda,q}_{V})^{i,j}\sigma)=\phi^{\lambda}_{q}(d^{i,j}\sigma),$$
for $i\neq j\in\{0,1\}$, leading us to conclude that
$(\phi^{\lambda}_{q}\,\Delta^{1,0},\phi^{\lambda}_{q}\,\Delta^{0,1})$
is still a $d$-surface. Furthermore:

\begin{thm}\label{thm8.3.1}
If $(\Delta^{1,0},\Delta^{0,1})$ is a $q$-constrained Willmore
surface admitting $V$ as a $(q,d)$-central sphere congruence, then
$(\phi^{\lambda}_{q}\,\Delta^{1,0},\phi^{\lambda}_{q}\,\Delta^{0,1})$
is a $\mathrm{Ad}_{\phi^{\lambda}_{q}}(q_{\lambda})$-constrained
Willmore surface admitting $\phi^{\lambda}_{q}\,V$ as a
$(\mathrm{Ad}_{\phi^{\lambda}_{q}}(q_{\lambda}),d)$-central sphere
congruence, for each $\lambda\in\C\backslash\{0\}$.
\end{thm}

The proof will consist of showing that, if
$(\Delta^{1,0},\Delta^{0,1})$ is a $q$-constrained Willmore surface
admitting $V$ as a $(q,d)$-central sphere congruence, then
$(\Delta^{1,0},\Delta^{0,1})$ is also a
$(q_{\lambda},d^{\lambda,q}_{V})$-constrained Willmore surface
admitting $V$ as a $(q_{\lambda},d^{\lambda,q}_{V})$-central sphere
congruence, for each $\lambda\in\C\backslash\{0\}$.
\begin{proof}
Fix $\lambda\in\C\backslash\{0\}$. The crucial fact that
\begin{equation}\label{eq:crucial}
d^{\lambda,q}_{V}\sigma=d\sigma,
\end{equation}
for $\sigma\in\Gamma(\Delta)$ never-zero, shows that, if
$(\Delta^{1,0},\Delta^{0,1})$ is a surface, then
$(\Delta^{1,0},\Delta^{0,1})$ is also a $d^{\lambda,q}_{V}$-surface.
On the other hand, as $q^{1,0},q^{0,1}\in\Omega^{1}(\wedge^{2}V)$
and $V$ and $V^{\perp}$ are $\mathcal{D}$-parallel, we have
$$\mathcal{N}_{V}^{d^{\lambda,q}_{V}}=\lambda^{-1}\mathcal{N}_{V}^{1,0}+\lambda\mathcal{N}_{V}^{0,1},$$
making clear that, if $V$ is a central sphere congruence of
$(\Delta^{1,0},\Delta^{0,1})$, then $V$ is, as well, a
$d^{\lambda,q}_{V}$-central sphere congruence of
$(\Delta^{1,0},\Delta^{0,1})$. Furthermore: according to Theorem
\ref{th5.0.4}, the $(q,d)$-constrained harmonicity of $V$ ensures
its $(q_{\lambda},d^{\lambda,q}_{V})$-constrained harmonicity. Since
$q_{\lambda}^{1,0}$ and $q_{\lambda}^{0,1}$ are scales of,
respectively, $q^{1,0}$ and $q^{0,1}$, we conclude that
$q_{\lambda}^{1,0}\in\Omega^{1,0}(\wedge^{2}\Delta^{0,1})$ and
$q_{\lambda}^{0,1}\in\Omega^{0,1}(\wedge^{2}\Delta^{1,0})$. The
result follows now from Proposition \ref{CWunderphi}.
\end{proof}

The $q$-constrained harmonicity of $V$, characterized by the
flatness of $d^{\lambda,q}_{V}$, for all $\lambda\in
\C\backslash\{0\}$, establishes, equivalently, the flatness of
$$d^{\mu,\mathrm{Ad}_{\phi ^{\lambda}_{q}}q_{\lambda}}_{\phi
^{\lambda}_{q}V}=\phi ^{\lambda}_{q}\circ
(d^{\lambda,q}_{V})^{\mu,q_{\lambda}}_{V} \circ (\phi
^{\lambda}_{q})^{-1}=\phi ^{\lambda}_{q}\circ d^{\lambda\mu,q}_{V}
\circ (\phi ^{\lambda}_{q})^{-1},$$ for all $\lambda,\mu\in
\C\backslash\{0\}$, or, equivalently, the $\mathrm{Ad}_{\phi
^{\lambda}_{q}}q_{\lambda}$-harmonicity of $\phi ^{\lambda}_{q}V$,
for all $\lambda$. On the other hand, for each
$\lambda\in\C\backslash\{0\}$, the deformation $\phi
^{\lambda}_{q}V$ of $V$ is a central sphere congruence to the
deformation
$(\phi^{\lambda}_{q}\,\Delta^{1,0},\phi^{\lambda}_{q}\,\Delta^{0,1})$
of $(\Delta^{1,0},\Delta^{0,1})$. For each $\lambda\in\C\backslash
\{0\}$, the flat metric connection $d^{\lambda,q}_{V}$ provides, in
this way, a deformation of the constrained Willmore surface
$(\Delta^{1,0},\Delta^{0,1})$ into another constrained Willmore
surface. Note that in the case $(\Delta^{1,0},\Delta^{0,1})$ is a
Willmore surface and $q=0$,
$(\phi^{\lambda}_{q}\,\Delta^{1,0},\phi^{\lambda}_{q}\,\Delta^{0,1})$
is still a Willmore surface. Such a $\C\backslash \{0\}$-spectral
deformation of constrained Willmore surfaces provides, in
particular, a $S^{1}$-spectral deformation of real constrained
Willmore surfaces, as we observe next.

\subsection{Real spectral deformation}\label{realspecofCW}

Suppose that $(\Delta^{1,0},\Delta^{0,1})$ defines a real
$q$-constrained Willmore surface
$\Lambda:=\Delta^{1,0}\cap\Delta^{0,1}$. In that case, cf. Remark
\ref{s1familyofconnections}, for each $\lambda\in S^{1}$,
$d^{\lambda,q}_{V}$ defines a connection on
$\underline{\R}^{n+1,1}$, so that we can take $\phi^{\lambda}_{q}$
to be real. Fix $\lambda\in S^{1}$. Consider $\phi^{\lambda}_{q}$ to
be real, i.e., $\phi^{\lambda}_{q}$ the complex linear extension to
$\underline{\C}^{n+2}$ of an isomorphism
$$\phi^{\lambda}_{q}:(\underline{\R}^{n+1,1},d^{\lambda,q}_{V})\rightarrow
(\underline{\R}^{n+1,1},d).$$ Then, obviously,
$\overline{\phi^{\lambda}_{q}\Lambda}=\phi^{\lambda}_{q}\Lambda$. On
the other hand, given $\sigma\in\Gamma(\Lambda)$ never-zero, the
crucial equation \eqref{eq:crucial} gives
$$\langle\phi^{\lambda}_{q}\sigma\rangle+d^{1,0}(\phi^{\lambda}_{q}\sigma)(TM)+d^{0,1}(\phi^{\lambda}_{q}\sigma)(TM)=
\langle\sigma\rangle+d^{1,0}\sigma(TM)+d^{1,0}\sigma(TM).$$ We
conclude that
$(\phi^{\lambda}_{q}\,\Delta^{1,0},\phi^{\lambda}_{q}\,\Delta^{0,1})$
still defines a real surface,
$$\phi^{\lambda}_{q}\Lambda=:\Lambda_{q}^{\lambda},$$
which we denote by the \textit{spectral deformation of parameter
$\lambda$ of $\Lambda$, corresponding to the multiplier $q$}.

For $\lambda\in S^{1}$, the deformation of
$(\Delta^{1,0},\Delta^{0,1})$ provided by $d^{\lambda,q}_{V}$
preserves reality conditions. It preserves, as well, the conformal
structure induced in $M$: yet again according to equation \eqref
{eq:crucial}, given $\sigma\in\Gamma(\Lambda)$ never-zero,
$g_{\sigma}^{d^{\lambda,q}_{V}}=g_{\sigma}$ and, therefore,
$$\mathcal{C}_{\Lambda^{\lambda}_{q}}=\mathcal{C}_{\Lambda}.$$
According to Theorem \ref{thm8.3.1}, it preserves the central sphere
congruence, as well,
\begin{equation}\label{eq:cscpreserved byCWdeform}
S_{\phi_{q}^{\lambda}\Lambda}=\phi^{\lambda}_{q}S_{\Lambda}.
\end{equation}

Following Theorem \ref{thm8.3.1}, we have:
\begin{thm}\label{spdeformCW}
Suppose that $\Lambda$ is a real $q$-constrained Willmore surface,
for some $q\in\Omega^{1}(\Lambda\wedge\Lambda^{(1)})$. Then, for
each $\lambda\in S^{1}$, the transformation
$\phi^{\lambda}_{q}\Lambda$ of $\Lambda$ defined by the flat metric
connection $d^{\lambda}_{q}$ is a real
$\mathrm{Ad}_{\phi^{\lambda}_{q}}(q_{\lambda})$-constrained Willmore
surface.
\end{thm}

The loop of flat metric connections $d^{\lambda}_{q}$ defines in
this way a $S^{1}$-deformation of the real $q$-constrained Willmore
surface $\Lambda$ into a family of real constrained Willmore
surfaces. In the particular case $\Lambda$ is Willmore and $q=0$,
the deformation remains within the class of Willmore surfaces; in
fact, it coincides  with the deformation presented in Section \ref
{subsec:willmfam}.

An alternative perspective on this spectral deformation of the real
$q$-constrained Willmore surface $\Lambda$ is that of the action
$$\Lambda\subset(\underline{\R}^{n+1,1},d)\mapsto\Lambda\subset(\underline{\R}^{n+1,1}, d^{\lambda}_{q}),$$
of the loop of flat metric connections $d^{\lambda}_{q}$ on
$\{\Lambda\}$, consisting of the change of the trivial flat
connection on $\underline{\R}^{n+1,1}$ into the flat metric
connection $d^{\lambda}_{q}$, for each $\lambda$. Equation
\eqref{eq:crucial} establishes
$\Lambda^{(1)}_{d^{\lambda}_{q}}=\Lambda^{(1)}$ and, therefore,
$\Lambda$ as $d^{\lambda}_{q}$-surface. Equation
\eqref{eq:conndeconn} establishes, furthermore, $\Lambda$ as a
$(q_{\lambda},d^{\lambda}_{q})$-constrained Willmore surface.

We complete this section by verifying that the deformation defined
by the loop of flat metric connections $d^{\lambda}_{q}$ coincides,
up to reparametrization, with the spectral deformation of a
$q$-constrained Willmore surface in spherical space defined in
\cite{SD}.\footnote{The omission,  in \cite{SD}, of reference to the
transformation rule of the normal connection shall be understood as
preservation.}Fix a holomorphic chart $z$ of
$(M,\mathcal{C}_{\Lambda^{\lambda}_{q}})=(M,\mathcal{C}_{\Lambda})$.
We have
$$g_{\phi_{q}^{\lambda}\sigma^{z}}=g_{\sigma^{z}}=g_{z},$$ showing
that $\phi_{q}^{\lambda}\sigma^{z}$ is the normalized section of
$\phi_{q}^{\lambda}\Lambda$ with respect to $z$. Since
\begin{eqnarray*}
(\phi_{q}^{\lambda}\sigma^{z})_{zz}&=&(\phi_{q}^{\lambda}\sigma_{z}^{z})_{z}\\&=&
\phi_{q}^{\lambda}((d^{\lambda}_{q})_{\delta_{z}}\sigma_{z}^{z})\\&=&
\phi_{q}^{\lambda}(\mathcal{D}_{\delta_{z}}\sigma_{z}^{z}+\lambda^{-1}\mathcal{N}_{\delta_{z}}\sigma_{z}^{z}+(\lambda^{-2}-1)q_{\delta_{z}}\sigma_{z}^{z})
\\&=&\phi_{q}^{\lambda}(\pi_{S}\sigma_{zz}^{z}+\lambda^{-1}\pi_{S^{\perp}}\sigma_{zz}^{z}-\frac{1}{2}(\lambda^{-2}-1)\,q^{z}\sigma^{z})
\end{eqnarray*}
and, ultimately,
$$(\phi_{q}^{\lambda}\sigma^{z})_{zz}=-\frac{1}{2}\,(c^{z}+(\lambda^{-2}-1)\,q^{z})\,\phi_{q}^{\lambda}\sigma^{z}+\lambda^{-1}\phi_{q}^{\lambda}k^{z},$$
we conclude that $(k^{\lambda}_{q})^{z}$ and
$(c^{\lambda}_{q})^{z}$, the Hopf differential and the Schwarzian
derivative, respectively, of $\phi_{q}^{\lambda}\Lambda$ with
respect to $z$, relate to those of $\Lambda$ by
$$(k^{\lambda}_{q})^{z}=\lambda^{-1}\phi_{q}^{\lambda}k^{z},\,\,\,\,\,\,(c^{\lambda}_{q})^{z}=c^{z}+(\lambda^{-2}-1)\,q^{z}.$$
By Lemma \ref{thmonHopfeSchwarz}, having in consideration
\eqref{eq:cscpreserved byCWdeform}, the conclusion follows.

\section{Dressing action}\label{sec:dress}

\markboth{\tiny{A. C. QUINTINO}}{\tiny{CONSTRAINED WILLMORE
SURFACES}}

We use a version of the dressing action theory of C.-L. Terng and K.
Uhlenbeck \cite{uhlenbeck} to build transformations of constrained
Willmore surfaces. We start by defining a local action of a group of
rational maps on the set of flat metric connections of the type
$\hat{d}^{\lambda,q}_{S}$, with $\hat{d}$ flat metric connection on
$\underline{\C}^{n+2}$ and
$q\in\Omega^{1}(\wedge^{2}S\oplus\wedge^{2}S^{\perp})$. Namely,
given $r=r(\lambda)\in\Gamma(O(\underline{\C}^{n+2}))$ holomorphic
at $\lambda=0$ and $\lambda=\infty$ and twisted in the sense that
$\rho r(\lambda)\rho=r(-\lambda)$, for $\rho$ reflection across $S$,
we define a $1$-form $\hat{q}$ with values in $\wedge^{2}S$ (note
that the fact that $r(\lambda)$ is twisted establishes that both
$r(0)$ and $r(\infty)$ preserve $S$) by
$\hat{q}^{1,0}:=\mathrm{Ad}_{r(0)}q^{1,0},\,\,\,\,\hat{q}\,^{0,1}:=\mathrm{Ad}_{r(\infty)}q^{0,1}$,
and a new family of metric connections from $d^{\lambda,q}_{S}$ by
$\hat{d}^{\lambda,\hat{q}}_{S}:=r(\lambda)\circ
d^{\lambda,q}_{S}\circ r(\lambda)^{-1}$. Obviously, for each
$\lambda$, the flatness of $\hat{d}^{\lambda,\hat{q}}_{S}$ is
equivalent to that of $d^{\lambda,q}_{S}$. Crucially, if
$\hat{d}^{\lambda,\hat{q}}_{S}$ admits a holomorphic extension to
$\lambda\in\C\backslash\{0\}$ through metric connections on
$\underline{\C}^{n+2}$, then the notation
$\hat{d}^{\lambda,\hat{q}}_{S}$ proves to be not merely formal, for
$\hat{d}:=\hat{d}^{1,\hat{q}}_{S}$. In that case, it follows that,
if $\Lambda$ is $q$-constrained Willmore, then $S$ is
$(\hat{q},\hat{d})$-constrained harmonic and, therefore, in the case
$1\in\mathrm{dom} (r)$, $S^{*}:=r(1)^{-1}S$ is $q^{*}$-constrained
harmonic, for $q^{*}:=\mathrm{Ad}_{r(1)^{-1}}\hat{q}$. The
transformation of $S$ into $S^{*}$, preserving constrained
harmonicity, leads, furthermore, to a transformation of $\Lambda$
into a new constrained Willmore surface, provided that
$\mathrm{det}\,r(0)\vert_{S}=\mathrm{det}\,r(\infty)\vert_{S}$. Set
$(\Lambda^{*})^{1,0}:=r(1)^{-1}r(\infty)\Lambda^{1,0}$,
$(\Lambda^{*})^{0,1}:=r(1)^{-1}r(0)\Lambda^{0,1}$ and
$\Lambda^{*}:=(\Lambda^{*})^{1,0}\cap (\Lambda^{*})^{0,1}$. The
condition above on the determinants of $r(0)\vert_{S}$ and
$r(\infty)\vert_{S}$ establishes $\Lambda^{*}$ as a line bundle (the
argument is based on the two families of lines on the Klein
quadric). The isotropy of $\Lambda^{1,0}$ and $\Lambda^{0,1}$
ensures that of $\Lambda^{*}$. It is not clear, though, that
$\Lambda^{*}$ is a real bundle. If $\Lambda^{*}$ is a real surface,
one proves that $S^{*}$ is the central sphere congruence of
$\Lambda^{*}$ and that the bundles $(\Lambda^{*})^{1,0}$ and
$(\Lambda^{*})^{0,1}$ defined above are not merely formal. The fact
that $q^{1,0}\in\Omega^{1,0}(\wedge^{2}\Lambda^{0,1})$ establishes
$(q^{*})^{1,0}\in\Omega^{1,0}(\wedge^{2}(\Lambda^{*})^{0,1})\subset\Omega^{1,0}(\Lambda^{*}\wedge(\Lambda^{*})^{(1)})$.
We conclude that, if, furthermore, $q^{*}$ is real, then
$\Lambda^{*}$ is a $q^{*}$-constrained Willmore surface. In fact, we
use a version of the dressing action theory of C.-L. Terng and K.
Uhlenbeck \cite{uhlenbeck} to build, more generally, transformations
of constrained harmonic bundles and complexified constrained
Willmore surfaces.\newline

Let $\rho\in\Gamma(O(\underline{\C}^{n+2}))$ be reflection across
$V$,
$$\rho=\pi_{V}-\pi_{V^{\perp}}.$$Obviously, given
$w\in\Gamma(\underline{\C}^{n+2})$, $w$ is a section of $V$
(respectively, a section of $V^{\perp}$) if and only if $\rho w=w$
(respectively, $\rho w=-w$). Note that $\rho^{-1}=\rho$. Let $q$ be
a $1$-form with values in $\wedge^{2}V\oplus \wedge^{2}V^{\perp}$.
The $\mathcal{D}_{V}$-parallelness of $V$ and $V^{\perp}$, together
with the fact that $\mathcal{N}_{V}$ intertwines $V$ and
$V^{\perp}$, whereas $q$ preserves them, makes clear that
\begin{equation}\label{eq:rholam}
d^{-\lambda,q}_{V}=\rho\circ
d_{V}^{\lambda,q}\circ\rho^{-1},\end{equation}for
$\lambda\in\C\backslash \{0\}$. Suppose we have
$r(\lambda)\in\Gamma(O(\underline{\C}^{n+2}))$ such that
$\lambda\mapsto r(\lambda)$ is rational in $\lambda$, $r$ is
holomorphic and invertible at $\lambda=0$ and $\lambda=\infty$ and
twisted in the sense that
\begin{equation}\label{eq:RcommuteRho}
\rho\,r(\lambda)\,\rho^{-1}=r(-\lambda),
\end{equation}
for $\lambda\in\mathrm{dom}(r)$. In particular, it follows that both
$r(0)$ and $r(\infty)$ commute with $\rho$, and, therefore, that
\begin{equation}\label{eq:r0rinfpreserveVperp}
r(0)_{\vert_{V}},r(\infty)_{\vert_{V}}\in\Gamma(O(V)),\,\,\,\,r(0)_{\vert_{V^{\perp}}},r(\infty)_{\vert_{V^{\perp}}}\in\Gamma(O(V^{\perp})).
\end{equation}
Define $\hat{q}\in\Omega^{1}(\wedge^{2}V\oplus \wedge^{2}V^{\perp})$
by setting
$$\hat{q}^{1,0}:=\mathrm{Ad}_{r(0)}q^{1,0},\,\,\,\,\,\,\hat{q}\,^{0,1}:=\mathrm{Ad}_{r(\infty)}q^{0,1}.$$
Define a new family of metric connections on $\underline{\C}^{n+2}$
by setting
$$\hat{d}^{\lambda,\hat{q}}_{V}:=r(\lambda)\circ
d^{\lambda,q}_{V}\circ r(\lambda)^{-1}.$$ Suppose that there exists
a holomorphic extension of
$\lambda\mapsto\hat{d}^{\lambda,\hat{q}}_{V}$ to
$\lambda\in\C\backslash\{0\}$ through metric connections on
$\underline{\C}^{n+2}$. In that case, as we, crucially, verify next,
the notation $\hat{d}^{\lambda,\hat{q}}_{V}$ is not merely formal:
\begin{prop}\label{cruxh}
$$\hat{d}^{\lambda,\hat{q}}_{V}=\mathcal{D}_{V}^{\hat{d}}+\lambda^{-1}(\mathcal{N}_{V}^{\hat{d}})^{1,0}+\lambda(\mathcal{N}_{V}^{\hat{d}})^{0,1}+(\lambda^{-2}-1)\hat{q}^{1,0}+(\lambda^{2}-1)\hat{q}^{0,1},$$
for the flat metric connection
$$\hat{d}:=\hat{d}^{1,\hat{q}}_{V}=\mathrm{lim}_{\lambda\rightarrow 1}r(\lambda)\circ
d^{\lambda,q}_{V}\circ r(\lambda)^{-1}$$and $\lambda\in\C\backslash
\{0\}$.
\end{prop}
\begin{proof}
First note that, as $r$ is holomorphic and invertible at $\lambda=0$
and
$$(d^{\lambda,q}_{V})^{0,1}=\mathcal{D}_{V}^{0,1}+\lambda
\mathcal{N}_{V}^{0,1}+(\lambda^{2}-1)q^{0,1}$$ is holomorphic on
$\C$, the connection
$$(\hat{d}^{\lambda,\hat{q}}_{V})^{0,1}=r(\lambda)\circ(d^{\lambda,q}_{V})^{0,1}\circ r(\lambda)^{-1}$$
which admits a holomorphic extension to
$\lambda\in\C\backslash\{0\}$, admits, furthermore, a holomorphic
extension to $\lambda\in\C$. Thus, locally,
$$(\hat{d}^{\lambda,\hat{q}}_{V})^{0,1}=A_{0}^{0,1}+\sum_{i\geq
1}\lambda^{i}A_{i}^{0,1},$$ with $A_{0}$ connection and
$A_{i}\in\Omega^{1}(o(\underline{\C}^{n+2}))$, for all $i$.
Considering then limits of
$$\lambda^{-2}A_{0}^{0,1}+\sum_{i\geq
1}\lambda^{i-2}A_{i}^{0,1}=r(\lambda)\circ
(\lambda^{-2}\mathcal{D}_{V}^{0,1}+\lambda^{-1}\mathcal{N}_{V}^{0,1}+(1-\lambda^{-2})q^{0,1})\circ
r(\lambda)^{-1},$$when $\lambda$ goes to infinity, we get
$$A_{2}^{0,1}+\mathrm{lim}_{\lambda\rightarrow \infty}\sum_{i\geq
3}\lambda^{i-2}A_{i}^{0,1}=\mathrm{Ad}_{r(\infty)}\,q^{0,1},$$which
shows that $A_{i}^{0,1}=0$, for all $i\geq 3$, and that
$A_{2}^{0,1}=\hat{q}^{0,1}$. Considering now limits of
$$A_{0}^{0,1}+\lambda A_{1}^{0,1}+\lambda^{2}\hat{q}^{0,1}=r(\lambda)\circ
(\mathcal{D}_{V}^{0,1}+\lambda\mathcal{N}_{V}^{0,1}+(\lambda^{2}-1)q^{0,1})\circ
r(\lambda)^{-1},$$when $\lambda$ goes to $0$, we conclude that
$A_{0}^{0,1}=r(0)\cdot(\mathcal{D}_{V}^{0,1}-q^{0,1})$ and,
therefore, that
$(\hat{d}^{\lambda,\hat{q}}_{V})^{0,1}=r(0)\cdot(\mathcal{D}_{V}^{0,1}-q^{0,1})+\lambda
A_{1}^{0,1}+\lambda^{2}\hat{q}^{0,1}$. As for
$$(\hat{d}^{\lambda,\hat{q}}_{V})^{1,0}=r(\lambda)\circ (\mathcal{D}_{V}^{1,0}+\lambda^{-1}\mathcal{N}^{1,0}+(\lambda^{-2}-1)q^{1,0})\circ r(\lambda)^{-1},$$
which has a pole at $\lambda=0$, we have, for $\lambda$ away from
$0$,
\begin{equation}\label{eq:dhat01holomwithpole0}
\sum_{i\geq 1}\lambda^{-i}A_{-i}^{1,0}+A^{1,0}_{0}+\sum_{i\geq
1}\lambda^{i}A_{i}^{1,0}=r(\lambda)\circ
(\mathcal{D}_{V}^{1,0}+\lambda^{-1}\mathcal{N}^{1,0}+(\lambda^{-2}-1)q^{1,0})\circ
r(\lambda)^{-1},
\end{equation}
with $A_{-i}^{1,0}\in\Omega^{1}(o(\underline{\C}^{n+2}))$, for all
$i\geq 1$. Considering limits of \eqref{eq:dhat01holomwithpole0}
when $\lambda$ goes to infinity, shows that $A_{i}^{1,0}=0$, for all
$i\geq 1$, and that
$A^{1,0}_{0}=r(\infty)\cdot(\mathcal{D}_{V}^{1,0}-q^{1,0})$.
Multiplying then both members of equation
\eqref{eq:dhat01holomwithpole0} by $\lambda^{2}$ and considering
limits when $\lambda$ goes to $0$, we conclude that
$A^{1,0}_{-2}=\hat{q}^{1,0}$ and that $A_{-i}^{1,0}=0$, for all
$i\geq 3$, and, ultimately, that
$(\hat{d}^{\lambda,\hat{q}}_{V})^{1,0}=r(\infty)\cdot(\mathcal{D}^{1,0}_{V}-q^{1,0})+\lambda^{-1}A_{-1}^{1,0}+\lambda^{-2}\hat{q}^{1,0}$.
Thus
\begin{eqnarray*}
\hat{d}_{V}^{\lambda,\hat{q}}&=&r(0)\cdot(\mathcal{D}_{V}^{0,1}-q^{0,1}+q^{1,0})+r(\infty)\cdot(\mathcal{D}^{1,0}_{V}-q^{1,0}+q^{0,1})\\
&& \mbox{}+ \lambda^{-1}A_{-1}^{1,0}+\lambda
A_{1}^{0,1}+(\lambda^{-2}-1)\hat{q}^{1,0}+(\lambda^{2}-1)\hat{q}^{0,1},
\end{eqnarray*}
for $\lambda\in\C\backslash \{0\}$, and, in particular,
$$\hat{d}=r(0)\cdot(\mathcal{D}_{V}^{0,1}-q^{0,1}+q^{1,0})+r(\infty)\cdot(\mathcal{D}^{1,0}_{V}-q^{1,0}+q^{0,1})+
A_{-1}^{1,0}+A_{1}^{0,1}.$$ The fact that $r(0)$ and $r(\infty)$
(and so $r(0)^{-1}$ and $r(\infty)^{-1}$), as well as $q$, preserve
$V$ and $V^{\perp}$, together with the
$\mathcal{D}_{V}$-parallelness of $V$ and of $V^{\perp}$, shows that
$\hat{d}-(A_{-1}^{1,0}+A_{1}^{0,1})$ preserves $\Gamma(V)$ and
$\Gamma(V^{\perp})$. On the other hand, equations \eqref{eq:rholam}
and \eqref{eq:RcommuteRho} combine to give
$$\hat{d}^{-\lambda,\hat{q}}_{V}=\rho\circ\hat{d}^{\lambda,\hat{q}}_{V}\circ\rho^{-1},$$
for all $\lambda\in\C\backslash\{0\}$ away from the poles of $r$ and
then, by continuity, on all of $\C\backslash\{0\}$. The particular
case of $\lambda=1$ gives $$\rho
(A_{-1}^{1,0}+A_{1}^{0,1})_{\vert_{V}}=-(A_{-1}^{1,0}+A_{1}^{0,1})_{\vert_{V}},\,\,\,\rho
(A_{-1}^{1,0}+A_{1}^{0,1})_{\vert_{V^{\perp}}}=-(A_{-1}^{1,0}+A_{1}^{0,1})_{\vert_{V^{\perp}}},$$showing
that that $A_{-1}^{1,0}+A_{1}^{0,1}\in\Omega^{1}(V\wedge
V^{\perp})$. We conclude that
\begin{equation}\label{eq:mathcalDdosr}
r(0)\cdot(\mathcal{D}_{V}^{0,1}-q^{0,1}+q^{1,0})+r(\infty)\cdot(\mathcal{D}^{1,0}_{V}-q^{1,0}+q^{0,1})=\mathcal{D}_{V}^{\hat{d}}
\end{equation}
and
$$A_{-1}^{1,0}=(\mathcal{N}_{V}^{\hat{d}})^{1,0},\,\,\,\,\,\,A_{1}^{0,1}=(\mathcal{N}_{V}^{\hat{d}})^{0,1},$$
completing the proof.
\end{proof}

Before proceeding any further, we remark on the $(1,0)$- and
$(0,1)$-components of $\mathcal{D}^{\hat{d}}_{V}$ and
$\mathcal{N}^{\hat{d}}_{V}$. According to \eqref{eq:mathcalDdosr},
\begin{equation}\label{eq:curlyDscomadjuntascomesemchapeustnne029876}
(\mathcal{D}_{V}^{\hat{d}})^{1,0}=r(\infty)\cdot(\mathcal{D}_{V}^{1,0}-q^{1,0})+\hat{q}^{1,0},\,\,\,\,(\mathcal{D}^{\hat{d}}_{V})^{0,1}=r(0)\cdot(\mathcal{D}_{V}^{0,1}-q^{0,1})+\hat{q}^{0,1}.
\end{equation}
As for $\mathcal{N}^{\hat{d}}_{V}$, according to Proposition
\ref{cruxh},
\begin{eqnarray*}
(\mathcal{N}^{\hat{d}}_{V})^{1,0}&=&\mathrm{lim}_{\lambda\rightarrow
0}\,\lambda((\hat{d}_{V}^{\lambda,\hat{q}})^{1,0}-(\mathcal{D}^{\hat{d}}_{V})^{1,0}-(\lambda^{-2}-1)\hat{q}^{1,0})\\&=&\mathrm{lim}_{\lambda\rightarrow
0}\,\lambda((\hat{d}_{V}^{\lambda,\hat{q}})^{1,0}-\lambda^{-2}\hat{q}^{1,0})\\&=&\mathrm{lim}_{\lambda\rightarrow
0}\,(r(\lambda)\circ(\lambda\, (d^{\lambda,q}_{V})^{1,0})\circ
r(\lambda)^{-1}-\lambda^{-1}\mathrm{Ad}_{r(0)}q^{1,0})\\&=&\mathrm{Ad}_{r(0)}\mathcal{N}^{1,0}_{V}+\mathrm{lim}_{\lambda\rightarrow
0}\,\frac{1}{\lambda}(\mathrm{Ad}_{r(\lambda)}-\mathrm{Ad}_{r(0)})q^{1,0}.
\end{eqnarray*}
so that
\begin{equation}\label{eq:calNhatd0,1}
(\mathcal{N}^{\hat{d}}_{V})^{1,0}=\mathrm{Ad}_{r(0)}\mathcal{N}^{1,0}_{V}+\frac{d}{d\lambda}_{\vert_{\lambda=0}}\mathrm{Ad}_{r(\lambda)}q^{1,0};
\end{equation}
and, similarly,
\begin{eqnarray*}
(\mathcal{N}^{\hat{d}}_{V})^{0,1}&=&\mathrm{lim}_{\lambda\rightarrow
\infty}\,\lambda^{-1}((\hat{d}_{V}^{\lambda,\hat{q}})^{0,1}-(\mathcal{D}^{\hat{d}}_{V})^{0,1}-(\lambda^{2}-1)\hat{q}^{0,1})\\&=&
\mathrm{Ad}_{r(\infty)}\mathcal{N}^{0,1}_{V}+\mathrm{lim}_{\lambda\rightarrow
\infty}\,(r(\lambda)\circ \lambda q^{0,1}\circ
r(\lambda)^{-1}-\lambda\mathrm{Ad}_{r(\infty)}q^{0,1})
\end{eqnarray*}
and, therefore,
\begin{equation}\label{eq:calNhatd1,0}
(\mathcal{N}^{\hat{d}}_{V})^{0,1}=\mathrm{Ad}_{r(\infty)}\mathcal{N}^{0,1}_{V}+\frac{d}{d\lambda}_{\vert_{\lambda=0}}\mathrm{Ad}_{r(\lambda^{-1})}q^{0,1}.
\end{equation}

Now suppose $V$ is a $q$-constrained harmonic bundle, in which case,
according to Theorem \ref{thm5.2}, the family $d^{\lambda,q}_{V}$,
on $\lambda\in\C\backslash \{0\}$, consists of a family of flat
metric connections on $\underline{\C}^{n+2}$. For non-zero
$\lambda\in \mathrm{dom}(r)$ and by definition of
$\hat{d}^{\lambda,\hat{q}}_{V}$, the isometry $r(\lambda)$ is made
into an isomorphism
$$r(\lambda):(\underline{\C}^{n+2},d^{\lambda,q}_{V})\rightarrow(\underline{\C}^{n+2},\hat{d}^{\lambda,\hat{q}}_{V}),$$
ensuring, in particular, that $\hat{d}^{\lambda,\hat{q}}_{V}$ is a
flat connection, as so is $d^{\lambda,q}_{V}$: the curvature tensors
of $\hat{d}^{\lambda,\hat{q}}_{V}$ and $d^{\lambda,q}_{V}$ are
related by
$$R^{\hat{d}^{\lambda,\hat{q}}_{V}}=r(\lambda)\,
R^{d^{\lambda,q}_{V}}\, r(\lambda)^{-1}.$$ Furthermore, the
vanishing of the curvature tensor of $\hat{d}^{\lambda,\hat{q}}_{V}$
for $\lambda\in\mathrm{dom}(r)\backslash \{0\}$ extends by
continuity to $\lambda\in\C\backslash \{0\}$. We are in this way
provided with a new $\C\backslash \{0\}$-family of flat metric
connections on $\underline{\C}^{n+2}$, that of
$\hat{d}^{\lambda,\hat{q}}_{V}$.

Suppose $1\in\mathrm{dom}(r)$. In the light of Proposition
\ref{cruxh}, and according to Theorem \ref {thm5.2}, the flatness of
the metric connection $\hat{d}^{\lambda,\hat{q}}_{V}$, for
$\lambda\in\C\backslash \{0\}$, is equivalent to the
$(\hat{q},\hat{d})$-constrained harmonicity of $V$, which, in its
turn and according to Proposition \ref {prop5.5}, is equivalent to
the $(\mathrm{Ad}_{r(1)^{-1}}\,\hat{q})$-constrained harmonicity of
$r(1)^{-1}V$.
\begin{thm}\label{CHtransf}
$r(1)^{-1}V$ is a $(\mathrm{Ad}_{r(1)^{-1}}\,\hat{q})$-constrained
harmonic bundle.
\end{thm}

Note that this transformation preserves the harmonicity condition.

This transformation of a constrained harmonic bundle into a new one
leads, furthermore, to a transformation of constrained Willmore
surfaces into new ones. Suppose, furthermore, that $V$ is a
$(q,d)$-central sphere congruence of some $q$-constrained Willmore
surface $(\Delta^{1,0},\Delta^{0,1})$. Set
$$\hat{\Delta}^{1,0}:=r(\infty)\Delta^{1,0},\,\,\,\,\,\,\hat{\Delta}^{0,1}:=r(0)\Delta^{0,1},$$
and suppose that
\begin{equation}\label{eq:condondet}
\mathrm{det}\,r(0)_{\vert_{V}}=\mathrm{det}\,r(\infty)_{\vert_{V}}.
\end{equation}
Then:

\begin{thm}\label{eq:thm8.4.2paraja}
$(r(1)^{-1}\hat{\Delta}^{1,0},r(1)^{-1}\hat{\Delta}^{0,1})$ is a
$(\mathrm{Ad}_{r(1)^{-1}}\hat{q})$-constrained Willmore surface
admitting $r(1)^{-1}V$ as a
$(\mathrm{Ad}_{r(1)^{-1}}\hat{q},d)$-central sphere congruence.
\end{thm}

\begin{proof}
The proof will consist of showing that
$(\hat{\Delta}^{1,0},\hat{\Delta}^{0,1})$ is a
$(\hat{q},\hat{d})$-constrained Willmore surface admitting $V$ as a
$(\hat{q},\hat{d})$-central sphere congruence. The result will then
follow from Proposition \ref{CWunderphi}.

First of all, note that, as $q$ is a multiplier to
$(\Delta^{1,0},\Delta^{0,1})$, we have, according to equation \eqref
{eq:adjwedge},
$$\hat{q}^{1,0}\in\Omega^{1,0}(\wedge^{2}\hat{\Delta}^{0,1}),\,\,\,\,\,\,\
\hat{q}^{0,1}\in\Omega^{0,1}(\wedge^{2}\hat{\Delta}^{1,0}).$$ Having
observed above that $V$ is $(\hat{q},\hat{d})$-constrained harmonic,
it remains to show that $(\hat{\Delta}^{1,0},\hat{\Delta}^{0,1})$ is
a $\hat{d}$-surface admitting $V$ as a $\hat{d}$-central sphere
congruence, as follows.

The fact $\Delta^{1,0}$ and $\Delta^{0,1}$ are rank $2$ isotropic
subbundles of $V$ ensures that so are $\hat{\Delta^{1,0}}$ and
$\hat{\Delta^{0,1}}$, as $r(0)$ and $r(\infty)$ are orthogonal
transformations and preserve $\Gamma(V)$.

The fact that $\Delta^{1,0}$ and $\Delta^{0,1}$ intersect in a rank
$1$ bundle ensures that so do $\hat{\Delta}^{1,0}$ and
$\hat{\Delta}^{0,1}$. The point is a general fact\footnote{That of
the two families of lines on the Klein quadric.} about the
Grassmannian $\mathcal{G}_{W}$ of isotropic $2$-planes in a complex
$4$-dimensional space $W$: it has two components, each an orbit of
the special orthogonal group $SO(W)$, intertwined by the action of
elements of $O(W)\backslash SO(W)$, and for which any element
intersects any element of the other component in a line while
distinct elements of the same component have trivial intersection.
Given that $\mathrm{rank}\,(\Delta^{1,0}\cap\Delta^{0,1})=1$,
$\Delta^{1,0}_{p}$ and $\Delta^{0,1}_{p}$ lie in different
components of $\mathcal{G}_{V_{p}}$ and the hypothesis
\eqref{eq:condondet} ensures that the same is true of
$\hat{\Delta}^{1,0}_{p}$ and $\hat{\Delta}^{0,1}_{p}$, for all $p$
in $M$.

Set
$$\hat{\Delta}:=\hat{\Delta}^{1,0}\cap\hat{\Delta}^{0,1}.$$We are left
to verify that
\begin{equation}\label{eq:chapeusds}
\hat{d}^{1,0}\Gamma(\hat{\Delta})\subset\Omega^{1}(\hat{\Delta}^{1,0}),\,\,\,\,\,\,\hat{d}^{0,1}\Gamma(\hat{\Delta})\subset\Omega^{1}(\hat{\Delta}^{0,1})
\end{equation}
and that
\begin{equation}\label{eq:chapeusNs}
(\mathcal{N}_{V}^{\hat{d}})^{1,0}\hat{\Delta}^{0,1}=0=(\mathcal{N}_{V}^{\hat{d}})^{0,1}\hat{\Delta}^{1,0}.
\end{equation}
Equation \eqref{eq:chapeusNs} forces
$\mathcal{N}_{V}^{\hat{d}}\hat{\Delta}=0$, in which situation,
condition \eqref{eq:chapeusds} reads, equivalently,
$$(\mathcal{D}_{V}^{\hat{d}})^{1,0}\Gamma(\hat{\Delta})\subset\Omega^{1}(\hat{\Delta}^{1,0}),\,\,\,\,\,\,(\mathcal{D}_{V}^{\hat{d}})^{0,1}\Gamma(\hat{\Delta})\subset\Omega^{1}(\hat{\Delta}^{0,1}),$$
which, in its turn, follows from
\begin{equation}\label{eq:vcsl'1'08763trfgvhbjvuf364rgbcn}
(\mathcal{D}_{V}^{\hat{d}})^{1,0}\Gamma(\hat{\Delta}^{1,0})\subset\Omega^{1}(\hat{\Delta}^{1,0}),\,\,\,\,\,\,(\mathcal{D}_{V}^{\hat{d}})^{0,1}\Gamma(\hat{\Delta}^{0,1})\subset\Omega^{1}(\hat{\Delta}^{0,1}).
\end{equation}
It is \eqref{eq:chapeusNs}
 and \eqref{eq:vcsl'1'08763trfgvhbjvuf364rgbcn} that we shall establish.

By the isotropy of $\Delta^{i,j}$,  we have
$q^{i,j}\Delta^{i,j}\subset\Delta\subset\Delta^{i,j}$, for $i\neq
j\in\{0,1\}$, which, together with Proposition
\ref{curlyDpreservam1001}, makes clear that
$$r(\infty)\cdot(\mathcal{D}_{V}^{1,0}-q^{1,0})\,\Gamma(\hat{\Delta}^{1,0})\subset\Omega^{1}(\hat{\Delta}^{1,0}),\,\,\,\,r(0)\cdot(\mathcal{D}_{V}^{0,1}-q^{0,1})\,\Gamma(\hat{\Delta}^{0,1})\subset\Omega^{1}(\hat{\Delta}^{0,1}).$$
On the other hand, as $\hat{\Delta}$ has rank $1$,
$\hat{\Delta}\wedge\hat{\Delta}^{i,j}=\wedge^{2}\hat{\Delta}^{i,j}$,
so that
$\hat{q}^{i,j}\in\Omega^{i,j}(\hat{\Delta}\wedge\hat{\Delta}^{j,i})$,
and, therefore, by the isotropy of $\hat{\Delta}^{i,j}$,
$\hat{q}^{i,j}\hat{\Delta}^{i,j}\subset\hat{\Delta}\subset\hat{\Delta}^{i,j}$,
for $i\neq j\in\{0,1\}$, completing the verification of
\eqref{eq:vcsl'1'08763trfgvhbjvuf364rgbcn}, in view of
\eqref{eq:curlyDscomadjuntascomesemchapeustnne029876}.

Finally, we establish \eqref{eq:chapeusNs}. According to
\eqref{eq:calNhatd0,1},
$$(\mathcal{N}^{\hat{d}}_{V})^{1,0}= \mathrm{Ad}_{r(0)}(\mathcal{N}^{1,0}+[r(0)^{-1}\frac{d}{d\lambda}_{\vert_{\lambda=0}}r(\lambda),q^{1,0}]).$$
The centrality of $V$ with respect to $(\Delta^{1,0},\Delta^{0,1})$
establishes, in particular,  $\mathcal{N}_{V}^{1,0}\Delta^{0,1}=0$,
whilst the isotropy of $\Delta^{0,1}$ ensures, in particular, that
$q^{1,0}\Delta^{0,1}=0$. Hence
$$\mathrm{Ad}_{r(0)}(\mathcal{N}^{1,0}+r(0)^{-1}\frac{d}{d\lambda}_{\vert_{\lambda=0}}r(\lambda)\,q^{1,0})\hat{\Delta}^{0,1}=0.$$
On the other hand, differentiation of  $r(\lambda)^{-1}=\rho\,
r(-\lambda)^{-1} \rho$, derived from equation
\eqref{eq:RcommuteRho}, gives
$$-r(\lambda)^{-1}\frac{d}{dk}_{\vert_{k=\lambda}}r(k)\,r(\lambda)^{-1}=\rho\, r(-\lambda)^{-1}\frac{d}{dk}_{\vert_{k=-\lambda}}r(k)\,r(-\lambda)^{-1}\rho,$$
or, equivalently,
$$\rho\,r(\lambda)^{-1}\frac{d}{dk}_{\vert_{k=\lambda}}r(k)\rho=-r(-\lambda)^{-1}\frac{d}{dk}_{\vert_{k=-\lambda}}r(k)\,r(-\lambda)^{-1}\rho\,r(\lambda)\rho,$$
and, therefore, yet again by equation \eqref{eq:RcommuteRho},
$$\rho\,r(\lambda)^{-1}\frac{d}{dk}_{\vert_{k=\lambda}}r(k)\rho=-r(-\lambda)^{-1}\frac{d}{dk}_{\vert_{k=-\lambda}}r(k).$$
Evaluation at $\lambda=0$ shows then that
$$\rho\,
r(0)^{-1}\frac{d}{d\lambda}_{\vert_{\lambda=0}}r(\lambda)\rho=-r(0)^{-1}\frac{d}{d\lambda}_{\vert_{\lambda=0}}r(\lambda).$$
Equivalently,
\begin{equation}\label{eq:rrsat0}
r(0)^{-1}\frac{d}{d\lambda}_{\vert_{\lambda=0}}r(\lambda)\in\Gamma(V\wedge
V^{\perp}).
\end{equation}
As $q^{1,0}V^{\perp}=0$, we conclude that
$$Ad_{r(0)}(q^{1,0}r(0)^{-1}\frac{d}{d\lambda}_{\vert_{\lambda=0}}r(\lambda))\hat{\Delta}^{0,1}=0$$
and, ultimately, that
$(\mathcal{N}_{V}^{\hat{d}})^{1,0}\hat{\Delta}^{0,1}=0$. A similar
argument near $\lambda=\infty$ establishes
$(\mathcal{N}_{V}^{\hat{d}})^{0,1}\hat{\Delta}^{1,0}=0$, completing
the proof.
\end{proof}

\section{B\"{a}cklund transformation of constrained harmonic bundles and complexified constrained
Willmore surfaces}\label{sec:tranfscentr}

\markboth{\tiny{A. C. QUINTINO}}{\tiny{CONSTRAINED WILLMORE
SURFACES}}

The classical B\"{a}cklund transformation was introduced by A.
B\"{a}cklund in the nineteenth century, providing a mean to generate
constant negative Gaussian curvature surfaces from a given one. Many
variants of this transformation have followed - for details, see
\cite{uhlenbeck}. In this section, we construct rational maps
$r(\lambda)$ satisfying the hypothesis of the dressing action
presented above, defining then a transformation of constrained
harmonic bundles and complexified constrained Willmore surfaces, the
\textit{B\"{a}cklund} \textit{transformation}. As the philosophy
underlying the work of C.-L. Terng and K. Uhlenbeck \cite{uhlenbeck}
suggests, we consider linear fractional transformations. We define
two different types of such transformations, \textit{type p} and
\textit{type q}, each one of them satisfying the hypothesis of the
dressing action with the exception of condition
\eqref{eq:detr0infty}. Iterating the procedure, in a $2$-step
process composing the two different types of transformations, will
produce a desired $r(\lambda)$. A \textit{Bianchi permutability} of
type $p$ and type $q$ transformations of constrained harmonic
bundles is established. For special choices of parameters, the
reality of $\Lambda$ as a bundle proves to establish that of
$\Lambda^{*}$, whilst the reality of $q$ establishes that of
$q^{*}$. For such a choice of parameters, $\Lambda^{*}$ is said to
be a \textit{B\"{a}cklund transform} of $\Lambda$, provided that it
immerses.\newline

Let $\rho$ denote reflection across $V$. Choose
$\alpha\in\C\backslash \{-1,0,1\}$ and a null line subbundle $L$ of
$\underline{\C}^{n+2}$ such that, locally,
\begin{equation}\label{eq:rhoLnotorthL}
\rho L\cap L^{\perp}=\{0\}.
\end{equation}
For that, note that condition \eqref{eq:rhoLnotorthL}, equivalently
characterized by $(\rho l,l)\neq 0$, fixing $l\in\Gamma(L)$
never-zero, is an open condition on points of $M$, so that it is
satisfied locally as long as, at some point $p\in M$, $L_{p}$ is not
orthogonal to $\rho_{p}L_{p}$, its reflection across $V_{p}$.
Condition \eqref{eq:rhoLnotorthL} ensures, on the one hand, that
$L\cap \rho L=\{0\}$, and, on the other hand, that $L\oplus \rho L$
is non-degenerate. Consider then projections
$\pi_{L}:\underline{\C}^{n+2}\rightarrow L$, $\pi_{\rho
L}:\underline{\C}^{n+2}\rightarrow \rho L$ and $\pi_{(L\oplus \rho
L)^{\perp}}:\underline{\C}^{n+2}\rightarrow (L\oplus \rho
L)^{\perp}$ with respect to the decomposition
$$\underline{\C}^{n+2}=L\oplus \rho L\oplus(L\oplus \rho L)^{\perp}.$$
For $\lambda\in\C\backslash\{\pm\alpha\}$, set
$$p_{\alpha,L}(\lambda):=\frac{\alpha-\lambda}{\alpha+\lambda}\,\pi_{L}+\pi_{(L\oplus \rho L)^{\perp}}+\,\frac{\alpha+\lambda}{\alpha-\lambda}\,\pi_{\rho L}$$and
$$q_{\alpha,L}(\lambda):=\frac{\lambda-\alpha}{\lambda+\alpha}\,\pi_{L}+\pi_{(L\oplus \rho L)^{\perp}}+\,\frac{\lambda+\alpha}{\lambda-\alpha}\,\pi_{\rho L},$$defining in this way
two maps of $\C\backslash\{\pm\alpha\}$ into
$\Gamma(O(\underline{\C}^{n+2}))$ that, clearly, extend
holomorphically to the Riemann sphere except $\pm\alpha$,
$$p_{\alpha,L}, q_{\alpha,L}:\mathbb{P}^{1}\backslash\{\pm
\alpha\}\rightarrow\Gamma(O(\underline{\C}^{n+2})),$$ by setting
$$p_{\alpha,L}(\infty):=-\pi_{L}+\pi_{(L\oplus \rho L)^{\perp}}-\pi_{\rho L}$$and
\begin{equation}\label{eq:qinfty=1}
q_{\alpha,L}(\infty):=I.
\end{equation}
We may, alternatively, denote $p_{\alpha,L}$ and $q_{\alpha,L}$ by,
specifically and respectively, $p_{V,\alpha,L}$ and
$q_{V,\alpha,L}$.

The transformations \textit{of type} $p$ and \textit{of type} $q$
are closely related: for $\lambda\in\C\backslash\{\pm\alpha,0\}$,
\begin{equation}\label{eq:conjap}
p_{\alpha,L}(\lambda)=q_{\alpha^{-1},L}(\lambda^{-1}),
\end{equation}
and
\begin{equation}\label{eq:p0qinftyetc}
p_{\alpha,L}(0)=q_{\alpha,L}(\infty),\,\,\,\,\,\,\,\,p_{\alpha,L}(\infty)=q_{\alpha,L}(0).
\end{equation}
Observe that
$$\mathrm{det}\,p_{\alpha,L}(\infty)_{\vert_{V}}=\mathrm{det}\,q_{\alpha,L}(0)_{\vert_{V}}=-1.$$
In fact,
$$p_{\alpha,L}(\infty)_{\vert_{V}}=q_{\alpha,L}(0)_{\vert_{V}}
=I\left\{
\begin{array}{ll} -1 & \mbox{$\mathrm{on}\,V\cap(L\oplus\rho L)$}\\ 1 &
\mbox{$\mathrm{on}\,V\cap(L\oplus\rho
L)^{\perp}$}\end{array}\right.$$and, as $\rho L\neq L$, $L$ is not a
subbundle of $V$ and, therefore, $\mathrm{rank}\,V\cap(L\oplus\rho
L)=1$ (noting that, for any subbundle $W$ of $\R^{n+1,1}$ and so, in
particular, for $L$, the intersection of $W+\rho W$ with $V$ is not
trivial). It follows that
$\mathrm{det}\,p_{\alpha,L}(0)_{\vert_{V}}\neq
\mathrm{det}\,p_{\alpha,L}(\infty)_{\vert_{V}}$ and
$\mathrm{det}\,q_{\alpha,L}(0)_{\vert_{V}}\neq
\mathrm{det}\,q_{\alpha,L}(\infty)_{\vert_{V}}$; neither
$r=p_{\alpha,L}$ nor $r=q_{\alpha,L}$ satisfies the hypothesis
\eqref{eq:condondet} of the dressing action. However, a two-step
process, composing a transformation of \textit{type $p$} with a
transformation of \textit{type $q$}, produces transformations
$r(\lambda)$ satisfying that hypothesis: by \eqref{eq:qinfty=1} and
\eqref{eq:p0qinftyetc},
$$\mathrm{det}\,p_{\alpha,L}(0)q_{\alpha',L'}(0)_{\vert_{V}}=
\mathrm{det}\,p_{\alpha,L}(\infty)q_{\alpha',L'}(\infty)_{\vert_{V}},$$
as well as
$$\mathrm{det}\,q_{\alpha,L}(0)p_{\alpha',L'}(0)_{\vert_{V}}=
\mathrm{det}\,q_{\alpha,L}(\infty)p_{\alpha',L'}(\infty)_{\vert_{V}},$$
for all $\alpha',L'$. As we shall verify next, for special choices
of parameters $\alpha ,L,\alpha'$ and $L'$, both
$p_{\alpha,L}q_{\alpha',L'}(\lambda)$ and
$q_{\alpha,L}p_{\alpha',L'}(\lambda)$ define $r(\lambda)$
satisfying, furthermore, all the hypotheses of the dressing action.

For that, first note that, for
$\lambda\in\mathbb{P}^{1}\backslash\{\pm\alpha\}$,
\begin{equation}\label{eq:pinvs}
p_{\alpha,L}(\lambda)^{-1}=p_{-\alpha,L}(\lambda)=p_{\alpha,L}(-\lambda),\,\,\,\,\,\,q_{\alpha,L}(\lambda)^{-1}=q_{-\alpha,L}(\lambda)=q_{\alpha,L}(-\lambda).
\end{equation}
On the other hand, the isometry $\rho=\rho ^{-1}$ intertwines $L$
and $\rho L$ and, therefore, preserves $(L\oplus \rho L)^{\perp}$,
which makes clear that $\rho\circ p_{\alpha,L}$ and
$p_{\alpha,L}^{-1}\circ\rho$ coincide in $L$, $\rho L$ and $(L\oplus
\rho L)^{\perp}$. Hence
\begin{equation}\label{eq:rhopqinv}
\rho
\,p_{\alpha,L}(\lambda)\rho=p_{\alpha,L}(-\lambda),\,\,\,\,\,\,\,\,\rho\,q_{\alpha,L}(\lambda)\rho=q_{\alpha,L}(-\lambda),
\end{equation}
for $\lambda\in\mathbb{P}^{1}\backslash\{\pm\alpha\}$; establishing
both $p_{\alpha,L}$ and $q_{\alpha,L}$ - as well as, therefore,
$p_{\alpha,L}q_{\alpha',L'}$ and $q_{\alpha,L}p_{\alpha',L'}$, for
all $\alpha',L'$ - as twisted in the sense of section
\ref{sec:dress}.

Now let $q$ be a $1$-form with values in $\wedge^{2}V\oplus
\wedge^{2}V^{\perp}$. For each $\lambda\in\C\backslash
\{-\alpha,0,\alpha\}$, define a new metric connection on
$\underline{\C}^{n+2}$ by setting
$$d^{\lambda,q}_{p_{_{\alpha,L}}}:=p_{\alpha,L}(\lambda)\circ
d^{\lambda,q}_{V}\circ p_{\alpha,L}(\lambda)^{-1}.$$
\begin{Lemma}\label{lemmaext}
Suppose $L$ is $d^{\alpha,q}_{V}$-parallel. In that case, there
exists a holomorphic extension of $\lambda\mapsto
d^{\lambda,q}_{p_{_{\alpha,L}}}$ to $\lambda\in\C\backslash\{0\}$
through metric connections on $\underline{\C}^{n+2}$.
\end{Lemma}
Before proceeding to the proof, observe that, as \eqref{eq:rholam}
makes clear, if $L$ is $d^{\alpha,q}_{V}$-parallel, then $\rho L$ is
$d^{-\alpha,q}_{V}$-parallel.
\begin{proof}
Since $d^{\alpha,q}_{V}$ is a metric connection,
$$d^{\alpha,q}_{V}\,\Gamma (\rho L)\subset \Omega^{1}((\rho
L)^{\perp}),$$ as well as, in view of the parallelness of $L$ with
respect to $d^{\alpha,q}_{V}$,
$$d^{\alpha,q}_{V}\,\Gamma ((L\oplus\rho L)^{\perp})\subset
\Omega^{1} (L^{\perp}).$$ For simplicity use $\pi_{\perp}$ to denote
the orthogonal projection of $\underline{\C}^{n+2}$ onto $(L\oplus
\rho L)^{\perp}$. As
$$L^{\perp}=L\oplus(L\oplus\rho L)^{\perp},\,\,\,\,\,(\rho
L)^{\perp}=\rho L\oplus(L\oplus\rho L)^{\perp},$$ we conclude that
$\pi_{L}\circ d^{\alpha,q}_{V}\circ\pi_{\rho L}=0=\pi_{\rho L}\circ
d^{\alpha,q}_{V}\circ \pi_{\perp}$, showing that $d^{\alpha,q}_{V}$
splits as $$d^{\alpha,q}_{V}=D^{\alpha}_{q}+\beta^{\alpha}_{q}$$ for
the connection
$$D^{\alpha}_{q}:=d^{\alpha,q}_{V}\circ\pi_{L}+\pi_{\rho L}\circ
d^{\alpha,q}_{V}\circ\pi_{\rho L}+\pi_{\perp}\circ
d^{\alpha,q}_{V}\circ\pi_{\perp},$$ on $\underline{\C}^{n+2}$, and
the $1$-form  $$\beta^{\alpha}_{q}:=\pi_{\perp}\circ
d^{\alpha,q}_{V}\circ\pi_{\rho L}+\pi_{L}\circ
d^{\alpha,q}_{V}\circ\pi_{\perp}\in\Omega^{1}(L\wedge(L\oplus\rho
L)^{\perp}).$$ Clearly, for each $\lambda$,
$$p_{\alpha,L}(\lambda)\circ D^{\alpha}_{q}\circ
p_{\alpha,L}(\lambda)^{-1}=D^{\alpha}_{q},\,\,\,\,\,\,\,
p_{\alpha,L}(\lambda)\,\beta^{\alpha}_{q}\,p_{\alpha,L}(\lambda)^{-1}=\frac{\alpha-\lambda}{\alpha+\lambda}\,\beta^{\alpha}_{q}.$$
Now decompose $d^{\lambda,q}_{V}$ as
$$d^{\lambda,q}_{V}=d^{\alpha,q}_{V}+(\lambda-\alpha)A(\lambda),$$for $\lambda\in\C\backslash \{0,\alpha\}$,
with $\lambda\mapsto
A(\lambda)\in\Omega^{1}(o(\underline{\C}^{n+2}))$ holomorphic.
Namely,
$$A(\lambda)=\frac{\alpha-\lambda}{\alpha\lambda^{2}-\alpha^{2}\lambda}\,\mathcal{N}^{1,0}+\mathcal{N}^{0,1}+\frac{\alpha^{2}-\lambda^{2}}{\lambda^{3}\alpha^{2}-\lambda^{2}\alpha^{3}}\,q^{1,0}+(\lambda+\alpha)\,q^{0,1},$$
for all $\lambda$. It follows that
$$d^{\lambda,q}_{p_{_{\alpha,L}}}=D^{\alpha}_{q}+\frac{\alpha-\lambda}{\alpha+\lambda}\,\beta^{\alpha}_{q}+(\lambda-\alpha)\,p_{\alpha,L}(\lambda)\,A(\lambda)\,p_{\alpha,L}(\lambda)^{-1},$$
for $\lambda\in\C\backslash\{-\alpha,0,\alpha\}$. For simplicity,
set
$\Upsilon(\lambda):=(\lambda-\alpha)\,p_{\alpha,L}(\lambda)\,A(\lambda)\,p_{\alpha,L}(\lambda)^{-1}$.
The skew-symmetry of $A(\lambda)$ makes clear that
$A(\lambda)L\subset L^{\perp}$, as well as $A(\lambda)\rho
L\subset(\rho L)^{\perp}$ and, consequently, that $\pi_{\rho
L}\,A(\lambda)\,\pi_{L}=0=\pi_{L}\,A(\lambda)\,\pi_{\rho L}$. On the
other hand, it is clear that
$$\pi_{\rho L
}\,p_{\alpha,L}(\lambda)\,A(\lambda)\,p_{\alpha,L}(\lambda)^{-1}\pi_{L}=\frac{\alpha+\lambda}{\alpha-\lambda}\,\pi_{\rho
L}A(\lambda)\,\frac{\alpha+\lambda}{\alpha-\lambda}\,\pi_{L}=\frac{(\alpha+\lambda)^{2}}{(\alpha-\lambda)^{2}}\,\pi_{\rho
L}A(\lambda)\pi_{L}$$ and, similarly,
$$\pi_{L}\,p_{\alpha,L}(\lambda)\,A(\lambda)\,p_{\alpha,L}(\lambda)^{-1}\pi_{\rho L}=\frac{(\alpha-\lambda)^{2}}{(\alpha+\lambda)^{2}}\,\pi_{L}\,A(\lambda)\pi_{\rho L}.$$ Hence
$\pi_{\rho
L}\,\Upsilon(\lambda)\,\pi_{L}=0=\pi_{L}\Upsilon(\lambda)\,\pi_{\rho
L}$. It follows that
\begin{eqnarray*}
\Upsilon(\lambda)&=&
(\lambda-\alpha)\,(\pi_{L}\,A(\lambda)\,\pi_{L}+\pi_{\rho
L}\,A(\lambda)\,\pi_{\rho
L}+\pi_{\perp}\,A(\lambda)\,\pi_{\perp})\\&&\mbox{}-(\alpha+\lambda)\,(\pi_{\perp}\,A(\lambda)\,\pi_{L}+\pi_{\rho
L}\,A(\lambda)\,\pi_{\perp})\\ & &
\mbox{}-\frac{(\alpha-\lambda)^{2}}{\alpha+\lambda}\,(\pi_{L}\,A(\lambda)\,\pi_{\perp}+\pi_{\perp}\,A(\lambda)\pi_{\rho
L}).
\end{eqnarray*} Hence, by setting
$$d^{\alpha,q}_{p_{_{\alpha,L}}}:=D^{\alpha}_{q}-2\alpha\,(\pi_{\perp}\,A(\alpha)\,\pi_{L}+\pi_{\rho
L}\,A(\alpha)\,\pi_{\perp}),$$ we extend holomorphically
$\lambda\mapsto d^{\lambda,q}_{p_{_{\alpha,L}}}$ to
$\lambda\in\C\backslash\{-\alpha,0\}$ through what, by continuity,
we verify to be a metric connection on $\underline{\C}^{n+2}$.

The existence of a holomorphic extension to $\C\backslash\{0\}$,
through a metric connection on $\underline{\C}^{n+2}$, can be proved
analogously, having in consideration the
$d^{-\alpha,q}_{V}$-parallelness of $\rho L$.
\end{proof}

The same argument establishes the existence, in the case $L$ is
$d^{\alpha,q}_{V}$-parallel, of a holomorphic extension of
$$\C\backslash\{-\alpha,0,\alpha\}\ni\lambda\mapsto d_{q_{_{\alpha,L}}}^{\lambda,q}:=q_{\alpha,L}(\lambda)\circ
d^{\lambda,q}_{V}\circ q_{\alpha,L}(\lambda)^{-1}$$to
$\C\backslash\{0\}$ through metric connections on
$\underline{\C}^{n+2}$. This argument uses nothing about the precise
form of the connection $d^{\lambda,q}_{V}$, only the holomorphicity
of $\lambda\mapsto d^{\lambda,q}_{V}$ in $\C\backslash\{0\}$ through
metric connections on $\underline{\C}^{n+2}$, the
$d^{\alpha,q}_{V}$-parallelness of $L$ and the consequent
$d^{-\alpha,q}_{V}$-parallelness of $\rho L$. We can iterate the
procedure, in a two-step process, starting with the connections
$d^{\lambda,q}_{p_{_{\alpha,L}}}$, defining a family of connections
of the form
$$q_{\alpha',L'}(\lambda)p_{\alpha,L}(\lambda)\circ
d^{\lambda,q}_{V}\circ
p_{\alpha,L}(\lambda)^{-1}q_{\alpha',L'}(\lambda)^{-1};$$ or,
equally, starting with the connections
$d^{\lambda,q}_{q_{_{\alpha,L}}}$, defining, in that case, a family
of connections of the form
$$p_{\alpha',L'}(\lambda) q_{\alpha,L}(\lambda)\circ d^{\lambda,q}_{V}\circ
q_{\alpha,L}(\lambda)^{-1}p_{\alpha',L'}(\lambda)^{-1};$$ for
suitable parameters $\alpha,\alpha',L,L'$, as follows.

Choose $L^{\alpha}$ a $d^{\alpha,q}_{V}$-parallel null line
subbundle of $\underline{\C}^{n+2}$ with
\begin{equation}\label{eq:Lbetabom}
\rho L^{\alpha}\cap (L^{\alpha})^{\perp}=\{0\},
\end{equation}
locally. Such $L^{\alpha}$ can be obtained by choosing a null line
$\langle l^{\alpha}_{p}\rangle\subset\C^{n+2}$, for some
$l^{\alpha}_{p}\in\C^{n+2}$ not orthogonal to $\rho_{_{V_{p}}}
l^{\alpha}_{p}$, for some $p\in M$, and extending it to a
$d^{\alpha,q}_{V}$-parallel null line subbundle $L^{\alpha}$ of
$\underline{\C}^{n+2}$ by $d^{\alpha,q}_{V}$-parallel transport of
$l^{\alpha}_{p}$. The non-orthogonality of $L^{\alpha}$ and $\rho
L^{\alpha}$ at $p$ is, equivalently, satisfied in some non-empty
open set, which we restrict to.

Condition \eqref{eq:Lbetabom} allows us to refer to
$q_{\alpha,L^{\alpha}}$. Now choose $\beta\neq \pm\alpha$ in
$\C\backslash\{-1,0,1\}$ and $L^{\beta}$ a
$d^{\beta,q}_{V}$-parallel null line subbundle of
$\underline{\C}^{n+2}$ and note that the null line bundle
$$\tilde{L}^{\beta}_{\alpha}:=q_{\alpha,L^{\alpha}}(\beta)L^{\beta}$$
is $d^{\beta,q}_{q_{_{\alpha,L^{\alpha}}}}$-parallel and,
consequently, $\rho \tilde{L}^{\beta}_{\alpha}$ is parallel with
respect to the connection
\begin{eqnarray*}
d^{-\beta,q}_{q_{_{\alpha,L^{\alpha}}}}&=&q_{\alpha,L^{\alpha}}(-\beta)\circ
d^{-\beta,q}_{V}\circ\rho\,
q_{\alpha,L^{\alpha}}(\beta)^{-1}\rho\\&=&q_{\alpha,L^{\alpha}}(-\beta)\rho\circ
d^{\beta,q}_{V}\circ q_{\alpha,L^{\alpha}}(\beta)^{-1}\rho.
\end{eqnarray*}
Choose $L^{\beta}$ satisfying, furthermore,
\begin{equation}\label{eq:tilderhoLnotorth}
\rho \tilde{L}^{\beta}_{\alpha}\cap
(\tilde{L}^{\beta}_{\alpha})^{\perp}=\{0\},
\end{equation}
locally. To see that such a choice is possible, first choose a point
$p\in M$ at which $\rho L^{\alpha}$ is not orthogonal to
$L^{\alpha}$. As $L^{\alpha}$ is an eigenspace of
$q_{\alpha,L^{\alpha}}(\beta)$,
$q_{\alpha,L^{\alpha}}(\beta)L^{\alpha}=L^{\alpha}$ and, therefore,
at $p$,
$$\rho q_{\alpha,L^{\alpha}}(\beta)L^{\alpha}\cap
(q_{\alpha,L^{\alpha}}(\beta)L^{\alpha})^{\perp}=\{0\}.$$ Choose
$l^{\alpha}_{p}\in L^{\alpha}_{p}$ non-zero. A
$d^{\beta,q}_{V}$-parallel null line bundle
$L^{\beta}\subset\underline{\C}^{n+2}$, satisfying equation
\eqref{eq:tilderhoLnotorth}, locally, can be obtained by
$d^{\beta,q}_{V}$-parallel transport of $l^{\alpha}_{p}$.

Condition \eqref{eq:tilderhoLnotorth} allows us to refer to
$p_{\beta,\tilde{L}^{\beta}_{\alpha}}$. Set then
$$r^{(\beta,\alpha)}_{L^{\alpha},L^{\beta}}:=p_{\beta,\tilde{L}^{\beta}_{\alpha}}\,q_{\alpha,L^{\alpha}},$$
defining, for each
$\lambda\in\mathbb{P}^{1}\backslash\{\pm\alpha,\pm\beta\}$, an
orthogonal transformation
$r^{(\beta,\alpha)}_{L^{\alpha},L^{\beta}}(\lambda)$ of
$\underline{\C}^{n+2}$. The $d^{\alpha,q}_{V}$-parallelness of
$L^{\alpha}$ ensures that $\lambda\mapsto
d^{\lambda,q}_{q_{_{\alpha,L^{\alpha}}}}$ admits a holomorphic
extension to $\lambda\in\C\backslash\{0\}$ through metric
connections on $\underline{\C}^{n+2}$ and, consequently, the
$d^{\beta,q}_{q_{_{\alpha,L^{\alpha}}}}$-parallelness of
$\tilde{L}^{\beta}_{\alpha}$ (together with the consequent
$d^{-\beta,q}_{q_{_{\alpha,L^{\alpha}}}}$-parallelness of $\rho
\tilde{L}^{\beta}_{\alpha}$) ensures that so does
$$\lambda\mapsto r^{(\beta,\alpha)}_{L^{\alpha},L^{\beta}}(\lambda)\circ
d^{\lambda,q}_{V}\circ
r^{(\beta,\alpha)}_{L^{\alpha},L^{\beta}}(\lambda)^{-1}=p_{\beta,\tilde{L}^{\beta}_{\alpha}}(\lambda)\,q_{\alpha,L^{\alpha}}(\lambda)\circ
d^{\lambda,q}_{V}\circ q_{\alpha,L^{\alpha}}(\lambda)^{-1}\,
p_{\beta,\tilde{L}^{\beta}_{\alpha}}(\lambda)^{-1}.$$ We conclude
that $r^{(\beta,\alpha)}_{L^{\alpha},L^{\beta}}$ satisfies the
hypothesis of the dressing action, defining a transformation of
constrained harmonic bundles and constrained Willmore surfaces. Our
next step is to investigate how these transformations relate to the
ones defined by $r^{(\alpha,\beta)}_{L^{\alpha},L^{\beta}}$,  in the
case they are both defined.

Suppose, furthermore, that, locally, $L^{\beta}$ is never-orthogonal
to $\rho L^{\beta}$,
\begin{equation}\label{eq:rhoLbetacap}
\rho L^{\beta}\cap (L^{\beta})^{\perp}=\{0\}.
\end{equation}
This is certainly the case for $L^{\beta}$ obtained by
$d^{\beta,q}_{V}$-parallel transport of $l^{\alpha}_{p}$, for
$l^{\alpha}_{p}\in L^{\alpha}_{p}$ non-zero and $p$ a point in $M$
at which $\rho L^{\alpha}$ is not orthogonal to $L^{\alpha}$. Note
that
$$\tilde{L}^{\alpha}_{\beta}=q_{\beta,L^{\beta}}(\alpha)L^{\alpha}$$ is
a $d^{\alpha,q}_{q_{_{\beta,L^{\beta}}}}$-parallel bundle. Observe
that, locally,
\begin{equation}\label{eq:2sxcmbvmkl?'098765escvbnm,87651qser678b35689oknbfrt}
\rho \tilde{L}^{\alpha}_{\beta}\cap
(\tilde{L}^{\alpha}_{\beta})^{\perp}=\{0\}.
\end{equation}
Indeed, given $l^{\alpha}\in\Gamma(L^{\alpha})$ never-zero, and
according to \eqref{eq:pinvs} and \eqref{eq:rhopqinv}, we have
\begin{eqnarray*}
(\rho \, q_{\beta,L^{\beta}}(\alpha)l^{\alpha},
q_{\beta,L^{\beta}}(\alpha)l^{\alpha})&=&
(q_{\beta,L^{\alpha}}(\alpha)^{-1}\rho\,l^{\alpha},
q_{\beta,L^{\beta}}(\alpha)l^{\alpha})\\&=&(\rho\,l^{\alpha},
q_{\beta,L^{\alpha}}(\alpha)^{2}l^{\alpha})\\&=&\frac{(\alpha-\beta)^{2}}{(\alpha+\beta)^{2}}\,(\rho\,l^{\alpha},
l^{\alpha});
\end{eqnarray*}
so that, given $p$ a point in $M$ at which $\rho L^{\alpha}$ is not
orthogonal to $L^{\alpha}$,
$$(\rho_{V_{p}}\, q_{\beta,L^{\beta}_{p}}(\alpha)l^{\alpha}_{p},
q_{\beta,L^{\beta}_{p}}(\alpha)l^{\alpha}_{p})\neq 0,$$ condition
\eqref{eq:2sxcmbvmkl?'098765escvbnm,87651qser678b35689oknbfrt} is
satisfied at $p$, or, equivalently, in some open neighbourhood of
$p$.

Set then
$$r^{(\alpha,\beta)}_{L^{\alpha},L^{\beta}}:=p_{\alpha,\tilde{L}^{\alpha}_{\beta}}\,q_{\beta,L^{\beta}},$$
defining, for each
$\lambda\in\mathbb{P}^{1}\backslash\{\pm\alpha,\pm\beta\}$, an
orthogonal transformation
$r^{(\alpha,\beta)}_{L^{\alpha},L^{\beta}}(\lambda)$ of
$\underline{\C}^{n+2}$. The $d^{\beta,q}_{V}$-parallelness of
$L^{\beta}$ ensures that $\lambda\mapsto
d^{\lambda,q}_{q_{_{\beta,L^{\beta}}}}$ admits a holomorphic
extension to $\lambda\in\C\backslash\{0\}$ through metric
connections on $\underline{\C}^{n+2}$ and, consequently, the
$d^{\alpha,q}_{q_{_{\beta,L^{\beta}}}}$-parallelness of
$\tilde{L}^{\alpha}_{\beta}$ (together with the consequent
$d^{-\alpha,q}_{q_{_{\beta,L^{\beta}}}}$-parallelness of $\rho
\tilde{L}^{\alpha}_{\beta}$) ensures that so does
$$\lambda\mapsto r^{(\alpha,\beta)}_{L^{\alpha},L^{\beta}}(\lambda)\circ
d^{\lambda,q}_{V}\circ
r^{(\alpha,\beta)}_{L^{\alpha},L^{\beta}}(\lambda)^{-1}=p_{\alpha,\tilde{L}^{\alpha}_{\beta}}(\lambda)\,q_{\beta,L^{\beta}}(\lambda)\circ
d^{\lambda,q}_{V}\circ q_{\beta,L^{\beta}}(\lambda)^{-1}
p_{\alpha,\tilde{L}^{\alpha}_{\beta}}(\lambda)^{-1}.$$ We conclude
that $r^{(\alpha,\beta)}_{L^{\alpha},L^{\beta}}$ satisfies the
hypothesis of the dressing action, defining a transformation of
constrained harmonic bundles and constrained Willmore surfaces. We
shall verify that these transformations coincide with the ones
defined by $r^{(\beta,\alpha)}_{L^{\alpha},L^{\beta}}$. On the way,
we verify that, starting with the connections
$d^{\lambda,q}_{p_{_{\alpha,L^{\alpha}}}}$, rather than with the
connections $d^{\lambda,q}_{q_{_{\beta,L^{\beta}}}}$, leads to the
same transformations of constrained harmonic bundles and of
constrained Willmore surfaces.

Set
$$\hat{L}^{\beta}_{\alpha}:=p_{\alpha,L^{\alpha}}(\beta)L^{\beta}$$
and observe that, locally,
$$\rho\hat{L}^{\beta}_{\alpha}\cap
(\hat{L}^{\beta}_{\alpha})^{\perp}=\{0\}.$$ Indeed, given
$l^{\beta}\in\Gamma(L^{\beta})$,
$$(\rho p_{\alpha,L^{\alpha}}(\beta)l^{\beta},
p_{\alpha,L^{\alpha}}(\beta)l^{\beta})=\frac{(\alpha-\beta)^{2}}{(\alpha+\beta)^{2}}\,(\rho\,l^{\beta},
l^{\beta}),$$and the conclusion follows, in view of
\eqref{eq:rhoLbetacap}.

Set then
$$\hat{r}^{(\alpha,\beta)}_{L^{\alpha},L^{\beta}}:=q_{\beta,\hat{L}^{\beta}_{\alpha}}\,p_{\alpha,L^{\alpha}},$$
defining, for each
$\lambda\in\mathbb{P}^{1}\backslash\{\pm\alpha,\pm\beta\}$, an
orthogonal transformation
$\hat{r}^{(\alpha,\beta)}_{L^{\alpha},L^{\beta}}(\lambda)$ of
$\underline{\C}^{n+2}$. The $d^{\alpha,q}_{V}$-parallelness of
$L^{\alpha}$ ensures that $\lambda\mapsto
d^{\lambda,q}_{p_{\alpha,L^{\alpha}}}$ admits a holomorphic
extension to $\lambda\in\C\backslash\{0\}$ through metric
connections on $\underline{\C}^{n+2}$ and, consequently, the
$d^{\beta,q}_{p_{_{\alpha,L^{\alpha}}}}$-parallelness of
$\hat{L}^{\beta}_{\alpha}$ (and the consequent
$d^{-\beta,q}_{p_{_{\alpha,L^{\alpha}}}}$-parallelness of $\rho
\hat{L}^{\beta}_{\alpha}$) ensures that so does $$\lambda\mapsto
\hat{r}^{(\alpha,\beta)}_{L^{\alpha},L^{\beta}}(\lambda)\circ
d^{\lambda,q}_{V}\circ
\hat{r}^{(\alpha,\beta)}_{L^{\alpha},L^{\beta}}(\lambda)^{-1}=q_{\beta,\hat{L}^{\beta}_{\alpha}}(\lambda)p_{_{\alpha,L^{\alpha}}}(\lambda)\circ
d^{\lambda,q}_{V}\circ p_{_{\alpha,L^{\alpha}}}(\lambda)^{-1}
q_{\beta,\hat{L}^{\beta}_{\alpha}}(\lambda)^{-1}.$$ We conclude that
$\hat{r}^{(\alpha,\beta)}_{L^{\alpha},L^{\beta}}$ satisfies the
hypothesis of the dressing action, defining a transformation of
constrained harmonic bundles and constrained Willmore surfaces.

For simplicity, set
$r^{*}:=r^{(\alpha,\beta)}_{L^{\alpha},L^{\beta}}$ and
$\hat{r}^{*}:=\hat{r}^{(\alpha,\beta)}_{L^{\alpha},L^{\beta}}$.

\begin{prop}\label{rstarvshatrstar} $r^{*}$ and $\hat{r}\,^{*}$ are related by
\begin{equation}\label{eq:rel}
r^{*}=K\,\hat{r}\,^{*},
\end{equation}
for $K:=q_{\beta,L^{\beta}}(0)\,
q_{\beta,\hat{L}^{\beta}_{\alpha}}(0)$.
\end{prop}

The proof of the proposition will be based on the following lemma,
cf. F. Burstall \cite{IS}.

\begin{Lemma}\label{TrioDeHolomorfia}
Let
$\gamma(\lambda)=\lambda\,\pi_{L_{1}}+\pi_{L_{0}}+\lambda^{-1}\,\pi_{L_{-1}}$
and
$\hat{\gamma}(\lambda)=\lambda\,\pi_{\hat{L}_{1}}+\pi_{\hat{L}_{0}}+\lambda^{-1}\,\pi_{\hat{L}_{-1}}$
 be homomorphisms of $\C^{n+2}$ corresponding to decompositions $$\C^{n+2}=L_{1}\oplus L_{0}\oplus L_{-1}=\hat{L}_{1}\oplus \hat{L}_{0}\oplus
 \hat{L}_{-1}$$ with $L_{\pm 1}$ and $\hat{L}_{\pm 1}$ null lines and
 $L_{0}=(L_{1}\oplus L_{-1})^{\perp}$, $\hat{L}_{0}=(\hat{L}_{1}\oplus
 \hat{L}_{-1})^{\perp}$. Suppose $\mathrm{Ad}\,\gamma$ and $\mathrm{Ad}\,\hat{\gamma}$ have simple poles. Suppose as well that $\xi$ is a map into
 $O(\C^{n+2})$ holomorphic near $0$ such that $$L_{1}=\xi
 (0)\hat{L}_{1}.$$ Then $\gamma\xi\hat{\gamma}^{-1}$ is holomorphic and invertible at $0$.
\end{Lemma}

Now we proceed to the proof of Proposition \ref{rstarvshatrstar}.

\begin{proof}
For simplicity, throughout this proof, we adopt $p_{\mu,L}^{-1}$ and
$q_{\mu,L}^{-1}$ to denote $\lambda\mapsto p_{\mu,L}(\lambda)^{-1}$
and, respectively, $\lambda\mapsto q_{\mu,L}(\lambda)^{-1}$, in the
case $p_{\mu,L}$ and, respectively, $q_{\mu,L}$ are defined.

As
$L^{\alpha}=q_{\beta,L^{\beta}}(\alpha)^{-1}\tilde{L}^{\alpha}_{\beta}$,
after an appropriate change of variable, we conclude, by Lemma \ref
{TrioDeHolomorfia}, that $p_{\alpha,L^{\alpha}}\,
q_{\beta,L^{\beta}}^{-1}\,
p_{\alpha,\tilde{L}^{\alpha}_{\beta}}^{-1}$ admits a holomorphic and
invertible extension to $\mathbb{P}^{1}\backslash\{\pm
\beta,-\alpha\}$. On the other hand, in view of \eqref{eq:rhopqinv},
the holomorphicity and invertibility of $p_{\alpha,L^{\alpha}}\,
q_{\beta,L^{\beta}}^{-1}\,
p_{\alpha,\tilde{L}^{\alpha}_{\beta}}^{-1}$ at the points $\alpha$
and $-\alpha$ are equivalent. Thus $p_{\alpha,L^{\alpha}}\,
q_{\beta,L^{\beta}}^{-1}\,
p_{\alpha,\tilde{L}^{\alpha}_{\beta}}^{-1}$ admits an holomorphic
and invertible extension to $\mathbb{P}^{1}\backslash\{\pm \beta\}$,
and so does, therefore, $(p_{\alpha,L^{\alpha}}\,
q_{\beta,L^{\beta}}^{-1}\,
p_{\alpha,\tilde{L}^{\alpha}_{\beta}}^{-1})^{-1}\,
q_{\beta,L^{\beta}}^{-1}$. A similar argument shows that
$p_{\alpha,\tilde{L}^{\alpha}_{\beta}}\, (q_{\beta,L^{\beta}}\,
p_{\alpha,L^{\alpha}}^{-1}\,
q_{\beta,\hat{L}^{\beta}_{\alpha}}^{-1})$ admits an holomorphic
extension to $\mathbb{P}^{1}\backslash\{\pm\alpha\}$. But
$$p_{\alpha,\tilde{L}^{\alpha}_{\beta}}\,q_{\beta,L^{\beta}}\,
p_{\alpha,L^{\alpha}}^{-1}\,
q_{\beta,\hat{L}^{\beta}_{\alpha}}^{-1}=(p_{\alpha,L^{\alpha}}\,
q_{\beta,L^{\beta}}^{-1}\,
p_{\alpha,\tilde{L}^{\alpha}_{\beta}}^{-1})^{-1}\,
q_{\beta,L^{\beta}}^{-1}.$$We conclude that
$p_{\alpha,\tilde{L}^{\alpha}_{\beta}}\,q_{\beta,L^{\beta}}\,
p_{\alpha,L^{\alpha}}^{-1}\,
q_{\beta,\hat{L}^{\beta}_{\alpha}}^{-1}$ extends holomorphically to
$\mathbb{P}^{1}$ and is, therefore, constant. Evaluating at
$\lambda=0$ gives
$p_{\alpha,\tilde{L}^{\alpha}_{\beta}}\,q_{\beta,L^{\beta}}\,
p_{\alpha,L^{\alpha}}^{-1}\,
q_{\beta,\hat{L}^{\beta}_{\alpha}}^{-1}=q_{\beta,L^{\beta}}(0)\,q_{\beta,\hat{L}^{\beta}_{\alpha}}(0)$,
completing the proof.
\end{proof}

According to \eqref{eq:rhopqinv}, $K$ commutes with $\rho$,
\begin{equation}\label{eq:rhoKrhoisK}
\rho\,K\,\rho=K.
\end{equation}
This ensures, in particular, that $K$ preserves $V$: given
$v\in\Gamma(V)$, $\rho Kv=\rho K\rho v=Kv$. Equivalently,
\begin{equation}\label{eq:KV=V}
K\,V=V.
\end{equation}
Together with equation \eqref {eq:rel}, equation \eqref{eq:KV=V}
shows that, for each $\lambda$, $$\hat{r}\,^{*}(\lambda)^{-1}\,V=
r^{*}(\lambda)^{-1}\,K\,V=r^{*}(\lambda)^{-1}\,V.$$ In particular,
\begin{equation}\label{eq:eqbase}
\hat{r}\,^{*}(1)^{-1}\,V=r^{*}(1)^{-1}\,V.
\end{equation}
As for a constrained Willmore surface $(\Delta^{1,0},\Delta^{0,1})$
admitting $V$ as a $(q,d)$-central sphere congruence, and, yet
again, by equation \eqref{eq:rel}, we have
$$r^{*}(1)^{-1}r^{*}(\infty)\Delta^{1,0}=
\hat{r}\,^{*}(1)^{-1}\,K^{-1}\,(K\,\hat{r}\,^{*})(\infty)\,\Delta^{1,0}
=\hat{r}\,^{*}(1)^{-1}\,\hat{r}\,^{*}(\infty)\,\Delta^{1,0},$$ as
well as
$$r^{*}(1)^{-1}r^{*}(0)\Delta^{0,1}=
\hat{r}\,^{*}(1)^{-1}\,\hat{r}\,^{*}(0)\,\Delta^{0,1}.$$ We conclude
that, despite not coinciding, $r^{*}$ and $\hat{r}\,^{*}$ produce
the same transformations of bundles and complexified surfaces.

Now set
$$\tilde{K}=p_{\alpha,L^{\alpha}}(\infty)\,p_{\beta,L^{\beta}}(\infty).$$
Note that
$$p_{\alpha',L'}(\lambda)\,p_{\alpha',L'}(\infty)=q_{\alpha',L'}(\lambda)=p_{\alpha',L'}(\infty)\,p_{\alpha',L'}(\lambda),$$
or, equivalently,
$$q_{\alpha',L'}(\lambda)\,p_{\alpha',L'}(\infty)=p_{\alpha',L'}(\lambda)=p_{\alpha',L'}(\infty)\,q_{\alpha',L'}(\lambda),$$
for all $\alpha',L'$ and $\lambda\in\mathbb{P}^{1}\backslash\{\pm
\alpha'\}$. Together with equation \eqref{eq:rel}, and having in
consideration that $p_{\alpha',L'}(\infty)^{2}=I$, this establishes
\begin{eqnarray*}
\tilde{K}\,r^{*}&=&p_{\alpha,L^{\alpha}}(\infty)\,p_{\beta,L^{\beta}}(\infty)^{2}\,
p_{\beta,\hat{L}^{\beta}_{\alpha}}(\infty)\,q_{\beta,\hat{L}^{\beta}_{\alpha}}p_{\alpha,L^{\alpha}}\\&=&p_{\alpha,L^{\alpha}}(\infty)\,
p_{\beta,\hat{L}^{\beta}_{\alpha}}\,p_{\alpha,L^{\alpha}}\\&=&
p_{\beta,\hat{L}^{\beta}_{\alpha}}\,p_{\alpha,L^{\alpha}}(\infty)p_{\alpha,L^{\alpha}}
\end{eqnarray*}
and, ultimately,
\begin{equation}\label{eq:rellinha}
r^{(\beta,\alpha)}_{L^{\alpha},L^{\beta}}=\tilde{K}\,r^{(\alpha,\beta)}_{L^{\alpha},L^{\beta}}.
\end{equation}
Ultimately, we conclude that, despite not coinciding,
$r^{(\alpha,\beta)}_{L^{\alpha},L^{\beta}}$ and
$r^{(\beta,\alpha)}_{L^{\alpha},L^{\beta}}$ produce the same
transformations of bundles and complexified surfaces.

Set
$$q^{*}:=\mathrm{Ad}_{r^{*}(1)^{-1}}(\mathrm{Ad}_{r^{*}(0)}q^{1,0}+
\mathrm{Ad}_{r^{*}(\infty)}q^{0,1}).$$

\begin{defn}\label{BTdefn}
Suppose $V$ is a $q$-constrained harmonic bundle. In that case, the
$q^{*}$-constrained harmonic bundle
$$V^{*}:=r^{*}(1)^{-1}\,V$$ is said to be the \emph{B\"{a}cklund
transform of $V$ of parameters $\alpha,\beta,L^{\alpha},L^{\beta}$}.
In the case $V$ is a $(q,d)$-central sphere congruence for some
constrained Willmore surface $(\Delta^{1,0},\Delta^{0,1})$, the
$q^{*}$-constrained Willmore surface
$$(\Delta^{1,0}, \,\Delta^{0,1})^{*}:=((\Delta^{*})^{1,0}, \,(\Delta^{*})^{0,1}):=(r^{*}(1)^{-1}\,r^{*}(\infty)\Delta^{1,0},r^{*}(1)^{-1}\,r^{*}(0)\Delta^{0,1})$$
is said to be the \emph{B\"{a}cklund transform of
$(\Delta^{1,0},\Delta^{0,1})$ of parameters
$\alpha,\beta,L^{\alpha},L^{\beta}$}.
\end{defn}

\begin{rem}
Set $\tilde{r}^{*}:=r^{(\beta,\alpha)}_{L^{\alpha},L^{\beta}}$.
According to equations \eqref{eq:rel} and \eqref{eq:rellinha},
$$\hat{q}^{*}:=\mathrm{Ad}_{\hat{r}^{*}(1)^{-1}}(\mathrm{Ad}_{\hat{r}^{*}(0)}q^{1,0}+
\mathrm{Ad}_{\hat{r}^{*}(\infty)}q^{0,1})=q^{*},$$ as well as
$$\tilde{q}^{*}:=\mathrm{Ad}_{\tilde{r}^{*}(1)^{-1}}(\mathrm{Ad}_{\tilde{r}^{*}(0)}q^{1,0}+
\mathrm{Ad}_{\tilde{r}^{*}(\infty)}q^{0,1})=q^{*}.$$
\end{rem}

Equation \eqref{eq:rel} establishes a \textit{Bianchi
permutability}\footnote{The terminology is motivated by the
permutability of this kind established by Bianchi with respect to
the original B\"{a}cklund transformations.} of type $p$ and type $q$
transformations of constrained harmonic bundles (into constrained
harmonic bundles), by means of the commutativity of the diagram in
Figure \ref{fig:im1}, below.
\begin{center}
\begin{figure}[H]
\includegraphics{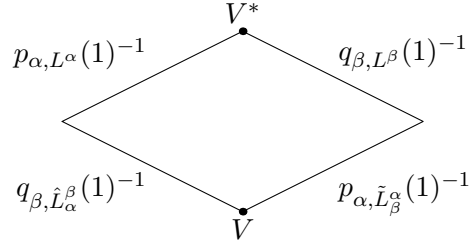}
\caption{A Bianchi permutability of type $p$ and type $q$
transformations.}\label{fig:im1}
\end{figure}
\end{center}
Equation \eqref{eq:rel} will play a crucial role when investigating
the preservation of reality conditions by B\"{a}cklund
transformations, in the next section.

\subsection{Real B\"{a}cklund transformation}

As we verify in this section, for special choices of parameters, the
B\"{a}cklund transformation preserves reality conditions.\newline

Suppose $V$ is a real $q$-constrained harmonic subbundle of
$\underline{\C}^{n+2}$ and let us focus on the particular case of a
B\"{a}cklund transformation of parameters
$\alpha,\beta,L^{\alpha},L^{\beta}$ for
$$\alpha\in\C\backslash (S^{1}\cup\{0\}),\,\,\,\,\,\,\,\,\,\,
\beta=\overline{\alpha}\,^{-1},\,\,\,\,\,\,\,\,\,\,L^{\beta}=\overline{L^{\alpha}}$$
and $L^{\alpha}$ a $d^{\alpha,q}_{V}$-parallel null line subbundle
of $\underline{\C}^{n+2}$ such that, locally,
$$\rho L^{\alpha}\cap (L^{\alpha})^{\perp}=\{0\}$$ and
\begin{equation}\label{eq:rhonotorthogonalnht236whum}
\rho\,q_{\overline{\alpha}\,^{-1},\overline{L^{\alpha}}}\,(\alpha)L^{\alpha}\cap
(q_{\overline{\alpha}\,^{-1},\overline{L^{\alpha}}}\,(\alpha)L^{\alpha})^{\perp}=\{0\}.
\end{equation}
We start by verifying that this is, indeed, a possible choice of
B\"{a}cklund transformation parameters in the case $V$ is real.

First note that the reality of $V$ establishes that of $\rho$,
$\overline{\rho}=\rho$, in which case the local non-orthogonality of
$L^{\alpha}$ and $\rho L^{\alpha}$ is equivalent to that of
$\overline{L^{\alpha}}$ and $\rho\overline{L^{\alpha}}$,
$$\rho\overline{L^{\alpha}}\cap (\overline{L^{\alpha}})^{\perp}=\{0\}.$$
On the other hand, as $V$ and $V^{\perp}$ are real, so are then
$\pi_{V}$ and $\pi_{V^{\perp}}$, as well as, therefore,
$\mathcal{D}_{V}$  and $\mathcal{N}_{V}$, so that, by the reality of
$q$,
$$d^{\overline{\alpha}\,^{-1},q}_{V}=\overline{d^{\alpha,q}_{V}}.$$
Hence the $d^{\alpha,q}_{V}$-parallelness of $L^{\alpha}$ is
equivalent to the $d^{\overline{\alpha}\,^{-1},q}_{V}$-parallelness
of $\overline{L^{\alpha}}$. The reality of $\rho$ establishes, on
the other hand,
\begin{equation}\label{eq:conj}
\overline{p_{V,\alpha',L'}(\lambda)}=p_{\overline{V},\overline{\alpha'},\overline{L'}}\,(\overline{\lambda}),\,\,\,\,\,\,\,\,\,\,\overline{q_{V,\alpha',L'}(\lambda)}=q_{\overline{V},\overline{\alpha'},\overline{L'}}\,(\overline{\lambda}),
\end{equation}
together with
$$\overline{p_{V,\alpha',L'}(\infty)}=p_{\overline{V},\overline{\alpha'},\overline{L'}}(\infty),$$
for all $\alpha',L'$ and $\lambda\in\C\backslash\{\pm\alpha'\}$.
According to  equation \eqref{eq:conjap}, it follows that, given
$l^{\alpha}\in\Gamma(L^{\alpha})$,
\begin{eqnarray*}
\overline{(\rho
q_{\overline{\alpha}\,^{-1},\overline{L^{\alpha}}}\,(\alpha)l^{\alpha},q_{\overline{\alpha}\,^{-1},\overline{L^{\alpha}}}\,(\alpha)l^{\alpha})}&=&
(\rho
q_{\alpha^{-1},L^{\alpha}}\,(\overline{\alpha})\overline{l^{\alpha}},q_{\alpha^{-1},L^{\alpha}}\,(\overline{\alpha})\overline{l^{\alpha}})\\&=&
(\rho
p_{\alpha,L^{\alpha}}(\overline{\alpha}\,^{-1})\overline{l^{\alpha}},p_{\alpha,L^{\alpha}}(\overline{\alpha}\,^{-1})\overline{l^{\alpha}}),
\end{eqnarray*}
showing that condition \eqref{eq:rhonotorthogonalnht236whum} is
equivalent to
$$\rho p_{\alpha,L^{\alpha}}(\overline{\alpha}\,^{-1})\overline{L^{\alpha}}\cap(p_{\alpha,L^{\alpha}}(\overline{\alpha}\,^{-1})\overline{L^{\alpha}})^{\perp}=\{0\}.$$

Next we establish the existence of a choice of $L^{\alpha}$ in the
conditions above.
\begin{Lemma}
Let $v$ and $w$ be sections of $V$ and $V^{\perp}$, respectively,
with $(v,v)$ never-zero, $(v,\overline{v})=0$ and $(w,w)=-(v,v)$.
Define a null section of $\underline{\C}^{n+2}$ by
$l^{\alpha}:=v+w$. Let $L^{\alpha}\subset\underline{\C}^{n+2}$ be a
$d^{\alpha,q}_{V}$-parallel null line bundle defined naturally by
$d^{\alpha,q}_{V}$-parallel transport of $l^{\alpha}_{p}$, for some
point $p\in M$. Then there is some (non-empty) open set in which
$L^{\alpha}$ is never orthogonal to $\rho L^{\alpha}$ and satisfies
equation \eqref{eq:rhonotorthogonalnht236whum}.
\end{Lemma}
\begin{proof}
At the point $p$, $L^{\alpha}$ is spanned by $l^{\alpha}_{p}$. The
fact that, at the point $p$,
$$(\rho l^{\alpha},l^{\alpha})=(v,v)-(w,w)=2(v,v)\neq 0$$
establishes the non-orthogonality of $L^{\alpha}$ and $\rho
L^{\alpha}$ at this point and, therefore, locally. Consider
projections
$\pi_{\overline{L^{\alpha}}}:\underline{\C}^{n+2}\rightarrow\overline{L^{\alpha}}$,
$\pi_{\rho\overline{L^{\alpha}}}:\underline{\C}^{n+2}\rightarrow
\rho\overline{L^{\alpha}}$ and
$\pi_{\perp}:\underline{\C}^{n+2}\rightarrow(\overline{L^{\alpha}}\oplus\rho
\overline{L^{\alpha}})^{\perp}$ with respect to the decomposition
$$\underline{\C}^{n+2}= \overline{L^{\alpha}}\oplus
(\overline{L^{\alpha}}\oplus\rho
\overline{L^{\alpha}})^{\perp}\oplus\rho\overline{L^{\alpha}},$$
established by the subsequent local non-orthogonality of
$\overline{L^{\alpha}}$ and $\rho \overline{L^{\alpha}}$. For
simplicity, denote
$q_{\overline{\alpha}\,^{-1},\overline{L^{\alpha}}}$ by $q$. Set
$$A:=\frac{\alpha-\overline{\alpha}\,^{-1}}{\alpha+\overline{\alpha}\,^{-1}}=\frac{\mid\alpha\mid ^{2}-1}{\mid\alpha\mid
^{2}+1}\in\R.$$ Then
$$\rho q(\alpha)l^{\alpha}=q(\alpha)^{-1}\rho l^{\alpha}=
A^{-1}\pi_{\overline{L^{\alpha}}}\,\rho l^{\alpha}+\pi_{\perp}\,\rho
l^{\alpha}+A\,\pi_{\rho\overline{L^{\alpha}}}\rho l^{\alpha},$$ and,
therefore,
$$(\rho q(\alpha)l^{\alpha},q(\alpha)l^{\alpha})=
A^{2}(\pi_{\overline{L^{\alpha}}}\,l^{\alpha},\pi_{\rho\overline{L^{\alpha}}}\,\rho
l^{\alpha})+ (\pi_{\perp}l^{\alpha},\pi_{\perp}\rho l^{\alpha})+
A^{-2}(\pi_{\rho\overline{L^{\alpha}}}\,l^{\alpha},\pi_{\overline{L^{\alpha}}}\,\rho
l^{\alpha}).$$ At the point $p$,
$\overline{L^{\alpha}}=\langle\overline{l^{\alpha}}\rangle$ and the
orthogonality relations show then that, at $p$,
$$l^{\alpha}=\frac{(l^{\alpha},\rho\overline{l^{\alpha}})}{\overline{(\rho l^{\alpha},l^{\alpha})}}\,\overline{l^{\alpha}}+
\frac{(l^{\alpha},\overline{l^{\alpha}})}{\overline{(\rho
l^{\alpha},l^{\alpha})}}\,\rho\overline{l^{\alpha}}+
\pi_{\perp}l^{\alpha}.$$ Hence, at $p$,
$$(\pi_{\overline{L^{\alpha}}}\,l^{\alpha},\pi_{\rho\overline{L^{\alpha}}}\,\rho l^{\alpha})=\frac{(l^{\alpha},\rho\overline{l^{\alpha}})^{2}}{\overline{(\rho l^{\alpha},l^{\alpha})}}$$
and
$$(\pi_{\rho\overline{L^{\alpha}}}\,l^{\alpha},\pi_{\overline{L^{\alpha}}}\,\rho l^{\alpha})=\frac{(l^{\alpha},\overline{l^{\alpha}})^{2}}{\overline{(\rho l^{\alpha},l^{\alpha})}},$$
and, therefore,
\begin{eqnarray*}
(\pi_{\perp}l^{\alpha},\pi_{ \perp}\rho l^{\alpha})&=& (\rho
l^{\alpha},l^{\alpha})-\frac{(l^{\alpha},\rho\overline{l^{\alpha}})^{2}+(l^{\alpha},\overline{l^{\alpha}})^{2}}{\overline{(\rho
l^{\alpha},l^{\alpha})}}.
\end{eqnarray*}
It follows that $(\rho q(\alpha)l^{\alpha},q(\alpha)l^{\alpha})$
vanishes at $p$ if and only if, at this point,
$$\mid(\rho l^{\alpha},l^{\alpha})\mid^{2}+(A^{2}-1)(l^{\alpha},\rho\overline{l^{\alpha}})^{2}+(A^{-2}-1)(l^{\alpha},\overline{l^{\alpha}})^{2}=0,$$
or, equivalently,
$$\mid(\rho l^{\alpha},l^{\alpha})\mid^{2}+(A^{2}+A^{-2}-2)(l^{\alpha},\overline{l^{\alpha}})^{2}=0,$$
as, since $(v,\overline{v})=0$, we have, at $p$,
$$(l^{\alpha},\rho\overline{l^{\alpha}})^{2}=((v,\overline{v})-(w,\overline{w}))^{2}=(w,\overline{w})^{2}=((v,\overline{v})+(w,\overline{w}))^{2}=
(l^{\alpha},\overline{l^{\alpha}})^{2}.$$ As
$(l^{\alpha},\overline{l^{\alpha}})$ is real, and, therefore,
$(l^{\alpha},\overline{l^{\alpha}})^{2}\geq 0$, and
$A^{2}+A^{-2}-2=(A-A^{-1})^{2}>0$, we conclude that, at the point
$p$,
$$(\rho q(\alpha)l^{\alpha},q(\alpha)l^{\alpha})\neq
0.$$ The proof is complete by the fact that the non-orthogonality of
$q(\alpha)L^{\alpha}$ and $\rho q(\alpha)L^{\alpha}$ is an open
condition on the points in $M$.
\end{proof}

We refer to a B\"{a}cklund transformation corresponding to this
particular choice of parameters, in the case $V$ is real, as a
\textit{B\"{a}cklund} \textit{transformation} \textit{of}
\textit{parameters} \textit{$\alpha$}, \textit{$L^{\alpha}$}. For
this choice of parameters, where, in particular, $\beta$ and
$L^{\beta}$ are defined by $\alpha$ and $L^{\alpha}$, we denote
$\tilde{L}^{\alpha}_{\beta}$, $\hat{L}^{\beta}_{\alpha}$ and
$r^{(\alpha,\beta)}_{L^{\alpha},L^{\beta}}$ simply by
$\tilde{L}^{\alpha}$, $\hat{L}^{\overline{\alpha}\,^{-1}}$ and
$r^{\alpha}_{L^{\alpha}}$, respectively. In what follows in this
section, consider this particular choice of parameters. Let $r^{*}$
denote $r^{\alpha}_{L^{\alpha}}$. According to \eqref{eq:conjap} and
\eqref{eq:conj},
\begin{eqnarray*}
\overline{r^{*}(1)^{-1}}&=&\overline{q_{\overline{\alpha}\,^{-1},\overline{L^{\alpha}}}(1)}\,^{-1}\,\overline{p_{\alpha,\tilde{L}^{\alpha}}(1)}\,^{-1}\\&=&
p_{\alpha,L^{\alpha}}(1)\,^{-1}\,p_{\overline{\alpha},\overline{\tilde{L}^{\alpha}}}\,(1)\,^{-1}.
\end{eqnarray*}
On the other hand,
\begin{equation}\label{eq:tildevsbarL}
\overline{\tilde{L}^{\alpha}}=q_{\alpha^{-1},L^{\alpha}}(\overline{\alpha})\overline{L^{\alpha}}=p_{\alpha,L^{\alpha}}(\overline{\alpha}\,^{-1})\overline{L^{\alpha}}=\hat{L}^{\overline{\alpha}\,^{-1}}.
\end{equation}
and, therefore, by equation \eqref{eq:rel},
\begin{eqnarray*}
r^{*}(1)^{-1}&=&(Kq_{\overline{\alpha}\,^{-1},\overline{\tilde{L}^{\alpha}}}\,(1)\,p_{\alpha,L^{\alpha}}(1))^{-1}\\&=&p_{\alpha,L^{\alpha}}(1)^{-1}\,p_{\overline{\alpha},\overline{\tilde{L}^{\alpha}}}\,(1)^{-1}\,K^{-1}.
\end{eqnarray*}
Hence
\begin{equation}\label{eq:conjugado de rstar(1)inv}
\overline{r^{*}(1)^{-1}}=r^{*}(1)^{-1}K.
\end{equation}
Note that, given $\alpha',L'$, $p_{\alpha',L'}(\infty)$ does not
depend on $\alpha'$. Thus
\begin{eqnarray*}
\overline{r^{*}(0)}&=&\overline{p_{\alpha,\tilde{L}^{\alpha}}(0)q_{\overline{\alpha}\,^{-1},\overline{L^{\alpha}}}\,(0)}\\&=&
p_{\alpha^{-1},L^{\alpha}}(\infty)\\&=&p_{\alpha,L^{\alpha}}(\infty).
\end{eqnarray*}
On the other hand, by equation \eqref{eq:rel},
\begin{eqnarray*}
r^{*}(\infty)&=&Kq_{\overline{\alpha}\,^{-1},\hat{L}^{\overline{\alpha}\,^{-1}}}(\infty)p_{\alpha,L^{\alpha}}(\infty)\\&=&Kp_{\alpha,L^{\alpha}}(\infty).
\end{eqnarray*}
We conclude that
\begin{equation}\label{eq:conjugado de rstar(0)}
\overline{r^{*}(0)}=K^{-1}r^{*}(\infty).
\end{equation}

\begin{rem}
It is opportune to remark that, by \eqref{eq:tildevsbarL},
\begin{eqnarray*}
\overline{r^{*}(\lambda)}&=&\overline{K\,q_{\overline{\alpha}\,^{-1},\overline{\tilde{L}^{\alpha}}}\,(\lambda)\,p_{\alpha,L^{\alpha}}(\lambda)}=\\&=&
\overline{K}\,p_{\alpha,\tilde{L}^{\alpha}}(\overline{\lambda}\,^{-1})\,q_{\overline{\alpha}\,^{-1},\overline{L^{\alpha}}}\,(\overline{\lambda}\,^{-1})\\&=&
\overline{K}\,r^{*}(\overline{\lambda}\,^{-1})
\end{eqnarray*}
for all $\lambda$, and, in particular,
\begin{equation}\label{eq:koverlineinversedadada?}
\overline{r^{*}(1)^{-1}}=r^{*}(1)^{-1}\overline{K}\,^{-1}.
\end{equation}
Together, equations \eqref{eq:conjugado de rstar(1)inv} and
\eqref{eq:koverlineinversedadada?} establish, in particular,
$\overline{K}=K^{-1}$. On the other hand, in view of the fact that
given $\alpha',L'$, $p_{\alpha',L'}(\infty)$ does not depend on
$\alpha'$, \eqref{eq:pinvs} establishes $K^{-1}=K$. Hence
\begin{equation}\label{eq:overlineKisK}
\overline{K}=K^{-1}=K.
\end{equation}
\end{rem}

Let $V^{*}$ be the B\"{a}cklund transform of $V$ of parameters
$\alpha,L^{\alpha}$. Equation \eqref{eq:conjugado de rstar(1)inv},
together with equation \eqref{eq:KV=V}, shows that the reality of
$V$ establishes that of $V^{*}$,
$$\overline{V^{*}}=V^{*}.$$
Following Theorem \ref{CHtransf}, we get:
\begin{thm}
If $V$ is a real $q$-constrained harmonic bundle, then the
B\"{a}cklund transform $V^{*}$ of $V$, of parameters
$\alpha,L^{\alpha}$, is a real $q^{*}$-constrained harmonic bundle.
\end{thm}

B\"{a}cklund transformations of constrained harmonic bundles
preserve reality conditions, for this specific choice of parameters.
As we verify next, this leads us, via the central sphere congruence,
to a transformation of real constrained Willmore surfaces,
preserving the Willmore surface condition.

Suppose $\Lambda$ is a real $q$-constrained Willmore surface having
$V$ as the complexification of its central sphere congruence. In
particular, $q$ is real. According to equations \eqref{eq:conjugado
de rstar(1)inv} and \eqref{eq:conjugado de rstar(0)}, the reality of
$q$ establishes that of $q^{*}$:
\begin{eqnarray*}
\overline{(q^{*})^{1,0}}&=&\overline{r^{*}(1)^{-1}}\,\,\overline{r^{*}(0)}\,\,\overline{q^{1,0}}\,\,\overline{r^{*}(1)^{-1}\,r^{*}(0)}\,^{-1}\\&=&
r^{*}(1)^{-1}r^{*}(\infty)\,q^{0,1}\,r^{*}(\infty)^{-1}\,r^{*}(1)\\&=&(q^{*})^{0,1}
\end{eqnarray*}
and, therefore, $$\overline{q^{*}}=q^{*}.$$ Let
$((\Lambda^{*})^{1,0},(\Lambda^{*})^{0,1})$ be the B\"{a}cklund
transform of $\Lambda$ of parameters $\alpha,L^{\alpha}$. Yet again
according to equations \eqref{eq:conjugado de rstar(1)inv} and
\eqref{eq:conjugado de rstar(0)},
\begin{eqnarray*}
\overline{(\Lambda^{*})^{0,1}}&=&\overline{r^{*}(1)^{-1}}\,\overline{r^{*}(0)}\,\overline{\Lambda^{0,1}}\\&=&
r^{*}(1)^{-1}\,r^{*}(\infty)\Lambda^{1,0}\\&=&(\Lambda^{*})^{1,0}
\end{eqnarray*}
establishing the reality of the bundle
$$\Lambda^{*}:=(\Lambda^{*})^{1,0}\cap (\Lambda^{*})^{0,1}.$$
If $\Lambda^{*}$ defines an immersion of $M$ into
$\mathbb{P}(\mathcal{L})$, then, according  to Theorem
\ref{eq:thm8.4.2paraja}, $V^{*}$ is the complexification of the
central sphere congruence of $\Lambda^{*}$, which establishes
$\Lambda^{*}$ as a real $q^{*}$-constrained Willmore surface.

\begin{thm}\label{backwillm}
If $\Lambda$ is a real $q$-constrained Willmore surface, then the
B\"{a}cklund transform $\Lambda^{*}$ of $\Lambda$, of parameters
$\alpha,L^{\alpha}$, is a real $q^{*}$-constrained Willmore surface,
provided that it immerses.
\end{thm}

Note that, if $\Lambda^{*}$ defines an immersion of $M$ into
$\mathbb{P}(\mathcal{L})$, then
$$\mathcal{C}_{\Lambda^{*}}=\mathcal{C}=\mathcal{C}_{\Lambda}.$$

The geometry of B\"{a}cklund transformation is not clear.
\begin{rem}
Note that
$$p_{\alpha',\rho
L'}(\lambda)=p_{\alpha',L'}(\lambda)^{-1},\,\,\,\,\,\,q_{\alpha',\rho
L'}(\lambda)=q_{\alpha',L'}(\lambda)^{-1},$$ for all $\alpha',L'$
and $\lambda\in\mathbb{P}^{1}\backslash\{\pm\alpha'\} $. By
\eqref{eq:pinvs}, and having in consideration the reality of $\rho$,
we get $q_{\overline{-\alpha}\,^{-1},\overline{\rho
L^{\alpha}}}\,(-\alpha)=q_{\overline{\alpha}\,^{-1},\rho\overline{L^{\alpha}}}\,(\alpha)=
q_{\overline{\alpha}\,^{-1},\overline{L^{\alpha}}}\,(\alpha)^{-1}$
and, therefore, by \eqref{eq:rhopqinv}
\begin{equation}\label{eq:qrgosl2w4567bdaw}
q_{\overline{-\alpha}\,^{-1},\overline{\rho L^{\alpha}}}\,(-\alpha)=
\rho\,q_{\overline{\alpha}\,^{-1},\overline{L^{\alpha}}}\,(\alpha)\,\rho,
\end{equation}
making clear that the orthogonality of $\rho
\,q_{\overline{\alpha}\,^{-1},\overline{L^{\alpha}}}\,(\alpha)L^{\alpha}$
 and
 $q_{\overline{\alpha}\,^{-1},\overline{L^{\alpha}}}\,(\alpha)L^{\alpha}$
 is equivalent to that of
 $\rho\,q_{\overline{-\alpha}\,^{-1},\overline{\rho L^{\alpha}}}\,(-\alpha)\rho L^{\alpha}$
 and
 $q_{\overline{-\alpha}\,^{-1},\overline{\rho L^{\alpha}}}\,(-\alpha)\rho L^{\alpha}$.
On the other hand, as we know, $L^{\alpha}$ is a
$d^{\alpha,q}_{V}$-parallel null line bundle if and only if $\rho
L^{\alpha}$ is a $d^{-\alpha,q}_{V}$-parallel null line bundle. We
conclude that $\alpha,L^{\alpha}$ are B\"{a}cklund transformation
parameters if and only if so are $-\alpha,\rho L^{\alpha}$.
Furthermore: the B\"{a}cklund transforms of parameters
$\alpha,L^{\alpha}$ and $-\alpha,\rho L^{\alpha}$ coincide. In fact,
according to equation \eqref{eq:qrgosl2w4567bdaw},
$$\tilde{L}^{-\alpha}=q_{\overline{-\alpha}\,^{-1},\overline{\rho L^{\alpha}}}\,(-\alpha)\rho
L^{\alpha}=
\rho\,q_{\overline{\alpha}\,^{-1},\overline{L^{\alpha}}}\,(\alpha)L^{\alpha}=\rho\tilde{L}^{\alpha}$$
and, therefore,
$$p_{-\alpha,\tilde{L}^{-\alpha}}(\lambda)\,q_{\overline{-\alpha}\,^{-1},\overline{\rho
L^{\alpha}}}\,(\lambda)=
p_{\alpha,\tilde{L}^{\alpha}_{\beta}}(\lambda)\,q_{\overline{\alpha}\,^{-1},\overline{L^{\alpha}}}\,(\lambda),$$
for all $\lambda\in\mathbb{P}^{1}\backslash\{\pm\alpha\}$.
\end{rem}

We complete this section with an interesting and useful result
relating the quadratic differentials $q_{Q}$ and $q^{*}_{Q}$:

\begin{prop}\label{quaddiffscoincide}
$$q^{*}_{Q}=q_{Q}.$$
\end{prop}
\begin{proof}
Fix $z$ a holomorphic chart of $(M,\mathcal{C}_{\Lambda})$. The
proof will consist of showing that $(q^{*})^{z}=q^{z}$, or,
equivalently, that
$q^{z}\tau^{*}=-2q^{*}_{\delta_{z}}((\mathcal{D}_{V^{*}})_{\delta_{z}}\tau^{*})$,
for all $\tau^{*}\in\Gamma((\Lambda^{*})^{0,1})$. Fix
$\tau\in\Gamma(\Lambda^{0,1})$ and set
$\tau^{*}:=r^{*}(1)^{-1}r^{*}(0)\tau$. Set also
$$\mathcal{R}:=r^{*}(\infty)^{-1}r^{*}(0).$$
Let $\hat{d}$ be as defined in Section \ref{sec:dress} for
$r=r^{*}$. According to \eqref{eq:isomV}, followed by equation
\eqref{eq:mathcalDdosr},
\begin{eqnarray*}
\mathcal{D}_{V^{*}}^{1,0}\circ
r^{*}(1)^{-1}r^{*}(0)&=&r^{*}(1)^{-1}\circ(\mathcal{D}_{V}^{\hat{d}})^{1,0}\circ
r^{*}(0)\\&=&
r^{*}(1)^{-1}r^{*}(0)q^{1,0}+r^{*}(1)^{-1}r^{*}(\infty)\circ
(\mathcal{D}_{V}^{1,0}-q^{1,0})\circ \mathcal{R}
\end{eqnarray*}
and, therefore,
$$q^{*}_{\delta_{z}}((\mathcal{D}_{V^{*}})_{\delta_{z}}\tau^{*})=
r^{*}(1)^{-1}r^{*}(0)\,q_{\delta_{z}}\mathcal{R}^{-1}\circ((\mathcal{D}_{V})_{\delta_{z}}-q_{\delta_{z}})\circ
\mathcal{R}\,\tau,$$as $q^{1,0}\Lambda^{0,1}=0$. On the other hand,
$$\mathcal{R}^{-1}\circ((\mathcal{D}_{V})_{\delta_{z}}-q_{\delta_{z}})\circ
\mathcal{R}=\mathcal{R}^{-1}\circ((\mathcal{D}_{V})_{\delta_{z}}-q_{\delta_{z}})\mathcal{R}+
(\mathcal{D}_{V})_{\delta_{z}}-q_{\delta_{z}},$$so we conclude that
$$q^{*}_{\delta_{z}}((\mathcal{D}_{V^{*}})_{\delta_{z}}\tau^{*})=r^{*}(1)^{-1}r^{*}(0)\,q_{\delta_{z}}\mathcal{R}^{-1}\circ((\mathcal{D}_{V})_{\delta_{z}}-q_{\delta_{z}})\mathcal{R}\tau-\frac{1}{2}\,
q^{z}\tau^{*}.$$ We complete this verification by showing that
$\mathcal{R}^{-1}\circ((\mathcal{D}_{V})_{\delta_{z}}-q_{\delta_{z}})\mathcal{R}\,\Gamma(\Lambda^{0,1})\subset\Gamma(\Lambda^{0,1})$,
as follows. Set, more generally,
$$\mathcal{R}_{\lambda}:=r^{*}(\infty)^{-1}r^{*}(\lambda).$$
In view of equation
\eqref{eq:curlyDscomadjuntascomesemchapeustnne029876},
$$\mathcal{R}_{\lambda}^{-1}\circ
(\mathcal{D}^{1,0}_{V}-q^{1,0})\mathcal{R}_{\lambda}=r^{*}(\lambda)^{-1}\circ
((\mathcal{D}_{V}^{\hat{d}})^{1,0}-\hat{q}^{1,0})\circ\,r^{*}(\lambda)-\mathcal{D}_{V}^{1,0}+q^{1,0}.$$
On the other hand,
$$(\mathcal{D}_{V}^{\hat{d}})^{1,0}-\hat{q}^{1,0}=r^{*}(\lambda)\circ
(\mathcal{D}^{1,0}_{V}-q^{1,0}+\lambda^{-1}\mathcal{N}^{1,0}_{V}+\lambda^{-2}q^{1,0})\circ
r^{*}(\lambda)^{-1}-
\lambda^{-1}(\mathcal{N}_{V}^{\hat{d}})^{1,0}-\lambda^{-2}\hat{q}^{1,0}.$$Thus
$$\mathcal{R}_{\lambda}^{-1}\circ
(\mathcal{D}^{1,0}_{V}-q^{1,0})\mathcal{R}_{\lambda}=\mathrm{Ad}_{r^{*}(\lambda)^{-1}}(-\lambda^{-1}(\mathcal{N}_{V}^{\hat{d}})^{1,0}-\lambda^{-2}\hat{q}^{1,0}+\mathrm{Ad}_{r^{*}(\lambda)}(\lambda^{-1}\mathcal{N}^{1,0}+\lambda^{-2}q^{1,0})).$$
In view of the holomorphicity of $\mathrm{Ad}_{r^{*}(\lambda)}$ at
$\lambda=0$,
$$\mathrm{Ad}_{r^{*}(\lambda)}=\mathrm{Ad}_{r^{*}(0)}+\sum_{k\geq
1}\frac{\lambda^{k}}{k!}\frac{d^{k}}{d\lambda^{k}}_{\vert_{\lambda=0}}\mathrm{Ad}_{r^{*}(\lambda)}$$
and, therefore, by \eqref{eq:calNhatd1,0},
$$\mathcal{R}_{\lambda}^{-1}\circ
(\mathcal{D}^{1,0}_{V}-q^{1,0})\mathcal{R}_{\lambda}=\mathrm{Ad}_{r^{*}(\lambda)^{-1}}(\frac{d}{d\lambda}_{\vert_{\lambda=0}}\mathrm{Ad}_{r^{*}(\lambda)}\mathcal{N}^{1,0}_{V}+\frac{1}{2}\frac{d^{2}}{d\lambda^{2}}_{\vert_{\lambda=0}}\mathrm{Ad}_{r^{*}(\lambda)}q^{1,0}+o(\lambda)),$$
for $\lambda$ near $0$. Considering limits when $\lambda$ goes to
$0$, we conclude that
$$\mathcal{R}^{-1}\circ
(\mathcal{D}^{1,0}_{V}-q^{1,0})\mathcal{R}=\mathrm{Ad}_{r^{*}(0)^{-1}}(\frac{d}{d\lambda}_{\vert_{\lambda=0}}\mathrm{Ad}_{r^{*}(\lambda)}\mathcal{N}^{1,0}_{V}+\frac{1}{2}\frac{d^{2}}{d\lambda^{2}}_{\vert_{\lambda=0}}\mathrm{Ad}_{r^{*}(\lambda)}q^{1,0}).$$
For simplicity, set
$\psi:=r^{*}(0)^{-1}\frac{d}{d\lambda}_{\vert_{\lambda=0}}r^{*}(\lambda)\in\Gamma(V\wedge
V^{\perp})$ (recalling \eqref{eq:rrsat0}). Note that
$$\frac{d}{d\lambda}_{\vert_{\lambda=0}}\mathrm{Ad}_{r^{*}(\lambda)}\mathcal{N}^{1,0}_{V}=\mathrm{Ad}_{r^{*}(0)}[\psi,\mathcal{N}^{1,0}_{V}].$$
The centrality of the central sphere congruence of $\Lambda$,
$\mathcal{N}_{V}^{1,0}\Lambda^{0,1}=0$, together with the
skew-symmetry of $\mathcal{N}_{V}$, establishes that
$\mathcal{N}^{1,0}_{V}$ takes values in the orthogonal to
$\Lambda^{0,1}$. In particular,
\begin{equation}\label{eq:N10VperpinLambda01}
\mathcal{N}^{1,0}_{V}V^{\perp}\subset
V\cap(\Lambda^{0,1})^{\perp}=\Lambda^{0,1}.
\end{equation}
It follows that
$$\mathrm{Ad}_{r^{*}(0)^{-1}}(\frac{d}{d\lambda}_{\vert_{\lambda=0}}\mathrm{Ad}_{r^{*}(\lambda)}\mathcal{N}^{1,0}_{V})\Lambda^{0,1}=-\mathcal{N}^{1,0}_{V}\psi\Lambda^{0,1}\subset \Lambda^{0,1}.$$
On the other hand,
$$\mathrm{Ad}_{r^{*}(0)^{-1}}(\frac{d^{2}}{d\lambda^{2}}_{\vert_{\lambda=0}}\mathrm{Ad}_{r^{*}(\lambda)}q^{1,0})=r^{*}(0)^{-1}\frac{d^{2}}{d\lambda^{2}}_{\vert_{\lambda=0}}r^{*}(\lambda)q^{1,0}
-2\psi q^{1,0}\psi+2q^{1,0}\psi\psi$$ and, therefore,
$$\mathrm{Ad}_{r^{*}(0)^{-1}}(\frac{d^{2}}{d\lambda^{2}}_{\vert_{\lambda=0}}\mathrm{Ad}_{r^{*}(\lambda)}q^{1,0})\Lambda^{0,1}=q^{1,0}\psi\psi\Lambda^{0,1},$$
as $q^{1,0}\Lambda^{0,1}=0=qV^{\perp}$. The fact that $q^{1,0}$
takes values in $\Lambda^{0,1}$ completes the proof.
\end{proof}

\section{B\"{a}cklund transformation vs. spectral deformation}\label{BTvsSp}

\markboth{\tiny{A. C. QUINTINO}}{\tiny{CONSTRAINED WILLMORE
SURFACES}}

B\"{a}cklund transformation and spectral deformation, of constrained
harmonic bundles or complexified constrained Willmore surfaces, are
closely related: as we verify in this section, the B\"{a}cklund
transform of parameters
$\frac{\alpha}{\lambda},\frac{\beta}{\lambda},\phi^{\lambda}
L^{\alpha},\phi^{\lambda} L^{\beta}$ of the spectral deformation
 of parameter $\lambda$, corresponding to
a multiplier $q$, defined by $\phi^{\lambda}$, coincides with the
spectral deformation of parameter $\lambda$, corresponding to the
multiplier $q^{*}$, of the B\"{a}cklund transform of parameters
$\alpha,\beta,L^{\alpha},L^{\beta}$.\newline

Suppose $V$ is a $q$-constrained harmonic bundle, for some
$q\in\Omega^{1}(\wedge^{2}V\oplus \wedge^{2}V^{\perp})$. Let $V^{*}$
be the B\"{a}cklund transform of $V$ of parameters
$\alpha,\beta,L^{\alpha},L^{\beta}$. Let $r^{*}$ denote
$r^{(\alpha,\beta)}_{L^{\alpha},L^{\beta}}$ and, for each
$\lambda\in\C\backslash\{0,\pm\alpha,\pm\beta\}$, let
$\hat{d}^{\lambda,\hat{q}}_{V}$ be as defined in Section
\ref{sec:dress} for $r=r^{*}$. By definition of
$\hat{d}^{\lambda,\hat{q}}_{V}$,
$r^{*}(\lambda):(\underline{\C}^{n+2},d^{\lambda,q}_{V})\rightarrow
(\underline{\C}^{n+2},\hat{d}^{\lambda,\hat{q}}_{V})$ is an isometry
of bundles preserving connections, for each $\lambda$. In
particular, so is $r^{*}(1):(\underline{\C}^{n+2},d)\rightarrow
(\underline{\C}^{n+2},\hat{d})$. Fix $\lambda$ in
$\C\backslash\{0,\pm\alpha,\pm\beta\}$. In view of
\eqref{eq:isomVd}, we conclude that
$$d^{\lambda,q^{*}}_{V^{*}}=r^{*}(1)^{-1}\,r^{*}(\lambda)\circ d^{\lambda,q}_{V}\circ
r^{*}(\lambda)^{-1}r^{*}(1).$$ It follows that, given
$$\phi^{\lambda}:(\underline{\C}^{n+2},d^{\lambda,q}_{V})\rightarrow
(\underline{\C}^{n+2},d),$$an isometry preserving connections, so is
$$\psi^{\lambda}:=\phi^{\lambda} \,r^{*}(\lambda)^{-1}\,r^{*}(1):(\underline{\C}^{n+2},d^{\lambda,q^{*}}_{V^{*}})\rightarrow
(\underline{\C}^{n+2},d),$$providing, therefore, the spectral
deformation $\psi^{\lambda}V^{*}$, of parameter $\lambda$,
corresponding to the multiplier $q^{*}$, of $V^{*}$. On the other
hand, such isometry $\phi^{\lambda}$ provides the spectral
deformation $\phi^{\lambda} V$, of parameter $\lambda$,
corresponding to the multiplier $q$, of $V$. Next we focus on this
$\mathrm{Ad}_{\phi^{\lambda}}q_{\lambda}$-constrained harmonic
bundle. For simplicity, set
$\tilde{q}_{\lambda}:=\mathrm{Ad}_{\phi^{\lambda}}q_{\lambda}$. Let
$\rho_{V}$ and $\rho_{\phi^{\lambda}V}$ denote reflection across $V$
and, respectively, $\phi^{\lambda}V$. Recalling
\eqref{eq:conndeconn}, we get
$$d^{\frac{\alpha}{\lambda},\tilde{q}_{\lambda}}_{\phi^{\lambda}
V}=\phi^{\lambda}\circ
(d^{\lambda,q}_{V})_{V}^{\frac{\alpha}{\lambda},q_{\lambda}}\circ(\phi^{\lambda})^{-1}=\phi^{\lambda}\circ
d^{\alpha,q}_{V}\circ(\phi^{\lambda})^{-1},$$ which makes clear
that, as $L^{\alpha}$ and $L^{\beta}$ are, respectively,
$d^{\alpha,q}_{V}$- and $d^{\beta,q}_{V}$-parallel, $\phi^{\lambda}
L^{\alpha}$ and $\phi^{\lambda} L^{\beta}$ are, respectively,
$d^{\frac{\alpha}{\lambda},\tilde{q}_{\lambda}}_{\phi^{\lambda} V}$-
and $d^{\frac{\beta}{\lambda},\tilde{q}_{\lambda}}_{\phi^{\lambda}
V}$-parallel. On the other hand, the fact that
\begin{equation}\label{eq:rhophiVvsrhoV}
\rho_{\phi^{\lambda}V}=\phi^{\lambda}\rho_{V}(\phi^{\lambda})^{-1}
\end{equation}
makes clear that, in view of the local non-orthogonality of
$L^{\beta}$ and $\rho_{V}L^{\beta}$, $\phi^{\lambda} L^{\beta}$ and
$\rho_{\phi^{\lambda}V}\phi^{\lambda} L^{\beta}$ are, locally,
non-orthogonal, as well. Equation \eqref{eq:rhophiVvsrhoV}
establishes, on the other hand,
$$p_{\phi^{\lambda} V,\alpha',\phi^{\lambda} L'}(\lambda)=\phi^{\lambda}\,
p_{V,\alpha',L'}(\lambda)\,(\phi^{\lambda})^{-1},\,\,\,\,\,\,
q_{\phi^{\lambda} V,\alpha',\phi^{\lambda}
L'}(\lambda)=\phi^{\lambda}\,
q_{V,\alpha',L'}(\lambda)\,(\phi^{\lambda})^{-1},$$ for all
$\alpha',L'$ and $\lambda\neq\pm\alpha' $. Note that
$$p_{\alpha',L'}(\lambda)=p_{\frac{\alpha'}{\lambda},L'}(1),\,\,\,\,\,q_{\alpha',L'}(\lambda)=q_{\frac{\alpha'}{\lambda},L'}(1).$$
for all $\alpha',L'$ and $\lambda\neq\pm\alpha'$. It follows that
$$\tilde{L}^{\alpha/\lambda}_{\beta/\lambda}=q_{\phi^{\lambda}V,\frac{\beta}{\lambda},\phi^{\lambda}L^{\beta}}(\frac{\alpha}{\lambda})\phi^{\lambda}L^{\alpha}=
\phi^{\lambda}\,q_{V,\frac{\beta}{\lambda},L^{\beta}}(\frac{\alpha}{\lambda})L^{\alpha}=
\phi^{\lambda}\,q_{V,\beta,L^{\beta}}(\alpha)L^{\alpha}=\phi^{\lambda}\,\tilde{L}^{\alpha}_{\beta}$$
establishing the local non-orthogonality of
$\rho_{\phi^{\lambda}V}\tilde{L}^{\alpha/\lambda}_{\beta/\lambda}=\phi^{\lambda}\,\rho_{V}\,\tilde{L}^{\alpha}_{\beta}$
and $\tilde{L}^{\alpha/\lambda}_{\beta/\lambda}$, in view of that of
$\rho_{V}\,\tilde{L}^{\alpha}_{\beta}$ and
$\tilde{L}^{\alpha}_{\beta}$. We conclude that, for each $\lambda$,
$\frac{\alpha}{\lambda},\frac{\beta}{\lambda},\phi^{\lambda}
L^{\alpha},\phi^{\lambda} L^{\beta}$ constitute B\"{a}cklund
transformation parameters to $\phi^{\lambda} V$. Set
$$r^{*}_{\lambda}:=p_{\phi^{\lambda}
V,\frac{\alpha}{\lambda},\tilde{L}^{\alpha/\lambda}_{\beta/\lambda}}\,q_{\phi^{\lambda}
V,\frac{\beta}{\lambda},\phi^{\lambda}
L^{\beta}}=\phi^{\lambda}\,p_{
V,\frac{\alpha}{\lambda},\tilde{L}^{\alpha}_{\beta}}\,q_{
V,\frac{\beta}{\lambda},L^{\beta}}\,(\phi^{\lambda})^{-1}.$$Then
$$r^{*}_{\lambda}(1)^{-1}\,\phi^{\lambda}=\phi^{\lambda}\,q_{V,\beta,L^{\beta}}(\lambda)^{-1}\,p_{V,\alpha,\tilde{L}^{\alpha}_{\beta}}(\lambda)^{-1}=
\phi^{\lambda}\,r^{*}(\lambda)^{-1}$$ and, therefore,
$r^{*}_{\lambda}(1)^{-1}\,\phi^{\lambda}=\psi^{\lambda}\,r^{*}(1)^{-1}$.
In particular,
$$r^{*}_{\lambda}(1)^{-1}\,\phi^{\lambda}V=\psi^{\lambda}\,r^{*}(1)^{-1}V,$$
the B\"{a}cklund transform of parameters
$\frac{\alpha}{\lambda},\frac{\beta}{\lambda},\phi^{\lambda}
L^{\alpha},\phi^{\lambda} L^{\beta}$ of the spectral deformation
$\phi^{\lambda}V$ of $V$, of parameter $\lambda$ [corresponding to
the multiplier $q$], coincides with the spectral deformation of
parameter $\lambda$ [corresponding to the multiplier $q^{*}$] of the
B\"{a}cklund transform of $V$ of parameters
$\alpha,\beta,L^{\alpha},L^{\beta}$.

This permutability between B\"{a}cklund transformation and spectral
deformation of constrained harmonic bundles extends to constrained
Willmore surfaces. Indeed, if $V$ is a $(q,d)$-central sphere
congruence to some $q$-constrained Willmore surface
$(\Delta^{1,0},\Delta^{0,1})$, the equality
$$r^{*}_{\lambda}(1)^{-1}=\phi^{\lambda}\,r^{*}(\lambda)^{-1}\,(\phi^{\lambda})^{-1}$$
establishes, furthermore,
\begin{eqnarray*}
r_{\lambda}^{*}(1)^{-1}r_{\lambda}^{*}(\infty)\,\phi^{\lambda}\Delta^{1,0}&=&\phi^{\lambda}r^{*}(\lambda)^{-1}p_{V,\frac{\alpha}{\lambda},\tilde{L}^{\alpha}_{\beta}}(\infty)\Delta^{1,0}\\&=&
\phi^{\lambda}r^{*}(\lambda)^{-1}p_{V,\alpha,\tilde{L}^{\alpha}_{\beta}}(\infty)\Delta^{1,0}
\end{eqnarray*}
and, therefore,
$$r_{\lambda}^{*}(1)^{-1}r_{\lambda}^{*}(\infty)\,\phi^{\lambda}\Delta^{1,0}=\psi^{\lambda}\,r^{*}(1)^{-1}r^{*}(\infty)\Delta^{1,0};$$
and, similarly,
$$r_{\lambda}^{*}(1)^{-1}r_{\lambda}^{*}(0)\,\phi^{\lambda}\Delta^{0,1}=\psi^{\lambda}\,r^{*}(1)^{-1}r^{*}(0)\Delta^{0,1}.$$

\begin{thm}
Let $V$ be a $q$-constrained harmonic bundle,
$\alpha$,$\beta$,$L^{\alpha}$,$L^{\beta}$ be B\"{a}cklund
transformation parameters to $V$, $\lambda\in\C\backslash\{0,
\pm\alpha,\pm\beta\}$ and
$\phi^{\lambda}:(\underline{\C}^{n+2},d^{\lambda,q}_{V})\rightarrow
(\underline{\C}^{n+2},d)$ be an isometry preserving connections. The
B\"{a}cklund transform of parameters
$\frac{\alpha}{\lambda},\frac{\beta}{\lambda}$, $\phi^{\lambda}
L^{\alpha}$, $\phi^{\lambda} L^{\beta}$ of the spectral deformation
$\phi^{\lambda}V$ of $V$, of parameter $\lambda$ [corresponding to
the multiplier $q$], coincides with the spectral deformation of
parameter $\lambda$ [corresponding to the multiplier $q^{*}$] of the
B\"{a}cklund transform of parameters
$\alpha,\beta,L^{\alpha},L^{\beta}$ of $V$. Furthermore, if $V$ is a
$q$-central sphere congruence to a constrained Willmore surface
$(\Delta^{1,0},\Delta^{0,1})$ and
$\phi^{\lambda}_{*}:(\underline{\C}^{n+2},d^{\lambda,q^{*}}_{V^{*}})\rightarrow
(\underline{\C}^{n+2},d)$ is an isometry preserving connections,
then the diagram in Figure \ref{fig:im2} commutes.
\begin{center}
\begin{figure}[H]
\includegraphics{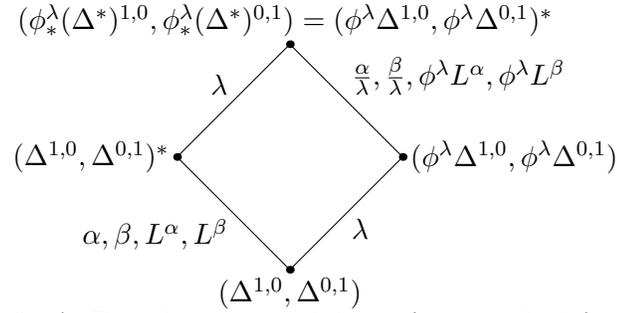}
\caption{A Bianchi permutability of spectral deformation and
B\"{a}cklund transformation of constrained Willmore
surfaces.}\label{fig:im2}
\end{figure}
\end{center}
\end{thm}
For $\lambda\in\{\pm\alpha,\pm\beta\}$, it is not clear how the
spectral deformation of parameter $\lambda$ relates to the
B\"{a}cklund transformation of parameters
$\alpha,\beta,L^{\alpha},L^{\beta}$.

\chapter{Constrained Willmore surfaces with a conserved
quantity}\label{chaptercq}

\markboth{\tiny{A. C. QUINTINO}}{\tiny{CONSTRAINED WILLMORE
SURFACES}}

A powerful result by E. Noether \cite{noether} establishes that any
symmetry of the action of a physical system has a corresponding
conservation law. Time translation symmetry gives conservation of
energy, space translation symmetry gives conservation of momentum,
symmetry under rotation gives conservation of angular momentum -
these are examples of physically conserved quantities one gets from
symmetries of the laws of nature. Noether's theorem has become a
fundamental tool of modern theoretical physics and the calculus of
variations. In this chapter, we introduce the concept of
\textit{conserved quantity} of a constrained Willmore surface in a
space-form, an idea by Fran Burstall and David Calderbank.
Constrained Willmore surfaces in space-forms admitting a conserved
quantity form a subclass of constrained Willmore surfaces preserved
by both spectral deformation and B\"{a}cklund transformation, for
special choices of parameters. In codimension $1$, this class
consists of the class of constant mean curvature surfaces in
space-forms. In codimension $2$, surfaces with holomorphic mean
curvature vector in some space-form are examples of constrained
Willmore surfaces admitting a conserved quantity.\newline

Let $\Lambda\subset\underline{\R}^{n+1,1}$ be a $q$-constrained
Willmore surface in the projectivized light-cone, for some real
$1$-form $q$ with values in $\Lambda\wedge\Lambda^{(1)}$. Let $S$ be
the central sphere congruence of $\Lambda$. Provide $M$ with the
conformal structure $\mathcal{C}_{\Lambda}$, induced by $\Lambda$.

\section{Conserved quantities of constrained Willmore
surfaces}\label{cqofCWs}

\markboth{\tiny{A. C. QUINTINO}}{\tiny{CONSTRAINED WILLMORE
SURFACES}}

Let
$$p(\lambda):=\lambda^{-1}v+v_{0}+\lambda\overline{v}$$ be a Laurent
 polynomial with $v_{0}\in\Gamma(S)$ real, $v\in\Gamma(S^{\perp})$
and $v_{\infty}:=p(1)=v_{0}+v+\overline{v}\neq 0$.

\begin{defn}
We say that $p(\lambda)$ is a $q$-\emph{conserved} \emph{quantity}
of $\Lambda$ if
\begin{equation}\label{eq:dlambdaqplambda0(CQ)}
d^{\lambda,q}_{S}p(\lambda)=0,
\end{equation}
for all $\lambda\in\C\backslash\{0\}$.
\end{defn}

\begin{thm}\label{charactdeCQ}
$p(\lambda)$ is a $q$-conserved quantity of $\Lambda$ if and only if
$$dv_{\infty}=0,\,\,\,\,\mathcal{D}^{0,1}v=0,\,\,\,\,\mathcal{N}^{1,0}v+q^{1,0}v_{0}=0.$$
\end{thm}

\begin{proof}
Given $\lambda\in\C\backslash\{0\}$,
\begin{eqnarray*}
d^{\lambda,q}_{S}p(\lambda)&=&\lambda^{-1}\mathcal{D}v+\mathcal{D}v_{0}+\lambda\mathcal{D}\overline{v}+\lambda^{-2}\mathcal{N}^{1,0}v+\lambda^{-1}\mathcal{N}^{1,0}v_{0}+\mathcal{N}^{1,0}\overline{v}\\&
& \mbox{}+\mathcal{N}^{0,1}v
+\lambda\mathcal{N}^{0,1}v_{0}+\lambda^{2}\mathcal{N}^{0,1}\overline{v}+(\lambda^{-2}-1)q^{1,0}v_{0}+(\lambda^{2}-1)q^{0,1}v_{0}.
\end{eqnarray*}
Organizing the terms in $d^{\lambda,q}_{S}p(\lambda)$ by powers of
$\lambda$ shows that equation \eqref{eq:dlambdaqplambda0(CQ)} holds
for all $\lambda\in\C\backslash\{0\}$ if and only if so do equations
$\mathcal{D}v+\mathcal{N}^{1,0}v_{0}=0=\mathcal{D}\overline{v}+\mathcal{N}^{0,1}v_{0}$,
together with
$\mathcal{N}^{1,0}v+q^{1,0}v_{0}=0=\mathcal{N}^{0,1}\overline{v}+q^{0,1}v_{0}$
and with
\begin{equation}\label{eq:CQ3}
\mathcal{D}v_{0}+\mathcal{N}^{1,0}\overline{v}+\mathcal{N}^{0,1}v-q^{1,0}v_{0}-q^{0,1}v_{0}=0.
\end{equation}
In view of the reality of $\mathcal{D}$, $\mathcal{N}$ and $q$, we
conclude that equation \eqref{eq:dlambdaqplambda0(CQ)} holds if and
only if so do equations \eqref{eq:CQ3},
\begin{equation}\label{eq:CQ1}
\mathcal{D}v+\mathcal{N}^{1,0}v_{0}=0
\end{equation}
and
\begin{equation}\label{eq:CQ2}
\mathcal{N}^{1,0}v+q^{1,0}v_{0}=0.
\end{equation}
Equation \eqref{eq:CQ1} is equivalent to the system of equations
\begin{equation}\label{eq:CQ4}
\mathcal{D}^{1,0}v+\mathcal{N}^{1,0}v_{0}=0
\end{equation}
and
\begin{equation}\label{eq:CQ5}
\mathcal{D}^{0,1}v=0.
\end{equation}
In the light of \eqref{eq:CQ2}, equation \eqref{eq:CQ3} reads
$\mathcal{D}v_{0}+\mathcal{N}^{1,0}\overline{v}+\mathcal{N}^{0,1}v+\mathcal{N}^{1,0}v+\mathcal{N}^{0,1}\overline{v}=0$,
i.e.,
\begin{equation}\label{eq:CQ1.5}
\mathcal{D}v_{0}+\mathcal{N}(v+\overline{v})=0.
\end{equation}
On the other hand, in view of equations \eqref{eq:CQ4},
\eqref{eq:CQ5} and \eqref{eq:CQ1.5},
\begin{eqnarray*}
dv_{\infty}&=&\mathcal{D}v_{0}+\mathcal{D}(v+\overline{v})+\mathcal{N}v_{0}+\mathcal{N}(v+\overline{v})\\&=&\mathcal{D}^{1,0}v+\mathcal{D}^{0,1}v+\mathcal{D}^{1,0}\overline{v}+\mathcal{D}^{0,1}\overline{v}+\mathcal{N}^{1,0}v_{0}+\mathcal{N}^{0,1}v_{0}
\\&=&\overline{\mathcal{D}^{0,1}v}+\overline{\mathcal{D}^{1,0}v+\mathcal{N}^{1,0}v_{0}}
\end{eqnarray*}
and, ultimately,
\begin{equation}\label{eq:CQ7}
dv_{\infty}=0.
\end{equation}
We complete the proof by observing that, together with equation
\eqref{eq:CQ5}, equation \eqref{eq:CQ7} establishes \eqref{eq:CQ4}
and \eqref{eq:CQ1.5}. For that, first note that
$\pi_{S}\,(dv_{\infty})=\mathcal{D}v_{0}+\mathcal{N}(v+\overline{v})$
and
$\pi_{S^{\perp}}(dv_{\infty})=\mathcal{D}(v+\overline{v})+\mathcal{N}v_{0}$.
Thus $dv_{\infty}$ vanishes if and only if
$$\mathcal{D}v_{0}+\mathcal{N}(v+\overline{v})=0=\mathcal{D}(v+\overline{v})+\mathcal{N}v_{0}.$$
In particular, $dv_{\infty}=0$ forces
$0=\mathcal{N}^{1,0}v_{0}+\mathcal{D}^{1,0}v+\mathcal{D}^{1,0}\overline{v}$
which, together with equation \eqref{eq:CQ5}, establishes equation
\eqref{eq:CQ4}, completing the proof.
\end{proof}

For later reference,
\begin{rem}\label{qCWnotinterveneonproofrenm}
The characterization of the $d^{\lambda,q}_{S}$-parallelism of
$p(\lambda)$ by the equations in Theorem \ref{charactdeCQ} does not
involve the fact that $\Lambda$ is $q$-constrained Willmore, only
the reality of $q\in\Omega^{1}(\Lambda\wedge\Lambda^{(1)})$.
\end{rem}

\begin{rem}\label{CQeqdeterminesq}
In the characterization above, of a $q$-conserved quantity of
$\Lambda$, equation $\mathcal{N}^{1,0}v+q^{1,0}v_{0}=0$ determines
$q$. In fact, $q$ is real, so it is determined by $q^{1,0}$, which,
given that $\Lambda$ is a $q$-constrained Willmore surface, is
ensured to be a $1$-form with values in
$\Lambda\wedge\Lambda^{0,1}$. Hence $q^{1,0}$ is determined by
$q^{1,0} u\in\Gamma(\Lambda^{0,1})$, fixing $u\in S\backslash
\Lambda^{\perp}$ (cf. Remark \ref{e10determinedby}). On the other
hand, $\mathcal{N}^{1,0}v\in\Gamma(\Lambda^{0,1})$ (cf.
\eqref{eq:N10VperpinLambda01}). Finally, if $dv_{\infty}=0$, then
$v_{\infty}$ defines a space-form $S_{v_{\infty}}$, and, for
$\sigma_{\infty}$ the surface in $S_{v_{\infty}}$ defined by
$\Lambda$,
$(\sigma_{\infty},v_{0})=(\sigma_{\infty},\pi_{S}(v_{\infty}))=(\sigma_{\infty},v_{\infty})=-1$
is never-zero.
\end{rem}

\begin{rem}
If $p(\lambda)$ is a $q$-conserved quantity of $\Lambda$, then
$$d(v,v)=0.$$ Indeed, by equation \eqref{eq:CQ2}, and in view of the
skew-symmetry of $q$,
$$(\mathcal{N}^{1,0}v,v_{0})=-(q^{1,0}v_{0},v_{0})=(v_{0},q^{1,0}v_{0})=-(v_{0},\mathcal{N}^{1,0}v)$$
and, therefore, $(\mathcal{N}^{1,0}v,v_{0})=0$. Equivalently, $
(v,\mathcal{N}^{1,0}v_{0})=0$, in view of the skew-symmetry of
$\mathcal{N}$, which, together with equation \eqref{eq:CQ4},
establishes $(v,\mathcal{D}^{1,0}v)=0$ and, consequently, by
equation \eqref{eq:CQ5}, $(v,\mathcal{D}v)=0$.
\end{rem}

\section{Examples}

\markboth{\tiny{A. C. QUINTINO}}{\tiny{CONSTRAINED WILLMORE
SURFACES}}

Two special cases follow from the characterization of a conserved
quantity provided by Theorem \ref{charactdeCQ}.
\subsection{The special case of codimension $1$: CMC surfaces in
$3$-space}\label{sec:casecodim1}

\markboth{\tiny{A. C. QUINTINO}}{\tiny{CONSTRAINED WILLMORE
SURFACES}}

The existence of a conserved quantity $p(\lambda)$ of $\Lambda$
establishes, in particular, the constancy of $v_{\infty}:=p(1)$. In
the particular case of $n=3$, we verify that $\Lambda$ has constant
mean curvature in the space-form $S_{v_{\infty}}$, that is, the
surface defined by $\Lambda$ in the space-form $S_{v_{\infty}}$ has
constant mean curvature. In fact, constant mean curvature surfaces
in $3$-dimensional space-forms are, precisely, the constrained
Willmore surfaces in $3$-space-form admitting a conserved quantity.
This case will be addressed in detail in Section \ref{sec:CMC},
dedicated to constant mean curvature surfaces in $3$-dimensional
space-forms.
\subsection{A special case in codimension
$2$: holomorphic mean curvature vector surfaces in
$4$-space}\label{HMCsurfs}

\markboth{\tiny{A. C. QUINTINO}}{\tiny{CONSTRAINED WILLMORE
SURFACES}}

In codimension $2$, the complexification of $S^{\perp}$ admits a
unique decomposition into the direct sum of two null complex lines,
complex conjugate of each other: given $v\in\Gamma(S^{\perp})$ null,
$S^{\perp}=\langle v\rangle\oplus \langle\overline{v}\rangle$. Such
a $v$ defines an almost-complex structure $J_{v}$ on $S^{\perp}$,
with eigenvalues $i$ and $-i$ and eigenspaces $\langle v\rangle$ and
$\langle\overline{v}\rangle$, associated to $i$ and $-i$,
respectively. In \cite{burstall+calderbank} (see, in particular,
Corollary 14.3), F. Burstall and D. Calderbank proved that a
codimension $2$ surface in a space-form, with holomorphic mean
curvature vector with respect to the complex structure induced by
$\nabla^{S^{\perp}}$, is constrained Willmore. In this section, we
prove it in our setting, proving, furthermore, that, in
$4$-dimensional space-form, the constrained Willmore surfaces
admitting a conserved quantity
$p(\lambda)=\lambda^{-1}v+v_{0}+\lambda\overline{v}$ with $v$ null
are the surfaces with holomorphic mean curvature vector in the
space-form $S_{p(1)}$, with respect to the complex structure on
$(S^{\perp},J_{v},\nabla^{S^{\perp}})$ determined by
Koszul-Malgrange Theorem.\newline

Suppose $n=4$. In that case, $S^{\perp}$ is a (non-degenerate real)
rank $2$ bundle, admitting, therefore, a unique decomposition
\begin{equation}\label{eq:SperpSpm}
S^{\perp}=S^{\perp}_{+}\oplus S^{\perp}_{-}
\end{equation}
of its complexification into the direct sum of two null complex
lines, complex conjugate of each other. In particular, $S^{\perp}$
admits an almost-complex structure (in fact, two, differing by
sign),
$$J_{S^{\perp}}=\pm\, I\left\{
\begin{array}{ll} i & \mbox{$\mathrm{on}\,S^{\perp}_{+}$}\\ -i &
\mbox{$\mathrm{on}\,S^{\perp}_{-}$}\end{array}\right..$$ Provide
then $S^{\perp}$ with the unique complex structure compatible with
the connection $\nabla^{S^{\perp}}=\mathcal{D}\vert _{\Gamma
(S^{\perp})}$, cf. Koszul-Malgrange Theorem, characterized by the
fact that a section $\nu$ of $S^{\perp}$ is holomorphic if and only
if $\mathcal{D}^{0,1}\nu=0$ (see, for example,
\cite{burstall+rawnsley}, Theorem $2.1$). Fix a choice of
$J_{S^{\perp}}$ and provide $S^{\perp}$ with the structure of
complex vector bundle defined by $i\nu:=J_{S^{\perp}}\nu$, for all
real $\nu\in\Gamma(S^{\perp})$. Then, given
$\nu\in\Gamma(S^{\perp})$ real and $z=x+iy$ a holomorphic chart of
$M$,
\begin{equation}\label{eq:cJcmestrevcompl}
\mathcal{D}_{\delta_{\bar{z}}}\nu=\frac{1}{2}\,(\mathcal{D}_{\delta_{x}}\nu+J_{S^{\perp}}\mathcal{D}_{\delta_{y}}\nu).
\end{equation}
Writing $\nu=v+\overline{v}$, with $v$ in the eigenspace of
$J_{S^{\perp}}$ associated to the eigenvalue $i$, and expanding
\eqref{eq:cJcmestrevcompl} out, we conclude that $\nu$ is
holomorphic if and only if
$\mathcal{D}^{0,1}v+\mathcal{D}^{1,0}\overline{v}=0$, or,
equivalently, $\mathcal{D}^{0,1}v=0$.

Now fix a non-zero $v_{\infty}\in\R^{5,1}$. Let $v_{\infty}^{\perp}$
be the orthogonal projection of $v_{\infty}$ onto $S^{\perp}$.  In
view of the reality of $v_{\infty}^{\perp}$, write
$v_{\infty}^{\perp}=v+\overline{v}$, with $v$ in the eigenspace of
$J_{S^{\perp}}$ associated to the eigenvalue $i$. Consider the
surface $\sigma _{\infty}:M\rightarrow S_{v_{\infty}}$, in the
space-form $S_{v_{\infty}}$, defined by $\Lambda$. Under the
isomorphism $\mathcal{Q}:N_{\infty}\rightarrow S^{\perp}$,
preserving connections, defined in Section \ref{normalbundle}, the
complex structure on $S^{\perp}$ induces naturally a complex
structure on $N_{\infty}$, preserving holomorphicity. We say that
$\Lambda$ has \textit{holomorphic mean curvature vector} in the
space-form $S_{v_{\infty}}$ if $\mathcal{H}_{\infty}$ is
holomorphic. By equation \eqref{eq:piperpvinfisQHinf}, it follows
that $\Lambda$ is a holomorphic mean curvature vector surface in
$S_{v_{\infty}}$ if and only if
\begin{equation}\label{eq:curlyDvinftyperp}
\mathcal{D}^{0,1}v_{\infty}^{\perp}=0,
\end{equation}
or, equivalently, $\mathcal{D}^{0,1}v=0$.

Now suppose $\Lambda$ has holomorphic mean curvature vector in
$S_{v_{\infty}}$. Define a real form
$q\in\Omega^{1}(\Lambda\wedge\Lambda^{(1)})$, with
$q^{1,0}\in\Omega^{1,0}(\Lambda\wedge\Lambda^{0,1})$, by setting
$q^{1,0}v_{\infty}^{T}:=-\mathcal{N}^{1,0}v$, for $v_{\infty}^{T}$
the orthogonal projection of $v_{\infty}$ onto $S$  (cf. Remark
\ref{CQeqdeterminesq}). Set
$p(\lambda):=\lambda^{-1}v+v_{\infty}^{T}+\lambda\overline{v}$.
According to Theorem \ref{charactdeCQ}, having in consideration
Remark \ref{qCWnotinterveneonproofrenm},
$d^{\lambda,q}_{S}p(\lambda)=0$ and, consequently,
\begin{equation}\label{eq:Rlambdaplambda0}
R^{\lambda}p(\lambda)=0,
\end{equation}
for $R^{\lambda}$ the curvature tensor of $d^{\lambda,q}_{S}$, for
all $\lambda\in\C\backslash\{0\}$. In the proof of Theorem
\ref{thm5.2}, we observed, in particular, that
$$R^{\lambda}=\frac{\lambda
^{-1}-\lambda}{2}\,i\,(d^{\mathcal{D}}*\mathcal{N}-2[q\wedge
*\mathcal{N}]) +(\lambda ^{-2}-1)\,d^{\mathcal{D}}q^{1,0}+(\lambda
^{2}-1)\,d^{\mathcal{D}}q^{0,1},$$ having in consideration that $
[q\wedge q]$ vanishes, cf. \eqref{eq:qq=0}. Note that, as
$qS^{\perp}=0$ and $S^{\perp}$ is $\mathcal{D}$-parallel,
$d^{\mathcal{D}}q^{1,0}S^{\perp}=0=d^{\mathcal{D}}q^{0,1}S^{\perp}$.
According to
\eqref{eq:cdfghyjuklguyeifrojhgfdxmnbvcxcghyrd5tr3276y348iu549549886}
and \eqref{eq:11111111qwertyuiopmjsj55s}, equation
\eqref{eq:Rlambdaplambda0} establishes
\begin{equation}\label{eq:nnahsdvvvvvlannnnnnnnnna}
\frac{\lambda^{-1}-\lambda}{2}\,(d^{\mathcal{D}}*\mathcal{N}-2[q\wedge
*\mathcal{N}])v_{\infty}^{T}=0
\end{equation}
and
$$0=\frac{i}{2}(d^{\mathcal{D}}*\mathcal{N}-2[q\wedge
*\mathcal{N}])((\lambda^{-2}-1)v+(1-\lambda^{2})\overline{v})
+(\lambda^{-2}-1)d^{\mathcal{D}}q^{1,0}v_{\infty}^{T}+(\lambda^{2}-1)d^{\mathcal{D}}q^{0,1}v_{\infty}^{T}.$$
Organizing the terms in equation \eqref{eq:nnahsdvvvvvlannnnnnnnnna}
by powers of $\lambda$, we conclude from the fact that equation
\eqref{eq:Rlambdaplambda0} holds for all
$\lambda\in\C\backslash\{0\}$ that, in particular,
\begin{equation}\label{eq:4444s911jhuh3u3uwesb}
d^{\mathcal{D}}*\mathcal{N}-2[q\wedge
*\mathcal{N}]v_{\infty}^{T}=0,
\end{equation}
or, equivalently, cf. Remark \ref{remmuyutilcurlyDNqCWeq},
$d^{\mathcal{D}}*\mathcal{N}=2[q\wedge *\mathcal{N}]$, in view of
the fact that
$(\sigma_{\infty},v_{\infty}^{T})=(\sigma_{\infty},v_{\infty})(=-1)$
is never-zero. Now we see that equation \eqref{eq:Rlambdaplambda0}
establishes, furthermore,
$$(\lambda^{-2}-1)\,d^{\mathcal{D}}q^{1,0}v_{\infty}^{T}+(\lambda^{2}-1)\,d^{\mathcal{D}}q^{0,1}v_{\infty}^{T}=0$$
and the fact that it does not depend on
$\lambda\in\C\backslash\{0\}$ establishes then
$d^{\mathcal{D}}q^{1,0}v_{\infty}^{T}=0$. Lastly, observe, in view
of \eqref{eq:curlyD1oLambdaijblablalbla}, that, as $q^{1,0}$ takes
values in $\Lambda\wedge\Lambda^{0,1}$, then so does
$d^{\mathcal{D}}q^{1,0}$,
$$d^{\mathcal{D}}q^{1,0}\in\Omega^{2}(\Lambda\wedge\Lambda^{0,1}),$$
to conclude, cf. Remark \ref{e10determinedby}, that
$d^{\mathcal{D}}q^{1,0}=0$ and, ultimately, by Lemma
\ref{withvswithoutdecomps}, that $d^{\mathcal{D}}q=0$. We conclude
that $\Lambda$ is a $q$-constrained Willmore surface admitting
$p(\lambda)$ as a $q$-conserved quantity.

Conversely, suppose $\Lambda$ is a constrained Willmore surface
admitting a conserved quantity
$p(\lambda)=\lambda^{-1}w+v_{0}+\lambda\overline{w}$ with
$v_{0}\in\Gamma(S)$ real and $w\in\Gamma(S^{\perp})$ null (and, in
particular, never-zero). In that case, $S^{\perp}=\langle
w\rangle\oplus \langle\overline{w}\rangle$ is the decomposition of
$S^{\perp}$ in \eqref{eq:SperpSpm}. Let $J_{w}$ be the
almost-complex structure on $S^{\perp}$ admitting $\langle w\rangle$
and $\langle \overline{w}\rangle$ as eigenspaces associated to the
eigenvalues $i$ and $-i$, respectively. According to the
characterization of conserved quantities presented in Theorem
\ref{charactdeCQ}, we conclude that $\Lambda$ has holomorphic mean
curvature vector in $S_{p(1)}$,  with respect to the complex
structure on $(S^{\perp},J_{w},\nabla^{S^{\perp}})$ determined by
Koszul-Malgrange Theorem.

\section{Spectral deformation of constrained Willmore surfaces with a conserved
quantity}

\markboth{\tiny{A. C. QUINTINO}}{\tiny{CONSTRAINED WILLMORE
SURFACES}}

The spectral deformation of constrained Willmore surfaces preserves
the existence of a conserved quantity, trivially:

\begin{thm}\label{specCWwcq}
Let $\mu$ be in $S^{1}$ and
$\phi^{\mu}_{q}:(\underline{\R}^{n+1,1},d^{\mu,q}_{S})\rightarrow
(\underline{\R}^{n+1,1},d)$ be an isomorphism. Suppose that either
$v_{0}$ is non-zero or $\overline{\mu}v+\mu\overline{v}$ is
non-zero. In that case, if $p(\lambda)$ is a $q$-conserved quantity
of $\Lambda$, then $\phi^{\mu}_{q}p(\mu\lambda)$ is a
$\mathrm{Ad}_{\phi^{\mu}_{q}}(q_{\mu})$-conserved quantity of the
spectral deformation $\phi^{\mu}_{q}\Lambda$ of parameter $\mu$ of
$\Lambda$.
\end{thm}
\begin{proof}
By hypothesis,
$$v_{\infty}^{\mu}:=\phi^{\mu}_{q}(\overline{\mu}v+v_{0}+\mu\overline{v})$$
is non-zero. On the other hand, as $\phi^{\mu}_{q}$ is real and
$\mu$ is unit, we have $\overline{\mu^{-1}\phi^{\mu}_{q}v}=\mu\,
\phi^{\mu}_{q}\overline{v}$. Having in consideration that
$\phi^{\mu}_{q}$ is an isometry, and, in particular,
$(\phi^{\mu}_{q}S)^{\perp}=\phi^{\mu}_{q}S^{\perp}$, we conclude
that
$$\phi^{\mu}_{q}p(\mu\lambda)=\lambda^{-1}(\mu^{-1}\phi^{\mu}_{q}v)+\phi^{\mu}_{q}v_{0}+\lambda(\mu\,\phi^{\mu}_{q}\overline{v})$$
is of the right form. The fact that
$\phi^{\mu}_{q}:(\underline{\R}^{n+1,1},d^{\mu,q}_{S})\rightarrow
(\underline{\R}^{n+1,1},d)$ preserves connections, and,
consequently,
$$d^{\lambda,\mathrm{Ad}_{\phi^{\mu}_{q}}(q_{\mu})}_{\phi^{\mu}_{q}S}=\phi^{\mu}_{q}\circ
(d^{\mu,q}_{S})^{\lambda,q_{\mu}}_{S}\circ
(\phi^{\mu}_{q})^{-1}=\phi^{\mu}_{q}\circ d^{\mu\lambda,q}_{S}\circ
(\phi^{\mu}_{q})^{-1},$$ completes the proof.
\end{proof}

\section{B\"{a}cklund transformation of constrained Willmore surfaces with a conserved
quantity}\label{BTCWcQ}

\markboth{\tiny{A. C. QUINTINO}}{\tiny{CONSTRAINED WILLMORE
SURFACES}}

B\"{a}cklund transformations of constrained Willmore surfaces
preserve the existence of a conserved quantity, in the following
terms:
\begin{thm}\label{BTsofCMCs}
Suppose $p(\lambda)$ is a $q$-conserved quantity of $\Lambda$. Let
$\alpha,L^{\alpha}$ be B\"{a}cklund transformation parameters to
$\Lambda$ corresponding to the multiplier $q$ and let $r^{*}$ denote
$r^{\alpha}_{L^{\alpha}}$. If
\begin{equation}\label{eq:palhopahaortogLalpha}
p(\alpha)\perp L^{\alpha},
\end{equation}
then
$$p^{*}(\lambda):=r^{*}(1)^{-1}\,r^{*}(\lambda)\,p(\lambda)$$
is a $q^{*}$-conserved quantity of the B\"{a}cklund transform
$\Lambda^{*}$ of $\Lambda$ of parameters $\alpha,L^{\alpha}$,
provided that $\Lambda^{*}$ immerses.
\end{thm}

\begin{proof}
Suppose $\Lambda^{*}$ immerses and let $S^{*}$ be its central sphere
congruence. First of all, note that
$$d^{\lambda,q^{*}}_{S^{*}}p^{*}(\lambda)=r^{*}(1)^{-1}\circ\hat{d}^{\lambda,\hat{q}}_{S}\circ
r^{*}(\lambda)\,p(\lambda)=r^{*}(1)^{-1}r^{*}(\lambda)\circ
d^{\lambda,q}_{S}\,p(\lambda)=0,$$ for
$\lambda\in\C\backslash\{0\}$. Let $\rho$ and $\rho^{*}$ denote,
respectively, $\rho_{S}$ and $\rho_{S^{*}}$. Consider projections
$\pi_{L^{\alpha}}:\underline{\C}^{n+2}\rightarrow L^{\alpha}$,
$\pi_{\rho L^{\alpha}}:\underline{\C}^{n+2}\rightarrow\rho
L^{\alpha}$ and $\pi_{\perp}:\underline{\C}^{n+2}\rightarrow
(L^{\alpha}\oplus\rho L^{\alpha})^{\perp}$ with respect to the
decomposition $\underline{\C}^{n+2}=L^{\alpha}\oplus \rho
L^{\alpha}\oplus (L^{\alpha}\oplus\rho L^{\alpha})^{\perp}$. As
$L^{\alpha}$ and $\rho L^{\alpha}$ are never orthogonal, condition
\eqref{eq:palhopahaortogLalpha} establishes, in particular,
$\pi_{\rho L^{\alpha}}p(\alpha)=0$. On the other hand, in view of
the specific form of $p(\lambda)$,
\begin{equation}\label{eq:rhopinf}
\rho p(\lambda)=p(-\lambda),
\end{equation}
for all $\lambda$. Hence $\pi_{L^{\alpha}}p(-\alpha)=\rho\pi_{\rho
L^{\alpha}}p(\alpha)$ and, therefore,
$\pi_{L^{\alpha}}p(-\alpha)=0$.  It follows that
$$p_{\alpha,L^{\alpha}}(\lambda)\,p(\lambda)=\frac{\alpha-\lambda}{\alpha+\lambda}\,\pi_{L^{\alpha}}p(\lambda)+\pi_{
\perp}p(\lambda)+\frac{\alpha+\lambda}{\alpha-\lambda}\,\pi_{\rho
L^{\alpha}}p(\lambda)$$ has no poles and, consequently, that
$$p^{*}(\lambda)=r^{*}(1)^{-1}K
q_{\overline{\alpha}\,^{-1},\hat{L}^{\overline{\alpha}\,^{-1}}}(\lambda)\,p_{\alpha,L_{\alpha}}(\lambda)\,p(\lambda)$$
has, at most, poles at $\lambda=\pm\overline{\alpha}\,^{-1}$.
Consider now projections
$\pi_{\overline{L^{\alpha}}}:\underline{\C}^{n+2}\rightarrow
\overline{L^{\alpha}}$ and $\pi_{\rho
\overline{L^{\alpha}}}:\underline{\C}^{n+2}\rightarrow \rho
\overline{L^{\alpha}}$ with respect to the decomposition
$\underline{\C}^{n+2}=\overline{L^{\alpha}}\oplus \rho
\overline{L^{\alpha}}\oplus (\overline{L^{\alpha}}\oplus\rho
\overline{L^{\alpha}})^{\perp}$. Given the specific form of
$p(\lambda)$, we have, on the other hand,
\begin{equation}\label{eq:pinfconjugation}
p(\overline{\lambda}\,^{-1})=\overline{p(\lambda)},
\end{equation}
for all $\lambda$, and, therefore,
$p(\overline{\alpha}\,^{-1})\in\Gamma(\overline{L^{\alpha}}\,^{\perp})$.
Thus
$\pi_{\overline{L^{\alpha}}}\,p(-\overline{\alpha}\,^{-1})=0=\pi_{\rho
\overline{L^{\alpha}}}\,p(\overline{\alpha}\,^{-1})$. It follows
that
$p^{*}(\lambda)=r^{*}(1)^{-1}\,p_{\alpha,\tilde{L}^{\alpha}}(\lambda)\,q_{\overline{\alpha}\,^{-1},\overline{L^{\alpha}}}\,(\lambda)\,p(\lambda)$
has, at most, poles at $\lambda=\pm\alpha$. We conclude that
$p^{*}(\lambda)$ has no poles. The fact that
$\mathrm{lim}_{\lambda\rightarrow
\infty}\,\lambda^{-1}p^{*}(\lambda)=r^{*}(1)^{-1}\,r^{*}(\infty)\,\overline{v}$
and $\mathrm{lim}_{\lambda\rightarrow 0}\,\lambda
p^{*}(\lambda)=r^{*}(1)^{-1}\,r^{*}(0)\,v$ are both finite
establishes then $p^{*}(\lambda)$ as a degree $1$ Laurent
polynomial. According to \eqref{eq:RcommuteRho},
$$\rho^{*}p^{*}(\lambda)=r^{*}(1)^{-1}\rho\, r^{*}(1)p^{*}(\lambda)=r^{*}(1)^{-1}\rho
\,r^{*}(\lambda)\,p(\lambda)=r^{*}(1)^{-1}r^{*}(-\lambda)\rho
p(\lambda)$$ and, therefore, following \eqref{eq:rhopinf},
$$\rho^{*}p^{*}(\lambda)=p^{*}(-\lambda),$$
showing that the coefficients on $\lambda$ and $\lambda^{-1}$ in
$p^{*}(\lambda)$ are sections of $(S^{*})^{\perp}$, whilst the
coefficient on $\lambda^{0}$ is a section of $S^{*}$. To complete
the proof, we are left to verify that
$p^{*}(\overline{\lambda}\,^{-1})=\overline{p^{*}(\lambda)}$,
equivalent to the complex conjugation conditions on the coefficients
in $p^{*}(\lambda)$. It comes as an immediate consequence of
equations \eqref{eq:koverlineinversedadada?} and
\eqref{eq:pinfconjugation}.
\end{proof}

\chapter{Constrained Willmore surfaces and isothermic
condition}\label{isoCW}

\markboth{\tiny{A. C. QUINTINO}}{\tiny{CONSTRAINED WILLMORE
SURFACES}}

A classical result of Thomsen \cite{thomsen} characterizes
isothermic Willmore surfaces in $3$-space as minimal surfaces in
some $3$-dimensional space-form. Constant mean curvature (CMC)
surfaces in $3$-dimensional space-forms are, in particular,
isothermic constrained Willmore surfaces, as proven by J. Richter
\cite{richter}. However, isothermic constrained Willmore surfaces in
$3$-space are not necessarily CMC surfaces in some space-form, as
proven by an example due to Fran Burstall and presented in
\cite{christoph2}, of a constrained Willmore cylinder that does not
have constant mean curvature in any space-form. We dedicate a
section to the very important class of CMC surfaces in $3$-space,
with constrained Willmore B\"{a}cklund transformations; both
constrained Willmore and isothermic spectral deformations; as well
as a spectral deformation of their own and, in the Euclidean case,
 isothermic Darboux transformations and Bianchi-B\"{a}cklund transformations.
S. Kobayashi and J.-I. Inoguchi \cite{kobayashi} proved that
isothermic Darboux transformation of a CMC surface in $\R^{3}$ is
equivalent to Bianchi-B\"{a}cklund transformation. We believe
isothermic Darboux transformation of a CMC surface in Euclidean
$3$-space can be obtained as a particular case of constrained
Willmore B\"{a}cklund transformation. This shall be the subject of
further work. We present the classical CMC spectral deformation by
means of the action of a loop of flat metric connections. We observe
that these three spectral deformations of CMC surfaces in $3$-space
are all closely related and, therefore, all closely related to
B\"{a}cklund transformation. We observe, in particular, that the
classical CMC spectral deformation can be obtained as composition of
isothermic and constrained Willmore spectral deformation and that,
in the particular case of minimal surfaces, the classical CMC
spectral deformation coincides with the constrained Willmore
spectral deformation corresponding to the zero multiplier. The
chapter starts with a section on the M\"{o}bius invariant class of
isothermic surfaces in space forms. We characterize isothermic
constrained Willmore surfaces by the non-uniqueness of multiplier.
The constrained Willmore spectral deformation is known to preserve
the isothermic condition, cf. \cite{SD}. As for B\"{a}cklund
transformation of constrained Willmore surfaces, we believe it does
not necessarily preserve the isothermic condition. We believe one
can obtain non-isothermic, non-Willmore constrained Willmore
surfaces as B\"{a}cklund transforms of non-minimal CMC surfaces in
space-forms. This shall be the subject of further work.
\newline

Throughout this chapter, let $\Lambda\subset\underline{\R}^{n+1,1}$
be a surface in the projectivized light-cone and $S$ be the central
sphere congruence of $\Lambda$. Consider $M$ provided with the
conformal structure $\mathcal{C}_{\Lambda}$.

\section{Isothermic surfaces}\label{Hopfeisothermal}

\markboth{\tiny{A. C. QUINTINO}}{\tiny{CONSTRAINED WILLMORE
SURFACES}}

It seems that the notion of isothermal lines, tracing back to the
early nineteenth century, was motivated by their physical
interpretation as lines of equal temperature, having led to the
notion of isothermic surfaces, that is, surfaces with isothermal
lines of curvature. This section is dedicated to the study of
isothermic surfaces merely from the point of view of constrained
Willmore surfaces. Classically, a surface in $\R^{3}$ is isothermic
if it admits conformal coordinate line coordinates at every point.
F. Burstall and U. Pinkall \cite{FPink} extended the isothermic
condition to surfaces in space forms, with a manifestly conformally
formulation, characterizing isothermic surfaces in the conformal
$n$-sphere by the existence of a non-zero real closed $1$-form
$\eta$ with values in a certain subbundle of the skew-symmetric
endomorphisms of $\underline{\R}^{n+1,1}$. We characterize
isothermic constrained Willmore surfaces by the non-uniqueness of
multiplier and establish the set of multipliers to an isothermic
$q$-constrained Willmore surface $(\Lambda,\eta)$ as the
$1$-dimensional affine space $q+\langle *\eta\rangle_{\R}$. The
constrained Willmore spectral deformation is known to preserve the
isothermic condition, cf. \cite{SD}. We derive it in our setting. As
for B\"{a}cklund transformation of constrained Willmore surfaces, it
is not clear that the isothermic condition is preserved. Isothermic
surfaces in $\R^{3}$ were studied intensively at the turn of the
20th century and a rich transformation theory of these surfaces was
developed in the works of Darboux \cite{darboux}, Calapso
\cite{calapso}, \cite{calapso2} and Bianchi \cite{bianchi},
\cite{bianchi2}. The loop group formalism provides a context in
which the results of Bianchi, Calapso and Darboux can be
generalized. Following the work of F. Burstall, D. Calderbank and U.
Pinkall \cite{burstall+calderbank}, \cite{FPink}, we characterize
isothermic surfaces by the flatness of a certain $\R$-family of
metric connections on $\underline{\R}^{n+1,1}$ and define, in terms
of this family of connections, both the isothermic spectral
deformation, discovered in the classical setting by Calapso and,
independently, by Bianchi; and the isothermic Darboux
transformation.\newline

\subsection{Isothermic surfaces: definition}
In this section, we present a manifestly conformally invariant
formulation of the isothermic condition, by F. Burstall and U.
Pinkall.\newline

Isothermic surfaces are classically defined to be immersions
$f:(M,g_{f})\rightarrow \R^{3}$ admitting, at every point, conformal
curvature line coordinates, i.e., conformal coordinates along the
principal directions. Equation \eqref{eq:secffvsshapeop}, relating
the shape operator to the second fundamental form of an isometric
immersion, makes clear that conformal coordinates $x$ and $y$ are
curvature line coordinates if and only if
\begin{equation}\label{eq:Pidiagonal}
\Pi(\delta_{x},\delta_{y})=0.
\end{equation}
In fact, the conformality of $x$ and $y$ ensures the existence of
some $u\in C^{\infty}(M,\R)$ for which
$$g_{f}=e^{u}(dx^{2}+dy^{2}),$$  establishing, in particular, that
$g_{f}(\delta_{x},\delta_{y})=0$, so that, if
$A^{\xi}\delta_{x}\in\langle\delta_{x}\rangle$, for either unit
$\xi\in\Gamma(N_{f})$, then equation \eqref{eq:Pidiagonal} is
established. It is obvious that, conversely, equation
\eqref{eq:Pidiagonal} forces, in particular,
$$A^{\xi}\delta_{x}\in\langle\delta_{x}\rangle,\,\,\,A^{\xi}\delta_{y}\in\langle\delta_{y}\rangle,$$
for either unit normal vector field $\xi$ to $f$. Classical
isothermic surfaces extend naturally to immersions
$f:(M,g_{f})\rightarrow \bar{M}$, into a Riemannian manifold
$\bar{M}$, admitting, at every point, conformal coordinates which
diagonalize the second fundamental form,\footnote{Or, equivalently,
which diagonalize simultaneously all shape operators (the
verification presented above for the particular case
$\bar{M}=\R^{3}$ clearly holds for general $\bar{M}$).} which we
still refer to as conformal curvature line coordinates. Equation
\eqref{eq:howPichanges} makes clear that the isothermic condition is
a conformal invariant (even though the second fundamental form is
not), having in consideration that under a conformal change of
metric in $\bar{M}$, the metric induced in $M$ changes conformally.
In fact, it makes clear, furthermore, that conformal curvature line
coordinates are preserved under conformal changes of the metric.
Hence, as very well-known:
\begin{thm}\label{isounderconfdiff}
Conformal curvature line coordinates are preserved by conformal
diffeomorphisms.
\end{thm}

Theorem \ref{isounderconfdiff} establishes, in particular, the
M\"{o}bius invariance of the class of isothermic surfaces. We define
$\Lambda:(M,\mathcal{C}_{\Lambda})\rightarrow
(\mathbb{P}(\mathcal{L}),\mathcal{C}_{\mathbb{P}(\mathcal{L})})$
 to be an \textit{isothermic surface} if, fixing
$h\in\mathcal{C}_{\mathbb{P}(\mathcal{L})}$ (independently of the
choice of $h$), $\Lambda:(M,\Lambda^{*}h)\rightarrow
(\mathbb{P}(\mathcal{L}),h)$ is an isothermic surface, with
$\Lambda^{*}h$ denoting the metric induced in $M$ by $\Lambda$ from
$h$. We formulate it next following a manifestly conformally
invariant characterization of isothermic surfaces established by F.
Burstall and U. Pinkall.

\begin{defn}\label{isocharac}
$\Lambda$ is said to be an \emph{isothermic surface} if there exists
a non-zero closed real $1$-form $\eta$ with values in
$\Lambda\wedge\Lambda^{\perp}$.
\end{defn}

This formulation of isothermic surfaces is discussed in
\cite{FPink}, \cite{burstall+calderbank}, \cite{susana} and
\cite{udo}. For the relationship between this formulation and the
classical one, see \cite{FPink}, \cite{susana} and \cite{udo} (see,
in particular, \S$5.3.19$).\footnote{Without wishing to go into
detail, it is worth remarking on this relationship. (For more
details, see \cite{IS} and \cite{susana}.) Given
$f,f^{c}:M\rightarrow \R^{n}$ immersions of $M$ in Euclidean
$n$-space, $f$ and $f^{c}$ are said to be \textit{Christoffel
transforms} of each other, or \textit{dual isothermic surfaces}, if
$f$ and $f^{c}$ have parallel tangent planes, i.e.,
$df(TM)=df^{c}(TM)$; $f$ and $f^{c}$ induce the same conformal
structure on $M$; and $f$ and $f^{c}$ induce opposite orientations
on $M$, i.e., $df^{c}\circ df^{-1}:df(TM)\rightarrow df(TM)$ has
negative determinant. A result by Christoffel \cite{christoffel},
for $n=3$, and B. Palmer \cite{palmer}, for arbitrary n,
characterizes isothermic surfaces immersed in $\R^{n}$ by the
existence of a dual isothermic surface. Since $n\geq 3$, we can
choose $v_{\infty}\in\mathcal{L}$ such that $\Lambda_{p}\neq \langle
v_{\infty}\rangle $, for all $p\in M$, and to define then a surface
$\sigma_{0,\infty}$ in Euclidean $n$-space by stereographic
projection of pole $x_{0}\in S^{n}$ of $\Lambda _{p}\in
\mathbb{P}(\mathcal{L})\backslash\{\langle v_{\infty}\rangle\}=
S^{n}\backslash\{x_{0}\}$, for each $p\in M$. According to Theorem
\ref{isounderconfdiff}, $\Lambda$ is isothermic if and only if so is
$\sigma_{0,\infty}$. Let $\sigma_{\infty}$ be the surface  defined
by $\Lambda$ in $S_{v_{\infty}}$. One verifies that, if
$\sigma_{0,\infty}$ is isothermic and $\sigma_{0,\infty}^{c}$ is a
dual isothermic surface to $\sigma_{0,\infty}$, then
$$\eta:=\sigma_{\infty}\wedge
(d\sigma_{0,\infty}^{c}+(\sigma_{0,\infty},d\sigma_{0,\infty}^{c})v_{\infty})$$
is a form in the conditions of Definition \ref{isocharac}; and,
conversely, that the existence of a form $\eta$ such that
$(\Lambda,\eta)$ is isothermic establishes the existence of a dual
isothermic surface to $\sigma_{0,\infty}$.}

Constant mean curvature surfaces in $3$-space are well-known
examples of isothermic surfaces (see Section \ref{sec:CMC}), as well
as surfaces of revolution, cones and cylinders (see, for example,
\cite{susana}).

Under the conditions of Definition \ref{isocharac}, we may,
alternatively, refer to the isothermic surface $\Lambda$ as the pair
$(\Lambda,\eta)$. As $\Lambda$ is not contained in any $2$-sphere
(cf. \eqref{eq:notinsmallerspheres}), such $\eta$ is unique up to
non-zero constant real scale, cf. \cite{susana} (see, in particular,
Proposition $1.25$). It is very useful to know that, cf.
\cite{susana} (see, in particular, Proposition $1.11$):
\begin{Lemma}\label{etainLwedgeL1}
If $(\Lambda,\eta)$ is an isothermic surface, then
$\eta\in\Omega^{1}(\Lambda\wedge\Lambda^{(1)})$.
\end{Lemma}

It follows, in particular, that:
\begin{Lemma}\label{etadetcvansisheszesesseparately}
If $(\Lambda,\eta)$ is an isothermic surface, then
\begin{equation}\label{eq:etaisoeqs}
d^{\mathcal{D}}\eta=0=[\mathcal{N}\wedge\eta].
\end{equation}
\end{Lemma}
\begin{proof}
Suppose $(\Lambda,\eta)$ is isothermic. Then, in particular, the
form $\eta$ is closed,
$d^{\mathcal{D}}\eta+[\mathcal{N}\wedge\eta]=0$. On the other hand,
according to Lemma \ref{etainLwedgeL1},
$\eta\in\Omega^{1}(\wedge^{2}S)$. Hence, by the
$\mathcal{D}$-parallelness of $S$ and $S^{\perp}$,
$d^{\mathcal{D}}\eta\in\Omega^{2}(\wedge^{2}S\oplus\wedge^{2}S^{\perp})$;
whereas, as $\mathcal{N}$ takes values in $S\wedge S^{\perp}$,
$[\mathcal{N}\wedge\eta]\in\Omega^{2}(S\wedge S^{\perp})$. We
conclude that $d^{\mathcal{D}}\eta$ and $[\mathcal{N}\wedge\eta]$
vanish separately.
\end{proof}

\begin{rem}\label{etawedgeta0}
Following equation \eqref{eq:wegdebrack}, we have
\begin{equation}\label{eq:lambdalambdaperpwedgemesmo0}
[(\Lambda\wedge\Lambda^{\perp})\wedge(\Lambda\wedge\Lambda^{\perp})]\subset\Lambda\wedge\Lambda=\{0\}
\end{equation}
and, therefore, given
$\eta\in\Omega^{1}(\Lambda\wedge\Lambda^{\perp})$,
\begin{equation}\label{eq:etawedgeeta0}
[\eta\wedge\eta]=0.
\end{equation}
\end{rem}

Theorem \ref{isounderconfdiff} establishes, in particular:

\begin{thm}\label{isoLambequivisosigmainf}
$\Lambda$ is isothermic if and only if, fixing
$v_{\infty}\in\R^{n+1,1}$ non-zero, so is the surface
$\sigma_{\infty}:M\rightarrow S_{v_{\infty}}$, in the space-form
$S_{v_{\infty}}$, defined by $\Lambda$. Furthermore: $\Lambda$
shares conformal curvature line coordinates with
$\sigma_{v_{\infty}}$, for all $v_{\infty}$.
\end{thm}

\subsection{Isothermic condition and Hopf
differential}\label{isoandhopf}

The Hopf differential is closely related to
$\Pi^{(2,0)}:=\Pi\vert_{T^{1,0}M\times T^{1,0}M}$, giving rise to
yet another characterization of isothermic surfaces in space-forms,
which we present in this section.\newline

Fix a non-zero $v_{\infty}$ in $\mathbb{R}^{n+1,1}$ and consider the
surface $\sigma _{\infty}:M\rightarrow S_{v_{\infty}}$, in the
space-form $S_{v_{\infty}}$, defined by $\Lambda$. Recall that the
pull-back bundle by $\sigma_{\infty}$ of the tangent bundle
$TS_{v_{\infty}}$ consists of the orthogonal complement in
$\underline{\R}^{n+1,1}$ of the non-degenerate bundle
$\langle\sigma_{\infty},v_{\infty}\rangle$,
$\sigma_{\infty}^{*}TS_{v_{\infty}}=\langle\sigma_{\infty},v_{\infty}\rangle^{\perp}$.
Let $\pi _{N_{\infty}}$ denote the orthogonal projection of
$\underline {\mathbb{R}}^{n+1,1}=d\sigma _{\infty}(TM)\oplus N
_{\infty}\oplus \langle v_{\infty},\sigma _{\infty}\rangle $ onto
$N_{\infty}$. Fix a holomorphic chart $z=x+iy$  of
$(M,\mathcal{C}_{\Lambda})$. Observe that
$$(\sigma_{\infty})_{zz}\in\Gamma(\sigma_{\infty}^{*}TS_{v_{\infty}}).$$
Indeed, differentiation of $(\sigma_{\infty},v_{\infty})=-1$ shows
that $((\sigma_{\infty})_{z},v_{\infty})=0$ and, consequently,
$((\sigma_{\infty})_{zz},v_{\infty})=0$; whereas differentiation of
$(\sigma_{\infty},(\sigma_{\infty})_{z})=0$ shows that
$(\sigma_{\infty})_{z}$ is orthogonal to $\sigma_{\infty}$. It
follows that
$(\sigma_{\infty})_{zz}-\pi_{N_{\infty}}(\sigma_{\infty})_{zz}\in\Gamma(d\sigma_{\infty}(TM))\subset\Gamma(S)$
and, therefore,
$(\sigma_{\infty})_{zz}-\pi_{N_{\infty}}(\sigma_{\infty})_{zz}-(\pi_{N_{\infty}}(\sigma_{\infty})_{zz},\mathcal{H}_{\infty})\sigma_{\infty}\in\Gamma(S)$.
We conclude that
\begin{equation}\label{piSperppiNinf}
\pi_{S^{\perp}}(\sigma_{\infty})_{zz}=\mathcal{Q}(\pi_{N_{\infty}}(\sigma_{\infty})_{zz}),
\end{equation}
for the isomorphism $\mathcal{Q}:N_{\infty}\rightarrow S^{\perp}$
defined in Section \ref{normalbundle}. Now write
$\sigma^{z}=\lambda\sigma_{\infty}$ with
$\lambda\in\Gamma(\underline{\R})$ never-zero. Then
$\sigma_{zz}^{z}=\lambda(\sigma_{\infty})_{zz}+2\lambda
_{z}(\sigma_{\infty})_{z}+\lambda_{zz}\sigma_{\infty}$ and,
therefore,
\begin{equation}\label{eq:kzpiSperpsigmainfzz}
k^{z}=\lambda\,\pi_{S^{\perp}}(\sigma_{\infty})_{zz}.
\end{equation}
It follows that, under the isomorphism $\mathcal{Q}$, the Hopf
differential $k^{z}$ is a real scale of
\begin{equation}\label{eq:ggbaddstbl?anhtq1534ascdu7qdfvgbnm}
\pi_{N_{\infty}}(\sigma_{\infty})_{zz}=\Pi_{\infty}(\delta_{z},\delta_{z})=\frac{1}{4}\,(\Pi_{\infty}(\delta_{x},\delta_{x})-2i\Pi_{\infty}(\delta_{x},\delta_{y})-\Pi_{\infty}(\delta_{y},\delta_{y})).
\end{equation}
We are led to the following characterization of isothermic surfaces
in terms of the Hopf differential, presented in \cite{SD}:
\begin{Lemma}\label{isoviakreal}
The surface $\Lambda$ is isothermic if and only if around each point
there exists a holomorphic chart of $(M,\mathcal{C}_{\Lambda})$ with
respect to which the Hopf differential of $\Lambda$ is a real
section of $S^{\perp}$. Furthermore: the conformal coordinates $x,y$
are curvature line coordinates to $\Lambda$ if and only if $k^{z}$
is real.
\end{Lemma}

\begin{proof}
The conformal coordinates  $x,y$ are curvature line coordinates of
$\sigma_{\infty}$ if and only if
$\Pi_{\infty}(\delta_{x},\delta_{y})=0$, or, equivalently,
$\pi_{N_{\infty}}(\sigma_{\infty})_{zz}$ is real. Since
$\mathcal{Q}$ is an isomorphism of real bundles, the reality of
$\pi_{N_{\infty}}(\sigma_{\infty})_{zz}$ is equivalent to that of
$k^{z}$.
\end{proof}

\subsection{Transformations of isothermic surfaces}\label{isotransform}

Isothermic surfaces in $\R^{3}$ were studied intensively at the turn
of the 20th century and a rich transformation theory of these
surfaces was developed in the works of Darboux \cite{darboux},
Calapso \cite{calapso}, \cite{calapso2} and Bianchi \cite{bianchi},
\cite{bianchi2}. The loop group formalism provides a context in
which the results of Bianchi, Calapso and Darboux can be
generalized. Following the work of F. Burstall, D. Calderbank and U.
Pinkall \cite{burstall+calderbank}, \cite{FPink}, we characterize
isothermic surfaces by the flatness of a certain $\R$-family of
metric connections on $\underline{\R}^{n+1,1}$ and define, in terms
of this family of connections, both the isothermic spectral
deformation, discovered in the classical setting by Calapso and,
independently, by Bianchi; and the isothermic Darboux
transformation.\newline

Let $\eta$ be a non-zero real $1$-form with values in
$\Lambda\wedge\Lambda^{\perp}$. For each $t\in\R$, set
$$d^{t}_{\eta}:=d+t\eta,$$defining a connection of
$\underline{\C}^{n+2}$. The reality of $\eta$ establishes that of
$d^{t}_{\eta}$, whereas its skew-symmetry establishes $d^{t}_{\eta}$
as a metric connection. As established in \cite{burstall+calderbank}
and \cite{FPink}:

\begin{thm}
$(\Lambda,\eta)$ is isothermic if and only if $d^{t}_{\eta}$ is a
flat connection, for each $t\in\R$.
\end{thm}
The proof is immediate, but worth presenting:
\begin{proof}
The curvature tensor $R^{t}$ of $d^{t}_{\eta}$ is given by
$R^{t}=R^{d}+td\eta+\frac{t^{2}}{2}\,[\eta\wedge \eta]=0$. Equation
\eqref{eq:etawedgeeta0} makes then clear that $\eta$ is closed if
and only if $R^{t}=0$, for all $t$.
\end{proof}

One shall be aware of the ambiguity that the notation $d^{t}_{\eta}$
carries, for $t=\pm 1$, with respect to the constrained Willmore
spectral deformation of parameter $t$, corresponding to the
multiplier $\eta$, in the case $(\Lambda,\eta)$ is an isothermic
surface admitting $\eta$ as a multiplier. $\newline$

\textbf{The isothermic spectral deformation.} As we verify next, if
$(\Lambda,\eta)$ is isothermic, then so is the transformation of
$\Lambda$ defined by the flat metric connection $d^{t}_{\eta}$, for
each $t\in\R$. Associated to an isothermic surface, we have a
one-parameter family of isothermic surfaces, discovered in the
classical setting by Calapso \cite{calapso}, \cite{calapso2} and,
independently, by Bianchi \cite{bianchi}, \cite{bianchi2}.\newline

Suppose $(\Lambda,\eta)$ is isothermic, so that, in particular,
$d^{t}_{\eta}$ is a flat metric connection on
$\underline{\R}^{n+1,1}$, for each $t\in\R$. Given $t\in\R$ and
$\sigma\in\Gamma(\Lambda)$,
\begin{equation}\label{eq:dtetasigmaigualdsigma}
d^{t}_{\eta}\sigma=d\sigma,
\end{equation}
showing that $\Lambda$ is still a $d^{t}_{\eta}$-surface, or,
equivalently, the transformation $\Lambda^{t}_{\eta}$ of $\Lambda$
defined by the connection $d^{t}_{\eta}$ is still a surface.
Furthermore:
\begin{thm}
Let $(\Lambda,\eta)$ be an isothermic surface. Then, for each
$t\in\R$, the transformation $\Lambda^{t}_{\eta}$ of $\Lambda$
defined by the flat metric connection $d^{t}_{\eta}$ is still
isothermic.
\end{thm}

Fix $t\in\R$ and
$\phi^{t}_{\eta}:(\underline{\R}^{n+1,1},d^{t}_{\eta})\rightarrow(\underline{\R}^{n+1,1},d)$
an isomorphism. The proof of the theorem will consist of showing
that $(\Lambda^{t}_{\eta},\mathrm{Ad}_{\phi^{t}_{\eta}}\eta)$ is
isothermic.

\begin{proof}
If $\eta=\sigma\wedge\mu$, with $\sigma\in\Gamma(\Lambda)$ and
$\mu\in\Omega^{1}(\Lambda^{\perp})$, then, recalling
\eqref{eq:adjwedge},
$\mathrm{Ad}_{\phi^{t}_{\eta}}\eta=\phi^{t}_{\eta}\sigma\wedge\phi^{t}_{\eta}\mu$
is a non-zero real $1$-form with values in
$\Omega^{1}((\phi^{t}_{\eta}\Lambda)\wedge
(\phi^{t}_{\eta}\Lambda)^{\perp})$ and
$$d(\mathrm{Ad}_{\phi^{t}_{\eta}}\eta)=\phi^{t}_{\eta}\circ d^{d^{t}_{\eta}}\eta\circ (\phi^{t}_{\eta})^{-1}=\phi^{t}_{\eta}\circ (d\eta+t[\eta\wedge\eta])\circ (\phi^{t}_{\eta})^{-1}=0.$$
\end{proof}
We may refer to $\Lambda^{t}_{\eta}$ as the \textit{isothermic
$(t,\eta)$-transformation} of $\Lambda$. Note that, if
$(\Lambda,\eta')$ is isothermic, for some non-zero real $1$-form
$\eta'$, with values in $\Lambda\wedge\Lambda^{\perp}$, then
$\eta'=t_{\eta}\eta$, for some $t_{\eta}\in\R$ and the isothermic
$(t,\eta')$-transformation of $\Lambda$ coincides with the
$(t\,t_{\eta},\eta)$-transformation.

Observe that, given $t'\in\R$ and
$\phi^{t'}_{\mathrm{Ad}_{\phi^{t}_{\eta}}\eta}:(\underline{\R}^{n+1,1},d^{t'}_{\mathrm{Ad}_{\phi^{t}_{\eta}}\eta})\rightarrow(\underline{\R}^{n+1,1},d)$
an isomorphism,
$\phi^{t'}_{\mathrm{Ad}_{\phi^{t}_{\eta}}\eta}\phi^{t}_{\eta}:(\underline{\R}^{n+1,1},d^{t+t'}_{\eta})\rightarrow(\underline{\R}^{n+1,1},d)$
is an isomorphism, to conclude that
$$\Lambda^{t+t'}_{\eta}=(\Lambda^{t}_{\eta})^{t'}_{\mathrm{Ad}_{\phi^{t}_{\eta}}\eta},$$
we have a one-parameter family of isothermic surfaces. In the
particular case of Euclidean $3$-space, this is the T-transform,
found by Calapso \cite{calapso}, \cite{calapso2} and, independently,
by Bianchi \cite{bianchi}, \cite{bianchi2}.

We complete this section by verifying that, up to reparametrization,
this isothermic spectral deformation coincides with the one
presented in \cite{SD}.\footnote{The omission, in \cite{SD}, of
reference to the transformation rules of the Hopf differential and
of the normal connection shall be understood as preservation.} First
of all, note that, in view of \eqref{eq:dtetasigmaigualdsigma},
given $\sigma\in\Gamma(\Lambda)$ never-zero,
\begin{equation}\label{eq:gdeisotharsnf}
g_{\phi^{t}_{\eta}\sigma}=g_{\sigma}^{d^{t}_{\eta}}=g_{\sigma},
\end{equation}
showing that the deformation defined by $d^{t}_{\eta}$ preserves the
conformal structure,
$$\mathcal{C}_{\Lambda^{t}_{\eta}}=\mathcal{C}_{\Lambda}.$$ It
preserves the central sphere congruence, as well. Indeed, fixing a
holomorphic chart $z$ of
$(M,\mathcal{C}_{\Lambda^{t}_{\eta}})=(M,\mathcal{C}_{\Lambda})$, we
have
$(d^{t}_{\eta})_{\delta_{\bar{z}}}(d^{t}_{\eta})_{\delta_{z}}\sigma=(\sigma)_{z\bar{z}}+
t\eta_{\delta_{\bar{z}}}\sigma_{z}=\sigma_{z\bar{z}}\,\mathrm{mod}\Lambda$,
in view of equation \eqref{eq:dtetasigmaigualdsigma}, showing that
$S^{d^{t}_{\eta}}=S$ and, ultimately, according to \eqref{eq:csc},
that
\begin{equation}\label{eq:cscpreservedbyisospecdeform}
S_{\phi^{t}_{\eta}\Lambda}=\phi^{t}_{\eta}S.
\end{equation}
According to \eqref{eq:gdeisotharsnf}, on the other hand,
$g_{\phi_{\eta}^{t}\sigma^{z}}=g_{\sigma^{z}}=g_{z}$, showing that
$\phi_{\eta}^{t}\sigma^{z}$ is the normalized section of
$\phi_{\eta}^{t}\Lambda$ with respect to $z$. Note that, as $\eta$
takes values in $\Lambda\wedge\Lambda^{\perp}$,
$\eta\Lambda^{\perp}$ takes values in $\Lambda$, and define
$\eta^{z}\in\C^{\infty}(M,\R)$ by
$\eta_{\delta_{z}}\sigma_{z}^{z}=\eta^{z}\sigma^{z}$. Then
$$(\phi_{\eta}^{t}\sigma^{z})_{zz}=(\phi_{\eta}^{t}\sigma_{z}^{z})_{z}=
\phi_{\eta}^{t}(\sigma_{zz}^{z}+t\eta_{\delta_{z}}\sigma_{z}^{z})=-\frac{1}{2}\,(c^{z}-2t\eta^{z})\,\phi_{\eta}^{t}\sigma^{z}+\phi_{\eta}^{t}k^{z}.$$
We conclude that $k^{z}_{t}$ and $c^{z}_{t}$, the Hopf differential
and the Schwarzian derivative, respectively, of
$\phi_{\eta}^{t}\Lambda$ with respect to $z$, relate to those of
$\Lambda$ by
$$k^{z}_{t}=\phi_{\eta}^{t}k^{z},\,\,\,\,\,\,c^{z}_{t}=c^{z}-2t\eta^{z}.$$
By Lemma \ref{thmonHopfeSchwarz}, having in consideration
\eqref{eq:cscpreservedbyisospecdeform}, the conclusion follows.
$\newline$

\textbf{Isothermic Darboux transformation.} Darboux \cite{darboux}
discovered a transformation of isothermic surfaces in $\R^{3}$: the
surface and its Darboux transform are characterized by being
conformal and curvature line preserving and enveloping some
$2$-sphere congruence. In this section, we present a manifestly
conformally invariant formulation of Darboux transforms of
isothermic surfaces, due to F. Burstall and U. Pinkall, in terms of
the family of flat metric connections, presented above,
characterizing the isothermic condition. For further reference, we
make a very brief description of the Darboux transformation of
isothermic surfaces in Euclidean $n$-space via solutions of a
Ricatti equation, presented in \cite{IS} as a direct extension of
the case $n=3,4$, discovered by Hertrich-Jeromin$-$Pedit
\cite{udo+pedit}.\newline

In \cite{udo+pedit}, U. Hertrich-Jeromin and F. Pedit develop
isothermic surface theory in Euclidean $n$-space, for $n=3,4$. In
particular, they define Darboux transformation, based on the
solution of a Ricatti equation, that, when restricted to codimension
$1$, becomes classical. In \cite{IS}, F. Burstall presents a direct
extension to general $n$ of this approach to isothermic surfaces and
Darboux transformation. For further reference, we describe it here
very briefly (for more details, see \cite{IS}). The starting point
is the fact that an immersion $f:M\rightarrow \R^{n}$ is isothermic
if and only if there exists another immersion
$f^{c}:M\rightarrow\R^{n}$ such that\footnote{Equation
\eqref{eq:dfdfc0} characterizes $f^{c}$ as a Christoffel transform
of $f$.}
\begin{equation}\label{eq:dfdfc0}
df\wedge df^{c}=0,
\end{equation}
where we multiply the coefficients of these $\R^{n}$-valued
$1$-forms using the product of the Clifford algebra
$\mathcal{C}l_{n}$ of $\R^{n}$. Equation \eqref{eq:dfdfc0} is the
integrability condition for a Ricatti equation involving an
auxiliary parameter $r\in\R\backslash\{0\}$:
$$dg=rgdf^{c}g-df,$$where again all multiplications take place in
$\mathcal{C}l_{n}$. We construct a new isothermic surface $\hat{f}$
by setting
$$\hat{f}=f+g$$ and verify that, just as in the classical case, $f$
and $\hat{f}$ are characterized by the conditions that they have the
same conformal structure and curvature lines and are enveloping
surfaces of a $2$-sphere congruence. We refer to $\hat{f}$ as the
\textit{Darboux transform of $f$ of parameters $r,g$}.

F. Burstall and U. Pinkall generalized the Darboux transformation to
isothermic surfaces in general space-form as follows, in a
manifestly conformally invariant formulation:

\begin{defn} Let $(\Lambda,\eta)$ be an isothermic surface. A surface
$\hat{\Lambda}$ is an \emph{isothermic Darboux transform} of
$\Lambda$ if $\hat{\Lambda}\cap\Lambda=\{0\}$ and there is a
non-zero real constant $m$ for which $\hat{\Lambda}$ is
$(d+m\eta)$-parallel.
\end{defn}

This formulation is discussed in \cite{FPink}, \cite{udo} (see, in
particular, \S$5.4.8$), \cite{susana} and
\cite{burstall+calderbank}. For the relationship between this
approach to isothermic Darboux transformation and the classical one,
see \cite{susana}.

\subsection{Isothermic condition and uniqueness of
multiplier}\label{nonisoCW}

In this section, we characterize isothermic constrained Willmore
surfaces in space-forms by the non-uniqueness of multiplier and
establish the set of multipliers to an isothermic $q$-constrained
Willmore surface $(\Lambda,\eta)$ as the $1$-dimensional affine
space $q+\langle *\eta\rangle_{\R}$.\newline

\begin{thm}\label{uniquenessofq}
Suppose $\Lambda$ is a constrained Willmore surface. The uniqueness
of multiplier to $\Lambda$ is equivalent to $\Lambda$ being not
isothermic.
\end{thm}

\begin{proof}
Suppose $\Lambda$ is a constrained Willmore surface and $q_{1}\neq
q_{2}$ are multipliers to $\Lambda$. Set $\eta:=*(q_{1}-q_{2})$,
defining, in this way, a non-zero real $1$-form with values in
$\Lambda\wedge\Lambda^{(1)}\subset\Lambda\wedge\Lambda^{\perp}$. The
fact that $d^{\mathcal{D}}q_{1}=0=d^{\mathcal{D}}q_{2}$ establishes
$d^{\mathcal{D}}*\eta=0$, or, equivalently, cf. Lemma
\ref{withvswithoutdecomps}, $d^{\mathcal{D}}\eta=0$; whereas
$[q_{1}\wedge
*\mathcal{N}]=\frac{1}{2}\,d^{\mathcal{D}}*\mathcal{N}=[q_{2}\wedge*\mathcal{N}]$
gives $[\mathcal{N}\wedge\eta]=[*\eta\wedge
*\mathcal{N}]=[(q_{2}-q_{1})\wedge *\mathcal{N}]=0$, recalling
equation \eqref{eq:bracstarstarbrac}. We conclude that $\eta$ is
closed and, therefore, that $(\Lambda,\eta)$ is isothermic.

Conversely, suppose $(\Lambda,\eta)$ is an isothermic
$q$-constrained Willmore surface, for some
$\eta\in\Omega^{1}(\Lambda\wedge\Lambda^{\perp})$ and
$q\in\Omega^{1}(\Lambda\wedge\Lambda^{(1)})$. Recalling Lemma
\ref{etainLwedgeL1}, set $q':=q+*\eta$, defining a real $1$-form
$q'\neq q$ with values in $\Lambda\wedge\Lambda^{(1)}$. Cf. Lemma
\ref{etadetcvansisheszesesseparately},
$d^{\mathcal{D}}\eta=0=[\mathcal{N}\wedge\eta]$. Equivalently
(recall Lemma \ref{withvswithoutdecomps}),
$d^{\mathcal{D}}*\eta=0=[*\eta\wedge *\mathcal{N}]$. Hence
$d^{\mathcal{D}}q'=d^{\mathcal{D}}q=0$ and $2[q'\wedge
*\mathcal{N}]=2[q\wedge*\mathcal{N}]=d^{\mathcal{D}}*\mathcal{N}$,
showing that $q'$ is a multiplier to $\Lambda$, as well as $q$, and
completing the proof.
\end{proof}

Constant mean curvature surfaces in $3$-space are examples of
isothermic constrained Willmore surfaces, as proven by J. Richter
\cite{richter}. A classical result by Thomsen \cite{thomsen}
characterizes isothermic Willmore surfaces in $3$-space as minimal
surfaces in some $3$-dimensional space-form, showing that the zero
multiplier is not necessarily the only multiplier to a constrained
Willmore surface with no constraint on the conformal structure.

Suppose $(\Lambda,\eta)$ is isothermic. Analogously to what was
observed in the proof of Theorem \ref{uniquenessofq} for the
particular case $t=1$, we verify, that, if $q$ is a multiplier to
$\Lambda$, then so is
$$q^{t}:=q+t*\eta,$$for each $t\in\R$. In the proof of
Theorem \ref{uniquenessofq}, we have, on the other hand, verified
that, if $q_{1}$ and $q_{2}$ are distinct multipliers to $\Lambda$,
then $(\Lambda,*(q_{1}-q_{2}))$ is isothermic and, therefore,
$q_{2}=q_{1}+t*\eta$, for some (non-zero) constant $t\in\R$. We
conclude that, if $\Lambda$ is constrained Willmore and $q$ is a
multiplier to $\Lambda$, then the set of multipliers to $\Lambda$ is
the affine space $q+\langle *\eta\rangle_{\R}$. In particular, the
set of multipliers to an isothermic Willmore surface
$(\Lambda,\eta)$ consists of the $1$-dimensional vector space
$\langle *\eta\rangle_{\R}$.

\subsection{Isothermic condition under constrained Willmore
transformation}\label{isoCWtransforms}

The constrained Willmore spectral deformation is known to preserve
the isothermic condition, cf. \cite{SD}. Next we derive it in our
setting.

\begin{thm}
The constrained Willmore spectral deformation preserves the
isothermic condition.
\end{thm}

The proof of the theorem will consist of showing that, if
$(\Lambda,\eta)$ is an isothermic $q$-constrained Willmore surface,
for some $\eta\in\Omega^{1}(\Lambda\wedge\Lambda^{\perp})$ and
$q\in\Omega^{1}(\Lambda\wedge \Lambda^{(1)})$, then, fixing
$\lambda\in S^{1}$ and an isomorphism
$\phi^{\lambda}_{q}:(\underline{\R}^{n+1,1},d^{\lambda}_{q})\rightarrow
(\underline{\R}^{n+1,1},d)$,
$(\phi^{\lambda}_{q}\Lambda,\mathrm{Ad}_{\phi^{\lambda}_{q}}\eta_{\lambda})$
is isothermic, for
$$\eta_{\lambda}:=\lambda^{-1}\eta^{1,0}+\lambda\,\eta^{0,1}.$$

\begin{proof}
Suppose $(\Lambda,\eta)$ is an isothermic $q$-constrained Willmore
surface, for some $\eta\in\Omega^{1}(\Lambda\wedge\Lambda^{\perp})$
and $q\in\Omega^{1}(\Lambda\wedge \Lambda^{(1)})$. Fix $\lambda\in
S^{1}$ and
$\phi^{\lambda}_{q}:(\underline{\R}^{n+1,1},d^{\lambda}_{q})\rightarrow
(\underline{\R}^{n+1,1},d)$ an isomorphism. Write
$\eta=\sigma\wedge\mu$, with $\sigma\in\Gamma(\Lambda)$ and
$\mu\in\Omega^{1}(\Lambda^{\perp})$. Then
$\mathrm{Ad}_{\phi^{\lambda}_{q}}\eta_{\lambda}=\phi^{\lambda}_{q}\sigma\wedge\phi^{\lambda}_{q}(\lambda^{-1}\mu^{1,0}+\lambda\mu^{0,1})$
is a non-zero real $1$-form with values in
$(\phi^{\lambda}_{q}\Lambda)\wedge
(\phi^{\lambda}_{q}\Lambda)^{\perp}$ and
$d(\mathrm{Ad}_{\phi^{\lambda}_{q}}\eta_{\lambda})=\phi^{\lambda}_{q}\circ
d^{d^{\lambda}_{q}}\eta_{\lambda}\circ (\phi^{\lambda}_{q})^{-1}$
vanishes if and only if
$$d^{d^{\lambda}_{q}}\eta_{\lambda}=d^{\mathcal{D}}\eta_{\lambda}+[(\lambda^{-1}\mathcal{N}^{1,0}+\lambda\mathcal{N}^{0,1}+(\lambda^{-2}-1)q^{1,0}+(\lambda^{2}-1)q^{0,1})\wedge\eta_{\lambda}]$$
does. Note that $q\in\Omega^{1}(\Lambda\wedge\Lambda^{\perp})$.
Thus, by \eqref{eq:lambdalambdaperpwedgemesmo0},
$[q^{1,0}\wedge\eta^{0,1}]=0=[q^{0,1}\wedge \eta^{1,0}]$ and,
therefore,
$d^{d^{\lambda}_{q}}\eta_{\lambda}=d^{\mathcal{D}}\eta_{\lambda}+[\mathcal{N}\wedge\eta]$.
Lemma \ref{etadetcvansisheszesesseparately}, together with Lemma
\ref{withvswithoutdecomps} (having in consideration Lemma
\ref{etainLwedgeL1}), establishes
$d^{\mathcal{D}}\eta^{1,0}=0=d^{\mathcal{D}}\eta^{0,1},\,\,\,[\mathcal{N}\wedge\eta]=0$,
completing the proof.
\end{proof}

As for B\"{a}cklund transformation of constrained Willmore surfaces,
we believe it does not necessarily preserve the isothermic
condition. This shall be the subject of further work.

\section{Constant mean curvature surfaces in $3$-space}\label{sec:CMC}

\markboth{\tiny{A. C. QUINTINO}}{\tiny{CONSTRAINED WILLMORE
SURFACES}}

Minimal surfaces arose originally as surfaces that minimized the
surface area, subject to some constraint, such as total volume
enclosed. Physical processes which can be modeled by minimal
surfaces include the formation of soap bubbles. A soap bubble can
can be thought of as an excellent approximation of some ideal
elastic matter, which encloses a volume and exists in an equilibrium
where slightly greater pressure inside the bubble is balanced by the
area-minimizing forces of the bubble itself. Minimal surfaces are
defined as surfaces with zero mean curvature and can be extended to
surfaces with constant, not necessarily zero, mean curvature.
Constant mean curvature surfaces in $3$-dimensional space-forms form
a very important class of isothermic constrained Willmore surfaces,
as proven by J. Richter \cite{richter}, with constrained Willmore
B\"{a}cklund transformations; both constrained Willmore and
isothermic spectral deformations; as well as a spectral deformation
of their own and, in the Euclidean case,
 isothermic Darboux transformations and Bianchi-B\"{a}cklund
transformations. The isothermic spectral deformation is known to
preserve the constancy of the mean curvature of a surface in some
space-form, cf. \cite{SD}. Characterized as the class of constrained
Willmore surfaces in $3$-dimensional space-forms admitting a
conserved quantity, the class of CMC surfaces in $3$-space is known
to be preserved by both constrained Willmore spectral deformation
and B\"{a}cklund transformation, for special choices of parameters.
We verify that both the space-form and the mean curvature are
preserved by constrained Willmore B\"{a}cklund transformation and
investigate how these change under constrained Willmore and
isothermic spectral deformation. We present the classical CMC
spectral deformation by means of the action of a loop of flat metric
connections on the class of CMC surfaces in $3$-space (preserving
the space-form and the mean curvature) and observe that the
classical CMC spectral deformation can be obtained as composition of
isothermic and constrained Willmore spectral deformation. These
spectral deformations of CMC surfaces in $3$-space are, in this way,
all closely related and, therefore, closely related to constrained
Willmore B\"{a}cklund transformation. S. Kobayashi and J.-I.
Inoguchi \cite{kobayashi} proved that isothermic Darboux
transformation of CMC surfaces in Euclidean $3$-space is equivalent
to Bianchi-B\"{a}cklund transformation. We believe isothermic
Darboux transformation of a CMC surface in Euclidean $3$-space can
be obtained as a particular case of constrained Willmore
B\"{a}cklund transformation. This shall be the subject of further
work. In contrast to isothermic or constrained Willmore surfaces in
space-forms, surfaces of constant mean curvature are not conformally
invariant objects.\newline

Throughout this section, consider $n=3$. For simplicity, we use $T$
and $\perp$ to indicate the orthogonal projections of
$\underline{\R}^{4,1}$ onto $S$ and $S^{\perp}$, respectively.

\begin{rem}
In contrast to isothermic or constrained Willmore surfaces in
space-forms, surfaces of constant mean curvature are not conformally
invariant objects (recall equation
\eqref{eq:conformalvarianceoftheMCV}).
\end{rem}

Fix $v_{\infty}\in\R^{4,1}$ non-zero. Consider the surface
$\sigma_{\infty}:M\rightarrow S_{v_{\infty}}$, in the space-form
$S_{v_{\infty}}$, defined by $\Lambda$. Given
$\xi\in\Gamma(N_{\infty})$ unit, the mean curvature
$H^{\xi}_{\infty}$ of $\sigma_{\infty}$ with respect to $\xi$ is
given by
$H^{\xi}_{\infty}=(\xi,\mathcal{H}_{\infty})=-(\xi+(\xi,\mathcal{H}_{\infty})\sigma_{\infty},v_{\infty})=-(\mathcal{Q}\xi,v_{\infty}^{\perp})$,
for $\mathcal{Q}:N_{\infty}\rightarrow S^{\perp}$ the isometry
defined in Section \ref{normalbundle}. Because we are in codimension
$1$, $S^{\perp}=\langle \mathcal{Q}\xi\rangle$ and we conclude that
$H^{\xi}_{\infty}=\pm(v_{\infty}^{\perp},v_{\infty}^{\perp})^{\frac{1}{2}}$,
depending on the sign of $\xi$. We define the \textit{mean curvature
of $\Lambda$ in the space-form $S_{v_{\infty}}$} to be
$$H_{\infty}:=(v_{\infty}^{\perp},v_{\infty}^{\perp})^{\frac{1}{2}}$$
and define $\Lambda$ to be a \textit{constant mean curvature
surface} (respectively, a \textit{minimal surface}) \textit{in the
space-form} $S_{v_{\infty}}$ if $\sigma_{\infty}$ is so:
\begin{defn}
$\Lambda$ is said to be  a \emph{constant} \emph{mean}
\emph{curvature} (CMC) \emph{surface} in the space-form
$S_{v_{\infty}}$ if $(v_{\infty}^{\perp},v_{\infty}^{\perp})$ is
constant. In the case $v_{\infty}\in\Gamma(S)$, $\Lambda$ is said to
be, specifically, a \emph{minimal surface} in $S_{v_{\infty}}$.
\end{defn}

Let $N$ be the real unit section of $S^{\perp}$ for which
\begin{equation}\label{eq:vinfperpHinfN}
v_{\infty}^{\perp}=H_{\infty}N,
\end{equation}
which, in the particular case $\Lambda$ is minimal in
$S_{v_{\infty}}$, is defined only up to sign.\footnote{There are
exactly two possible choices of real unit sections of $S^{\perp}$,
symmetrical of each other - they are  $\mathcal{Q}\xi$ with $\xi$ a
unit normal vector field to $\sigma_{\infty}$. Unless
$H_{\infty}=0$, condition \eqref{eq:vinfperpHinfN} determines $N$ as
$\mathcal{Q}\xi$ for $\xi$ the unit normal vector field to
$\sigma_{\infty}$ with $H_{\infty}^{\xi}=-H_{\infty}$ (or,
equivalently, $H_{\infty}^{\xi}<0$).}

\begin{rem}\label{vingTnot0}
If $v_{\infty}^{T}=0$, then $v_{\infty}^{\perp}$ is a constant
section of $S^{\perp}$, and so is then
$H_{\infty}^{-1}v_{\infty}^{\perp}=N$. The constancy of the normal
to $S$, in its turn, establishes the constancy of $S$, establishing,
in particular, that $\Lambda$ lies in a $2$-sphere:
$\P(\mathcal{L}\cap S)$; which contradicts
\eqref{eq:notinsmallerspheres}. Thus
\begin{equation}\label{eq:vinfTnonzero} v_{\infty}^{T}\neq0.
\end{equation}
\end{rem}

It is useful to note that, in view of the constancy of $(N,N)$, and
because we are in codimension $1$ (and, therefore,
$S^{\perp}=\langle N\rangle$), we have
\begin{equation}\label{eq:dNinS}
dN\in\Omega^{1}(S).
\end{equation}

\subsection{CMC surfaces in $3$-space as isothermic constrained
Willmore surfaces with a conserved quantity}\label{CMCisoCW}

Constant mean curvature surfaces in $3$-dimen\-sio\-nal space-forms
are examples of isothermic constrained Willmore surfaces, as proven
by J. Richter \cite{richter}. In this section, we establish it in
our setting. We present a $1$-form, derived\footnote{From the notion
of \textit{conserved quantity of an isothermic surface}, presented
in \cite{susana}, similarly to how a $q$-conserved quantity of a
constrained Willmore surface determines $q$ (cf. Remark
\ref{CQeqdeterminesq}).} by F. Burstall and D. Calderbank from a
surface with constant mean curvature in $3$-space, which establishes
the surface as an isothermic surface and for which scaling by the
mean curvature provides a multiplier to the surface. We prove also
that constant mean curvature surfaces in $3$-dimensional space-forms
are the constrained Willmore surfaces in $3$-space admitting a
conserved quantity.\newline

Suppose $\Lambda$ has constant mean curvature $H_{\infty}$ in
$S_{v_{\infty}}$. Set
$$\eta_{\infty}:=\frac{1}{2}\,\sigma_{\infty}\wedge dN.$$
We may, alternatively, use $\eta_{\infty}^{N}$ to denote
$\eta_{\infty}$, in order to avoid the ambiguity with respect to the
sign of $N$ in the particular case $H_{\infty}=0$.

\begin{thm}\label{thm905}
$(\Lambda,\eta_{\infty})$ is isothermic.
\end{thm}

\begin{proof}
The reality of $\eta_{\infty}$ is immediate, in view of the reality
of both $\sigma_{\infty}$ and $N$. Fix a holomorphic chart $z$ of
$M$. The fact that $N$ is orthogonal to $S$ shows, in particular,
that $(N_{z},\sigma_{\infty})=0=(N_{z},(\sigma_{\infty})_{\bar{z}})$
and, therefore, in view of the maximal isotropy of $\Lambda^{0,1}$
in $S$, that $N_{z}\in\Gamma(\Lambda^{0,1})$. On the other hand,
differentiation of $(\sigma_{\infty},v_{\infty})=-1$ shows that
$((\sigma_{\infty})_{\bar{z}},v_{\infty})=0$ and, consequently, that
the component of $N_{z}$ with respect to $\sigma_{\infty}$ in the
frame $(\sigma_{\infty},(\sigma_{\infty})_{\bar{z}})$ of
$\Lambda^{0,1}$ is
$-(N_{z},v_{\infty})=-(N,v_{\infty}^{\perp})_{z}=(H_{\infty})_{z}=0$.
Thus, having in consideration the reality of $N$,
\begin{equation}\label{eq:Nzjkl}
N_{z}\in\Gamma\langle(\sigma_{\infty})_{\bar{z}}\rangle,\,\,\,\,N_{\bar{z}}\in\Gamma\langle(\sigma_{\infty})_{z}\rangle.
\end{equation}
In particular, $dN\in\Omega^{1}(\Lambda^{\perp})$ and, therefore,
$\eta_{\infty}\in\Omega^{1}(\Lambda\wedge\Lambda^{\perp})$.\footnote{Furthermore,
$\eta_{\infty}\in\Omega^{1}(\Lambda\wedge(\Lambda^{\perp}\cap
S))\subset \Omega^{1}(\Lambda\wedge\Lambda^{(1)})$.} It follows from
\eqref{eq:Nzjkl}, on the other hand, having in consideration
\eqref{eq:LiBrdelta0}, that
$$d(\sigma_{\infty}\wedge
dN)(\delta_{z},\delta_{\bar{z}})=(\sigma_{\infty})_{z}\wedge
N_{\bar{z}}-(\sigma_{\infty})_{\bar{z}}\wedge
N_{z}+\sigma_{\infty}\wedge(N_{z\bar{z}}-N_{\bar{z}z})=0,$$ or,
equivalently, $d\eta_{\infty}=0$. Lastly, note that, if
$\sigma_{\infty}\wedge dN=0$, or, equivalently,
$dN\in\Omega^{1}(\langle\sigma_{\infty}\rangle)$, then, in
particular, $(N_{\bar{z}},(\sigma_{\infty})_{\bar{z}})=0$. Together
with \eqref{eq:Nzjkl}, and having in consideration that
$((\sigma_{\infty})_{z},(\sigma_{\infty})_{\bar{z}})$ is never-zero,
this forces $N_{\bar{z}}$ to vanish and, therefore, in view of the
reality of $N$, $N$ to be constant. But, as observed in Remark
\ref{vingTnot0}, the constancy of $N$ forces $\Lambda$ to lie in a
$2$-sphere, which is not the case. Hence $\sigma_{\infty}\wedge dN$
is non-zero, completing the proof.
\end{proof}

\begin{rem}
In the proof of Theorem \ref{thm905}, we have observed, in
particular, that $N_{z}\in\Gamma(\Lambda^{0,1})$ and, therefore,
\begin{equation}\label{eq:eta10LambdaLambda01}
\eta_{\infty}^{1,0}\in\Omega^{1,0}(\Lambda\wedge\Lambda^{0,1}),\,\,\,\,\eta_{\infty}^{0,1}\in\Omega^{0,1}(\Lambda\wedge\Lambda^{1,0}),
\end{equation}
in view of the reality of $\eta_{\infty}$.
\end{rem}

Set $$q_{\infty}=H_{\infty}\eta_{\infty}.$$

\begin{thm}\label{qinfCQ}
$\Lambda$ is a $q_{\infty}$-constrained Willmore surface.
\end{thm}

The proof of the theorem will follow a few considerations. First
note that the constancy of $(v_{\infty}^{\perp},v_{\infty}^{\perp})$
ensures, in particular, that $v_{\infty}^{\perp}$ is either zero or
never-zero, and, on the other hand, that
\begin{equation}\label{eq:auxmaisuma4356}
((dv_{\infty}^{\perp})^{\perp},v_{\infty}^{\perp})=0.
\end{equation}
In the case $v_{\infty}^{\perp}$ is never-zero, and, therefore,
$S^{\perp}=\langle v_{\infty}^{\perp}\rangle$, equation
\eqref{eq:auxmaisuma4356} establishes
\begin{equation}\label{eq:perpdvper0}
(dv_{\infty}^{\perp})^{\perp}=0,
\end{equation}
an equality that, obviously, still holds in the case
$v_{\infty}\in\Gamma(S)$. Equivalently,
\begin{equation}\label{eq:utily}
\mathcal{N}v_{\infty}^{\perp}=dv_{\infty}^{\perp}.
\end{equation}
Since $dv_{\infty}^{\perp}=H_{\infty}dN$, we conclude that
\begin{equation}\label{eq:qinfifLambdaCMC}
q_{\infty}=\frac{1}{2}\,\sigma_{\infty}\wedge\mathcal{N}v_{\infty}^{\perp}.
\end{equation}

Now we proceed to the proof of Theorem \ref{qinfCQ}.

\begin{proof}
According to Theorem \ref{thm905}, $(\Lambda,\eta_{\infty})$ is
isothermic, which, according to Lemma \ref{etainLwedgeL1},
establishes that the $1$-form $q_{\infty}$ takes values in
$\Lambda\wedge\Lambda^{(1)}$. The reality of $q_{\infty}$ is
obviously equivalent to that of $\eta_{\infty}$. On the other hand,
by \eqref{eq:etaisoeqs},
$$d^{\mathcal{D}}q_{\infty}=H_{\infty}\,d^{\mathcal{D}}\eta_{\infty}=0.$$
To complete the proof, we are left to verify that
$d^{\mathcal{D}}*\mathcal{N}=2[q_{\infty}\wedge*\mathcal{N}]$.\footnote{The
scaling  of $\eta_{\infty}$ by $H_{\infty}$ in order to obtain a
multiplier is determined by this equation.}

Notation: given
$\Psi\in\Omega^{1}(\mathrm{End}(\underline{\R}^{4,1}))$ and
$\psi\in\Omega^{1}(\underline{\R}^{4,1})$, $[\Psi,\psi]$ denotes the
$2$-form with values in $\underline{\R}^{4,1}$ defined by
$[\Psi,\psi](X,Y):=\Psi _{X}\psi_{Y}-\Psi _{Y}\psi_{X}$, for
$X,Y\in\Gamma(TM)$.

In view of \eqref{eq:codcurlyN},
$$(d^{\mathcal{D}}*\mathcal{N})v_{\infty}^{T}=-2i\,(d^{\mathcal{D}}\mathcal{N}^{1,0})v_{\infty}^{T}=-2i\,
d^{\mathcal{D}}(\mathcal{N}^{1,0}v_{\infty}^{T})-2i[\mathcal{N}^{1,0},\mathcal{D}^{0,1}v_{\infty}^{T}].$$
On the other hand, by the constancy of the section $v_{\infty}$ of
$\underline{\R}^{4,1}$,
\begin{equation}\label{eq:dvinfis0}
\mathcal{D}v_{\infty}^{T}+\mathcal{D}v_{\infty}^{\perp}+\mathcal{N}v_{\infty}^{T}+\mathcal{N}v_{\infty}^{\perp}=0,
\end{equation}
and, in particular, considering the orthogonal projection of
$\underline{\R}^{4,1}$ onto $S^{\perp}$,
\begin{equation}\label{eq:hagrt}
\mathcal{N}v_{\infty}^{T}=-\mathcal{D}v_{\infty}^{\perp}.
\end{equation}
 But, according to \eqref{eq:perpdvper0},
\begin{equation}\label{eq:conseqof65}
\mathcal{D}v_{\infty}^{\perp}=0.
\end{equation}
Hence
$$(d^{\mathcal{D}}*\mathcal{N})\,v_{\infty}^{T}=-2i\,[\mathcal{N}^{1,0},\mathcal{D}^{0,1}v_{\infty}^{T}].$$
On the other hand,
$$2[q_{\infty}\wedge
*\mathcal{N}]\,v_{\infty}^{T}=-2i[q_{\infty}\wedge\mathcal{N}^{1,0}]\,v_{\infty}^{T}+2i[q_{\infty}\wedge\mathcal{N}^{0,1}]\,v_{\infty}^{T}.$$
The fact that $q_{\infty}$ takes values in
$\Lambda\wedge\Lambda^{(1)}$ establishes, in particular,
$q_{\infty}\,S^{\perp}=0$ and, therefore,
$$[q_{\infty}\wedge\mathcal{N}^{1,0}]\,v_{\infty}^{T}=[\mathcal{N}^{1,0},
q_{\infty}^{0,1}v_{\infty}^{T}].$$In view of
\eqref{eq:qinfifLambdaCMC},
$$2\,q_{\infty}^{0,1}v_{\infty}^{T}=(\sigma_{\infty}\wedge\mathcal{N}^{0,1}v_{\infty}^{\perp})v_{\infty}^{T}=-\mathcal{N}^{0,1}v_{\infty}^{\perp}-(\mathcal{N}^{0,1}v_{\infty}^{\perp},v_{\infty}^{T})\sigma_{\infty},$$
and, therefore,
$$2\,[q_{\infty}\wedge\mathcal{N}^{1,0}]\,v_{\infty}^{T}=-[\mathcal{N}^{1,0},\mathcal{N}^{0,1}v_{\infty}^{\perp}].$$
Similarly,
$$2\,[q_{\infty}\wedge\mathcal{N}^{0,1}]\,v_{\infty}^{T}=-[\mathcal{N}^{0,1},\mathcal{N}^{1,0}v_{\infty}^{\perp}].$$
Now observe, in view of equation \eqref{eq:perpdvper0}, that, given
$X,Y\in\Gamma(TM)$,
$$\mathcal{N}_{X}\mathcal{N}_{Y}v_{\infty}^{\perp}-\mathcal{N}_{Y}\mathcal{N}_{X}v_{\infty}^{\perp}=
\pi_{S^{\perp}}(d_{X}d_{Y}v_{\infty}^{\perp}-d_{Y}d_{X}v_{\infty}^{\perp})=\pi_{S^{\perp}}(d_{[X,Y]}v_{\infty}^{\perp})=0.$$
In particular,
$$[\mathcal{N}^{1,0},\mathcal{N}^{0,1}v_{\infty}^{\perp}]=-[\mathcal{N}^{0,1},\mathcal{N}^{1,0}v_{\infty}^{\perp}].$$
Hence $$2[q_{\infty}\wedge *\mathcal{N}]\,v_{\infty}^{T}=
2i\,[\mathcal{N}^{1,0},\mathcal{N}^{0,1}v_{\infty}^{\perp}].$$ Going
back to equation \eqref{eq:dvinfis0}, and considering, this time,
the orthogonal projection of $\underline{\R}^{4,1}$ onto $S$,
establishes
\begin{equation}\label{eq:NperpDT}
\mathcal{N}v_{\infty}^{\perp}=-\mathcal{D}v_{\infty}^{T},
\end{equation}
and, ultimately,
$$(d^{\mathcal{D}}*\mathcal{N})\,v_{\infty}^{T}=2\,[q_{\infty}\wedge
*\mathcal{N}]\,v_{\infty}^{T},$$
which, in view of the fact that
$(\sigma_{\infty},v_{\infty}^{T})=(\sigma_{\infty},v_{\infty})(=-1)$
is never-zero, completes the proof, cf. Remark
\ref{remmuyutilcurlyDNqCWeq}.
\end{proof}

Note that, if $\Lambda$ is minimal in $S_{v_{\infty}}$ (i.e.,
$H_{\infty}=0$), then $q_{\infty}=0$, which, according to Theorem
\ref{qinfCQ}, establishes $\Lambda$ as a Willmore surface.

\begin{corol}
A CMC surface in $3$-space is, in particular, a constrained Willmore
surface. A minimal surface in $3$-space is, in particular, a
Willmore surface.
\end{corol}
Minimal surfaces in $3$-space are, in particular, isothermic
Willmore surfaces. Furthermore, a classical result by Thomsen
\cite{thomsen} characterizes isothermic Willmore surfaces in
$3$-space as minimal surfaces in some $3$-dimensional space-form.
Hence a CMC surface in a $3$-dimensional space-form is a Willmore
surface if and only if it is minimal.

Next we establish a conserved quantity of a CMC surface in $3$-space
(see also Proposition \ref{pinftyallt}).
\begin{prop}\label{CMChasCQ}
$\Lambda$ admits
$$p_{\infty}(\lambda):=\lambda^{-1}\frac{1}{2}\,v_{\infty}^{\perp}+v_{\infty}^{T}+\lambda\frac{1}{2}\,
v_{\infty}^{\perp}$$ as a $q_{\infty}$-conserved quantity.
\end{prop}
\begin{proof}
According to equations \eqref{eq:hagrt} and \eqref{eq:conseqof65},
we have $\mathcal{N}v_{\infty}^{T}=0$, so that
$(\mathcal{N}v_{\infty}^{\perp},v_{\infty}^{T})=-(v_{\infty}^{\perp},\mathcal{N}v_{\infty}^{T})=0$
and, consequently, by \eqref{eq:qinfifLambdaCMC},
\begin{equation}\label{eq:qNemCMC}
q_{\infty}^{1,0}v_{\infty}^{T}=\frac{1}{2}\,(-\mathcal{N}^{1,0}v_{\infty}^{\perp}-(\mathcal{N}^{1,0}v_{\infty}^{\perp},v_{\infty}^{T})\sigma_{\infty})=-\frac{1}{2}\,\mathcal{N}^{1,0}v_{\infty}^{\perp}.
\end{equation}
Equation \eqref{eq:conseqof65} completes the proof, according to
Theorem \ref{charactdeCQ}.
\end{proof}

\begin{prop}\label{CW+areCMC}
A constrained Willmore surface in $3$-space admitting a conserved
quantity $p(\lambda)$ is, in particular, a CMC surface in the
space-form $S_{p(1)}$.
\end{prop}
\begin{proof}
Let $\hat{\Lambda}\subset\underline{\R}^{4,1}$ be a constrained
Willmore surface in the projectivized light-cone. Let $\perp$
indicate, temporarily, the orthogonal projection of
$\underline{\R}^{4,1}$ onto the normal bundle to the central sphere
congruence of $\hat{\Lambda}$. The existence of a conserved quantity
$p(\lambda)$ of $\hat{\Lambda}$ establishes, in particular, cf.
Theorem \ref{charactdeCQ}, the constancy of
$\hat{v}_{\infty}:=p(1)$. Furthermore, by \eqref{eq:utily},
$d(\hat{v}_{\infty}^{\perp},\hat{v}_{\infty}^{\perp})=2(d\hat{v}_{\infty}^{\perp},\hat{v}_{\infty}^{\perp})=(\mathcal{N}_{\hat{\Lambda}}\hat{v}_{\infty}^{\perp},\hat{v}_{\infty}^{\perp})=0$,
establishing $\hat{\Lambda}$ as a CMC surface in the space-form
$S_{\hat{v}_{\infty}}$.
\end{proof}

Theorem \ref{qinfCQ} combines with Propositions \ref{CMChasCQ} and
\ref{CW+areCMC} to establish, in particular, the following:
\begin{thm}\label{CMCsareCWswithCQ}
CMC surfaces in $3$-dimensional space-forms are the constrained
Willmore surfaces in $3$-space admitting a $q$-conserved quantity,
for some multiplier $q$.
\end{thm}

Next we establish a conserved quantity with respect to a general
multiplier to a CMC surface in a $3$-dimensional space-form. For
each $t\in\R$, set
$$q_{\infty}^{t}:=q_{\infty}+t*\eta_{\infty}.$$
Combined, Theorems \ref{thm905} and \ref{qinfCQ} establish, in
particular, the set of multipliers to $\Lambda$ as the family
$q_{\infty}^{t}$, with $t\in\R$. In generalization of Proposition
\ref{CMChasCQ}, we have:

\begin{prop}\label{pinftyallt}
$\Lambda$ admits
$$p_{\infty}^{t}(\lambda):=\lambda^{-1}\frac{1}{2}\,(H_{\infty}-it)N+ v_{\infty}^{T}+\lambda\frac{1}{2}\,(H_{\infty}+it)N$$
as a $q_{\infty}^{t}$-conserved quantity.
\end{prop}

\begin{proof}
First of all, note that
$d(v_{\infty}^{T}+\frac{1}{2}((H_{\infty}-it)+(H_{\infty}+it))N)=dv_{\infty}=0$.
The fact that $H_{\infty}=(N,v_{\infty}^{\perp})$ is constant
establishes
\begin{equation}\label{eq:dNvinfperp0}
0=d(N,v_{\infty}^{\perp})=d(N,v_{\infty})=(dN,v_{\infty}).
\end{equation}
By \eqref{eq:dNinS}, it follows that $(dN,v_{\infty}^{T})=0$ and,
consequently,
$\eta_{\infty}^{1,0}v_{\infty}^{T}=-\frac{1}{2}\,d^{1,0}N$. By
\eqref{eq:qNemCMC}, and, yet again,  \eqref{eq:dNinS}, we conclude
then that
$$(q_{\infty}^{t})^{1,0}v_{\infty}^{T}=q_{\infty}^{1,0}v_{\infty}^{T}-it\eta_{\infty}^{1,0}v_{\infty}^{T}= -\frac{1}{2}\,\mathcal{N}^{1,0}v_{\infty}^{\perp}+\frac{it}{2}\,d^{1,0}N=-\frac{1}{2}\,\mathcal{N}^{1,0}
(H_{\infty}-it)N.$$ On the other hand, \eqref{eq:dNinS} establishes
$\mathcal{D}^{0,1}N=0$, and then
$\mathcal{D}^{0,1}(H_{\infty}-it)N=0$, in view of the constancy of
$H_{\infty}$. The conclusion follows, according to Theorem
\ref{charactdeCQ}.
\end{proof}

We may, alternatively, use $q_{\infty}^{N,t}$ and $p_{\infty}^{N,t}$
to denote $q_{\infty}^{t}$ and $p_{\infty}^{t}$, respectively, in
order to avoid the ambiguity with respect to the sign of $N$ in the
particular case $H_{\infty}=0$.

We complete this section by remarking on the close relationship
between the multiplier $q_{\infty}$ and the Hopf differential. Fix a
holomorphic chart $z$ of $M$. For simplicity, write $q_{\infty}^{z}$
for $(q_{\infty})^{z}$, alternatively. Note that
$$(q_{\infty})_{\delta_{z}}(\sigma_{\infty})_{z}=\frac{1}{2}\,(\sigma_{\infty}\wedge
dv_{\infty}^{\perp})_{\delta_{z}}(\sigma_{\infty})_{z}=-\frac{1}{2}\,(d_{\delta_{z}}v_{\infty}^{\perp},(\sigma_{\infty})_{z})\sigma_{\infty}$$
and, therefore, by equation \eqref{eq:conseqof65},
$$-2(q_{\infty})_{\delta_{z}}(\sigma_{\infty})_{z}=-(v_{\infty}^{\perp},\mathcal{N}_{\delta_{z}}(\sigma_{\infty})_{z})\sigma_{\infty}=-(v_{\infty}^{\perp},((\sigma_{\infty})_{zz})^{\perp})\sigma_{\infty}.$$
We conclude that
$$q_{\infty}^{z}=-(v_{\infty}^{\perp},((\sigma_{\infty})_{zz})^{\perp}).$$
Observe that, if $v_{\infty}^{\perp}\neq 0$, in which case
$S^{\perp}=\langle v_{\infty}^{\perp}\rangle$, then, in view of the
non-degeneracy of $S^{\perp}$ and according to equation
\eqref{eq:kzpiSperpsigmainfzz}, $q_{\infty}^{z}$ is real if and only
if so is $k^{z}$. Note that
$$q_{\infty}^{z}=-\lambda^{-1}(v_{\infty}^{\perp},k^{z}),$$
for $\lambda\in\Gamma(\underline{\R})$ as in equation
\eqref{eq:kzpiSperpsigmainfzz}. For later reference, note that, in
the case $v_{\infty}^{\perp}$ is non-zero,
\begin{equation}\label{eq:qinftzvsqinfz}
(q_{\infty}^{t})^{z}=(1-itH_{\infty}^{-1})(q_{\infty})^{z}.
\end{equation}

\subsection{CMC surfaces in $3$-space: an equation on the Hopf differential
and the Schwarzian derivative} Constant mean curvature surfaces in
$3$-dimensional space-forms are, in particular, isothermic
constrained Willmore surfaces. In view of the characterization of
isothermic surfaces in space-forms in terms of the reality of the
Hopf differential, the close relationship between the Hopf
differential and the set of multipliers to a CMC surface in
$3$-space leads to a characterization of these surfaces in terms of
the Schwarzian derivative and the Hopf differential, following the
characterization of constrained Willmore surfaces in space-forms
presented in Section \ref{cW+kc}.\newline

Let $z$ be a holomorphic chart of $M$, which, in the case $\Lambda$
has constant mean curvature in some space-form, we can choose so
that $k^{z}$ is real, cf. Lemma \ref{isoviakreal}. If $\Lambda$ is
minimal in $S_{v_{\infty}}$, then $\Lambda$ is Willmore, so that,
according to Lemma \ref{ckCWeq},
$\nabla^{S^{\perp}}_{\delta_{\bar{z}}}\nabla^{S^{\perp}}_{\delta_{\bar{z}}}k^{z}+\frac{\overline{c^{z}}}{2}\,k^{z}=0$
and, therefore,
\begin{equation}\label{eq:CMCeqoncandkbajhhsdajhgklasjdhgf}
\nabla^{S^{\perp}}_{\delta_{\bar{z}}}\nabla^{S^{\perp}}_{\delta_{\bar{z}}}k^{z}+\frac{\overline{c^{z}}}{2}\,k^{z}=H_{\infty}k^{z}.
\end{equation}
On the other hand, if $\Lambda$ has non-zero constant mean curvature
in $S_{v_{\infty}}$ and $k^{z}$ is real, then, according to what was
observed in Section \ref{CMCisoCW}, together with Lemma
\ref{ckCWeq}, $q_{\infty}^{z}$ is a non-zero real-valued holomorphic
function and, therefore, constant. We conclude that, if $\Lambda$
has constant mean curvature $H_{\infty}\neq 0$ in $S_{v_{\infty}}$,
then we can choose $z$ such that $k^{z}$ is real and
$q^{z}_{\infty}=H_{\infty}$ and, therefore, according to Lemma
\ref{ckCWeq} equation \eqref{eq:CMCeqoncandkbajhhsdajhgklasjdhgf}
still holds. Furthermore, cf. \cite{SD}:
\begin{Lemma}\label{CMCviacharts}
$\Lambda$ has constant mean curvature $H\in\R$ in some space-form if
and only if around each point there exists a holomorphic chart $z$
of $(M,\mathcal{C}_{\Lambda})$ such that $k^{z}$ is real and
$$\nabla^{S^{\perp}}_{\delta_{\bar{z}}}\nabla^{S^{\perp}}_{\delta_{\bar{z}}}k^{z}+\frac{\overline{c^{z}}}{2}k^{z}=Hk^{z}.$$
In that case, $\Lambda$ is a $q$-constrained Willmore surface for
$q$ the quadratic differential defined locally by $q^{z}dz^{2}$ for
$q^{z}:=H$, under the correspondence given by
\eqref{eq:correspondenec between qzand qsharp}.
\end{Lemma}

\subsection{Spectral deformations of CMC surfaces in $3$-space}\label{CWspdeform}
As a class of constrained Willmore surfaces admitting a conserved
quantity $p(\lambda)$ with $p(1)$ with non-zero orthogonal
projection onto the central sphere congruence, the class of CMC
surfaces in $3$-space is known to be preserved by constrained
Willmore spectral deformation. In this section, we investigate how
the space-form and the mean curvature change under this deformation.
The isothermic spectral deformation is known to preserve the
constancy of the mean curvature of a surface in some space-form, cf.
\cite{SD}. We establish it in our setting, along with verifying how
the space-form and the mean curvature change under this deformation.
We present the classical CMC spectral deformation by means of the
action of a loop of flat metric connections on the class of CMC
surfaces in $3$-space (preserving the space-form and the mean
curvature).  We verify that all these deformations of CMC surfaces
are closely related. We observe, in particular, that the classical
CMC spectral deformation can be obtained as composition of
isothermic and constrained Willmore spectral deformation and that,
in the particular case of minimal surfaces, the classical CMC
spectral deformation coincides, up to reparametrization, with the
constrained Willmore spectral deformation corresponding to the zero
multiplier.\newline

Suppose $\Lambda$ has constant mean curvature $H_{\infty}$ in
$S_{v_{\infty}}$. In that case, $(\Lambda,\eta_{\infty})$ is an
isothermic $q_{\infty}^{t}$-constrained Willmore surface admitting
$p_{\infty}^{t}(\lambda)$ as a $q_{\infty}^{t}$-conserved quantity,
for each $t\in\R$.\newline

\textbf{The constrained Willmore spectral deformation.} For each
$t\in\R$ and $\lambda\in S^{1}$, the constrained Willmore spectral
deformation of $\Lambda$ of parameter $\lambda$ corresponding to the
multiplier $q_{\infty}^{t}$ is known to have constant mean curvature
in some space-form, cf. Theorem \ref{specCWwcq}, in view of
\eqref{eq:vinfTnonzero}. Next we investigate how the space-form and
the mean curvature change with this deformation.\newline

Fix $t\in\R$, $\lambda\in S^{1}$ and
$\phi^{\lambda}_{q_{\infty}^{t}}:(\underline{\R}^{4,1},d^{\lambda}_{q_{\infty}^{t}})\rightarrow
(\underline{\R}^{4,1},d)$ an isomorphism. Set
$$v_{q_{\infty}^{t}}^{\lambda}:=
\phi^{\lambda}_{q_{\infty}^{t}}(v_{\infty}^{T}+
((\mathrm{Re}\lambda)H_{\infty}+\frac{it}{2}(\lambda-\lambda^{-1}))N).\footnote{Recall
that $v_{\infty}^{\perp}=H_{\infty}N$, so that, in the particular
case $t=0$,
$v_{q_{\infty}^{t}}^{\lambda}=\phi^{\lambda}_{q_{\infty}^{t}}(v_{\infty}^{T}+(\mathrm{Re}\,\lambda)v_{\infty}^{\perp})$.}$$
\begin{Lemma}
$v_{q_{\infty}^{t}}^{\lambda}$ is a non-zero constant section of
$\underline{\R}^{4,1}$.
\end{Lemma}
\begin{proof}
According to \eqref{eq:vinfTnonzero}, $v_{\infty}^{T}$ is non-zero
and so is then $v_{q_{\infty}^{t}}^{\lambda}$. The fact that
$\lambda$ is unit establishes the reality of
$\frac{it}{2}(\lambda-\lambda^{-1})$ and, consequently, that of
$v_{q_{\infty}^{t}}^{\lambda}\in\Gamma((\underline{\R}^{4,1})^{\C})$.
In its turn, the fact that $p_{\infty}^{t}$ is a
$q_{\infty}^{t}$-conserved quantity of $\Lambda$ ensures the
constancy of the section
$\phi^{\lambda}_{q_{\infty}^{t}}(p_{\infty}^{t}(\lambda))=v_{q_{\infty}^{t}}^{\lambda}$
of $\underline{\R}^{4,1}$:
$d(\phi^{\lambda}_{q_{\infty}^{t}}(p_{\infty}^{t}(\lambda)))=\phi^{\lambda}_{q_{\infty}^{t}}(d^{\lambda}_{q_{\infty}^{t}}p_{\infty}^{t}(\lambda))=0$.
\end{proof}

In view of the fact that $p_{\infty}^{t}(\lambda)$ is a
$q_{\infty}^{t}$-conserved quantity of $\Lambda$, together with
\eqref{eq:vinfTnonzero}, Theorem \ref{specCWwcq} establishes the
constrained Willmore spectral deformation of $\Lambda$, of parameter
$\lambda$, corresponding to the multiplier $q_{\infty}^{t}$ as a
constrained Willmore surface (in $3$-dimensional space-form)
admitting a conserved quantity, or, equivalently, as a CMC surface
in some $3$-space. Specifically, according to Proposition
\ref{CW+areCMC}:
\begin{thm}
The spectral deformation of $\Lambda$ of parameter $\lambda$
corresponding to the multiplier $q_{\infty}^{t}$ is a CMC surface in
$S_{v_{q_{\infty}^{t}}^{\lambda}}$.
\end{thm}

Recall that the constrained Willmore spectral deformation preserves
the central sphere congruence, as observed in \eqref{eq:cscpreserved
byCWdeform}, to conclude that:
\begin{prop}
Let $\Lambda^{\lambda}_{q_{\infty}^{t}}$ be the spectral deformation
of $\Lambda$, of parameter $\lambda$, corresponding to the
multiplier $q_{\infty}^{t}$. The mean curvature
$H_{q_{\infty}^{t}}^{\lambda}$ of
$\Lambda^{\lambda}_{q_{\infty}^{t}}$ in
$S_{v_{q_{\infty}^{t}}^{\lambda}}$ relates to the mean curvature of
$\Lambda$ in $S_{v_{\infty}}$ by
$$H_{q_{\infty}^{t}}^{\lambda}=\mid\mathrm{Re}\,(\lambda H_{\infty}+\frac{it}{2}(\lambda-\lambda^{-1}))\mid.$$
\end{prop}
\begin{proof}
The mean curvature of $\phi^{\lambda}_{q_{\infty}^{t}}\Lambda$ in
$S_{v_{q_{\infty}^{t}}^{\lambda}}$ is given by
$H_{q_{\infty}^{t}}^{\lambda}=(\pi_{\lambda,t}^{\perp}\,v_{q_{\infty}^{t}}^{\lambda},\pi_{\lambda,t}^{\perp}\,v_{q_{\infty}^{t}}^{\lambda})^{\frac{1}{2}}$,
for $\pi_{\lambda,t}^{\perp}$ the orthogonal projection of
$\underline{\R}^{4,1}$ onto the normal bundle to
$S_{\phi^{\lambda}_{q_{\infty}^{t}}\Lambda}=\phi^{\lambda}_{q_{\infty}^{t}}S$,
and, therefore,
\begin{eqnarray*}
H_{q_{\infty}^{t}}^{\lambda}&=&(((\mathrm{Re}\lambda)H_{\infty}+\frac{it}{2}(\lambda-\lambda^{-1}))N,((\mathrm{Re}\lambda)H_{\infty}+\frac{it}{2}(\lambda-\lambda^{-1}))N)^{\frac{1}{2}}\\&=&
\mid(\mathrm{Re}\lambda)H_{\infty}+\frac{it}{2}(\lambda-\lambda^{-1})\mid\\&=&\mid\mathrm{Re}\,(\lambda
H_{\infty}+\frac{it}{2}(\lambda-\lambda^{-1}))\mid,
\end{eqnarray*}
having in consideration that, as $\Lambda$ is unit,
$\frac{it}{2}(\lambda-\lambda^{-1})$ is real.
\end{proof}

It is interesting to remark that, for $t=0$, the deformations
corresponding to the parameters $i$ and $-i$ are minimal surfaces
and, therefore, Willmore surfaces, even when $\Lambda$ is not a
Willmore surface.
\newline

\textbf{The isothermic spectral deformation.} Cf. \cite{SD}, the
isothermic spectral deformation preserves the class of CMC surfaces
in $3$-dimensional space-forms. Next we establish it in our setting
and investigate how the space-form and the mean curvature change
with this deformation.

First of all, note that, if $H_{\infty}\neq 0$, then $\eta_{\infty}=
H_{\infty}^{-1}q_{\infty}$, and, therefore, the fact that
$q_{\infty}$ is a multiplier to $\Lambda$ ensures that
$$d^{\mathcal{D}}*\mathcal{N}\neq 2[\eta_{\infty}\wedge *\mathcal{N}],$$
$\eta_{\infty}$ is not a multiplier to $\Lambda$. If, on the other
hand, $\Lambda$ is minimal in $S_{v_{\infty}}$, then the set of
multipliers to $\Lambda$ is the vector space $\langle
*\eta_{\infty}\rangle_{\R}\neq 0$, in which case we conclude, yet
again, that $\eta_{\infty}$ is not a multiplier to $\Lambda$. There
is, therefore, no risk of ambiguity on the notation
$d_{\eta_{\infty}}^{t}$.

Fix $t\in\R$ and
$\phi^{t}_{\eta_{\infty}}:(\underline{\R}^{4,1},d_{\eta_{\infty}}^{t})\rightarrow
(\underline{\R}^{4,1},d)$ an isomorphism. Set
$$v_{\eta_{\infty}}^{t}:=\phi^{t}_{\eta_{\infty}}(v_{\infty}+\frac{t}{2}\,N).$$
\begin{Lemma}
$v_{\eta_{\infty}}^{t}$ is a non-zero constant section of
$\underline{\R}^{4,1}$.
\end{Lemma}
\begin{proof}
The orthogonal projection of
$(\phi^{t}_{\eta_{\infty}})^{-1}v_{\eta_{\infty}}^{t}$ onto $S$ is
$v_{\infty}^{T}$, which is non-zero, cf. \eqref{eq:vinfTnonzero}.
Thus $v_{\eta_{\infty}}^{t}$ is non-zero. On the other hand, the
constancy of $v_{\infty}$, together with the fact that
$\eta_{\infty}$ vanishes on $S^{\perp}$, gives
\begin{eqnarray*}
dv_{\eta_{\infty}}^{t}&=&\phi^{t}_{\eta_{\infty}}(d^{t}_{\eta_{\infty}}v_{\eta_{\infty}}^{t})\\&=&
\phi^{t}_{\eta_{\infty}}(dv_{\infty}+\frac{t}{2}\,dN+t\eta_{\infty}v_{\infty}+\frac{t^{2}}{2}\,\eta_{\infty}N)\\&=&
\phi^{t}_{\eta_{\infty}}(\frac{t}{2}\,(dN,v_{\infty})\sigma_{\infty}).
\end{eqnarray*}
The constancy of $H_{\infty}$, and, in particular,
\eqref{eq:dNvinfperp0}, establishes then the constancy of
$v_{\eta_{\infty}}^{t}$, completing the proof.
\end{proof}

The isothermic spectral deformation provides, in particular, a
deformation of CMC surfaces in $3$-space. In fact:
\begin{thm}\label{isospCMC}
The isothermic $(t,\eta_{\infty})$-transformation of $\Lambda$ is a
CMC surface in $S_{v_{\eta_{\infty}}^{t}}$.
\end{thm}

\begin{proof}
The mean curvature $H^{t}_{\eta_{\infty}}$ of
$\phi^{t}_{\eta_{\infty}}\Lambda$ in the space-form
$S_{v_{\eta_{\infty}}^{t}}$ is given by
$H^{t}_{\eta_{\infty}}=(\pi_{t}^{\perp}v_{\eta_{\infty}}^{t},\pi_{t}^{\perp}v_{\eta_{\infty}}^{t})^{\frac{1}{2}}$,
for $\pi_{t}^{\perp}$ the orthogonal projection of
$\underline{\R}^{4,1}$ onto the normal bundle to
$S_{\phi^{t}_{\eta_{\infty}}\Lambda}=\phi^{t}_{\eta_{\infty}}S$,
and, therefore,
$$(H^{t}_{\eta_{\infty}})^{2}=(v_{\infty}^{\perp}+\frac{t}{2}\,N,v_{\infty}^{\perp}+\frac{t}{2}\,N)=
H_{\infty}^{2}+t(v_{\infty}^{\perp},N)+\frac{t^{2}}{4}=(H_{\infty}+\frac{t}{2})^{2}.$$
\end{proof}
In the proof of Theorem \ref{isospCMC}, we have verified, in
particular, that:
\begin{prop}
Let $\Lambda^{t}_{\eta_{\infty}}$ be the isothermic
$(t,\eta_{\infty})$-transformation of $\Lambda$. The mean curvature
$H_{\eta_{\infty}}^{t}$ of $\Lambda^{t}_{\eta_{\infty}}$ in
$S_{v_{\eta_{\infty}}^{t}}$ relates to the mean curvature of
$\Lambda$ in $S_{v_{\infty}}$ by
$$(H^{t}_{\eta_{\infty}})^{2}=(H_{\infty}+\frac{t}{2})^{2}.$$
\end{prop}

\begin{rem}
Let $k^{t}_{\infty}$ denote the curvature of $S_{v_{\infty}^{t}}$.
According to \eqref{eq:gdeisotharsnf}, the family
$\phi^{t}_{\eta_{\infty}}\sigma_{\infty}$, on $t\in\R$, constitutes
an isothermic deformation of $\sigma_{\infty}$ with
$$k^{t}_{\eta_{\infty}}+(H^{t}_{\eta_{\infty}})^{2}=-(v_{\eta_{\infty}}^{t},v_{\eta_{\infty}}^{t})+(v_{\infty}^{\perp}+\frac{t}{2}N,v_{\infty}^{\perp}+\frac{t}{2}N)=-(v_{\infty}^{T},v_{\infty}^{T}),$$
independent of $t$.
\end{rem}
$\newline$

\textbf{The classical CMC spectral deformation.} The isometric
deformation of surfaces in $\R^{3}$ preserving the mean curvature,
or, equivalently,\footnote{Cf. Gauss's \emph{theorema egregium},
isometric deformation of surfaces in Euclidean $3$-space preserves
the Gaussian curvature, which, in view of \eqref{eq:mcineucld3} and
\eqref{GaussiancurvinEucli3}, makes clear that the mean curvature is
preserved by such a deformation if and only if so are both principal
curvatures.} both principal curvatures, was first studied by O.
Bonnet. Bonnet \cite{bonnet} (see also, for example, \cite{bobenko}
and \cite{chern}) proved that a CMC surface in Euclidean $3$-space
admits a (non-trivial) one-parameter family of isometrical
deformations preserving both principal curvatures. In \cite{SD}, F.
Burstall et al. present an action of $\C\backslash\{0\}$ on the
class of constant mean curvature surfaces in $3$-space, in terms of
the Hopf differential and the Schwarzian derivative. The particular
action of $S^{1}$ preserves the metric, the space-form and the mean
curvature - this is the classical CMC spectral deformation. In this
section, we present the classical CMC spectral deformation by means
of the action of a loop of flat metric connections on the class of
CMC surfaces in $3$-space.\newline

For each $\lambda\in S^{1}$, set
$$d^{\lambda}_{\infty}:=\mathcal{D}+\lambda\mathcal{N}^{1,0}+\lambda^{-1}\mathcal{N}^{0,1}+2(\lambda-1)q_{\infty}^{1,0}+2(\lambda^{-1}-1)q_{\infty}^{0,1},$$
defining a real connection on $(\underline{\R}^{4,1})^{\C}$. As
$\mathcal{D}$ is metric, and both $\mathcal{N}$ and $q_{\infty}$ are
skew-symmetric, $d^{\lambda}_{\infty}$ is metric too, for each
$\lambda$.
\begin{thm}
$d^{\lambda}_{\infty}$ is a flat connection, for each $\lambda\in
S^{1}$.
\end{thm}

\begin{proof}
The curvature tensor of $d^{\lambda}_{\infty}$ is given by
\begin{eqnarray*}
R^{d^{\lambda}_{\infty}}&=&(\lambda-\lambda^{-1})d^{\mathcal{D}}\mathcal{N}^{1,0}+2(\lambda-1)d^{\mathcal{D}}q^{1,0}_{\infty}+2(\lambda^{-1}-1)d^{\mathcal{D}}q^{0,1}_{\infty}\\
&& \mbox{}+(2-2\lambda)\,[\mathcal{N}^{1,0}\wedge q_{\infty}^{0,1}]
+(2-2\lambda^{-1})\,[\mathcal{N}^{0,1}\wedge
q_{\infty}^{1,0}]-(\lambda^{-1}+\lambda-2)[q^{1,0}_{\infty}\wedge
q_{\infty}^{0,1}],
\end{eqnarray*}
according to Codazzi and Gauss-Ricci equations. The fact that
$\Lambda$ is a CMC surface in $S_{v_{\infty}}$, and, therefore, a
$q_{\infty}$-constrained Willmore surface, establishes, in
particular,
$d^{\mathcal{D}}q_{\infty}^{1,0}=0=d^{\mathcal{D}}q_{\infty}^{0,1}$
and, as observed in Section \ref{CCWS}, $[q_{\infty}^{1,0}\wedge
q_{\infty}^{0,1}]=0$. The fact that $\Lambda$ is a CMC surface in
$S_{v_{\infty}}$ establishes, on the other hand,
$(\Lambda,\eta_{\infty})$ as isothermic, and, therefore,
$[\mathcal{N}\wedge\eta_{\infty}]=0$, cf. Lemma
\ref{etadetcvansisheszesesseparately}. Hence
$R^{d^{\lambda}_{\infty}}=0$ if and only if either $\lambda=\pm 1$
or $d^{\mathcal{D}}\mathcal{N}^{1,0}=2[\mathcal{N}^{1,0}\wedge
q_{\infty}^{0,1}]$. The fact that $\Lambda$ is a
$q_{\infty}$-constrained Willmore surface gives, on the other hand,
$d^{\mathcal{D}}(-i\mathcal{N}^{1,0}+i\mathcal{N}^{0,1})=-2i[q_{\infty}\wedge\mathcal{N}^{1,0}]+2i[q_{\infty}\wedge\mathcal{N}^{0,1}]$,
which completes the proof.
\end{proof}

Fix $\lambda\in S^{1}$ and
$\phi^{\lambda}_{\infty}:(\underline{\R}^{4,1},d^{\lambda}_{\infty})\rightarrow
(\underline{\R}^{4,1},d)$ an isomorphism. Set
$$v_{\infty}^{\lambda}:=\phi^{\lambda}_{\infty}v_{\infty}.$$
\begin{Lemma}
$v_{\infty}^{\lambda}$ is a non-zero constant section of
$\underline{\R}^{4,1}$.
\end{Lemma}
\begin{proof}
According to equations \eqref{eq:hagrt} and \eqref{eq:conseqof65},
together with the fact that $q_{\infty}$ vanishes on $S^{\perp}$,
$$d^{\lambda}_{\infty}v_{\infty}=\mathcal{D}v_{\infty}^{T}+\lambda\mathcal{N}^{1,0}v_{\infty}^{\perp}+\lambda^{-1}\mathcal{N}^{0,1}v_{\infty}^{\perp}+2(\lambda-1)q^{1,0}_{\infty}v_{\infty}^{T}+2(\lambda^{-1}-1)q^{0,1}_{\infty}v_{\infty}^{T},$$
and, consequently, by \eqref{eq:qNemCMC}, followed by
\eqref{eq:NperpDT},
$d^{\lambda}_{\infty}v_{\infty}=\mathcal{D}v_{\infty}^{T}+\mathcal{N}v_{\infty}^{\perp}=0$.
We conclude that $v_{\infty}^{\lambda}$ is constant:
$dv_{\infty}^{\lambda}=d(\phi^{\lambda}_{\infty}v_{\infty})=\phi^{\lambda}_{\infty}(d^{\lambda}_{\infty}v_{\infty})=0$.
Inequation \eqref{eq:vinfTnonzero} establishes
$v_{\infty}^{\lambda}$ as non-zero.
\end{proof}

As $\mathcal{N}\Lambda=0=q_{\infty}\Lambda$, given
$\sigma\in\Gamma(\Lambda)$,
\begin{equation}\label{eq:dchapeusigmadsigma}
d^{\lambda}_{\infty}\sigma=d\sigma,
\end{equation}
showing that $\Lambda$ is still a $d^{\lambda}_{\infty}$-surface,
or, equivalently, that the transformation
$\Lambda^{\lambda}_{\infty}$ of $\Lambda$, defined by the flat
metric connection $d^{\lambda}_{\infty}$, is still a surface.
Furthermore:
\begin{thm}\label{CMCclassicdef}
The transformation $\Lambda^{\lambda}_{\infty}$ of $\Lambda$ defined
by the flat metric connection $d^{\lambda}_{\infty}$ is a CMC
surface in
$$S_{v_{\infty}^{\lambda}}=\phi^{\lambda}_{\infty}S_{v_{\infty}}.$$
\end{thm}

Before proceeding to the proof of the theorem, observe that
$\Lambda$ and $\Lambda^{\lambda}_{\infty}$ share the central sphere
congruence. For that, first note that, according to equation
\eqref{eq:dchapeusigmadsigma}, given $\sigma\in\Gamma(\Lambda)$
never-zero,
\begin{equation}\label{eq:glambdainftyISg}
g_{\phi^{\lambda}_{\infty}\sigma}=g_{\sigma}^{d^{\lambda}_{\infty}}=g_{\sigma},
\end{equation}
establishing
$$\mathcal{C}_{\Lambda^{\lambda}_{\infty}}=\mathcal{C}_{\Lambda}.$$
Yet again, in view of equation \eqref{eq:dchapeusigmadsigma}, it
follows that, given a holomorphic chart $z$ of
$(M,\mathcal{C}_{\Lambda}^{d^{\lambda}_{\infty}})=(M,\mathcal{C}_{\Lambda})$,
$(d^{\lambda}_{\infty})_{\delta_{\bar{z}}}(d^{\lambda}_{\infty})_{\delta_{z}}\sigma=(d^{\lambda}_{\infty})_{\delta_{\bar{z}}}\sigma_{z}=\sigma_{z\bar{z}}+\lambda^{-1}\pi_{S^{\perp}}\sigma_{z\bar{z}}+2(\lambda^{-1}-1)q_{\infty}^{0,1}\sigma_{z}$
and, therefore,
$(d^{\lambda}_{\infty})_{\delta_{\bar{z}}}(d^{\lambda}_{\infty})_{\delta_{z}}\sigma=\sigma_{z\bar{z}}$,
as $q_{\infty}^{0,1}\in\Omega^{1}(\Lambda\wedge\Lambda^{1,0})$. We
conclude that $S^{d^{\lambda}_{\infty}}=S$ and, ultimately,
according to \eqref{eq:csc}, that
\begin{equation}\label{eq:CMCdefpreservescsc}
S_{\phi^{\lambda}_{\infty}\Lambda}=\phi^{\lambda}_{\infty}S.
\end{equation}

Now we proceed to the proof of Theorem \ref{CMCclassicdef}.
\begin{proof}
According to \eqref{eq:CMCdefpreservescsc}, the mean curvature
$H^{\lambda}_{\infty}$ of $\phi^{\lambda}_{\infty}\Lambda$ in
$S_{v_{\infty}^{\lambda}}$ is given by
$H^{\lambda}_{\infty}=(\pi_{\lambda}^{\perp}v_{\infty}^{\lambda},\pi_{\lambda}^{\perp}v_{\infty}^{\lambda})^{\frac{1}{2}}$,
for $\pi_{\lambda}^{\perp}$ the orthogonal projection of
$\underline{\R}^{4,1}$ onto $\phi^{\lambda}_{\infty}S^{\perp}$, and,
therefore,
$H^{\lambda}_{\infty}=(v_{\infty}^{\perp},v_{\infty}^{\perp})^{\frac{1}{2}}=H_{\infty}$.
\end{proof}
In the proof of Theorem \ref{CMCclassicdef}, we have, in particular,
verified that:
\begin{prop}\label{mcpreservedbyclassicalCMCdeform}
The mean curvature $H_{\infty}^{\lambda}$ of
$\Lambda^{\lambda}_{q_{\infty}}$ in $S_{v_{\infty}^{\lambda}}$
coincides with the mean curvature of $\Lambda$ in $S_{v_{\infty}}$,
$$H_{\infty}^{\lambda}=H_{\infty}.$$
\end{prop}

The loop of flat metric connections $d^{\lambda}_{\infty}$ defines a
conformal $S^{1}$-deformation of $\Lambda$ into CMC surfaces in a
fixed space-form, preserving the mean curvature.
\begin{rem}
The family $\phi^{\lambda}_{\infty}\sigma_{\infty}$, with
$\lambda\in S^{1}$, is a family of isometrical deformations of
$\sigma_{\infty}$ in a fixed space-form, preserving the mean
curvature.
\end{rem}

We complete this section by verifying that the deformation defined
by the loop of flat metric connections $d^{\lambda}_{\infty}$ is the
classical CMC spectral deformation, described in \cite{SD} in terms
of the Hopf differential and the Schwarzian derivative. Fix a
holomorphic chart $z$ of
$(M,\mathcal{C}_{\Lambda^{\lambda}_{\infty}})=(M,\mathcal{C}_{\Lambda})$.
According to \eqref{eq:glambdainftyISg},
$g_{\phi_{\infty}^{\lambda}\sigma^{z}}=g_{z}$, showing that
$\phi_{\infty}^{\lambda}\sigma^{z}$ is the normalized section of
$\phi_{\infty}^{\lambda}\Lambda$ with respect to $z$. For
simplicity, write  $q_{\infty}^{z}$ for $(q_{\infty})^{z}$. In view
of \eqref{eq:dchapeusigmadsigma},
\begin{eqnarray*}
(\phi_{\infty}^{\lambda}\sigma^{z})_{zz}&=&(\phi_{\infty}^{\lambda}\sigma_{z}^{z})_{z}\\&=&
\phi_{\infty}^{\lambda}((d^{\lambda}_{\infty})_{\delta_{z}}\sigma_{z}^{z})\\&=&
\phi_{\infty}^{\lambda}(\mathcal{D}_{\delta_{z}}\sigma_{z}^{z}+\lambda\mathcal{N}_{\delta_{z}}\sigma_{z}^{z}+2(\lambda-1)(q_{\infty})_{\delta_{z}}\sigma_{z}^{z})
\\&=&\phi_{\infty}^{\lambda}(\pi_{S}\sigma_{zz}^{z}+\lambda\pi_{S^{\perp}}\sigma_{zz}^{z}-(\lambda-1)\,q^{z}_{\infty}\sigma^{z})
\end{eqnarray*}
and, ultimately,
$$(\phi_{\infty}^{\lambda}\sigma^{z})_{zz}=-\frac{1}{2}\,(c^{z}+2(\lambda-1)\,q^{z}_{\infty})\,\phi_{\infty}^{\lambda}\sigma^{z}+\lambda\phi_{\infty}^{\lambda}k^{z}.$$
We conclude that $(k^{\lambda}_{\infty})^{z}$ and
$(c^{\lambda}_{\infty})^{z}$, the Hopf differential and the
Schwarzian derivative, respectively, of
$\phi_{\infty}^{\lambda}\Lambda$ with respect to $z$, relate to
those of $\Lambda$ by
$$(k^{\lambda}_{\infty})^{z}=\lambda\phi_{\infty}^{\lambda}k^{z},\,\,\,\,\,\,(c^{\lambda}_{\infty})^{z}=c^{z}+2(\lambda-1)\,q^{z}_{\infty}.$$
By Lemma \ref{thmonHopfeSchwarz}, the conclusion follows. $\newline$

\textbf{Isothermic vs. constrained Willmore vs. classical CMC
spectral deformations.} How are the constrained Willmore, isothermic
and classical CMC spectral deformations of a CMC surface in
$3$-space related? In this section, we compare the families of flat
metric connections that define each of these deformations. We
observe, in particular, that the classical CMC spectral deformation
can be obtained as composition of isothermic and constrained
Willmore spectral deformation and that, in the particular case of
minimal surfaces, the classical CMC spectral deformation coincides,
up to reparametrization, with the constrained Willmore spectral
deformation corresponding to the zero multiplier.\newline

We start by introducing some terminology. Note that, according to
\eqref{eq:komega}, given $z$ and $\omega$ holomorphic charts of
$(M,\mathcal{C}_{\Lambda})$, $k^{z}$ vanishes if and only if
$k^{\omega}$ does. We refer to the points where the Hopf
differential of $\Lambda$ vanishes as \emph{umbilic points} of
$\Lambda$, in coherence with the classical notion of umbilic point
of a surface in Euclidean $3$-space, cf. Proposition
\ref{umbilicvshopfprop} below, in Appendix \ref{appHopf+umbilics}
(having in consideration that $\Lambda$ is isothermic).

Fix $z$ a holomorphic chart of $M$ and let $\sigma^{z}$ be the
normalized section of $\Lambda$ with respect to $z$. Given
$t,t'\in\R$ and $\lambda\in S^{1}$,
$d^{\lambda}_{q_{\infty}^{t}}=d^{t'}_{\eta_{\infty}}$ forces, in
particular,
$d^{\lambda}_{q_{\infty}^{t}}\sigma^{z}_{z}=d^{t'}_{\eta_{\infty}}\sigma^{z}_{z}$,
or, equivalently,
$$\mathcal{D}_{\delta_{z}}\sigma_{z}^{z}+\lambda^{-1}\mathcal{N}_{\delta_{z}}\sigma_{z}^{z}+(\lambda^{-2}-1)(q_{\infty}^{t})_{\delta_{z}}\sigma_{z}^{z}=
\mathcal{D}_{\delta_{z}}\sigma_{z}^{z}+\mathcal{N}_{\delta_{z}}\sigma_{z}^{z}+t'(\eta_{\infty})_{\delta_{z}}\sigma_{z}^{z},$$
or, yet again, equivalently,
$$(\lambda^{-1}-1)k^{z}=0,\,\,\,\,(\lambda^{-2}-1)(q_{\infty}^{t})_{\delta_{z}}\sigma_{z}^{z}=t'(\eta_{\infty})_{\delta_{z}}\sigma_{z}^{z},$$
in view of the fact that
$\mathrm{Im}\,q_{\infty}^{t},\mathrm{Im}\,\eta_{\infty}\subset S$.
By \eqref{eq:eta10LambdaLambda01} and according to Remark
\ref{e10determinedby},
$(\eta_{\infty})_{\delta_{z}}\sigma_{z}^{z}\neq 0$. We conclude
that, away from umbilics,
$d^{\lambda}_{q_{\infty}^{t}}=d^{t'}_{\eta_{\infty}}$ holds if and
only if $\lambda=1$ and $t'=0$, in which case,
$$d^{1}_{q_{\infty}^{t}}=d=d^{0}_{\eta_{\infty}},$$for all $t\in\R$.
Similarly, we conclude that, away from umbilics, given
$\lambda,\lambda'\in S^{1}$ and $t,t'\in\R$,
$d^{\lambda}_{q_{\infty}^{t}}=d^{\lambda'}_{q_{\infty}^{t'}}$ if and
only if $\lambda=\lambda'$ and either $\lambda=\pm 1$ or $t=t'$.
Lastly, if $d^{\lambda}_{\infty}=d^{\lambda'}_{q_{\infty}^{t}}$, for
some $\lambda,\lambda'\in S^{1}$ and $t\in\R$, then
$$2(\lambda-1)(q_{\infty})_{\delta_{z}}\sigma_{z}^{z}=((\lambda')^{-2}-1)(H_{\infty}-it)(\eta_{\infty})_{\delta_{z}}\sigma_{z}^{z}$$
and either $\lambda'=\lambda^{-1}$ or $k^{z}=0$; and, therefore,
away from umbilics, $\lambda'=\lambda^{-1}$ and either $\lambda=1$
or
\begin{equation}\label{eq:33q2wesdfghjkl8waq777777sw4}
(2H_{\infty}-(\lambda+1)(H_{\infty}-it))(\eta_{\infty})_{\delta_{z}}\sigma_{z}^{z}=0.
\end{equation}
In its turn, unless $\lambda=-1$, equation
\eqref{eq:33q2wesdfghjkl8waq777777sw4} forces $t$ to be
$$t_{\lambda}:=iH_{\infty}\frac{1-\lambda}{1+\lambda}\in\R.$$
Conversely, one verifies that, for all $\lambda\neq\pm 1$,
$$d^{\lambda}_{\infty}=d^{\lambda^{-1}}_{q_{\infty}^{t_{\lambda}}},$$
which, in fact, extends to all $\lambda\neq -1$:
$d^{1}_{\infty}=d^{1}_{q_{\infty}^{0}}$. One verifies that, given
$t\in\R$, we have $d^{-1}_{\infty}=d^{-1}_{q_{\infty}^{t}}$ if and
only if $H_{\infty}=0$. Note, in particular, that, if
$H_{\infty}=0$, then $d_{\infty}^{\lambda}=d^{\lambda^{-1}}_{0}$,
for all $\lambda\in S^{1}$.

For each $t\in\R$, we have a multiplier $q_{\infty}^{t}$ and,
therefore, a loop of $q_{\infty}^{t}$-constrained Willmore spectral
deformation parameters $\lambda\in S^{1}$. The set of constrained
Willmore spectral deformation parameters is described, in this way,
as the cylinder of points $(t,\lambda)$ with $t\in\R$ and
$\lambda\in S^{1}$. A transformation corresponding to a parameter in
the line $(t,1)$, with $t\in\R$, is trivial, and so is the
transformation corresponding to the parameter $0$ in the real line
of isothermic spectral deformation parameters. The transformations
corresponding to parameters in the line $(t,-1)$ do not depend on
$t\in\R$. For minimal surfaces, the classical CMC spectral
deformation coincides, up to reparametrization, with the constrained
Willmore spectral deformation corresponding to the zero multiplier.
Furthermore, for a general surface, the classical CMC spectral
deformation of parameter other than $-1$ can be obtained as
constrained Willmore spectral deformation of parameters in the curve
$(t_{\lambda},\lambda^{-1})$, with $\lambda\in S^{1}$. As for the
classical CMC deformation of parameter $-1$ of non-minimal surfaces,
it is not clear that it can be obtained by constrained Willmore
spectral deformation alone. However, such transformations can be
obtained as composition of isothermic and constrained Willmore
spectral deformation, as we verify next. Given $\lambda\in S^{1}$,
\begin{eqnarray*}
d_{\infty}^{\lambda^{-1}}&=&
d^{\lambda}_{q_{\infty}}+(-\lambda^{-2}+2\lambda^{-1}-1)\,q^{1,0}_{\infty}+(-\lambda^{2}+2\lambda-1)\,q^{0,1}_{\infty}\\&=&
d^{\lambda}_{q_{\infty}}+H_{\infty}(2-\lambda-\lambda^{-1})(\lambda^{-1}\eta_{\infty}^{1,0}+\lambda\,\eta_{\infty}^{0,1})
\end{eqnarray*}
and, ultimately,
\begin{equation}\label{eq:eqwiththe3connections}
d_{\infty}^{\lambda^{-1}}=d^{\lambda}_{q_{\infty}}+2H_{\infty}(1-\mathrm{Re}\,\lambda)\,\eta_{\infty}^{\lambda},
\end{equation}
for $\eta_{\infty}^{\lambda}:=(\eta_{\infty})_{\lambda}$ as in
Section \ref{isoCWtransforms}. Fix $\lambda\in S^{1}$ and
$\phi^{\lambda}_{q_{\infty}}:(\underline{\R}^{4,1},d^{\lambda}_{q_{\infty}})\rightarrow
(\underline{\R}^{4,1},d)$ an isomorphism. Cf. Section
\ref{isoCWtransforms},
$(\phi^{\lambda}_{q_{\infty}}\Lambda,\hat{\eta}_{\infty}^{\lambda})$
is isothermic, for
$$\hat{\eta}_{\infty}^{\lambda}:=\mathrm{Ad}_{\phi^{\lambda}_{q_{\infty}}}\eta_{\infty}^{\lambda}.$$
Set
$$r_{\lambda}:=2H_{\infty}(1-\mathrm{Re}\,\lambda)$$
and fix an isomorphism
$\phi^{r}_{\lambda}:(\underline{\R}^{4,1},d+r_{\lambda}\hat{\eta}_{\infty}^{\lambda})\rightarrow
(\underline{\R}^{4,1},d)$. Following equation
\eqref{eq:eqwiththe3connections}, we get that
$$\phi^{r}_{\lambda}\phi^{\lambda}_{q_{\infty}}\circ d^{\lambda^{-1}}_{\infty}=
\phi^{r}_{\lambda}\circ(d\circ
\phi^{\lambda}_{q_{\infty}}+r_{\lambda}\hat{\eta}_{\infty}^{\lambda}\phi^{\lambda}_{q_{\infty}})=
d\circ\phi^{r}_{\lambda}\phi^{\lambda}_{q_{\infty}},$$ the isometry
$$\phi^{r}_{\lambda}\phi^{\lambda}_{q_{\infty}}:(\underline{\R}^{4,1},d^{\lambda^{-1}}_{\infty})\rightarrow
(\underline{\R}^{4,1},d)$$ preserves connections. We conclude that
the isothermic
$(r_{\lambda},\hat{\eta}_{\infty}^{\lambda})$-transformation of the
constrained Willmore spectral deformation of parameter $\lambda$ of
$\Lambda$ corresponding to the multiplier $q_{\infty}$ coincides
with the classical CMC spectral deformation of parameter
$\lambda^{-1}$ of $\Lambda$,
$$\Lambda^{{\lambda^{-1}}}_{\infty}=(\Lambda^{\lambda}_{q_{\infty}})^{r_{\lambda}}_{\hat{\eta}_{\infty}^{\lambda}}.$$

\subsection{Constrained Willmore B\"{a}cklund transformation of CMC surfaces in $3$-space}\label{BTCMC}
Characterized as the class of constrained Willmore surfaces in
$3$-dimen\-sio\-nal space-forms admitting a conserved quantity, the
class of CMC surfaces in $3$-space is known to be preserved by
constrained Willmore B\"{a}cklund transformation, for special
choices of parameters. In this section, we verify that both the
space-form and the mean curvature are preserved under this
transformation. We remark on the fact that constrained Willmore
B\"{a}cklund transformation of non-minimal CMC surfaces in $3$-space
preserves conformal curvature line coordinates.\newline

Suppose $\Lambda$ has constant mean curvature $H_{\infty}$ in
$S_{v_{\infty}}$. Then, for each $t\in\R$, $\Lambda$ is a
$q_{\infty}^{t}$-constrained Willmore surface admitting
$p_{\infty}^{t}(\lambda)$ as a $q_{\infty}^{t}$-conserved quantity.
Fix $t\in\R$. Suppose $\alpha_{t},L^{\alpha_{t}}$ are constrained
Willmore B\"{a}cklund transformation parameters to $\Lambda$,
corresponding to the multiplier $q_{\infty}^{t}$, satisfying the
condition
\begin{equation}\label{eq:pinalphaLalph}
p_{\infty}^{t}(\alpha_{t})\perp L^{\alpha_{t}},
\end{equation}
and let $\Lambda^{*_{t}}$ be the constrained Willmore B\"{a}cklund
transform of $\Lambda$ of parameters $\alpha_{t}, L^{\alpha_{t}}$.
Suppose $\Lambda^{*_{t}}$ immerses. Let $r^{*_{t}}$ denote
$r^{\alpha_{t}}_{L^{\alpha_{t}}}$. According to Theorem
\ref{BTsofCMCs}, $\Lambda^{*_{t}}$ is a
$(q_{\infty}^{t})^{*}$-constrained Willmore surface admitting
$$(p_{\infty}^{t})^{*}(\lambda)=r^{*_{t}}(1)^{-1}r^{*_{t}}(\lambda)p_{\infty}^{t}(\lambda)$$
as a $(q_{\infty}^{t})^{*}$-conserved quantity, and, therefore, a
CMC surface in some space-form. Furthermore, according to
Proposition \ref{CW+areCMC}:
\begin{thm}
Suppose $\Lambda$ is a CMC surface in $S_{v_{\infty}}$. Suppose
$\alpha_{t}, L^{\alpha_{t}}$ are constrained Willmore B\"{a}cklund
transformation parameters to $\Lambda$, corresponding to the
multiplier $q_{\infty}^{t}$, verifying condition
\eqref{eq:pinalphaLalph}. Then the constrained Willmore B\"{a}cklund
transform of $\Lambda$ of parameters $\alpha_{t}, L^{\alpha_{t}}$
still is a CMC surface in $S_{v_{\infty}}$, provided that it
immerses.
\end{thm}

Next we investigate how the mean curvature changes under this
transformation. Let $S^{*_{t}}$ be the central sphere congruence of
$\Lambda^{*_{t}}$ and $T_{*_{t}}$ and $\perp_{*_{t}}$ indicate the
orthogonal projections of $\underline{\R}^{4,1}$ onto $S^{*_{t}}$
and $(S^{*_{t}})^{\perp}$, respectively.
\begin{prop}\label{pinf*}
Set $N^{*_{t}}:=-r^{*_{t}}(1)^{-1}N$. Then
\begin{equation}\label{eq:vinfstartHinfNstart}
v_{\infty}^{\perp_{*_{t}}}=H_{\infty}N^{*_{t}}
\end{equation}
and
$$(p_{\infty}^{t})^{*}(\lambda)=\lambda^{-1}\frac{1}{2}\,(H_{\infty}-it)N^{*_{t}}+v_{\infty}^{T_{*_{t}}}+\lambda\frac{1}{2}\,
(H_{\infty}+it)N^{*_{t}},$$ for all $\lambda$.
\end{prop}

Before proceeding to the proof of the proposition:

\begin{rem}
Let $\rho$ denote reflection across $S$. According to
\eqref{eq:r0rinfpreserveVperp}, $r^{*_{t}}(0)\vert_{ S^{\perp}}$ and
$r^{*_{t}}(\infty)\vert_{S^{\perp}}$ are orthogonal transformations
of $S^{\perp}$ and, therefore, as $S^{\perp}$ is a
$\mathrm{rank}\,1$ bundle,
$r^{*_{t}}(0)\vert_{S^{\perp}},r^{*_{t}}(\infty)\vert_{S^{\perp}}\in\{\pm
\,I\}$. On the other hand,
$$r^{*_{t}}(0)\vert_{S^{\perp}}=q_{\overline{\alpha_{t}}\,^{-1},\overline{L^{\alpha_{t}}}}\,(0)\vert_{S^{\perp}}
=I\left\{
\begin{array}{ll} -1 & \mbox{$\mathrm{on}\,\,S^{\perp}\cap(\overline{L^{\alpha_{t}}}\oplus\rho \overline{L^{\alpha_{t}}})$}\\ 1 &
\mbox{$\mathrm{on}\,\,S^{\perp}\cap(\overline{L^{\alpha_{t}}}\oplus\rho
\overline{L^{\alpha_{t}}})^{\perp}$}\end{array}\right.,$$so that, if
$r^{*_{t}}(0)\vert_{S^{\perp}}$ were identity, then
$S^{\perp}\cap(\overline{L^{\alpha_{t}}}\oplus\rho
\overline{L^{\alpha_{t}}})^{\perp}$ would be $S^{\perp}$ and,
therefore, $L^{\alpha_{t}}\subset S$, in which case, $\rho
L^{\alpha_{t}}=L^{\alpha_{t}}$, which, as $L^{\alpha_{t}}$ is null,
contradicts the fact that $\rho L^{\alpha_{t}}\cap
(L^{\alpha_{t}})^{\perp}=\{0\}$. Hence
$r^{*_{t}}(0)\vert_{S^{\perp}}=-I$. Similarly, the fact that
$\tilde{L}^{\alpha_{t}}$ is null and $\rho
\tilde{L}^{\alpha_{t}}\cap (\tilde{L}^{\alpha_{t}})^{\perp}=\{0\}$
establishes
$r^{*_{t}}(\infty)\vert_{S^{\perp}}=p_{\alpha_{t},\tilde{L}^{\alpha_{t}}}(\infty)\vert_{S^{\perp}}=-I$.
Thus
\begin{equation}\label{eq:r(0),r(infty)incodimension1}
r^{*_{t}}(0)\vert_{S^{\perp}}=-I=r^{*_{t}}(\infty)\vert_{S^{\perp}}.
\end{equation}
\end{rem}

Now we proceed to the proof of Proposition \ref{pinf*}.
\begin{proof}
Write
$(p_{\infty}^{t})^{*}(\lambda)=\lambda^{-1}v_{t}+v_{0}^{t}+\lambda
\overline{v_{t}}$ with $v_{0}^{t}\in\Gamma(S^{*_{t}})$ and
$v_{t}\in\Gamma((S^{*_{t}})^{\perp})$. The fact that
$v_{\infty}=(p_{\infty}^{t})^{*}(1)=v_{t}+v_{0}^{t}+\overline{v_{t}}$
establishes $v_{0}^{t}=v_{\infty}^{T_{*_{t}}}$ and
\begin{equation}\label{eq:vinfperp*t}
v_{t}+\overline{v_{t}}=v_{\infty}^{\perp_{*_{t}}}.
\end{equation}
On the other hand,
\begin{eqnarray*}
v_{t}&=&\mathrm{lim}_{\lambda\rightarrow
0}\lambda(\lambda^{-1}v_{t}+v_{0}^{t}+\lambda
\overline{v_{t}})\\&=&\mathrm{lim}_{\lambda\rightarrow
0}\lambda(p_{\infty}^{t})^{*}(\lambda)\\&=&r^{*_{t}}(1)^{-1}r^{*_{t}}(0)\mathrm{lim}_{\lambda\rightarrow
0}\lambda p_{\infty}^{t}(\lambda)\\&=&
\frac{1}{2}\,r^{*_{t}}(1)^{-1}r^{*_{t}}(0)(H_{\infty}-it)N
\end{eqnarray*}
and, therefore, by \eqref{eq:r(0),r(infty)incodimension1},
$$v_{t}=\frac{1}{2}\,(H_{\infty}-it)N^{*_{t}}.$$
Next let $K^{t}\in\Gamma(O(\underline{\C}^{n+2}))$ be $K$, as
defined in Proposition \ref{rstarvshatrstar}, for
$\alpha=\alpha_{t}$ and $\beta=\overline{\alpha_{t}}\,^{-1}$.
According to \eqref{eq:rhoKrhoisK}, $K^{t}$ preserves $S^{\perp}$.
Hence $K^{t}N$ is a unit section of $S^{\perp}$, which equation
\eqref{eq:overlineKisK} ensures to be real. Thus $K^{t}N=\pm N$. If
$K^{t}N=-N$, then, according to equation \eqref{eq:conjugado de
rstar(1)inv},
$\overline{N^{*_{t}}}=-r^{*_{t}}(1)^{-1}K^{t}N=-N^{*_{t}}$ and,
therefore, $iN^{*_{t}}$ is a real section of $(S^{*_{t}})^{\perp}$
with $(iN^{*_{t}},iN^{*_{t}})=-1$, which contradicts the fact that
the metric in $(S^{*_{t}})^{\perp}$ has no signature. We conclude
that $K^{t}N=N$. It follows that
$\overline{N^{*_{t}}}=-r^{*_{t}}(1)^{-1}K^{t}N=N^{*_{t}}$,
$\overline{N^{*_{t}}}$ is real, and, therefore, by
\eqref{eq:vinfperp*t},
$v_{\infty}^{\perp_{*_{t}}}=H_{\infty}N^{*_{t}}$, completing the
proof.
\end{proof}

Equation \eqref{eq:vinfstartHinfNstart} establishes immediately
that:
\begin{prop}\label{Hsparaqinft}
The mean curvature $H_{\infty}^{*_{t}}$ of $\Lambda^{*_{t}}$ in
$S_{v_{\infty}}$ relates to the mean curvature of $\Lambda$ in
$S_{v_{\infty}}$ by
$$H_{\infty}^{*_{t}}=H_{\infty}.$$
\end{prop}

Let $\eta_{\infty}^{N^{*_{t}}}$ be the \textit{canonical} form
establishing the isothermic condition of $\Lambda^{*_{t}}$ as a CMC
surface in $S_{v_{\infty}}$, as established in Section
\ref{CMCisoCW},
$\eta_{\infty}^{N^{*_{t}}}:=\frac{1}{2}\,\sigma_{\infty}^{*_{t}}\wedge
dN^{*_{t}}$, for $\sigma_{\infty}^{*_{t}}$ the surface in
$S_{v_{\infty}}$ defined by $\Lambda^{*_{t}}$. Let
$q_{\infty}^{*_{t}}$ be the \textit{canonical} multiplier to
$\Lambda^{*_{t}}$ as a CMC surface in $S_{v_{\infty}}$, i.e.,
$q_{\infty}^{*_{t}}:=H_{\infty}\eta_{\infty}^{N^{*_{t}}}$. For each
$t'\in\R$, let $(q_{\infty}^{*_{t}})^{t'}$ denote the multiplier
$q_{\infty}^{*_{t}}+t'*\eta_{\infty}^{N^{*_{t}}}$ to
$\Lambda^{*_{t}}$ and $(p_{\infty}^{*_{t}})^{t'}$ be the
\textit{canonical} $(q_{\infty}^{*_{t}})^{t'}$-conserved quantity of
$\Lambda^{*_{t}}$, as established in Proposition \ref{pinftyallt},
$(p_{\infty}^{*_{t}})^{t'}:=\lambda^{-1}\frac{1}{2}\,(H_{\infty}-it')N^{*_{t}}+
v_{\infty}^{T_{*_{t}}}+\lambda\frac{1}{2}\,(H_{\infty}+it')N^{*_{t}}$.
Then
$$(p_{\infty}^{t})^{*}=(p_{\infty}^{*_{t}})^{t}$$ and, therefore,
according to Remark \ref{CQeqdeterminesq},
$$(q_{\infty}^{t})^{*}=(q_{\infty}^{*_{t}})^{t}.$$ By Proposition
\ref{quaddiffscoincide}, it follows that the quadratic differentials
$(q_{\infty}^{t})_{Q}$ and $((q_{\infty}^{*_{t}})^{t})_{Q}$
coincide, $(q_{\infty}^{t})_{Q}=((q_{\infty}^{*_{t}})^{t})_{Q}$.
Recall that $\Lambda$ and $\Lambda^{*_{t}}$ induce the same
conformal structure in $M$. We conclude that, given $z$ a
holomorphic chart of $(M,\mathcal{C}_{\Lambda}=
\mathcal{C}_{\Lambda^{*_{t}}})$,
$$(q_{\infty}^{t})^{z}=((q_{\infty}^{*_{t}})^{t})^{z}$$
and, therefore, by \eqref{eq:qinftzvsqinfz}, that, if $\Lambda$
(respectively, $\Lambda^{*_{t}}$) is not minimal in
$S_{v_{\infty}}$, then
$$(q_{\infty})^{z}=(q_{\infty}^{*_{t}})^{z}.$$By Lemma
\ref{isoviakreal}, we conclude, ultimately, that, unless $\Lambda$
(respectively, $\Lambda^{*_{t}}$) is minimal in $S_{v_{\infty}}$,
$\Lambda$ and $\Lambda^{*_{t}}$ share conformal curvature line
coordinates (recall the close relationship between $q_{\infty}$
(respectively, $q_{\infty}^{*_{t}}$) and the Hopf differential of
$\Lambda$ (respectively, $\Lambda^{*_{t}}$), cf. observed in Section
\ref{CMCisoCW}).

\subsection{Isothermic Darboux transformation vs. constrained Willmore B\"{a}cklund transformation of CMC surfaces in $3$-space}

U. Hertrich-Jeromin and F. Pedit \cite{udo+pedit} proved that, for
special choices of parameters, the isothermic Darboux transformation
of CMC surfaces in Euclidean $3$-space preserves the constancy of
the mean curvature and, furthermore, the mean curvature. This is
also the case for constrained Willmore B\"{a}cklund transformation
of these surfaces, cf. Section \ref{BTCMC}. Moreover, as proven by
S. Kobayashi and J.-I. Inoguchi \cite{kobayashi}, isothermic Darboux
transformation of CMC surfaces in $\R^{3}$ is equivalent to
Bianchi-B\"{a}cklund transformation, the latter described in
\cite{rui} in terms of a dressing action, following the work of
Terng and Uhlenbeck \cite{uhlenbeck}, just like the constrained
Willmore B\"{a}cklund transformation presented above. How are these
transformations related? And what to say for general $3$-space? This
shall be the subject of further work. In \cite{quaternionsbook}, a
description of Darboux transformation of constant mean curvature
surfaces in Euclidean $3$-space is presented in the quaternionic
setting. It is based on the solution of a Riccati equation and it
displays a striking similarity with the Darboux transformation of
constrained Willmore surfaces in $4$-space presented in Chapter
\ref{surfacesinS4} below. We prove that all non-trivial Darboux
transforms of constrained Willmore surfaces can be obtained by
constrained Wilmore B\"{a}cklund transformation. We believe
isothermic Darboux transformation of a CMC surface in Euclidean
$3$-space can be obtained as a particular case of constrained
Willmore B\"{a}cklund transformation. $\newline$

\chapter{The special case of surfaces in
$4$-space}\label{surfacesinS4}

\markboth{\tiny{A. C. QUINTINO}}{\tiny{CONSTRAINED WILLMORE
SURFACES}}

This chapter is dedicated to the special case of surfaces in
$4$-space. Our approach is quaternionic, based on the model of the
conformal $4$-sphere on the quaternionic projective space, and
follows the work of F. Burstall et al. \cite{quaternionsbook}. We
identify $\mathbb{H}^{2}$ with $\C^{4}$ and provide
$\wedge^{2}\C^{4}$ with the real structure $\wedge^{2}j$ and with a
certain metric inducing a metric with signature $(5,1)$ on the space
of real vectors of $\wedge^{2}\C^{4}$. Via the Pl\"{u}cker
embedding, we identify a $j$-stable $2$-plane $L$ in $\C^{4}$ with
the real null line $\wedge^{2}L$ in
$(\mathrm{Fix}(\wedge^{2}j))^{\C}$, presenting, in this way, the
quaternionic projective space as a model for the conformal
$4$-sphere, $\mathbb{H}P^{1}\cong S^{4}$. Surfaces in $S^{4}$ are
described in this model as the immersed bundles of $j$-stable
$2$-planes in $\C^{4}$. We extend the Darboux transformation of
Willmore surfaces in $S^{4}$ presented in \cite{quaternionsbook},
based on the solution of a Riccati equation, to a transformation of
constrained Willmore surfaces in $4$-space.  We apply, yet again,
the dressing action presented in Chapter \ref{transformsofCW} to
define another transformation of constrained Willmore surfaces in
$4$-space, the \textit{untwisted B\"{a}cklund transformation},
referring then to the original one as the \textit{twisted
B\"{a}cklund transformation}. We verify that, when both are defined,
twisted and untwisted B\"{a}cklund transformations coincide. We
prove that constrained Willmore Darboux transformation of parameters
$\rho,T$ with $\rho
> 1$ is equivalent to untwisted B\"{a}cklund transformation of
parameters $\alpha,L^{\alpha}$ with $\alpha^{2}$ real. Constrained
Willmore Darboux transformation of parameters $\rho,T$ with $\rho
\leq 1$ is trivial.\newline

\section{Surfaces in $S^{4}\cong\mathbb{H}P^{1}$}

\markboth{\tiny{A. C. QUINTINO}}{\tiny{CONSTRAINED WILLMORE
SURFACES}}

Consider the natural identification of $\mathbb{H}$ with $\R^{4}$
and then the natural identification of $\mathbb{H}^{2}$ with
$\langle 1,i\rangle^{4}=\C^{4}$. Provide $\wedge^{2}\C^{4}$ with the
real structure $\wedge^{2}j$. Define a metric on $\wedge^{2}\C^{4}$
by $(v_{1}\wedge v_{2},v_{3}\wedge v_{4}):=-\mathrm{det}(v_{1},
v_{2},v_{3},v_{4})$, for $v_{1}, v_{2},v_{3}, v_{4}\in\C^{4}$, with
$\mathrm{det}(v_{1}, v_{2},v_{3}, v_{4})$ denoting the determinant
of the matrix whose columns are the components of $v_{1},
v_{2},v_{3}$ and $v_{4}$, respectively, on the canonical basis of
$\C^{4}$. This metric induces a metric with signature $(5,1)$ on the
space of real vectors of $\wedge^{2}\C^{4}$,
$\mathrm{Fix}(\wedge^{2}j)=\R^{5,1}$. Via the Pl\"{u}cker embedding,
we identify a $j$-stable $2$-plane $L$ in $\C^{4}$ with the real
null line $\wedge^{2}L$ in $(\mathrm{Fix}(\wedge^{2}j))^{\C}$,
presenting, in this way, the quaternionic projective space
$\mathbb{H}P^{1}$ as a model for the conformal $4$-sphere, and
describing, in this model, surfaces in $S^{4}$ as the immersed
bundles $L\cong\wedge^{2}L:M\rightarrow S^{4}$ of $j$-stable
$2$-planes in $\C^{4}$.\newline

\subsection{Linear algebra}\label{LinAlg}
Consider the quaternions, the unitary $\R$-algebra $\mathbb{H}$
generated by $i,j,k$ with the relations
$$i^{2}=j^{2}=k^{2}=-1,$$
$$ij=-ji=k,\,\,jk=-kj=i,\,\,ki=-ik=j.$$
We identify the real vector space $\mathbb{H}$ in the obvious way
with $\R^{4}$: given $a,b,c,d\in\R$,
$$a+ib+jc+kd\cong (a,b,c,d);$$ and identify then $\mathbb{H}^{2}$ with $\C^{4}$ for
$\C:=\langle 1,i\rangle$, i.e., given $a,b,c,d, a',b',c',d'\in\R$,
$$((a,b,c,d),(a',b',c',d'))\cong (a+ia',b+ib',c+ic',d+id').$$
Under this identification, the left multiplication by the
quaternionic unit $j$ defines a complex anti-linear map
$j:\C^{4}\rightarrow\C^{4}$. The projective space $\mathbb{H}P^{1}$
of the quaternionic lines in $\mathbb{H}^{2}$ is, in this way,
described as the set of the $j$-stable $2$-planes in $\C^{4}$.
\begin{rem}
A $2$-plane $U$ in $\C^{4}$ is either $j$-stable or complementary to
$jU$: $U\cap jU$ is, obviously, $j$-stable, and, therefore,
$\mathrm{rank}_{\C}(U\cap jU)\in\{0,2\}$.
\end{rem}
Fix $\mathrm{det}\in\wedge ^{4}(\C^{4})^{*}\backslash\{0\}$ such
that
\begin{equation}\label{eq:detwedge4}
\overline{\mathrm{det}\circ\wedge^{4}j}=\mathrm{det}
\end{equation}
and that
\begin{equation}\label{eq:det>0}
\mathrm{det}(v_{1},v_{2},jv_{1},jv_{2})>0
\end{equation}
for $v_{1},v_{2},jv_{1},jv_{2}\in\C^{4}$ linearly independent. To
see that such a $\mathrm{det}$ exists, first observe that condition
\eqref{eq:detwedge4} is equivalent to the reality of
$\mathrm{det}(e_{1},e_{2},e_{3},e_{4})$, for
$e_{1},e_{2},e_{3}=je_{1},e_{4}=-je_{2}$ the canonical basis of
$\C^{4}$. This makes clear the existence of $\mathrm{det}$ in
$\wedge ^{4}(\C^{4})^{*}\backslash\{0\}$ satisfying condition
\eqref{eq:detwedge4}, determined up to a (non-zero) real scale. On
the other hand, for such a $\mathrm{det}$, given
$v_{1},v_{2},jv_{1},jv_{2}\in\C^{4}$ linearly independent,
$\mathrm{det}(v_{1},v_{2},jv_{1},jv_{2})$ is a non-zero real number,
whose sign does not, in fact, depend on the basis
$\{v_{1},v_{2},jv_{1},jv_{2}\}$ of $\C^{4}$. \footnote{The argument
above shows that a possible choice of such a $\mathrm{det}$ is the
one given by $\mathrm{det}(v_{1}\wedge v_{2},v_{3}\wedge
v_{4}):=-\mid v_{1}, v_{2},v_{3},v_{4}\mid$, for $v_{1},
v_{2},v_{3}, v_{4}\in\C^{4}$, with $\mid v_{1}, v_{2},v_{3},
v_{4}\mid$ denoting the determinant of the matrix whose columns are
the components of $v_{1}, v_{2},v_{3}$ and $v_{4}$, respectively, on
the canonical basis of $\C^{4}$.}

Consider $\wedge ^{2}\C^{4}$ equipped with the real structure
$\wedge^{2}j$ and the metric defined by
$$(v_{1}\wedge v_{2},v_{3}\wedge v_{4}):=\mathrm{det}(v_{1},
v_{2},v_{3}, v_{4}),$$ for $v_{1}, v_{2},v_{3},v_{4}\in\C^{4}$.
Condition \eqref{eq:detwedge4}, equivalent to
$$\overline{(v_{1}\wedge v_{2},v_{3}\wedge
v_{4})}=(\overline{v_{1}\wedge v_{2}},\overline{v_{3}\wedge
v_{4}}),$$ for $v_{1}, v_{2},v_{3},v_{4}\in\C^{4}$, ensures that
$\wedge^{2}\underline{\C}^{4}$ induces a (real) metric in
$\mathrm{Fix}\,(\wedge^{2} j)$, which condition \eqref{eq:det>0}
ensures to have signature $(5,1)$, as we verify next.
\begin{prop}\label{Fix51}
$\wedge^{2}\underline{\C}^{4}$ induces in $\mathrm{Fix}\,(\wedge^{2}
j)$ a metric with signature $(5,1)$.
\end{prop}
We therefore write
$$\mathrm{Fix}\,(\wedge^{2} j)=\R^{5,1}$$ and so
$$\wedge^{2}\C^{4}=(\R^{5,1})^{\C}.$$
The proof of the proposition will follow a very important remark,
presented next.

By definition of the metric on $\wedge^{2}\C^{4}$, the
complementarity of $2$-planes $W$ and $\hat{W}$ in $\C^{4}$,
$$\C^{4}=W\oplus\hat{W},$$ is equivalent to $(w_{1}\wedge
w_{2},\hat{w}_{1}\wedge \hat{w}_{2})\neq 0$, for $w_{1},w_{2}$ frame
of $W$ and $\hat{w}_{1},\hat{w}_{2}$ frame of $\hat{W}$, or,
equivalently,
\begin{equation}\label{eq:capsSs}
\wedge^{2}W\cap (\wedge^{2}\hat{W})^{\perp}=\{0\}.
\end{equation}
Condition \eqref{eq:capsSs}, in its turn, establishes the
complementarity of $\wedge^{2}W$ and $\wedge^{2}\hat{W}$ in
$\wedge^{2}W+\wedge^{2}\hat{W}$, together with the non-degeneracy of
$\wedge^{2}W\oplus\wedge^{2}\hat{W}$, ensuring, in particular, that
\begin{equation}\label{eq:MwedgehatMistheorthogonaltowedges}
W\wedge\hat{W}=(\wedge^{2}W\oplus\wedge^{2}\hat{W})^{\perp}
\end{equation}
and establishing the decomposition
\begin{equation}\label{eq:wedge2C4MhatM}
\wedge^{2}\C^{4}=\wedge^{2}W\oplus\wedge^{2}\hat{W}\oplus
W\wedge\hat{W}.
\end{equation}
Conversely, a decomposition
\begin{equation}\label{eq:wedge2C4decomp}
\wedge^{2}\C^{4}= V\oplus V_{+}^\perp\oplus V_{-}^{\perp},
\end{equation}
with $V_{+}^\perp$ and $V_{-}^{\perp}$ null lines such that
$V_{+}^{\perp}\cap (V_{-}^{\perp})^{\perp}=\{0\},$ and
$V^{\perp}=V_{+}^\perp\oplus V_{-}^{\perp}$, determines
complementary $2$-planes $W$ and $\hat{W}$ in $\C^{4}$ such that
$\wedge^{2}W=V_{+}^{\perp}$, $\wedge^{2}\hat{W}=V_{-}^{\perp}$ and
$V=W\wedge\hat{W}$.

For the particular case of $W$ a non-$j$-stable $2$-plane and
$\hat{W}=jW$, and in view of the reality of $W\wedge jW$ and
$\wedge^{2}W\oplus\wedge^{2}jW$, \eqref{eq:wedge2C4MhatM} gives, in
particular,
$$\mathrm{Fix}(\wedge^{2}j)=(\wedge^{2}W\oplus\wedge^{2}jW)\cap\mathrm{Fix}(\wedge^{2}j)\oplus
(W\wedge jW)\cap\mathrm{Fix}(\wedge^{2}j).$$

We now prove Proposition \ref{Fix51}.
\begin{proof}
Let $W$ be a non-$j$-stable $2$-plane in $\C^{4}$. Given $w_{i}\in
W$, for $i=1,2,3,4$, the vector $w_{1}\wedge w_{2}+jw_{3}\wedge
jw_{4}$ is real if and only if $w_{1}\wedge w_{2}=w_{3}\wedge
w_{4}$. Hence, if $w_{1}\wedge w_{2}+jw_{3}\wedge jw_{4}$ is
non-zero, then $w_{3}\wedge w_{4}\neq 0$ and
$w_{i}=\lambda_{i}^{1}w_{1}+\lambda_{i}^{2}w_{2}$, for $i=3,4$. In
that case,
$$(w_{1}\wedge w_{2}+jw_{3}\wedge jw_{4},w_{1}\wedge
w_{2}+jw_{3}\wedge
jw_{4})=2\,\overline{\lambda_{3}^{1}\lambda_{4}^{2}-\lambda_{4}^{1}\lambda_{3}^{2}}\,\mathrm{det}(w_{1},w_{2},jw_{1},jw_{2}),$$
whilst $$0\neq jw_{3}\wedge
jw_{4}=\overline{\lambda_{3}^{1}\lambda_{4}^{2}-\lambda_{4}^{1}\lambda_{3}^{2}}\,jw_{1}\wedge
jw_{2}=\overline{\lambda_{3}^{1}\lambda_{4}^{2}-\lambda_{4}^{1}\lambda_{3}^{2}}\,jw_{3}\wedge
jw_{4}$$ and, therefore,
$\overline{\lambda_{3}^{1}\lambda_{4}^{2}-\lambda_{4}^{1}\lambda_{3}^{2}}=1$.
By equation \eqref{eq:det>0}, we conclude that, for non-zero real
$w_{1}\wedge w_{2}+jw_{3}\wedge jw_{4}\in\wedge^{2}W\oplus
\wedge^{2}jW$,
$$(w_{1}\wedge w_{2}+jw_{3}\wedge jw_{4},w_{1}\wedge w_{2}+jw_{3}\wedge jw_{4})=2\,\mathrm{det}(w_{1},w_{2},jw_{1},jw_{2})>0,$$
i.e., that the metric induced in
$(\wedge^{2}W\oplus\wedge^{2}jW)\cap\mathrm{Fix}(\wedge^{2}j)$ is
positive definite. Now let $w_{1},w_{2}$ be a basis of $W$, so that
$w_{1}\wedge jw_{1},w_{1}\wedge jw_{2}, w_{2}\wedge jw_{1}$ and
$w_{2}\wedge jw_{2}$ form a basis of $W\wedge jW$. The vectors
$w_{1}\wedge jw_{1}$ and $w_{2}\wedge jw_{2}$ are real, null and not
orthogonal, spanning, therefore, a subspace of $(W\wedge jW)\cap
\mathrm{Fix}\,(\wedge^{2} j)$ with signature $(1,1)$. On the other
hand, $\langle w_{1}\wedge jw_{2},w_{2}\wedge jw_{1}\rangle\cap
\mathrm{Fix}\,(\wedge^{2} j)=\{ (a+ib)w_{1}\wedge
jw_{2}+(a-ib)w_{2}\wedge jw_{1}:a,b\in\R\}$, clearly orthogonal to
$\langle w_{1}\wedge jw_{1},w_{2}\wedge jw_{2}\rangle$ and with
positive definite metric inherited from $\wedge^{2}\C^{4}$: given
$w:=(a+ib)w_{1}\wedge jw_{2}+(a-ib)w_{2}\wedge jw_{1}$ with
$a+ib\in\C\backslash\{0\}$,
$$(w,w)=2\,(a^{2}+b^{2})\,\mathrm{det}(w_{1},w_{2},jw_{1},jw_{2})>0.$$
By \eqref{eq:MwedgehatMistheorthogonaltowedges}, we conclude the
existence of an orthogonal basis of $\mathrm{Fix}\,(\wedge^{2} j)$
composed by five space-like vectors and one time-like vector.
\end{proof}
\begin{rem}\label{inprodvoverlv}
In the proof of Proposition \ref{Fix51}, we have observed, in
particular, that, given $W$ a non-$j$-stable $2$-plane in $\C^{4}$,
the metric induced in the space of the real vectors in
$\wedge^{2}W\oplus\wedge^{2}jW$ is positive definite. Hence, given
$v\in\wedge^{2}W$,
$$(v,\overline{v}\,)=\frac{1}{2}\,(v+\overline{v},v+\overline{v})\geq 0,$$
vanishing if and only if $v=0$. It follows, in particular, that,
given $v\in\wedge^{2}\C^{4}$ decomposable,
\begin{equation}\label{eq:vconjvispositive}
(v,\overline{v}\,)\geq 0.
\end{equation}
\end{rem}

We identify $2$-planes in $\C^{4}$ with null lines in
$\wedge^{2}\C^{4}$, spanned by a decomposable vector, via the famous
Pl\"{u}cker embedding, i.e., via the correspondence
$$L=\langle u,v\rangle\leftrightarrow\langle u\wedge v\rangle=\wedge^{2}L,$$given
complex linearly independent $u$ and $v$ in $\C^{4}$. In this way,
we identify, in particular, $j$-stable $2$-planes in $\C^{4}$ with
real null lines in $\wedge^{2}\C^{4}$ and then, naturally, with null
lines in $\mathrm{Fix}(\wedge^{2}j)$. This presents the quaternionic
projective space as a model for the conformal $4$-sphere.

\begin{prop}
$$\mathbb{H}P^{1}\cong S^{4}.$$
\end{prop}
This chapter is dedicated to the special case of surfaces in
$S^{4}$, described in this model as the immersed bundles
$$L\cong\wedge^{2}L:M\rightarrow S^{4}$$ of $j$-stable
$2$-planes in $\C^{4}$.

We use $\mathrm{End}_{j}(\C^{4})$ to denote the set of endomorphisms
of $\C^{4}$ commuting with $j$, and $\mathrm{Gl}_{j}(\C^{4})$
(respectively, $sl_{j}(\C^{4})$) to denote the invertible
(respectively, trace-free) $j$-commuting endomorphisms of $\C^{4}$.
\begin{Lemma}\label{formadosendomorfismoscomsquare}
Given $\xi\in\mathrm{End}_{j}(\C^{4})$ with $\xi^{2}=aI$, for some
$a\in\R$, if $a\geq 0$, then
$$\xi=\sqrt{a}\,I,$$for one of the square roots of $a$; whereas,
if $a< 0$, then, given a choice of $\sqrt{a}$,
$$\xi=I\left\{
\begin{array}{ll} \sqrt{a} & \mbox{$\mathrm{on}\,W$}\\ \overline{\sqrt{a}}=-\sqrt{a} &
\mbox{$\mathrm{on}\,jW$}\end{array}\right.,$$for some non-j-stable
$2$-plane $W$ in $\C^{4}$.
\end{Lemma}
\begin{proof}
As a complex endomorphism, $\xi$ admits at least one complex
eigenvalue. Given $\lambda\in\C$ an eigenvalue of $\xi$ and $u$ an
eigenvector of $\xi$ associated to $\lambda$,
$$au=\xi^{2}u=\xi(\lambda u)=\lambda\xi u =\lambda^{2}u$$ and,
therefore, $\lambda=\pm\sqrt{a}$. On the other hand, $$\xi(ju)=j\xi
u=j\lambda u=\overline{\lambda}ju,$$ showing that
$\overline{\lambda}$ is an eigenvalue of $\xi$, as well, and that,
for $E_{\lambda}$ and $E_{\overline{\lambda}}$ the eigenspaces
associated to $\lambda$ and $\overline{\lambda}$, respectively, we
have $jE_{\lambda}\subset E_{\overline{\lambda}}$, or, equivalently,
in view of the symmetry of roles between $\lambda$ and
$\overline{\lambda}$, $jE_{\lambda}= E_{\overline{\lambda}}$. It
follows, in particular, that, if $a<0$, and, therefore, $\sqrt{a}$
is purely imaginary, $\overline{\sqrt{a}}=-\sqrt{a}$, the
eigenvalues of $\xi$ are exactly $\sqrt{a}$ and $-\sqrt{a}$ and, as
$\xi$ is diagonalizable, $\underline{\C}^{4}=E_{\sqrt{a}}\oplus
jE_{\sqrt{a}}$. If $a>0$, and so $\sqrt{a}$ is real, $-\sqrt{a}$ is
not an eigenvalue of $\xi$ and, therefore,
$E_{\sqrt{a}}=\underline{\C}^{4}$, completing the proof.
\end{proof}
According to the previous lemma, given
$S\in\mathrm{End}_{j}(\C^{4})$ with $S^{2}=-I$, we have a
decomposition
$$\C^{4}=S_{+}\oplus S_{-},$$
for $S_{+}$ and $S_{-}=jS_{+}$ the eigenspaces of $S$ associated to
$i$ and $-i$, respectively, a notation that will be kept throughout
the chapter.
\begin{rem}
Given $S\in\mathrm{End}_{j}(\C^{4})$, such that $S^{2}=-I$, and
$v_{\pm}\in\wedge^{2}S_{\pm},$ we have
$(v_{\pm},\overline{v_{\pm}})\geq 0$, vanishing if and only if
$v_{\pm}=0$, respectively.
\end{rem}

We define a $2$-\textit{sphere} \textit{in}
\textit{$\wedge^{2}\C^{4}$} to be the complexification of a
$2$-sphere in $\R^{5,1}$. Recall that the $2$-spheres in $\R^{5,1}$
are described as the non-degenerate $4$-planes in $\R^{5,1}$.
\begin{prop}\label{2spheres}
The $2$-spheres in $\wedge^{2}\C^{4}$ are described as the (real
non-degenerate) $4$-planes $S_{+}\wedge S_{-}$ for
$S\in\mathrm{End}_{j}(\C^{4})$ such that $S^{2}=-I$, determined up
to sign.
\end{prop}

\begin{proof}
The complementarity in $\C^{4}$ of the $2$-planes $S_{+}$ and
$S_{-}$ ensures that $S_{+}\wedge S_{-}$ is a $4$-plane, as well as,
by recalling \eqref{eq:MwedgehatMistheorthogonaltowedges} and
\eqref{eq:wedge2C4MhatM}, its non-degeneracy in $\wedge^{2}\C^{4}$.
The reality of $S_{+}\wedge S_{-}$ is immediate from the fact that
$S_{+}$ and $S_{-}$ are intertwined by $j$. Conversely, a real
non-degenerate $4$-plane $V$ in $\wedge^{2}\C^{4}$ determines, up to
sign, a $j$-commuting endomorphism $S$ of $\C^{4}$, such that
$S^{2}=-I$, for which $V=S_{+}\wedge S_{-}.$ Indeed, for such a $V$,
$V^{\perp}$ is a real non-degenerate complex $2$-plane, admitting,
therefore, a unique decomposition $V^{\perp}=V_{+}^{\perp}\oplus
V_{-}^{\perp}$. as the direct sum of two null complex lines, complex
conjugate of each other. The corresponding decomposition
\eqref{eq:wedge2C4decomp} determines a $2$-plane $W$ in $\C^{4}$ for
which $\wedge^{2}W=V_{+}^{\perp}$ and $V=W\wedge jW$. By
$$S:=I\left\{
\begin{array}{ll} i & \mbox{$\mathrm{on}\,W$}\\-i & \mbox{$\mathrm{on}\,jW$}\end{array}\right.,$$
we define a $j$-commuting endomorphism $S$ of $\C^{4}$ such that
$S^{2}=-I$, determined by $V$ up to sign, for which $V=S_{+}\wedge
S_{-}.$
\end{proof}

Throughout this chapter, we will use the standard identification
$$sl(\C^{4})\cong o((\R^{5,1})^{\C}),$$
of the special linear algebra $sl(\C^{4})$ with the orthogonal
algebra $o(\wedge^{2}\C^{4})$, given by $f\mapsto \hat{f}$ with
\begin{equation}\label{eq:leinitzidentif}
\hat{f} (x\wedge y)= (fx)\wedge y+x\wedge (fy),
\end{equation}
for $f\in sl(\C^{4})$, $\hat{f}\in o(\wedge^{2}\C^{4})$ and $x,y\in
\C^{4}$. Observe that, under this identification, the reality of an
endomorphism of $(\R^{5,1})^{\C}$ corresponds to the
$j$-commutativity of an endomorphism of $\C^{4}$.

We complete this section by presenting a last identification that
shall be used throughout this chapter. Observe that by $g\mapsto
\hat{g}$ with $$\hat{g}(u\wedge v)=gu\wedge gv,$$ for
$g\in\mathrm{Sl}(\C^{4})$, $\hat{g}\in O((\R^{5,1})^{\C})$ and
$u,v\in\C^{4}$, we define a $2:1$ mapping, $$g,-g\mapsto \hat{g},$$
of the special linear group $\mathrm{Sl}(\C^{4})$ onto the
orthogonal group $O((\R^{5,1})^{\C})$. For an orthogonal
transformation
$$\xi=I\left\{
\begin{array}{ll} a & \mbox{$\mathrm{on}\,\wedge^{2}W$}\\ 1 &
\mbox{$\mathrm{on}\,W\wedge \hat{W}$}\\
a^{-1} & \mbox{$\mathrm{on}\,\wedge
^{2}\hat{W}$}\end{array}\right.\in O((\R^{5,1})^{\C}),$$ with $W$
and $\hat{W}$ complementary $2$-planes in $\C^{4}$ and
$a\in\C\backslash\{0\}$, and fixing a choice of $\sqrt a$, we will
still denote by $\xi$ the special linear transformation
$$\xi=I\left\{
\begin{array}{ll} \sqrt{a} & \mbox{$\mathrm{on}\,W$}\\
\sqrt a\,^{-1} & \mbox{$\mathrm{on}\,\hat{W}$}\end{array}\right.\in
\mathrm{Sl}(\C^{4}),$$defined up to sign depending on the choice of
$\sqrt a$. Obviously,
$$\xi(u\wedge v)=\xi u\wedge\xi v,$$for $u,v\in\C^{4}$,
independently of the choice of $\sqrt a$. Observe that, in the
particular case $W$ is a non-$j$-stable $2$-plane and $\hat{W}=jW$,
$\xi$ and $\overline{\xi}$ are related via $j$ by
\begin{equation}\label{eq:xijjxi}
\xi\circ j=j\circ\overline{\xi}.
\end{equation}

\subsection{The mean curvature sphere congruence}
We present the mean curvature sphere congruence  and the Hopf fields
of a surface in $S^{4}$, as defined in
\cite{quaternionsbook}.\newline

Let $L$ be an immersed bundle of $j$-stable $2$-planes in $\C^{4}$.
Consider $M$ provided with the conformal structure induced by the
surface $\wedge^{2}L:M\rightarrow S^{4}$. Given
$S\in\Gamma(\mathrm{End}_{j}(\underline{\C}^{4}))$ such that
\begin{equation}\label{eq:SL=L}
SL=L,
\end{equation}
we define a section of $\mathrm{End}(\underline{\C}^{4}/L)$, which
we still denote by $S$, by $S(x+L):=Sx+L$, for all
$x\in\Gamma(\underline{\C}^{4})$. Consider the derivative
$$\delta:=\pi_{\underline{\C}^{4}/L}\circ d_{\vert L},$$for $\pi_{\underline{\C}^{4}/L}:\underline{\C}^{4}\rightarrow
\underline{\C}^{4}/L$ the canonical projection.
\begin{rem}
Consider the projection
$\pi:\mathcal{L}\rightarrow\mathbb{P}(\mathcal{L})$ for the
light-cone $\mathcal{L}$ in $\R^{5,1}$. Given $l$ a never-zero
section of $L\simeq\wedge^{2}
L:M\rightarrow\mathbb{P}(\mathcal{L})$,  $d\pi_{l}$ gives an
isomorphism $L^{\perp}/L\rightarrow T_{L}\mathbb{P}(\mathcal{L})$
under which the derivative of $L$, $dL=d\pi_{l}\circ
dl=d\pi_{l}\circ \delta l$, is given by $\delta l$,
$$dL=\delta l.$$ Therefore $L$ immerses if and only if
$\mathrm{rank}\,\delta L(TM)=2$.
\end{rem}
As presented in \cite{quaternionsbook},
\begin{defn}\label{eq:mcs}
The \emph{mean curvature sphere (congruence)} of $L$ is defined to
be the unique $j$-commuting complex structure $S$ on
$\underline{\C}^{4}$ such that
$$SL=L,\,\,\,\,(dS)L\subset L,\,\,\,\, *\,\delta=S\circ\delta=\delta\circ S_{\vert L}$$ and
\begin{equation}\label{eq:SDSL=O}
(SdS-*dS)L=0.
\end{equation}
\end{defn}
Once and for all, fix $S$ as the mean curvature sphere of $L$.
Decompose the trivial flat connection on $\underline{\C}^{4}$ as
$d=\mathcal{D}_{S}+\mathcal{N}_{S}$ by the conditions
\begin{equation}\label{eq:NS-SN}
\mathcal{D}_{S}S=0,\,\,\,\mathcal{N}_{S}S=-S\mathcal{N}_{S}.
\end{equation}
Equivalently, set $\mathcal{N}_{S}:=\frac{1}{2}SdS$ and
$\mathcal{D}_{S}:=d-\mathcal{N}_{S}$. Note that conditions
\eqref{eq:NS-SN} are described, equivalently, by
\begin{equation}\label{eq:DS+}
\mathcal{D}_{S}\Gamma(S_{\pm})\subset\Omega^{1}(S_{\pm}),
\end{equation}
together with
\begin{equation}\label{eq:NS+}
\mathcal{N}_{S}S_{\pm}\subset S_{\mp}
\end{equation}
respectively. Consider the congruence $ V= S_{+}\wedge S_{-}$ of
$2$-spheres in $\wedge^{2}\underline{\C}^{4}$. Let
$\pi_{\wedge^{2}S_{+}\oplus\wedge^{2}S_{-}}$ denote the orthogonal
projection of
\begin{equation}\label{eq:w2C4w2s+w2s-s+s-}
\wedge^{2}\underline{\C}^{4}=\wedge^{2}S_{+}\oplus S_{+}\wedge
S_{-}\oplus\wedge^{2}S_{-}
\end{equation}
onto $\wedge^{2}S_{+}\oplus\wedge^{2}S_{-}$. The trivial flat
connection on $\wedge^{2}\underline{\C}^{4}$ relates to the trivial
flat connection on $\underline{\C}^{4}$ by
$$d(u\wedge v)=(du)\wedge v+u\wedge (dv),$$for
$u,v\in\Gamma(\underline{\C}^{4})$. In particular, given
$u\in\Gamma(S_{+}), v\in\Gamma(S_{-})$,
\begin{eqnarray*}
\mathcal{N}_{V}(u\wedge v)&=&(\pi_{V}\circ
d\circ\pi_{V^{\perp}}+\pi_{V^{\perp}}\circ d\circ\pi_{V})(u\wedge
v)\\&=&\pi_{\wedge^{2}S_{+}\oplus\wedge^{2}S_{-}}((du)\wedge
v+u\wedge(dv))\\&=&\pi_{\wedge^{2}S_{+}\oplus\wedge^{2}S_{-}}((\mathcal{D}_{S}u)\wedge
v+(\mathcal{N}_{S}u)\wedge v+ u\wedge (\mathcal{D}_{S}v)+u\wedge
(\mathcal{N}_{S}v))\
\end{eqnarray*}
and, therefore, according to \eqref{eq:DS+} and \eqref{eq:NS+},
$$\mathcal{N}_{V}(u\wedge v)=(\mathcal{N}_{S}u)\wedge v+u\wedge
(\mathcal{N}_{S}v).$$Similarly, we get to the same conclusion for
$u\wedge v\in\Gamma(\wedge^{2}S_{+})$ and $u\wedge
v\in\Gamma(\wedge^{2}S_{-})$. We conclude that
$\mathcal{N}_{V}\cong\mathcal{N}_{S}$, under the identification
defined by \eqref{eq:leinitzidentif}. In particular,
$$\mathcal{N}_{S}j=j\mathcal{N}_{S}.$$

We define a decomposition
$$\mathcal{N}_{S}=A+Q$$ with
$A,Q\in\Omega^{1}(\mathrm{End}_{j}(\underline{\C}^{4}))$ given by
$$*A=SA\,\,\,,\,\,\,*\,Q=-SQ,$$ said to be the
\textit{Hopf} \textit{fields} of $L$, following the terminology of
\cite{quaternionsbook}, as we verify later in this section. Clearly,
$$A=\frac{1}{2}(\mathcal{N}_{S}-S*\mathcal{N}_{S}),\,\,\,\,\,\,Q=\frac{1}{2}(\mathcal{N}_{S}+S*\mathcal{N}_{S}),$$
and, therefore, having in consideration \eqref{eq:NS-SN}, both $A$
and $Q$ anti-commute with $S$,
$$AS=-SA\,\,\,,\,\,\,QS=-SQ.$$
In particular, this makes clear that
\begin{equation}\label{eq:AQSpm}
AS_{\pm}\subset S_{\mp},\,\,\,\,QS_{\pm}\subset S_{\mp},
\end{equation}
respectively. On the other hand, $iA^{1,0}=-*A^{1,0}=-SA^{1,0}$,
showing that $\mathrm{Im}\,A^{1,0}\subset S_{-}$ and, therefore,
that
\begin{equation}\label{eq:A10S-}
A^{1,0}S_{-}=0.
\end{equation}
Analogously, we verify that
\begin{equation}\label{eq:todosondesao0}
A^{0,1}S_{+}=Q^{1,0}S_{+}=Q^{0,1}S_{-}=0.
\end{equation}
Note that, as
\begin{equation}\label{eq:dSNS}
dS=2\mathcal{N}_{S}S,
\end{equation}
we have $Q=\frac{1}{4}\,(SdS-*dS)$. Thus $$QL=0.$$ Lastly, observe
that
$$A^{1,0}s_{+}=\frac{1}{2}\,(\mathcal{N}_{S}^{1,0}s_{+}-iS\mathcal{N}_{S}^{1,0}s_{+})=\mathcal{N}_{S}^{1,0}s_{+}\in\Omega^{1,0}(L_{-}),$$
for $s_{+}\in\Gamma(S_{+})$,
\begin{equation}\label{eq:A10S+L-}
A^{1,0}S_{+}\subset L_{-};
\end{equation}
and, similarly,$$A^{0,1}S_{-}\subset L_{+}$$and
$$Q^{1,0}S_{-}\subset S_{+}\,\,\,,\,\,\,Q^{0,1}S_{+}\subset S_{-}.$$
We complete this section by verifying that our description follows
the description presented in \cite{quaternionsbook}. Given
$\psi\in\Gamma(\underline{\C}^{4})$,
$d(S\psi)=-2S\mathcal{N}_{S}\psi+S\mathcal{D}_{S}\psi+S\mathcal{N}_{S}\psi=-S\mathcal{N}_{S}\psi+S\mathcal{D}_{S}\psi$
and, therefore,
$S(d(S\psi))=\mathcal{N}_{S}\psi-\mathcal{D}_{S}\psi$, as well as
$*\,d(S\psi)=-S*\mathcal{N}_{S}\psi+S*\mathcal{D}_{S}\psi$. Thus
$Q\psi= \frac{1}{4}(d\psi+S*d\psi+S(d(S\psi))-*\,d(S\psi))$, showing
that $Q$ coincides with the one defined in \cite{quaternionsbook}.
The similar conclusion with respect to $A$ follows then from the
fact that
\begin{equation}\label{eq:dSvsAeQ}
dS=2(*\,Q-*A),
\end{equation}
made clear by \eqref{eq:dSNS}.

\subsection{Mean curvature sphere congruence and central sphere
congruence}\label{mcscVScsc}

In codimension $2$, the complexification of the normal $S^{\perp}$
to the central sphere congruence of a surface $\Lambda$ admits a
unique decomposition $S^{\perp}=S^{\perp}_{+}\oplus S^{\perp}_{-}$
into the direct sum of two null complex line bundles, complex
conjugate of each other. Via the Pl\"{u}cker embedding, we identify
$S_{+}^{\perp}$ with some bundle $S_{+}$ of $2$-planes in $\C^{4}$,
and write then $S=S_{+}\wedge jS_{+}$. The mean curvature sphere of
$L\cong\wedge^{2}L=\Lambda$ is defined, up to sign, as the
$j$-commuting complex structure on $\underline{\C}^{4}$ admitting
$S_{+}$ as the eigenspace associated to the eigenvalue $i$ (and,
therefore, $jS_{+}$ as the eigenspace associated to $-i$). This
establishes the close relationship between mean curvature sphere
congruence and the central sphere congruence.\newline

Let $L$ be an immersed bundle of $j$-stable $2$-planes in $\C^{4}$
and consider the surface $$\Lambda:=\wedge^{2}L:M\rightarrow
S^{4}.$$ Let $S$ be the mean curvature sphere of $L$.

\begin{prop}\label{VvsS}
The congruence $V:= S_{+}\wedge S_{-}$, of $2$-spheres in
$\wedge^{2}\C^{4}$, is the complexification of the central sphere
congruence of $\Lambda$.
\end{prop}
Before proceeding to the proof of the proposition, time for a few
considerations.

Equation \eqref{eq:SL=L} ensures that, given
$l_{i}=l_{i}^{+}+l_{i}^{-}\in\Gamma(L)$ with
$l_{i}^{\pm}\in\Gamma(S_{\pm})$, respectively, for $i=1,2$,
$l_{1}^{+}$ and $l_{2}^{+}$ are linearly dependent, and so are
$l_{1}^{-}$ and $l_{2}^{-}$. We conclude the existence of a frame of
$L$ composed by a section of $S_{+}$ and a section of $S_{-}$, which
provides a decomposition $$L=L_{+}\oplus L_{-}$$ of $L$ into a sum
of complex line bundles $L_{+}$ and $L_{-}$ with
$$L_{-}=jL_{+},$$namely, $$L_{+}:=L\cap S_{+}\,\,\,,\,\,\,L_{-}:=L\cap
S_{-}.$$ Obviously,
\begin{equation}\label{eq:wedge2LL+L-}
\wedge^{2}L= L_{+}\wedge L_{-}.
\end{equation}
Now we proceed to the proof of Proposition \ref{VvsS}.
\begin{proof}
By \eqref{eq:wedge2LL+L-}, $\wedge^{2}L\subset S_{+}\wedge S_{-}$.
Equations \eqref{eq:SL=L} and \eqref{eq:dSNS}, together with the
fact that $(dS)L\subset L$, show that $\mathcal{N}_{S}L\subset L$
and, therefore, by \eqref{eq:NS+}, that
\begin{equation}\label{eq:NLpmsubsetmp}
\mathcal{N}_{S}L_{\pm}\subset L_{\mp},
\end{equation}
respectively. Provide $M$ with a conformal structure $\mathcal{C}$.
In view of \eqref{eq:NLpmsubsetmp} and \eqref{eq:SL=L}, given
$l=l_{+}+l_{-}$ in $\Gamma(L)$, with $l_{+}\in\Gamma(L_{+})$ and
$l_{-}\in\Gamma(L_{-})$, $S\circ \delta
l=S(dl)+L=S(\mathcal{D}_{S}l_{+}+\mathcal{D}_{S}l_{-})+L$.
Therefore, by \eqref{eq:DS+}, $S\circ \delta
l=i(\mathcal{D}_{S}l_{+}-\mathcal{D}_{S}l_{-})+L$,  whilst $$*\delta
l=-i(d^{1,0}l-d^{0,1}l)+L=-i(\mathcal{D}_{S}^{1,0}l-\mathcal{D}_{S}^{0,1}l)+L.$$
Thus $*\delta l=S\circ \delta l$ if and only if
$\mathcal{D}_{S}^{1,0}l_{+}-\mathcal{D}_{S}^{0,1}l_{-}\in\Gamma(L)$.
By equation $*\,\delta=S\circ\delta$, we conclude that
$$\mathcal{D}_{S}^{1,0}\,\Gamma(L_{+})\subset\Omega^{1,0}(L_{+}),$$
or, equivalently,
$$\mathcal{D}_{S}^{0,1}\,\Gamma(L_{-})\subset\Omega^{0,1}(L_{-}).$$
It follows that, given $z$ holomorphic chart of
$(M,\mathcal{C}_{\Lambda})$ and never-zero $l_{+}\in\Gamma(L_{+})$
and $l_{-}\in\Gamma(L_{-})$,
$$
\Lambda^{1,0}=\langle l_{+}\wedge l_{-},d_{\delta_{z}}(l_{+}\wedge
l_{-})\rangle=\langle l_{+}\wedge
l_{-},l_{+}\wedge(\mathcal{D}_{S})_{\delta_{z}}l_{-}
\rangle,$$showing, in particular, the linear independency of $l_{-}$
and $(\mathcal{D}_{S})_{\delta_{z}}l_{-}$, and, consequently, that
$$\Lambda^{1,0}= L_{+}\wedge S_{-}$$and, by complex
conjugation, $$\Lambda^{0,1}=L_{-}\wedge S_{+}.$$In particular, it
is clear that $V$ envelops the surface $\Lambda$. Lastly, according
to equation \eqref{eq:SDSL=O}, we have, in particular, that, given
$l_{-}\in\Gamma(L_{-})$, $(Sd^{1,0}S+id^{1,0}S)l_{-}=0$, or,
equivalently, $(d^{1,0}S)l_{-}\in\Omega^{1,0}(S_{-})$, which, in its
turn, according to \eqref{eq:NS+}, is equivalent to
$\mathcal{N}_{S}^{1,0}l_{-}=0$. Hence
$$\mathcal{N}_{S}^{1,0}L_{-}=0.$$Similarly, we verify that
$$\mathcal{N}_{S}^{0,1}L_{+}=0.$$On the other hand, according to \eqref{eq:todosondesao0} and \eqref{eq:A10S+L-}, we have
$$\mathcal{N}_{S}^{1,0}S_{+}\subset L_{-}$$and, similarly, $$\mathcal{N}_{S}^{0,1}S_{-}\subset L_{+}.$$
It follows that $\mathcal{N}_{S}^{1,0}(L_{-}\wedge
S_{+})=0=\mathcal{N}_{S}^{0,1}(L_{+}\wedge S_{-})$, i.e.,
$\mathcal{N}_{V}^{1,0}\Lambda^{0,1}=0=\mathcal{N}_{V}^{0,1}\Lambda^{1,0}$,
$V$ is central with respect to $\Lambda$, completing the proof.
\end{proof}

\begin{rem}
The mean curvature sphere of $L$ is, conversely, determined by the
complexification of the central sphere congruence of $\Lambda$, as
follows. Since the complexification $V$ of the central sphere
congruence of $\Lambda$ is a $2$-sphere, $V=S_{+}\wedge S_{-}$, for
some $S\in\mathrm{End}_{j}(\C^{4})$ such that $S^{2}=-I$, determined
up to sign. One of them is the mean curvature sphere of $L$, cf.
Proposition \ref{VvsS}. For the other one, define $\mathcal{D}_{S}$
analogously to the case of the mean curvature sphere. Of course,
$$\mathcal{D}_{-S}=\mathcal{D}_{S}.$$The conformality of sections of $\Lambda$,
characterized by the isotropy of
$$d^{1,0}(l_{+}\wedge
l_{-})=(\mathcal{D}_{S}^{1,0}l_{+})\wedge l_{-}+l_{+}\wedge
(\mathcal{D}_{S}^{1,0}l_{-}),$$or, equivalently,
$$(\mathcal{D}_{S}^{1,0}l_{+})\wedge l_{-}\wedge l_{+}\wedge
(\mathcal{D}_{S}^{1,0}l_{-})=0,$$for all $l_{+}\in\Gamma(L_{+})$ and
$l_{-}\in\Gamma(L_{-})$, amounts, according to \eqref{eq:DS+}, to
either
\begin{equation}\label{eq:nbureckpvnv ucr938uyncm}
\mathcal{D}_{S}^{1,0}\,\Gamma(L_{+})\subset\Omega^{1,0}(L_{+})
\end{equation}
or
$$\mathcal{D}_{S}^{1,0}\,\Gamma(L_{-})\subset\Omega^{1,0}(L_{-}),$$
depending on the sign of $S$. Equation \eqref{eq:nbureckpvnv
ucr938uyncm} determines $S$ for which, in particular, given
$l=l_{+}+l_{-}\in\Gamma(L)$ with $l_{\pm}\in\Gamma(L_{\pm})$
respectively,
$\mathcal{D}_{S}^{1,0}l_{+}-\mathcal{D}_{S}^{0,1}l_{-}\in\Gamma(L),$
or, equivalently, $*\,\delta=S\circ\delta$. $S$ determined in this
way is the mean curvature sphere of $L$.
\end{rem}

\section{Constrained Willmore surfaces in $4$-space}

\markboth{\tiny{A. C. QUINTINO}}{\tiny{CONSTRAINED WILLMORE
SURFACES}}

In this section, we characterize constrained Willmore surfaces in
$4$-space in the quaternionic setting. We present, in particular, a
characterization of this class of surfaces in terms of the closeness
of a certain form, following the characterization of the harmonicity
of the mean curvature sphere presented in
\cite{quaternionsbook}.\newline

Let $L\cong\Lambda\subset\underline{\R}^{5,1}$ be a surface in
$S^{4}.$ Let $S$ be the mean curvature sphere of $L$ and $V$ be the
complexification of the central sphere congruence of $\Lambda$,
$$V=S_{+}\wedge S_{-},$$with orthogonal complement
$$V^{\perp}=\wedge^{2}S_{+}\oplus\wedge^{2} S_{-}.$$
Provide $M$ with the conformal structure induced by $\Lambda$.
Consider the map $j:\C^{4}/L\rightarrow \C^{4}/L$ induced naturally
by the quaternionic structure $j$ in $\C^{4}$, having in
consideration that $L$ is $j$-stable. Constrained Willmore surfaces
in $S^{4}$ are, alternatively, characterized as follows:
\begin{thm}\label{CWinS4vkutywsfgdredhgvcertyq2932h}
An immersed bundle $L$ of $j$-stable $2$-planes in $\C^{4}$ is a
constrained Willmore surface if and only if there exists
$q\in\Omega^{1}(\mathrm{End}_{j}(\underline{\C}^{4}/L,L))$ such that
\begin{equation}\label{eq:Sq=qS=*q}
Sq=*\,q=qS
\end{equation}
and
$$d^{\mathcal{D}_{S}}q=0,\,\,\,\,\,\,d^{\mathcal{D}_{S}}*\mathcal{N}_{S}=2[q\wedge
*\mathcal{N}_{S}].$$
\end{thm}
In the proof of the theorem, we establish a correspondence between
$1$-forms with values in $\Lambda\wedge\Lambda^{(1)}$ and
$S$-commuting $1$-forms with values in
$\mathrm{End}_{j}(\underline{\C}^{4}/L,L)$, under the the
identification $sl(\C^{4})\cong o((\R^{5,1})^{\C})$ presented in
Section \ref{LinAlg}. Under this correspondence, the conditions on
$q$ in Theorem \ref{CWinS4vkutywsfgdredhgvcertyq2932h} characterize
$q$ as a multiplier to $\Lambda$.

Next we prove Theorem \ref{CWinS4vkutywsfgdredhgvcertyq2932h}.
\begin{proof}
First note that
$$\Lambda^{\perp}=\underline{\C}^{4}\wedge L.$$Let $q$ be a $1$-form with values in
$o((\underline{\R}^{5,1})^{\C})=sl(\underline{\C}^{4})$. Suppose
that $q\Lambda=0$ and $q\Lambda^{\perp}\subset\Lambda$.  Let
$l,l',u,v$ be a frame of $\underline{\C}^{4}$ with
$l,l'\in\Gamma(L)$. The fact that $q$ vanishes in $\Lambda$,
equivalent to
\begin{equation}\label{eq:qll'l'ql=0}
(ql)\wedge l'+l\wedge (ql')=0,
\end{equation}
shows that $ql$ has no component in $\langle u\rangle$ nor in
$\langle v\rangle$. On the other hand, since
$q\Lambda^{\perp}\subset \Lambda$, we have, in particular,
\begin{equation}\label{eq:quwedgel+lwedgequ}
(qu)\wedge l+u\wedge (ql)\in\Gamma(\langle l\wedge l'\rangle),
\end{equation}
forcing the component of $ql$ in $\langle l'\rangle $ to vanish.
Hence $ql\in\Gamma(\langle l\rangle)$, and, by symmetry of roles,
$ql'\in\Gamma(\langle l'\rangle)$. Equation \eqref{eq:qll'l'ql=0}
shows now that, if $ql=al$, with $a\in\Gamma(\underline{\C})$, then
$ql'=-al'$. By the skew-symmetry of $q$, it follows that
$$-a^{2}(l\wedge l',u\wedge v)=-(l\wedge l',qu\wedge qv).$$ Equation
\eqref{eq:quwedgel+lwedgequ} forces, on the other hand, the
component of $qu$ in $\langle v\rangle$ to vanish. Similarly, the
fact that
$$(qv)\wedge l+v\wedge (ql)\in\Gamma(\langle l\wedge l'\rangle)$$shows,
in particular, that the component of $qv$ in $\langle u\rangle$
vanishes. Thus $a=0$. We conclude that $qL=0$. Yet again by
$q\Lambda^{\perp}\subset \Lambda$, it follows that
$\mathrm{Im}\,q\subset L$. It is clear that, conversely, if $qL=0$
and $\mathrm{Im}\,q\subset L$, then $q\Lambda=0$ and
$q\Lambda^{\perp}\subset\Lambda$. We conclude that $q$
 takes values in $\Lambda\wedge\Lambda^{\perp}$, or, equivalently, $q\Lambda=0$ and $q\Lambda^{\perp}\subset\Lambda,$
if and only if $$qL=0,\,\,\,\mathrm{Im}\,q\subset L,$$i.e., $q$
defines a $1$-form with values in
$\mathrm{End}(\underline{\C}^{4}/L,L)$. Note that
$$\Lambda\wedge\Lambda^{(1)}=\Lambda\wedge\Lambda^{\perp}\cap(\wedge^{2}V\oplus \wedge^{2}V^{\perp}).$$
Now observe that if $q$ preserves $S_{+}\wedge S_{-}$ then, given
$s_{+}\in\Gamma(S_{+})$ and $s_{-}\in\Gamma(S_{+})$, the component
of $qs_{+}$ in $\langle s_{-}\rangle$ vanishes, as well as the
component of $qs_{-}$ in $\langle s_{+}\rangle$ does, which amounts
to $q$ preserving both $S_{+}$ and $S_{-}$. It is clear that,
conversely, if $q$ preserves both $S_{+}$ and $S_{-}$, then $q$
preserves $S_{+}\wedge S_{-}$, $\wedge^{2}S_{+}$ and
$\wedge^{2}S_{-}$. Hence $q$ takes values in $(\wedge^{2}V\oplus
\wedge^{2}V^{\perp})$, or, equivalently, $q$ preserves both $V$ and
$V^{\perp}$, if and only if $q$ preserves the eigenspaces of $S$. In
its turn, $q$ preserves $S_{\pm}$ if and only if $q$ commutes with
$S$ in $S_{\pm}$, respectively.

Now recall Lemma \ref{withvswithoutdecomps}, establishing that, if
$q\in\Omega^{1}(\Lambda\wedge\Lambda^{(1)})$ is real and
$d^{\mathcal{D}_{V}}q$ vanishes, then $q^{1,0}\in\Omega
^{1,0}(\Lambda \wedge\Lambda ^{0,1})$, or, equivalently,
$q^{0,1}\in\Omega ^{0,1}(\Lambda \wedge\Lambda ^{1,0})$. In that
case, $q^{1,0}(L_{-}\wedge S_{+})=0=q^{0,1}(L_{+}\wedge S_{-})$, or,
equivalently, $L_{-}\wedge q^{1,0}S_{+}=0=L_{+}\wedge q^{0,1}S_{-}$,
and, therefore,
$$q^{1,0}S_{+}=0=q^{0,1}S_{-},$$as $q^{1,0}S_{+}\subset L_{+}$ and $q^{0,1}S_{-}\subset L_{-}$.
Hence
$$\mathrm{Im}\,q^{1,0}\subset S_{-},\,\,\,\,\mathrm{Im}\,q^{0,1}\subset
S_{+},$$showing that
$$Sq^{1,0}=-iq^{1,0}=*\,q^{1,0},\,\,\,\,\,\,Sq^{0,1}=iq^{0,1}=*\,q^{0,1}$$and,
ultimately, that $Sq=*\,q$, completing the proof.
\end{proof}

Following the characterization of the harmonicity of $S$ presented
in \cite{quaternionsbook}, we have, more generally:

\begin{thm}\label{CWinS4charact}
Suppose $q\in\Omega^{1}(\mathrm{End}_{j}(\underline{\C}^{4}/L,L))$
satisfies condition \eqref{eq:Sq=qS=*q}. An immersed bundle $L$ of
$j$-stable $2$-planes in $\C^{4}$ is a $q$-constrained Willmore
surface if and only if either of the following equations is
verified:

$i)$\,\,$d*\mathcal{N}_{S}+2\,d*\,q=0;$

$ii)$\,\,$d*(A+q)=0;$

$iii)$\,\,$d*(Q+q)=0.$
\end{thm}

Before proceeding to the proof of the theorem, observe that equation
\eqref{eq:dSvsAeQ} establishes, in particular,
\begin{equation}\label{dstarAisdstarQ}
d*A=d*\,Q.
\end{equation}

Now we prove Theorem \ref{CWinS4charact}.
\begin{proof}
In view of \eqref{eq:Sq=qS=*q},
\begin{equation}\label{eq:ddsdddddddddddddddddddddpzz2}
d^{\mathcal{D}_{S}}*\,q=d^{\mathcal{D}_{S}}Sq=(\mathcal{D}_{S}S)\wedge
q+Sd^{\mathcal{D}_{S}}q=Sd^{\mathcal{D}_{S}}q,
\end{equation}
and, therefore, we have
$d*\,q=[\mathcal{N}_{S}\wedge*\,q]=-[q\wedge*\mathcal{N}_{S}]$ if
and only if $d^{\mathcal{D}_{S}}q=0.$ In particular, if $L$ is a
constrained Willmore surface, then
$$d*\mathcal{N}_{S}+2\,d*\,q=d^{\mathcal{D}_{S}}*\mathcal{N}_{S}+[\mathcal{N}_{S}\wedge
*\mathcal{N}_{S}]-2[q\wedge*\mathcal{N}_{S}]=0.$$Yet again in view of
\eqref{eq:Sq=qS=*q},
$(d^{\mathcal{D}_{S}}*\,q)S=d^{\mathcal{D}_{S}}(*\,qS)+*\,q\wedge(\mathcal{D}_{S}S)=-d^{\mathcal{D}_{S}}q$,
whereas, following \eqref{eq:ddsdddddddddddddddddddddpzz2},
$S\,d^{\mathcal{D}_{S}}*\,q=-d^{\mathcal{D}_{S}}q$. Thus
$$S\,d^{\mathcal{D}_{S}}*\,q=(d^{\mathcal{D}_{S}}*\,q)S.$$
Similarly, as
\begin{equation}\label{eq:NSSN-Q+A}
*\mathcal{N}_{S}S=-Q+A=-S*\mathcal{N}_{S},
\end{equation}
we have
$$(d^{\mathcal{D}_{S}}*\mathcal{N}_{S})S=d^{\mathcal{D}_{S}}(*\mathcal{N}_{S}S)+*\mathcal{N}_{S}\wedge(\mathcal{D}_{S}S)=d^{\mathcal{D}_{S}}(-Q+A),$$
whereas
$$d^{\mathcal{D}_{S}}*\mathcal{N}_{S}=d^{\mathcal{D}_{S}}(S(-Q+A))=(\mathcal{D}_{S}S)\wedge
(-Q+A)+S\,d^{\mathcal{D}_{S}}(-Q+A)=S\,d^{\mathcal{D}_{S}}(-Q+A),$$and,
therefore,
$$S\,d^{\mathcal{D}_{S}}*\mathcal{N}_{S}=-(d^{\mathcal{D}_{S}}*\mathcal{N}_{S})S.$$
On the other hand, together, \eqref{eq:Sq=qS=*q} and
\eqref{eq:NSSN-Q+A} establish
$$[\mathcal{N}_{S}\wedge *\,q]S=-S[\mathcal{N}_{S}\wedge*\,q].$$It follows that, if equation $i)$ holds,
$0=d^{\mathcal{D}_{S}}*\mathcal{N}_{S}+2d^{\mathcal{D}_{S}}*\,q+2[\mathcal{N}_{S}\wedge
*\,q]$, then, equivalently,
$d^{\mathcal{D}_{S}}*\mathcal{N}_{S}+2[\mathcal{N}_{S}\wedge
*\,q]=0=2d^{\mathcal{D}_{S}}*\,q$, by separating $S$-commuting and
$S$-anti-commuting parts.

Equation \eqref{dstarAisdstarQ} establishes the equivalence of $i)$
and either equations $ii)$ or $iii)$, completing the proof.
\end{proof}

We complete this section with an approach to constrained Willmore
surfaces in $4$-space in terms of flatness of connections. Let
$q\in\Omega^{1}(\mathrm{End}_{j}(\underline{\C}^{4}/L,L))$. For each
$\lambda\in\C\backslash\{0\}$, define a connection on
$\underline{\C}^{4}$ by
$$d^{\lambda,q}_{S}:=\mathcal{D}_{S}+\lambda^{-1}\mathcal{N}_{S}^{1,0}+\lambda\mathcal{N}^{0,1}_{S}+(\lambda^{-2}-1)q^{1,0}+(\lambda^{2}-1)q^{0,1}.$$
For each $\lambda$, the connection $d^{\lambda,q}_{V}$, on
$\wedge^{2}\underline{\C}^{4}$, relates to $d^{\lambda,q}_{S}$ via
equation \eqref{eq:leinitzidentif}, and so do, therefore, the
respective curvature tensors,
$$R^{d^{\lambda,q}_{V}}(\sigma_{1}\wedge\sigma_{2})=(R^{d^{\lambda,q}_{S}}\sigma_{1})\wedge\sigma_{2}+\sigma_{1}\wedge
(R^{d^{\lambda,q}_{S}}\sigma_{2}),$$for
$\sigma_{1},\sigma_{2}\in\Gamma(\underline{\C}^{4})$. Hence $L$ is
$q$-constrained Willmore if and only if $d^{\lambda,q}_{S}$ is flat,
for all $\lambda\in\C\backslash\{0\}$. For
$\lambda\in\C\backslash\{0\}$, define an orthogonal transformation
of $\wedge^{2}\underline{\C}^{4}$ by
$$\tau (\lambda)
:=I\left\{
\begin{array}{ll} \lambda^{-1} & \mbox{$\mathrm{on}\,\wedge^{2}S_{+}$}\\ 1 &
\mbox{$\mathrm{on}\,S_{+}\wedge S_{-}$}\\ \lambda &
\mbox{$\mathrm{on}\,\wedge^{2}S_{-}$}\end{array}\right..$$ Given
$\lambda\in\C\backslash\{0\}$, we define a new connection on
$\wedge^{2}\underline{\C}^{4}$ by
$$d_{\lambda^{2},q}^{V}:=\tau(\lambda)\circ d^{\lambda,q}_{V}\circ
\tau(\lambda)^{-1};$$as well as, fixing a choice of
$\sqrt{\lambda}$, and, with respect to (and independently of) this
choice, a new connection on $\underline{\C}^{4}$ by
$$d_{\lambda^{2},q}^{S}:=\tau(\lambda)\circ d^{\lambda,q}_{S}\circ
\tau(\lambda)^{-1},$$for $\tau \in\Gamma(Sl(\underline{\C}^{4}))$ as
defined in Section \ref{LinAlg}. These definitions only depend,
indeed, on $\lambda^{2}$, as, in particular, the next lemma makes
clear (note that $d_{\lambda^{2},q}^{V}$ is related to
$d_{\lambda^{2},q}^{S}$ by equation \eqref{eq:leinitzidentif}). The
curvature tensors of $d^{V}_{\lambda^{2},q}$ and $d^{\lambda,q}_{V}$
are related by
$$R^{d^{V}_{\lambda^{2},q}}=\tau(\lambda)\,R^{d^{\lambda,q}_{V}}\,\tau(\lambda)^{-1},$$showing
that the flatness of $d^{V}_{\lambda^{2},q}$, or, equivalently, that
of $d_{\lambda^{2},q}^{S}$, for all $\lambda\in\C\backslash\{0\}$,
provides another characterization of the $q$-constrained Willmore
condition of $L$. It is obvious but it will be useful to remark
that, for each $\lambda$, the $d_{\lambda^{2},q}^{S}$-parallelness
of a bundle $W\subset\underline{\C}^{4}$ is equivalent to the
$d_{\lambda^{2},q}^{V}$-parallelness of $\wedge^{2}W\subset
\wedge^{2}\underline{\C}^{4}$.

\begin{Lemma}\label{d^S_lambda^2orm}
For each $\lambda\in\C\backslash\{0\}$,
$$d_{\lambda^{2},q}^{S}=
d+(\lambda^{-2}-1)(Q^{1,0}+q^{1,0})+(\lambda^{2}-1)(Q^{0,1}+q^{0,1}).$$
\end{Lemma}
\begin{proof}
Fix $\lambda\in\C\backslash\{0\}$. The fact that both $q^{1,0}$ and
$q^{0,1}$ preserve $S_{+}$ and $S_{-}$ establishes, in particular,
$$\tau(\lambda)\circ q\circ
\tau(\lambda)^{-1}=q.$$Similarly, in view of the
$\mathcal{D}_{S}$-parallelness of $S_{+}$ and $S_{-}$ (cf.
\eqref{eq:DS+}), $$\tau(\lambda)\circ \mathcal{D}_{S}\circ
\tau(\lambda)^{-1}=\mathcal{D}_{S}.$$ On the other hand, according
to \eqref{eq:AQSpm}, \eqref{eq:A10S-} and \eqref{eq:todosondesao0},
$$\tau(\lambda)\circ A^{1,0}\circ \tau(\lambda)^{-1}=\lambda
A^{1,0},\,\,\,\,\tau(\lambda)\circ A^{0,1}\circ
\tau(\lambda)^{-1}=\lambda ^{-1}A^{0,1}$$and
$$\tau(\lambda)\circ Q^{1,0}\circ \tau(\lambda)^{-1}=\lambda ^{-1} Q^{1,0},\,\,\,\,\tau(\lambda)\circ Q^{0,1}\circ
\tau(\lambda)^{-1}=\lambda Q^{0,1}.$$Hence
$$
d_{\lambda^{2},q}^{S}=\mathcal{D}_{S}+(\lambda^{-2}-1)(Q^{1,0}+q^{1,0})+(\lambda^{2}-1)(Q^{0,1}+q^{0,1})+A+Q,$$
completing the proof.
\end{proof}

\section{Transformations of constrained Willmore surfaces in $4$-space}

\markboth{\tiny{A. C. QUINTINO}}{\tiny{CONSTRAINED WILLMORE
SURFACES}}

Under the standard identification $sl(\C^{4})\cong
o(\wedge^{2}\C^{4})$, $1$-forms with values in
$\Lambda\wedge\Lambda^{(1)}$ correspond to $S$-commuting $1$-forms
with values in $\mathrm{End}_{j}(\underline{\C}^{4}/L,L)$, for $S$
the mean curvature sphere congruence of $L\cong\Lambda$. For such a
form $q$, condition $d^{\mathcal{D}}q=0$ establishes $Sq=*q$. A
surface $L$ in $S^{4}$ is a $q$-constrained Willmore surface, for
some $1$-form $q$ with values in
$\mathrm{End}_{j}(\underline{\C}^{4}/L,L)$ such that $Sq=*\,q=qS$,
if and only if $d*(Q+q)=0$, for the Hopf field
$Q\in\Omega^{1}(\mathrm{End}_{j}(\underline{\C}^{4}))$. The
closeness of the $1$-form $*(Q+q)$ ensures the existence of
$G\in\Gamma(\mathrm{End}_{j}(\underline{\C}^{4}))$ with
$dG=2*(Q+q)$, as well as the integrability of the Riccati equation
$dT=\rho T(dG)T-dF+4\rho qT$, for each $\rho\in\R\backslash\{0\}$,
fixing such a $G$ and setting $F:=G-S$. For a local solution
$T\in\Gamma(Gl_{j}(\underline{\C}^{4}))$ of the $\rho$-Ricatti
equation, we define the \textit{constrained Willmore Darboux
transform} of $L$ of parameters $\rho,T$ by setting
$\hat{L}:=T^{-1}L$, and extend, in this way, the Darboux
transformation of Willmore surfaces in $S^{4}$ presented in
\cite{quaternionsbook} to a transformation of constrained Willmore
surfaces in $4$-space. We apply, yet again, the dressing action
presented in Chapter \ref{transformsofCW} to define another
transformation of constrained Willmore surfaces in $4$-space, the
\textit{untwisted B\"{a}cklund transformation}, referring then to
the original one as the \textit{twisted B\"{a}cklund
transformation}. We verify that, when both are defined, twisted and
untwisted B\"{a}cklund transformations coincide. We prove that
constrained Willmore Darboux transformation of parameters $\rho,T$
with $\rho
> 1$ is equivalent to untwisted B\"{a}cklund transformation of
parameters $\alpha,L^{\alpha}$ with $\alpha^{2}$ real. Constrained
Willmore Darboux transformation of parameters $\rho,T$ with $\rho
\leq 1$ is trivial.\newline

Let $L\cong\Lambda\subset\underline{\R}^{5,1}$ be a $q$-constrained
Willmore surface in $S^{4}$, for some $1$-form $q$ with values in
$\mathrm{End}_{j}(\underline{\C}^{4}/L,L)$. Let $S$ be the mean
curvature sphere of $L$ and $V$ be the complexification of the
central sphere congruence of $\Lambda$. Let $\rho_{V}$ be reflection
across $V$. Provide $M$ with the conformal structure induced by
$\Lambda$. For simplicity, in what follows, given
$\alpha\in\C\backslash S^{1}$ non-zero and $L^{\alpha}$ a
$d^{\alpha,q}_{V}$-parallel null line subbundle of
$\wedge^{2}\underline{\C}^{4}$ such that, locally,
\begin{equation}\label{eq:BTparameterscondition}
\rho_{V} L^{\alpha}\cap
(L^{\alpha})^{\perp}=\{0\}=\rho_{V}\tilde{L}^{\alpha}\cap
(\tilde{L}^{\alpha})^{\perp},
\end{equation}
we consider the alternative notation
$$p:=p_{\alpha,\tilde{L}^{\alpha}},\,\,\,q:=q_{\overline{\alpha}\,^{-1},\overline{L^{\alpha}}},$$
with no risk of ambiguity with the multiplier to $\Lambda$. Recall
that, for such $\alpha$ and $L^{\alpha}$, we define another
constrained Willmore surface in $S^{4}$, the B\"{a}cklund transform
$\Lambda^{*}$ of $\Lambda$ of parameters $\alpha,L^{\alpha}$, or,
equally, $-\alpha,\rho_{V} L^{\alpha}$, by setting
\begin{equation}\label{eq:BTofCWdef}
(\Lambda^{*})^{0,1}:=(pq)^{-1}(1)(pq)(0)\Lambda^{0,1},
\end{equation}
provided that $\Lambda^{*}$ immerses.

\subsection{Untwisted B\"{a}cklund transformation of constrained Willmore surfaces in $4$-space}\label{genBT}

In this section, we apply, yet again, the dressing action presented
in Section \ref{sec:dress} to define another transformation of
constrained Willmore surfaces in $4$-space, the \textit{untwisted
B\"{a}cklund transformation}, referring then to the B\"{a}cklund
transformation defined by \eqref{eq:BTofCWdef} as the
\textit{twisted  B\"{a}cklund transformation}. The terminology is
motivated by the fact that, whilst $pq(\lambda)$ relates to
$pq(-\lambda)$ via the twisting $\rho_{V}
pq(\lambda)\rho_{V}=pq(-\lambda)$, performed by $\rho_{V}$, the
untwisted transformation will be given by
$(\Lambda^{*})^{0,1}=P(1)^{-1}P(0)\Lambda^{0,1}$ for
$P(\lambda)=g^{-1}R(\lambda)\in\Gamma(O(\wedge^{2}\underline{\C}^{4}))$,
with $g\in\Gamma(SO(\underline{\R}^{5,1}))$, for some
$R(\lambda)\in\Gamma(O(\wedge^{2}\underline{\C}^{4}))$ for which no
such a relation with $R(-\lambda)$ is established.\newline

First of all, observe that, given $w:=v_{0}+v_{+}+v_{-}\neq 0$, with
$v_{0}\in\Gamma(S_{+}\wedge S_{-})$,
$v_{+}\in\Gamma(\wedge^{2}S_{+})$ and
$v_{-}\in\Gamma(\wedge^{2}S_{-})$, $w$ is null if and only if
\begin{equation}\label{eq:characterizdenull}
(v_{0},v_{0})=-2(v_{+},v_{-}).
\end{equation}
The nullity of $w$ is equivalently characterized by, at each point,
either $v_{-}=0=(v_{0},v_{0})$ or both $v_{-}\neq 0$ and
\begin{equation}\label{eq:v+vsv-}
v_{+}=-\frac{1}{2}\,(v_{0},v_{0})(v_{-},\overline{v_{-}})^{-1}\overline{v_{-}}.
\end{equation}

\begin{Lemma}\label{whentauLreal}
Let $\alpha\in \C\backslash\{0\}$. Let $v_{0}\in\Gamma(S_{+}\wedge
S_{-})$, $v_{+}\in\Gamma(\wedge^{2}S_{+})$ and
$v_{-}\in\Gamma(\wedge^{2}S_{-})$ be such that $v_{0}+v_{+}+v_{-}$
is null. In that case, $\tau(\alpha)\langle v_{0}+v_{+}+v_{-}\rangle
$ is real at a point $p$ in $M$ if and only if, at $p$, either
$$v_{+}=0=v_{-},\,\,\,\langle\overline{v_{0}}\rangle=\langle
v_{0}\rangle$$ or
$$\mid(v_{0},v_{0})\mid =1,\,\,\,(v_{-},\overline{v_{-}})=\frac{1}{2}\mid\alpha\mid^{-2},\,\,\,\overline{v_{0}}=-v_{0}.$$
\end{Lemma}
\begin{proof}
Our argument is pointwise, so we work in a single fibre.

Set $w:=v_{0}+v_{+}+v_{-}$. From the definition of $\tau(\alpha)$,
it follows that
$\langle\tau(\alpha)w,\overline{\tau(\alpha)w}\rangle$ has rank $1$
if and only if either
$$v_{+}=0=v_{-},\,\,\,\overline{v_{0}}=\lambda v_{0}$$ or $$ v_{-}\neq
0,\,\,\,
\overline{v_{+}}=\lambda\mid\alpha\mid^{2}v_{-},\,\,\,\overline{v_{0}}=\lambda
v_{0},\,\,\,\mid\lambda\mid^{2}=1,$$ for some
$\lambda\in\Gamma(\underline{\C})$. In particular, because $w$ is
null, if $v_{-}$ is non-zero, then, according to \eqref{eq:v+vsv-},
$\tau(\alpha)\langle w\rangle$ is real if and only if
\begin{equation}\label{eq:55433555455ghsg nxxhjzkl?}
(v_{-},\overline{v_{-}})=\frac{1}{2}\mid\alpha\mid^{-2}\mid(v_{0},v_{0})\mid
\end{equation}
and
\begin{equation}\label{eq:2qwasdxfgyu90'p,mn ji89765resd6786}
\overline{v_{0}}=-\frac{1}{2}\mid\alpha\mid^{-2}\overline{(v_{0},v_{0})}(v_{-},\overline{v_{-}})^{-1}v_{0}.
\end{equation}
Note that, if $v_{-}$ is non-zero, equation
\eqref{eq:55433555455ghsg nxxhjzkl?} forces $(v_{0},v_{0})$ to be
non-zero. On the other hand, in view of the nullity of $w$, if
$(v_{0},v_{0})$ is non-zero then so is $v_{-}$. The proof is
complete by observing that, in that case, together, equations
\eqref{eq:55433555455ghsg nxxhjzkl?} and
\eqref{eq:2qwasdxfgyu90'p,mn ji89765resd6786} force
$\overline{v_{0}}=-(v_{0},v_{0})^{-1}v_{0}$, so that $v_{0}=\mid
(v_{0},v_{0})\mid ^{-2}v_{0}$ and, ultimately,
$\mid(v_{0},v_{0})\mid=1$.
\end{proof}

Choose a non-zero $\alpha\in\C\backslash S^{1}$ and a
$d^{\alpha,q}_{V}$-parallel null line subbundle $L^{\alpha}$ of
$\wedge^{2}\underline{\C}^{4}$ such that, locally,
$\tau(\alpha)L^{\alpha}$ is never real. The existence of such a
choice of $L^{\alpha}$ is established in the following lemma.
\begin{Lemma}\label{tauLneverreal}
Let $l^{\alpha}$ be a non-zero section of $\wedge^{2}S_{+}$. Let
$L^{\alpha}\subset\wedge^{2}\underline{\C}^{4}$ be a
$d^{\alpha,q}_{V}$-parallel null line bundle defined naturally by
$d^{\alpha,q}_{V}$-parallel transport of $l^{\alpha}_{p}$, for some
point $p\in M$. Then $\tau(\alpha)L^{\alpha}$ is never real in some
open set containing $p$.
\end{Lemma}
Before proceeding to the proof of the lemma, observe that, at each
point, the non-reality of $\tau(\alpha)L^{\alpha}$ is equivalent to
the real subspace
$\tau(\alpha)L^{\alpha}+\overline{\tau(\alpha)L^{\alpha}}$ of
$(\R^{5,1})^{\C}$ being $2$-dimensional, or, equivalently, non-null,
\begin{equation}\label{eq:taulalphanotreal}
\overline{\tau(\alpha)L^{\alpha}}\cap
(\tau(\alpha)L^{\alpha})^{\perp}=\{0\},
\end{equation}
which characterizes the non-degeneracy of
$\tau(\alpha)L^{\alpha}\oplus\overline{\tau(\alpha)L^{\alpha}}.$
\begin{proof}
At the point $p$, $L^{\alpha}$ is spanned by $l^{\alpha}_{p}$, so
that, by construction, at this point, $\tau(\alpha)L^{\alpha}$ is
not real. On the other hand, at each point, the non-reality of
$\tau(\alpha)L^{\alpha}$ is characterized by
\eqref{eq:taulalphanotreal}, or, equivalently, by
$$(\tau(\alpha)l,\overline{\tau(\alpha)l}\,)\neq 0,$$fixing
$l\in\Gamma(L^{\alpha})$ non-zero; showing that the non-reality of
$(\tau(\alpha)L^{\alpha})_{x}$ is an open condition on $x\in M$.
\end{proof}
At each point, the non-reality of $\tau(\alpha)L^{\alpha}$,
characterized by equation \eqref{eq:taulalphanotreal}, establishes a
decomposition
\begin{equation}\label{eq:taudecomp}
\wedge^{2}\underline{\C}^{4}=\tau(\alpha)L^{\alpha}\oplus
(\tau(\alpha)L^{\alpha}\oplus\overline{\tau(\alpha)L^{\alpha}})^{\perp}\oplus
\overline{\tau(\alpha)L^{\alpha}}.
\end{equation}
For
$\lambda\in\C\backslash\{\pm\alpha,\pm\overline{\alpha}\,^{-1}\}$
and
\begin{equation}\label{eq:ralphalambdathe number}
r_{\alpha}(\lambda):=\frac{1-\overline{\alpha}\,^{-2}}{1-\alpha^{2}}\frac{\lambda-\alpha^{2}}{\lambda-\overline{\alpha}\,^{-2}}\neq
0,
\end{equation}
set then
$$r_{\alpha,L^{\alpha}} (\lambda)
:=I\left\{
\begin{array}{ll} r_{\alpha}(\lambda) & \mbox{$\mathrm{on}\,\tau (\alpha)L^{\alpha}$}\\ 1 &
\mbox{$\mathrm{on}\,(\tau (\alpha)L^{\alpha}+\overline{\tau (\alpha)L^{\alpha}})^{\perp}$}\\
r_{\alpha}(\lambda)^{-1} & \mbox{$\mathrm{on}\,\overline{\tau
(\alpha)L^{\alpha}}$}\end{array}\right.,$$defining this way a map
$r_{\alpha,L^{\alpha}}$ into
$\Gamma(O(\wedge^{2}\underline{\C}^{4}))$ with
$$\overline{r_{\alpha,L^{\alpha}}(\overline{\lambda}\,^{-1})}=r_{\alpha,L^{\alpha}}(\lambda),$$for
all $\lambda$, which, therefore, extends holomorphically to
$\mathbb{P}^{1}\backslash\{\pm\alpha,\pm\overline{\alpha}\,^{-1}\}$
by setting
$$r_{\alpha,L^{\alpha}}(\infty):=\overline{r_{\alpha,L^{\alpha}}(0)}.$$
We may, alternatively and for simplicity, write $r(\lambda)$ for
$r_{\alpha,L^{\alpha}} (\lambda)$ with
$\lambda\in\mathbb{P}^{1}\backslash\{\pm\alpha,\pm\overline{\alpha}\,^{-1}\}$.
It will be useful to observe that, in the case $\alpha^{2}$ is real,
we have
\begin{equation}\label{eq:rinf=r0-1}
r(\infty)=r(0)^{-1}.
\end{equation}
Once and for all, fix a choice of $\sqrt{r_{\alpha}(0)}$, set
$$\sqrt{\overline{r_{\alpha}(0)}}:=\overline{\sqrt{r_{\alpha}(0)}},$$
and consider the corresponding
$r(0),r(\infty)\in\Gamma(\mathrm{Sl(\underline{\C}^{4})})$, for
which equation \eqref{eq:xijjxi} then applies:
\begin{equation}\label{eq:r0jjrinfty}
r(0)\circ j=j\circ r(\infty).
\end{equation}
It follows, in particular, that, at each point, the $j$-stability of
$r(0)S_{+}$ is characterized by
$$r(0)S_{+}=r(\infty)S_{-},$$
or, equivalently, by the non-complementarity of $r(0)S_{+}$ and
$r(\infty)S_{-}$ in $\underline{\C}^{4}$,
\begin{equation}\label{eq:jstablityofr(0)=S+}
(\wedge^{2}r(0)S_{+})\cap
(\wedge^{2}r(\infty)S_{-})^{\perp}\neq\{0\}.
\end{equation}

Next we characterize the $j$-stability of $r(0)S_{+}$ in terms of
the projections of $L^{\alpha}$ with respect to the decomposition
\eqref{eq:w2C4w2s+w2s-s+s-}.

\begin{Lemma}\label{whenr0S+jstable}
Let $l^{\alpha}:=v_{0}+v_{+}+v_{-}$, with
$v_{0}\in\Gamma(S_{+}\wedge S_{-})$,
$v_{+}\in\Gamma(\wedge^{2}S_{+})$ and
$v_{-}\in\Gamma(\wedge^{2}S_{-})$, be a non-zero section of
$L^{\alpha}$. Let $\pi_{\oplus}$ and $\pi_{\perp}$ denote the
orthogonal projections of $\wedge^{2}\underline{\C}^{4}$ onto
$\tau(\alpha)L^{\alpha}\oplus\overline{\tau(\alpha)L^{\alpha}}$ and
$(\tau(\alpha)L^{\alpha}\oplus\overline{\tau(\alpha)L^{\alpha}})^{\perp}$,
respectively. At each point at which $l^{\alpha}$ does not vanish,
the bundle $r(0)S_{+}$ is $j$-stable if and only if
\begin{equation}\label{eq:v+v-not0}
v_{+}\neq 0\neq v_{-}
\end{equation}
and, given $u_{+}\in\Gamma(\wedge^{2}S_{+})$ not zero,

\begin{equation}\label{eq:hags3eddy678tjn uy8 wej7 e---,nvfc}
\mid(l^{\alpha},u_{+})\mid\mid(l^{\alpha},\overline{u_{+}}\,)\mid^{-1}=\mid
r_{\alpha}(0)\mid^{2}\mid\alpha\mid^{-2},
\end{equation}
together with

\begin{equation}\label{eq:12345rfdcx78ujhnmop98765432sdcvb,mjhnl?s.ag}
\pi_{\perp}u_{+}\neq 0,\,\,\,\,\,\,\langle\pi_{\perp}u_{+}\rangle
=\langle\overline{\pi_{\perp}u_{+}}\,\rangle
\end{equation}
and

\begin{equation}\label{eq:nolambdae-lambda}
(\pi_{\oplus}r(0)u_{+},\pi_{\oplus}\overline{r(0)u_{+}}\,)\neq
(\pi_{\perp}r(0)u_{+},\pi_{\perp}\overline{r(0)u_{+}}\,).
\end{equation}
\end{Lemma}

\begin{proof}
Our argument is pointwise, so we work in a single fibre.

Let $s_{+},s_{+}'$ be a frame of $S_{+}$. Then $r(0)S_{+}$ is
$j$-stable if and only if
\begin{equation}\label{eq:rinftr0js+s+'}
r(0)(s_{+}\wedge s_{+}')=\lambda r(\infty)(js_{+}\wedge js_{+}'),
\end{equation}
for some $\lambda\in\Gamma(\underline{\C})$. Note that such a
$\lambda$ has necessarily unit length. Indeed,
$$r(\infty)(js_{+}\wedge js_{+}')=\overline{r(0)(s_{+}\wedge
s_{+}')},$$ so that
$$r(\infty)(js_{+}\wedge
js_{+}')=\overline{\lambda}\,\overline{r(\infty)(js_{+}\wedge
js_{+}')}=\overline{\lambda}\,r(0)(s_{+}\wedge
s_{+}')=\mid\lambda\mid^{2}r(\infty)(js_{+}\wedge js_{+}').$$ Write
$s_{+}\wedge
s_{+}'=a\,\tau(\alpha)l^{\alpha}+b\,\overline{\tau(\alpha)l^{\alpha}}+\pi_{\perp}(s_{+}\wedge
s_{+}')$, with $a,b\in\Gamma(\underline{\C})$. Then
$$r(0)(s_{+}\wedge
s_{+}')=a\,r_{\alpha}(0)\tau(\alpha)l^{\alpha}+b\,r_{\alpha}(0)^{-1}\overline{\tau(\alpha)l^{\alpha}}+\pi_{\perp}(s_{+}\wedge
s_{+}'),$$whereas $$r(\infty)(js_{+}\wedge
js_{+}')=\overline{b}\,\overline{r_{\alpha}(0)}\,^{-1}\tau(\alpha)l^{\alpha}+\overline{a}\,\overline{r_{\alpha}(0)}\,\overline{\tau(\alpha)l^{\alpha}}+\overline{\pi_{\perp}(s_{+}\wedge
s_{+}')}.$$Hence equation \eqref{eq:rinftr0js+s+'} holds if and only
if
$$a\,r_{\alpha}(0)=\lambda\,\overline{b}\,\overline{r_{\alpha}(0)}\,^{-1},\,\,\,b\,r_{\alpha}(0)^{-1}=\lambda\,\overline{a}\,\overline{r_{\alpha}(0)},\,\,\,\pi_{\perp}(s_{+}\wedge
s_{+}')=\lambda\,\overline{\pi_{\perp}(s_{+}\wedge s_{+}')};$$or,
equivalently, at each point, either
\begin{equation}\label{eq:cccvfartaeeee}
a=0=b,\,\,\,\,\pi_{\perp}(s_{+}\wedge
s_{+}')=\lambda\,\overline{\pi_{\perp}(s_{+}\wedge s_{+}')},
\end{equation}
or $a,b\neq 0$ and
$$a\overline{b}\,^{-1}\mid
r_{\alpha}(0)\mid^{2}=\lambda=\overline{a}\,^{-1}b\mid
r_{\alpha}(0)\mid^{-2},\,\,\,\,\,\,\pi_{\perp}(s_{+}\wedge
s_{+}')=\lambda\overline{\pi_{\perp}(s_{+}\wedge s_{+}')}.$$ Note
that \eqref{eq:cccvfartaeeee} contradicts the complementarity of
$\wedge^{2}S_{+}$ and $\wedge^{2}S_{-}$. Set
$$X:=(\tau(\alpha)l^{\alpha},\overline{\tau(\alpha)l^{\alpha}})\neq 0.$$
In view of the nullity of $s_{+}\wedge s_{+}'$,
\begin{equation}\label{eq:nullityofs+wedges+'}
(\pi_{\perp}(s_{+}\wedge s_{+}'),\pi_{\perp}(s_{+}\wedge s_{+}'))=-2
abX,
\end{equation}
equation $a\overline{b}\,^{-1}\mid
r_{\alpha}(0)\mid^{2}=\overline{a}\,^{-1}b\mid
r_{\alpha}(0)\mid^{-2},$ equivalent to
$$\mid b\mid =\mid a\mid\mid r_{\alpha}(0)\mid^{2},$$ establishes $$2\mid
a\mid^{2}\mid r_{\alpha}(0)\mid^{2}\mid X\mid =\mid
(\pi_{\perp}(s_{+}\wedge s_{+}'),\pi_{\perp}(s_{+}\wedge
s_{+}'))\mid$$and, in particular, that, if, at some point, $a\neq
0$, then, at that point,
$$(\pi_{\perp}(s_{+}\wedge s_{+}'),\pi_{\perp}(s_{+}\wedge
s_{+}'))\neq0.$$ It follows, in particular, that, if
$\pi_{\perp}(s_{+}\wedge
s_{+}')=\lambda\overline{\pi_{\perp}(s_{+}\wedge s_{+}')}$, then
\begin{equation}\label{eq:piperprooverlpiperpronot0}
(\pi_{\perp}r(0)(s_{+}\wedge
s_{+}'),\overline{\pi_{\perp}r(0)(s_{+}\wedge s_{+}')}\,)\neq 0,
\end{equation}
as $\pi_{\perp}r(0)(s_{+}\wedge s_{+}')=\pi_{\perp}(s_{+}\wedge
s_{+}')$. On the other hand, as remarked previously, $r(0)S_{+}$ is
$j$-stable if and only if $(r(0)(s_{+}\wedge s_{+}'),\overline{
r(0)(s_{+}\wedge s_{+}')}\,)=0$, or, equivalently,
$$(\pi_{\oplus}r(0)(s_{+}\wedge s_{+}'),\pi_{\oplus}\overline{r(0)(s_{+}\wedge s_{+}')}\,)=
-(\pi_{\perp}r(0)(s_{+}\wedge
s_{+}'),\pi_{\perp}\overline{r(0)(s_{+}\wedge s_{+}')}\,),$$in which
case, in particular, having in consideration
\eqref{eq:piperprooverlpiperpronot0},
$$(\pi_{\oplus}r(0)(s_{+}\wedge
s_{+}'),\pi_{\oplus}\overline{r(0)(s_{+}\wedge s_{+}')}\,)\neq
(\pi_{\perp}r(0)(s_{+}\wedge
s_{+}'),\pi_{\perp}\overline{r(0)(s_{+}\wedge s_{+}')}\,).$$ We
conclude that the bundle $r(0)S_{+}$ is $j$-stable if and only if
$a,b,\pi_{\perp}(s_{+}\wedge s_{+}')\neq 0$ and
\begin{equation}\label{eq:modbmodar0}
\mid b\mid =\mid a\mid\mid r_{\alpha}(0)\mid^{2},
\end{equation}
together with
$$\pi_{\perp}(s_{+}\wedge s_{+}')=a\overline{b}\,^{-1}\mid
r_{\alpha}(0)\mid^{2}\overline{\pi_{\perp}(s_{+}\wedge s_{+}')}$$and
condition \eqref{eq:nolambdae-lambda}. Now write
$v_{+}=a_{+}\,s_{+}\wedge s_{+}'$ and $v_{-}=a_{-}\,js_{+}\wedge
js_{+}'$, with $a_{+},a_{-}\in\Gamma(\underline{\C})$. Set
$$N:=(s_{+}\wedge s_{+}',js_{+}\wedge js_{+}')>0. $$We have
$$a=X^{-1}(s_{+}\wedge
s_{+}',\overline{\tau(\alpha)l^{\alpha}})=\overline{\alpha}\,^{-1}\,\overline{a_{+}}\,X^{-1}N,$$
as well as $$b=X^{-1}(s_{+}\wedge
s_{+}',\tau(\alpha)l^{\alpha})=\alpha\, a_{-}\,X^{-1}N.$$In
particular, $a$ (respectively, $b$) vanishes at some point if and
only if $a_{+}$ (respectively, $a_{-}$), or, equivalently, $v_{+}$
(respectively, $v_{-}$) does. We verify, on the other hand, that
equation \eqref{eq:modbmodar0} holds if and only if
$$\mid a_{-}\mid=\mid a_{+}\mid \mid r_{\alpha}(0)\mid^{2}\mid
\alpha\mid^{-2},$$or, equivalently, if condition
\eqref{eq:hags3eddy678tjn uy8 wej7 e---,nvfc} is satisfied. It
follows, in particular, that, together, conditions
\eqref{eq:v+v-not0} and \eqref{eq:hags3eddy678tjn uy8 wej7
e---,nvfc} ensure that
$$\pi_{\oplus}r(0)(s_{+}\wedge s_{+}')=\lambda\,\pi_{\oplus}r(\infty)(js_{+}\wedge js_{+}'),$$
for $\lambda:=a\overline{b}\,^{-1}\mid r_{\alpha}(0)\mid^{2}\in
S^{1}$, which together with condition
\eqref{eq:12345rfdcx78ujhnmop98765432sdcvb,mjhnl?s.ag}, and given
that $$\pi_{\perp}r(0)(s_{+}\wedge s_{+}')=\pi_{\perp}(s_{+}\wedge
s_{+}')=\overline{\pi_{\perp}r(\infty)(js_{+}\wedge
js_{+}')},$$amounts to
$$r(0)(s_{+}\wedge s_{+}')=\lambda\,\pi_{\oplus}r(\infty)(js_{+}\wedge js_{+}')+\xi\,\pi_{\perp}r(\infty)(js_{+}\wedge
js_{+}'),$$for some $\xi\in\Gamma(\underline{\C})$. In fact, we
verify that $\xi$ is unit:
$$\mid(\pi_{\perp}(s_{+}\wedge s_{+}'),\pi_{\perp}(s_{+}\wedge
s_{+}'))\mid=\mid \xi\mid^{2}\mid(\pi_{\perp}(s_{+}\wedge
s_{+}'),\pi_{\perp}(s_{+}\wedge s_{+}'))\mid,$$and condition
\eqref{eq:v+v-not0} ensures, on the other hand, according to
\eqref{eq:nullityofs+wedges+'}, that
\begin{equation}\label{eq:eqnotzeropiperpofs+wedges*':}
(\pi_{\perp}(s_{+}\wedge s_{+}'),\pi_{\perp}(s_{+}\wedge
s_{+}'))\neq 0.
\end{equation}
For simplicity, write $\pi_{\oplus}^{r}$ for
$\pi_{\oplus}r(\infty)(js_{+}\wedge js_{+}')$.  By the nullity of
$s_{+}\wedge s_{+}'$, it follows that
$$0=(r(0)(s_{+}\wedge s_{+}'),r(0)(s_{+}\wedge
s_{+}'))=\lambda^{2}(\pi_{\oplus}^{r},\pi_{\oplus}^{r})+\xi^{2}(\pi_{\perp}(s_{+}\wedge
s_{+}'),\pi_{\perp}(s_{+}\wedge s_{+}')),$$ as well as
$$0=(r(\infty)(js_{+}\wedge js_{+}')),r(\infty)(js_{+}\wedge
js_{+}')))=(\pi_{\oplus}^{r},\pi_{\oplus}^{r})+(\pi_{\perp}(s_{+}\wedge
s_{+}'),\pi_{\perp}(s_{+}\wedge s_{+}')).$$ Hence
\begin{equation}\label{eq:perp-oplus}
(\pi_{\oplus}^{r},\pi_{\oplus}^{r})=-(\pi_{\perp}(s_{+}\wedge
s_{+}'),\pi_{\perp}(s_{+}\wedge s_{+}'))
\end{equation}
and, therefore, $\xi^{2}=\lambda^{2}$, which, as $\xi$ and $\lambda$
are both unit, forces $\lambda,\xi\in\{-1,1,-i,i\}$. Together with
\eqref{eq:eqnotzeropiperpofs+wedges*':} and \eqref{eq:perp-oplus},
the nullity of $r(0)(s_{+}\wedge s_{+}')$ excludes the cases $\xi
=\pm i\lambda$. Indeed, if $r(0)(s_{+}\wedge
s_{+}')=\lambda\,\pi_{\oplus}^{r}\pm
i\lambda\,\pi_{\perp}(s_{+}\wedge s_{+}')$, then
$$0=(\lambda\,\pi_{\oplus}^{r}\pm i\lambda\,\pi_{\perp}(s_{+}\wedge s_{+}'),\lambda\,\pi_{\oplus}^{r}\pm
i\lambda\,\pi_{\perp}(s_{+}\wedge s_{+}')),$$ respectively, leading
to a contradiction:
$$0=\lambda^{2}(\pi_{\oplus}^{r},\pi_{\oplus}^{r})-\lambda^{2}(\pi_{\perp}(s_{+}\wedge
s_{+}'),\pi_{\perp}(s_{+}\wedge
s_{+}'))=-2\lambda^{2}(\pi_{\perp}(s_{+}\wedge s_{+}'),
\pi_{\perp}(s_{+}\wedge s_{+}'))\neq 0.$$ Together with condition
\eqref{eq:nolambdae-lambda}, and given that $r(0)(s_{+}\wedge
s_{+}')$ and $r(\infty)(js_{+}\wedge s_{+}')$ are complex conjugate
of each other, the nullity of $r(0)(s_{+}\wedge s_{+}')$ excludes,
on the other hand, the case $\xi=-\lambda$: if $r(0)(s_{+}\wedge
s_{+}')=\lambda\,\pi_{\oplus}^{r}-\lambda\,\pi_{\perp}(s_{+}\wedge
s_{+}')$, then
$$0=(\lambda\,\pi_{\oplus}^{r}-\lambda\,\pi_{\perp}(s_{+}\wedge
s_{+}'), \overline{\pi_{\oplus}^{r}+\pi_{\perp}(s_{+}\wedge
s_{+}')}\,)=\lambda(\pi_{\oplus}^{r},\overline{\pi_{\oplus}^{r}})-\lambda(\pi_{\perp}(s_{+}\wedge
s_{+}'),\overline{\pi_{\perp}(s_{+}\wedge s_{+}')}\,)$$ and,
therefore,
$(\pi_{\oplus}^{r},\overline{\pi_{\oplus}^{r}})=(\pi_{\perp}(s_{+}\wedge
s_{+}'),\overline{\pi_{\perp}(s_{+}\wedge s_{+}')}\,)$, or,
equivalently, $$(\pi_{\oplus}\overline{r(0)(s_{+}\wedge
s_{+}')},\pi_{\oplus}r(0)(s_{+}\wedge s_{+}'))=
(\pi_{\perp}r(0)(s_{+}\wedge
s_{+}'),\pi_{\perp}\overline{r(0)(s_{+}\wedge s_{+}')}\,).$$We
conclude that $\xi=\lambda$, completing the proof.
\end{proof}

At each point, the non-$j$-stability of the bundle $r(0)S_{+}$, of
$2$-planes in $\C^{4}$, is equivalent to its complementarity to
$jr(0)S_{+}=r(\infty)S_{-}$ in $\underline{\C}^{4}$. Choose
$L^{\alpha}$ for which, furthermore, locally, $r(0)S_{+}$ and
$r(\infty)S_{-}$ are complementary in $\underline{\C}^{4}$. The
existence of such a choice is established in the following lemma.

\begin{Lemma}
Let $l^{\alpha}$ be a non-zero section of $\wedge^{2}S_{+}$. Let
$L^{\alpha}\subset\wedge^{2}\underline{\C}^{4}$ be a
$d^{\alpha,q}_{V}$-parallel null line bundle defined naturally by
$d^{\alpha,q}_{V}$-parallel transport of $l^{\alpha}_{p}$, for some
point $p\in M$. Then there is some open set containing $p$ in which
$\tau(\alpha)L^{\alpha}$ is never real and $r(0)S_{+}$ and
$r(\infty)S_{-}$ are complementary in $\underline{\C}^{4}$.
\end{Lemma}
\begin{proof}
By construction,
$$L^{\alpha}_{p}=\langle l^{\alpha}_{p}\rangle,$$
so that, at the point $p$, $\tau(\alpha)L^{\alpha}$ is not real and,
therefore, $r$ is, indeed, defined; whilst Lemma
\ref{whenr0S+jstable} ensures that, at $p$, $r(0)S_{+}\cap
r(\infty)S_{-}=\{0\}$. On the other hand, according to
\eqref{eq:jstablityofr(0)=S+}, at each point, the non-$j$-stability
of $r(0)S_{+}$ is characterized by
$$(r(0)\wedge^{2}S_{+})\cap
(\,\overline{r(0)\wedge^{2}S_{+}}\,)^{\perp}=\{0\},$$ or,
equivalently, by $(r(0)\,u_{+},\overline{r(0)\,u_{+}}\,)\neq 0$,
fixing $u_{+}\in\Gamma(\wedge^{2}S_{+})$ non-zero; which shows that
the non-$j$-stability of $(r(0)S_{+})_{x}$ is an open condition on
$x\in M$. And so is the non-reality of
$(\tau(\alpha)L^{\alpha})_{x}$, as remarked previously.
\end{proof}

Set
$$S^{*}_{+}:=r(0)S_{+},\,\,\,S^{*}_{-}:=jS^{*}_{+}=r(\infty)S_{-}$$
and, considering projections
$\pi_{S^{*}_{+}}:\underline{\C}^{4}\rightarrow S^{*}_{+}$ and
$\pi_{S^{*}_{-}}:\underline{\C}^{4}\rightarrow S^{*}_{-}$ with
respect to the decomposition
$$\underline{\C}^{4}=S^{*}_{+}\oplus S^{*}_{-},$$
define two line bundles by
\begin{equation}\label{eq:Lspm}
L_{+}^{*}:=\pi_{S^{*}_{+}}r(\infty)L_{+},\,\,\,\,\,\,L_{-}^{*}:=jL_{+}^{*}=\pi_{S^{*}_{-}}r(0)L_{-}.
\end{equation}
As we know, $L^{\alpha}$ is a $d^{\alpha,q}_{V}$-parallel null line
bundle if and only if $\rho_{V} L^{\alpha}$ is a
$d^{-\alpha,q}_{V}$-parallel null line bundle. Note that the reality
of $\tau(\alpha)L^{\alpha}$ is equivalent to that of
$\tau(-\alpha)\rho_{V} L^{\alpha}$, and that
$$r_{\alpha,L^{\alpha}}=r_{-\alpha,\rho_{V} L^{\alpha}}.$$
We define the \textit{untwisted} \textit{B\"{a}cklund}
\textit{transform} \textit{$L^{*}$} \textit{of} \textit{$L$}
\textit{of} \textit{parameters} \textit{$\alpha,L^{\alpha}$},
\textit{or}, \textit{equally}, \textit{$-\alpha,\rho_{V}
L^{\alpha}$}, by setting
$$\wedge^{2}L^{*}:=L^{*}_{+}\wedge L^{*}_{-}.$$ The real line bundle
$\wedge^{2}L^{*}$ determines a $j$-stable bundle $L^{*}$ of
$2$-planes in $\underline{\C}^{4}$.
\begin{thm}\label{L*CWpowihjugfvbaq9873tfgvwbuyqwrf}
$L^{*}$ is a constrained Willmore surface in $S^{4}$, provided that
it immerses.
\end{thm}

The proof of the theorem will follow some preliminaries, presented
next.

At each point, the complementarity of $S^{*}_{+}$ and $S^{*}_{-}$ in
$\underline{\C}^{4}$ establishes a decomposition
$$
\wedge ^{2}\underline{\C}^{4}=\wedge^{2}S^{*}_{+}\oplus
S^{*}_{+}\wedge S^{*}_{-} \oplus \wedge^{2}S^{*}_{-}.
$$
Define then, for $\lambda\in\C\backslash\{0\}$, another orthogonal
transformation of $\wedge^{2}\underline{\C}^{4}$ by
$$\tau ^{*} (\lambda)
:=I\left\{
\begin{array}{ll} \lambda^{-1} & \mbox{$\mathrm{on}\,\wedge^{2}S^{*}_{+}$}\\ 1 &
\mbox{$\mathrm{on}\,S^{*}_{+}\wedge S^{*}_{-}$}\\
\lambda & \mbox{$\mathrm{on}\,\wedge
^{2}S^{*}_{-}$}\end{array}\right..$$ Set
\begin{equation}\label{eq:VstarvsSstar}
V^{*}:=S^{*}_{+}\wedge S^{*}_{-}
\end{equation}
and let $\rho_{V^{*}}$ denote reflection across $V^{*}$. Observe
that $\tau(-1)=\rho_{V}$ and $\tau^{*}(-1)=\rho_{V^{*}}$, and that,
for $\lambda\in\C\backslash\{0\}$,
\begin{equation}\label{eq:tauand-}
\tau(-\lambda)=\tau(\lambda)\rho_{V},\,\,\,\,\,\,\tau^{*}(-\lambda)=\tau^{*}(\lambda)\rho_{V^{*}}
\end{equation}
and
\begin{equation}\label{eq:taueoverline}
\overline{\tau(\lambda)}=\tau(\overline{\lambda}\,^{-1}),\,\,\,\,\overline{\tau^{*}(\lambda)}=\tau^{*}(\overline{\lambda}\,^{-1}).
\end{equation}
For $\lambda\in\C\backslash\{0\}$, set
$$R(\lambda):=\tau
^{*}(\lambda)^{-1}\,r_{\alpha,L^{\alpha}}(\lambda^{2})\,\tau(\lambda)\in\Gamma(O(\wedge^{2}\underline{\C}^{4})),$$which
will play a crucial role on what follows. Fix
$g\in\Gamma(SO(\underline{\R}^{5,1}))$ such that
$$g\rho_{V}g^{-1}=\rho_{V^{*}}$$ and
set also
$$P(\lambda):=g^{-1}R(\lambda).$$
The holomorphicity of $P(\lambda)$ at $\lambda=0$, equivalent to
that of $R(\lambda)$, is directly ensured by Lemma
\ref{TrioDeHolomorfia}, and, similarly, after an appropriate change
of variable, so is the holomorphicity of $P(\lambda)$ at
$\lambda=\infty$. Set then
$$
(\Lambda^{*})^{1,0}:=P(1)^{-1}P(\infty)\,\Lambda^{1,0}$$ and
$$(\Lambda^{*})^{0,1}:=P(1)^{-1}P(0)\,\Lambda^{0,1}.$$
Equivalently,
$$(\Lambda^{*})^{1,0}=R(\infty)\,(L_{+}\wedge S_{-}),\,\,\,\,\,\,(\Lambda^{*})^{0,1}=R(0)\,(L_{-}\wedge S_{+}),$$as $R(1)=I$.
Observe that $(\Lambda^{*})^{1,0}$ and $ (\Lambda^{*})^{0,1}$ are
complex conjugate of each other: having in consideration
\eqref{eq:taueoverline},
\begin{eqnarray*}
\overline{(\Lambda^{*})^{1,0}}&=&\mathrm{lim}_{\lambda\rightarrow
\infty}\,\tau^{*}(\overline{\lambda})\,\overline{r(\lambda^{2})}\,\tau(\overline{\lambda}\,^{-1})\,\overline{L_{+}\wedge
S_{-}}\\&=&\mathrm{lim}_{\lambda\rightarrow
0}\,\tau^{*}(\lambda^{-1})\,r(0)\,\tau(\lambda)\,(L_{-}\wedge
S_{+})\\&=&(\Lambda^{*})^{0,1}.
\end{eqnarray*}
Fixing a choice of $\sqrt{\lambda}$,
$$\tau (\lambda)
=I\left\{
\begin{array}{ll} \sqrt{\lambda}\,^{-1} & \mbox{$\mathrm{on}\,S_{+}$}\\ \sqrt{\lambda} & \mbox{$\mathrm{on}\,S_{-}$}\end{array}\right.$$
and
$$\tau ^{*}(\lambda)
=I\left\{
\begin{array}{ll} \sqrt{\lambda}\,^{-1} & \mbox{$\mathrm{on}\,S^{*}_{+}$}\\ \sqrt{\lambda} & \mbox{$\mathrm{on}\,S^{*}_{-}$}\end{array}\right.,$$
and, therefore, given $s_{+}\in\Gamma(S_{+})$,
$R(0)s_{+}=\mathrm{lim}_{\lambda\rightarrow
0}\,(\pi_{S^{*}_{+}}\,r(\lambda^{2})\,s_{+}+\lambda^{-1}\pi_{S^{*}_{-}}\,r(\lambda^{2})\,s_{+})$.
On the other hand, as $r(\lambda^{2})$ is holomorphic at
$\lambda=0$, it admits a Taylor expansion
$$
r(\lambda^{2})=r(0)+\lambda^{2}\frac{d}{d\lambda}_{\vert_{\lambda=0}}r(\lambda^{2})+\frac{\lambda^{4}}{2}\frac{d^{2}}{d\lambda^{2}}_{\vert_{\lambda=0}}r(\lambda^{2})+....$$
around $0$, making clear that
\begin{equation}\label{eq:conclusaodeTaylorexparound0}
\mathrm{lim}_{\lambda\rightarrow
0}\,\lambda^{-1}r(\lambda^{2})=\mathrm{lim}_{\lambda\rightarrow
0}\,\lambda^{-1}r(0).
\end{equation}
Hence
$R(0)s_{+}=r(0)s_{+}+\pi_{S^{*}_{-}}(\mathrm{lim}_{\lambda\rightarrow
0}\,\lambda^{-1} \,r(0)\,s_{+})$. Now the holomorphicity of
$R(\lambda)$ at $\lambda=0$ forces $\pi_{S^{*}_{-}}\,r(0)s_{+}$ to
be $0$, establishing $R(0)s_{+}=r(0)s_{+}$. Thus
$R(0)S_{+}=S^{*}_{+}.$ A similar, even slightly simpler computation
establishes $R(0)L_{-}=\pi_{S^{*}_{-}}r(0)L_{-}.$ We conclude that
\begin{equation}\label{eq:Lambda01}
(\Lambda^{*})^{0,1}=\pi_{S^{*}_{-}}r(0)L_{-}\wedge S^{*}_{+},
\end{equation}
or, equivalently, by complex conjugation,$$
(\Lambda^{*})^{1,0}=\pi_{S^{*}_{+}}r(\infty)L_{+}\wedge S^{*}_{-}.$$
Thus
$$(\Lambda^{*})^{1,0}\cap(\Lambda^{*})^{0,1}=\wedge^{2}L^{*}=:\Lambda^{*}.$$
The choice of the notation $\Lambda^{*}$, already attributed to a
twisted B\"{a}cklund transform of $\Lambda$, is not casual, as we
shall verify in the next section.

Now we proceed to the proof of Theorem
\ref{L*CWpowihjugfvbaq9873tfgvwbuyqwrf}.
\begin{proof}
Note that, according to \eqref{eq:tauand-},
\begin{equation}\label{eq:P-lambda}
P(-\lambda)=g^{-1}\tau^{*}(-\lambda)\,r_{\alpha,L^{\alpha}}(\lambda^{2})\,\tau(-\lambda)=
g^{-1}\rho_{V^{*}}R(\lambda)\rho_{V}=\rho_{V}P(\lambda)\rho_{V}.
\end{equation}
The proof will consist of showing that
$$\lambda\mapsto d^{\lambda,q}_{P}:=P(\lambda)\circ
d^{\lambda,q}_{V}\circ P(\lambda)^{-1}$$admits a holomorphic
extension to $\lambda\in\C\backslash\{0\}$ through metric
connections on $\wedge^{2}\underline{\C}^{4}$. It will then follow
from Theorem \ref{eq:thm8.4.2paraja} that $\Lambda^{*}$ is a
constrained Willmore surface, provided that it immerses, with no
need for condition \eqref{eq:condondet} to be verified, as it would
solely intend to ensure that $(\Lambda^{*})^{1,0}$ and
$(\Lambda^{*})^{0,1}$ intersect in a rank $1$ bundle, a fact that is
already known to us.

Since $d_{\alpha^{2},q}^{V}$ is metric,
$$d_{\alpha^{2},q}^{V}\,\Gamma (\,\overline{\tau(\alpha)L^{\alpha}}\,)\subset \Omega^{1}
(\,\overline{\tau(\alpha)L^{\alpha}}\,^{\perp}),$$ as well as, in
view of the parallelness of $\tau(\alpha)L^{\alpha}$ with respect to
$d_{\alpha^{2},q}^{V}$,
$$d_{\alpha^{2},q}^{V}\,\Gamma ((\tau(\alpha)L^{\alpha}\oplus\overline{\tau(\alpha)L^{\alpha}})^{\perp})\subset \Omega^{1}
((\tau(\alpha)L^{\alpha})^{\perp}).$$ Let $\pi_{\tau
L}:\wedge^{2}\underline{\C}^{4}\rightarrow\tau(\alpha)L^{\alpha}$
and  $\pi_{\overline{\tau
L}}:\wedge^{2}\underline{\C}^{4}\rightarrow
\overline{\tau(\alpha)L^{\alpha}}$ be projections with respect to
the decomposition \eqref{eq:taudecomp}. Let $\pi_{\oplus}$ and
$\pi_{\perp}$ be as in Lemma \ref{whenr0S+jstable}. As
$$(\tau(\alpha)L^{\alpha})^{\perp}=\tau(\alpha)L^{\alpha}\oplus(\tau(\alpha)L^{\alpha}\oplus
\overline{\tau(\alpha)L^{\alpha}})^{\perp}$$ and
$$\overline{\tau(\alpha)L^{\alpha}}\,^{\perp}=\overline{\tau(\alpha)L^{\alpha}}\oplus(\tau(\alpha)L^{\alpha}\oplus
\overline{\tau(\alpha)L^{\alpha}})^{\perp},$$ we conclude that
$\pi_{\tau L}\circ d_{\alpha^{2},q}^{V}\circ\pi_{\overline{\tau
L}}=0=\pi_{\overline{\tau L}}\circ d_{\alpha^{2},q}^{V}\circ
\pi_{\perp}$, showing that $d_{\alpha^{2},q}^{V}$ splits as
$$d_{\alpha^{2},q}^{V}=D_{\alpha^{2}}^{q}+\beta_{\alpha^{2}}^{q},$$ for the connection
$$D_{\alpha^{2}}^{q}:=d_{\alpha^{2},q}^{V}\circ\pi_{\tau L}+\pi_{\overline{\tau L}}\circ
d_{\alpha^{2},q}^{V}\circ\pi_{\overline{\tau L}}+\pi_{\perp}\circ
d_{\alpha^{2},q}^{V}\circ\pi_{\perp},$$ on
$\wedge^{2}\underline{\C}^{4}$, and
$$\beta_{\alpha^{2}}^{q}:=\pi_{\perp}\circ
d_{\alpha^{2},q}^{V}\circ\pi_{\overline{\tau L}}+\pi_{\tau L}\circ
d_{\alpha^{2},q}^{V}\circ\pi_{\perp}\in\Omega^{1}(\tau(\alpha)L^{\alpha}\wedge(\tau(\alpha)L^{\alpha}\oplus\overline{\tau(\alpha)L^{\alpha}}\,)
\,^{\perp}).$$ For simplicity, denote $r_{\alpha}(\lambda^{2})\in\C$
(as defined in \eqref{eq:ralphalambdathe number}) by
$\alpha_{\lambda^{2}}$. Clearly, for each $\lambda$,
$$r(\lambda^{2})\circ D_{\alpha^{2}}^{q}\circ r(\lambda^{2})^{-1}=D_{\alpha^{2}}^{q},\,\,\,\,
r(\lambda^{2})\,\beta_{\alpha^{2}}^{q}\,r(\lambda^{2})^{-1}=\alpha_{\lambda^{2}}\,\beta_{\alpha^{2}}^{q}.$$
Now decompose $d_{\lambda^{2},q}^{V}$ as
$$d_{\lambda^{2},q}^{V}=d_{\alpha^{2},q}^{V}+(\lambda^{2}-\alpha^{2})A(\lambda^{2}),$$
for $\lambda\in\C\backslash \{0,\pm\alpha\}$, with $\lambda\mapsto
A(\lambda^{2})\in\Omega^{1}(o(\wedge^{2}\underline{\C}^{4}))$
holomorphic. It follows that
\begin{eqnarray*}
d^{\lambda,q}_{P}&=&g^{-1}\tau^{*}(\lambda)^{-1}r_{\alpha,L^{\alpha}}(\lambda^{2})\circ
d^{V}_{\lambda^{2},q}\circ
r_{\alpha,L^{\alpha}}(\lambda^{2})^{-1}\tau^{*}(\lambda)g\\&=&
g^{-1}\tau^{*}(\lambda)^{-1}D_{\alpha^{2}}^{q}\,\tau^{*}(\lambda)\,g+\alpha_{\lambda^{2}}\,g^{-1}\tau^{*}(\lambda)^{-1}\beta_{\alpha^{2}}^{q}\,\tau^{*}(\lambda)\,g+\Psi(\lambda),
\end{eqnarray*}
for
$\Psi(\lambda):=g^{-1}\tau^{*}(\lambda)^{-1}r(\lambda^{2})(\lambda^{2}-\alpha^{2})A(\lambda^{2})r(\lambda^{2})^{-1}\tau^{*}(\lambda)\,g$
and
$\lambda\in\C\backslash\{0,\pm\alpha,\pm\overline{\alpha}\,^{-1}\}$.
Set
$\Upsilon(\lambda):=(\lambda^{2}-\alpha^{2})r(\lambda^{2})A(\lambda^{2})r(\lambda^{2})^{-1}.$
The skew-symmetry of $A(\lambda^{2})$ establishes
$$A(\lambda^{2})\tau(\alpha)L^{\alpha}\subset(\tau(\alpha)L^{\alpha})\,^{\perp},\,\,\,\,\,A(\lambda^{2})\overline{\tau(\alpha)L^{\alpha}}\subset\overline{\tau(\alpha)L^{\alpha}}\,^{\perp}$$
and, consequently,
$$\pi_{\overline{\tau L}}\,A(\lambda^{2})\,\pi_{\tau L}=0=\pi_{\tau L}\,A(\lambda^{2})\,\pi_{\overline{\tau L}}.$$
On the other hand, it is clear that $$\pi_{\overline{\tau
L}}\,\,r(\lambda^{2})\,A(\lambda^{2})\,r(\lambda^{2})^{-1}\pi_{\tau
L}=a_{\lambda^{2}}^{-1}\,\pi_{\overline{\tau
L}}\,\,A(\lambda^{2})\,a_{\lambda^{2}}^{-1}\,\pi_{\tau
L}=a_{\lambda^{2}}^{-2}\,\pi_{\overline{\tau
L}}\,A(\lambda^{2})\pi_{\tau L}$$and, similarly,
$$\pi_{\tau
L}\,r(\lambda^{2})\,A(\lambda^{2})\,r(\lambda^{2})^{-1}\pi_{\overline{\tau
L}}=a_{\lambda^{2}}^{2}\,\pi_{\tau
L}\,A(\lambda^{2})\pi_{\overline{\tau L}}.$$ Hence
$\pi_{\overline{\tau L}}\,\Upsilon(\lambda)\,\pi_{\tau
L}=0=\pi_{\tau L}\Upsilon(\lambda)\,\pi_{\overline{\tau L}}.$ It
follows that
\begin{eqnarray*}
\Upsilon(\lambda)&=& (\lambda^{2}-\alpha^{2})\,(\pi_{\tau
L}\,A(\lambda^{2})\,\pi_{\tau L}+\pi_{\overline{\tau
L}}\,\,A(\lambda^{2})\,\pi_{\overline{\tau
L}}+\pi_{\perp}A(\lambda^{2})\,\pi_{\perp})\\ & &
\mbox{}+\frac{1-\alpha^{2}}{1-\overline{\alpha}\,^{-2}}\,(\lambda^{2}-\overline{\alpha}\,^{-2})\,(\pi_{\perp}\,A(\lambda^{2})\,\pi_{\tau
L}+\pi_{\overline{\tau L}}\,\,A(\lambda^{2})\,\pi_{\perp})\\ & &
\mbox{}+
 \frac{1-\overline{\alpha}\,^{-2}}{1-\alpha^{2}}\frac{(\lambda^{2}-\alpha^{2})^{2}}{\lambda^{2}-\overline{\alpha}\,^{-2}}\,(\pi_{\perp}\,A(\lambda^{2})\,\pi_{\overline{\tau
L}}+\pi_{\tau L}\,A(\lambda^{2})\,\pi_{\perp}).
\end{eqnarray*}
Hence, by setting
\begin{eqnarray*}
d^{\alpha,q}_{P}&:=&g^{-1}\tau^{*}(\alpha)^{-1}D_{\alpha^{2}}^{q}\,\tau^{*}(\alpha)\,g
\\ & & \mbox{}+g^{-1}\tau^{*}(\alpha)^{-1}\frac{1-\alpha^{2}}{1-\overline{\alpha}\,^{-2}}\,(\alpha^{2}-\overline{\alpha}\,^{-2})\,(\pi_{\perp}\,A(\alpha^{2})\,\pi_{\tau
L}+\pi_{\overline{\tau
L}}\,\,A(\alpha^{2})\,\pi_{\perp})\tau^{*}(\alpha)g
\end{eqnarray*}
and
\begin{eqnarray*}
d^{-\alpha,q}_{P}&:=&g^{-1}\tau^{*}(-\alpha)^{-1}D_{\alpha^{2}}^{q}\,\tau^{*}(-\alpha)\,g
\\ & & \mbox{}+g^{-1}\tau^{*}(-\alpha)^{-1}\frac{1-\alpha^{2}}{1-\overline{\alpha}\,^{-2}}\,(\alpha^{2}-\overline{\alpha}\,^{-2})\,(\pi_{\perp}\,A(\alpha^{2})\,\pi_{\tau
L}+\pi_{\overline{\tau
L}}\,\,A(\alpha^{2})\,\pi_{\perp})\tau^{*}(-\alpha)g,
\end{eqnarray*}
we extend holomorphically $(\lambda\mapsto d^{\lambda,q}_{P})$ to
$\lambda\in\C\backslash\{0,\pm\overline{\alpha}\,^{-1}\}$ through
what, by continuity, are metric connections on
$\wedge^{2}\underline{\C}^{4}$.

The existence of a holomorphic extension to $\C\backslash\{0\}$,
through metric connections on $\wedge^{2}\underline{\C}^{4}$, can be
proved analogously, having in consideration the parallelness of
$\overline{\tau(\alpha)L^{\alpha}}$ with respect to the connection
$$d_{\overline{\alpha}\,^{-2},q}^{V}=\overline{\tau(\alpha)\circ d^{\alpha,q}_{V}\circ \tau(\alpha)^{-1}}.$$

\end{proof}

Suppose $L^{*}$ immerses. Note that
$P(1)^{-1}P(-1)=R(-1)=\rho_{V^{*}}\rho_{V}$, whilst, on the other
hand, according to \eqref{eq:P-lambda},
$P(1)^{-1}P(-1)=P(1)^{-1}\rho_{V}P(1)\rho_{V}$. We conclude that
$\rho_{V^{*}}P(1)^{-1}=P(1)^{-1}\rho_{V}$ and, therefore, as
$\rho_{V}V=V,$ that $\rho_{V^{*}}P(1)^{-1}V=P(1)^{-1}V$.
Equivalently, $$V^{*}=P(1)^{-1}V,$$$V^{*}$ is the complexification
of the central sphere congruence of $\Lambda^{*}$. Denote the mean
curvature sphere of $L^{*}$ by $S^{*}$. In view of
\eqref{eq:VstarvsSstar}, the eigenspace of $S^{*}$ associated to the
eigenvalue $i$ is either $S^{*}_{+}$ or $S_{-}^{*}$. Suppose it is
$S^{*}_{-}$. In that case, $(\Lambda^{*})^{0,1}=(L^{*}\cap
S^{*}_{+})\wedge S^{*}_{-},$ which, in view of the complementarity
of $S_{+}^{*}$ and $S_{-}^{*}$ in $\underline{\C}^{4}$, contradicts
\eqref{eq:Lambda01}. Hence $S^{*}_{+}$ is the eigenspace of $S^{*}$
associated to $i$. The choice of notation is consistent. It follows,
in particular, that the equalities \eqref{eq:Lspm} are not merely
formal either.

\subsection{Twisted vs. untwisted B\"{a}cklund transformation of constrained Willmore surfaces in $4$-space}

Twisted B\"{a}cklund transformation of constrained Willmore surfaces
in $4$-space is closely related to untwisted B\"{a}cklund
transformation. As we verify in this section, when twisted
B\"{a}cklund transformation parameters constitute untwisted
B\"{a}cklund transformation parameters, the corresponding twisted
and untwisted B\"{a}cklund transforms coincide.\footnote{In Appendix
B below, we verify that twisted and untwisted B\"{a}cklund
transformation parameters conditions at a point are not
equivalent.}\newline

Choose a non-zero $\alpha\in\C\backslash S^{1}$ and a
$d^{\alpha,q}_{V}$-parallel null line subbundle $L^{\alpha}$ of
$\wedge^{2}\underline{\C}^{4}$ such that, locally,
$\tau(\alpha)L^{\alpha}$ is never real, $r(0)S_{+}$ and
$r(\infty)S_{-}$ are complementary in $\underline{\C}^{4}$, and
condition \eqref{eq:BTparameterscondition} is satisfied. The
existence of such a choice of $L^{\alpha}$ is established in the
next lemma. First observe that, according to
\eqref{eq:characterizdenull}, given $w:=v_{0}+v_{+}+v_{-}$ null,
with $v_{0}\in\Gamma(S_{+}\wedge S_{-})$,
$v_{+}\in\Gamma(\wedge^{2}S_{+})$ and
$v_{-}\in\Gamma(\wedge^{2}S_{-})$, $\rho_{V} w$ is orthogonal to
$w$, at some point, if and only if, at that point,
$(v_{0},v_{0})=0,$ or, equivalently, as $\wedge^{2}S_{+}\cap
(\wedge^{2}S_{-})^{\perp}=\{0\}$, either $v_{+}=0$ or $v_{-}=0$.

\begin{Lemma}\label{buildlalphataunotreal}
Let $v_{-}$ be a non-zero section of $\wedge^{2}S_{-}$ with
$(v_{-},\overline{v_{-}})\neq 1.$ Let $v_{0}$ be a section of
$S_{+}\wedge S_{-}$ with $(v_{0},v_{0})=(v_{-},\overline{v_{-}})$
and
$(v_{0},\overline{v_{0}})=\frac{1}{4}\,(v_{-},\overline{v_{-}}).$
Define a null section of $\wedge^{2}\underline{\C}^{4}$ by
$l^{\alpha}:=v_{0}-\frac{1}{2}\,\overline{v_{-}}+v_{-}.$ Let
$L^{\alpha}\subset\wedge^{2}\underline{\C}^{4}$ be a
$d^{\alpha,q}_{V}$-parallel null line bundle defined naturally by
$d^{\alpha,q}_{V}$-parallel transport of $l^{\alpha}_{p}$, for some
point $p\in M$. Then there is a (non-empty) open set where
$L^{\alpha}$ is never orthogonal to $\rho_{V} L^{\alpha}$,
$\tilde{L}^{\alpha}$ is never orthogonal to
$\rho_{V}\tilde{L}^{\alpha}$, $\tau(\alpha)L^{\alpha}$ is never real
and $r(0)S_{+}$ and $r(\infty)S_{-}$ are complementary in
$\underline{\C}^{4}$.
\end{Lemma}
\begin{proof}
At the point $p$, $L^{\alpha}$ is spanned by $l^{\alpha}_{p}$. The
fact that, at the point $p$, in particular, $(v_{0},v_{0})$ is not
zero establishes the non-orthogonality of $L^{\alpha}$ and $\rho_{V}
L^{\alpha}$ at this point. The non-reality of
$\tau(\alpha)L^{\alpha}$ at $p$ follows then, according to Lemma
\ref{whentauLreal}, from the fact that $\mid
(v_{0},v_{0})\mid=(v_{-},\overline{v_{-}})\neq 1.$ Let $\pi_{\perp}$
denote the orthogonal projection of $\wedge^{2}\underline{\C}^{4}$
onto
$(\tau(\alpha)L^{\alpha}\oplus\overline{\tau(\alpha)L^{\alpha}})^{\perp}$.
The non-$j$-stability of $r(0)S_{+}$ at the point $p$ is
established, according to Lemma \ref{whenr0S+jstable}, by the fact
that $\langle \pi_{\perp}v_{-}\rangle\neq \langle
\pi_{\perp}\overline{v_{-}}\rangle,$ as we verify next. First of
all, note that
$\tau(\alpha)l^{\alpha}=-\frac{1}{2}\alpha\,^{-1}\overline{v_{-}}+v_{0}+\alpha
v_{-}.$ Set
$X:=(\tau(\alpha)l^{\alpha},\overline{\tau(\alpha)l^{\alpha}}\,)\neq
0.$ Then
$\overline{v_{-}}=X^{-1}(\overline{v_{-}},\overline{\tau(\alpha)l^{\alpha}}\,)\tau(\alpha)l^{\alpha}+X^{-1}(\overline{v_{-}},\tau(\alpha)l^{\alpha})\overline{\tau(\alpha)l^{\alpha}}+\pi_{\perp}\overline{v_{-}},$
or, equivalently,
\begin{eqnarray*}
(v_{-},\overline{v_{-}}\,)^{-1}X\,\overline{v_{-}}&=&\alpha\overline{\alpha}\,^{-1}v_{-}+(\frac{1}{4}\mid\alpha\mid^{-2}+\mid\alpha\mid^{2})\overline{v_{-}}\\
& & \mbox{}-\frac{1}{2}\,\overline{\alpha}\,^{-1}v_{0}+\alpha\,
\overline{v_{0}}+(v_{-},\overline{v_{-}}\,)^{-1}X\pi_{\perp}\overline{v_{-}},
\end{eqnarray*}
and, therefore,
\begin{eqnarray*}
(v_{-},\overline{v_{-}}\,)^{-1}X\,\pi_{\perp}\overline{v_{-}}&=&
-\alpha\overline{\alpha}\,^{-1}v_{-}+\frac{1}{2}\,\overline{\alpha}\,^{-1}v_{0}-\alpha\,
\overline{v_{0}}\\& &
\mbox{}-(\frac{1}{4}\mid\alpha\mid^{-2}+\mid\alpha\mid^{2}-(v_{-},\overline{v_{-}})^{-1}X)\overline{v_{-}}.
\end{eqnarray*}
Hence the linear dependency of $\pi_{\perp}\overline{v_{-}}$ and its
complex conjugate implies, in particular,
\begin{equation}\label{eq:realityofalphasquare,7654VDSH}
\alpha\overline{\alpha}\,^{-1}=\overline{\alpha}\alpha^{-1}
\end{equation}
and
\begin{equation}\label{eq:nbs4ffcaSCF6SY7VD48W3222DNBV 892}
(\frac{1}{2}\,\alpha^{-1}+\alpha)\overline{v_{0}}=(\frac{1}{2}\,\overline{\alpha}\,^{-1}-\overline{\alpha})v_{0}.
\end{equation}
Condition \eqref{eq:realityofalphasquare,7654VDSH} consists of the
reality of $\alpha^{2}$, whereas equation
\eqref{eq:nbs4ffcaSCF6SY7VD48W3222DNBV 892} forces, in particular,
$$v_{0}=\frac{(1-2\overline{\alpha}^{2})(1-2\alpha^{2})}{(1+2\alpha^{2})(1+2\overline{\alpha}^{2})}\,v_{0}.$$
As $v_{0}$ is non-zero, we conclude that, if $\langle
\pi_{\perp}\overline{v_{-}}, \pi_{\perp}v_{-}\rangle$ has rank $1$,
then $\frac{1-2\alpha^{2}}{1+2\alpha^{2}}=\pm 1,$ which contradicts
the fact that $\alpha$ is non-zero.

Now consider projections
$\pi_{\overline{L^{\alpha}}}:\wedge^{2}\underline{\C}^{4}\rightarrow
\overline{L^{\alpha}}$,
$\pi_{\rho_{V}\overline{L^{\alpha}}}:\wedge^{2}\underline{\C}^{4}\rightarrow
\rho_{V}\overline{L^{\alpha}}$ and
$\pi_{\perp'}:\wedge^{2}\underline{\C}^{4}\rightarrow
(\overline{L^{\alpha}}\oplus\rho_{V}\overline{L^{\alpha}})^{\perp}$
with respect to the decomposition
$$\wedge^{2}\underline{\C}^{4}= \overline{L^{\alpha}}\oplus
(\overline{L^{\alpha}}\oplus\rho_{V}\overline{L^{\alpha}})^{\perp}\oplus\rho_{V}\overline{L^{\alpha}},$$
provided by the non-orthogonality of $\overline{L^{\alpha}}$ and
$\rho_{V}\overline{L^{\alpha}}$, consequent to the non-orthogonality
of $L^{\alpha}$ and $\rho_{V} L^{\alpha}$. Set
$$A:=\frac{\alpha-\overline{\alpha}\,^{-1}}{\alpha+\overline{\alpha}\,^{-1}}=\frac{\mid\alpha\mid
^{2}-1}{\mid\alpha\mid ^{2}+1}\in\R.$$ Then
$$
\rho_{V} q(\alpha)l^{\alpha}
=A\,\rho_{V}\pi_{\overline{L^{\alpha}}}\,l^{\alpha}+\rho_{V}\pi_{\perp'}\,l^{\alpha}+A^{-1}\rho_{V}\pi_{\rho_{V}\overline{L^{\alpha}}}\,l^{\alpha},\
$$
and, therefore,
\begin{eqnarray*}
(\rho_{V} q(\alpha)l^{\alpha},q(\alpha)l^{\alpha})&=&
A^{2}(\pi_{\overline{L^{\alpha}}}\,l^{\alpha},\rho_{V}\pi_{\overline{L^{\alpha}}}\,l^{\alpha})
\\ & & \mbox{}+
(\pi_{\perp '}l^{\alpha},\rho_{V}\pi_{\perp '}\,l^{\alpha})\\ & &
\mbox{}+A^{-2}(\pi_{\rho_{V}\overline{L^{\alpha}}}\,l^{\alpha},\rho_{V}\pi_{\rho_{V}\overline{L^{\alpha}}}\,l^{\alpha}).
\end{eqnarray*}
At the point $p$,
$$\pi_{\overline{L^{\alpha}}}\,l^{\alpha}=a(\overline{v_{0}}-\frac{1}{2}\,v_{-}+\overline{v_{-}}),\,\,\,\,\,\,\,\,\pi_{\rho_{V}\overline{L^{\alpha}}}\,l^{\alpha}=b(\overline{v_{0}}+\frac{1}{2}\,v_{-}-\overline{v_{-}}),$$
for some $a,b\in\C$, and, therefore,
$$\pi_{\perp '}l^{\alpha}=v_{0}-(a+b)\overline{v_{0}}-(\frac{1}{2}+a-b)\overline{v_{-}}+(1+\frac{1}{2}(a-b))v_{-}.$$
The orthogonality relations $(\pi_{\perp
'}l^{\alpha},\overline{l^{\alpha}})=0=(\pi_{
\perp^{\alpha}}l^{\alpha},\rho_{V}\overline{l^{\alpha}})$ establish
then $a=\frac{-1}{2}$ and $b=\frac{3}{4}.$ It follows that, at the
point $p$,
$$(v_{-},\overline{v_{-}})^{-1}(\rho_{V} q(\alpha)l^{\alpha},q(\alpha)l^{\alpha})=\frac{1}{2}A^{2}+\frac{3}{8}+\frac{9}{8}A^{-2}.$$
Ultimately, the fact that $A$ is real, and, therefore, $A^{2}$ and
$A^{-2}$ are positive, shows that, at $p$,
$$(\rho_{V} q(\alpha)l^{\alpha},q(\alpha)l^{\alpha})\neq
0.$$

The proof is complete by observing that the non-orthogonality of
$\rho_{V} L^{\alpha}$ and $L^{\alpha}$, characterized by $(\rho_{V}
l,l)\neq 0,$ fixing $l\in\Gamma(L^{\alpha})$ non-zero, is an open
condition on the points in $M$. And so is, similarly, the
non-orthogonality of $\rho_{V} \tilde{L}^{\alpha}$ and
$\tilde{L}^{\alpha}$. And so are the non-reality of
$\tau(\alpha)L^{\alpha}$ and the non-$j$-stability of $r(0)S_{+}$,
as observed previously.
\end{proof}

For such a choice of parameters, we are in a position to refer to
both $R(\lambda)$, as defined in the previous section, and
$(pq)^{-1}(1)pq(\lambda)$.
\begin{prop}
\begin{equation}\label{eq:taupqs}
R(\lambda)=(pq)^{-1}(1)(pq)(\lambda),
\end{equation}
for all $\lambda$.
\end{prop}

The proof of the proposition is based on Lemma
\ref{TrioDeHolomorfia}.

\begin{proof}
Equation \eqref{eq:taupqs}  holds for $\lambda=1$. The proof will
consist of showing that $\xi:=\tau ^{*}(\lambda)^{-1}\,r(\lambda
^{2})\,\tau(\lambda)q^{-1}(\lambda)p^{-1}(\lambda)(pq)(1)$ is
holomorphic in $\mathbb{P}^{1}$ (and, therefore, constant). For
that, first note that, at most, $\xi$ has simple poles at
$0,\infty,\pm\alpha$ and $\pm\overline{\alpha}^{-1}$. The
holomorphicity of $\xi$ at $0$, equivalent to the holomorphicity of
$\tau ^{*}(\lambda)^{-1}\,r(\lambda ^{2})\,\tau(\lambda)$ at
$\lambda=0$, is already known to us, from the previous section, as
well as the holomorphicity of $\xi$ at $\lambda=\infty$. Observe, on
the other hand, that we can decompose $r(\lambda^{2})$ as
$$r(\lambda^{2})=r_{1}(\lambda)r_{2}(\lambda)=r_{2}(\lambda)r_{1}(\lambda),$$
for
$$r_{1} (\lambda)
:=I\left\{
\begin{array}{ll} c\,\frac{\lambda-\alpha}{\lambda-\overline{\alpha}^{-1}}& \mbox{$\mathrm{on}\,\tau (\alpha)L^{\alpha}$}\\ 1 &
\mbox{$\mathrm{on}\,(\tau (\alpha)L^{\alpha}+\overline{\tau (\alpha)L^{\alpha}})^{\perp}$}\\
c^{-1} \frac{\lambda-\overline{\alpha}^{-1}}{\lambda-\alpha}&
\mbox{$\mathrm{on}\,\overline{\tau
(\alpha)L^{\alpha}}$}\end{array}\right.$$ and
$$r_{2} (\lambda)
:=I\left\{
\begin{array}{ll} \frac{\lambda+\alpha}{\lambda+\overline{\alpha}^{-1}}& \mbox{$\mathrm{on}\,\tau (\alpha)L^{\alpha}$}\\ 1 &
\mbox{$\mathrm{on}\,(\tau (\alpha)L^{\alpha}+\overline{\tau (\alpha)L^{\alpha}})^{\perp}$}\\
 \frac{\lambda+\overline{\alpha}^{-1}}{\lambda+\alpha}&
\mbox{$\mathrm{on}\,\overline{\tau
(\alpha)L^{\alpha}}$}\end{array}\right.,$$ with
$c:=\frac{1-\overline{\alpha}^{-2}}{1-\alpha^{2}}$. After
appropriate changes of variable and having in consideration
\eqref{eq:tauand-} and \eqref{eq:taueoverline}, we conclude the
holomorphicity of $\xi$ at any of the other candidates for poles.
\end{proof}

Thus:
\begin{thm}
When both are defined, twisted B\"{a}cklund transformation of
parameters $\alpha,L^{\alpha}$ coincides with untwisted B\"{a}cklund
transformation of parameters $\alpha,L^{\alpha}$.
\end{thm}

Proposition \ref{D-BT+BTforalpha2real} below provides a
characterization, for the particular case $\alpha^{2}$ is real, of
untwisted B\"{a}cklund transformations of parameters
$\alpha,L^{\alpha}$ defining twisted B\"{a}cklund transformations of
the same parameters.

\subsection{Darboux transformation of constrained Willmore surfaces in
$4$-space}\label{DTsCWS4}

Characterized by the equation $d*(Q+q)=0$, for some
$q\in\Omega^{1}(\mathrm{End}_{j}(\underline{\C}^{4}/L,L))$ in
certain conditions, a constrained Willmore surface $L$ in $S^{4}$
ensures the existence of
$G\in\Gamma(\mathrm{End}_{j}(\underline{\C}^{4}))$ with
$dG=2*(Q+q)$, as well as the integrability of the Riccati equation
$dT=\rho T(dG)T-dF+4\rho qT$, for each $\rho\in\R\backslash\{0\}$,
fixing such a $G$ and setting $F:=G-S$. For a local solution
$T\in\Gamma(Gl_{j}(\underline{\C}^{4}))$ of the $\rho$-Ricatti
equation, we define the \textit{constrained Willmore Darboux
transform} of $L$ of parameters $\rho,T$ by setting
$\hat{L}:=T^{-1}L$, and extend, in this way, the Darboux
transformation of Willmore surfaces in $S^{4}$ presented in
\cite{quaternionsbook} to a transformation of constrained Willmore
surfaces in $4$-space.\newline

Consider $G\in\Gamma(\mathrm{End}_{j}(\underline{\C}^{4}))$ with
$$dG=2*\,(Q+q)$$ (cf. Proposition \ref{CWinS4charact}) and set $$F:=G-S.$$
For $\rho\in\R\backslash\{0\}$, consider the $\rho$-Riccati equation
\begin{equation}\label{eq:Ricattieq}
dT=\rho T(dG)T-dF+4\rho qT.
\end{equation}
Because $L$ is $q$-constrained Willmore, we have
$$dq=[\mathcal{N}_{S}\wedge q]=\mathcal{N}_{S}\wedge q + q\wedge\mathcal{N}_{S},$$so that the integrability condition for equation \eqref{eq:Ricattieq},
\begin{eqnarray*}
0&=&d(\rho T(dG)T-dF+4\rho qT)\\&=&\rho \,dT\wedge (dG)T+\rho
T(d^{2}G)T-\rho\,T(dG)\wedge dT-d^{2}F+4\rho\, dq\, T-4\rho\,
q\wedge dT\\&=&\rho (\rho T(dG)T-dF+4\rho\, qT)\wedge
(dG)T-\rho\,T(dG)\wedge (\rho T(dG)T-dF+4\rho\, qT)\\ & & \mbox{}+
4\rho\, dq\, T-4\rho \,q\wedge (\rho T(dG)T-dF+4\rho\, qT)\\&=&-\rho
dF\wedge (dG)T+4\rho^{2}qT\wedge (dG)T+\rho T(dG)\wedge
dF-4\rho^{2}T(dG)\wedge qT\\ & &
\mbox{}+4\rho\,[\mathcal{N}_{S}\wedge q]T-4\rho^{2}q\wedge
T(dG)T+4\rho\,q\wedge dF\\&=&-\rho dF\wedge (dG)T+\rho T(dG)\wedge
dF-4\rho^{2}T(dG)\wedge qT+4\rho\,[\mathcal{N}_{S}\wedge
q]T+4\rho\,q\wedge dF\
\end{eqnarray*}
is, equivalently, characterized by
\begin{eqnarray*}
0&=&-4\rho\,(*A\wedge *\,Q+*A\wedge *\,q+*\,q\wedge *\,Q)T\\ & &
\mbox{}+
 4\rho\,
T(*\,Q\wedge *A+*\,Q\wedge *\,q+*\,q\wedge *A)\\ & &
\mbox{}-8\rho^{2}T*(Q+q)\wedge qT+8\rho\,q\wedge *(A+q)\\ & &
\mbox{}+ 4\rho\,(A\wedge q+Q\wedge q+q\wedge A+q\wedge Q)T.
\end{eqnarray*}
Obviously, given $\omega$ and $\gamma$ $1$-forms with values in a
same bundle over $M$, $*\,\omega\wedge \gamma=-\omega\wedge
*\,\gamma$ and, in particular, $ *\,\omega\wedge*\,
\gamma=\omega\wedge \gamma.$ Hence equation \eqref{eq:Ricattieq} is
integrable if and only if \begin{eqnarray*} 0&=&-4\rho\,*A\wedge
*\,QT+ 4\rho\, T(*\,Q\wedge *A+*\,Q\wedge *\,q+*\,q\wedge *A)\\ & &
\mbox{}-8\rho^{2}T*(Q+q)\wedge qT+8\rho\,q\wedge *(A+q)+
4\rho\,(Q\wedge q+q\wedge A)T.
\end{eqnarray*}
We introduce now the concept of left and right-$K$ and
$\overline{K}$ type, which will prove very efficient in showing the
vanishing of each of the terms above. Given
$\xi\in\Omega^{1}(\mathrm{End}(\C^{4}))$, we say that $\xi$ is of
\textit{left}-$K$ (respectively, \textit{right}-$K$) \textit{type}
if $*\,\xi=S\xi$ (respectively, $*\,\xi=\xi S$), referring to
\textit{left}-$\overline{K}$ and \textit{right}-$\overline{K}$ type
in the case $*\,\xi=-S\xi $ or, respectively, $*\,\xi=-\xi S$. For
example, $q$ is of both left-$K$ and right-$K$ type, whilst $A$ is
of left-$K$ type and, therefore, given that it anti-commutes with
$S$, of right-$\overline{K}$ type, as well; whereas $Q$ is of both
left-$\overline{K}$ and right-$K$ type. It is obvious that the Hodge
$*\,$-$\,$operator preserves types. Observe also that, given
$\xi_{1}$ of right-K (respectively, right-$\overline{K}$) type and
$\xi_{2}$ of left-K (respectively, left-$\overline{K}$) type,
$\xi_{1}\wedge\xi_{2}=*\,\xi_{1}\wedge*\,\xi_{2}=\xi_{1}S\wedge
S\xi_{2}=-\xi_{1}\wedge\xi_{2},$ and, therefore,
$\xi_{1}\wedge\xi_{2}=0.$ We conclude in this way the integrability
of equation \eqref{eq:Ricattieq}, ensuring the existence, for each
$\rho\in\R\backslash\{0\}$, of a solution
$T\in\Gamma(\mathrm{End}_{j}(\underline{\C}^{4}))$. Observe that the
Riccati equation \eqref{eq:Ricattieq} has a conserved quantity,
namely, given a solution $T$, if $(T-S)^{2}(p_{0})=(\rho^{-1}-1)I,$
at some $p_{0}\in M$, then
\begin{equation}\label{eq:conservedquantity}
(T-S)^{2}=(\rho^{-1}-1)I
\end{equation}
everywhere. In fact, setting
$$X_{0}:=(T-S)^{2}-\rho^{-1}I+I=T^{2}-TS-ST-\rho^{-1}I,$$we have
\begin{eqnarray*}
d X_{0}&=&(dT-dS)(T-S)+(T-S)(dT-dS)\\&=&(\rho T(dG)T-dF+4\rho
qT-dS)T-(\rho T(dG)T-dF+4\rho qT-dS)S\\ & & \mbox{}+T(\rho
T(dG)T-dF+4\rho qT-dS)-S(\rho T(dG)T-dF+4\rho qT-dS)\
\end{eqnarray*}
or, equivalently, as $dS=dG-dF$,
\begin{eqnarray*}
d X_{0}&=&(\rho T(dG)T-dG+4\rho qT)T-(\rho T(dG)T-dG+4\rho qT)S\\
& & \mbox{}+T(\rho T(dG)T-dG+4\rho qT)-S(\rho T(dG)T-dG+4\rho
qT)\\&=&\rho\,T(dG)(T^{2}-TS-\rho^{-1}I)+(T^{2}-ST-\rho^{-1}I)\rho
(dG)T\\ & & \mbox{}+4\rho\,TqT+4\rho\,q(T^{2}-TS-ST-\rho^{-1}I),\
\end{eqnarray*}
having in consideration that $S*Q=S(-SQ)=SQS=-(*\,Q)S$ and that,
according to \eqref{eq:Sq=qS=*q},
\begin{equation}\label{eq:scommuteswith*qvs-q}
S*\,q=-q=(*\,q)\,S.
\end{equation}
These relations make clear, on the other hand, that
$$\rho\,T(dG)ST+\rho\,TS(dG)T=\rho\,T(-4q)T.$$ We conclude that
$X_{0}$ solves the first order linear equation
$$dX=\rho\,T(dG)X+X\rho
(dG)T+4\rho\,qX,$$ on $X$, and, therefore, that, if $X_{0}(p_{0})=0$
at some point $p_{0}\in M$, then $X_{0}=0$. It follows that, by
imposing, as initial condition,
$$T(p_{0})=S(p_{0})+I
\left\{\begin{array}{ll} \sqrt{\rho^{-1}-1} &
\mbox{$\mathrm{on}\,\,W(p_{0})$}\\ \overline{\sqrt{\rho^{-1}-1}} &
\mbox{$\mathrm{on}\,\,jW(p_{0})$}\end{array}\right.,$$ for some
$p_{0}\in M$, some choice of $ \sqrt{\rho^{-1}-1}$, and some
subspace $W(p_{0})$ of $\C^{4}$, chosen as $\C^{4}$ itself, in the
case $\rho\leq 1$, and as a non-j-stable $2$-plane, otherwise; we
define a local solution
$T\in\Gamma(\mathrm{Gl}_{j}(\underline{\C}^{4}))$ of equation
\eqref{eq:Ricattieq} with a conserved quantity satisfying equation
\eqref{eq:conservedquantity}. For such a solution $T$ of the
$\rho$-Riccati equation \eqref{eq:Ricattieq}, we define the
\textit{constrained} \textit{Willmore} \textit{Darboux}
\textit{transform} \textit{of} $L$ \textit{of} \textit{parameters}
$\rho,T$ by setting
$$\hat{L}:=T^{-1}L,$$
for $T^{-1}$ the section of $\mathrm{End}_{j}(\underline{\C}^{4})$
given by $T^{-1}(p):=T(p)^{-1}$, for all $p$. Observe that
constrained Willmore Darboux transformation of parameters $\rho,T$
with $\rho \leq 1$ is trivial. In fact, if $\rho ^{-1}-1\geq0$, then
$T=S+\sqrt{\rho ^{-1}-1}\,I$, for one of the square roots of $\rho
^{-1}-1$, and, therefore, $\hat{L}=L$, by equation \eqref{eq:SL=L}.
Of course, since $L$ is a $j$-stable $2$-plane, so is $\hat{L}$. As
$$dT^{-1}=-T^{-1}(dT)T^{-1}=-\rho dG+T^{-1}(dF)T^{-1}-4\rho
T^{-1}q$$ and both $Q$ and $q$ vanish on $L$, whilst
$$\mathrm{Im}\,dF\subset L,$$we have
$$(dT^{-1})L\subset \hat{L}$$and, therefore, given $l\in\Gamma(L)$,
$$d(T^{-1}l)+\hat{L}=(dT^{-1})l+T^{-1}dl+\hat{L}=T^{-1}dl+\hat{L},$$showing
that the derivative $\hat{\delta}$ of $\hat{L}$ relates to the one
of $L$ by $$\hat{\delta}=T^{-1}\delta T_{\vert_{\hat{L}}}.$$In
particular, $\hat{L}$ is immersed. Equation
\eqref{eq:conservedquantity} establishes, in particular,
$$0=(T-S)^{2}S-S(T-S)^{2}=T^{2}S-ST^{2}$$and, therefore,
$$T^{2}S=ST^{2}.$$Thus
$$T^{-1}ST=T^{-1}ST^{2}T^{-1}=TST^{-1}.$$Set then
$$\hat{S}:=TST^{-1}=T^{-1}ST,$$which, as we verify next, consists of the mean curvature sphere of $\hat{L}$. Obviously,
$\hat{S}\hat{L}=\hat{L}$ and $*\,\hat{\delta}=\hat{S}\circ
\hat{\delta}$. According to \eqref{eq:conservedquantity}, on the
other hand,
\begin{equation}\label{eq:Tsquare}
T^{2}=TS+ST+\rho^{-1}I,
\end{equation}
so that
$$\hat{S}=T^{-1}(T^{2}-TS-\rho^{-1})=T-S-\rho^{-1}
T^{-1}=F+T-(G+\rho^{-1} T^{-1})$$ and then $$d\hat{S}=\rho
(T(dG)T+4qT)-(\rho^{-1}T^{-1}(dF)T^{-1}-4T^{-1}q).$$ As $QL=0=qL$
and $\mathrm{Im}\,q,\, \mathrm{Im}\,A\subset L,$ it is now clear
that $(d\hat{S})\hat{L}\subset\hat{L},$ as well as, from $S*\,Q=Q$
and $S*A=-A$, that
\begin{eqnarray*}
\hat{S}d\hat{S}-*d\hat{S}&=&\rho TS(dG)T+4\rho
TST^{-1}qT-\rho^{-1}T^{-1}S(dF)T^{-1}+4T^{-1}Sq\\ & & \mbox{}-\rho
T*(dG)T-4\rho*\,qT+\rho^{-1}T^{-1}(*dF)T^{-1}-4T^{-1}*q\\&=&
4\rho(TQT+TST^{-1}qT-*\,qT).
\end{eqnarray*}
Thus $(\hat{S}d\hat{S}-*d\hat{S})\hat{L}=0.$ This proves that
$\hat{S}$ is the mean curvature sphere of $\hat{L}$, as well as that
the Hopf fields of $\hat{L}$ relate to the Hopf fields of $L$ by
$$\hat{Q}=\rho(TQT+TST^{-1}qT-*qT)$$ and, in view of the fact that $\hat{A}=\frac{1}{2}*d\hat{S}+\hat{Q}$,
$$\hat{A}=\rho
TST^{-1}qT-\rho*\,qT-\rho TqT+2\rho
*\,qT+\rho^{-1}T^{-1}AT^{-1}+\rho^{-1}T^{-1}qT^{-1}+2T^{-1}*\,q.$$
Aiming for a simpler relation between $\hat{A}$ and $A$, observe
that, since $d\hat{S}$ is $\hat{S}$-anti-commuting,
$$\hat{S}d\hat{S}=\hat{S}(d\circ \hat{S}-\hat{S}\circ
d)=-(d\hat{S})\hat{S},$$ $d\hat{S}$ reduces to the difference of the
$\hat{S}$-anti-commuting parts of $\rho (T(dG)T+4qT)$ and
$\rho^{-1}T^{-1}(dF)T^{-1}-4T^{-1}q$. The $\hat{S}$-anti-commuting
part of $T(dG)T$ is
$$\frac{1}{2}(T(dG)T+\hat{S}T(dG)T\hat{S})=
T(*\,Q+S*\,QS+*\,q+S*\,qS)T=2 T*\,QT,$$and, similarly, that of
$T^{-1}(dF)T^{-1}$ is $2 T^{-1}*AT^{-1};$ while the
$\hat{S}$-anti-commuting parts of $ qT$ and $T^{-1}q$ are
$\frac{1}{2}(qT+TST^{-1}qST)$ and
$\frac{1}{2}(T^{-1}q+T^{-1}SqT^{-1}ST),$ respectively. It follows
that
$$d\hat{S}=2(*\,\hat{Q}-(\rho^{-1}T^{-1}*AT^{-1}-T^{-1}q-T^{-1}*\,qT^{-1}ST))$$
and, ultimately, that
$$\hat{A}=\rho^{-1}T^{-1}AT^{-1}-T^{-1}*\,q+T^{-1}qT^{-1}ST.$$

Set
$$\hat{q}:=T^{-1}qT\in\Omega^{1}(\mathrm{End}(\underline{\C}^{4}/\hat{L},\hat{L})).$$

\begin{thm}\label{DTofCW}
$\hat{L}$ is a $\hat{q}$-constrained Willmore surface in $S^{4}$.
\end{thm}

\begin{proof}
Obviously, $\hat{S}\hat{q}=*\,\hat{q}=\hat{q}\hat{S}$. As we have
previously observed, the $\hat{S}$-anti-commuting part of
$d(F+T)=\rho (T(dG)T+4qT)$ is $2*\,\hat{Q}$. On the other hand, as
$T^{2}$ commutes with $S$, so does $T^{-2}$,
$$ST^{-2}=T^{-2}T^{2}ST^{-2}=T^{-2}S,$$which, together with
equation \eqref{eq:Tsquare}, shows that
$$I=(TS+ST+\rho^{-1}I)T^{-2}=T^{-1}S+ST^{-1}+\rho^{-1}T^{-2}.$$Thus
\begin{eqnarray*}
\hat{S}(T(dG)T+4qT)\hat{S}&=&TS(dG)ST+4T(I-T^{-1}S-\rho^{-1}T^{-2})qST\\&=&TS(dG)ST+4(TqST-SqST-\rho^{-1}T^{-1}qST).
\end{eqnarray*}
It follows that the $\hat{S}$-commuting part of $d(F+T)$ is
\begin{eqnarray*}
\frac{1}{2}\,(d(F+T)-\hat{S}d(F+T)\hat{S})&=&\rho
T(*\,Q-S*\,QS+*\,q-S*\,qS)T\\
& & \mbox{}+2\rho qT-2\rho T*\,qT+2\rho
SqST+2T^{-1}qST\\&=&2\,T^{-1}*\,q T\\&=&2*\,\hat{q}.
\end{eqnarray*}
Hence $$d(F+T)=2*(\hat{Q}+\hat{q}),$$and, therefore,
$$d*(\hat{Q}+\hat{q})=0,$$completing the proof.
\end{proof}
We complete this section by noting that
$$\hat{S}_{+}=T^{-1}S_{+}=TS_{+},\,\,\,\,\,\hat{S}_{-}=T^{-1}S_{-}=TS_{-}$$ and, consequently, $$\hat{L}_{+}=T^{-1}L_{+},\,\,\,\,\,\hat{L}_{-}=T^{-1}L_{-}.$$

\subsection{B\"{a}cklund transformation vs. Darboux
transformation of constrained Willmore surfaces in
$4$-space}\label{BTvsDT}

Constrained Willmore Darboux transformation of parameters $\rho,T$
with $\rho \leq 1$ is trivial. We establish a correspondence between
constrained Willmore Darboux transformation parameters $\rho,T$ with
$\rho
> 1$ and pairs $\alpha,L^{\alpha};\,-\alpha,\rho_{V} L^{\alpha}$ of
untwisted B\"{a}cklund transformation parameters with $\alpha^{2}$
real, and show that the corresponding transformations coincide.
Non-trivial Darboux transformation of constrained Willmore surfaces
in $4$-space is, in this way, established as a particular case of
constrained Wilmore B\"{a}cklund transformation.
\newline

Suppose $\rho
>1$ and $T$ be constrained Willmore Darboux transformation parameters to $L$. Fix a choice of $\sqrt{\rho^{-1}-1}$.
Then
$$T-S=I\left\{
\begin{array}{ll} \sqrt{\rho^{-1}-1} &
\mbox{$\mathrm{on}\,W$}\\-\sqrt{\rho^{-1}-1} &
\mbox{$\mathrm{on}\,jW$}\end{array}\right.,$$for some non-$j$-stable
bundle $W$ of $2$-planes in $\C^{4}$. Set
\begin{equation}\label{eq:alphasquarerho}
\alpha^{2}:=2i\rho\sqrt{\rho^{-1}-1}-2\rho+1.
\end{equation}

\begin{Lemma}\label{Wd...parallel}
The bundle $W$ is $d^{S}_{\alpha^{2},q}$-parallel.
\end{Lemma}
\begin{proof}
For simplicity, set $\mu:=\sqrt{\rho^{-1}-1}$ and $X:=T-S$. Note
that $$\alpha^{-2}=-2i\rho\mu-2\rho+1.$$ Hence, according to Lemma
\ref{d^S_lambda^2orm},
\begin{eqnarray*}
d^{S}_{\alpha^{2},q}&=&d+2\rho\mu*(Q+q)-2\rho\,(Q+q)\\&=&d+\rho\mu\,dG+\rho\,*dG,\
\end{eqnarray*}
as well as
$$d^{S}_{\alpha^{-2},q}=d-\rho\mu\,dG+\rho\,*dG.$$
On the one hand, straightforward computation establishes
$$
(\mu I-X)\circ d^{S}_{\alpha^{2},q}\circ(X+\mu I)+(\mu I+X)\circ
d^{S}_{\alpha^{-2},q}\circ (X-\mu I)=$$ $$=2\mu(
dX+\rho(*\,dG)X+\rho\,\mu^{2}dG-\rho X(dG)X-\rho X*\,dG),$$ or,
equivalently,
\begin{equation}\label{eq:8udcvbns5wedfghjx9owq}
(\mu I-X)\circ\, d^{S}_{\alpha^{2},q}\circ(X+\mu I)+(\mu I+X)\circ\,
d^{S}_{\alpha^{-2},q}\circ (X-\mu I)=2\mu(dX-\Phi),
\end{equation}
for $$\Phi:=\rho\,X(dG)X+\rho[X,*dG]+(\rho-1)dG.$$ On the other
hand, in view of \eqref{eq:scommuteswith*qvs-q}, together with
$S*Q=Q=-(*\,Q)S,$ we have
\begin{eqnarray*}
\Phi&=&\rho\,T(dG)T-\rho\,T(dG)S-\rho\,S(dG)T+\rho\,S(dG)S\\
& &
\mbox{}+\rho\,T*dG-\rho\,S*dG-\rho\,(*\,dG)T+\rho\,(*\,dG)S+\rho\,dG-dG\\&=&
\rho\,T(dG)T+2\rho\,TQ+2\rho\,Tq-2\rho\,QT+2\rho\,qT-2\rho\,SQ-2\rho\,Sq-2\rho\,TQ\\
& &
\mbox{}-2\rho\,Tq+2\rho\,SQ+2\rho\,Sq+2\rho\,QT+2\rho\,qT-2\rho\,QS-2\rho*\,q+2\rho\,QS\\
& & \mbox{}+2\rho*\,q-2\,*\,Q-2*\,q\
\end{eqnarray*}
and, therefore,
\begin{equation}\label{eq:njhy73r890rhjsk}
\Phi=\rho\,T(dG)T-dG+4\rho\,qT,
\end{equation}
or, equivalently, $dX=\Phi.$ Thus
$$
(\mu I-X)\circ d^{S}_{\alpha^{2},q}\circ(X+\mu I)+(\mu I+X)\circ
d^{S}_{\alpha^{-2},q}\circ (X-\mu I)=0.$$Obviously,
$$X+\mu I=2\mu\pi_{W},\,\,\,\,X-\mu I=-2\mu\pi_{jW},$$ for
$\pi_{W}:\underline{\C}^{4}\rightarrow W$ and
$\pi_{jW}:\underline{\C}^{4}\rightarrow jW$ projections with respect
to the decomposition $\underline{\C}^{4}=W\oplus jW.$ It follows
that, given $\sigma\in\Gamma(W)$, $(\mu I-X)\circ
d^{S}_{\alpha^{2},q}\sigma=0,$ or, equivalently,
$$Xd^{S}_{\alpha^{2},q}\sigma=\mu\, d^{S}_{\alpha^{2},q}\sigma,$$completing the proof.
\end{proof}

For either choice of $\alpha=\sqrt{\alpha^{2}}$, define a null line
subbundle of $\wedge^{2}\underline{\C}^{4}$ by
\begin{equation}\label{eq:alphaM}
L^{\alpha}:=\tau(\alpha)^{-1}\wedge^{2}W.
\end{equation}
Note that $L^{-\alpha}=\rho_{V} L^{\alpha}$. The
$d^{S}_{\alpha^{2},q}$-parallelness of $W$ is equivalent to the
$d^{\alpha,q}_{V}$-parallelness of $L^{\alpha}$. The complementarity
of $W$ and $jW$ in $\underline{\C}^{4}$,
$\wedge^{2}jW\cap(\wedge^{2}W)^{\perp}=\{0\}$, is equivalent to the
non-reality of $\tau(\alpha)L^{\alpha}$. Observe that
$r(0):=r_{\sqrt{\alpha^{2}},\tau(\sqrt{\alpha^{2}}\,)^{-1}\wedge^{2}W}(0)$
and $T-S$ share eigenspaces:
$$r(0)
=I\left\{
\begin{array}{ll} \sqrt{r_{\alpha}(0)} & \mbox{$\mathrm{on}\,W$}\\
\sqrt{r_{\alpha}(0)}\, ^{-1} &
\mbox{$\mathrm{on}\,jW$}\end{array}\right..$$ Fixing a choice of
$\sqrt{-\rho}$ according to the choices of $\sqrt{\rho^{-1}-1}$ and
$\sqrt{r_{\alpha}(0)}$, we get
$\sqrt{r_{\alpha}(0)}=\sqrt{-\rho}\,(\sqrt{\rho^{-1}-1}+i),$ and,
consequently,
$\sqrt{r_{\alpha}(0)}\,^{-1}=\sqrt{-\rho}\,(-\sqrt{\rho^{-1}-1}+i).$
We conclude that
\begin{equation}\label{eq:r0vsT(veryclose)}
r(0)=\sqrt{-\rho}\,(T-S+i).
\end{equation}
Hence
$$r(0)S_{+}=\hat{S}_{+},$$establishing, in particular, the
non-$j$-stability of $r(0)S_{+}$.

Note that, since $\sqrt{\rho^{-1}-1}$ and $-\sqrt{\rho^{-1}-1}$ are
complex conjugate of each other, $\alpha^{2}$ is real. Thus
$\alpha^{2}$ is unit if and only $\alpha^{2}=\pm 1$, which, in view
of $\rho\neq 1$, is impossible. It is immediate to verify that
$\alpha^{2}$ is non-zero.

Conversely, given a non-zero $\alpha\in\C\backslash S^{1}$, with
$\alpha^{2}$ real, and $L^{\alpha}$ a $d^{\alpha,q}_{V}$-parallel
null line subbundle of $\wedge^{2}\underline{\C}^{4}$,  with
$\tau(\alpha)L^{\alpha}$ non-real, equation
\eqref{eq:alphasquarerho} determines
$$\rho=\frac{2\alpha^{2}-1-\alpha^{4}}{4\alpha^{2}}>1,$$as well as a choice of
$\sqrt{\rho^{-1}-1}$, whereas equation \eqref{eq:alphaM} determines
a non-$j$-stable $d^{S}_{\alpha^{2},q}$-parallel bundle $W$ of
$2$-planes in $\C^{4}$. Obviously, the pair ($-\alpha,\rho_{V}
L^{\alpha})$ determines the same pair $(\rho,W)$.  Set
$$T :=I\left\{
\begin{array}{ll} \sqrt{\rho^{-1}-1} & \mbox{$\mathrm{on}\,W$}\\
-\sqrt{\rho^{-1}-1} &
\mbox{$\mathrm{on}\,jW$}\end{array}\right.+S.$$Observe that
$$\overline{\tau(\alpha)\circ
d^{\alpha,q}_{V}\circ\tau(\alpha)^{-1}
\wedge^{2}W}=\overline{\tau(\alpha)\circ
d^{\alpha,q}_{V}\circ\tau(\alpha)^{-1}} \,\wedge^{2}jW,$$the
$d^{S}_{\alpha^{2},q}$-parallelness of $W$ is equivalent to the
parallelness of $jW$ with respect to the connection
$$\overline{\tau(\alpha)\circ
d^{\alpha,q}_{S}\circ\tau(\alpha)^{-1}}=\tau(\overline{\alpha}\,^{-1})\circ
d^{\overline{\alpha}\,^{-1},q}_{S}\circ\tau(\overline{\alpha})=d^{S}_{\overline{\alpha}\,^{-2},q}.$$
Given that $\alpha^{2}$ is real, we conclude that $jW$ is
$d^{S}_{\alpha^{-2},q}$-parallel. It follows, in particular, that,
for $\mu:=\sqrt{\rho^{-1}-1}$ and $X:=T-S$,
$$d^{S}_{\alpha^{2},q}\circ (X+\mu I)\Gamma(\underline{\C}^{4})\subset\Omega^{1} (W),\,\,\,\,d^{S}_{\alpha^{-2},q}\circ (X-\mu I)\Gamma(\underline{\C}^{4})\subset\Omega^{1} (jW),$$
or, equivalently, $$(\mu I-X)\circ d^{S}_{\alpha^{2},q}\circ(X+\mu
I)=0=(\mu I+X)\circ d^{S}_{\alpha^{-2},q}\circ (X-\mu I).$$ From
equations \eqref{eq:8udcvbns5wedfghjx9owq} and
\eqref{eq:njhy73r890rhjsk} - which derive solely from the fact that
$T$ satisfies equation \eqref{eq:conservedquantity}, independently
of $T$ being a solution of equation \eqref{eq:Ricattieq} or not - we
conclude that $T$ is a solution of $\rho$-Riccati equation
\eqref{eq:Ricattieq} (with a conserved quantity satisfying equation
\eqref{eq:conservedquantity}).

This correspondence between constrained Willmore Darboux
transformation parameters $\rho,T$ with $\rho > 1$ and pairs
$\alpha,L^{\alpha};\,-\alpha,\rho_{V} L^{\alpha}$ of untwisted
B\"{a}cklund transformation parameters, with $\alpha^{2}$ real,
establishes, furthermore, a correspondence between transforms, as we
verify next.

Suppose that the parameters $\alpha,L^{\alpha}$ define an untwisted
B\"{a}cklund transform of $L$ (i.e., that $L^{*}$ immerses).
Following \eqref{eq:r0vsT(veryclose)}, and in view of
$$\mathrm{Im}(-S+i)\subset S_{-},$$we conclude that, with respect to
the decomposition of $\underline{\C}^{4}$ into the direct sum of
$S_{+}$ and $S_{-}$,
$$\pi_{S_{+}}r(0)T^{-1}L_{+}=L_{+}.$$ Equivalently,
$$\pi_{S_{+}^{*}}r(\infty)L_{+}=T^{-1}L_{+},$$ with respect to
the decomposition of $\underline{\C}^{4}$ into the direct sum of
$S_{+}^{*}$ and $S_{-}^{*}$. In fact, given $l_{+}\in\Gamma(L_{+})$,
$\pi_{S_{+}^{*}}r(\infty)l_{+}$ is a section of $T^{-1}L_{+}$ if and
only if, for some $\lambda\in\C$, $r(\infty)l_{+}-\lambda
T^{-1}l_{+}$ is a section of $r(\infty)S_{-}$, or, equivalently,
recalling \eqref{eq:rinf=r0-1},
$$l_{+}=\lambda\pi_{S_{+}}( r(0)T^{-1}l_{+}).$$We conclude that
$L^{*}_{+}=\hat {L}_{+},$ or, equivalently,
$$L^{*}=\hat{L}.$$

If $L^{*}$ does not immerse, we define the untwisted B\"{a}cklund
transform of $L$ of parameters $\alpha,L^{\alpha}$ to be $\hat{L}$.
In this sense, we have just proven the following:
\begin{thm}\label{DvsDBT}
Constrained Willmore Darboux transformation of parameters $\rho,T$
with $\rho
>1$ is equivalent to untwisted B\"{a}cklund transformation of
parameters $\alpha, L^{\alpha}$ with $\alpha^{2}$ real. Constrained
Willmore Darboux transformation of parameters $\rho,T$ with $\rho
\leq 1$ is trivial.
\end{thm}

In particular, twisted B\"{a}cklund transforms of parameters
$\alpha, L^{\alpha}$ with $\alpha^{2}$ real,
$\tau(\alpha)L^{\alpha}$ non-real and $r(0)S_{+}$ non-j-stable, are
constrained Willmore Darboux transforms. Next we examine what they
correspond to under the correspondence established above between
untwisted B\"{a}cklund transformation of parameters $\alpha,
L^{\alpha}$ with $\alpha^{2}$ real and constrained Willmore Darboux
transformation of parameters $\rho,T$ with $\rho
>1$.

In what follows, let $\rho,T$ be constrained Willmore Darboux
transformation parameters, with $\rho>1$, and $\alpha,L^{\alpha}$ be
corresponding untwisted B\"{a}cklund transformation parameters,
under the correspondence established above.

\begin{Lemma}\label{5esdfcgvhjj766wrdxu7892q.'v}
$L^{\alpha}$ is orthogonal to $\rho_{V} L^{\alpha}$ at a point in
$M$ if and only if, at that point, either $W\cap S_{+}\neq\{0\}$ or
$W\cap S_{-}\neq\{0\}$.
\end{Lemma}
Before proceeding to the proof of the lemma, it is opportune to
emphasize the following fact.
\begin{Lemma}\label{bbbbvzbxxxxx5x5x5x5x4x3f}
Let $w_{1}:=w_{1}^{+}+w_{1}^{-},w_{2}:=w_{2}^{+}+w_{2}^{-}$ be a
frame of $W$ with $w_{i}^{\pm}\in\Gamma(S_{\pm})$, respectively, for
$i=1,2$. Then, at each point, $W\cap S_{\pm}\neq\{0\}$ if and only
if $w_{1}^{\mp}\wedge w_{2}^{\mp}=0,$ respectively.
\end{Lemma}
\begin{proof}
Let $p$ be a point in $M$. Let $w:=aw_{1}+bw_{2}\in\Gamma(W)$, with
$a,b\in\Gamma(\underline{\C})$, be not-zero at $p$. Then $w(p)$ is
in $S_{\pm}(p)$ if and only if, at the point $p$, one has either
$b\neq 0$ and $w_{2}^{\mp}=-\frac{a}{b}\,w_{1}^{\mp}$ or $b=0\neq a$
and $w_{1}^{\mp}=0$, respectively. On the other hand, if, at $p$,
$w_{1}^{\mp}\wedge w_{2}^{\mp}=0,$ then either $w_{i}(p)\in W(p)\cap
S_{\pm}(p)$, respectively, for some $i=1,2$; or
$w_{2}^{\mp}(p)=\lambda w_{1}^{\mp}(p)$, for some
$\lambda\in\C\backslash\{0\}$, in which case $-\lambda
w_{1}(p)+w_{2}(p)\in W(p)\cap S_{\pm}(p)\backslash\{0\}$,
respectively.

\end{proof}
Now we proceed to the proof of Lemma
\ref{5esdfcgvhjj766wrdxu7892q.'v}.
\begin{proof}
Let $w_{1}:=w_{1}^{+}+w_{1}^{-},w_{2}:=w_{2}^{+}+w_{2}^{-}$ be a
frame of $W$ with $w_{i}^{\pm}\in\Gamma(S_{\pm})$, respectively, for
$i=1,2$. Recall that $\rho_{V}=\tau(-1)$. At each point, the
orthogonality of $L^{\alpha}$ to $\rho_{V} L^{\alpha}$ is
characterized by
$$(\rho(w_{1}\wedge w_{2}),w_{1}\wedge w_{2})=0,$$or,
equivalently,
\begin{equation}\label{eq:ms1e2+e-vsrho}
(w_{1}^{+}\wedge w_{2}^{+},w_{1}^{-}\wedge w_{2}^{-})=0.
\end{equation}
On the other hand, at each point, equation \eqref{eq:ms1e2+e-vsrho}
holds if and only if either $w_{1}^{+}\wedge w_{2}^{+}=0$ or
$w_{1}^{-}\wedge w_{2}^{-}=0$, which, according to Lemma
\ref{bbbbvzbxxxxx5x5x5x5x4x3f}, completes the proof.
\end{proof}
\begin{Lemma}\label{qnotorthtorhoq}
If, at some point , $W\cap S_{+}=\{0\}=W\cap S_{-}$, then, at that
point,
$q_{\overline{\alpha}\,^{-1},\overline{L^{\alpha}}}\,(\alpha)L^{\alpha}$
is not orthogonal to $\rho_{V}
q_{\overline{\alpha}\,^{-1},\overline{L^{\alpha}}}\,(\alpha)L^{\alpha}$.
\end{Lemma}
\begin{proof}

Our argument is pointwise, so we work in a single fibre.

Suppose $W\cap S_{+}=\{0\}=W\cap S_{-}$. In that case, according to
Lemma \ref{5esdfcgvhjj766wrdxu7892q.'v}, $\overline{L^{\alpha}}$ is
not orthogonal to $\rho_{V}\overline{L^{\alpha}}$, or, equivalently,
$(\wedge^{2}\tau(\overline{\alpha})jW)\cap(\wedge^{2}\tau(-\overline{\alpha})jW)^{\perp}=\{0\},$
which characterizes the complementarity of
$\tau(\overline{\alpha})jW$ and $\tau(-\overline{\alpha})jW$ in
$\underline{\C}^{4}$. Set
$A:=\frac{\alpha-\overline{\alpha}\,^{-1}}{\alpha+\overline{\alpha}\,^{-1}}$
and fix a choice of $\sqrt{A}$. Then
$$q:=q_{\overline{\alpha}\,^{-1},\overline{L^{\alpha}}}\,(\alpha)=I\left\{
\begin{array}{ll} \sqrt{A} & \mbox{$\mathrm{on}\,\tau(\overline{\alpha})jW$}\\ \sqrt{A}\,^{-1} &
\mbox{$\mathrm{on}\,\tau(-\overline{\alpha})jW$}\end{array}\right..$$
Let $w_{1}:=w_{1}^{+}+w_{1}^{-},w_{2}:=w_{2}^{+}+w_{2}^{-}$ be a
frame of $W$ with $w_{i}^{\pm}\in\Gamma(S_{\pm})$, respectively, for
$i=1,2$. $L^{\alpha}$ is spanned by
\begin{eqnarray*}
l^{\alpha}&:=&\tau(\alpha)^{-1}(w_{1}\wedge w_{2})\\&=&\alpha
\,w_{1}^{+}\wedge w_{2}^{+}+w_{1}^{+}\wedge
w_{2}^{-}+w_{1}^{-}\wedge w_{2}^{+}+\alpha^{-1}w_{1}^{-}\wedge
w_{2}^{-}.
\end{eqnarray*}
The proof will consist of showing that $(\rho_{V}
ql^{\alpha},ql^{\alpha})\neq 0$, or, equivalently, in view of the
nullity of $ql^{\alpha}$, that
$$(\pi_{\wedge^{2}S_{+}}ql^{\alpha},\pi_{\wedge^{2}S_{-}}ql^{\alpha})\neq
0,$$ for
$\pi_{\wedge^{2}S_{+}}:\wedge^{2}\underline{\C}^{4}\rightarrow\wedge^{2}S_{+}$
and $\pi_{\wedge^{2}S_{-}}:\wedge^{2}\underline{\C}^{4}\rightarrow
\wedge^{2}S_{-}$ projections with respect to decomposition
\eqref{eq:w2C4w2s+w2s-s+s-}.

Fix a choice of $\sqrt{\alpha}$. Then
$$\tau(\overline{\alpha}) =I\left\{
\begin{array}{ll} \overline{\sqrt{\alpha}}\,^{-1} & \mbox{$\mathrm{on}\,S_{+}$}\\ \overline{\sqrt{\alpha}}&
\mbox{$\mathrm{on}\,S_{-}$}\end{array}\right.,$$ whereas
$$\tau(-\overline{\alpha}) =I\left\{
\begin{array}{ll} i\overline{\sqrt{\alpha}}\,^{-1} & \mbox{$\mathrm{on}\,S_{+}$}\\ -i\overline{\sqrt{\alpha}}&
\mbox{$\mathrm{on}\,S_{-}$}\end{array}\right.,$$and, therefore,
$$\tau(\overline{\alpha})jW=\langle
u_{1}:=\overline{\alpha}jw_{1}^{+}+jw_{1}^{-},u_{2}:=\overline{\alpha}jw_{2}^{+}+jw_{2}^{-}\rangle,$$
whereas $$\tau(-\overline{\alpha})jW=\langle
v_{1}:=-\overline{\alpha}jw_{1}^{+}+jw_{1}^{-},v_{2}:=-\overline{\alpha}jw_{2}^{+}+jw_{2}^{-}\rangle.$$
Given $i=1,2$ and
$a_{i}^{\pm},b_{i}^{\pm},c_{i}^{\pm},d_{i}^{\pm}\in\Gamma(\underline{\C})$
for which
$$w_{i}^{\pm}=a_{i}^{\pm}u_{1}+b_{i}^{\pm}u_{2}+c_{i}^{\pm}v_{1}+d_{i}^{\pm}v_{2},$$
respectively, the fact that $W\cap S_{+}=\{0\}=W\cap S_{-}$, or,
equivalently,
\begin{equation}\label{eq:jm12+jm12-neq0}
jw_{1}^{+}\wedge jw_{2}^{+}\neq 0\neq jw_{1}^{-}\wedge jw_{2}^{-},
\end{equation}
forces
\begin{equation}\label{eq:coeficoemntsforcados}
c_{i}^{+}=a_{i}^{+},\,\,d_{i}^{+}=b_{i}^{+},\,\,c_{i}^{-}=-a_{i}^{-},\,\,d_{i}^{-}=-b_{i}^{-}.
\end{equation}
Hence
\begin{eqnarray*}
qw_{i}^{\pm}&=&a_{i}^{\pm}\overline{\alpha}(\sqrt{A}\mp\sqrt{A}\,^{-1})jw_{1}^{+}+a_{i}^{\pm}(\sqrt{A}\pm\sqrt{A}\,^{-1})jw_{1}^{-}
\\ & &
\mbox{}+b_{i}^{\pm}\overline{\alpha}(\sqrt{A}\mp\sqrt{A}\,^{-1})jw_{2}^{+}+b_{i}^{\pm}(\sqrt{A}\pm\sqrt{A}\,^{-1})jw_{2}^{-},
\end{eqnarray*}
respectively, for $i=1,2$; and, therefore,
\begin{eqnarray*}
qw_{1}^{\pm}\wedge
qw_{2}^{\pm}&=&(a_{1}^{\pm}b_{2}^{\pm}-b_{1}^{\pm}a_{2}^{\pm})\,\alpha^{2}(\sqrt{A}\mp\sqrt{A}\,^{-1})^{2}jw_{1}^{+}\wedge
jw_{2}^{+}\\ & &
\mbox{}+(a_{1}^{\pm}b_{2}^{\pm}-b_{1}^{\pm}a_{2}^{\pm})\,\overline{\alpha}\,(A-A^{-1})(jw_{1}^{+}\wedge
jw_{2}^{-}+jw_{1}^{-}\wedge jw_{2}^{+})\\ & & \mbox{}+
(a_{1}^{\pm}b_{2}^{\pm}-b_{1}^{\pm}a_{2}^{\pm})\,(\sqrt{A}\pm\sqrt{A}\,^{-1})^{2}jw_{1}^{-}\wedge
jw_{2}^{-},
\end{eqnarray*}
respectively, as well as
\begin{eqnarray*}
qw_{i}^{+}\wedge
qw_{k}^{-}&=&(A-A^{-1})(a_{i}^{+}b_{k}^{-}-b_{i}^{+}a_{k}^{-})(\alpha^{2}jw_{1}^{+}\wedge
jw_{2}^{+}+jw_{1}^{-}\wedge jw_{2}^{-})\\ & &
\mbox{}+\overline{\alpha}\,a_{i}^{+}a_{k}^{-}((\sqrt{A}-\sqrt{A}\,^{-1})^{2}-(\sqrt{A}+\sqrt{A}\,^{-1})^{2})jw_{1}^{+}\wedge
jw_{1}^{-}\\ & &
\mbox{}+\overline{\alpha}\,(a_{i}^{+}b_{k}^{-}(\sqrt{A}-\sqrt{A}\,^{-1})^{2}-b_{i}^{+}a_{k}^{-}(\sqrt{A}+\sqrt{A}\,^{-1})^{2})jw_{1}^{+}\wedge
jw_{2}^{-}\\ & &
\mbox{}+\overline{\alpha}\,(a_{i}^{+}b_{k}^{-}(\sqrt{A}+\sqrt{A}\,^{-1})^{2}-b_{i}^{+}a_{k}^{-}(\sqrt{A}-\sqrt{A}\,^{-1})^{2})jw_{1}^{-}\wedge
jw_{2}^{+}\\ & &
\mbox{}+\overline{\alpha}\,b_{i}^{+}b_{k}^{-}((\sqrt{A}-\sqrt{A}\,^{-1})^{2}-(\sqrt{A}+\sqrt{A}\,^{-1})^{2})jw_{2}^{+}\wedge
jw_{2}^{-},
\end{eqnarray*}
for $i\neq k$. It follows that
$$\pi_{\wedge^{2}S_{+}}ql^{\alpha}=X_{+}\,jw_{1}^{-}\wedge jw_{2}^{-}$$and
$$\pi_{\wedge^{2}S_{-}}ql^{\alpha}=X_{-}\,jw_{1}^{+}\wedge
jw_{2}^{+},$$for
\begin{eqnarray*}
X_{+}&=&\alpha(\sqrt{A}+\sqrt{A}\,^{-1})^{2}(a_{1}^{+}b_{2}^{+}-b_{1}^{+}a_{2}^{+})+\alpha^{-1}(\sqrt{A}-\sqrt{A}\,^{-1})^{2}(a_{1}^{-}b_{2}^{-}-b_{1}^{-}a_{2}^{-})
\\ & & \mbox{}+
(A-A^{-1})(a_{1}^{+}b_{2}^{-}-b_{1}^{+}a_{2}^{-}+a_{1}^{-}b_{2}^{+}-b_{1}^{-}a_{2}^{+})
\end{eqnarray*}
and
\begin{eqnarray*}
\alpha^{-2}X_{-}&=&\alpha(\sqrt{A}-\sqrt{A}\,^{-1})^{2}(a_{1}^{+}b_{2}^{+}-b_{1}^{+}a_{2}^{+})+\alpha^{-1}(\sqrt{A}+\sqrt{A}\,^{-1})^{2}(a_{1}^{-}b_{2}^{-}-b_{1}^{-}a_{2}^{-})
\\ & & \mbox{}+
(A-A^{-1})(a_{1}^{+}b_{2}^{-}-b_{1}^{+}a_{2}^{-}+a_{1}^{-}b_{2}^{+}-b_{1}^{-}a_{2}^{+}).
\end{eqnarray*}
According to \eqref{eq:jm12+jm12-neq0}, $$(jw_{1}^{+}\wedge
jw_{2}^{+},jw_{1}^{-}\wedge jw_{2}^{-})\neq 0,$$so that $(\rho_{V}
ql^{\alpha},ql^{\alpha})$ vanishes if and only if
$X_{+}\,\alpha^{-2}X_{-}$ does. As $\alpha^{2}$ is real, or,
equivalently, $\alpha$ is either real or pure imaginary, we have
$$\alpha^{\pm 1}(\sqrt{A}\pm\sqrt{A}\,^{-1})^{2}=\frac{4\alpha^{\pm
3}}{\alpha^{2}-\alpha^{-2}}$$ and
$$\alpha^{\pm 1}(\sqrt{A}\mp\sqrt{A}\,^{-1})^{2}=\frac{4\alpha^{\mp 1}}{\alpha^{2}-\alpha^{-2}},$$
respectively, whilst, depending, respectively, on $\alpha$ being
real or pure imaginary,
$$A-A^{-1}=\mp\,\frac{4}{\alpha^{2}-\alpha^{-2}}.$$ Hence the
orthogonality of $\rho_{V} ql^{\alpha}$ to $ql^{\alpha}$ is
characterized, equivalently, by the equation
\begin{eqnarray*}\label{eq:gd}
0&=&
\alpha^{2}(a_{1}^{+}b_{2}^{+}-b_{1}^{+}a_{2}^{+})^{2}+(\alpha^{4}+\alpha^{-4})(a_{1}^{+}b_{2}^{+}-b_{1}^{+}a_{2}^{+})(a_{1}^{-}b_{2}^{-}-b_{1}^{-}a_{2}^{-})\\
& &
\mbox{}\mp(\alpha^{3}+\alpha^{-1})(a_{1}^{+}b_{2}^{+}-b_{1}^{+}a_{2}^{+})(a_{1}^{+}b_{2}^{-}-b_{1}^{+}a_{2}^{-}+a_{1}^{-}b_{2}^{+}-b_{1}^{-}a_{2}^{+})
\\ & &
\mbox{}\mp(\alpha^{-3}+\alpha)(a_{1}^{-}b_{2}^{-}-b_{1}^{-}a_{2}^{-})(a_{1}^{+}b_{2}^{-}-b_{1}^{+}a_{2}^{-}+a_{1}^{-}b_{2}^{+}-b_{1}^{-}a_{2}^{+})
\\ & &
\mbox{}+\alpha^{-2}(a_{1}^{-}b_{2}^{-}-b_{1}^{-}a_{2}^{-})^{2}+(a_{1}^{+}b_{2}^{-}-b_{1}^{+}a_{2}^{-}+a_{1}^{-}b_{2}^{+}-b_{1}^{-}a_{2}^{+})^{2},
\end{eqnarray*}
depending on $\alpha$ being real or pure imaginary, respectively.
Observe now that, on the other hand, according to
\eqref{eq:coeficoemntsforcados},
$$w_{i}^{+}=2\,(a_{i}^{+}jw_{1}^{-}+b_{i}^{+}jw_{2}^{-}),\,\,\,\,\,\,w_{i}^{-}=2\,\overline{\alpha}\,(a_{i}^{-}jm_{1}^{+}+b_{i}^{-}jw_{2}^{+}),$$
and, therefore,
$$jw_{i}^{+}=-2\,(\overline{a_{i}^{+}}\,w_{1}^{-}+\overline{b_{i}^{+}}\,w_{2}^{-}),\,\,\,\,\,\,jw_{i}^{-}=-2\,\alpha\,(\overline{a_{i}^{-}}\,w_{1}^{+}+\overline{b_{i}^{-}}\,w_{2}^{+}),$$
for $i=1,2$. It follows, in particular, that
$$jw_{1}^{+}\wedge
jw_{2}^{+}=4\,\overline{(a_{1}^{+}b_{2}^{+}-b_{1}^{+}a_{2}^{+})}\,w_{1}^{-}\wedge
w_{2}^{-}$$and, consequently, that $$(jw_{1}^{+}\wedge
jw_{2}^{+},w_{1}^{+}\wedge
w_{2}^{+})=4\,\overline{(a_{1}^{+}b_{2}^{+}-b_{1}^{+}a_{2}^{+})}\,(w_{1}^{-}\wedge
w_{2}^{-},w_{1}^{+}\wedge w_{2}^{+}),$$or, equivalently,
$$a_{1}^{+}b_{2}^{+}-b_{1}^{+}a_{2}^{+}=\frac{1}{4}\frac{(w_{1}^{+}\wedge
w_{2}^{+},jw_{1}^{+}\wedge jw_{2}^{+})}{(jw_{1}^{-}\wedge
jw_{2}^{-},jw_{1}^{+}\wedge jw_{2}^{+})}.$$Similarly, from
$$jw_{1}^{-}\wedge
jw_{2}^{-}=4\,\alpha^{2}\overline{(a_{1}^{-}b_{2}^{-}-b_{1}^{-}a_{2}^{-})}\,w_{1}^{+}\wedge
w_{2}^{+},$$ we conclude that
$$a_{1}^{-}b_{2}^{-}-b_{1}^{-}a_{2}^{-}=\frac{1}{4\alpha^{2}}\frac{(w_{1}^{-}\wedge
w_{2}^{-},jw_{1}^{-}\wedge jw_{2}^{-})}{(jw_{1}^{+}\wedge
jw_{2}^{+},jw_{1}^{-}\wedge jw_{2}^{-})}.$$On the other hand,
$$jw_{1}^{+}\wedge
jw_{2}^{-}=4\,\alpha\,(\overline{a_{1}^{+}a_{2}^{-}}w_{1}^{-}\wedge
w_{1}^{+}+\,\overline{a_{1}^{+}b_{2}^{-}}w_{1}^{-}\wedge
w_{2}^{+}+\,\overline{b_{1}^{+}a_{2}^{-}}w_{2}^{-}\wedge
w_{1}^{+}+\,\overline{b_{1}^{+}b_{2}^{-}}w_{2}^{-}\wedge
w_{2}^{+}),$$ showing that
$$a_{1}^{+}b_{2}^{-}=-\frac{1}{4\overline{\alpha}}\frac{(w_{1}^{+}\wedge
w_{2}^{-},jw_{1}^{+}\wedge jw_{2}^{-})}{(jw_{1}^{+}\wedge
jw_{2}^{+},jw_{1}^{-}\wedge jw_{2}^{-})},$$as well as
$$b_{1}^{+}a_{2}^{-}=\frac{1}{4\overline{\alpha}}\frac{(w_{1}^{+}\wedge
w_{2}^{-},jw_{1}^{-}\wedge jw_{2}^{+})}{(jw_{1}^{+}\wedge
jw_{2}^{+},jw_{1}^{-}\wedge jw_{2}^{-})}.$$ Similarly, the fact that
$$jw_{1}^{-}\wedge
jw_{2}^{+}=4\,\alpha\,(\overline{a_{1}^{-}a_{2}^{+}}w_{1}^{+}\wedge
w_{1}^{-}+\,\overline{a_{1}^{-}b_{2}^{+}}w_{1}^{+}\wedge
w_{2}^{-}+\,\overline{b_{1}^{-}a_{2}^{+}}w_{2}^{+}\wedge
w_{1}^{-}+\,\overline{b_{1}^{-}b_{2}^{+}}w_{2}^{+}\wedge
w_{2}^{-})$$ establishes
$$a_{1}^{-}b_{2}^{+}=-\frac{1}{4\overline{\alpha}}\frac{(w_{1}^{-}\wedge
w_{2}^{+},jw_{1}^{-}\wedge jw_{2}^{+})}{(jw_{1}^{+}\wedge
jw_{2}^{+},jw_{1}^{-}\wedge jw_{2}^{-})}$$and
$$b_{1}^{-}a_{2}^{+}=\frac{1}{4\overline{\alpha}}\frac{(w_{1}^{-}\wedge
w_{2}^{+},jw_{1}^{+}\wedge jw_{2}^{-})}{(jw_{1}^{+}\wedge
jw_{2}^{+},jw_{1}^{-}\wedge jw_{2}^{-})}.$$Set
$$Y_{\pm}:=(w_{1}^{\pm}\wedge w_{2}^{\pm},jw_{1}^{\pm}\wedge
jw_{2}^{\pm}),$$respectively, and
\begin{eqnarray*}
Y&:=&(w_{1}^{+}\wedge w_{2}^{-},jw_{1}^{+}\wedge
jw_{2}^{-})+(w_{1}^{+}\wedge w_{2}^{-},jw_{1}^{-}\wedge jw_{2}^{+})\\
& & \mbox{}+(w_{1}^{-}\wedge w_{2}^{+},jw_{1}^{-}\wedge jw_{2}^{+})+
(w_{1}^{-}\wedge w_{2}^{+},jw_{1}^{+}\wedge jw_{2}^{-}).
\end{eqnarray*}
It follows that $\rho_{V} l^{\alpha}$ is orthogonal to $ql^{\alpha}$
if and only if
$$\alpha^{2}Y_{+}^{2}+(\alpha^{2}+\alpha^{-6})Y_{+}Y_{-}+\alpha^{-6}Y_{-}^{2}\pm\frac{\alpha^{3}+\alpha^{-1}}{\overline{\alpha}}Y_{+}Y\pm\frac{\alpha^{-3}+\alpha}{\alpha^{2}\overline{\alpha}}Y_{-}Y+\alpha^{-2}Y^{2}=0,$$
depending on $\alpha$ being real or pure imaginary, respectively.
But, obviously, depending on $\alpha$ being real or pure imaginary,
respectively,
$$\frac{\alpha^{3}+\alpha^{-1}}{\overline{\alpha}}=\pm\,(\alpha^{2}+\alpha^{-2}),\,\,\,\,\,
\frac{\alpha^{-3}+\alpha}{\alpha^{2}\overline{\alpha}}=\pm\,(\alpha^{-6}+\alpha^{-2}).$$
Set $$\hat{Y}:=Y+Y_{+}+Y_{-}=(w_{1}\wedge w_{2},jw_{1}\wedge
jw_{2})\neq 0.$$To complete the proof, we are left to verify that
$$\alpha^{2}Y_{+}^{2}+\alpha^{-6}Y_{-}^{2}+\alpha^{-2}Y^{2}+\alpha^{2}Y_{+}(\hat{Y}-Y_{+})+\alpha^{-2}Y(\hat{Y}-Y)+\alpha^{-6}Y_{-}(\hat{Y}-Y_{-})\neq 0,$$
or, equivalently, as $\hat{Y}$ is not zero, that
$$\alpha^{2}Y_{+}+\alpha^{-2}Y+\alpha^{-6}Y_{-}\neq 0.$$
According to \eqref{eq:vconjvispositive}, for $i\neq k$,
$$(w_{i}^{+}\wedge w_{k}^{-},jw_{i}^{+}\wedge
jw_{k}^{-})\geq 0,$$ and, therefore,
$$Y\geq 0.$$On the other hand, the fact that
$W\cap S_{+}=\{0\}=W\cap S_{-}$, or, equivalently,
$$w_{1}^{+}\wedge w_{2}^{+}\neq 0\neq w_{1}^{-}\wedge w_{2}^{-},$$
intervenes, yet again, to ensure, according to Remark
\ref{inprodvoverlv}, that
$$Y_{\pm}>0. $$ The fact that $\alpha^{2}\in\R\backslash\{0\}$ completes
the proof.
\end{proof}

\begin{Lemma}
The following are equivalent, respectively and pointwise:

$i)$\,\,$W\cap S_{\pm}\neq \{0\};$

$ii)$\,\,$L^{\alpha}\subset\wedge^{2}S_{\pm}\oplus S_{\pm}\wedge
S_{\mp};$

$iii)$\,\,$\wedge^{2}S_{\pm}\cap (L^{\alpha})^{\perp}\neq\{0\}.$
\end{Lemma}

\begin{proof}
Since $W$ has rank $2$, or, equivalently, $\wedge^{2}W$ has rank
$1$, it is clear that $W\cap S_{\pm}\neq \{0\}$ if and only if
$\wedge^{2}W\subset\wedge^{2}S_{\pm}\oplus S_{\pm}\wedge S_{\mp}$,
respectively. The fact that $\tau(\alpha)$ preserves
$\wedge^{2}S_{\pm}$ and $S_{\pm}\wedge S_{\mp}$ establishes then the
equivalence of $i)$ and $ii)$. The equivalence of $ii)$ and $iii)$
is obvious, in view of the complementarity of $S_{+}$ and $S_{-}$ in
$\C^{4}$. Lastly, it is obvious that a subbundle of
$\wedge^{2}\underline{\C}^{4}$ is orthogonal to $\wedge^{2}S_{\pm}$
if and only if its projection onto  $\wedge^{2}S_{\mp}$,
respectively, with respect to the decomposition
\eqref{eq:w2C4w2s+w2s-s+s-}, is zero. The fact that $L^{\alpha}$ is
a line bundle completes the proof.
\end{proof}

We refer to an untwisted B\"{a}cklund transformation of parameters
$\alpha,L^{\alpha}$ for which, locally,
$$L^{\alpha}\nsubseteq\wedge^{2}S_{+}\oplus S_{+}\wedge S_{-}\,\,\,\,\mathrm{and}\,\,\,\,L^{\alpha}\nsubseteq \wedge^{2}S_{-}\oplus S_{+}\wedge S_{-},$$
as \textit{regular}; as well as to a constrained Willmore Darboux
transformation of parameters $\rho,T$ for which, locally,
$$W\cap S_{+}=\{0\}=W\cap S_{-},$$
for $W$ the eigenspace of $T-S$ associated to the eigenvalue
$\sqrt{\rho^{-1}-1}$. The regularity of an untwisted B\"{a}cklund
transformation of parameters $\alpha,L^{\alpha}$ with $\alpha^{2}$
real is equivalent to the regularity of the corresponding Darboux
transformation of parameters $\rho,T$ with $\rho>1$. The fact that
regularity can, equivalently, be characterized by
$$\wedge^{2}S_{+}\cap (L^{\alpha})^{\perp}=\{0\}=\wedge^{2}S_{-}\cap
(L^{\alpha})^{\perp}$$ makes clear that it is an open condition on
the points of $M$.

According to Lemmas \ref{5esdfcgvhjj766wrdxu7892q.'v} and \ref
{qnotorthtorhoq}, we have:

\begin{prop}\label{D-BT+BTforalpha2real}
An untwisted B\"{a}cklund transformation of parameters
$\alpha,L^{\alpha}$ with $\alpha^{2}$ real defines a twisted
B\"{a}cklund transformation of parameters $\alpha,L^{\alpha}$ if and
only if it is regular.
\end{prop}
Hence, following Theorem \ref{DvsDBT}:
\begin{prop}
Twisted B\"{a}cklund transformation of parameters $\alpha,
L^{\alpha}$ with $\alpha^{2}$ real, $\tau(\alpha)L^{\alpha}$
non-real and $r(0)S_{+}$ non-j-stable is equivalent to regular
constrained Willmore Darboux transformation of parameters $\rho,T$
with $\rho>1$.
\end{prop}

\appendix\chapter{Hopf differential and
umbilics}\label{appHopf+umbilics}

\markboth{\tiny{A. C. QUINTINO}}{\tiny{CONSTRAINED WILLMORE
SURFACES}}

According to \eqref{eq:komega}, given $z$ and $\omega$ holomorphic
charts of $(M,\mathcal{C}_{\Lambda})$, $k^{z}$ vanishes if and only
if $k^{\omega}$ does.

\begin{prop}\label{umbilicvshopfprop}
Suppose $\Lambda\subset\underline{\R}^{4,1}$ is an isothermic
surface in $3$-space. Then, given $v_{\infty}\in\R^{4,1}$
light-like, the umbilic points of the surface in $S_{v_{\infty}}$
defined by $\Lambda$ are the points at which the Hopf differential
of $\Lambda$ vanishes.
\end{prop}

To prove the proposition, we start by establishing a relation
between the Hopf differential of $\Lambda$ and the (classical) Hopf
differential of $\sigma^{z}$:

\begin{Lemma}\label{hopfdiffervsclassical}
Let $v_{\infty}\in\R^{n+1,1}$ be non-zero, $\sigma_{\infty}$ be the
surface defined by $\Lambda$ in $S_{v_{\infty}}$, and $\xi$ be a
normal vector field to $\sigma_{\infty}$. Let $z$ be a holomorphic
chart of $M$. Then
$$(\sigma_{zz}^{z},\xi)=\lambda(k^{z},\mathcal{Q}\xi),$$for $\mathcal{Q}:N_{\infty}\rightarrow S^{\perp}$ the
isometry defined in Section \ref{normalbundle} and
$\lambda\in\Gamma(\underline{\R})$ defined by
$\sigma^{z}=\lambda\sigma_{\infty}$.
\end{Lemma}
\begin{proof}
Recall, yet again, that $\sigma_{\infty}^{*}TS_{v_{\infty}}$
consists of the orthogonal complement in $\underline{\R}^{n+1,1}$ of
the non-degenerate bundle
$\langle\sigma_{\infty},v_{\infty}\rangle$. Let $\pi _{N_{\infty}}$
denote the orthogonal projection of $\underline
{\mathbb{R}}^{n+1,1}=d\sigma _{\infty}(TM)\oplus N _{\infty}\oplus
\langle v_{\infty},\sigma _{\infty}\rangle $ onto $N_{\infty}$.
Following equation \eqref{piSperppiNinf}, we have
$((\sigma_{\infty})_{zz},\xi)=(\pi_{N_{\infty}}(\sigma_{\infty})_{zz},\xi)=(\pi_{S^{\perp}}(\sigma_{\infty})_{zz},\mathcal{Q}\xi).$
On the other hand, writing $\sigma^{z}=\lambda\sigma_{\infty}$ with
$\lambda\in\Gamma(\underline{\R})$, we have
$\sigma_{zz}^{z}=\lambda(\sigma_{\infty})_{zz}+2\lambda
_{z}(\sigma_{\infty})_{z}+\lambda_{zz}\sigma_{\infty}$ and,
therefore,
\begin{equation}\label{eq:8765r4edfghjkl?76543ewsdcvbnm}
(\sigma_{zz}^{z},\xi)=\lambda ((\sigma_{\infty})_{zz},\xi),
\end{equation}
which, together with equation \eqref{eq:kzpiSperpsigmainfzz},
completes the proof.
\end{proof}

Now suppose $\Lambda\subset\underline{\R}^{4,1}$ is an isothermic
surface in $3$-space. Let $v_{\infty}\in\R^{4,1}$ be light-like and
$\sigma_{\infty}$ be the surface in $S_{v_{\infty}}$ defined by
$\Lambda$. Cf. Theorem \ref{isoLambequivisosigmainf},
$\sigma_{\infty}$ is isothermic. Let $x,y$ be conformal curvature
line coordinates of $\sigma_{\infty}$. Fix $\xi\in
\Gamma(N_{\infty})$ unit. Let $A_{\infty}^{\xi}$ denote the shape
operator of $\sigma_{\infty}$ with respect to $\xi$. Then
$$A^{\xi}_{\infty}(\delta_{x})=k_{1}\delta_{x},\,\,\,\,A^{\xi}_{\infty}(\delta_{y})=k_{2}\delta_{y},$$
for some $k_{1},k_{2}\in\Gamma(\underline{\R})$, the principal
curvatures of $\sigma_{\infty}$, locally, in the domain of the chart
$z:=x+iy$ of $M$. In view of the conformality of the coordinates $x$
and $y$, $z$ is, up to a change of orientation in $M$, a holomorphic
chart of $(M,\mathcal{C}_{\Lambda})$. In these conditions:
\begin{Lemma}
The (local) principal curvatures, $k_{1}$ and $k_{2}$, of
$\sigma_{\infty}$ relate to the Hopf differential of $\Lambda$ by
$$e^{u}(k_{1}-k_{2})=4(k^{z},\mathcal{Q}\xi),$$with
$u\in\Gamma(\underline{\R})$.
\end{Lemma}
\begin{proof}
The fact that $x,y$ are conformal,
$dx^{2}+dy^{2}\in\mathcal{C}_{\Lambda}\ni g_{\infty},$ establishes,
in particular,
$$g_{\infty}(\delta_{x},\delta_{x})=e^{u}=g_{\infty}(\delta_{y},\delta_{y}),$$
for some $u\in\Gamma(\underline{\R})$, and, therefore,
$$g_{\infty}(A^{\xi}_{\infty}(\delta_{x}),\delta_{x})-g_{\infty}(A^{\xi}_{\infty}(\delta_{y}),\delta_{y})=e^{u}(k_{1}-k_{2}).$$
On the other hand, according to equation \eqref{eq:secffvsshapeop},
$$g_{\infty}(A^{\xi}_{\infty}(\delta_{x}),\delta_{x})-g_{\infty}(A^{\xi}_{\infty}(\delta_{y}),\delta_{y})=(\Pi_{\infty}(\delta_{x},\delta_{x}),\xi)-(\Pi_{\infty}(\delta_{y},\delta_{y}),\xi).$$
Since the conformal coordinates $x,y$ are curvature line, it
follows, by equation \eqref{eq:Pidiagonal}, that
$$e^{u}(k_{1}-k_{2})=4(\Pi_{\infty}(\delta_{z},\delta_{z}),\xi)=4(\pi_{N_{\infty}}(\sigma_{\infty})_{zz},\xi),$$
for $\pi_{N_{\infty}}$ the orthogonal projection of $\underline
{\mathbb{R}}^{4,1}=d\sigma _{\infty}(TM)\oplus N _{\infty}\oplus
\langle v_{\infty},\sigma _{\infty}\rangle $ onto
$$N_{\infty}\subset\sigma_{\infty}^{*}TS_{v_{\infty}}=\langle v_{\infty},\sigma
_{\infty}\rangle^{\perp},$$ and, consequently,
$e^{u}(k_{1}-k_{2})=4((\sigma_{\infty})_{zz},\xi).$ Now write
$\sigma^{z}=\lambda\,\sigma_{\infty}$, with
$\lambda\in\Gamma(\underline{\R})$. By equation
\eqref{eq:8765r4edfghjkl?76543ewsdcvbnm}, we get
$e^{u}(k_{1}-k_{2})=4\lambda^{-1}(\sigma_{zz}^{z},\xi)$ and the
conclusion follows then from Lemma \ref{hopfdiffervsclassical}.
\end{proof}
Proposition \ref{umbilicvshopfprop} follows.

\chapter{Twisted vs. untwisted B\"{a}cklund transformation parameters}

\markboth{\tiny{A. C. QUINTINO}}{\tiny{CONSTRAINED WILLMORE
SURFACES}}

Twisted and untwisted B\"{a}cklund transformation parameters
conditions at a point are not equivalent. On the one hand, the
choice of $L^{\alpha}$ as a null line bundle defined naturally by
$d^{\alpha,q}_{V}$-parallel transport of $l^{\alpha}_{p}$, for
$l^{\alpha}$ a non-zero section of $\wedge^{2}S_{+}$ and $p$ a point
in $M$, establishes the existence of untwisted B\"{a}cklund
transformation parameters $\alpha, L^{\alpha}$, at the point $p$,
with $\rho_{V}L^{\alpha}$ orthogonal to $L^{\alpha}$ at $p$. On the
other hand, the choice of $L^{\alpha}$ as a null line bundle defined
naturally by $d^{\alpha,q}_{V}$-parallel transport of
$l^{\alpha}_{p}$, for some point $p\in M$ and some
$l^{\alpha}:=v_{0}+v_{+}+v_{-}\in\Gamma(\wedge^{2}\underline{\C}^{4})$
with $v_{0}\in\Gamma(S_{+}\wedge S_{-})$ purely imaginary unit,
$v_{-}\in\Gamma(\wedge^{2}S_{-})$ with
\begin{equation}\label{eq:v_-relatedtov_0}
(v_{-},\overline{v_{-}})=\frac{1}{2}\mid\alpha\mid^{-2},
\end{equation}
and
$v_{+}=-\frac{1}{2}\,(v_{-},\overline{v_{-}})^{-1}\overline{v_{-}}$,
establishes the existence of twisted B\"{a}cklund transformation
parameters $\alpha, L^{\alpha}$, at the point $p$, with
$\tau(\alpha)L^{\alpha}$ real at $p$. Indeed, the conditions on
$v_{0}$, $v_{-}$ and $v_{+}$ ensure immediately the
non-orthogonality of $L^{\alpha}$ and $\rho_{V}L^{\alpha}$ and the
reality of $\tau(\alpha)L^{\alpha}$, at the point $p$, which,
together, ensure the non-orthogonality of $\tilde{L}^{\alpha}$ and
$\rho_{V}\tilde{L}^{\alpha}$ at $p$, as we verify next. We work in
the fibre at $p$. Computation shows that, as $L^{\alpha}$ is null
and not orthogonal to $\rho_{V}L^{\alpha}$,\footnote{The
intervention of the reality of $\tau(\alpha)L^{\alpha}$ is
deliberately left to the next stage.}
$$q(\alpha)l^{\alpha}=v_{0}+X_{0}\overline{v_{0}}+X_{-}v_{-}+X_{+}\overline{v_{-}}$$
for
\begin{eqnarray*}
X_{0}&=&A(\frac{(v_{0},\overline{v_{0}})}{\overline{2(v_{0},v_{0})}}-\frac{(v_{-},\overline{v_{-}})}{2\overline{(v_{0},v_{0})}}-\frac{(v_{0},v_{0})}{8(v_{-},\overline{v_{-}})})\\
& &
\mbox{}-\frac{(v_{0},\overline{v_{0}})}{\overline{(v_{0},v_{0})}}+
A^{-1}(\frac{(v_{0},\overline{v_{0}})}{2\overline{(v_{0},v_{0})}}+\frac{(v_{-},\overline{v_{-}})}{2\overline{(v_{0},v_{0})}}+\frac{(v_{0},v_{0})}{8(v_{-},\overline{v_{-}})}),
\end{eqnarray*}
\begin{eqnarray*}
X_{-}&=&A(
\frac{\mid(v_{0},v_{0})\mid^{2}}{16(v_{-},\overline{v_{-}})^{2}}-\frac{(v_{0},\overline{v_{0}})}{4(v_{-},\overline{v_{-}})}+\frac{1}{4})
+\frac{1}{2}\\& &
\mbox{}-\frac{\mid(v_{0},v_{0})\mid^{2}}{8(v_{-},\overline{v_{-}})^{2}}+A^{-1}(\frac{\mid(v_{0},v_{0})\mid^{2}}{16(v_{-},\overline{v_{-}})^{2}}+\frac{(v_{0},\overline{v_{0}})}{4(v_{-},\overline{v_{-}})}+\frac{1}{4})
\end{eqnarray*}
and
\begin{eqnarray*}
X_{+}&=&A(\frac{(v_{0},\overline{v_{0}})}{2\overline{(v_{0},v_{0})}}-\frac{(v_{-},\overline{v_{-}})}{2\overline{(v_{0},v_{0})}}
-\frac{(v_{0},v_{0})}{8(v_{-},\overline{v_{-}})})+\frac{(v_{-},\overline{v_{-}})}{\overline{(v_{0},v_{0})}}\\&
&
\mbox{}-\frac{(v_{0},v_{0})}{2(v_{-},\overline{v_{-}})}+\frac{(v_{0},v_{0})}{4(v_{-},\overline{v_{-}})}-A^{-1}(\frac{(v_{0},\overline{v_{0}})}{2\overline{(v_{0},v_{0})}}+\frac{(v_{0},v_{0})}{8(v_{-},\overline{v_{-}})}+\frac{(v_{-},\overline{v_{-}})}{2\overline{(v_{0},v_{0})}});
\end{eqnarray*}
and, therefore,
\begin{eqnarray*}
(\rho_{V}q(\alpha)l^{\alpha},q(\alpha)l^{\alpha})&=&
A^{2}\frac{(v_{0},\overline{v_{0}})^{2}-2(v_{-},\overline{v_{-}})(v_{0},\overline{v_{0}})+(v_{-},\overline{v_{-}})^{2}}{2\overline{(v_{0},v_{0})}}+A^{2}\frac{(v_{0},v_{0})}{4}\\&
&
\mbox{}-A^{2}\frac{(v_{0},\overline{v_{0}})(v_{0},v_{0})}{4(v_{-},\overline{v_{-}})}
+A^{2}\frac{(v_{0},v_{0})^{2}\overline{(v_{0},v_{0})}}{32(v_{-},\overline{v_{-}})^{2}}
\\ & &
\mbox{}-\frac{(v_{0},\overline{v_{0}})^{2}+(v_{-},\overline{v_{-}})^{2}}{\overline{(v_{0},v_{0})}}+\frac{3(v_{0},v_{0})}{2}-\frac{(v_{0},v_{0})^{2}\overline{(v_{0},v_{0})}}{16(v_{-},\overline{v_{-}})^{2}}
\\& &
\mbox{}+A^{-2}\frac{(v_{0},\overline{v_{0}})^{2}+2(v_{0},\overline{v_{0}})(v_{-},\overline{v_{-}})+(v_{-},\overline{v_{-}})^{2}}{2\overline{(v_{0},v_{0})}}+
\\ & &
+A^{-2}\frac{(v_{0},v_{0})}{4}+A^{-2}\frac{(v_{0},v_{0})(v_{0},\overline{v_{0}})}{4(v_{-},\overline{v_{-}})}+A^{-2}\frac{(v_{0},v_{0})^{2}\overline{(v_{0},v_{0})}}{32(v_{-},\overline{v_{-}})^{2}};
\end{eqnarray*}
with
$$A:=\frac{\alpha-\overline{\alpha}\,^{-1}}{\alpha+\overline{\alpha}\,^{-1}}=\frac{\mid\alpha\mid
^{2}-1}{\mid\alpha\mid ^{2}+1}\in\R.$$ On the other hand, in view of
the reality of  $\tau(\alpha)l^{\alpha}$, or, equivalently, of the
fact that $v_{0}$ is a purely imaginary unit, together with equation
\eqref{eq:v_-relatedtov_0}, the coefficients $X_{0}$, $X_{-}$ and
$X_{+}$ simplify to
$$X_{0}+1=A(\frac{1}{2}\mid\alpha\mid^{-1}+\frac{1}{2}\mid\alpha\mid)^{2}-A^{-1}(\frac{1}{2}\mid\alpha\mid^{-1}-\frac{1}{2}\mid\alpha\mid)^{2},$$
$$X_{-}=A(\frac{1}{2}\mid\alpha\mid^{2}+\frac{1}{2})^{2}+A^{-1}(\frac{1}{2}\mid\alpha\mid^{2}-\frac{1}{2})^{2}$$
and
$$X_{+}=A(\frac{1}{2}\mid\alpha\mid^{-1}+\frac{1}{2}\mid\alpha\mid)^{2}+A^{-1}(\frac{1}{2}\mid\alpha\mid-\frac{1}{2}\mid\alpha\mid^{-1})^{2},$$
respectively. In that case,
$$
(v_{0},v_{0})^{-1}(\rho_{V}q(\alpha)l^{\alpha},q(\alpha)l^{\alpha})=(X_{0}+1)^{2}+\mid\alpha\mid^{-2}X_{-}X_{+},$$
whilst the reality of $A$ ensures the reality of $X_{0}$, as well as
the positiveness of $X_{-}X_{+}$, leading us to conclude that
$$(\rho_{V}q(\alpha)l^{\alpha},q(\alpha)l^{\alpha})\neq
0.$$

\backmatter

\renewcommand{\bibname}{References}

$\newline$

\noindent\tiny{CENTRO DE MATEM\'{A}?TICA E APLICA\c{C}\~{O}ES
FUNDAMENTAIS DA UNIVERSIDADE DE LISBOA, PORTUGAL}

$\newline$

\noindent\textit{E-mail address:} aurea@ptmat.fc.ul.pt

\newpage
$\newline\newline$

\begin{center}
\Large{ERRATA}
\end{center}
\begin{center}
January 2010
\end{center}

$\newline\newline\newline$

Page $132$, lines $18-21$: replace "The existence of a conserved
quantity $p(\lambda)$ of $\hat{\Lambda}$ establishes (...)
establishing $\hat{\Lambda}$ as a CMC surface in the space-form
$S_{\hat{v}_{\infty}}$." by "The existence of a conserved quantity
$p(\lambda)=\lambda^{-1}v+v_{0}+\lambda\overline{v}$ of
$\hat{\Lambda}$ establishes, in particular, the constancy of
$\hat{v}_{\infty}:=p(1)$. To complete the proof, we establish
$d(\hat{v}_{\infty}^{\perp},\hat{v}_{\infty}^{\perp})=0$, or,
equivalently,
$d^{0,1}(\hat{v}_{\infty}^{\perp},\hat{v}_{\infty}^{\perp})=0$, for
the case $\hat{v}_{\infty}^{\perp}\neq 0$. First note that, if
$\hat{v}_{\infty}^{\perp}\neq 0$, then $v\neq 0$ and, therefore,
$S^{\perp}=\langle v\rangle$ (locally), so that
$\overline{v}=\lambda v$, with
$\lambda\in\Gamma(\underline{\mathbb{C}})$. According to Remark
7.5., it follows that
\begin{eqnarray*}
d(\hat{v}_{\infty}^{\perp},\hat{v}_{\infty}^{\perp})&=&d(v,v)+2d(v,\overline{v})+\overline{d(v,v)}\\&=&
2d(v,\lambda v)\\&=&2(d\lambda)(v,v).
\end{eqnarray*}
On the other hand, the constancy of $\hat{v}_{\infty}$,
$(d\lambda)v=-dv_{0}-(\lambda+1)dv$, establishes, in particular,
$(d\lambda)v=-\mathcal{N}v_{0}-(\lambda+1)\mathcal{D}v$, by
orthogonal projection onto the normal bundle to the central sphere
congruence of $\hat{\Lambda}$. Considering $(0,1)$-parts, we
conclude then, according to Theorem $7.2.$, that
$(d^{0,1}\lambda)v=-\mathcal{N}^{0,1}v_{0}$. Hence
$$d^{0,1}(\hat{v}_{\infty}^{\perp},\hat{v}_{\infty}^{\perp})=-2(\mathcal{N}^{0,1}v_{0},v)=2(v_{0},\mathcal{N}^{0,1}v)=2\overline{\lambda(v_{0},\mathcal{N}^{1,0}v)},$$
having in consideration the reality of $v_{0}$. The fact that
$(v_{0},\mathcal{N}^{1,0}v)=0$, as observed in Remark $7.5.$,
completes the proof."$\newline$

$\newline\newline$

\end{document}